\documentclass[a4paper,11pt]{article}

\usepackage{amsmath,amsthm}
\usepackage{amssymb}
\usepackage{color}
\usepackage{breqn}
\usepackage{enumerate}
\usepackage{graphicx}
\usepackage{bm}
\usepackage{bbm}
\usepackage[normalem]{ulem}
\usepackage[affil-it]{authblk}
\usepackage{tabu}
\usepackage{bold-extra}
\usepackage[hang,flushmargin]{footmisc} 
\usepackage{hyperref}
\usepackage{mathtools}
\usepackage{relsize}
\usepackage{titlesec}
\usepackage[dvipsnames]{xcolor}
\usepackage{comment}

\newtheorem{theorem}{Theorem}
\newtheorem{corollary}[theorem]{Corollary}
\newtheorem{lemma}[theorem]{Lemma}
\newtheorem{proposition}[theorem]{Proposition}

\newtheorem{claim}[theorem]{Claim}
\newtheorem{fact}[theorem]{Fact}
\newtheorem*{claim*}{Claim}
\newtheorem*{multclaim*}{Multilinearity claim}
\newtheorem*{claima*}{Claim A}
\newtheorem*{claimb*}{Claim B}
\newtheorem*{claimc*}{Claim C}
\newtheorem*{claimd*}{Claim D}

\theoremstyle{definition}
\newtheorem{defin}[theorem]{Definition}

\titleformat{\section}[hang]{\scshape\large\bfseries\filcenter}{\S\thesection}{4pt}{}
\titleformat{\subsection}[hang]{\scshape\bfseries}{\thesubsection.}{4pt}{}

\allowdisplaybreaks

\def \ls#1#2 {^{#1}\!#2}

\newcommand\id{\mathbbm{1}}	

\newcommand{\sgn}{
\operatorname{sgn}}
\newcommand{\bias}{
\operatorname{bias}}

\newcommand\mder{\partial}
\newcommand{\tdt}{
	\times\cdots\times
}

\newcommand{\exx}{
  \mathop{
    \mathchoice{\vcenter{\hbox{\larger[4]$\mathbb{E}$}}}
               {\kern0pt\mathbb{E}}
               {\kern0pt\mathbb{E}}
               {\kern0pt\mathbb{E}}
  }\displaylimits
}

\newcommand{\on}[1]{\operatorname{#1}}

\newcommand{\cmm}[1]{\ignorespaces}

\newcommand{\img}{\operatorname{Im}}
\newcommand{\tightoverset}[2]{
  \mathop{#2}\limits^{\vbox to -.5ex{\kern-1.15ex\hbox{$#1$}\vss}}}
\newcommand{\conv}{
	\tightoverset{\boldsymbol{-}}{\ast}
}
\newcommand\apps[1]{
	\overset{#1}{\approx}
}

\newcommand{\chg}[1]{\textcolor{Black}{#1}}

\newcommand{\bigconv}[1]{
	\mathbf{C}_{#1}
}

\newcommand{\du}[1]{\underline{\underline{#1}}}

\newcommand{\prank}{
\operatorname{prank}}

\newcommand{\dom}{
	\operatorname{dom}
}	
\newcommand{\supp}{
	\operatorname{supp}
}			

\newcommand\con{\bm{\mathrm{C}}} 
\newcommand\cons{\bm{\mathrm{c}}}
	
\newcommand\restr[2]{{
  \left.\kern-\nulldelimiterspace 
  #1 
  \vphantom{\big|} 
  \right|_{#2} 
}}

\makeatletter
\newcommand{\subalign}[1]{%
  \vcenter{%
    \Let@ \restore@math@cr \default@tag
    \baselineskip\fontdimen10 \scriptfont\tw@
    \advance\baselineskip\fontdimen12 \scriptfont\tw@
    \lineskip\thr@@\fontdimen8 \scriptfont\thr@@
    \lineskiplimit\lineskip
    \ialign{\hfil$\m@th\scriptstyle##$&$\m@th\scriptstyle{}##$\hfil\crcr
      #1\crcr
    }%
  }%
}
\makeatother

\newcommand\blfootnote[1]{%
  \begingroup
  \renewcommand\thefootnote{}\footnote{#1}%
  \addtocounter{footnote}{-1}%
  \endgroup
}

\newcommand\ex{
	\mathop{\mathbb{E}}
}
\newcommand\codim{
	\operatorname{codim}
}

\makeatletter
\newcommand*\bcdot{\mathpalette\bigcdot@{0.5}}
\newcommand*\bigcdot@[2]{\mathbin{\vcenter{\hbox{\scalebox{#2}{$\m@th#1\bullet$}}}}}
\makeatother

\def\colon{:}
	
\makeatletter
\def\blfootnote{\gdef\@thefnmark{}\@footnotetext}
\makeatother	

\newcommand\fco{\lbrack}
\newcommand\fcc{\rbrack^\wedge}
	
\begin{document}
\begin{center}\Large\noindent{\bfseries{\scshape An inverse theorem for Freiman multi-homomorphisms}}\\[24pt]\normalsize\noindent{\scshape W. T. Gowers\dag\hspace{3pt}and\hspace{3pt}L. Mili\'cevi\'c\ddag}\\[6pt]
\end{center}
\blfootnote{\noindent\dag\ Coll\`ege de France and University of Cambridge\\\phantom{\dag} Email: \texttt{wtg10@dpmms.cam.ac.uk}\\
\noindent\ddag\ Mathematical Institute of the Serbian Academy of Sciences and Arts\\\phantom{\dag\ }Email: \texttt{luka.milicevic@turing.mi.sanu.ac.rs}}

\footnotesize
\begin{changemargin}{1in}{1in}
\centerline{\sc{\textbf{Abstract}}}
\phantom{a}\hspace{12pt}~Let $G_1, \dots, G_k$ and $H$ be vector spaces over a finite field $\mathbb{F}_p$ of prime order. Let $A \subset G_1 \tdt G_k$ be a set of size $\delta |G_1| \cdots |G_k|$. Let a map $\phi \colon A \to H$ be a multi-homomorphism, meaning that for each direction $d \in [k]$, and each element $(x_1, \dots, x_{d-1}, x_{d+1}, \dots, x_k)$ of $G_1\tdt G_{d-1}\times G_{d+1}\tdt G_k$, the map that sends each $y_d$ such that $(x_1, \dots,$ $x_{d-1},$ $y_d,$ $x_{d+1}, \dots,$ $x_k) \in A$ to $\phi(x_1, \dots,$ $x_{d-1},$ $y_d,$ $x_{d+1}, \dots,$ $x_k)$ is a Freiman homomorphism (of order 2). In this paper, we prove that for each such map, there is a multiaffine map $\Phi \colon G_1 \tdt G_k \to H$ such that $\phi = \Phi$ on a set of density $\Big(\exp^{(O_k(1))}(O_{k,p}(\delta^{-1}))\Big)^{-1}$, where $\exp^{(t)}$ denotes the $t$-fold exponential.\\
\phantom{a}\hspace{12pt}~Applications of this theorem include:
\begin{itemize}
\item a quantitative inverse theorem for approximate polynomials mapping $G$ to $H$, for finite-dimensional $\mathbb{F}_p$-vector spaces $G$ and $H$, in the high-characteristic case,
\item a quantitative inverse theorem for uniformity norms over finite fields in the high-characteristic case, and
\item a quantitative structure theorem for dense subsets of $G_1 \tdt G_k$ that are subspaces in the principal directions (without additional characteristic assumptions).
\end{itemize}
\end{changemargin} 
\normalsize

\section{Introduction}

The finite-fields version of Freiman's theorem can be stated as follows.

\begin{theorem}[Freiman's theorem in $\mathbb{F}_p^n$]Let $p$ be a prime, let $V$ be a finite-dimensional vector space over $\mathbb{F}_p$, let $A \subset V$, and suppose that $|A+A| \leq K|A|$. Then there is a subspace $U \leq V$ such that $A \subseteq U$ and $|U| \leq O_{K, p}(|A|)$.\end{theorem}

Freiman's original theorem concerned finite sets of integers~\cite{Freiman}. In~\cite{Ruzsa}, Ruzsa found a highly influental new proof of this result, which was generalized to all Abelian groups by Green and Ruzsa in~\cite{greenRuzsaFreiman}. The result of this form with the strongest bounds known so far\footnote{\chg{Sanders' result was the state of the art at the time of writing of this paper. Since then, there have been a few important developments in higher order Fourier analysis that concern the results appearing in the introduction. However, as our main results have not been affected, we opted to keep this section in the original form.}} was given by Sanders in~\cite{Sanders}.\\
\indent An important ingredient in Freiman's proof was the notion of a Freiman homomorphism. When $G$ and $H$ are Abelian groups and $A \subset G$ is a subset, we say that $\phi \colon A \to H$ is a \emph{(Freiman) homomorphism of order $k$}, or simply a \emph{$k$-homomorphism}, if whenever $a_1, \dots, a_k, b_1, \dots, b_k \in A$ satisfy $\sum_{i=1}^k a_i = \sum_{i=1}^k b_i$, then $\sum_{i=1}^k \phi(a_i) = \sum_{i=1}^k \phi(b_i)$ holds as well. In particular, 2-homomorphisms are maps that respect all \emph{additive quadruples} (quadruples $(a, b, c, d)$ such that $a+b = c+d$) in the given set $A$. Freiman homomorphisms can be thought of as approximate analogues of linear maps. Indeed, if we combine Freiman's theorem with the Balog--Szemer\'edi--Gowers theorem~\cite{BalogSzemeredi},~\cite{TimSze}, we may obtain the following result.\\

\begin{theorem}\label{introFreAndBSG}Let $V$ and $H$ be vector spaces over $\mathbb{F}_p$, let $A \subset V$ be a subset of size at least $\delta |V|$ and let $\phi \colon A \to H$ be a 2-homomorphism. Then there exist an affine map $\psi \colon V \to H$ and a subset $A' \subset A$ of size $\Omega_{\delta}(|V|)$ such that $\phi(a) = \psi(a)$ for all $a \in A'$.\end{theorem}

The main result of this paper is a generalization of Theorem \ref{introFreAndBSG} to the setting of multivariate maps. We begin by giving a definition of a class of functions that have the same relationship to Freiman homomorphisms that multilinear maps have to linear maps. That is, they are functions of several variables that are Freiman homomorphisms in each variable separately. A formal definition is as follows.

\begin{defin} Let $G_1, \dots, G_k$ and $H$ be finite-dimensional vector spaces over $\mathbb{F}_p$, and let $A$ be a subset of $G_1 \tdt G_k$. A function $\phi\colon A \to H$ is a \emph{Freiman multi-homomorphism of order $k$} if for every $d\in\{1,2,\dots,k\}$ and every $(a_1,\dots,a_{d-1},a_{d+1},\dots,a_k)\in G_1\tdt G_{d-1}\times G_{d+1}\tdt G_k$, the map from $\{x_d \in G_d \colon (a_1,\dots,a_{d-1},x_d, a_{d+1}, \dots, a_k) \in A\}$ to $H$ defined by the formula $x_d\mapsto\phi(a_1,\dots,a_{d-1},x_d,a_{d+1},\dots,a_k)$ is a Freiman homomorphism of order $k$. \end{defin}

We shall also call these multi-$k$-homomorphisms. Indeed, often we shall simply call them multi-homomorphisms, in which case, as with Freiman homomorphisms, it should be understood that $k=2$. \chg{Let us define other important classes of maps that frequently occur in this paper.}\\

\begin{defin}
    \chg{Let $G, G_1, \dots, G_k$ and $H$ be finite-dimensional vector spaces over $\mathbb{F}_p$, and let $A$ be a subset of $G_1 \tdt G_k$. A function $\phi : G \to H$ is \emph{affine} if $\phi(a + b - c) = \phi(a) + \phi(b) - \phi(c)$ holds for all $a,b,c \in G$. A function $\phi\colon A \to H$ is \textit{multiaffine} if for every $d\in\{1,2,\dots,k\}$ and every $(a_1,\dots,a_{d-1},a_{d+1},\dots,a_k)\in G_1\tdt G_{d-1}\times G_{d+1}\tdt G_k$, the map from $\{x_d \in G_d \colon (a_1,\dots,a_{d-1},x_d, a_{d+1}, \dots, a_k) \in A\}$ to $H$, defined by the formula $x_d\mapsto\phi(a_1,\dots,a_{d-1},x_d,a_{d+1},\dots,a_k)$, is a restriction of an affine map from $G_d$ to $H$. We also use the term \textit{global map} if the domain $A$ equals the whole product $G_1 \tdt G_k$ inside which it is contained.}
\end{defin}

\chg{Hence, an affine map is the same as a global Freiman homomorphism.}\\

Our main theorem is an inverse theorem for multi-homomorphisms. It is trivial that any global multiaffine map $\phi \colon G_1\tdt G_k \to H$ is a multi-homomorphism. Moreover, the restriction of a global multiaffine map to any subset of $G_1 \tdt G_k$ is also a multi-homomorphism. The theorem gives a sort of converse: given a multi-homomorphism $\phi$ defined on a dense subset $A$ of $G_1\tdt G_k$, it must agree on a dense subset of $A$ with some global multiaffine map.

\begin{theorem}[Inverse theorem for multi-homomorphisms]\label{multiaffineInvThm}For every $k \in \mathbb{N}$ there is a constant $D^{\mathrm{mh}}_k$ such that the following statement holds. Let $G_1, \dots, G_k$ and $H$ be finite-dimensional vector spaces over $\mathbb{F}_p$. Let $A \subset G_1 \tdt G_k$ be a set of size at least $\delta|G_1|\cdots |G_k|$, and let $\phi \colon A \to H$ be a multi-homomorphism. Then there is a multiaffine map $\Phi \colon G_1\tdt G_k \to H$ such that $\phi(x_1, \dots, x_k) = \Phi(x_1, \dots, x_k)$ for at least $\Big(\exp^{(D^{\mathrm{mh}}_k)}(O_{k,p}(\delta^{-1}))\Big)^{-1} |G_1| \dots |G_k|$ elements $(x_1, \dots, x_k) \in A$, where $\exp^{(t)}$ denotes the $t$-fold iterated exponential.\end{theorem}

\noindent\textbf{Remark.} We may bound $D^{\text{mh}}_k$ by \chg{$C 2^k (k+1)!$} for some absolute constant $C$.\\

\vspace{\baselineskip}
\noindent\textbf{Uniformity norms.} In order to give further motivation for Theorem~\ref{multiaffineInvThm}, we need to recall the definition of the sequence of uniformity norms $\|\cdot\|_{U^k}$. These norms were introduced in~\cite{TimSze}, and played an essential role in obtaining a new proof of Szemer\'edi's theorem that gave quantitative bounds.

\begin{defin}[Uniformity norms]Let $G$ be a finite Abelian group and let $f \colon G \to \mathbb{C}$. The \emph{$U^k$ norm} of $f$ is given by the formula
\[\|f\|_{U^k}^{2^k} = \exx_{x, a_1, \dots, a_k \in G} \prod_{\varepsilon \in \{0,1\}^k}\operatorname{Conj}^{|\varepsilon|} f\Big(x - \sum_{i = 1}^k \varepsilon_i a_i\Big),\]
where $\operatorname{Conj}^{l}$ stands for the conjugation operator being applied $l$ times and $|\varepsilon|$ is shorthand for $\sum_{i = 1}^k \varepsilon_i$.\end{defin}

The relevance of these norms lies in the fact that whenever $f$ has small $U^k$ norm, it behaves like a randomly chosen function when it comes to counting objects of `complexity' $k-1$. We shall not define complexity here, but in the context of arithmetic progressions, where the complexity of an arithmetic progression of length $k$ is $k-2$, this statement can be formalized as follows.

\begin{proposition}Let $N$ be a sufficiently large prime, let $A \subset \mathbb{Z}_N$ be a set of size $\delta N$ and suppose that $\|\mathbbm{1}_A - \delta\|_{U^k} \leq \varepsilon$. Then the number $\on{AP}_{k+1}(A)$ of arithmetic progressions of length $k+1$ (and hence complexity $k-1$) inside $A$ satisfies $|N^{-2} \on{AP}_{k+1}(A) - \delta^{k+1}| = O_k(\varepsilon)$.\end{proposition}

\noindent Thus, in order to prove Szemer\'edi's theorem, one needs to understand the structure of functions with large uniformity norms. This was the strategy of the proof in~\cite{TimSze}, where a local inverse theorem for uniformity norms was obtained: given any $f \colon \mathbb{Z}_N \to \mathbb{D} = \{z \in \mathbb{C} \colon |z| \leq 1\}$ with $\|f\|_{U^k} \geq c$, there exist a polynomial $\psi : \mathbb{Z}_N \to \mathbb{Z}_N$ of degree at most $k-1$ and an arithmetic progression $P$ of length $N^{\Omega(1)}$ such that $\sum_{x \in P} f(x) \exp\Big(\frac{2 \pi i}{N} \psi(x)\Big) = \Omega_c(|P|)$.\\

This led to efforts to generalize the result to a strong inverse theorem, where one has a global correlation with a structured function such as a polynomial phase function. There are a couple of remarkable results along these lines. In~\cite{StrongUkZ}, Green, Tao and Ziegler proved such a result in the setting of $\mathbb{Z}_N$, while Bergelson, Tao and Ziegler had previously obtained an analogous result~\cite{BergelsonTaoZiegler} when the ambient group was $\mathbb{F}_p^n$ (with a further refinement by Tao and Ziegler~\cite{TaoZiegler}). In both cases, the family of structured functions is explicitly described, but it is more complicated than just the polynomial phases, so we shall not give the definitions here. However, in the so-called `high-characteristic case', when $k \leq p$, polynomial phases are again sufficient. Another approach which led to similar results in this direction is provided by the theory of nilspaces developed initially by Szegedy in~\cite{Szeg} and jointly by Camarena and Szegedy in~\cite{CamSzeg}. (For additional developments of nilspace theory and applications to other fields, see works of Candela~\cite{CandelaNotes1},~\cite{CandelaNotes2} detailing measure-theoretic aspects, works of Gutman, Manners and Varj\'u~\cite{GMV1}, \cite{GMV2}, \cite{GMV3} which have a particular focus on topological and dynamical applications, works of Candela and Szegedy~\cite{CandelaSzegedy1}, \cite{CandelaSzegedy2} generalizing the theory further and the work~\cite{nilspacesCharp} of Candela, Gonz\'alez-S\'anchez and Szegedy which specializes to low characteristic.)\\
\indent None of the results mentioned so far gave quantitative bounds on the correlation when $k>3$. Relatively recently,\footnote{The result appeared on arXiv in November 2018, and the first version of this paper was posted on arXiv in March 2020.} there was another major breakthrough when Manners~\cite{Manners} proved quantitative bounds\footnote{If we write $c_k(\delta)$ for the guaranteed correlation bound $|\ex_x f(x) g_{\text{str}}(x)| \geq c_k(\delta)$, where $g_{\text{str}}$ is the structured function, when $\|f\|_{U^k} \geq \delta$, and if $\mathcal{E}_0 \subset \mathcal{E}_1 \subset \dots$ is the Grzegorczyk hierarchy (i.e. functions in $\mathcal{E}_1$ are linear, functions in $\mathcal{E}_2$ are polynomial, functions in $\mathcal{E}_3$ use a bounded number of exponentials, etc.) then \chg{one possible definition of \emph{quantitative bounds}} means that all $n \mapsto c_k(n^{-1})^{-1}$ belong to some fixed $\mathcal{E}_i$. Before Manners's result, the proofs depended on regularity lemmas of increasing order, which led to $n \mapsto c_k(n^{-1})^{-1}$ being higher and higher in the Grzegorczyk hierarchy.} for the strong inverse theorem in the $\mathbb{Z}_N$ case.\\
\indent When it comes to quantitative bounds in the $\mathbb{F}_p^n$ case, the only result until recently was proved by Green and Tao for the ${U^3}$ norm~\cite{StrongU3}. Then with a much more involved argument, the authors obtained quantitative bounds for the $U^4$ norm~\cite{U4paper} in the case of large characteristic ($p \geq 5$). The key ingredient, from which the inverse theorem for $U^4$ norm follows easily, is Theorem~\ref{multiaffineInvThm} for the case of two variables. (This theorem is implicit in that paper.) Our main application is thus a generalization of the main result of~\cite{U4paper} and a quantitative version of the strong inverse theorem for the $U^k$ norm in $\mathbb{F}_p^n$, in the high-characteristic case $p\geq k$. This application is the main motivation for Theorem~\ref{multiaffineInvThm}. 

\begin{theorem}\label{inverseUniformityMain} Let $p\geq k$ and let $f \colon \mathbb{F}_p^n \to \mathbb{D}$ be a function such that $\|f\|_{U^k} \geq \delta > 0$ (where $\mathbb{D}$ is the unit disc in $\mathbb C$). Then there is a polynomial $g \colon \mathbb{F}_p^n \to \mathbb{F}_p$ of degree at most $k-1$ such that 
\[\Big|\exx_{x \in \mathbb{F}_p^n} f(x) \exp\Big(\frac{2 \pi i}{p} g(x)\Big)\Big| = \Omega_{k,p}\Big(\Big(\exp^{(O_k(1))} O_{k,p}(\delta^{-1})\Big)^{-1}\Big).\]
\end{theorem}

As in the case of $U^4$ norm in~\cite{U4paper}, this theorem follows reasonably straightforwardly from Theorem~\ref{multiaffineInvThm}. Given that the deduction is not hard and given that Theorem~\ref{multiaffineInvThm} does not require any characteristic assumption, it is plausible that a proof of the full quantitative inverse theorem for uniformity norms over finite fields (that is, without a large-characteristic assumption) is now within reach.\footnote{Since this work was posted online, Tidor~\cite{Tidor} was able to deduce the inverse theorem for $\|\cdot\|_{U^4}$ norm in $\mathbb{F}_2^n$ from the case $k = 2$ of Theorem~\ref{multiaffineInvThm} (which was proved in~\cite{U4paper}).}
 \medskip
 
\noindent\textbf{Other applications.} As well as Theorem~\ref{inverseUniformityMain}, some other results also follow from Theorem~\ref{multiaffineInvThm}. Among these, the closest in spirit to the inverse theorem for uniformity norms is an inverse theorem for approximate polynomials. For groups $G, H$ and an element $a \in G$, define the \emph{discrete derivative} $\Delta_a$ as the operator that maps a function $f \colon G \to H$ to the function $\Delta_a f$ defined by the formula $\Delta_a f(x) = f(x+a) - f(x)$. It is not hard to prove that when $G$ and $H$ are finite-dimensional vector spaces over $\mathbb{F}_p$ and $d < p$, a function $f\colon G\to H$ is a polynomial of degree at most $d$ if and only if the condition
\[\Delta_{a_1} \dots \Delta_{a_{d+1}} f(x) = 0\]
holds for all $a_1, \dots, a_{d+1}, x \in G$. By an \emph{approximate polynomial} we mean a function that satisfies this condition for a large collection of parameters but not necessarily for all. Our next result is that such functions are necessarily related to polynomials of the usual kind.\\

\begin{theorem}[Inverse theorem for approximate polynomials]\label{approxPolyMain} Let $p > d$, let $G$ and $H$ be finite-dimensional vector spaces over $\mathbb{F}_p$, and let $f \colon G \to H$ be a function such that
\[\Delta_{a_1} \dots \Delta_{a_{d+1}} f(x) = 0\]
for at least $\delta|G|^{d+2}$ choices of $a_1, \dots, a_{d+1}, x \in G$. Then there is a polynomial $\psi \colon G \to H$ of degree at most $d$ such that $f(x) = \psi(x)$ for at least $c|G|$ elements $x \in G$, where $c = \Omega_{d,p}\Big(\Big(\exp^{(O_d(1))} (O_{d,p}(\delta^{-1}))\Big)^{-1}\Big)$.\end{theorem}

\indent Before proceeding further, we pause for a moment to discuss the relationship between Theorems~\ref{multiaffineInvThm} and~\ref{approxPolyMain} (in a qualitative sense, and ignoring the assumption that $p > d$). While Theorem~\ref{multiaffineInvThm} implies Theorem~\ref{approxPolyMain} reasonably straightforwardly, the reverse implication is not entirely clear. To see why not, fix a function $\phi \colon G_1 \tdt G_k \to H$ that satisfies the assumptions of Theorem~\ref{multiaffineInvThm}. We may define a vector space $G^+ = G_1 \oplus G_2 \oplus \dots \oplus G_k$ and view $\phi$ as a function on $G^+$. It is not hard to see that $\phi$ becomes an approximate polynomial of degree at most $k$ on $G^+$ in the sense of Theorem~\ref{approxPolyMain}. Assuming that Theorem~\ref{approxPolyMain} has been proved, we can find a polynomial $\psi \colon G^+ \to H$ of degree at most $k$ that agrees with $\phi$ on a dense set. However, the structure of polynomials on $G^+$ is more general than that of multiaffine maps on $G_1 \tdt G_k$, so additional arguments are needed to complete the implication of theorems in this direction. For example, if we set $G_1 = G_2 = \mathbb{F}_p^n$ and $H = \mathbb{F}_p^m$, and set $\phi_i(x, y) = 2x_1y_1$ and $\psi_i(x,y) = x_1^2 + y_1^2$ for all $i \in [m]$, then $\phi = \psi$ on a dense set. In particular, the argument above would identify $\phi$ with a polynomial of degree 2, but not a multiaffine map, and it would still be correct (even if the conclusion would not be as strong as possible). 

At first sight, we expect that an inverse theorem should be significantly easier to prove when the object with approximate properties that we wish to understand also has a strong algebraic structure.\footnote{In our case, we want to show that a polynomial of degree $d$ on $G^+$ (instead of an arbitrary function) that is simultaneously a multi-2-homomorphism on a dense subset of $G_1 \tdt G_k$ (the approximate property) necessarily comes from a global multiaffine map.} However, the paper~\cite{TaoZiegler}, in which Tao and Ziegler prove the inverse theorem for uniformity norms in the case of \chg{low-characteristic, which requires structural results for algebraic objects called `non-classical polynomials' that do not exhibit quasirandom behaviour}, demonstrates that this intuition can be misleading.\\
\indent A final point is that Theorem~\ref{multiaffineInvThm} does not suffer from the low-characteristic issues that are present in Theorem~\ref{approxPolyMain}.\\

\noindent Another application that we have included in this paper is related to Bogolyubov's method. The classical version of this method can be stated as follows.

\begin{proposition}[Bogolyubov lemma]\label{bogIntro}Suppose that $A \subset G$ is a set of density $\delta$ and that $G$ is a finite-dimensional vector space over $\mathbb{F}_p$. Then $A + A - A - A$ contains a subspace of codimension $O_\delta(1)$.\end{proposition}

The proof of the above statement actually gives more information than this. It shows that given any function $f \colon G \to \mathbb{D}$, and any $\varepsilon > 0$, there is a subspace $V$ of codimension $O_\varepsilon(1)$ such that on each coset of $V$ the values of $f \ast f \ast \overline{f} \ast \overline{f}$ vary by at most $\varepsilon$. Thus, we may approximate the iterated convolution of any function in the $L^\infty$ norm by a highly structured function. In this paper, we also prove a multilinear generalization of this phenomenon. To state it, given a function $f \colon G_1 \tdt G_k\to\mathbb C$, we define its \emph{convolution in direction $d$} as  
\begin{align*}
\bigconv{d} f(x_1, \dots,x_{d-1}&,y_d,x_{d+1},\dots, x_k) \\
&= \exx_{z_d \in G_d} f(x_1, \dots, x_{d-1}, y_d + z_d, x_{d+1}, \dots, x_k) \overline{f(x_1, \dots, x_{d-1}, z_d, x_{d+1}, \dots, x_k)}.\\
\end{align*}

One way of describing what happens in the one-variable case is to say that we find a linear map $\phi\colon G\to\mathbb F_p^l$ for some small $l$ such that the iterated convolution is roughly constant on the inverse image of each $y\in\mathbb F_p^l$. Our next theorem is a very similar statement for iterated convolutions in different directions: the main difference is that $\phi$ is now a multilinear map, and a technical difference is that the approximation is valid for a set of inverse images that covers most of the domain rather than all of it.

It will be convenient to use the shorthand $G_{[k]}$ for $G_1\tdt G_k$ and $x_{[k]}$ for $(x_1,\dots,x_k)$.

\begin{theorem}[Approximating multiconvolutions in $L^\infty$ norm; multilinear Bogolyubov argument]\label{convApproxMain}Let $f \colon G_{[k]} \to \mathbb{D}$ be a function, let $d_1, \dots, d_r \in [k]$ be directions such that every direction $d\in[k]$ is included at least once, and let $\varepsilon > 0$. Then there exist
\begin{itemize}
\item a positive integer $l = \exp^{(O_k(1))}\Big(O_{k,p}(2^{O_{k,p}(r)}\varepsilon^{-O_{k,p}(1)})\Big)$, 
\item a multiaffine map $\phi \colon G_{[k]} \to \mathbb{F}^l_p$,
\item a subset $M \subset \mathbb{F}_p^l$, such that $|\phi^{-1}(M)| \geq (1-\varepsilon) |G_{[k]}|$,
\item a function $c \colon M \to \mathbb{D}$
\end{itemize}
such that for every $x_{[k]}\in\phi^{-1}(M)$ we have that
\[\Big|\bigconv{d_r} \dots \bigconv{d_1}\bigconv{k} \dots \bigconv{1} f(x_{[k]}) - c\Big(\phi(x_{[k]})\Big)\Big|\leq \varepsilon.\]\end{theorem}
To put this more loosely, if $f$ is a bounded function defined on $G_{[k]}$, then any sufficiently thorough multiconvolution of $f$ will be roughly constant on almost all the level sets of some low-dimensional multiaffine map, with the exceptional sets having a small total density. Or even more loosely: multiconvolutions have low multiaffine complexity.

Unlike with the previous applications, deducing Theorem~\ref{convApproxMain} from Theorem~\ref{multiaffineInvThm} requires significantly more work. It is also central to the proof of Theorem~\ref{multiaffineInvThm}, and in fact our argument takes the following recursive form: for each $k$ we use Theorem~\ref{multiaffineInvThm} for \chg{$k-1$} to prove \chg{a slightly weaker version of} Theorem~\ref{convApproxMain} for $k$, and then we use \chg{that variant of} Theorem~\ref{convApproxMain} for $k$ to prove Theorem~\ref{multiaffineInvThm} for \chg{$k$}.\footnote{The slightly weaker version of Theorem~\ref{convApproxMain} appears as Theorem~\ref{MixedConvApprox}. It turns out that additional control provided by $L^\infty$-approximation is not important in this proof; in fact, $L^2$ approximation is sufficient. The stronger version is useful in another application \chg{and appears as Theorem~\ref{AlmostLinftyApproxThm} much later in the paper}.} We shall return to a discussion of Theorem~\ref{convApproxMain} in the next section, where we give an outline of the proof.\\

Finally, we use Theorem~\ref{convApproxMain} in an easy and straightforward manner to deduce statements that are closer in spirit to Proposition~\ref{bogIntro}. Such results were established in the case of two variables (where one convolves a two-dimensional set several times in the principal directions to obtain biaffine structure) by Bienvenu and L\^{e}~\cite{BienHe} and by the authors~\cite{U4paper} independently, and the bounds in the problem were improved by Hosseini and Lovett~\cite{HosseiniLovett}. (See~\cite{BGSM3} for further discussion of that result and the structure obtained.) The two-variable version was also used by Bienvenu and L\^{e} in the study of correlations of the M\"obius function with quadratic polynomials over $\mathbb{F}_q[t]$ in~\cite{BienHe2}.\\
\indent One appealing further corollary is a structure theorem for subsets of $G_1 \tdt G_k$ that are subspaces in each principal direction.

\begin{theorem}\label{alldirsubspacestheorem}Let $G_1, \dots, G_k$ be $\mathbb{F}_p$-vector spaces. Suppose that $X \subset G_1 \tdt G_k$ is a set of density $\delta$ such that for each $d \in [k]$ and each $(x_1, \dots, x_{d-1}, x_{d+1}, \dots, x_k)\in G_1\tdt G_{d-1}\times G_{d+1}\tdt G_k$, the set $\{y_d \in G_d \colon (x_1, \dots, x_{d-1}, y_d, x_{d+1}, \dots, x_k) \in X\}$ is a (possibly empty) subspace. Then there exist $r \leq \exp^{(O_k(1))}(O_{k,p}(\delta^{-1}))$, sets $I_1, \dots, I_r \subset [k]$ and multilinear forms $\alpha_i \colon \prod_{j \in I_i} G_j \to \mathbb{F}_p$ for $i \in [r]$ such that 
\[\Big\{(x_1, \dots, x_k) \in G_1 \tdt G_k \colon (\forall i \in [r]) \, \alpha_i\Big((x_j \colon j \in I_i)\Big) = 0\Big\} \subset X.\]\end{theorem}
We also derive related results for general finite fields, when the choice of the field plays a non-trivial role.\\

For each of our main results and for each fixed $k$, the dependence on the other parameters is given by a bounded number of exponentials. However, this number increases with $k$, so we obtain a tower-type dependence on $k$. This is very unlikely to be best possible, but the inductive structure of our proof described in the discussion following Theorem~\ref{convApproxMain} and the iterative approach to extending Freiman multi-homomorphisms defined on $1-o(1)$ proportion of varieties (see Theorem~\ref{veryDenseToFullMultiaffineMap}) essentially force a dependence of this type. It seems that to obtain a substantially better bound, significant new ideas would be needed.\\

\noindent\textbf{Comparison with other works.} Before Manners's proof~\cite{Manners}, all proofs of the inverse theorem for general $U^k$ norms (as opposed to results for specific small values of $k$) relied on regularity lemmas or non-standard analysis. The  main novelty of our proof is that we avoid such arguments in the case of vector spaces over finite fields. Let us now say a few words about how this paper differs from~\cite{Manners}. Although similar in spirit, the approaches taken in~\cite{Manners} and in this paper are nevertheless disjoint in the sense that there is no obvious way of generalizing our proof to deal with $\mathbb Z_N$ or Manners's proof to deal with $\mathbb F_p^n$. The main obstacle that stands in the way of adapting our proof to $\mathbb{Z}_N$ is that we rely heavily on a quantitative inverse theorem for biased multilinear forms (i.e.,\ the partition versus analytic rank problem). Such a result is not known in the setting of $\mathbb{Z}_N$. Indeed, even the corresponding conjecture has not yet been articulated. On the other hand, Manners's proof depends heavily on the assumption that the ambient group has bounded rank and no small subgroups, which is the opposite situation to that of $\mathbb{F}_p^n$. Also, the main result of~\cite{Manners} is actually a variant of Theorem~\ref{approxPolyMain} for $\mathbb{Z}_N$ rather than of Theorem~\ref{multiaffineInvThm} -- that is, it concerns polynomials rather than multilinear functions.\\

\noindent\textbf{Acknowledgements.} LM is supported by the Serbian Ministry of Education, Science and Technological Development through Mathematical Institute of the Serbian Academy of Sciences and Arts. We thank Pierre-Yves Bienvenu, Pablo Candela and Olof Sisask for comments on an earlier version of the paper. We are grateful to the anonymous referee for numerous helpful comments that have greatly improved the exposition.\\

\section{Overview of proof and main ideas}

Recall that Theorem~\ref{multiaffineInvThm} is the inverse theorem for multi-homomorphisms. Let us state the accurate version of Theorem~\ref{convApproxMain} that we use in the the proof of Theorem~\ref{multiaffineInvThm}. It appears in this form in Section~\ref{sec4labelapprox} as Theorem~\ref{MixedConvApprox}.

\begin{theorem}\label{MixedConvApproxIntroVer} Let $G_1, \dots, G_k$ be finite-dimensional vector spaces over $\mathbb{F}_p$ and let $f\colon G_{[k]} \to \mathbb{D}$. Let $\varepsilon > 0$. Then there exist 
\begin{itemize}
\item a positive integer $r =\exp^{(O_k(1))}\Big(O(\varepsilon^{-O(1)})\Big)$, 
\item multiaffine map $\phi \colon G_{[k]} \to \mathbb{F}_p^r$, and
\item map $c : \mathbb{F}_p^r \to \mathbb{D}$
\end{itemize}
such that
\[\bigconv{k}\dots\bigconv{1} f \apps{\varepsilon}_{L^2} c \circ \phi.\]
\end{theorem}

As expected, we may add additional convolutional directions and the same kind of conclusion will still hold. The proof our main result splits into the following stages.

\begin{itemize}
\item[\textbf{Step 1.}] Assuming Theorem~\ref{multiaffineInvThm} for $k - 1$, we prove Theorem~\ref{MixedConvApproxIntroVer} for $k$. The starting point of the proof is the use of the standard fact that convolutions of a function $g$ in a single variable can be well-approximated in $L^2$ norm using exponential sums involving the set $S$ of large Fourier coefficients of $g$, i.e. $g \conv g(x) \approx \sum_{r \in S} |\hat{g}(r)|^2 \omega^{-r\cdot x}$. Using this fact, to control the multivariate convolution $\bigconv{k}\dots\bigconv{1} f$, we thus need to understand the large Fourier coefficients of functions 
\[x_k \mapsto \bigconv{k-1}\dots\bigconv{1} f(x_{[k]})\]
for many $x_{[k-1]} \in G_{[k-1]}$, since these will give us information about the convolution in direction $k$. We denote the function above by $(\bigconv{k-1} \dots \bigconv{1} f)_{x_{[k-1]}}$: it can be thought of as the restriction of $\bigconv{k-1} \dots \bigconv{1} f$ to the 1-dimensional cross-section of $G_{[k]}$ given by fixing the first $k-1$ coordinates to $x_{[k-1]}$. Writing $S_{x_{[k-1]}}$ for the set of Fourier coefficients of appropriate scale for the function above, we show that these sets depend on elements $x_1, \dots, x_{k-1}$ in a multilinear way. Namely, using Theorem~\ref{multiaffineInvThm} for $k - 1$, we may find a reasonably small collection of global multiaffine maps $\mu_1, \dots, \mu_m : G_{[k - 1]} \to G_k$ such that the set $S_{x_{[k-1]}}$ is contained inside $\{\mu_1(x_{[k-1]}, \dots, \mu_r(x_{[k-1]}))\}$ for the vast majority of $x_{[k-1]}$. Hence, by the convolution law and the Fourier inversion formula, we essentially obtain an approximation of the form
\begin{align*}&\bigconv{k} \dots \bigconv{1} f(x_{[k]}) \approx \sum_{i \in [m]} \Big|\fco(\bigconv{k-1} \dots \bigconv{1} f)_{x_{[k-1]}}\fcc(\mu_i(x_{[k-1]}))\Big|^2 \omega^{\mu_i(x_{[k-1] }) \cdot x_{k}}.\end{align*}
We are not done yet as we need to approximate the Fourier coefficients $\fco(\bigconv{k-1} \dots \bigconv{1} f)_{x_{[k-1]}}\fcc(\mu_i(x_{[k-1]}))$ using phases of multiaffine forms as well. After expanding, we see that an approximation for $\bigconv{k-1} \dots \bigconv{1} f$ will be needed, leading to an inductive argument. Additionally, to get the desired approximation, since $\fco(\bigconv{k-1} \dots \bigconv{1} f)_{x_{[k-1]}}\fcc(\mu_i(x_{[k-1]}))$ involves a linear combination of $|G|$ terms, we need an additional random sampling argument to reduce the number terms in the linear combination to a bounded quantity.\\
\phantom{a}\hspace{6pt} The new ingredient that was not present in the two-dimensional case is the use of a solution to the partition rank versus analytic rank problem, which is required in the more general case and about which we shall say more later. Section~\ref{sec4labelapprox} is devoted to this step.

\item[\textbf{Step 2.}] We define certain sets of points, which we call arrangements, that correspond to taking convolutions in certain sequences of directions. For example, in two dimensions, convolving first in the vertical direction and then in the horizontal direction gives rise to vertical parallelograms, which are configurations of the form $(x,y_1), (x,y_1+h), (x+w,y_2), (x+w,y_2+h)$. With each arrangement we associate a sequence of lengths, which is an element of $G_1 \tdt G_k$ that can be obtained from points in the arrangement by using convolution operations. For instance, the lengths associated with a vertical parallelogram are its width and height -- the elements $w\in G_1$ and $h\in G_2$ above. Since all points in the arrangements we consider belong to the domain of $\phi$, we may define the $\phi$ value of an arrangement as an appropriate linear combination (in fact, a $\pm 1$ combination) of the values of $\phi$ at its points. Again, in the case of a vertical parallelogram $(x,y_1), (x,y_1+h), (x+w,y_2), (x+w,y_2+h)$, its $\phi$ value is
\[\phi(x,y_1) - (x,y_1 + h) - (x + w,y_2) + (x + w,y_2 + h).\]
In our case, we need to consider arrangements arising from a total of $2k$ directional convolutions, with each direction appearing twice. In fact, the final sets of points that we shall look at will be $k$-tuples of such arrangements. The reason for having such an involved set of points will be explained in the discussion of the fourth step, and is closely related to the use of Theorem~\ref{MixedConvApproxIntroVer}.\\    
\phantom{a}\hspace{6pt} In this step, carried out in Section~\ref{sec5labelexresparr}, we show that a positive proportion of pairs of arrangements with the same lengths have the same $\phi$ value. Eventually, this will allow us to define a new map $\psi$ related to $\phi$ at a point $a_{[k]}$ as a common $\phi$ value $\psi(a_{[k]})$ of significant proportion of arrangements of lengths $a_{[k]}$. As we shall see, the map $\psi$ will have improved properties compared to $\phi$. As the conclusion of the work in this and the next two steps, we will be able to guarantee that $\psi$ is a Freiman multi-homomorphism on a set with significant algebraic structure. Namely, the domain of $\psi$ will be a very dense subset of a variety\footnote{In this paper, by \emph{variety} in $G_1 \tdt G_k$ we mean the zero set of a multiaffine map $\alpha \colon G_1 \tdt G_k \to H$. We also say that the \emph{codimension} of $\alpha$ is the dimension $\dim H$ and we say that a variety $V$ is of \emph{codimension at most $r$} if $V$ can be defined by a multiaffine map of codimension $r$. Equivalently, we may define the \emph{codimension} of a variety $V$ to be the minimum codimension of a multiaffine map that defines it.} of reasonably small codimension.

\item[\textbf{Step 3.}] In Section~\ref{sec6densification}, we use an algebraic form of the dependent random choice method to find a subset of points in the domain of $\phi$ such that the proportion of pairs of arrangements of same lengths that have the same $\phi$ value is not merely positive, but close to $1$. 

\item[\textbf{Step 4.}] Combining the work in the previous step with Theorem~\ref{MixedConvApproxIntroVer}, we find a new map $\psi \colon B \to H$, where $B \subset V$ for some variety $V$, $|B| \geq (1-o(1))|V|$, and $\psi$ is a Freiman multi-homomorphism. Due to additional structure in the domain, $\psi$ can be made to have the stronger property of being a restriction of an affine map in each principal direction. We call such maps multiaffine, and use the term global multiaffine map for maps whose domain is the whole of $G_1 \tdt G_k$. Moreover, $\psi$ is related to the initial map $\phi$ in a sufficiently algebraically strong sense that once we show that $\psi$ coincides with a global multiaffine map on a dense set, we may deduce the same for $\phi$. Section~\ref{sec7labnearlymultiaffinepiece} contains this step. Let us sketch the main details in this part of the proof.\\
\phantom{a}\hspace{6pt} As the previous two steps show, we may pass to a subset $A' \subseteq A$ on which almost all pairs $(q_1, \dots, q_k)$ and $(q'_1, \dots, q'_k)$ of $k$-tuples of arrangements of the same tuple of lengths, where $q_i$ and $q'_i$ are arrangements of the shape corresponding to convolutions in directions $i+1, i + 2, \dots, k, 1, \dots, k, 1, 2, \dots, i$, have the same $\phi$ value. If the lengths tuple is $a_{[k]}$, this allows us to define $\psi(a_{[k]})$ as the most frequent $\phi$ value. Theorem~\ref{MixedConvApproxIntroVer} allows us to have a precise control over the counts of all arrangements of interest. In particular, the counts are essentially constant along multiaffine varieties of small codimension.\\
\phantom{a}\hspace{6pt} It remains to show that $\psi$ is a multi-homomorphism. Just as we defined $\psi(a_{[k]})$ to be the most frequent value of arrangements of lengths $a_{[k]}$ arising from convolutions in directions $i+1, i + 2, \dots, k, 1, \dots, k, 1, 2, \dots, i$, we can define an auxiliary map $\psi'(a_{[k]})$ as the most frequent value of arrangements of lengths $a_{[k]}$ for the first $2k-1$ directional convolutions only, ignoring the convolution in the last direction, which is $i$. Since the large arrangements of points (for the whole sequence of $2k$ directions) arise as concatenations of the smaller arrangements (with the last convolution step not carried out), we expect the equality $\psi(a_{[k]}) = \psi'(a_{[k]\setminus \{i\}}, b_i + a_i) - \psi'(a_{[k]\setminus \{i\}}, b_i)$ to hold frequently. Such an argument will allow us to conclude that $\psi$ is a Freiman homomorphism in direction $i$. For this reason we consider tuples of arrangements of different shapes, as the last direction of convolution in each arrangement will be used for proving that $\psi$ is a Freiman homomorphism that direction.

\item[\textbf{Step 5.}] The next step, found in Section~\ref{sec8labelbiaffinemapstheory}, is a slight digression from the main flow of the argument, in which we study extensions of biaffine maps defined on quasirandom biaffine varieties. This is similar to arguments from~\cite{U4paper}, but the arguments presented here are more streamlined and there are some new ones as well. In general, this step can be thought of as a bilinear algebraic regularity method. Indeed, the key fact is an algebraic regularity lemma,\footnote{The algebraic regularity lemma appears as Corollary 5.2 in~\cite{U4paper}, but we do not require it in that form in this paper. Rather, we use a `simultaneous biaffine regularity lemma' which the next step is about.} which tells us that a biaffine variety $V$ inside a product of vector spaces $G_1 \times G_2$ can be partitioned into quasirandom pieces of the form $(t_1 + U_1) \times (t_2 + U_2)$ for subspaces $U_1 \leq G_1, U_2 \leq G_2$ of small codimension. The quasirandomness can then be exploited to find appropriate structures in the given variety. Ultimately, it allows us to extend biaffine maps on biaffine varieties to global biaffine maps and it provides good bounds.\\
\phantom{a}\hspace{6pt} When it comes to new results in the bilinear setting, of particular importance is Theorem~\ref{biaffineExtnConv}, which we now summarize. Consider vector spaces $G_1, G_2$ and $H$, and a quasirandom biaffine variety $V \subseteq G_1 \times G_2$. Let $X \subseteq G_1$ be an arbitrary set of reasonable density and let $S \subseteq (X \times G_2) \cap V$ be a set whose columns $\{y \in G_2 : (x,y) \in S\}$ are $1-o(1)$ dense inside columns of the variety $V$ for all $x \in X$. Let $\varphi : S \to H$ be a biaffine map. In order to extend the columns of $S$ to the whole of $V$, it is natural to consider a \textit{convolutional extension}, namely $\varphi^{\text{conv}}(x,y) = \varphi(x,z_1) + \varphi(x,z_2) - \varphi(x, z_1 + z_2 - y)$. It is easy to see that the convolutional extension will remain affine in the second coordinate, which is an elementary fact from the theory of a single variable. The important fact is that $\varphi^{\text{conv}}$ remains affine in the first coordinate as well, once a small error set has been removed from $X$, which follows from the fact that the variety $V$ is quasirandom. The fact that we do not require that $X$ is $1-o(1)$ dense inside $G_1$ is crucial, and such a statement is different from results in the previous work~\cite{U4paper}, where all results required very high density.

\item[\textbf{Step 6.}] Let us first set up some terminology. By a \emph{line} we mean a set of the form 
\[\{(x_1, \dots, x_{d-1}, y_d, x_{d+1}, \dots, x_k) : y_d \in G_d\}.\]  We frequently use the phrase a line \textit{in principal direction} to emphasize that we consider special kind of lines. Similarly, by a \emph{plane} we mean a set of the form 
\[\{(x_1, \dots, x_{d-1}, y_d, x_{d+1}, \dots, x_{d' -1}, y_{d'}, x_{d' + 1}, \dots, x_k) : y_d \in G_d, y_{d'} \in G_{d'}\}.\] 
In Section~\ref{sec9labelsimbiaffinereglemma}, we prove a `simultaneous biaffine regularity lemma', Theorem~\ref{simReg}, which tells us that we may partition any given variety using varieties of lower complexity (defined by multiaffine maps in fewer variables) so that on most planes in the principal directions we get quasirandom biaffine varieties. This fits well with results on extensions of biaffine maps and allows us to extend them to the multiaffine setting. The proof relies heavily on the results concerning the partition rank versus analytic rank problem and their consequences regarding the behaviour of multiaffine forms and varieties.\\
\phantom{a}\hspace{6pt} Once the simultaneous biaffine regularity lemma has been proved, we use it to deduce a couple of further important results, Theorem~\ref{mainExtensionConvStep} and Theorem~\ref{globalBiaffineInAllPlanes}. The former is a multilinear generalization of the convolutional extension fact described in the previous step and the latter is a seemingly obvious, but ultimately significant, result that Freiman multihomomorphisms defined on $1- o(1)$ dense subsets of varieties can be assumed to be restrictions of global biaffine maps on planes in principal directions. The latter theorem will allow us easier control over various cocycle identities required in the last step.

\item[\textbf{Step 7.}] Finally, in Section~\ref{sec10labelextendingmultiaffine}, we show that a multiaffine map defined on almost all of a variety coincides on a large set with a global multiaffine map. This is the culmination of the work in the previous two steps and is done via a somewhat subtle iterative procedure, which we now describe. Let $\phi : A \to H$ be a multiaffine map on a $1-\varepsilon$ dense subset $A$ of multiaffine variety $V$. Assume that $V$ is defined by $r$ forms that depend on \textit{all} $k$ variables, and $s$ forms that depend on fewer variables, which we think of as being of lower order. We first show, in Theorem~\ref{extn1codim}, that we may extend the map $\phi$ to a variety of codimension one smaller, effectively removing one of the $r$ forms depending on all variables defining $V$, at the cost of introducing additional lower order forms, thus increasing $s$ to $(s +  \log \varepsilon^{-1}){O(1)}$ and decreasing the density of $A$ from $1-\varepsilon$ to $1 - \varepsilon^{\Omega(1)}$. The issue is that the density drops with the removal of a single form. Fortunately, in Theorem~\ref{bogolyubovSingleStep}, we show that, up to the cost of increasing $s$ to $(s +  \log \xi^{-1}){O(1)}$, the density of the domain $A$ inside the variety $V$ can be increased to $1-\xi$, for any $\xi$. Thus, in Proposition~\ref{StabBogFnMultThm}, we first densify the domain sufficiently using Theorem~\ref{bogolyubovSingleStep} and then remove all $r$ forms that depend on all $k$ variables one by one using Theorem~\ref{extn1codim}. We repeat the procedure for the next maximal set of arguments, say, $x_1, \dots, x_{k-1}$, until, after $2^{k} - 1$ repetitions for all possible maximal sets, in Theorem~\ref{veryDenseToFullMultiaffineMap}, we obtain a global multiaffine map related to $\phi$.
\end{itemize}

The main result, Theorem~\ref{multiaffineInvThm} is quickly deduced in Section~\ref{sec11labelPET}.\\

\noindent\textbf{Organization of the paper.} The rest of the paper mainly follows the steps above. We begin with a preliminary Section~\ref{sec3prelimsec} that contains a number of auxiliary results that will be used frequently later in the paper. Then, Sections~\ref{sec4labelapprox}--\ref{sec10labelextendingmultiaffine} follow the above seven steps, and all ingredients are put together in Section~\ref{sec11labelPET}. Finally, Section~\ref{sec12labelapps} is devoted to applications.\\

\subsection{Key ideas and comparison with our earlier work~\cite{U4paper}}

\chg{In this subsection, we describe the most significant differences from the previous work~\cite{U4paper} which focused on the the bihomomorphisms and the ${U}^4$ norm. The following discussion does not focus so much on supplementing the sketch provided earlier as on explaining why some parts of the argument take a possibly unexpected form. It may therefore be skipped on a first reading.}\\
\indent On a conceptual level, the main contribution of this paper is that, in the arithmetic setting, one may use (weak) algebraic regularity lemmas instead of the usual (strong) hypergraph regularity lemmas~\cite{TimReg, HyperReg1, HyperReg2, TaoReg} to obtain reasonable bounds in combinatorial problems in situations where other proofs tend to give bounds of Grzegorczky type. In our case, as anticipated in the introduction to~\cite{LukaRank}, we employ Theorem~\ref{strongInvARankThm} (a polynomial bound for the partition rank versus analytic rank problem) as a substitute for hypergraph regularity lemmas. This is one of the major differences from our previous work~\cite{U4paper}, where we did not require a result such as Theorem~\ref{strongInvARankThm} since we considered only the case of bi-homomorphisms. In that setting, Theorem~\ref{strongInvARankThm} becomes an easy linear-algebraic fact that for an $n \times n$ matrix $A$ with entries in the field $\mathbb{F}_p$ we have
\[p^{-\on{rank} A} = \exx_{x, y \in \mathbb{F}_p^n} \omega^{x^T A y},\]
where $\omega = \exp(2 \pi i / p)$. In fact, in the bi-homomorphism case our algebraic regularity lemma (see Lemma 5.1 and Corollary 5.2 in~\cite{U4paper}) can be seen as a direct analogue of Szemer\'edi's regularity lemma applied to the graph defined by a bilinear variety. Namely, for a bilinear map $\beta \colon \mathbb{F}_p^n \times \mathbb{F}_p^n \to \mathbb{F}_p^r$ and a vector $\lambda \in \mathbb{F}_p^r$ we consider the bipartite graph $\Gamma$ with the vertex classes given by two copies of the vector space $\mathbb{F}_p^n$ and an edge for each pair $xy$ such that $\beta(x,y) = \lambda$. The reason for the improved bounds in that case is the algebraic structure of the graph $\Gamma$. The regularity lemma in~\cite{U4paper} becomes the fact that for any positive integer $t$ we may find a subspace $U \leq \mathbb{F}_p^n$ of codimension at most $rt$ such that $\on{rank} \beta|_{U \times U} \geq t$. However, in the higher order case, applying a hypergraph regularity lemma to an appropriate algebraically defined hypergraph leads to Grzegorczky-type bounds that grow just slightly slower than in the general hypergraph case. Let us illustrate this by an example.\\
\indent Let $G$ be a finite-dimensional vector space over $\mathbb{F}_p$ and suppose that $\gamma \colon G \times G \times G \to \mathbb{F}^r_p$ is a trilinear map. As with hypergraph regularity lemmas, we look for a highly quasirandom partition of $G \times G \times G$ that refines the partition $\bigcup_{\lambda \in \mathbb{F}_p^r} \{(x,y,z) \in G^3 \colon \gamma(x,y,z) = \lambda\}$. We shall end up with further trilinear maps $\gamma' \colon G \times G \times G \to \mathbb{F}^s_p$, bilinear maps $\beta^{1\,2}\colon G \times G \to \mathbb{F}_p^{t_{1\,2}}$, $\beta^{1\,3}\colon G \times G \to \mathbb{F}_p^{t_{1\,3}}$, $\beta^{2\,3}\colon G \times G \to \mathbb{F}_p^{t_{2\,3}}$ and linear maps $\alpha^1 \colon G \to \mathbb{F}_p^{v_1}, \alpha^2 \colon G \to \mathbb{F}_p^{v_2}, \alpha^3 \colon G \to \mathbb{F}_p^{v_3}$ such that the partition into varieties of the form
\begin{align*}\{(x,y,z) \in G^3\colon \gamma'(x,y,z) = \mu, \beta^{1\,2}(x,y) = \nu^{1\,2}, \beta^{1\,3}(x,z) = \nu^{1\,3},& \beta^{2\,3}(y,z) = \nu^{2\,3}, \\
&\alpha^1(x) = \tau^1, \alpha^2(y) = \tau^2, \alpha^3(z) = \tau^3\}\end{align*}
is a highly quasirandom refinement of the partition $\bigcup_{\lambda \in \mathbb{F}_p^r} \{(x,y,z) \in G^3 \colon \gamma(x,y,z) = \lambda\}$. Let $U$ be the subspace defined by $\{u \in G \colon \alpha^1(u) = 0, \alpha^2(u) = 0, \alpha^3(u) = 0\}$. The quasirandomness requirement translates into a condition on the bias\footnote{The \emph{bias} of a multilinear map $\alpha(x_1, \dots, x_k)$ is defined as $\on{bias} \alpha= \ex_{x_1, \dots, x_k \in G} \omega^{\alpha(x_1, \dots, x_k)}$ and it measures how much the distribution of values of $\alpha$ deviates from the uniform distribution.} of multilinear maps. That is, we need 
\[\on{bias} \Big(\lambda_1 \gamma'_1|_{U \times U \times U} + \dots + \lambda_s \gamma'_s|_{U \times U \times U}\Big)\leq p^{-C(s + t_{1\,2} + t_{1\,3} + t_{2\,3})}\] 
for all non-zero choices of scalars $\lambda_1, \dots, \lambda_s$, 
\[\on{bias} \Big(\lambda_1 \beta^{1\,2}_1|_{U \times U \times U} + \dots + \lambda_{t_{1\,2}} \beta^{1\,2}_{t^{1\,2}}|_{U \times U \times U}\Big) \leq p^{-C t_{1\,2}},\] 
etc., with similar bias bounds for $\beta^{1\,3}$ and $\beta^{2\,3}$ (for some sufficiently large $C > 0$). To prove the existence of such a system of multilinear maps, as in the proofs of hypergraph regularity lemmas, we apply an iterative procedure where at each step we take the lowest-order linear combination $\lambda$ of forms that violate the quasirandomness condition and use Theorem~\ref{strongInvARankThm} to replace $\lambda$ by further forms of even lower order. For example, if we start with just the trilinear map $\gamma \colon G \times G \times G \to \mathbb{F}^r_p$ which fails to be quasirandom, at the first step of the procedure we remove one $\gamma_i$ from $\gamma_1, \dots, \gamma_r$ but we also add at least $r$ new forms involving at most two variables. If we get a linear combination of $\beta^{1\,2}_1, \dots, \beta^{1\,2}_{t_{1\,2}}$ that is not sufficiently quasirandom, we remove one of those forms, but we add at least $t_{1\,2}$ new linear forms to the lists $\alpha^1_1, \dots, \alpha^1_{v_1}$ and $\alpha^2_1, \dots, \alpha^2_{v_2}$. Such a proof shows that even in the arithmetic setting (\chg{where the ambient space comes from abelian groups and addition plays a role in the definition of the hypergraphs of interest}) one gets exponential bounds for a strong algebraic regularity lemma for trilinear forms (i.e.\ if we start with $\gamma \colon G \times G \times G \to \mathbb{F}^r_p$ we end up with $\exp(\Omega(r))$ lower order forms defining a quasirandom decomposition). Going one step further, the strong regularity lemma for multilinear forms in four variables already has worse bounds than Theorem~\ref{multiaffineInvThm}. To overcome this obstacle, we use Theorem~\ref{strongInvARankThm} directly for regularization purposes.\\

\indent Having discussed the differences between the algebraic regularity lemmas in this work and in~\cite{U4paper} let us now say something more concrete about how these two papers differ. Recall that the main result in~\cite{U4paper} (which is implicit in that paper) is the inverse theorem for Freiman bihomomorphisms (the case $k = 2$ of Theorem~\ref{multiaffineInvThm}). The structure of that proof roughly consists of less general versions of the first 5 steps in the outline above. To prove the general case of Theorem~\ref{multiaffineInvThm} we need a new approach to the problem of extending a multiaffine map defined on a $1-o(1)$ proportion of a variety of bounded codimension. Since our algebraic regularity lemma was at least as strong as Szemer\'edi's regularity lemma, in~\cite{U4paper} we could complete this extension quite directly. (See section~\ref{biaffinemapssection} of this paper for details on extensions of bi-homomorphisms.) In particular, we could first extend the map to a full biaffine subvariety (rather than having $1-o(1)$ proportion of a variety as the domain), and then extend the map to the full space. However, as the strong regularity lemmas are no longer available to us, we need to find a different extension argument this time. It turns out that the weaker regularity lemma does not allow us to `fill' the domain in the general case and we are forced to work with maps defined on a $1-o(1)$ proportion of a variety during the whole extension process (the reason for this is that the weak regularity lemma can produce at most a $1-o(1)$ proportion of quasirandom atoms). This difficulty forces us to use a significantly more involved extension argument (see the description of \textbf{Step 7}) which requires the use of the bilinear extension results (\textbf{Step 5}), which are further strengthened, and the development of tools that allow us to use them in the multilinear setting (\textbf{Step 6}).\\

\indent As a final important difference between~\cite{U4paper} and this paper, we note the inductive structure of this proof, in which we prove Theorem~\ref{convApproxMain} for $k$ using Theorem~\ref{multiaffineInvThm} for $k$ and then use Theorem~\ref{convApproxMain} for $k$ (and other ingredients) to prove Theorem~\ref{multiaffineInvThm} for $k+1$. In~\cite{U4paper} we only needed Theorem~\ref{convApproxMain} for $k = 2$ (the analytic version of the so-called bilinear Bogolyubov lemma), the proof of which relied heavily on Freiman's theorem. Thus to generalize, it was important to realize that Freiman's theorem can be replaced by Theorem~\ref{convApproxMain} in that proof.\\

\noindent \textbf{Uses of the inverse theorem for biased multilinear forms.} A recurring theme in this paper is the use of the inverse theorem for biased multilinear forms -- Theorem~\ref{strongInvARankThm}. The uses of the theorem can be roughly split into three categories:
\begin{itemize}
\item[\textbf{(i)}] applications of its corollaries (see Corollary~\ref{simFibresThm}) to varieties with the goal of finding regular pieces of varieties,
\item[\textbf{(ii)}] application of the simultaneous biaffine regularity lemma (Theorem~\ref{simReg}), which allows us to partition any given variety in a structured way so that almost every piece intersected with almost every plane in principal directions becomes quasirandom, and
\item[\textbf{(iii)}] direct applications of the theorem itself to relevant multilinear forms (see the proofs of Theorem~\ref{multiaffineInvThmGen} and Theorem~\ref{ApproxPolyThmGen}).
\end{itemize}

\noindent\textbf{Remark on the length of the paper.} Given the unusual length of the paper, we feel that we should explain why this is the case. First of all, as is suggested by the discussion above, avoiding the use of the strong regularity lemma generates further difficulties and using its weaker replacement (Theorem~\ref{strongInvARankThm}) results in longer arguments and additional steps (\textbf{Step 6}, \textbf{Step 7} and parts of \textbf{Step 5}). If we were allowed to use strong regularity lemma, we would have ended up with a single step, not significantly more involved than \textbf{Step 5}. Secondly, as is usual in works concerning uniformity norms, the proof structure itself is involved and almost all steps in the outline of the proof above result in technical difficulties. A less fundamental point, but it still makes a difference, is that dealing with functions of $k$ variables for general $k$ necessarily complicates the notation and increases the space taken up by our calculations.

\section{Preliminaries}\label{sec3prelimsec}
\phantom{a}\\
\indent \chg{This section sets up the notation and contains mostly auxiliary results, either because they will be used frequently or because they would be a distraction from the main line of the argument when they are applied. With this in mind, we suggest to the reader that, on a first reading, the focus should be on the notation, which now follows, and on Subsection~\ref{appmvsubseclabel}, which contains most of the technology for manipulating multiaffine maps, which is of key significance for the paper. The other results are more technical.}\\

\noindent\textbf{Notation.} As above, we write $\mathbb{D} = \{z \in \mathbb{C}\colon |z| \leq 1\}$ for the unit disk. We use the standard expectation notation $\ex_{x \in X}$ as shorthand for the average $\frac{1}{|X|} \sum_{x \in X}$, and when the set $X$ is clear from the context we simply write $\ex_x$. As in~\cite{LukaRank}, we use the following convention to save writing in situations where we have many indices appearing in predictable patterns. Instead of denoting a sequence of length $m$ by $(x_1, \dots, x_m)$, we write $x_{[m]}$, and for $I\subset[m]$ we write $x_I$ for the subsequence with indices in $I$. This applies to products as well: $G_{[k]}$ stands for $\prod_{i \in [k]} G_i$ and $G_I = \prod_{i \in I} G_i$. For example, instead of writing $\alpha \colon \prod_{i \in I} G_i \to \mathbb{F}$ and $\alpha(x_i \colon i \in I)$, we write $\alpha \colon G_I \to \mathbb{F}$ and $\alpha(x_I)$. This notation is particularly useful when $I=[k]\setminus\{d\}$ as it saves us writing expressions such as $(x_1,\dots,x_{d-1},x_{d+1},\dots,x_k)$ and $G_1\tdt G_{d-1}\times G_{d+1}\tdt G_k$.\\
\indent We extend the use of the dot product notation to any situation where we have two sequences $x=x_{[n]}$ and $y=y_{[n]}$ and a meaningful multiplication between elements $x_i y_i$, writing $x\cdot y$ as shorthand for the sum $\sum_{i=1}^n x_i y_i$. For example, if $\lambda=\lambda_{[n]}$ is a sequence of scalars, and $A=A_{[n]}$ is a suitable sequence of maps, then $\lambda\cdot A$ is the map $\sum_{i=1}^n\lambda_iA_i$.\\  
\indent Frequently we shall consider `slices' of sets $S\subset G_{[k]}$, by which we mean sets $S_{x_I} = \{y_{[k]\setminus I} \in G_{[k] \setminus I} \colon (x_I, y_{[k] \setminus I}) \in S\}$, for $I \subset [k], x_I \in G_I$. (Here we are writing $(x_I,y_{[k]\setminus I})$ not for the concatenation of the sequences $x_I$ and $y_{[k]\setminus I}$ but for the `merged' sequence $z_{[n]}$ with $z_i=x_i$ when $i\in I$ and $z_i=y_i$ otherwise.) If $I$ is a singleton $\{i\}$ and $z_i\in G_i$, then we shall write $S_{z_i}$ instead of $S_{z_{\{i\}}}$. Sometimes, the index $i$ will be clear from the context and it will be convenient to omit it. For example, $f(x_{[k]\setminus\{i\}},a)$ stands for $f(x_1,\dots,x_{i-1},a,x_{i+1},\dots,x_k)$. If the index is not clear, we emphasize it by writing it as a superscript to the left of the corresponding variable, e.g.\ $f(x_{[k]\setminus\{i\}},{}^i\,a)$.\\ 

\indent More generally, when $X_1, \dots, X_k$ are finite sets, $Z$ is an arbitrary set, $f \colon X_1 \tdt X_k = X_{[k]} \to Z$ is a function, $I \subsetneq [k]$ and $x_i \in X_i$ for each $i \in I$, we define a function $f_{x_I} \colon X_{[k] \setminus I} \to Z$, by mapping each $y_{[k] \setminus I} \in X_{[k] \setminus I}$ as $f_{x_I}(y_{[k] \setminus I}) = f(x_I, y_{[k] \setminus I})$. When the number of variables is small -- for example, when we have a function $f(x,y)$ that depends only on two variables $x$ and $y$ instead of on indexed variables -- we also write $f_x$ for the map $f_x(y)=f(x,y)$.\\

Let $G, G_1, \dots, G_k$ be finite-dimensional vector spaces over a finite field $\mathbb{F}$, and let $\chi \colon \mathbb{F} \to \mathbb{D}$ be a non-trivial additive character on $\mathbb{F}$. For maps $f,g \colon G \to \mathbb{C}$, we write $f\conv g$ for the function defined by $f \conv g(x) = \ex_{y \in G} f(x + y) \overline{g(y)}$. Given a map $f \colon G_1 \tdt G_k \to \mathbb{C}$, we can rewrite the definition of the convolution in direction $d$ as 
\[\bigconv{d} f(x_{[k]\setminus\{d\}},y_d) = f_{x_{[k] \setminus \{d\}}} \conv f_{x_{[k] \setminus \{d\}}} (y_d) = \mathbb{E}_{x_d \in G_d} f(x_{[k] \setminus \{d\}}, x_d + y_d) \overline{f(x_{[k] \setminus \{d\}}, x_d)}.\]
\chg{\indent We shall make frequent use of Fourier analysis. Following Section 4 of~\cite{TaoVu}, we define a \textit{dot product} on $G$ as a non-degenerate symmetric bilinear form $\cdot : G^2\to \mathbb{F}_p$. Fix a dot product $\cdot$ on $G$. The Fourier transform of $f \colon G \to \mathbb{C}$ is the function $\hat{f} \colon G \to \mathbb{C}$ defined by $\hat{f}(r) = \ex_{x \in G} f(x) \chi(-r\cdot x)$. Although the dot product $\cdot$ is not unique for a given group, the results in the paper are not affected by the particular choice of $\cdot$. (Note that the choice of additive character similarly does not affect the results.)\\
\indent Another way we use dot products is to pass from multilinear forms to multilinear maps with one argument removed. Namely, the dot product induces an isomorphism between $G$ and $\on{Hom}(G,\mathbb{F})$ by sending $a \in G$ to the form $b \mapsto a \cdot b$. Thus, if $\alpha : G_1 \tdt G_k \to \mathbb{F}$ is a multilinear form, for each choice of $x_1, \dots, x_{k-1}$, the map $y_k \mapsto \alpha(x_1, \dots, x_{k-1}, y_k)$ is a linear form on $G_k$ and so there exists a unique element $A(x_1, \dots, x_{k-1}) \in G_k$ such that $\alpha(x_1, \dots, x_{k-1}, y_k) = A(x_1, \dots, x_{k-1}) \cdot y_k$ for every $y_k \in G_k$. It is easy to see that $A : G_1 \tdt G_{k-1} \to G_k$ is multilinear.}\\

Throughout the paper, unless explicitly stated otherwise the implicit constants in the big-O notation depend on $p$ and $k$ only.\\

\noindent \textbf{Additional asymptotic notation.} In the later parts of the paper, we use the notation $\con$ as a placeholder for a positive constant whose value is not important. For example,
\[(\forall x, y > 1) \hspace{3pt} \Big(x \geq \con y^{\con}\hspace{3pt}\implies\hspace{3pt}x \geq 100 y^2 \log y\Big)\]
is a shorthand for 
\[(\exists C_1, C_2 > 0)(\forall x, y > 1) \hspace{3pt} \Big(x \geq C_1 y^{C_2}\implies x \geq 100 y^2 \log y\Big).\]
Formally, let $x, y_1, \dots, y_m$ be variables, let $p_1, \dots, p_n$ let parameters, let $X, Y_1, \dots, Y_m \subset \mathbb{R}$ be sets, let $f \colon \mathbb{R}^m \times \mathbb{R}^n \to \mathbb{R}$ be a function and let $P(x, y_1, \dots, y_m)$ be a proposition whose truth value depends on $x, y_1, \dots, y_m$. Then we define 
\[(\forall x \in X)(\forall y \in Y_1)\dots (\forall y \in Y_m)\ \Big( x \geq f(y_1, \dots, y_m; \con, \dots, \con) \implies P(x,y_1,\dots,y_m)\Big)\] 
to be a shorthand for
\begin{equation}(\exists C_1 > 0) \dots (\exists C_n > 0)(\forall x \in X)(\forall y \in Y_1)\dots (\forall y \in Y_m) \hspace{3pt} \Big(x \geq f(y_1, \dots, y_m; C_1, \dots, C_n) \implies P(x,y_1,\dots,y_m)\Big).\label{fullConEqn}\end{equation}

We also use the notation 
\[(\forall x \in X)(\forall y \in Y_1)\dots (\forall y \in Y_m) \hspace{3pt} \Big(x \leq f(y_1, \dots, y_m; \con, \dots, \con) \implies P(x,y_1,\dots,y_m)\Big)\] 
which is defined by replacing $\geq$ with $\leq$ in the full expression above. Since this is an unusual notation, we provide a few more examples.

\[(\forall x > 2)(\forall y > 0)\hspace{3pt} \Big(x \geq \con y \implies x \log x \geq 1000 y\Big)\]
\[(\forall x > 0)(\forall y > 2)\hspace{3pt} \Big(x \leq \con y \implies 1000 x \leq y \log y\Big)\]
\[(\forall x, y,z > 1)\hspace{3pt} \Big(x \geq \con y^{\con} + \con \cdot z^{\con} \implies x \geq yz + y + 1\Big)\]
\[(\forall x > 1)(\forall y > 1)\hspace{3pt} \Big(x \geq \con  y^{\con} \implies \exp(\sqrt{x}) \geq y^{100}\Big)\]
\[(\forall x \in (0,1))(\forall y \in (0,1))\hspace{3pt} \Big(x \leq \con y^{\con} \implies \sqrt{x} \leq \frac{y}{100}\Big)\]

When the sets $X, Y_1, \dots, Y_m$ are clear from the context, we drop the universal quantifier part in the expressions above. For example, if we already know that $x$ and $y$ take values in $(0,1)$, then the last example may be written as
\[x \leq \con y^{\con} \implies \sqrt{x} \leq \frac{y}{100}.\]

To help readability, we also adopt $\cons$ notation, which has the same logical meaning as $\con$, but indicates that the property holds provided the implicit constant is sufficiently small. On the other hand, we shall think of $\con$ as a sufficiently large constant. Using $\cons$, the last example may be written as
\[x \leq \cons y^{\con} \implies \sqrt{x} \leq \frac{y}{100}.\]

Finally, if we have some parameters $s_1, \dots, s_m$ that are fixed beforehand (in our case these will almost always be the size $p$ of the field and the number of variables, $k$) we may write $\con_{s_1, \dots, s_m}$ and $\cons_{s_1, \dots, s_m}$ to indicate that implicit constants depend on these parameters. This has the effect of adding $(\forall s_1, \dots, s_m)$ at the beginning of the expression~\eqref{fullConEqn}, and replacing $C_i$ by $C_i(s_1, \dots, s_m)$. However, these dependencies will be clear from the context, so we shall mostly be using $\con$ and $\cons$ notation without the parameters explicitly written out.

\subsection{Useful inequalities and identities}

\chg{
We make use of the basic fact that, with high probability, the averages of long sequences can be approximated by averages of their random short subsequences.}

\begin{lemma}[Random sampling approximation]\label{randomsampling}\chg{Let $a_1, \dots, a_n \in \mathbb{D}$. Let $k \in [n]$ and $\varepsilon > 0$. Pick indices $i_1, \dots, i_k \in [n]$ uniformly and independently at random. Then
\[\mathbb{P}\Big(\Big|\frac{1}{n} \sum_{j \in [n]} a_j - \frac{1}{k} \sum_{j \in [k]} a_{i_j}\Big| \leq \varepsilon\Big) \geq 1 - 4\exp\Big(-\frac{\varepsilon^2 k}{8}\Big).\]}
\end{lemma}

\begin{proof}\chg{Set $\alpha = \frac{1}{n} \sum_{j \in [n]} a_j$. We define random variables $X_j = \on{Re} a_{i_j}$ and $Y_j = \on{Im} a_{i_j}$. The random variables $X_1, \dots, X_k$ are independent and take values in $[-1,1]$. Similarly, $Y_1, \dots, Y_k$ are independent and take values in $[-1,1]$. Observe that $\ex X_j = \on{Re} \alpha$ and $\ex Y_j = \on{Im} \alpha$. Hoeffding's inequality implies the bounds
\[\mathbb{P}\Big(\Big|\frac{1}{k} \sum_{j \in [k]} X_j - \on{Re} \alpha\Big| \geq \varepsilon/2\Big) \leq 2\exp\Big(-\frac{\varepsilon^2 k}{8}\Big)\]
and
\[\mathbb{P}\Big(\Big|\frac{1}{k} \sum_{j \in [k]} Y_j - \on{Im} \alpha\Big| \geq \varepsilon/2\Big) \leq 2\exp\Big(-\frac{\varepsilon^2 k}{8}\Big).\]
These two bounds imply the lemma.}\end{proof}

We record the following standard fact as a lemma. It is a direct consequence of Parseval's identity.

\begin{lemma}\label{LargeFCs}Let $f \colon G \to \mathbb{D}$ be a map. For $\varepsilon > 0$ there are at most $\varepsilon^{-2}$ values of $r \in G$ such that $|\hat{f}(r)| \geq \varepsilon$.\end{lemma}






\chg{Let us introduce a piece of notation that will be useful for the following lemma. For a function of variable $x$, where $x$ is known to belong to a set $X$, we write $\|f(x)\|_{L^q(x)}$ as the shorthand for $\Big(\exx_{x \in X} |f(x)|^q\Big)^{1/q}$. This is useful when $f$ is a given implicitly, for example 
\[\Big\|\sum_{i \in [k]} \chi(a_i \cdot x)\Big\|_{L^q(x)}\]
stands for the $L^q$ norm of the function $x \mapsto \sum_{i \in [k]} \chi(a_i \cdot x)$ and allows us to evade unnecessary definitions.} 

The next lemma tells us that we may use the large spectrum to approximate convolutions of functions (of a single variable). 

\begin{lemma}\label{basicL2app}\chg{Let $f,g \colon G \to \mathbb{D}$ and let $\varepsilon > 0$. Let $S \subset G$ be any set that contains all $r \in G$ such that $|\hat{f}(r)|$ or $|\hat{g}(r)| \geq \varepsilon$. Then we have
\[\Big\|f \conv g(x) - \sum_{r \in S} \hat{f}(r) \overline{\hat{g}(r)} \chi(x \cdot r)\Big\|_{L^2(x)} \leq \varepsilon.\]}\end{lemma}

\noindent\textbf{Remark.} There are much more sophisticated results on approximations of convolutions when it comes to $L^q$ control for large $q$. For an improved Fourier analytic approximation see Bourgain~\cite{Bourgain}. A more combinatorial argument with substantially better bounds was given by Croot and Sisask~\cite{CrootSisask}. However, the most basic estimate above is sufficient for our purposes.\\

\begin{proof}We simply expand the square of the $L^2$ norm to get
\begin{align*}\Big\|f \conv g(x) - &\sum_{r \in S} \hat{f}(r) \overline{\hat{g}(r)} \chi(x \cdot r)\Big\|^2_{L^2(x)} = \exx_{x \in G} \Big|f \conv g(x) - \sum_{r \in S} \hat{f}(r) \overline{\hat{g}(r)} \chi(x \cdot r)\Big|^2 = \exx_{x \in G} \Big|\sum_{r \notin S} \hat{f}(r) \overline{\hat{g}(r)} \chi(x \cdot r)\Big|^2\\
= & \sum_{r, s \notin S} \exx_{x \in G} \hat{f}(r) \overline{\hat{g}(r)}\, \overline{\hat{f}(s)} \hat{g}(s) \chi(x \cdot (r-s))\, = \,\sum_{r \notin S} |\hat{f}(r)|^2 |\hat{g}(r)|^2.\end{align*}
Since $|\hat{f}(r)|, |\hat{g}(r)| \leq \varepsilon$ holds for $r \notin S$, the expression above is at most $\varepsilon^2 \sum_{r \in G} |\hat{f}(r)||\hat{g}(r)|$. By the Cauchy--Schwarz inequality, this is at most $\varepsilon^2 \sqrt{\sum_{r \in G} |\hat{f}(r)|^2} \sqrt{\sum_{r \in G} |\hat{g}(r)|^2}$. Using Parseval's identity we get the upper bound of $\varepsilon^2 \sqrt{\exx_{x \in G} |f(x)|^2} \sqrt{\exx_{x \in G} |g(x)|^2} \leq \,\varepsilon^2$,
which proves the lemma.\end{proof}

The following elementary inequality will allows us to restrict various complex values to the unit disk. It is easily proved by elementary Euclidean geometry or by direct algebraic manipulation.

\begin{fact}\label{resttodiskfact} Let $a, b \in \mathbb{C}$ be such that $|a| \leq 1 \leq |b|$. Then
\[|a - b| \geq \Big|a - \frac{b}{|b|}\Big|.\]
\end{fact}

\begin{corollary}\label{basicLqapp}Let $f,g \colon G \to \mathbb{D}$ and $\varepsilon > 0$. Let $S = \{s_1, \dots, s_\ell\} \subset G$ be any set that contains all $r \in G$ such that $|\hat{f}(r)|$ or $|\hat{g}(r)| \geq \varepsilon$. Define a linear map $\phi : G \to \mathbb{F}_p^\ell$ by $\phi(x) = (s_1 \cdot x, \dots, s_\ell \cdot x)$. Then there exists a map $c : \mathbb{F}_p^\ell \to \mathbb{D}$ such that
\[\Big\|f \conv g - c \circ \phi\Big\|_{L^2} \leq \varepsilon.\]\end{corollary}

\begin{proof}Using Lemma~\ref{basicL2app}, we get an approximation $\sum_{r \in S} \hat{f}(r) \overline{\hat{g}(r)} \chi(x \cdot r)$. Note that this sum is constant along cosets of the subspace $\langle S \rangle^\perp$. Hence, it can be written as $c \circ \phi$ for some $c : \mathbb{F}_p^\ell \to \mathbb{C}$, where it may happen that $|c(\lambda)| > 1$ for some $\lambda \in \mathbb{F}_p^\ell$. If so, replace it by $c(\lambda)/|c(\lambda)|$, which, by Fact~\ref{resttodiskfact}, can only improve the approximation as $\|f \conv g\|_\infty \leq 1$.\end{proof}

\begin{lemma}\label{LargeFCs}Let $f \colon G \to \mathbb{D}$ be a map. For $\varepsilon > 0$ there are at most $\varepsilon^{-2}$ elements $r \in G$ such that $|\hat{f}(r)| \geq \varepsilon$.\end{lemma}

We recall the following lemma which is implicit in~\cite{TimSze}.

\begin{lemma}\label{l4bound}Let $f,g \colon G \to \mathbb{D}$. Then
\[\Big(\exx_d \Big|\exx_x \overline{f}(x) g(x + d)\Big|^2\Big)^2 \leq \min\Big\{\sum_r |\hat{f}(r)|^4, \sum_r |\hat{g}(r)|^4\Big\}.\]\end{lemma}

\begin{proof} We prove that the expression is at most $\sum_r |\hat{f}(r)|^4$, which is sufficient.
\begin{align*}\Big(\exx_d \Big|\exx_x \overline{f}(x) g(x + d)\Big|^2\Big)^2 =& \Big(\exx_d |g \conv f(d)|^2\Big)^2 = \Big(\sum_r \Big|\widehat{g \conv f}(r)\Big|^2\Big)^2 = \Big(\sum_r |\hat{g}(r)|^2|\hat{f}(r)|^2\Big)^2 \\
\leq& \Big(\sum_r |\hat{g}(r)|^4\Big) \Big(\sum_r |\hat{f}(r)|^4\Big) \leq \Big(\sum_r |\hat{g}(r)|^2\Big) \Big(\sum_r |\hat{f}(r)|^4\Big)\\
\leq& \sum_r |\hat{f}(r)|^4.\end{align*}
\end{proof}

The next lemma is somewhat technical and it is used to give bounds on $\|\cdot\|_{L^q}$ norms of certain expressions in terms of norms of simpler ones. 

\begin{lemma}\label{squareNormBounds}Let $f,g,h \colon X \to \mathbb{C}$ be functions such that $\|f\|_{L^{\infty}}, \|h\|_{L^{\infty}} \leq 1$. Then
\[\Big\||f|^2h - |g|^2h\Big\|_{L^q} \leq (2 + \|f-g\|_{L^{2q}}) \|f-g\|_{L^{2q}}.\]\end{lemma}

\begin{proof}We have
\begin{align*}\Big\||f|^2h - |g|^2h\Big\|^q_{L^q} =& \exx_x \Big|(|f(x)|^2- |g(x)|^2 )h(x)\Big|^q \leq \exx_x \Big||f(x)|^2- |g(x)|^2\Big|^q\\
 = &\exx_x \Big||f(x)|- |g(x)|\Big|^q \Big||f(x)| + |g(x)|\Big|^q \leq \sqrt{\exx_x \Big||f(x)|- |g(x)|\Big|^{2q}} \sqrt{\exx_x \Big||f(x)| + |g(x)|\Big|^{2q}}\\
\leq & \sqrt{\exx_x \Big|f(x)- g(x)\Big|^{2q}} \sqrt{\exx_x \Big|2|f(x)| + (|g(x)| - |f(x)|)\Big|^{2q}}\\
= &\|f-g\|_{L^{2q}}^q\hspace{2pt} \| 2|f| + (|g| - |f|)\|_{L^{2q}}^q.\end{align*}
Taking $q$\textsuperscript{th} roots, we obtain
\[\||f|^2h - |g|^2h\|_{L^q} \leq \|f-g\|_{L^{2q}} \| 2|f| + (|g| - |f|)\|_{L^{2q}} \leq \|f-g\|_{L^{2q}} (2\|f\|_{L^{2q}} + \|f-g\|_{L^{2q}}) \leq (2 + \|f-g\|_{L^{2q}}) \|f-g\|_{L^{2q}},\]
as claimed.
\end{proof}

Next, we show that if two functions are sufficiently close in the $\|\cdot\|_{L^{2q}}$ norm, then their convolutions in the given direction are close as well, though in the $\|\cdot\|_{L^{q}}$ norm.

\begin{lemma}\label{2pConvpBnd}Let $f,g \colon G_{[k]} \to \mathbb{C}$ be maps such that $\|f\|_{L^\infty} \leq 1$ and $\|f-g\|_{L^{2q}} \leq \varepsilon$. Then for any $d\in [k]$, 
\[\|\bigconv{d}f - \bigconv{d}g\|_{L^q} \leq 4\varepsilon + 2\varepsilon^2.\]
\end{lemma}

\begin{proof}This is a simple consequence of the Cauchy--Schwarz and $L^q$-norm triangle inequalities. Without loss of generality $d = k$. We have
\begin{align*}\|\bigconv{k}f - \bigconv{k}g\|^q_{L^q} = &\exx_{x_{[k]} \in G_{[k]}} \Big|\exx_{y_k \in G_k} f(x_{[k-1]}, x_k + y_k) \overline{f(x_{[k-1]}, y_k)} - \exx_{y_k \in G_k} g(x_{[k-1]}, x_k + y_k) \overline{g(x_{[k-1]}, y_k)}\Big|^q\\
= & \exx_{x_{[k]} \in G_{[k]}} \Big|\exx_{y_k \in G_k} f(x_{[k-1]}, x_k + y_k) (\overline{f(x_{[k-1]}, y_k)} - \overline{g(x_{[k-1]}, y_k)}) \\
&\hspace{1cm}+ \exx_{y_k \in G_k} (f(x_{[k-1]}, x_k + y_k) - g(x_{[k-1]}, x_k + y_k)) \overline{g(x_{[k-1]}, y_k)}\Big|^q\\
\leq &2^{q-1}\exx_{x_{[k]} \in G_{[k]}} \Big|\exx_{y_k \in G_k} f(x_{[k-1]}, x_k + y_k) (\overline{f(x_{[k-1]}, y_k)} - \overline{g(x_{[k-1]}, y_k)})\Big|^q\\
&\hspace{1cm}+ 2^{q-1}\exx_{x_{[k]} \in G_{[k]}} \Big|\exx_{y_k \in G_k}(f(x_{[k-1]}, x_k + y_k) - g(x_{[k-1]}, x_k + y_k)) \overline{g(x_{[k-1]}, y_k)}\Big|^q\\
\leq &2^{q-1}\exx_{x_{[k-1]} \in G_{[k-1]}} \Big(\sqrt{\exx_{y_k \in G_k} |f(x_{[k-1]}, y_k)|^2} \hspace{2pt}\sqrt{\exx_{z_k \in G_k} |f(x_{[k-1]}, z_k) - g(x_{[k-1]}, z_k)|^2}\Big)^q\\
&\hspace{1cm}+2^{q-1}\exx_{x_{[k-1]} \in G_{[k-1]}} \Big(\sqrt{\exx_{y_k \in G_k} |f(x_{[k-1]}, y_k) - g(x_{[k-1]}, y_k)|^2} \hspace{2pt}\sqrt{\exx_{z_k \in G_k} |g(x_{[k-1]}, z_k)|^2}\Big)^q\\
= &2^{q-1}\exx_{x_{[k-1]} \in G_{[k-1]}} \Big(\sqrt{\exx_{y_k \in G_k} |f(x_{[k-1]}, y_k)|^2}\Big)^q \hspace{2pt}\Big(\sqrt{\exx_{z_k \in G_k} |f(x_{[k-1]}, z_k) - g(x_{[k-1]}, z_k)|^2}\Big)^q\\
&\hspace{1cm}+2^{q-1}\exx_{x_{[k-1]} \in G_{[k-1]}} \Big(\sqrt{\exx_{y_k \in G_k} |f(x_{[k-1]}, y_k) - g(x_{[k-1]}, y_k)|^2}\Big)^q \hspace{2pt}\Big(\sqrt{\exx_{z_k \in G_k} |g(x_{[k-1]}, z_k)|^2}\Big)^q\\
\leq &2^{q-1}\sqrt{ \exx_{x_{[k-1]} \in G_{[k-1]}} \Big(\exx_{y_k \in G_k} |f(x_{[k-1]}, y_k)|^2}\Big)^q \\
&\hspace{2cm}\cdot\,\sqrt{\exx_{x_{[k-1]} \in G_{[k-1]}} \Big(\exx_{z_k \in G_k} |f(x_{[k-1]}, z_k) - g(x_{[k-1]}, z_k)|^2}\Big)^q\\
&\hspace{1cm}+2^{q-1}\sqrt{\exx_{x_{[k-1]}} \Big(\exx_{y_k \in G_k} |f(x_{[k-1]}, y_k) - g(x_{[k-1]}, y_k)|^2}\Big)^q \hspace{2pt}\sqrt{\exx_{x_{[k-1]}} \Big(\exx_{z_k \in G_k} |g(x_{[k-1]}, z_k)|^2}\Big)^q\\
\leq&2^{q-1}(\|f\|_{L^{2q}}^q + \|g\|_{L^{2q}}^q) \|f-g\|_{L^{2q}}^q \leq 2^{q-1}(\|f\|_{L^{2q}} + \|g\|_{L^{2q}})^q \|f-g\|_{L^{2q}}^q.\end{align*}
The lemma follows after taking $q$\textsuperscript{th} roots.\end{proof}

If we allow more convolutions, then we can obtain approximations in the $L_\infty$ norm.

\begin{lemma}\label{LinftyConvApprox}Let $f,g \colon G_{[k]}  \to \mathbb{D}$ be two maps and let $d_1, \dots, d_r \in [k]$ be directions such that $\{d_1, \dots, d_r\} = [k]$ (allowing repetition of directions). Then
\[\Big\|\bigconv{d_1} \dots \bigconv{d_r} f - \bigconv{d_1} \dots \bigconv{d_r}g\Big\|_{L_\infty} \leq 2^r \|f-g\|_{L^1}.\]
\end{lemma}

Before proceeding with the proof, we derive an explicit formula coming from expanding out convolutions. It is notationally complex, but it is essentially a straightforward generalization of the following simple special case. If $k=2$ and $h:G_1\times G_2\to\mathbb{C}$, then $\bigconv{2}h (x,b)=\ex_y h(x,y+b)\overline{h(x,y)}$ and 
\[\bigconv{1}\bigconv{2} h(a,b)=\exx_{x,y_1,y_2}h(x+a,y_1+b)\overline{h(x+a,y_1)}\,\overline{h(x,y_2+b)}h(x,y_2).\]
That is, $\bigconv{2}h(x,b)$ is the average of a suitable product over `vertical edges' of height $b$ in column $x$, and $\bigconv{1}\bigconv{2}h(a,b)$ is the average over `vertical parallelograms' made out of pairs of such edges. If we were to convolve again in direction 1, then the value at $(t,b)$ would be an average over pairs of vertical parallelograms of the same height, and with widths that differ by $t$, and so on. In general, each time we convolve in some direction, we duplicate the previous configuration in a certain way, so after $r$ convolutions the number of points in a configuration is $2^r$. 

\begin{lemma}\label{convFormulaExplicit}Let $h \colon G_{[k]} \to \mathbb{C}$ be a function, let $d_1, \dots, d_r \in [k]$ be directions and for each $d \in [k]$, let $j_{d, 1}, \dots, j_{d, l_d}$ be those $i$ for which $d_i = d$, sorted in increasing order. Let $x_{[k]} \in G_{[k]}$. For parameters $\bm{a} = (a^1, a^2, a^3, \dots, a^r) \in G_{d_1} \times G_{d_2}^{\{0,1\}} \times G_{d_3}^{\{0,1\}^2} \tdt G_{d_r}^{\{0,1\}^{r-1}}$, and $\varepsilon \in \{0,1\}^r$, define a point $\bm{p}^{\bm{a}, \varepsilon} \in G_{[k]}$ by setting
\begin{equation}\label{confFormulaPtsDefn}\bm{p}^{\bm{a}, \varepsilon}_d = \varepsilon_{j_{d, 1}} \cdots \varepsilon_{j_{d, l_d}} x_d + \varepsilon_{j_{d, 2}} \cdots \varepsilon_{j_{d, l_d}} a^{j_{d,1}}_{\varepsilon|_{[j_{d, 1} - 1]}} + \dots + \varepsilon_{j_{d,l_d}} a^{j_{d,l_d - 1}}_{\varepsilon|_{[j_{d, l_d - 1} - 1]}} + a^{j_{d,l_d}}_{\varepsilon|_{[j_{d, l_d} - 1]}}.\end{equation}
Then
\[\bigconv{d_1} \dots \bigconv{d_r} h(x_{[k]}) = \exx_{\bm{a}} \prod_{\varepsilon \in \{0,1\}^r} \operatorname{Conj}^{r - |\varepsilon|} h(\bm{p}^{\bm{a}, \varepsilon}),\]
where $|\varepsilon| = \varepsilon_1 + \dots + \varepsilon_r$, and $\bm{a}$ ranges over all choices of parameters in $G_{d_1} \times G_{d_2}^{\{0,1\}} \times G_{d_3}^{\{0,1\}^2} \tdt G_{d_r}^{\{0,1\}^{r-1}}$.\end{lemma}

\begin{proof}[Proof of Lemma~\ref{convFormulaExplicit}]We prove the claim by induction on $r$. For $r = 1$, the claim is trivial. Assume that $r \geq 2$ and that the claim holds for smaller values of $r$ and let $d_1, \dots, d_r$ be directions. Then
\[\bigconv{d_1} \dots \bigconv{d_r} h(x_{[k]}) = \exx_{a^1 \in G_{d_1}} \bigconv{d_2} \dots \bigconv{d_{r}} h(x_{[k] \setminus \{d_1\}},x_{d_1} + a^1)\overline{\bigconv{d_2} \dots \bigconv{d_{r}} h(x_{[k] \setminus \{d_1\}},a^1)}.\]
For parameters $a^1 \in G_1$, $\bm{b} = (b^2, b^3, \dots, b^r), \bm{c} = (c^2, c^3, \dots, c^r) \in G_{d_2} \times G_{d_3}^{\{0,1\}} \tdt G_{d_r}^{\{0,1\}^{r-2}}$ and $\varepsilon \in \{0,1\}^{[2, r]}$ (note indexing by $2,\dots, r$ instead of $1, \dots, r-1$), define points $\bm{s}^{\bm{b}, \varepsilon}, \bm{t}^{\bm{c}, \varepsilon} \in G_{[k]}$ (we suppress $a^1$ from the notation) by setting
\[\bm{s}^{\bm{b}, \varepsilon}_d =\begin{cases} \varepsilon_{j_{d, 1}} \cdots \varepsilon_{j_{d, l_d}} x_d + \varepsilon_{j_{d, 2}} \cdots \varepsilon_{j_{d, l_d}} b^{j_{d,1}}_{\varepsilon|_{[2, j_{d, 1} - 1]}} + \dots + \varepsilon_{j_{d,l_d}} b^{j_{d,l_d - 1}}_{\varepsilon|_{[2, j_{d, l_d - 1} - 1]}} + b^{j_{d,l_d}}_{\varepsilon|_{[2, j_{d, l_d} - 1]}}, &\text{when }d \not= d_1\\
\varepsilon_{j_{d, 2}} \cdots \varepsilon_{j_{d, l_d}} (a^1 + x_d) + \varepsilon_{j_{d, 3}} \cdots \varepsilon_{j_{d, l_d}} b^{j_{d,2}}_{\varepsilon|_{[2, j_{d, 2} - 1]}} + \dots + \varepsilon_{j_{d,l_d}} b^{j_{d,l_d - 1}}_{\varepsilon|_{[2, j_{d, l_d - 1} - 1]}} + b^{j_{d,l_d}}_{\varepsilon|_{[2, j_{d, l_d} - 1]}} &\text{when }d = d_1
\end{cases}\]
and
\[\bm{t}^{\bm{c}, \varepsilon}_d =\begin{cases} \varepsilon_{j_{d, 1}} \cdots \varepsilon_{j_{d, l_d}} x_d + \varepsilon_{j_{d, 2}} \cdots \varepsilon_{j_{d, l_d}} c^{j_{d,1}}_{\varepsilon|_{[2, j_{d, 1} - 1]}} + \dots + \varepsilon_{j_{d,l_d}} c^{j_{d,l_d - 1}}_{\varepsilon|_{[2, j_{d, l_d - 1} - 1]}} + c^{j_{d,l_d}}_{\varepsilon|_{[2, j_{d, l_d} - 1]}}, &\text{when }d \not= d_1\\
\varepsilon_{j_{d, 2}} \cdots \varepsilon_{j_{d, l_d}} a^1 + \varepsilon_{j_{d, 3}} \cdots \varepsilon_{j_{d, l_d}} c^{j_{d,2}}_{\varepsilon|_{[2, j_{d, 2} - 1]}} + \dots + \varepsilon_{j_{d,l_d}} c^{j_{d,l_d - 1}}_{\varepsilon|_{[2, j_{d, l_d - 1} - 1]}} + c^{j_{d,l_d}}_{\varepsilon|_{[2, j_{d, l_d} - 1]}} &\text{when }d = d_1.
\end{cases}\]
By the induction hypothesis, we have
\[\bigconv{d_2} \dots \bigconv{d_r} h(x_{[k] \setminus \{d_1\}},x_{d_1} + a^1) = \exx_{\bm{b}} \prod_{\varepsilon \in \{0,1\}^{[2, r]}} \operatorname{Conj}^{r - 1 - |\varepsilon|} h(\bm{s}^{\bm{b}, \varepsilon}),\]
and
\[\bigconv{d_2} \dots \bigconv{d_r} h(x_{[k] \setminus \{d_1\}}, a^1) = \exx_{\bm{c}} \prod_{\varepsilon \in \{0,1\}^{[2, r]}} \operatorname{Conj}^{r - 1 - |\varepsilon|} h(\bm{t}^{\bm{c}, \varepsilon}).\]
Rename the parameters and points by setting $a^i_{1, \varepsilon} = b^i_\varepsilon, a^i_{0, \varepsilon} = c^i_\varepsilon$ for each $i \in [2,r]$ and each $\varepsilon \in \{0,1\}^{[2,i]}$, and setting $\bm{p}^{\bm{a}, (1,\varepsilon)} = \bm{s}^{\bm{b}, \varepsilon}, \bm{p}^{\bm{a}, (0,\varepsilon)} = \bm{t}^{\bm{c}, \varepsilon}$ for each $\varepsilon \in \{0,1\}^{[2, r]}$, where we write $\bm{a}$ for $(a^1, a^2, a^3, \dots, a^r)$. Then
\begin{align*}\bigconv{d_1} \dots \bigconv{d_r} h(x_{[k]}) = &\exx_{a^1 \in G_{d_1}} \bigconv{d_2} \dots \bigconv{d_{r}} h(x_{[k] \setminus \{d_1\}},x_{d_1} + a^1)\overline{\bigconv{d_2} \dots \bigconv{d_{r}}h(x_{[k] \setminus \{d_1\}},a^1)}\\
=&\exx_{a^1 \in G_{d_1}} \Big(\exx_{\bm{b}} \prod_{\varepsilon \in \{0,1\}^{[2, r]}} \operatorname{Conj}^{r - 1 - |\varepsilon|} h(\bm{s}^{\bm{b}, \varepsilon})\Big) \Big(\exx_{\bm{c}} \prod_{\varepsilon \in \{0,1\}^{[2, r]}} \operatorname{Conj}^{r - |\varepsilon|} h(\bm{t}^{\bm{c}, \varepsilon})\Big)\\
= &\exx_{\bm{a}} \Big(\prod_{\varepsilon \in \{0,1\}^{[2, r]}} \operatorname{Conj}^{r - 1 - |\varepsilon|} h(\bm{p}^{\bm{a}, (1,\varepsilon)})\Big)\Big(\prod_{\varepsilon \in \{0,1\}^{[2, r]}} \operatorname{Conj}^{r - |\varepsilon|} h(\bm{p}^{\bm{a}, (0,\varepsilon)})\Big)\\
= &\exx_{\bm{a}}\prod_{\varepsilon \in \{0,1\}^{r}} \operatorname{Conj}^{r - |\varepsilon|} h(\bm{p}^{\bm{a}, \varepsilon}).\end{align*}
It remains to check that points $\bm{p}^{\bm{a}, \varepsilon}$ have the form described in the statement. For coordinates $d \not= d_1$, this is clear, since $j_{d, i} = 1$ if and only if $d = d_1, i = 1$. Let $\varepsilon \in \{0,1\}^{[2, r]}$. Then
\begin{align*}\bm{p}^{\bm{a}, (1,\varepsilon)}_{d_1} = \bm{s}^{\bm{b}, \varepsilon}_{d_1} &= \varepsilon_{j_{d_1, 2}} \cdots \varepsilon_{j_{d_1, l_{d_1}}} (a^1 + x_{d_1}) + \varepsilon_{j_{d_1, 3}} \cdots \varepsilon_{j_{d_1, l_{d_1}}} b^{j_{d_1,2}}_{\varepsilon|_{[2, j_{d_1, 2} - 1]}} + \dots\\
& \hspace{5cm} + \varepsilon_{j_{d_1,l_{d_1}}} b^{j_{d_1,l_{d_1} - 1}}_{\varepsilon|_{[2, j_{d_1, l_{d_1} - 1} - 1]}} + b^{j_{d_1,l_{d_1}}}_{\varepsilon|_{[2, j_{d_1, l_{d_1}} - 1]}}\\
&= 1\cdot \varepsilon_{j_{d_1, 2}} \cdots \varepsilon_{j_{d_1, l_{d_1}}}x_{d_1}  + \varepsilon_{j_{d_1, 2}} \cdots \varepsilon_{j_{d_1, l_{d_1}}}a^1  + \varepsilon_{j_{d_1, 3}} \cdots \varepsilon_{j_{d_1, l_{d_1}}} a^{j_{d_1,2}}_{(1, \varepsilon|_{[2, j_{d_1, 2} - 1]})} + \dots\\
&\hspace{2cm}  + \varepsilon_{j_{d_1,l_{d_1}}} a^{j_{d_1,l_{d_1} - 1}}_{(1, \varepsilon|_{[2, j_{d_1, l_{d_1} - 1} - 1])}} + a^{j_{d_1,l_{d_1}}}_{(1,\varepsilon|_{[2, j_{d_1, l_{d_1}} - 1]})}\end{align*}
and
\begin{align*}\bm{p}^{\bm{a}, (0,\varepsilon)}_{d_1} = \bm{t}^{\bm{c}, \varepsilon}_{d_1} &= \varepsilon_{j_{d_1, 2}} \cdots \varepsilon_{j_{d_1, l_{d_1}}}  a^1 + \varepsilon_{j_{d_1, 3}} \cdots \varepsilon_{j_{d_1, l_{d_1}}} c^{j_{d_1,2}}_{\varepsilon|_{[2, j_{d_1, 2} - 1]}} + \dots + \varepsilon_{j_{d_1,l_{d_1}}} c^{j_{d_1,l_{d_1} - 1}}_{\varepsilon|_{[2, j_{d_1, l_{d_1} - 1} - 1]}} + c^{j_{d_1,l_{d_1}}}_{\varepsilon|_{[2, j_{d_1, l_{d_1}} - 1]}}\\
&= 0 \cdot \varepsilon_{j_{d_1, 2}} \cdots \varepsilon_{j_{d_1, l_{d_1}}}x_{d_1}  + \varepsilon_{j_{d_1, 2}} \cdots \varepsilon_{j_{d_1, l_{d_1}}} a^1  + \varepsilon_{j_{d_1, 3}} \cdots \varepsilon_{j_{d_1, l_{d_1}}} a^{j_{d_1,2}}_{(0, \varepsilon|_{[2, j_{d_1, 2} - 1]})} + \dots\\
&\hspace{2cm}  + \varepsilon_{j_{d_1,l_{d_1}}} a^{j_{d_1,l_{d_1} - 1}}_{(0, \varepsilon|_{[2, j_{d_1, l_{d_1} - 1} - 1])}} + a^{j_{d_1,l_{d_1}}}_{(0,\varepsilon|_{[2, j_{d_1, l_{d_1}} - 1]})},\end{align*}
as desired.\end{proof}

\begin{proof}[Proof of Lemma~\ref{LinftyConvApprox}]Let $x_{[k]} \in G_{[k]}$. Use the same notation as in the proof of Lemma~\ref{convFormulaExplicit}. That lemma implies that
\[\bigconv{d_1} \dots \bigconv{d_r} f(x_{[k]}) - \bigconv{d_1} \dots \bigconv{d_r}g(x_{[k]}) = \exx_{\bm{a}} \prod_{\varepsilon \in \{0,1\}^r} \operatorname{Conj}^{r - |\varepsilon|} f(\bm{p}^{\bm{a}, \varepsilon}) - \prod_{\varepsilon \in \{0,1\}^r} \operatorname{Conj}^{r - |\varepsilon|} g(\bm{p}^{\bm{a}, \varepsilon}).\]
Order all elements of $\{0,1\}^r$ as $\varepsilon^1, \dots, \varepsilon^{2^r}$. Then
\begin{align*}\Big|\bigconv{d_1} \dots \bigconv{d_r} f(x_{[k]}) - \bigconv{d_1} \dots \bigconv{d_r}g(x_{[k]})\Big| =& \Big|\exx_{\bm{a}} \sum_{i \in [2^r]} \Big(\prod_{j \in [i-1]} \operatorname{Conj}^{r - |\varepsilon^j|} f(\bm{p}^{\bm{a}, \varepsilon^j})\Big) \operatorname{Conj}^{r - |\varepsilon^i|}\Big(f(\bm{p}^{\bm{a}, \varepsilon^i}) - g(\bm{p}^{\bm{a}, \varepsilon^i})\Big)\\
&\hspace{2cm}\Big(\prod_{j \in [i+1, 2^r]} \operatorname{Conj}^{r - |\varepsilon^j|} g(\bm{p}^{\bm{a}, \varepsilon^j})\Big)\Big|\\
\leq&\sum_{i \in [2^r]} \exx_{\bm{a}} \Big|(f- g)(\bm{p}^{\bm{a}, \varepsilon^i})\Big|\\
=&2^r \|f-g\|_{L^1},\end{align*}
since for each $\varepsilon \in \{0,1\}^r$, the point $\bm{p}^{\bm{a}, \varepsilon}$ ranges uniformly over all $G_{[k]}$ as $\bm{a}$ ranges over $G_{d_1} \times G_{d_2}^{\{0,1\}} \times G_{d_3}^{\{0,1\}^2} \tdt G_{d_r}^{\{0,1\}^{r-1}}$, which is due to~\eqref{confFormulaPtsDefn} and the fact that all directions are present in $d_1, \dots, d_r$ (just look at the parameters $a^{j_{d,l_d}}_{\varepsilon|_{[j_{d, l_d} - 1]}}$, $d \in [k]$ for any fixed choice of the other ones).\end{proof}

\subsection{Linear algebra results}
Let $\mathbb{F}$ be a finite field of size $\mathbf{f}$. We begin this subsection with a simple criterion for solubility of systems of linear equations.

\begin{lemma}\label{solCrit}Let $G$ be a \chg{finite-dimensional} vector space over $\mathbb F$. Let $x_1, \dots, x_r \in G$ and let $\lambda_1, \dots, \lambda_r \in \mathbb{F}$. Let $u_0 + U$ be a coset in $G$. Then the following are equivalent.
\begin{itemize}
\item[\textbf{(i)}]There exists $y \in u_0 + U$ such that $x_i \cdot y = \lambda_i$ for each $i \in [r]$.
\item[\textbf{(ii)}]$\sum_{i \in [r]} \mu_i(\lambda_i - x_i \cdot u_0) = 0$ for every $\mu \in \mathbb{F}^r$ such that $\sum_{i \in [r]} \mu_i x_i \in U^\perp$. 
\end{itemize}
\end{lemma}

\begin{proof} \chg{Let us define a linear map $\phi : U \to \mathbb{F}^r$ by $\phi(u) = (x_1 \cdot u, \dots, x_r \cdot u)$. Then the property \textbf{(i)} becomes equivalent to the vector $\tilde{\lambda} = (\lambda_i - x_i \cdot u_0)_{i \in [r]}$ being in $\on{Im} \phi$. On the other hand, the set of all $\mu \in \mathbb{F}_p^r$ such that $\sum_{i \in [r]} \mu_i x_i \in U^\perp$ is precisely $\on{Im}\phi^\perp$. Thus, the property \textbf{(i)} becomes equivalent to the vector $\tilde{\lambda}$ belonging to $(\on{Im}\phi^\perp)^\perp$. The lemma follows from the fact that $(\on{Im}\phi^\perp)^\perp = \on{Im} \phi$.}\end{proof}

The following lemma controls the size of the intersection of a fixed dense set $S$ with random cosets of fixed dimension $r$.

\begin{lemma}[Random coset intersection lemma]\label{randomCosetInt}Let $G$ be a finite-dimensional vector space over $\mathbb{F}$, and let $S$ be a subset of size $\delta |G|$. Suppose that $x_0, x_1, \dots, x_r \in G$ are chosen uniformly and independently at random. Let $N = |\{\lambda \in \mathbb{F}^r \colon x_0 + \lambda \cdot x \in S\}|$. Then
\[\mathbb{E}N = \delta \mathbf{f}^r\]
and
\[\mathbb{P}\Big(|N - \mathbb{E} N| \leq \lambda \mathbb{E} N\Big) \geq 1 - \mathbf{f}^{-r}\lambda^{-2}\delta^{-1}.\]\end{lemma}

\begin{proof}A simple calculation gives
\[\mathbb{E} N = \ex \sum_{\lambda \in \mathbb{F}^r} \mathbbm{1}(x_0 + \lambda \cdot x \in S) = \sum_{\lambda \in \mathbb{F}^r} \mathbb{P}(x_0 + \lambda \cdot x \in S) = \mathbf{f}^r \delta.\]
We have an equally simple calculation for the second moment:
\begin{align*}\mathbb{E} N^2 &= \ex \sum_{\lambda,\mu \in \mathbb{F}^r} \mathbbm{1}(x_0 + \lambda \cdot x \in S)\mathbbm{1}(x_0 + \mu \cdot x \in S) = \sum_{\lambda \not= \mu} \mathbb{P}(x_0 + \lambda \cdot x, x_0 + \mu \cdot x \in S) + \sum_{\lambda} \mathbb{P}(x_0 + \lambda \cdot x \in S)\\
&\chg{=} (\mathbf{f}^{2r} - \mathbf{f}^r) \delta^2 + \mathbf{f}^r \delta.\end{align*}
Hence, $\operatorname{var} N \leq \mathbf{f}^r \delta$, which gives 
\begin{align*}\mathbb{P}\Big(|N - \mathbb{E} N| \leq  \lambda \mathbb{E} N\Big) &\geq 1 - \mathbb{P}\Big(|N - \mathbb{E} N|^2 \geq \lambda^2 (\mathbb{E} N)^2\Big) \geq 1 - \frac{\lambda^{-2} \operatorname{var} N}{(\mathbb{E} N)^2}\\
&\geq  1 - \mathbf{f}^{-r}\lambda^{-2}\delta^{-1},\end{align*}
as desired.\end{proof}

We say that a map $\psi \colon v_0 + V \to H$, where $V \leq G$ and $H$ are vector spaces, is \emph{affine} if there are a linear map $\tilde{\psi} \colon V \to H$ and a value $h_0 \in H$ such that $\psi(v_0 + x) = h_0 + \tilde{\psi}(x)$ holds for all $x \in V$. Note that being affine is equivalent to being a 2-homomorphism, since the domain is a coset in $G$. The next result tells us that a 2-homomorphism on a very dense subset of a coset $C$ necessarily extends to an affine map on the whole of $C$. 

\begin{lemma}\label{easyExtn}Let $p$ be a prime. Suppose that $V \leq G$ is a subspace of a vector space over $\mathbb{F}_p$. Let $v_0 \in G$, let $A \subset v_0 + V$ be a set of size greater than $\frac{4}{5} |v_0 + V|$ and let $\phi \colon A \to H$ be a map such that $\phi(a) + \phi(b) = \phi(c) + \phi(d)$ whenever $a,b,c,d \in A$ satisfy $a+b = c +d$. Then there is a unique affine map $\psi \colon v_0 + V \to H$ that extends $\phi$.\end{lemma}

\begin{proof}
    \chg{Consider the graph $\Gamma \subseteq G \times H$ of the given map, namely $\Gamma = \{(a, \phi(a)) : a \in A$. Since $\phi$ is a Freiman homomorphism, we have $|\Gamma + \Gamma| = |A + A| \leq |V|$. By Corollary 5.6 in~\cite{TaoVu}, there exists a coset $t + W$ of a subgroup $W \leq G \times H$ such that $|W| \leq \frac{3}{2}|\Gamma|$, and in fact $W = \Gamma - \Gamma$. Since the sizes of subgroups are powers of $p$, we must have $|W| = |V|$. It is not hard to see that $t + W$ corresponds to the graph of an affine map $\psi$ from $v_0 + V$ to $H$, where $v_0$ is the first half of $t$ and that $\psi$ extends $\phi$.}
\end{proof}

Next, we generalize the previous lemma to the case where the map is no longer a 2-homomorphism, but it respects the vast majority of additive structures. It is stated in a slightly more flexible form because this will be more convenient for the applications later in the paper.

\begin{lemma}\label{3approxHom}Let $p$ be a prime and let $\epsilon \in (0, \frac{1}{100})$. Suppose that $G$ and $H$ are finite-dimensional $\mathbb{F}_p$-vector spaces, and that $\rho_1, \rho_2, \rho_3 \colon G \to H$ are three maps such that $\rho_1(x_1) - \rho_2(x_2) = \rho_3(x_1 - x_2)$ for at least a $1 - \epsilon$ proportion of the pairs $(x_1, x_2) \in G \times G$. Then there is an affine map $\alpha \colon G \to H$ such that $\rho_3(x) = \alpha(x)$ for at least $(1-\sqrt{\epsilon})|G|$ of the $x \in G$.\end{lemma}

\begin{proof} Let $\Omega = \{(x_1, x_2) \in G \times G\colon \rho_1(x_1) - \rho_2(x_2) = \rho_3(x_1 - x_2)\}$. Call an element $y \in G$ a \emph{popular difference} if $y = x_1 - x_2$ for at least $(1- \sqrt{\epsilon}) |G|$ of $(x_1, x_2) \in \Omega$. Then there are at least $(1- \sqrt{\epsilon})|G|$ popular differences in $V$. We claim that $\rho_3$ is a 2-homomorphism on the set of popular differences $D$. For each $d \in D$, define the set $R(d) = \{x \in G\colon (x+d,x) \in \Omega\}$.\\
\indent Let $d_1, d_2, d_3, d_4 \in D$ be an additive quadruple: that is, a quadruple such that $d_4 - d_3 + d_2 - d_1 = 0$. Consider the set
\[\Big(\Big(R(d_1) \cap R(d_2)\Big) + d_2\Big) \cap \Big(\Big(R(d_3) \cap R(d_4)\Big) + d_3\Big),\]
which is non-empty, since $\sqrt{\epsilon} < 1/4$. Let $y$ be an arbitrary element of that set. Then $y - d_2 \in R(d_1) \cap R(d_2)$ and $y - d_3 \in R(d_3) \cap R(d_4)$, so we have $(y - d_2 + d_1, y - d_2), (y, y - d_2), (y, y- d_3), (y-d_3 + d_4, y-d_3) \in \Omega$. Hence,
\begin{align*}\rho_3(d_4) - \rho_3(d_3) &+ \rho_3(d_2) - \rho_3(d_1) \\
= &\Big(\rho_1(y-d_3 + d_4) - \rho_2(y- d_3)\Big) - \Big(\rho_1(y) - \rho_2(y- d_3)\Big)\\
&\hspace{1cm}+ \Big(\rho_1(y) - \rho_2(y- d_2)\Big) - \Big(\rho_1(y-d_2 + d_1) - \rho_2(y- d_2)\Big)\\
= & \Big(\rho_1(y-d_3 + d_4) - \rho_1(y) + \rho_1(y) - \rho_1(y-d_2 + d_1)\Big)\\
 &\hspace{1cm}-\Big(\rho_2(y- d_3) -\rho_2(y- d_3) + \rho_2(y- d_2) - \rho_2(y- d_2) \Big) = 0,\end{align*}
as desired. But $\rho_3$ is a 2-homomorphism on a subset of $V$ of size at least $(1-\sqrt{\epsilon})|G|$, so the claim follows from Lemma~\ref{easyExtn}.\end{proof}

In the context of respecting additive quadruples, we have the following corollary.

\begin{corollary}\label{approxF2homm}Let $p$ be a prime and let $\epsilon \in (0, \frac{1}{4 \cdot 10^4})$. Suppose that $V \leq G$ and $H$ are finite-dimensional $\mathbb{F}_p$-vector spaces and that $v_0 \in G$. Let $X \subset v_0 + V$ and let $\phi \colon X \to H$ be a map such that $\phi(a) + \phi(b) = \phi(c) + \phi(a + b - c)$ holds for at least $(1-\varepsilon) |V|^3$ of $(a,b,c) \in X^3$. Then there is an affine map $\alpha \colon v_0 + V \to H$ such that $\phi(x) = \alpha(x)$ for at least $(1-5\sqrt[4]{\epsilon})|V|$ of $x \in X$.\end{corollary}

\begin{proof}Note that the assumptions tacitly imply that $|X| \geq (1- \varepsilon)|V|$. Extend $\phi$ arbitrarily to $v_0 + V$ for technical reasons; we shall remove these additional elements at the end of the proof. We have $\phi(a) - \phi(c) = \phi(a + b - c) - \phi(b)$ for at least $(1-\varepsilon)|V|^3$ choices of $(a,b,c)$. Making a change of variables, we obtain
\[\phi(x + d) - \phi(x) = \phi(y + d) - \phi(y)\]
for at least $(1-\varepsilon)|V|^3$ choices of $(x,y,d) \in (v_0 + V) \times (v_0 + V) \times V$. Define $\rho_1 = \rho_2 \colon V \to H$ to be the map $\rho_1(u) = \phi(v_0 + u)$. For each $d \in V$ let $\rho_3$ be the most frequent value of $\phi(x + d) - \phi(x)$ as $x$ ranges over $v_0 + V$ (if there is a tie, choose an arbitrary winner). By averaging, for at least $(1-\sqrt{\varepsilon}) |V|$ elements $d \in V$ we have $\phi(x + d) - \phi(x) = \phi(y + d) - \phi(y)$ for at least $(1-\sqrt{\varepsilon})|V|^2$ elements $(x,y) \in (v_0 + V)^2$. Let $D$ be the set of all such $d \in V$. In particular, when $d \in D$, we obtain $\phi(x + d) - \phi(x) = \rho_3(d)$ for at least $(1-\sqrt{\varepsilon})|V|$ of $x \in v_0 + V$. Replacing $\phi$ with $\rho_1$ and $\rho_2$, we conclude that 
\[\rho_1(u) - \rho_2(v) = \rho_3(u-v)\]
for at least $(1-2\sqrt{\varepsilon})|V|^2$ of $(u,v) \in V^2$. By Lemma~\ref{3approxHom}, there is an affine map $\alpha \colon V \to H$ such that $\rho_3(u) = \alpha(u)$ for at least $(1-2\sqrt[4]{\epsilon})|V|$ of $u \in V$. By restricting our attention only to $d \in D$, we get that $\phi(x + d) - \phi(x) = \alpha(d)$ for at least $(1-4\sqrt[4]{\varepsilon})|V|^2$ of $(x,d) \in (v_0 + V) \times V$. Average over $x$ and ignore the elements of $(v_0 + V) \setminus X$ to finish the proof.\end{proof}

We also need a combination of Lemma~\ref{easyExtn} and Corollary~\ref{approxF2homm}, in the case where the domain of the map $\phi$ is a dense subset of the coset but the number of additive quadruples respected by $\phi$ is significantly higher than expected. This stronger assumption allows us to remove fewer points from the domain and still get a restriction of an affine map.

\begin{corollary}\label{easyExtnDoubleApprox}There is an absolute constant $\varepsilon_0 > 0$ such that the following holds. Let $p$ be a prime, let $\epsilon \in (0, \varepsilon_0)$ and let $\eta > 0$. Suppose that $V \leq G$ and $H$ are finite-dimensional $\mathbb{F}_p$-vector spaces, that $v_0 \in G$, and that $X \subset v_0 + V$ is a set of size at least $(1-\varepsilon)|V|$. Let $\phi \colon X \to H$ be a map such that the number of quadruples $a,b,c,d\in X$ with $\phi(a) + \phi(b) \not= \phi(c) + \phi(d)$ and $a + b = c + d$ is at most $\eta |V|^3$. Then there are a subset $X' \subset X$ and an affine map $\psi \colon v_0 + V \to H$ such that $|X \setminus X'| \leq O(\eta^{1/4})|V|$ and $\psi(x) = \phi(x)$ for all $x \in X'$.\end{corollary}

\begin{proof}We say that a pair $(a,b) \in X^2$ is \emph{good} if there are at most $\sqrt{\eta}|V|$ elements $c \in X$ such that $a - b + c \in X$ and $\phi(a) - \phi(b) \not= \phi(a-b+c) - \phi(c)$. Otherwise, the pair is \emph{bad}. The number of bad pairs in $X$ is at most $2 \sqrt{\eta} |V|^2$.\\
\indent We now show that there are at most $O(\sqrt{\eta})|V|^5$ sextuples $(a,b,c,d,e,f) \in X^6$ such that $a - b + c = d - e + f$, but $\phi(a) - \phi(b) + \phi(c) \not= \phi(d) - \phi(e) + \phi(f)$. There are at most $O(\sqrt{\eta})|V|^5$ such sextuples where additionally $(a,b)$ or $(d,e)$ is a bad pair, so without loss of generality we may assume that both these pairs are good. Fix such a sextuple $(a,b,c,d,e,f)$. There are at least $|X \cap (X - a + b) \cap (X - d + e)| \geq 3|X| - 2|V| \geq (1-3\varepsilon)|V|$ elements $x \in X$ such that $a - b + x, d - e + x \in X$. Hence, for $(1-3\varepsilon - 2 \sqrt{\eta})|V|$ elements $x \in X$, we have $a - b + x, d - e + x \in X$, $\phi(a) - \phi(b) = \phi(a - b + x) - \phi(x)$ and $\phi(d) - \phi(e) = \phi(d - e + x) - \phi(x)$. Therefore, provided $\varepsilon, \eta \leq \frac{1}{100}$, each such sextuple gives rise to at least $\frac{1}{2}|V|$ elements $x \in X$ such that $a - b + x, d - e + x \in X$ and
\[\phi(a - b + x) - \phi(x) + \phi(c) = \phi(a) - \phi(b) + \phi(c) \not= \phi(d) - \phi(e) + \phi(f) =\phi(d - e + x) - \phi(x) + \phi(f),\]
from which it follows that
\[\phi(a - b + x) + \phi(c) \not= \phi(d - e + x) + \phi(f).\]
Let $N_6$ be the number of the considered sextuples, i.e.\ all six elements belong to $X^6$, they are additive and the pairs $(a,b)$ and $(d,e)$ are good pairs. Define $\Omega$ to be the set
\begin{align*}\Big\{(a,b,c,d,e,f,x) \in (v_0 + V)^7 \colon\, &a - b + x, c, d - e + x, f \in X,\,\, a - b + c = d - e + f\\
&\hspace{3cm}\phi(a - b + x) + \phi(c) \not= \phi(d - e + x) + \phi(f)\Big\}.\end{align*}
The argument above shows that $|\Omega| \geq \frac{1}{2} |V| N_6$. On the other hand, every additive quadruple in $X$ not respected by $\phi$ contributes $|V|^3$ to $|\Omega|$. Double-counting proves that $N_6 \leq O(\eta|V|^5)$ and the claimed upper bound on additive sextuples in $X$ not respected by $\phi$ follows.\\

Now define a map $\tilde{\phi}$ as follows. For fixed $x \in v_0 + V$, if there is a value $h \in H$ such that $\phi(a) - \phi(b) + \phi(c) = h$ for all but at most $\sqrt[4]{\eta}|V|^2$ triples $(a,b,c) \in X^3$ such that $a - b + c = x$, set $\tilde{\phi}(x) = h$. Let $\tilde{X}$ be the set of all such $x \in v_0 + V$. Thus, $\tilde{\phi}$ is a map from $\tilde{X}$ to $H$. Note that $|(v_0 + V) \setminus \tilde{X}| = O(\eta^{1/4} |V|)$.\\
\indent We claim that $\tilde{\phi}$ respects all but $O(\eta|V|^3)$ additive quadruples in $\tilde{X}$. Consider any $x_1, x_2, x_3, x_4 \in \tilde{X}$ such that $x_1 + x_2 = x_3 + x_4$ but $\tilde{\phi}(x_1) + \tilde{\phi}(x_2) \not= \tilde{\phi}(x_3) + \tilde{\phi}(x_4)$. For each $a \in X$ consider the set $X \cap \Big(\bigcap_{i=1}^4 X  - x_i + a\Big)$. The size of this intersection is at least $(1 - 5\varepsilon)|V|$, so there are at least $\frac{1}{2}|V|^2$ pairs $a,b \in X$ such that $c_i = x_i - a + b \in X$ for $i=1,2,3,4$. By definition of $\tilde{X}$, this means that for at least $\Big(\frac{1}{2} - 4\sqrt[4]{\eta}\Big)|V|^2$ pairs $a,b \in X$, we additionally have $\tilde{\phi}(x_i) = \phi(a) - \phi(b) + \phi(c_i)$ for each $i$. Thus, 
\begin{align*}\phi(c_1) + \phi(c_2) - \phi(c_3) - \phi(c_4) = &\Big(\tilde{\phi}(x_1) - \phi(a) + \phi(b)\Big) + \Big(\tilde{\phi}(x_2) - \phi(a) + \phi(b)\Big) \\
&\hspace{2cm}- \Big(\tilde{\phi}(x_3) - \phi(a) + \phi(b)\Big) - \Big(\tilde{\phi}(x_4) - \phi(a) + \phi(b)\Big) \not= 0.\end{align*}
We conclude that for each additive quadruple in $\tilde{X}$ not respected by $\tilde{\phi}$, we may find at least $\frac{1}{3}|V|$ quadruples $(c_1, c_2, c_3, c_4) \in X^4$ such that there is some $d \in V$ with $c_i = x_i + d$ for each $i$. The claimed upper bound on the number of non-respected additive quadruples in $\tilde{X}$ now follows from double-counting.\\
\indent We may now apply Corollary~\ref{approxF2homm} to find an affine map $\psi \colon v_0 + V \to H$ such that $\phi(x) = \tilde{\psi}(x)$ for all but $O(\sqrt[4]{\eta}|V|)$ elements $x \in \tilde{X}$. Hence, 
\begin{equation} \label{goodTriplesEqnDoubleApp}
\phi(a) - \phi(b) + \phi(c) = \psi(a-b+c)
\end{equation}
for all but at most $O(\eta^{1/4}|V|^3)$ triples $(a,b,c) \in X^3$.\\
\indent Finally, define $X' \subset X$ as the set of all $x \in X$ such that for all but at most $\sqrt{\eta}|V|^2$ of choices of $a,b \in X$ such that $a + b - x \in X$, we have $\phi(x) = \phi(a) + \phi(b) - \phi(a +b - x)$. Thus, $|X \setminus X'| \leq O(\sqrt{\eta}|V|)$. If $\phi(x) \not= \psi(x)$ for $x \in X'$, then, we get at least $\frac{1}{2}|V|^2$ of $(a,b) \in X^2$ such that $a + b - x \in X$ and
\[\psi(x) \not= \phi(x) = \phi(a) - \phi(a + b - x) + \phi(b).\]
By~\eqref{goodTriplesEqnDoubleApp}, we see that $\phi(x) \not= \psi(x)$ may happen for at most $O(\eta^{1/4}|V|)$ elements $x \in X'$, which completes the proof.\end{proof}

We also need the combination of Freiman's theorem and the Balog--Szemer\'edi--Gowers theorem that we mentioned in the introduction (Theorem~\ref{introFreAndBSG}). Using Sanders's bound for the Bogolyubov-Ruzsa lemma~\cite{Sanders}, it takes the following form.

\begin{theorem}\label{FreBSG}Let $p$ be a prime and let $G$ and $H$ be finite-dimensional vector spaces over $\mathbb{F}_p$. Let $A \subset G$ and let $\psi \colon A \to H$ be a map that respects at least $c |G|^3$ additive quadruples -- that is, there are at least $c|G|^3$ choices of $(x_1, x_2, x_3, x_4) \in A^4$ such that $x_1 + x_2 = x_3 + x_4$ and $\phi(x_1) + \phi(x_2) = \phi(x_3) + \phi(x_4)$. Then there is an affine map $\alpha \colon G \to H$ such that $\alpha(x) = \psi(x)$ for $\exp(-(\log(c^{-1}))^{O(1)})|G|$ values $x \in G$.\end{theorem}

\subsection{Approximating multiaffine varieties}\label{appmvsubseclabel}

\noindent As above, let $\mathbb{F}$ be a finite field of size $\mathbf{f}$. Let $\mathcal{G} \subset \mathcal{P}([k])$ be a collection of sets. We say that $\mathcal{G}$ is a \emph{down-set} if it is closed under taking subsets. We also say that a multiaffine map $\alpha \colon G_{[k]} \to \mathbb{F}^r$ is \emph{$\mathcal{G}$-supported} if it can be written in the form $\alpha(x_{[k]}) = \sum_{I \in \mathcal{G}} \alpha_I(x_I)$ for some multilinear maps $\alpha_I \colon G_I \to \mathbb{F}^r$. We say that a variety is \emph{$\mathcal{G}$-supported} if its defining map is $\mathcal{G}$-supported. For example, any multiaffine form on $G_{[k]}$ is $\mathcal{P}([k])$-supported, any constant map on $G_{[k]}$ is $\{\emptyset\}$-supported and, as a concrete example, the form $\phi \colon G_1 \times G_2 \to \mathbb{F}$ given by $\phi(x_1, x_2) = (x_{1})_i + (x_{2})_i$ for any coordinate $i$ is $\{\emptyset, \{1\}, \{2\}\}$-supported.\\

\noindent\textbf{Remark.} In this subsection, the explicit constants and the implicit constants in big-O notation depend on $k$ only, and no longer on the field $\mathbb{F}$.\\

The most basic fact about varieties is that their codimension gives an easy lower bound on their size.

\begin{lemma}[Lemma 11 in~\cite{LukaRank}]\label{varsizelemma}Suppose that $B \subset G_{[k]}$ is a non-empty variety of codimension at most $r$. Then $|B| \geq \mathbf{f}^{-kr} |G_{[k]}|$.\end{lemma}

\indent We recall the following results from~\cite{LukaRank}. The first states that a variety can always be approximated from the outside by a variety of low codimension. (This statement is mainly interesting when the given variety is dense, since otherwise any sufficiently small low-rank variety containing it will work.) The second is a generalization that states that a collection of varieties defined by multilinear maps that belong to a low-dimensional subspace can be simultaneously approximated from the outside by a similar collection where the varieties all have low codimension. 

\begin{lemma}[Approximating dense varieties externally, Lemma 12 in~\cite{LukaRank}]\label{varOuterApprox} Let $A \colon G_{[k]} \to H$ be a multiaffine map. Then for every positive integer $s$ there is a multiaffine map $\phi \colon G_{[k]} \to \mathbb{F}^s$ such that $A^{-1}(0) \subset \phi^{-1}(0)$ and $|\phi^{-1}(0) \setminus A^{-1}(0)| \leq \mathbf{f}^{-s}|G_{[k]}|$. If, additionally, $A$ is linear in coordinate $c$, then so is $\phi$. Moreover, if $\mathcal{G} \subset \mathcal{P}([k])$ is a down-set and if $A$ is $\mathcal{G}$-supported, then $\phi$ is also $\mathcal{G}$-supported.\footnote{The claim that $\phi$ is $\mathcal{G}$-supported does not appear in~\cite{LukaRank}, but follows from the proof given in that paper, since for each $i$, $\phi_i(x_{[k]}) = A(x_{[k]}) \cdot h_i$ for some $h_i \in H$.}\end{lemma}

\begin{lemma}[Approximating dense varieties externally simultaneously, Lemma 13 in~\cite{LukaRank}]\label{BohrApproxSim}Let $A_1, \dots, A_r \colon G_{[k]} \to H$ be multiaffine maps. Let $\epsilon > 0$. Then there exist a positive integer $s \leq r + \log_{\mathbf{f}} \epsilon^{-1}$ and multiaffine maps $\phi_1, \dots, \phi_r \colon G_{[k]} \to \mathbb{F}^s$ such that for each $\lambda \in \mathbb{F}^r$ we have $(\lambda \cdot A)^{-1}(0) \subset (\lambda \cdot \phi)^{-1}(0)$ and $|(\lambda \cdot \phi)^{-1}(0)\setminus(\lambda \cdot A)^{-1}(0)| \leq \epsilon |G_{[k]}|$. If additionally each map $A_i$ is linear in coordinate $c$, then so are the maps $\phi_i$. Moreover, if $\mathcal{G} \subset \mathcal{P}([k])$ is a down-set and if each $A_i$ is $\mathcal{G}$-supported, then so is each $\phi_i$.\footnote{Again, the claim that the $\phi_i$ are $\mathcal{G}$-supported does not appear in~\cite{LukaRank}, but follows easily from the proof.}\end{lemma}

\noindent Recall that the notation $\lambda \cdot \phi$ appearing in the lemma denotes the multilinear map $\sum_{i \in [r]} \lambda_i \phi_i$, and similarly for $\lambda \cdot A$.\\
\indent We also recall the following definitions from~\cite{LukaRank}. Let $S \subset G_{[k]}$ and let $\alpha \colon G_{[k]} \to H$ be a multiaffine map. A \emph{layer} of $\alpha$ is any set of the form $\{x_{[k]} \in G_{[k]} \colon \alpha(x_{[k]}) = \lambda\}$, for $\lambda \in H$. We say that layers of $\alpha$ \emph{internally $\epsilon$-approximate }$S$, if there are layers $L_1, \dots, L_m$ of $\alpha$ such that $S \supset L_i$ and $\Big|S \setminus \Big(\bigcup_{i \in [m]} L_i\Big)\Big| \leq \epsilon |G_{[k]}|$. Similarly, we say that layers of $\alpha$ \emph{externally $\epsilon$-approximate }$S$, if there are layers $L_1, \dots, L_m$ of $\alpha$ such that $S \subset \bigcup_{i \in [m]} L_i$ and $\Big|\Big(\bigcup_{i \in [m]} L_i\Big)  \setminus S\Big| \leq \epsilon |G_{[k]}|$.

In the next lemma, $\chi$ is an arbitrary non-trivial additive character on $\mathbb{F}$.

\begin{theorem}[Strong inverse theorem for maps of low analytic rank, Theorem 6 in~\cite{LukaRank}]\label{strongInvARankThm}For every positive integer $k$ there are constants $C = C_k, D = D_k > 0$ with the following property. Suppose that $\alpha \colon G_{[k]} \to \mathbb{F}$ is a multilinear form such that $\ex_{x_{[k]}} \chi(\alpha(x_{[k]})) \geq c$, for some $c > 0$. Then there exist a positive integer $r \leq C \log_{\mathbf{f}}^D (\mathbf{f}c^{-1})$ and multilinear maps $\beta_i \colon G_{I_i} \to \mathbb{F}$ and $\gamma_i \colon G_{[k] \setminus I_i} \to \mathbb{F}$ with $\emptyset \not= I_i \subset [k-1]$ for each $i\in[r]$, such that 
\[\alpha(x_{[k]}) = \sum_{i \in [r]} \beta_i(x_{I_i}) \gamma_i(x_{[k] \setminus I_i}).\]
for every $x_{[k]} \in G_{[k]}$.
\end{theorem}

\noindent \textbf{Remark.} In a qualitative sense, this theorem was first proved by Bhowmick and Lovett in~\cite{BhowLov}, generalizing an approach of Green and Tao~\cite{GreenTaoPolys}. An almost identical result\footnote{There is a slight difference in bounds: in~\cite{Janzer2} the constant $C$ depends on the field as well.} to the one stated here was obtained independently by Janzer in~\cite{Janzer2} (who had previously obtained tower-type bounds in this problem~\cite{Janzer1}).\\

The least number $r$ such that $\alpha$ can be expressed in terms of $r$ pairs of forms $(\beta_i, \gamma_i)$ as above is called \emph{the partition rank} of $\alpha$, and is denoted $\prank \alpha$. This notion was introduced by Naslund in~\cite{Naslund}. Also, the quantity $\ex_{x_{[k]}} \chi(\alpha(x_{[k]}))$ is called the \emph{bias} of $\alpha$, written $\bias \alpha$. The quantity $-\log_{|\mathbb{F}|} \bias \alpha$ was called the \emph{analytic rank} of $\alpha$ by the first author and J. Wolf, who showed that it has useful properties~\cite{TimWolf}. Thus, high bias, or equivalently low analytic rank, implies low partition rank. \chg{We remark that the bias does not depend on the character $\chi$.}\\

The next result that we recall from~\cite{LukaRank} says that we may approximate varieties both internally and externally using low-codimensional varieties.

\begin{theorem}[Simultaneous approximation of varieties, Theorem 7 in~\cite{LukaRank}]\label{simVarAppThm}For every positive integer $k$, there are constants $C=C_k, D=D_k > 0$ with the following property. Let $\epsilon > 0$ and let $B_1, \dots, B_r \colon G_{[k]} \to H$ be multiaffine maps. For each $\lambda \in \mathbb{F}^r$, let $Z_\lambda = \{x_{[k]} \in G_{[k]} \colon \sum_{i \in [r]} \lambda_i B_i(x_{[k]}) = 0\}$. Then there exist a positive integer $s \leq C \Big(r \log_{|\mathbb{F}|}(|\mathbb{F}|\epsilon^{-1}) \Big)^D$ and a multiaffine map $\beta \colon G_{[k]} \to \mathbb{F}^s$ such that for each $\lambda \in \mathbb{F}^r$, the layers of $\beta$ internally and externally $\epsilon$-approximate $Z_\lambda$.\\
\indent Moreover, if $\mathcal{G}$ is a down-set and the maps $B_1, \dots, B_r$ are $\mathcal{G}$-supported, then so is the map $\beta$.\footnote{This follows from the proof in~\cite{LukaRank}. One can generalize Proposition 20 in that paper by changing $\mathcal{P}[k-1]$ to the down-set $\mathcal{G}$ and adding the property that the map $\beta$ is $\mathcal{G}$-supported. The induction base is again trivial and the new maps come from inverse theorem for maps of high bias, which keeps maps $\mathcal{G}$-supported.}\end{theorem}

\noindent\textbf{Remark.} The original statement only has the internal approximation part of the claim. The external approximation is easily obtained using Lemma~\ref{BohrApproxSim}. \chg{The conclusion in the above version of the theorem is deduced by joining the two multiaffine maps provided by the internal and external approximation results together. In general, for multiaffine maps $\alpha : G_{[k]} \to \mathbb{F}^r$ and $\beta : G_{[k]} \to \mathbb{F}^s$, we frequently consider their join $(\alpha, \beta)$, which is itself a multiaffine map to $\mathbb{F}^{r + s}$. Thinking in terms of partitions, collections of layers $\{\alpha = \lambda\}$ and $\{\beta = \mu\}$ have a common refinement coming from the layers of $(\alpha, \beta)$.}\\

\chg{Using Theorem~\ref{simVarAppThm} we may control the vanishing linear combinations of a collection of multiaffine maps.}

\begin{theorem}\label{newthmvanishinglcs}
    \chg{Let $\Phi_1, \dots, \Phi_r : G_{[k]} \to H$ be multiaffine maps. Then there exist a positive integer $t \leq C (r \log_p\xi^{-1})^D$, a multiaffine map $\beta : G_{[k]} \to \mathbb{F}_p^t$ and a union $U$ of layers of $\beta$ such that $|U| \geq (1 - \xi) |G_{[k]}|$, and for each layer $L$ of $\beta$ inside $U$ there exists a subspace $\Lambda \leq \mathbb{F}_p^r$ such that 
    \[\{\lambda \in \mathbb{F}_p^r \colon \lambda \cdot \Phi(x_{[k]}) = 0\} = \Lambda\]
    for every $x_{[k]} \in L$.}
\end{theorem}

\begin{proof}
\chg{Let $\xi > 0$. Apply Theorem~\ref{simVarAppThm} to $\Phi_1, \dots, \Phi_l$, to obtain a positive integer $t = O\Big((r \log_p\xi^{-1})^{O(1)}\Big)$ and a multiaffine map $\beta \colon G_{[k]} \to \mathbb{F}_p^t$ such that for each $\lambda \in \mathbb{F}_p^r$, the set $Z_\lambda = \{x_{[k]} \in G_{[k]} \colon \sum_{i \in [r]} \lambda_i \Phi_i(x_{[k]}) = 0\}$ is internally and externally $(\xi p^{-r-r^2})$-approximated by the layers of $\beta$. For $\Lambda \leq \mathbb{F}_p^r$, consider the set $\Big(\bigcap_{\lambda \in \Lambda} Z_\lambda\Big) \setminus \Big(\bigcup_{\lambda \notin \Lambda} Z_\lambda\Big)$. Approximate the sets $Z_\lambda$ in the intersection internally and the sets $Z_\lambda$ in the union externally by collections of layers $\mathcal{L}_1$ and $\mathcal{L}_2$ of $\beta$, respectively. The union of layers in $\mathcal{L}_1 \setminus \mathcal{L}_2$ thus internally $(\xi p^{-r^2})$-approximates the given set. Since the number of subspaces $\Lambda \leq \mathbb{F}_p^r$ is at most $p^{r^2}$, we deduce that there is a union $U$ of layers of $\beta$ of size
\begin{equation}|U| \geq (1-\xi) |G_{[k]}|\label{uSizeBound}\end{equation}
such that for each layer $L_\mu = \{x_{[k]} \in G_{[k]}:\beta(x_{[k]}) = \mu\} \subset U$ there is a subspace $\Lambda_\mu \leq \mathbb{F}_p^r$ such that
\[\{\lambda \in \mathbb{F}_p^r \colon \lambda \cdot \Phi(x_{[k]}) = 0\} = \Lambda_\mu\]
for every $x_{[k]} \in L_\mu$.
}
\end{proof}

Let $V \subset G_{[k]}$ be a variety. The next theorem shows that the set of points $x_{[k-1]} \in G_{[k-1]}$ such that $V_{x_{[k-1]}}$ is dense in $G_k$ is almost a variety as well.

\begin{theorem}[Structure of a set of dense columns of a variety, Theorem 8 in~\cite{LukaRank}]\label{denseColumnsThm}For every positive integer $k$ there are constants $C=C_k, D=D_k > 0$ with the following property. Let $\alpha\colon G_{[k]} \to \mathbb{F}^r$ be a multiaffine map. Let $S \subset \mathbb{F}^r$ and $\epsilon > 0$. Let $X$ be the set of $\epsilon$-dense columns: that is, 
\[\{x_{[k-1]} \in G_{[k-1]} \colon |\{y \in G_k \colon \alpha(x_{[k-1]}, y) \in S\}| \geq \epsilon |G_k|\}.\]
Then there exist a positive integer $s \leq C \Big(r \log_{|\mathbb{F}|} (|\mathbb{F}|\epsilon^{-1})\Big)^D$ and a multiaffine map $\beta \colon G_{[k-1]} \to \mathbb{F}^s$ such that the layers of $\beta$ $\epsilon$-internally and $\epsilon$-externally approximate $X$.\\
\indent Moreover, if $\mathcal{G}$ is a down-set and $\alpha$ is $\mathcal{G}$-supported, then we may take $\beta$ to be $\mathcal{G}'$-supported, where $\mathcal{G}' = \mathcal{G} \cap \mathcal{P}([k-1])$.\footnote{Once again, the proof in~\cite{LukaRank} can be straightforwardly modified to give this slightly stronger version. Theorem 7 and Lemma 13 of that paper (which are combined together to Theorem~\ref{simVarAppThm} of this paper) are used to define the desired map, and both respect the notion of $\mathcal{G}'$-supported maps. Note that $\alpha_i(x_{[k]}) = \alpha'_i(x_{[k-1]}) + A_i(x_{[k-1]}) \cdot x_k$ implies that $\alpha'_i$ and $A_i$ are $\mathcal{G}'$-supported. (We also note that Lemma 13 was referred to as Proposition 13 in the above-mentioned proof in~\cite{LukaRank}.)}\end{theorem}

We say that a multiaffine variety $V \subset G_{[k]}$ is \emph{multilinear} if it is of the form $V = \bigcap_{i \in [s]} \{x_{[k]} \in G_{[k]} \colon \alpha_i(x_{I_i}) = 0\}$ for some multilinear maps $\alpha_i \colon G_{I_i} \to \mathbb{F}^{r_i}$ -- that is, if the multiaffine maps used to define it are in fact multilinear \chg{on the appropriate subset of $[k]$}. We now deduce that dense multilinear varieties necessarily contain multilinear varieties of low codimension. 

\chg{For the next lemma, we introduce the notion of a multilinear set. A subset $M \subset G_{[k]}$ is a \emph{multilinear set} if, for every direction $d \in [k]$ and every $x_{[k]\setminus \{d\}} \in G_{[k] \setminus \{d\}}$, the set $M_{x_{[k] \setminus \{d\}}}$ is a (possibly empty) subspace of $G_d$.}

\begin{lemma}\label{mlSetVarML}Let $M \subset G_{[k]}$ be a \emph{multilinear set} and let $B \subset M$ be a non-empty variety of codimension at most $s$. Then $M$ contains a multilinear variety of codimension $O(s)$.\end{lemma}

\begin{proof} Splitting the multiaffine map $\beta \colon G_{[k]} \to \mathbb{F}^s$ that defines $B$ into its multilinear pieces, we get multilinear maps $\beta_I \colon G_I \to \mathbb{F}^s$ and \chg{we pick} values $\lambda_I \in \mathbb{F}^s$, $\emptyset \not= I \subseteq [k]$ such that the variety $B^0$ defined as
\[B^0 = \bigcap_{\emptyset \not= I \subseteq [k]}\{x_{[k]} \in G_{[k]} \colon \beta_I(x_I) = \lambda_I\}\]
satisfies $\emptyset \not= B^0 \subset M$.\\
\indent We shall show that for each $d \in [0, k]$ there are a positive integer $s_d = O(s)$, multilinear maps $\beta^d_I \colon G_I \to \mathbb{F}^{s_d}$ and values $\lambda^d_I \in \mathbb{F}^{s_d}$, $\emptyset \not= I \subseteq [k]$, such that $\lambda^d_I = 0$ when $I \cap [d] \not= \emptyset$ and the variety $B^d$ defined as
\[B^d = \bigcap_{\emptyset \not= I \subseteq [k]}\{x_{[k]} \in G_{[k]} \colon \beta^d_I(x_I) = \lambda^d_I\}\]
satisfies $\emptyset \not= B^d \subset M$.\\
\indent We prove this by induction on $d$, taking $d = 0$ as the base case, which plainly holds. Suppose that the statement holds for some $d \in [0, k-1]$. Let $y_{[k]} \in B^d$ be an arbitrary element. Consider the set
\begin{align*}B^{d+1} = \Big(\bigcap_{\substack{\emptyset \not= I \subseteq [k]\\d + 1 \in I}}\{x_{[k]} \in G_{[k]} \colon \beta^d_I(x_I) = 0\}\Big)& \cap \Big(\bigcap_{\substack{\emptyset \not= I \subseteq [k]\\d + 1 \notin I}}\{x_{[k]} \in G_{[k]} \colon \beta^d_I(x_I) = \lambda^d_I\}\Big)\\
 &\cap \Big(\bigcap_{\substack{\emptyset \not= I \subseteq [k]\\d + 1 \in I}}\{x_{[k]} \in G_{[k]} \colon \beta^d_I(x_{I \setminus \{d+1\}}, y_{d+1}) = \lambda^d_I\}\Big).\end{align*}
Note that $B^{d+1}$ is non-empty since $(y_{[k] \setminus \{d+1\}}, {}^{d+1}\,0) \in B^{d+1}$ (recall that the notation ${}^{d+1}\,0$ means that we put zero at coordinate $d+1$). It remains to check that $B^{d+1} \subseteq M$: other properties are evident. Let $x_{[k]} \in B^{d+1}$. When $d+1 \in I$, we have $\beta^d_I(x_I) = 0$ and $\beta^d_I(x_{I \setminus \{d+1\}}, y_{d+1}) = \lambda^d_I$, and hence that $\beta^d_I(x_{I \setminus \{d+1\}}, x_{d+1} + y_{d+1}) = \lambda^d_I$. On the other hand, when $d + 1 \notin I$, we simply have $\beta^d_I(x_I) = \lambda^d_I$. We conclude that $(x_{I \setminus \{d+1\}}, y_{d+1})$ and $(x_{I \setminus \{d+1\}}, x_{d+1} + y_{d+1})$ belong to $M$. Since $M$ is a multilinear set, it contains $x_{[k]}$, as desired.\end{proof}

\begin{corollary}\label{denseToLowCodimVar}Let $V \subset G_{[k]}$ be a variety of density $\delta > 0$. Then $V$ contains a non-empty variety $B$ of codimension $O\Big(\log^{O(1)}_{\mathbf{f}} (\mathbf{f} \delta^{-1})\Big)$. Moreover, if $V$ is multilinear, then we may take $B$ to be multilinear as well.\end{corollary}

\begin{proof}Let $V=\{x \in G_{[k]}:\alpha(x) = 0\}$ for some multiaffine map $\alpha \colon G_{[k]} \to H$. Apply Theorem~\ref{simVarAppThm} (with $r = 1$) to find a non-empty variety of codimension $O\Big(\log^{O(1)}_{\mathbf{f}} (\mathbf{f} \delta^{-1})\Big)$ inside $V$.\\
\indent If $V$ is in fact a multilinear variety, then it is also a multilinear set. The additional conclusion follows from Lemma~\ref{mlSetVarML}.\end{proof}

In the rest of this subsection, we slightly strengthen the stated results in a straightforward manner. For maps $f,g \colon X \to \mathbb{C}$, we shall write 
\[f\apps{\varepsilon}_{L^p} g\]
as shorthand for $\|f - g\|_{L^p} \leq \varepsilon$. It will also sometimes be convenient to have a similar notation when we have two complicated expressions that both define functions of a variable such as $x_{[k]}$ and we want to say that the functions are close in $L_p$. In such a case we will write
\[f(x_{[k]})\apps{\varepsilon}_{L^p,x_{[k]}} g(x_{[k]})\]
The variable $x_{[k]}$ should be understood as a dummy variable in this notation, so the left-hand side is referring not to the complex number $f(x_{[k]})$ but to the function $x_{[k]}\mapsto f(x_{[k]})$.\\

\chg{The next proposition tells us that functions that are constant along layers of multiaffine maps are well-behaved when convolved in a principal direction.}

\begin{proposition}\label{ExpSumConvComp}\chg{Let $\phi: G_{[k]} \to \mathbb{F}_p^r$ be a multiaffine map and let $c : \mathbb{F}_p^r \to \mathbb{D}$. Let $\varepsilon > 0$ and $d \in [k]$. Then we may find a positive integer $\ell = O(r \log_p\varepsilon^{-1})^{O(1)}$, a multiaffine map $\tilde{\phi}: G_{[k]} \to \mathbb{F}_p^\ell$ and a function $\tilde{c} : \mathbb{F}_p^\ell \to \mathbb{D}$ such that the convolution of $c \circ \phi$ in direction $d$ satisfies
\[\bigconv{d}(c \circ \phi) \apps{\varepsilon}_{L^2} \tilde{c}\circ \tilde{\phi}.\]
Moreover, if $\mathcal{G}$ is a down-set and the forms $\phi_i$ are $\mathcal{G}$-supported, then so are the forms $\tilde{\phi}_i$.}
\end{proposition}

\begin{proof}
    Without loss of generality, $d = k$. Write each form $\phi_i$ as $\phi_i(x_{[k]}) = \Phi_i(x_{[k-1]}) \cdot x_k + \phi'_i(x_{[k-1]})$ for multiaffine maps $\Phi_i : G_{[k-1]} \to G_k$ and $\phi'_i : G_{[k-1]} \to \mathbb{F}_p$. Apply Proposition~\ref{newthmvanishinglcs} with error parameter $\varepsilon^2$ to the maps $\Phi_1, \dots, \Phi_r$ to find a positive integer $t \leq O(r \log_p\varepsilon^{-1})^{O(1)}$, a multiaffine map $\beta : G_{[k - 1]} \to \mathbb{F}_p^t$ and a union $U$ of layers of $\beta$ such that $|U| \geq (1 - \varepsilon^2) |G_{[k - 1]}|$ and the vanishing linear combinations of $\Phi_1(x_{[k-1]}), \dots, \Phi_r(x_{[k-1]})$ are the same for all points $x_{[k-1]}$ inside the same layer inside $U$.\\

    Let $\lambda \in \mathbb{F}_p^r, \mu \in \mathbb{F}_p^r, \nu \in \mathbb{F}_p^t$ be values of $\phi, \phi'$ and $\beta$ and assume that the layer $\beta^{-1}$ is inside $U$. Let $W \leq \mathbb{F}_p^r$ be the subspace of the vanishing linear combinations of $\Phi_i$ for the layer $\beta^{-1}(\nu)$. We claim that $\bigconv{d}(c \circ \phi)(x_{[k]})$ is the same for all $x_{[k]} \in \phi^{-1}(\lambda) \cap ({\phi'}^{-1}(\mu) \times G_k) \cap ({\beta}^{-1}(\nu) \times G_k)$. We shall denote this value $\tilde{c}(\lambda, \mu, \nu)$.\\

    To prove the claim, note that 
    \begin{align*}\bigconv{d}(c \circ \phi)(x_{[k]}) = &\exx_{y_k} c(\phi(x_{[k-1]}, y_k + x_k))\overline{c(\phi(x_{[k-1]}, y_k))}\\
    =& \exx_{y_k} c(\Phi(x_{[k-1]})\cdot(y_k + x_k) + \phi'(x_{[k-1]}))\overline{c(\Phi(x_{[k-1]})\cdot y_k + \phi'(x_{[k-1]}))}\\
    =&\exx_{y_k} c(\Phi(x_{[k-1]})\cdot y_k + \lambda + \mu)\overline{c(\Phi(x_{[k-1]})\cdot y_k + \mu)}.\end{align*}

    By Lemma~\ref{solCrit}, the values taken by the map $y_k \in G_k \mapsto (\Phi_1(x_{[k-1]})\cdot y_k, \dots, \Phi_r(x_{[k-1]})\cdot y_k)$ are precisely $W^\perp$, and each value is taken the same number of times. Hence,  
    \[\bigconv{d}(c \circ \phi)(x_{[k]}) = \frac{1}{|W^\perp|} \sum_{w \in W^\perp} c(w + \lambda + \mu)\overline{w + \mu)},\]
    which depends only on $\lambda, \mu$ and $\nu$, as claimed.\\
    
    By the work above, we may define $\ell = 2r + t$, $\tilde{\phi} : G_{[k]} \to \mathbb{F}_p^\ell$ by $\tilde{\phi}(x_{[k]}) = (\phi(x_{[k]}), \phi'(x_{[k-1]}), \beta(x_{[k-1]}))$ and, for $\lambda \in \mathbb{F}_p^r, \mu \in \mathbb{F}_p^r, \nu \in \mathbb{F}_p^t$, $\tilde{c}(\lambda,\mu,\nu)$ as above when ${\beta}^{-1}(\nu) \subseteq U$, and $\tilde{c}(\lambda,\mu,\nu) = 0$ otherwise. With these definitions, we have $\bigconv{d}(c \circ \phi)(x_{[k]}) = \tilde{c}\circ \tilde{\phi}(x_{[k]})$ for all $x_{[k]} \in U \times G_k$. Thus, we have 
    \[\Big\|\bigconv{d}(c \circ \phi) - \tilde{c}\circ \tilde{\phi}\Big\|_{L^2}^2 = \exx_{x_{[k]}} \Big|\bigconv{d}(c \circ \phi)(x_{[k]}) -  \tilde{c}\circ \tilde{\phi}(x_{[k]})\Big|^2 \leq \frac{|G_{[k]} \setminus U|}{G_{[k]}} \leq \varepsilon^2.\qedhere\]
\end{proof}

We now deduce a strengthening of Theorem~\ref{denseColumnsThm} by obtaining control over sizes of slices of higher dimension.

\begin{theorem}[Fibres theorem]\label{fibresThm}For every $i \leq k$, there are constants $C = C_{i, k}$ and $D = D_{i, k}$ such that the following holds. Let $B \subset G_{[k]}$ be a non-empty variety of codimension at most $r$ and let $\varepsilon \in (0,1)$. Then we may find
\begin{itemize}
\item a positive integer $s \leq C \Big(r + \log_{\mathbf{f}} \varepsilon^{-1}\Big)^D$,
\item a multiaffine map $\beta \colon G_{[i]} \to \mathbb{F}^s$,
\item a union $U$ of layers of $\beta$ of size $|U| \geq (1-\varepsilon) |G_{[i]}|$, and
\item a map $c \colon \mathbb{F}^s \to [0,1]$,\end{itemize} 
such that 
\[\Big||G_{[i+1,k]}|^{-1} |B_{x_{[i]}}| - c(\beta(x_{[i]}))\Big| \leq \varepsilon\]
for every $x_{[i]} \in U$. Moreover, if $\mathcal{G} \subset \mathcal{P}[k]$ is a down-set, $\mathcal{G}^{(i)} = \mathcal{G} \cap \mathcal{P}([i])$, and $B$ is defined by maps that are $\mathcal{G}$-supported, then $\beta$ can be taken to be $\mathcal{G}^{(i)}$-supported.
\end{theorem}

\begin{proof}We prove the claim by downwards induction on $i$. For $i = k$, let $\alpha \colon G_{[k]} \to \mathbb{F}^r$ be a multiaffine map such that $B = \{x_{[k]} \in G_{[k]}:\alpha(x_{[k]}) = 0\}$. Set $\beta = \alpha$, so $s=r$, and for each $u\in\mathbb{F}^s$ let $c(u) = 1$ if $u = 0$ and let $c(u)=0$ otherwise.\\
\indent Assume now that the claim holds for $i+1$ and let $\eta > 0$ be a positive quantity to be specified later. Apply the inductive hypothesis with $\eta$ as the approximation parameter \chg{(which was denoted by $\varepsilon$ in the statement)}. We get 
\begin{itemize}
\item $s\leq C_{i+1, k} \Big(r + \log_{\mathbf{f}} \eta^{-1}\Big)^D_{i+1,k}$,
\item a $\mathcal{G}^{(i+1)}$-supported multiaffine map $\beta \colon G_{[i+1]} \to \mathbb{F}^s$,
\item a union $U$ of layers of $\beta$ of size $|U| \geq (1-\eta) |G_{[i+1]}|$, and
\item a map $c \colon \mathbb{F}^s \to [0,1]$,\end{itemize} 
such that for each $x_{[i+1]} \in U$,
\[\Big||G_{[i+2,k]}|^{-1}|B_{x_{[i+1]}}| - c(\beta(x_{[i+1]}))\Big| \leq \eta.\]
Let $F = G_{[i+1]} \setminus U$. Suppose that $x_{[i]} \in G_{[i]}$ is such that $|F_{x_{[i]}}| \leq \frac{\varepsilon}{4} |G_{i+1}|$. Then
\begin{align*}\Big||G_{[i+1,k]}|^{-1}|B_{x_{[i]}}| &- \exx_{x_{i+1} \in G_{i+1}} c(\beta(x_{[i+1]}))\Big|\\
&\hspace{1cm}=|G_{i+1}|^{-1}\bigg|\sum_{x_{i+1} \in G_{i+1}}\bigg(|G_{[i+2,k]}|^{-1}|B_{x_{[i+1]}}| - c(\beta(x_{[i+1]}))\bigg)\bigg|\\
&\hspace{1cm}\leq |G_{i+1}|^{-1}\sum_{x_{i+1} \in G_{i+1} \setminus F_{x_{[i]}}} \Big| |G_{[i+2,k]}|^{-1}|B_{x_{[i+1]}}| - c(\beta(x_{[i+1]}))\Big|\\
&\hspace{2cm}+ |G_{i+1}|^{-1}\sum_{x_{i+1} \in F_{x_{[i]}}}\Big||G_{[i+2,k]}|^{-1}|B_{x_{[i+1]}}| - c(\beta(x_{[i+1]}))\Big|\\
&\hspace{1cm}\leq \eta + \varepsilon/4.\end{align*}
Apply Theorem~\ref{denseColumnsThm} to find $s' \leq O((s \log_{\mathbf{f}}(\mathbf{f} \varepsilon^{-1}))^{O(1)})$ and a $\mathcal{G}^{(i)}$-supported multiaffine map $\beta' \colon G_{[i]} \to \mathbb{F}^{s'}$ whose layers externally $(\varepsilon/100)$-approximate the set $\{x_{[i]} \in G_{[i]} \colon |F_{x_{[i]}}| \geq \frac{\varepsilon}{4} |G_{i+1}|\}$. Note that $|F| \leq \frac{\varepsilon^2}{100} |G_{[i + 1]}|,$ provided $\eta \leq \frac{\varepsilon^2}{100}$. Hence, there is a union $U'$ of layers of $\beta'$ of size $|U'| \geq (1-\varepsilon/10)|G_{[i]}|$ such that for each $x_{[i]} \in U'$, $|F_{x_{[i]}}| \leq \frac{\varepsilon}{4} |G_{i+1}|$.\\ 

We now need to understand how the image of the map $y_{i+1} \mapsto \beta(x_{[i]}, y_{i+1})$ depends on $x_{[i]}$. 

For each $j \in [s]$ the map $\beta_j \colon G_{[i+1]} \to \mathbb{F}$ is a $\mathcal{G}^{(i+1)}$-supported multiaffine form, so we may find $\mathcal{G}^{(i)}$-supported multiaffine maps $\Gamma_j \colon G_{[i]} \to G_{i+1}$ and $\gamma_j \colon G_{[i]} \to \mathbb{F}$, such that $\beta_j(x_{[i+1]}) = \Gamma_j(x_{[i]}) \cdot x_{i+1} + \gamma_j(x_{[i]})$ for each $x_{[i+1]} \in G_{[i+1]}$. Apply Theorem~\ref{simVarAppThm} to the maps $\Gamma_1, \dots, \Gamma_s$. We obtain a positive integer $t = O\Big((s + \log_{\mathbf{f}} \eta^{-1})^{O(1)}\Big)$ and a $\mathcal{G}^{(i)}$-supported multiaffine map $\rho \colon G_{[i]} \to \mathbb{F}^t$ such that the layers of $\rho$ internally and externally $(\mathbf{f}^{-s^2-s}\eta)$-approximate the sets $\{x_{[i]} \in G_{[i]} \colon \lambda \cdot \Gamma(x_{[i]}) = 0\}$, for each $\lambda \in \mathbb{F}^s$. For a subspace $\Lambda \leq \mathbb{F}^s$ with a basis $\lambda^{(1)}, \dots, \lambda^{(v)}$, first approximate each variety $\{x_{[i]} \in G_{[i]} \colon \lambda^{(j)} \cdot \Gamma(x_{[i]}) = 0\}$ internally by layers of $\rho$. Then for each $\mu \in \mathbb{F}^s \setminus \Lambda$ approximate externally the variety $\{x_{[i]} \in G_{[i]} \colon \mu \cdot \Gamma(x_{[i]}) = 0\}$. This gives us an internal approximation of 
\begin{align*} \Big\{x_{[i]} \in G_{[i]} \colon (\forall \lambda \in \Lambda) \lambda \cdot \Gamma(x_{[i]}) = 0\Big\} &\setminus \bigg(\bigcup_{\mu \in \mathbb{F}^s \setminus \Lambda}\Big\{x_{[i]} \in G_{[i]} \colon \mu \cdot \Gamma(x_{[i]}) = 0\Big\}\bigg)\\& = \Big\{x_{[i]} \in G_{[i]} \colon \{\lambda \in \mathbb{F}^s \colon \lambda \cdot \Gamma(x_{[i]}) = 0\} = \Lambda\Big\}\end{align*}
by layers of $\rho$ with error of density at most $\mathbf{f}^{-s^2}\eta$.\\ 
\indent The number of subspaces of $\mathbb{F}^s$ is at most $\mathbf{f}^{s^2}$. Therefore there is a union $U''$ of layers of $\rho$ of size $|U''| \geq (1-\eta) |G_{[i]}|$ such that for each layer $L \subset U''$, there is a subspace $\Lambda_L \leq \mathbb{F}^s$ such that
\[\{\lambda \in \mathbb{F}^s \colon \lambda \cdot \Gamma(x_{[i]}) = 0\} = \Lambda_L\]
for every $x_{[i]} \in L$. Thus by Lemma~\ref{solCrit}, when $x_{[i]} \in L$, the image of the map $y_{i+1} \mapsto \beta(x_{[i]}, y_{i+1})$ is $\gamma(x_{[i]}) + \Lambda_L^\perp$.

When $x_{[i]} \in L \subset U''$, we may rewrite
\begin{align*}\exx_{x_{i+1} \in G_{i+1}}c(\beta(x_{[i+1]})) = &|G_{i+1}|^{-1} \sum_{x_{i+1} \in G_{i+1}} c\Big(\Gamma(x_{[i]}) \cdot x_{i+1} + \gamma(x_{[i]})\Big)\\
=&\frac{1}{|\Lambda_L^\perp|}\sum_{u \in \gamma(x_{[i]}) + \Lambda_L^\perp} c(u).\end{align*}
The quantity in the last line depends only on $\gamma(x_{[i]})$ and $\rho(x_{[i]})$. Hence, for any set of the form ${\beta'}^{-1}(\lambda_1) \cap \gamma^{-1}(\lambda_2) \cap \rho^{-1}(\lambda_3)$ such that ${\beta'}^{-1}(\lambda_1) \subset U'$ and $\rho^{-1}(\lambda_2) \subset U''$, there is a quantity $\tilde{c}(\lambda_1, \lambda_2, \lambda_3) \in [0,1]$ such that for each $x_{[i]} \in {\beta'}^{-1}(\lambda_1) \cap \gamma^{-1}(\lambda_2) \cap \rho^{-1}(\lambda_3)$,
\[\Big||G_{[i+1,k]}|^{-1}|B_{x_{[i]}}| - \tilde{c}(\lambda_1, \lambda_2, \lambda_3)\Big| = \Big||G_{[i+1,k]}|^{-1}|B_{x_{[i]}}| - \tilde{c}(\beta'(x_{[i]}), \gamma(x_{[i]}), \rho(x_{[i]}))\Big| \leq \eta + \varepsilon/4.\]
Choose $\eta = \frac{\varepsilon^2}{100}$ so that the necessary bounds are satisfied, and the proof is complete.\end{proof}

We deduce a simultaneous version of Theorem~\ref{fibresThm}, where we find a multiaffine map of low codimension that works for all layers of a given multiaffine map $\beta$.

\begin{corollary}[Simultaneous Fibres theorem]\label{simFibresThm}For every $i \leq k$ there are constants $C = C_{i, k}$ and $D = D_{i, k}$ such that the following holds. Let $\beta \colon G_{[k]} \to \mathbb{F}^r$ be a multiaffine map and let $\varepsilon \in (0,1)$. Then we may find
\begin{itemize}
\item a positive integer $s \leq C \Big(r + \log_{\mathbf{f}} \varepsilon^{-1}\Big)^D$,
\item a multiaffine map $\gamma \colon G_{[i]} \to \mathbb{F}^s$, and
\item a union $U$ of layers of $\gamma$ of size $|U| \geq (1-\varepsilon) |G_{[i]}|$,
\end{itemize}
such that for each layer $L$ of $\gamma$ inside $U$, there is a map $\chg{c_L} \colon \mathbb{F}^r \to [0,1]$ with the property that
\[\Big||G_{[i+1,k]}|^{-1} |\{y_{[i + 1, k]} \in G_{[i+1, k]} \colon \beta(x_{[i]}, y_{[i+1,k]}) = \lambda\}|\,\, -\,\, \chg{c_L}(\lambda)\Big| \leq \varepsilon\]
for every $\lambda \in \mathbb{F}^r$ and every $x_{[i]} \in L$. Moreover, if $\mathcal{G} \subset \mathcal{P}[k]$ is a down-set, $\mathcal{G}^{(i)} = \mathcal{G} \cap \mathcal{P}([i])$, and $\beta$ is $\mathcal{G}$-supported, then $\gamma$ can be taken to be $\mathcal{G}^{(i)}$-supported.
\end{corollary}

\begin{proof}Consider a modified map $\tilde{\beta} \colon G_{[k]} \times \mathbb{F}^r \to \mathbb{F}^r$, defined by the formula $\tilde{\beta}(x_{[k]}, \lambda) = \beta(x_{[k]}) - \lambda$, which is $\tilde{\mathcal{G}}$-supported, where $\tilde{\mathcal{G}} = \mathcal{G} \cup \{\{k+1\}\}$. Let $B = \{(x_{[k]}, \lambda) \in G_{[k]} \times \mathbb{F}^r \colon \tilde{\beta}(x_{[k]}, \lambda) = 0\}$ and write $\mathcal{H} = \mathcal{G}^{(i)} \cup \{\{k+1\}\}$. Apply Theorem~\ref{fibresThm} to $B$, \chg{(with a slight reordering of directions, so that the space $\mathbb{F}^r$ is considered to precede $G_{[k]}$)} with parameter $\tilde{\varepsilon} = \varepsilon \mathbf{f}^{-r}$, to get an $\mathcal{H}$-supported map $\psi \colon G_{[i]} \times \mathbb{F}^r \to \mathbb{F}^s$, for a positive integer $s = O\Big((r + \log_{\mathbf{f}} \varepsilon^{-1})^{O(1)}\Big)$ and a union of layers $\tilde{U}$ of $\psi$ of size $|\tilde{U}| \geq (1- \varepsilon \mathbf{f}^{-r}) |G_{[i]}| |\mathbb{F}^r|$, such that for each layer $L$ of $\psi$ inside $\tilde{U}$ there is some constant $c$ such that
\begin{equation}\label{approxEqncB} \Big||G_{[i+1,k]}|^{-1} |B_{x_{[i]}, \lambda}| - c \Big| \leq \mathbf{f}^{-r}\varepsilon\end{equation}
for every $(x_{[i]}, \lambda) \in L$.\\
\indent Let $\gamma \colon G_{[i]} \to (\mathbb{F}^s)^{r+1}$ be the $\mathcal{G}^{(i)}$-supported multiaffine map defined by
\[\gamma(x_{[i]}) = \Big(\psi(x_{[i]}, 0), \psi(x_{[i]}, e_1), \dots, \psi(x_{[i]}, e_r)\Big),\]
where $e_1, \dots, e_r$ is the standard basis of $\mathbb{F}^r$. Given a layer $L = \{(x_{[i]},\lambda') \in G_{[i]} \times \mathbb{F}^r:\psi(x_{[i]},\lambda') = u\}$ and given some $\lambda \in \mathbb{F}^r$, we have
\begin{align}L_\lambda &= \Big\{x_{[i]} \in G_{[i]}\colon \psi(x_{[i]}, \lambda) = u\Big\}\nonumber\\
&=  \Big\{x_{[i]} \in G_{[i]}\colon \Big(1 - \sum_{j \in [r]} \lambda_j\Big) \psi(x_{[i]}, 0) + \lambda_1 \psi(x_{[i]}, e_1) + \dots + \lambda_r \psi(x_{[i]}, e_r) = u\Big\}\nonumber\\
&= \bigcup\limits_{\substack{w \in (\mathbb{F}^s)^{r+1}\\w_1 (1 - \sum_{j \in [r]} \lambda_j) + w_2 \lambda_1 + \dots + w_{r+1} \lambda_r = u}} \{x_{[i]} \in G_{[i]}\colon \gamma(x_{[i]}) = w\}.\label{simfibrdecEqn}\end{align}
Hence, $L_\lambda$ is a union of layers of $\gamma$. Let $U$ be the union of those layers $K$ of $\gamma$ such that for each $\lambda \in \mathbb{F}^r$, the layer $K$ appears in the decomposition~\eqref{simfibrdecEqn} above for some $L_\lambda$, where $L \subset \tilde{U}$ is a layer of $\psi$. In other words, $U$ is the union of all layers $K$ of $\gamma$ such that $K \times \{\lambda\} \subset \tilde{U}$ for every $\lambda \in \mathbb{F}^r$. Since $|\tilde{U}| \geq (1- \varepsilon \mathbf{f}^{-r}) |G_{[i]}| |\mathbb{F}^r|$, we have that $|U| \geq (1-\varepsilon) |G_{[i]}|$. Finally, if $x_{[i]} \in K$ for some layer $K$ of $\gamma$ inside $U$, then for every $\lambda$ there is some layer $L \subset \tilde{U}$ of $\psi$ such that $K \subset L_\lambda$. Let $c(\lambda)$ be the relevant constant for $L_\lambda$ in~\eqref{approxEqncB}. Then for every $x_{[i]} \in K$ and every $\lambda \in \mathbb{F}^r$,
\begin{align*}\Big||G_{[i+1,k]}|^{-1} |\{y_{[i + 1, k]} \in G_{[i+1, k]} \colon \beta(x_{[i]}, y_{[i+1,k]}) = \lambda\}| - c(\lambda)\Big|= \Big||G_{[i+1,k]}|^{-1} |B_{x_{[i]}, \lambda}| - c(\lambda) \Big| \leq \varepsilon,\end{align*}
which completes the proof.
\end{proof}

Let $\alpha \colon G_{[k]} \to H$ be a multiaffine map. Let $\mathcal{F}$ be the collection of subsets of $[k]$ such that $\alpha(x_{[k]}) = \sum_{I \in \mathcal{F}} \alpha_I(x_I)$ for non-zero multilinear maps $\alpha_I \colon G_I \to H$. We call these maps the \emph{multilinear parts} of $\alpha$. Equivalently, $\mathcal{F}$ is the minimal collection of sets such that $\alpha$ is $\mathcal{F}$-supported. We define $\supp \alpha$ to be the down-set generated by $\mathcal{F}$, in other words the set $\{A \subset [k] \colon (\exists B \in \mathcal{F})\, A \subseteq B\}$. That is, $\supp \alpha$ is the minimal down-set $\mathcal{G}$ such that $\alpha$ is $\mathcal{G}$-supported. We say that a multiaffine map $\beta \colon G_{[k]} \to H$ is \emph{of lower order than} $\alpha$ if  $\supp \beta$ contains none of the maximal elements of $\supp \alpha$ (with respect to inclusion). When $\alpha$ is fixed, we simply say that $\beta$ is \emph{of lower order}. Moreover, if $\beta \colon G_{[l]} \to H$ is multiaffine for some $l < k$, we still say that $\beta$ is of lower order than $\alpha$ if $\supp \beta$ contains no maximal set in $\supp \alpha$. For varieties, we say that $U$ is \emph{of lower order than} $V$ if the multiaffine maps used to define $U$ and $V$ satisfy the corresponding condition. Again, when $V$ is fixed, we say that $U$ is \emph{of lower order}.\\
\indent We also need the following observation of Lovett~\cite{Lovett}.

\begin{lemma}[Lovett, Lemma 2.1 in~\cite{Lovett}]\label{biasHomog}Let $\alpha \colon G_{[k]} \to \mathbb{F}$ be a multiaffine form with multilinear parts $\alpha_I$, and let $\chi:\mathbb F\to\mathbb C$ be an additive character. Then
\[\Big|\exx_{x_{[k]}} \chi(\alpha(x_{[k]}))\Big| \leq \exx_{x_{[k]}} \chi(\alpha_{[k]}(x_{[k]})) \in \mathbb{R}_{\geq 0},\]
\chg{where we interpret $\alpha_{[k]}$ as zero if there is no contribution of a form multilinear on $G_{[k]}$ in $\alpha$}.\end{lemma}

\indent Further, when $\chg{\alpha = (\alpha_1, \dots, \alpha_r)} \colon G_{[k]} \to \mathbb{F}^r$ is a multiaffine map such that each component \chg{form $\alpha_i : G_{[k]} \to \mathbb{F}$} is multilinear on $G_{I_i}$, where $I_i$ is the set of coordinates on which $\alpha_i$ depends, we say that $\alpha$ is \emph{mixed-linear}. \chg{We allow $I_i = \emptyset$, but as we shall typically be interested in partitions into layers $\alpha^{-1}(\lambda)$ of the given map, the constant terms are mostly irrelevant.} We say that a variety is \emph{mixed-linear} if it is a layer of a mixed-linear map.\\

The next theorem is more technical than previous results. Suppose that we are given a very dense subset $X$ of a variety $V$ of bounded codimension. The following theorem allows us to find a point $a_{[k]}$ and a further subset $X'$ and variety $V'$ such that for each $x_{[k]} \in X'$ the sequence $a_{[k]},$ $(x_1, a_{[2,k]}), \dots,$ $(x_{[k-1]}, a_k),$ $x_{[k]}$ of points lies in $X'$, $X'$ is still very dense in $V'$ and $V'$ is given by an intersection of $V$ with a lower order variety of bounded codimension. Note that the sequence $a_{[k]},$ $(x_1, a_{[2,k]}), \dots,$ $(x_{[k-1]}, a_k),$ $x_{[k]}$ has the property that every two consecutive points differ in exactly one coordinate.\\
\indent\chg{We also need an additional piece of notation. Given multiaffine forms $\alpha_i \colon G_{I_i} \to \mathbb{F}$ for $i \in [r]$ and a variety $V = \{x_{[k]} \in G_{[k]} : (\forall i \in [r]) \alpha_i(x_{I_i}) = \nu_i\}$, we write $V^{[k-1]}$ for the variety arising from the forms that only involve directions in $[k-1]$, namely
\[V^{[k-1]} = \{x_{[k-1]} \in G_{[k-1]} \colon (\forall i \in [r])\ I_i \subset [k-1]\implies \alpha_i(x_{I_i}) = \nu_i\}.\]
We stress that we will not be using the typical meaning of such notation, which is the space of functions $[k-1] \to V$, but rather we think of $V^{[k-1]}$ as a natural restriction of $V$ to $G_{[k-1]}$.}

\begin{theorem}\label{connectedVeryDensePiece}Let $\alpha \colon G_{[k]} \to \mathbb{F}^r$ be a mixed-linear map, with $\alpha_i \colon G_{I_i} \to \mathbb{F}$. Let $\nu \in \mathbb{F}^r$ and let $V = \{x_{[k]}:\alpha(x_{[k]}) = \nu\}$ be a non-empty layer of $\alpha$. Let $X \subset V$ be a set of size $|X| \geq (1-\varepsilon)|V|$ and let $\xi > 0$. Then there exist a variety $U \subset G_{[k-1]}$ of codimension $O((r+\log_{\mathbf{f}} \xi^{-1})^{O(1)})$ and of lower order than $V$, a subset $X' \subset X \cap (U \times G_k)$, and an element $a_{[k]} \in X'$, such that $|X'| \geq (1- O(\varepsilon^{\Omega(1)} + \xi^{\Omega(1)}))|V \cap (U \times G_k)|$ and for each $x_{[k]} \in X'$, $(x_{[i]}, a_{[i+1,k]}) \in X'$ for each $i \in [0,k]$. Moreover, there is $\delta > 0$ such that for each $x_{[k-1]} \in U \cap V^{[k-1]}$, $|V_{x_{[k-1]}}| = \delta |G_k|$.
\end{theorem}

Note that since $a_{[k]} \in X' \subset X \cap (U \times G_k) \subset V \cap (U \times G_k)$ we have in particular $V \cap (U \times G_k) \not= \emptyset$.

\begin{proof}We prove the claim by induction on $k$. \chg{The base case is trivial.} Suppose the claim holds for $k-1$. Let $\mathcal{I} = \{i \in [r] \colon k \in I_i\}$. Let $W = \{x_{[k]} \in G_{[k]} \colon (\forall i \in \mathcal{I}) \alpha_i(x_{I_i}) = \nu_i\}$. Then $V = W \cap (V^{[k-1]} \times G_k)$. In fact, when $x_{[k-1]} \in V^{[k-1]}$ then $V_{x_{[k-1]}} = W_{x_{[k-1]}}$, while $V_{x_{[k-1]}} = \emptyset$ when $x_{[k-1]} \in G_{[k-1]} \setminus V^{[k-1]}$. Apply Theorem~\ref{fibresThm} in direction $G_k$ to the variety $W$ with parameter $\mathbf{f}^{-kr}\xi$ to find a positive integer $s \leq C \Big(r + \log_{\mathbf{f}} \xi^{-1}\Big)^D$, a multiaffine map $\gamma \colon G_{[k-1]} \to \mathbb{F}^s$ of lower order than $\alpha$, and a collection of values $M \subset \mathbb{F}^s$ such that $|\{x_{[k-1]} \in G_{[k-1]} \colon \gamma(x_{[k-1]}) \in M\}| \geq (1-\mathbf{f}^{-kr}\xi) |G_{[k-1]}|$, with the property that for every $\mu\in M$ there exists a constant $c_\mu\in[0,1]$ such that $|V_{x_{[k-1]}}| =|W_{x_{[k-1]}}| = c_\mu |G_k|$ for every $x_{[k-1]} \in V^{[k-1]}$ such that $\gamma(x_{[k-1]}) = \mu$.\footnote{Note that Theorem~\ref{fibresThm} actually gives the approximation $\Big||W_{x_{[k-1]}}| - c_\mu |G_k|\Big| \leq \mathbf{f}^{-kr}\xi |G_k|.$ However, in the special case when Theorem~\ref{fibresThm} is applied for slices in a single direction, the slices $W_{x_{[k-1]}}$ are in fact cosets of subspaces of codimension at most $r$. Hence, the possible values of $|G_k|^{-1}|W_{x_{[k-1]}}|$ are only $\chg{0}, \mathbf{f}^{-r}, \mathbf{f}^{-r + 1}, \dots, \mathbf{f}^{-1}, 1$, so since $\mathbf{f}^{-kr}\xi < \frac{1}{2}\mathbf{f}^{-r}$, we get that all slices $W_{x_{[k-1]}}$ have the exact same size. We shall frequently use such an exact version of Theorem~\ref{fibresThm} instead of its original approximate variant in the case of 1-directional slices without further comment.}\\
\indent Lemma~\ref{varsizelemma} implies that $|V| \geq \mathbf{f}^{-kr} |G_{[k]}|$. We also have
\begin{align*}\sum_{\mu \in M} |X \cap (\gamma^{-1}(\mu) \times G_k)|\hspace{2pt}\geq\hspace{2pt}|X| - \mathbf{f}^{-kr}\xi |G_{[k]}|\hspace{2pt}&\geq\hspace{2pt}(1-\varepsilon)|V| - \mathbf{f}^{-kr}\xi |G_k|\hspace{2pt}\geq\hspace{2pt}(1-\varepsilon - \xi)|V|\hspace{2pt}\\
&\geq\hspace{2pt}(1-\varepsilon - \xi) \sum_{\mu \in M} |V \cap (\gamma^{-1}(\mu) \times G_k)|\\
&\geq\hspace{2pt}(1-\varepsilon - \xi) |V| - \mathbf{f}^{-kr}\xi |G_{[k]}|\,> 0,\end{align*}
provided $\varepsilon, \xi < \frac{1}{4}$ (otherwise, the statement of the theorem is vacuous, \chg{as we may take empty $X'$}). Hence, there is a choice of $\mu \in M$ such that $V \cap (\{x_{[k-1]}:\gamma(x_{[k-1]}) = \mu\} \times G_k)$ is non-empty and 
\[|X \cap (\{x_{[k-1]}:\gamma(x_{[k-1]}) = \mu\} \times G_k)| \geq (1-\varepsilon - \xi) |V \cap (\{x_{[k-1]}:\gamma(x_{[k-1]}) = \mu\} \times G_k)|.\]

Write $V' = V \cap (\{x_{[k-1]}:\gamma(x_{[k-1]}) = \mu\} \times G_k)$ and set 
\[X' = \{x_{[k]} \in X \colon \gamma(x_{[k-1]}) = \mu, |X_{x_{[k-1]}}| \geq (1-\sqrt{\varepsilon} - \sqrt{\xi})|V_{x_{[k-1]}}|\}.\] 
Since there is $c_\mu > 0$ such that $|V_{x_{[k-1]}}| = c_\mu |G_k|$ for each $x_{[k-1]} \in V^{[k-1]}$ with $\gamma(x_{[k-1]}) = \mu$,  we have that 
\[|X'| \geq (1-\sqrt{\varepsilon} - \sqrt{\xi}) |V'| > 0.\]

By averaging, there is some $a_k \in G_k$ such that 
\[|(X')_{a_k}| \geq (1-\sqrt{\varepsilon} - \sqrt{\xi}) |(V')_{a_k}| > 0.\]
We may now apply the induction hypothesis to $(X')_{a_k}$ and $(V')_{a_k}$. We obtain a variety $U' \subset G_{[k-1]}$ (actually the induction hypothesis gives a variety $U'' \subset G_{[k-2]}$, here we immediately put $U' = U'' \times G_{k-1}$) of codimension $O((r + \log_{\mathbf{f}} \xi^{-1})^{O(1)})$ and of lower order than $V'_{a_k}$ and hence of lower order than $V$, a subset $X'' \subset (X')_{a_k} \cap U'$, and an element $a_{[k-1]} \in X''$, such that $|X''| \geq (1- O(\varepsilon^{\Omega(1)} + \xi^{\Omega(1)}))|(V')_{a_k} \cap U'|$ and $(x_{[i]}, a_{[i+1,k-1]}) \in X''$ for each $x_{[k-1]} \in X''$ and each $i \in [k-1]$. We claim that $a_{[k]}$, $(X'' \times G_k) \cap X'$, and $V' \cap ((V'_{a_k} \cap U') \times G_{k})$ have the desired properties.\\

Firstly, 
\[V' \cap ((V'_{a_k} \cap U') \times G_{k}) = V \cap ((\{x_{[k-1]}:\gamma(x_{[k-1]})  = \mu\} \cap U' \cap W_{a_k})\times G_k),\] 
is indeed a variety obtained from $V$ by intersecting it with a lower order variety of the codimension claimed. (Recall that $W$ was previously defined as $\{x_{[k]} \in G_{[k]} \colon (\forall i \in \mathcal{I}) \alpha_i(x_{I_i}) = \nu_i\}$. Also recall that $V = W \cap (V^{[k-1]} \times G_k)$ holds.) Next, 
\begin{align*}|(X'' \times G_k) \cap X'| =& \sum_{x_{[k-1]} \in X''} |X'_{x_{[k-1]}}| \\
= &\sum_{x_{[k-1]} \in X''} |X_{x_{[k-1]}}|\\
 \geq& \sum_{x_{[k-1]} \in X''} (1-\sqrt{\varepsilon} - \sqrt{\xi}) |V_{x_{[k-1]}}|\\
 = &(1-\sqrt{\varepsilon} - \sqrt{\xi}) c_\mu |G_k| |X''|\\
=& (1- O(\varepsilon^{\Omega(1)} + \xi^{\Omega(1)}))c_\mu |G_k| |(V')_{a_k} \cap U'|\\
=&(1- O(\varepsilon^{\Omega(1)} + \xi^{\Omega(1)})) \sum_{x_{[k-1]} \in (V')_{a_k} \cap U'} |V_{x_{[k-1]}}|\\
=&(1- O(\varepsilon^{\Omega(1)} + \xi^{\Omega(1)})) |V \cap (((V')_{a_k} \cap U') \times G_k)|\\
=&(1- O(\varepsilon^{\Omega(1)} + \xi^{\Omega(1)})) |V' \cap (((V')_{a_k} \cap U') \times G_k)|.\end{align*}

From the facts that $a_{[k-1]} \in X''$ and $X'' \subset X'_{a_k}$ we deduce that $a_{[k]} \in  (X'' \times G_k) \cap X'$.

Finally, for each $x_{[k]} \in (X'' \times G_k) \cap X'$ we need to show that $(x_{[i]}, a_{[i+1,k]}) \in (X'' \times G_k) \cap X'$ for $i \in [k]$ as well. Since $X'' \subset X'_{a_k}$, we already know that $(x_{[k-1]}, a_k) \in (X'' \times G_k) \cap X'$. Now suppose that $i \in [k-2]$. Then since $x_{[k-1]} \in X''$, we know that $(x_{[i]}, a_{[i+1,k-1]}) \in X''$. But again $X'' \subset X'_{a_k}$, so $(x_{[i]}, a_{[i+1,k]}) \in X'$, which completes the proof.\end{proof}

Let $A \subset G_{[k]}$. We say that a map $\phi \colon A \to H$ is \emph{multiaffine} if for each $d \in [k]$ and $x_{[k] \setminus \{d\}} \in G_{[k] \setminus \{d\}}$ there is an affine map $\psi \colon G_d \to H$ such that for every $y_d \in A_{x_{[k] \setminus \{d\}}}$ we have $\phi(x_{[k] \setminus \{d\}}, y_d)= \psi(y_d)$.\\

The following proposition shows that multi-homomorphisms on very dense subsets of \chg{mixed-linear} varieties are essentially multiaffine.

\begin{proposition}\label{MltHommToMltAff}Assume that the underlying field is $\mathbb{F}_p$ for a prime $p$. Let $\alpha \colon G_{[k]} \to \mathbb{F}_p^r$ be a mixed-linear map, with $\alpha_i \colon G_{I_i} \to \mathbb{F}_p$. Let $\nu \in \mathbb{F}_p^r$ and let $V = \{x_{[k]}:\alpha(x_{[k]}) = \nu\}$ be a non-empty layer of $\alpha$. Suppose that $X \subset V$ is a set of size $|X| \geq (1-\varepsilon)|V|$ and let $\phi \colon X \to H$ be a multi-homomorphism. Let $\xi > 0$. Then there exist a variety $U \subset G_{[k]}$ of lower order than $V$ and of codimension $O((r+\log_p \xi^{-1})^{O(1)})$ for which the intersection $U \cap V$ is non-empty and a subset $X' \subset X \cap U$ of size
\[(1- O(\varepsilon^{\Omega(1)} + \xi^{\Omega(1)}))|V \cap U|\]
such that $\phi|_{X'}$ is multiaffine.\end{proposition}

\noindent\textbf{Remark.} By modifying $\xi$ appropriately, we may strengthen the conclusion slightly to
\begin{equation}|X'| \geq (1-O(\varepsilon^{\Omega(1)}) - \xi) |V \cap U|\label{MltHommToMltAffEqn}\end{equation}
without affecting the form of the bound on the codimension of $U$.

\begin{proof}We show that for each $d \in [0, k]$, there exist a lower order variety $U \subset G_{[k]}$ of codimension $O((r+\log_p \xi^{-1})^{O(1)})$ and a subset $X' \subset X \cap U$ of size $(1- O(\varepsilon^{\Omega(1)} + \xi^{\Omega(1)}))|V \cap U|$ such that for each $d' \in [d]$, $\phi|_{X'}$ is a restriction of an affine map on each line in direction $G_{d'}$. The statement of the proposition then becomes the case $d = k$.\\
\indent We prove the claim by induction on $d$. Note that the base case $d = 0$ is trivial. Suppose now that the claim has been proved for some $d \in [0,k-1]$ and let $X'$ and $U$ be the relevant objects. Let $U$ be defined by a multiaffine map of codimension $s = O((r+\log_p \xi^{-1})^{O(1)})$. Without loss of generality $U$ is defined by a mixed-linear map.\footnote{For a multiaffine map $\gamma \colon G_{[k]} \to \mathbb{F}^t$ that defines $U = \{\gamma = 0\}$, we may write it as $\gamma(x_{[k]}) = \sum_{I \in \mathcal{F}} \gamma_I(x_{[k]})$, where $\gamma_I \colon G_I \to \mathbb{F}^t$ is multilinear. Let $\tilde{\gamma} \colon G_{[k]} \to (\mathbb{F}^t)^{\mathcal{F}}$ be the concatenation of maps $\gamma_I$. Then layers of $\gamma$ (and in particular $U$) are unions of layers of $\tilde{\gamma}$, so we may average over latter, and $\tilde{\gamma}$ is mixed-linear.} Let $|X'| = (1-\varepsilon') |V \cap U|$, where $\varepsilon' = O(\varepsilon^{\Omega(1)} + \xi^{\Omega(1)})$. Let 
\[V^0 = \{x_{[k] \setminus \{d+1\}} \in G_{[k] \setminus \{d+1\}} \colon (\forall i \in [r])\ d+1 \notin I_i\implies \alpha_i(x_{I_i}) = \nu_i\}\] 
and let $\mathcal{I} = \{i \in [r] \colon d+1 \in I_i\}$. Apply Theorem~\ref{fibresThm} in direction ${d+1}$ to the variety
\[W = \{x_{[k]} \in G_{[k]} \colon (\forall i \in \mathcal{I}) \alpha_i(x_{I_i}) = \nu_i\} \cap U\]
to find a positive integer $t = O((r+\log_p \xi^{-1})^{O(1)})$, a multiaffine map $\beta \colon G_{[k] \setminus \{d+1\}} \to \mathbb{F}_p^t$ of lower order than $\alpha$, a collection of values $M \subset \mathbb{F}_p^t$, and a map $c \colon M \to [0,1]$ such that
\[|\beta^{-1}(M)| \geq (1 - p^{-kr - ks} \xi)|G_{[k] \setminus \{d+1\}}|\]
and
\begin{equation}\label{samecolsizeMltHommProp}|(V \cap U)_{x_{[k] \setminus \{d+1\}}}| = |W_{x_{[k] \setminus \{d+1\}}}| = c(\mu)|G_{d+1}|\end{equation}
for every $\mu \in M$ and every $x_{[k] \setminus \{d+1\}} \in V^0 \cap \beta^{-1}(\mu)$.\\
\indent By Lemma~\ref{varsizelemma}, we have $|V \cap U| \geq p^{-kr -ks}|G_{[k]}|$. Thus, we have
\begin{align*}\sum_{\mu \in M} |X' \cap (\beta^{-1}(\mu) \times G_{d+1})| \geq &|X'| - |G_{d+1}| \cdot |G_{[k] \setminus \{d+1\}} \setminus \beta^{-1}(M)|\\
\geq & (1-\varepsilon') |V \cap U| - p^{-kr - ks} \xi |G_{[k]}|\\
\geq & (1-\varepsilon'-\xi) |V \cap U|\\
\geq & (1-\varepsilon'-\xi) \sum_{\mu \in M} |V \cap U \cap (\beta^{-1}(\mu) \times G_{d+1})|.\end{align*}
We may assume that $\varepsilon' + \xi < 1$, as otherwise the claim is trivial. Note also that the fact that $|V \cap U| \geq p^{-kr -ks}|G_{[k]}|$ in particular implies that $V \cap U \cap \beta^{-1}(M) \not=\emptyset$.\\
\indent By averaging, there exists $\mu \in M$ such that the variety $V \cap U \cap (\beta^{-1}(\mu) \times G_{d+1})$ is non-empty and $|X' \cap (\beta^{-1}(\mu) \times G_{d+1})| \geq (1-\varepsilon' - \xi) |V \cap U \cap (\beta^{-1}(\mu) \times G_{d+1})|$. Let
\[X'' = \Big\{x_{[k]} \in X' \cap (\beta^{-1}(\mu) \times G_{d+1}) \colon |X'_{x_{[k] \setminus \{d+1\}}}| \geq (1-\sqrt{\varepsilon'} - \sqrt{\xi}) |(V \cap U)_{x_{[k] \setminus \{d+1\}}}|\Big\}.\]
Using~\eqref{samecolsizeMltHommProp} we see that $|X''| \geq  (1-\sqrt{\varepsilon'} - \sqrt{\xi}) |V \cap U \cap (\beta^{-1}(\mu) \times G_{d+1})|$. Finally, applying Lemma~\ref{easyExtn} we obtain that $\phi|_{X''}$ is a restriction of an affine map on each line in direction $G_{d+1}$. Here we need $\sqrt{\varepsilon'} + \sqrt{\xi} < \frac{1}{5}$ to hold, which we may assume as otherwise the claim is vacuous. Thus, the set $X''$ and variety $V \cap U \cap (\beta^{-1}(\mu) \times G_{d+1})$ satisfy the required properties.\end{proof}

Furthermore, it turns out that multi-homomorphisms defined on a $1-o(1)$ proportion of the whole space $G_{[k]}$ become restrictions of global multiaffine maps after omitting a few points if necessary.

\begin{proposition}\label{lastExtnStep}Assume that the underlying field is $\mathbb{F}_p$. There are constants $c_k, C_k, d_k > 0$ depending on $k$ only such that the following holds.\\
\indent Let $X \subset G_{[k]}$ be a set of density at least $1-\varepsilon$ for some $\varepsilon \in (0, c_k)$ and let $\phi \colon X \to H$ be a multi-homomorphism. Then there is a unique multiaffine map $\Phi \colon G_{[k]} \to H$ such that $\Phi(x_{[k]}) = \phi(x_{[k]})$ for a $1- C_k\varepsilon^{d_k}$ proportion of the $x_{[k]} \in G_{[k]}$.\end{proposition}

\begin{proof}We prove the result by induction on $k$. For $k = 1$, it follows from Lemma~\ref{easyExtn}. Assume now that the result holds for some $k-1 \geq 1$ and that $X$ and $\phi$ satisfying the hypothesis of the proposition are given. Let $Y= \{x_k \in G_k\colon |X_{x_k}| \geq (1 - \sqrt{\varepsilon})|G_{[k-1]}|\}$. Then $|Y| \geq (1-\sqrt{\varepsilon})|G_k|$. Provided $\varepsilon < c_{k-1}^2$, apply the induction hypothesis to $X_{x_k}$ for each $x_k \in Y$ to get a unique multiaffine map $\Phi_{x_k}$ that coincides on a subset of points $\tilde{X}_{x_k} \subset X_{x_k}$ with the map $\phi(\cdot , x_k)$. The induction hypothesis also gives a positive quantity $\varepsilon' \leq C_{k-1}\varepsilon^{d_{k-1}/2}$ such that $|\tilde{X}_{x_k}| \geq (1-\varepsilon')|G_{[k-1]}|$ for all $x_k \in Y$. For $x_k \in Y$ and $x_{[k-1]} \in G_{[k-1]}$ define $\Phi(x_{[k-1]}, x_k) = \Phi_{x_k}(x_{[k-1]})$. We claim that $\Phi$ is a multi-homomorphism on $G_{[k-1]} \times Y$. It suffices to prove that for all $x_k^{[4]}$ in $Y$ such that $x^1_k + x^2_k = x^3_k + x^4_k$ we also have $\Phi(x_{[k-1]}, x^1_k) + \Phi(x_{[k-1]}, x^2_k) = \Phi(x_{[k-1]}, x^3_k) + \Phi(x_{[k-1]}, x^4_k)$ for all $x_{[k-1]} \in G_{[k-1]}$. Let $Z = \bigcap_{i \in [4]} \tilde{X}_{x_k^i}$. Note that $|Z| \geq (1-4\varepsilon')|G_{[k-1]}|$ and that for each $x_{[k-1]} \in Z$ we have
\begin{align*}&\Phi(x_{[k-1]}, x^1_k) + \Phi(x_{[k-1]}, x^2_k) - \Phi(x_{[k-1]}, x^3_k) - \Phi(x_{[k-1]}, x^4_k)\\
&\hspace{2cm}= \phi(x_{[k-1]}, x^1_k) + \phi(x_{[k-1]}, x^2_k) - \phi(x_{[k-1]}, x^3_k) - \phi(x_{[k-1]}, x^4_k)\\
&\hspace{2cm} = 0.\end{align*}
Hence the map $\tau \colon G_{[k-1]} \to H$ given by
\[\tau(x_{[k-1]}) = \Phi(x_{[k-1]}, x^1_k) + \Phi(x_{[k-1]}, x^2_k) - \Phi(x_{[k-1]}, x^3_k) - \Phi(x_{[k-1]}, x^4_k)\]
is actually 0 on $Z$. But $\tau$ is a global multiaffine map on $G_{[k-1]}$. The next claim will imply that it is itself zero as long as $4\varepsilon' < p^{-k}$, which is true if $\varepsilon < \Big(\frac{p^{-k}}{4C_{k-1}}\Big)^{2/d_{k-1}}$.

\begin{claim}\label{nonzerovalsclaim}Suppose that $\psi \colon G_{[k]} \to H$ is a multiaffine map. Then $\psi = 0$ or $\psi(x_{[k]}) \not= 0$ for at least $p^{-k}|G_{[k]}|$ points $x_{[k]} \in G_{[k]}$.\end{claim}
\begin{proof}[Proof of claim] Suppose that $\psi \not =0$. Without loss of generality $H = \mathbb{F}_p$, i.e.\ $\psi$ is a multiaffine form (view $H$ as $\mathbb{F}_p^s$ and take some $\psi_i$ that is a non-zero map). Then, $\psi$ takes the value 1, so $\{x_{[k]}:\psi(x_{[k]}) = 1\}$ is a non-empty variety of codimension 1. By Lemma~\ref{varsizelemma} it has density at least $p^{-k}$.\end{proof}

\chg{Let $\mathcal{M}$ be the abelian group of all multiaffine maps from $G_{[k-1]}$ to $H$, in which the group operation is pointwise addition. Let us write $\Psi: Y \to \mathcal{M}$ for the map that sends $y \in Y$ to the map $x_{[k-1]} \mapsto \Phi(x_{[k-1]}, x_k)$. By the work above, $\Psi$ is a Freiman homomorphism and $|Y| \geq (1 - \sqrt{\varepsilon})|G_k|$. Hence, as long as $\varepsilon \leq 1/25$, Lemma~\ref{easyExtn} applies to give an affine extension $\Tilde{\Psi}$ of $\Psi$ to the whole $G_k$. Finally, define $\tilde{\Phi}(x_{[k]}) = \tilde{\Psi}(x_k)(x_{[k-1]})$, which gives a multiaffine map $\tilde{\Phi} : G_{[k]} \to H$. Hence, $\tilde{\Phi}$ coincides with $\Phi$ at at least $(1-\sqrt{\varepsilon} - \varepsilon')|G_{[k]}| \geq (1 - 2C_{k-1}\varepsilon^{d_{k-1}/2})|G_{[k]}|$ points of $G_{[k]}$. Uniqueness also follows from Claim~\ref{nonzerovalsclaim}, completing the proof.} 

\end{proof}

\subsection{Box norms}
\indent Let $X$ and $Y$ be finite sets. For a function $f \colon X \times Y \to \mathbb{C}$, we define its \emph{box norm} by the formula
\[\|f\|_{\square} = \Big(\exx_{x_1, x_2 \in X} \exx_{y_1, y_2 \in Y} f(x_1, y_1) \overline{f(x_1, y_2)}\hspace{2pt} \overline{f(x_2, y_1)} f(x_2, y_2)\Big)^{1/4}.\]

The box norm satisfies the following Cauchy--Schwarz-like inequality.

\begin{lemma}\label{boxCS}Let $f_{11}, f_{12}, f_{21}, f_{22} \colon X \times Y\to \mathbb{C}$ be four functions. Then
\[\Big|\exx_{x_1, x_2 \in X} \exx_{y_1, y_2 \in Y} f_{11}(x_1, y_1) \overline{f_{12}(x_1, y_2)}\hspace{1pt} \overline{f_{21}(x_2, y_1)} f_{22}(x_2, y_2)\Big| \leq \|f_{11}\|_{\square} \|f_{12}\|_{\square} \|f_{21}\|_{\square} \|f_{22}\|_{\square}.\]
\end{lemma}

In particular, we have the following useful corollary.

\begin{corollary}\label{boxUnifCor}Suppose that $f \colon X \times Y \to \mathbb{C}$, $u \colon X \to \mathbb{C}$ and $v \colon Y \to \mathbb{C}$ are three functions. Then
\[\Big|\exx_{x \in X, y \in Y} u(x) v(y) f(x,y)\Big| \leq \|u\|_2 \|v\|_2 \|f\|_{\square}.\]\end{corollary}

\section{An approximation theorem for mixed convolutions}\label{sec4labelapprox}

Recall the following piece of notation that we introduced earlier. For two maps $f, g \colon X \to \mathbb{C}$, we write $f(x) \apps{\epsilon}_{L^q,x} g(x)$ to mean that $\|f-g\|_{L^q} \leq \epsilon$. If the variable and the underlying space are clear, we omit the $x$ from $\apps{\epsilon}_{L^q,x}$ and write simply $\apps{\epsilon}_{L^q}$. The main result of this section is the following theorem which says that after applying convolutions in directions $1,2, \dots, k$ to a given function on $G_{[k]}$, we get a new function which is \chg{almost constant along layers of a mutiaffine map of low codimension}. We prove this theorem using Theorem~\ref{multiaffineInvThm} for $k-1$ as an inductive hypothesis. Throughout this section we write $\chi$ for the function defined on $\mathbb F_p$ that takes $x$ to $\omega^x$, where $\omega=e^{2\pi i/p}$.

\begin{theorem}\label{MixedConvApprox} Let $G_1, \dots, G_k$ be finite-dimensional vector spaces over $\mathbb{F}_p$ and let $f\colon G_{[k]} \to \mathbb{D}$. Let $\varepsilon > 0$. Then there exist 
\begin{itemize}
\item a positive integer $r =\exp^{\big(D^{\mathrm{mh}}_{k-1} + D^{\mathrm{mh}}_{k - 2} + \dots + D^{\mathrm{mh}}_{1} + 2k\big)}\Big(O(\varepsilon^{-O(1)})\Big)$, 
\item a multiaffine map $\phi \colon G_{[k]} \to \mathbb{F}_p^r$, and
\item a map $c : \mathbb{F}_p^r \to \mathbb{D}$
\end{itemize}
such that
\[\bigconv{k}\dots\bigconv{1} f \apps{\varepsilon}_{L^2} c \circ \phi.\]
\end{theorem}

The constant $D^{\mathrm{mh}}_{k-1}$ in the above statement comes from the conclusion of Theorem~\ref{multiaffineInvThm}. The proof of Theorem \ref{MixedConvApprox} consists of two parts. Firstly, we find multilinear structure in the set of large Fourier coefficients of the mixed convolutions. In the second part of the proof we exploit this structure to obtain the \chg{desired approximation by a function that is constant along layers of a multiaffine map}.\\

\subsection{Multilinearity of large Fourier coefficients}

We begin by generalizing Lemma 13.1 from~\cite{TimSze}. We say that $(x_{[k]},y_{[k]},z_{[k]},w_{[k]})$ is a $d$-\emph{additive quadruple} if $x_i=y_i=z_i=w_i$ for every $i\ne d$ and $(x_d, y_d, z_d, w_d)$ is an additive quadruple -- that is, $x_d - y_d + z_d - w_d = 0$. We say that a map $\sigma$ \emph{respects} this $d$-additive quadruple if $\sigma(x_{[k]}) - \sigma(y_{[k]}) + \sigma(z_{[k]}) - \sigma(w_{[k]}) = 0$. Our next lemma says that convolution in a given direction gives rise to many additive quadruples in that direction, in the following sense.

\begin{lemma}\label{convImpliesLin}Let $f \colon G_{[k]} \to \mathbb{D}$ and let $d < k$. Let $\ell \in [d]$. Set
\[g = \bigconv{d} \bigconv{d - 1}\dots\bigconv{1} f,\]
let $S \subset G_{[k] \setminus \{d + 1\}}$ be a set of size at least $\delta |G_{[k] \setminus \{d + 1\}}|$, and let $\sigma\colon S \to G_{d + 1}$ be a map such that 
\[\Big|\widehat{g_{x_{[k] \setminus \{d + 1\}}}}(\sigma(x_{[k] \setminus \{d + 1\}}))\Big| \geq c\]
whenever $x_{[k] \setminus \{d + 1\}} \in S$. Then the number of $\ell$-additive quadruples respected by $\sigma$ is at least $\delta^4 c^8 |G_{\ell}|^3|G_{[k] \setminus \{\ell, d + 1\}}|$.\end{lemma}

\begin{proof}By assumption,
\begin{align}\delta c^2 \leq & \exx_{x_{[k] \setminus \{d+1\}} \in G_{[k] \setminus \{d+1\}}} S(x_{[k] \setminus \{d+1\}}) \Big|\widehat{g_{x_{[k] \setminus \{d+1\}}}}(\sigma(x_{[k] \setminus \{d+1\}}))\Big|^2\nonumber\\
=&\exx_{x_{[k] \setminus \{d+1\}} \in G_{[k] \setminus \{d+1\}}} S(x_{[k] \setminus \{d+1\}}) \exx_{u_{d+1}, v_{d+1} \in G_{d+1}} g(x_{[k] \setminus \{d+1\}}, u_{d+1}) \nonumber\\
&\hspace{6cm}\overline{g(x_{[k] \setminus \{d+1\}}, v_{d+1})} \chi\Big((v_{d+1} - u_{d+1}) \cdot \sigma(x_{[k] \setminus \{d+1\}})\Big)\nonumber\\
=&\exx_{x_{[k] \setminus \{d+1\}} \in G_{[k] \setminus \{d+1\}}} S(x_{[k] \setminus \{d+1\}}) \exx_{u_{d+1}, v_{d+1} \in G_{d+1}} \bigconv{d}\dots\bigconv{1} f(x_{[k] \setminus \{d+1\}}, u_{d+1})\nonumber\\
&\hspace{6cm}\overline{\bigconv{d}\dots\bigconv{1} f(x_{[k] \setminus \{d+1\}}, v_{d+1})} \chi\Big((v_{d+1} - u_{d+1}) \cdot \sigma(x_{[k] \setminus \{d+1\}})\Big).\nonumber\\
\end{align}
Setting $h_{d+1} = v_{d+1} - u_{d+1}$, we can rewrite the last expression as
\begin{align}\exx_{x_{[k] \setminus \{d+1\}} \in G_{[k] \setminus \{d+1\}}} \exx_{u_{d+1}, h_{d+1} \in G_{d+1}} S(x_{[k] \setminus \{d+1\}})  &\bigconv{d}\dots\bigconv{1} f(x_{[k] \setminus \{d+1\}}, u_{d+1})\nonumber\\
&\overline{\bigconv{d}\dots\bigconv{1} f(x_{[k] \setminus \{d+1\}}, u_{d+1} + h_{d+1})} \chi\Big(h_{d+1} \cdot \sigma(x_{[k] \setminus \{d+1\}})\Big).\label{linFC1st}\end{align}
For parameters $u_{d+1}, h_{d+1} \in G_{d+1}$, $y_{d} \in G_{d}^{\{0,1\}}, \dots, y_{2} \in G_{2}^{\{0,1\}^{d-1}}, x_{[k] \setminus \{d+1\}} \in G_{[k] \setminus \{d+1\}}$ and for indices $i \in [d]$ and $\bm{a} \in \{0,1\}^{d + 1 - i}$ we define the point $\bm{x}^{i, \bm{a}} \in G_{[k]}$ as follows (which depends on the parameters mentioned, but we do not stress this in the notation). We write the indices of $y_i$ in the superscript, so $y_d = (y^0_d, y^1_d)$, and so on. First, set $\bm{x}^{d, 0} = (x_{[k] \setminus \{d+1\}}, u_{d+1})$ and $\bm{x}^{d, 1} = (x_{[k] \setminus \{d+1\}}, u_{d+1} + h_{d+1})$. Then we may rewrite~\eqref{linFC1st} as 
\begin{equation}\label{linFC2nd}\delta c^2 \leq \exx_{x_{[k] \setminus \{d+1\}}, u_{d+1}, h_{d+1}} S(x_{[k] \setminus \{d+1\}})  \bigconv{d}\dots\bigconv{1} f(\bm{x}^{d, 0}) \overline{\bigconv{d}\dots\bigconv{1} f(\bm{x}^{d, 1})} \chi\Big(h_{d+1}\cdot \sigma(x_{[k] \setminus \{d+1\}})\Big).\end{equation}
Next, for $i \in [d-1]$ and $\bm{a} \in \{0,1\}^{d-i}$, define
\[\bm{x}^{i, (\bm{a}, 0)} = \Big(\bm{x}^{i + 1, \bm{a}}_{[k] \setminus \{i+1\}},\ls{i+1}\,y^{\bm{a}}_{i+1}\Big)\hspace{1cm}\text{and}\hspace{1cm}\bm{x}^{i, (\bm{a}, 1)} = \Big(\bm{x}^{i + 1, \bm{a}}_{[k] \setminus \{i+1\}}, \ls{i+1}\,\bm{x}^{i + 1, \bm{a}}_{i+1} + y^{\bm{a}}_{i+1}\Big).\]
We introduced this notation in order to obtain the identity
\begin{equation}\label{convExpnNotn}\bigconv{i+1}\dots\bigconv{1} f(\bm{x}^{i+1, \bm{a}}) = \exx_{y_{i+1}^{\bm{a}}} \bigconv{i}\dots\bigconv{1} f(\bm{x}^{i, (\bm{a}, 1)}) \overline{\bigconv{i}\dots\bigconv{1} f(\bm{x}^{i, (\bm{a}, 0)})}.\end{equation}

Returning to~\eqref{linFC2nd} and using~\eqref{convExpnNotn}, we obtain
\begin{align*}\delta c^2 \leq& \exx_{x_{[k] \setminus \{d+1\}}, u_{d+1}, h_{d+1}} S(x_{[k] \setminus \{d+1\}})  \bigconv{d}\dots\bigconv{1} f(\bm{x}^{d, 0}) \overline{\bigconv{d}\dots\bigconv{1} f(\bm{x}^{d, 1})} \chi\Big(h_{d+1}\cdot \sigma(x_{[k] \setminus \{d+1\}})\Big)\\
=&\exx_{x_{[k] \setminus \{d+1\}}, u_{d+1}, h_{d+1}} \exx_{y_{d}} S(x_{[k] \setminus \{d+1\}}) \overline{\bigconv{d-1}\dots\bigconv{1} f(\bm{x}^{d-1; 0,0})} \bigconv{d-1}\dots\bigconv{1} f(\bm{x}^{d-1; 0,1}) \\
&\hspace{5cm}\bigconv{d-1}\dots\bigconv{1} f(\bm{x}^{d-1; 1,0}) \overline{\bigconv{d-1}\dots\bigconv{1} f(\bm{x}^{d-1; 1,1})} \chi\Big(h_{d+1}\cdot \sigma(x_{[k] \setminus \{d+1\}})\Big)\\
=& \cdots\\
=&\exx_{\substack{x_{[k] \setminus \{d+1\}}\\u_{d+1}, h_{d+1}}} \exx_{y_d, \dots, y_{\ell+1}}  S(x_{[k] \setminus \{d+1\}})  \chi\Big(h_{d+1}\cdot  \sigma(x_{[k] \setminus \{d+1\}})\Big) \prod_{\bm{a} \in \{0,1\}^{d + 1 - \ell}} \Big(\operatorname{Conj}^{d\,-\, \ell\, + \sum\limits_{i\in[d - \ell + 1]}a_i} \bigconv{\ell} \dots \bigconv{1} f(\bm{x}^{\ell;\bm{a}})\Big),
\end{align*}
where $\operatorname{Conj}^e$ stands for the conjugation operator applied $e$ times (so a conjugation is performed if and only if $e$ is odd). In order to simplify the notation, let us also write $|a| = \sum\limits_{i\in[d - \ell + 1]}a_i$.\\
\indent Another step of expansion of convolutions gives
\begin{align*}\delta c^2 \leq &\exx_{\substack{x_{[k] \setminus \{d+1\}}\\u_{d+1}, h_{d+1}}} \exx_{y_d, \dots, y_\ell}  S(x_{[k] \setminus \{d+1\}})  \chi\Big(h_{d+1}\cdot  \sigma(x_{[k] \setminus \{d+1\}})\Big)\\
&\hspace{1cm}\bigg(\prod_{\bm{a} \in \{0,1\}^{d + 1 - \ell}} \Big(\operatorname{Conj}^{d\,-\, \ell\, +\, 1\, + |a|} \bigconv{\ell-1} \dots \bigconv{1} f(\bm{x}^{\ell;\bm{a}}_{[k] \setminus \{\ell\}}, \ls{\ell}\,y^{\bm{a}}_{\ell})\Big)\bigg)\\
&\hspace{1cm}\bigg(\prod_{\bm{a} \in \{0,1\}^{d + 1 - \ell}} \Big(\operatorname{Conj}^{d\,-\, \ell\, +  |a|} \bigconv{\ell-1} \dots \bigconv{1} f(\bm{x}^{\ell;\bm{a}}_{[k] \setminus \{\ell\}}, \ls{\ell}\, y^{\bm{a}}_{\ell} + \bm{x}^{\ell;\bm{a}}_{\ell})\Big)\bigg).\end{align*}
Note that $\bm{x}^{\ell;\bm{a}}_{\ell} = x_\ell$. Therefore,

\begin{align}\delta c^2 \leq &\exx_{\substack{x_{[k] \setminus \{d+1, \ell\}}\\u_{d+1}, h_{d+1}}} \exx_{y_{d}, \dots, y_{\ell}} \exx_{x_\ell}  S(x_{[k] \setminus \{d+1\}})  \chi\Big(h_{d+1}\cdot  \sigma(x_{[k] \setminus \{d+1\}})\Big)\nonumber\\
&\hspace{1cm}\bigg(\prod_{\bm{a} \in \{0,1\}^{d + 1 - \ell}} \Big(\operatorname{Conj}^{d\,-\, \ell\, +\, 1\, + |a|} \bigconv{\ell - 1} \dots \bigconv{1} f(\bm{x}^{\ell;\bm{a}}_{[k] \setminus \{\ell\}},\ls{\ell}\,y^{\bm{a}}_{\ell})\Big)\bigg)\nonumber\\
&\hspace{1cm}\bigg(\prod_{\bm{a} \in \{0,1\}^{d + 1 - \ell}} \Big(\operatorname{Conj}^{d\,-\, \ell\, + |a|} \bigconv{\ell  - 1} \dots \bigconv{1} f(\bm{x}^{\ell;\bm{a}}_{[k] \setminus \{\ell\}}, \ls{\ell}\,y^{\bm{a}}_{\ell} + x_\ell)\Big)\bigg)\nonumber\\
\leq&\exx_{\substack{x_{[k] \setminus \{d+1, \ell\}}\\u_{d+1}, h_{d+1}}} \exx_{y_{d}, \dots, y_{\ell}}\bigg|\exx_{x_\ell}  S(x_{[k] \setminus \{d+1\}})  \chi\Big(h_{d+1}\cdot  \sigma(x_{[k] \setminus \{d+1\}})\Big)\nonumber\\
&\hspace{1cm}\prod_{\bm{a} \in \{0,1\}^{d + 1 - \ell}} \Big(\operatorname{Conj}^{d\,-\, \ell\, + |a|} \bigconv{\ell-1} \dots \bigconv{1} f(\bm{x}^{ \ell ;\bm{a}}_{[k] \setminus \{\ell\}},\ls{\ell}\,y^{\bm{a}}_{\ell} + x_\ell)\Big)\bigg|\nonumber.\\
\end{align}
We introduce a new variable $t_\ell \in G_\ell$. After the change of variables $y^{\bm{a}}_{\ell} \mapsto y^{\bm{a}}_{\ell} + t_{\ell}$ for all $\bm{a}$ the expression above becomes
\begin{align}\exx_{\substack{x_{[k] \setminus \{d+1, \ell\}}\\u_{d+1}, h_{d+1}}} \exx_{y_{d}, \dots, y_{\ell}}\hspace{3pt}\exx_{t_{\ell}}&\bigg|\exx_{x_\ell}  S(x_{[k] \setminus \{d+1\}})  \chi\Big(h_{d+1}\cdot  \sigma(x_{[k] \setminus \{d+1\}})\Big)\nonumber\\
&\prod_{\bm{a} \in \{0,1\}^{d + 1 - \ell}} \Big(\operatorname{Conj}^{d\,-\, \ell\, + |a|} \bigconv{\ell - 1} \dots \bigconv{1} f(\bm{x}^{\ell;\bm{a}}_{[k] \setminus \{\ell\}},\ls{\ell}\, y^{\bm{a}}_{\ell} + t_\ell + x_\ell)\Big)\bigg|.\label{linFCmaineq}\end{align}

Write $\bm{p}$ for the tuple (of elements and sequences of elements of spaces $G_1, \dots, G_k$) 
\[\bm{p} = (x_{[k] \setminus \{d+1, \ell\}}, u_{d+1}, h_{d+1}, y_{d}, \dots, y_{\ell}),\]
and for fixed $\bm{p}$, define maps $F_{\bm{p}}, H_{\bm{p}} \colon G_\ell \to \mathbb{D}$ by
\[F_{\bm{p}}(w_\ell) = S(x_{[k] \setminus \{d+1,\ell\}}, w_\ell)  \chi\Big(h_{d+1}\cdot  \sigma(x_{[k] \setminus \{d+1, \ell\}}, w_\ell)\Big)\]
and
\[H_{\bm{p}}(w_\ell) = \prod_{\bm{a} \in \{0,1\}^{d + 1 - \ell}} \Big(\operatorname{Conj}^{d\,-\, \ell\, + |a|} \bigconv{\ell - 1} \dots \bigconv{1} f(\bm{x}^{\ell;\bm{a}}_{[k] \setminus \{\ell\}}, \ls{\ell}\,y^{\bm{a}}_{\ell} + w_\ell)\Big).\]
Rewriting~\eqref{linFCmaineq}, we obtain
\[\delta c^2 \leq \exx_{\bm{p}} \exx_{t_\ell} \Big|\exx_{x_\ell} F_{\bm{p}}(x_\ell) H_{\bm{p}}(x_\ell + t_\ell)\Big|.\]
Lemma~\ref{l4bound} yields
\begin{align*}\delta^4 c^8 &\leq \exx_{\bm{p}} \sum_{r \in G_\ell} \Big|\widehat{F_{\bm{p}}}(r)\Big|^4\\
&= \exx_{x_{[k] \setminus \{d+1, \ell\}}, h_{d+1}} \sum_{r \in G_\ell} \exx_{z^1_{\ell}, z^2_{\ell}, z^3_{\ell}, z^4_{\ell} \in G_\ell} S(x_{[k] \setminus \{d+1, \ell\}}, z^1_{\ell})S(x_{[k] \setminus \{d+1, \ell\}}, z^2_{\ell})S(x_{[k] \setminus \{d+1, \ell\}}, z^3_{\ell})S(x_{[k] \setminus \{d+1, \ell\}}, z^4_{\ell})\\
&\hspace{2cm} \chi\Big(h_{d + 1} \cdot (\sigma(x_{[k] \setminus \{d+1, \ell\}}, z^4_{\ell}) - \sigma(x_{[k] \setminus \{d+1, \ell\}}, z^3_{\ell}) + \sigma(x_{[k] \setminus \{d+1, \ell\}}, z^2_{\ell}) - \sigma(x_{[k] \setminus \{d+1, \ell\}}, z^1_{\ell}))\Big)\\
&\hspace{2cm}\chi\Big(-r \cdot (z^4_{\ell} - z^3_{\ell} + z^2_{\ell} - z^1_{\ell})\Big).\end{align*}
This is exactly the density of $\ell$-additive quadruples with the property that all their points lie in $S$ and are respected by $\sigma$.\end{proof}

\chg{For the next corollary, we say that a map $\mu \colon G_{[k]\setminus \{d + 1\}} \to G_{d + 1}$ is \textit{multiaffine in coordinates $[d]$} if the map $y_i \in G_i \mapsto \mu(x_{[k] \setminus \{i, d + 1\}},\ls{i}{y_i})$ is affine for all $i \in [d]$ and $x_{[k] \setminus \{i, d + 1\}} \in G_{[k] \setminus \{i, d + 1\}}$. We may now apply Theorem~\ref{multiaffineInvThm} in the case of $d$ variables to deduce that the map $\sigma$ in the proposition above is related to a map which is multiaffine in coordinates $[d]$.}

\begin{corollary}\label{FClinMidStep}Let $f \colon G_{[k]} \to \mathbb{D}$ be a map and let $d < k$. Write 
\[g = \bigconv{d} \dots \bigconv{1} f.\]
Let $S \subseteq G_{[k] \setminus \{d + 1\}}$ be a set of density at least $\delta$ and let $\sigma \colon S \to G_{d + 1}$ be a map such that for each $x_{[k] \setminus \{d + 1\}} \in S$,
\[\Big|\widehat{g_{x_{[k] \setminus \{d + 1\}}}}(\sigma(x_{[k] \setminus \{d + 1\}}))\Big|\geq c.\]
Then there exist $\delta'$ such that
\[{\delta'}^{-1}= \exp^{(D^{\mathrm{mh}}_{d})}\Big(\exp\Big(O((\log \delta^{-1} c^{-1})^{O(1)})\Big)\Big),\]
a subset $S' \subseteq S$ of size $|S'| \geq \delta' |G_{[k] \setminus \{d + 1\}}|$, and a map $\mu \colon G_{[k]\setminus \{d + 1\}} \to G_{d + 1}$, multiaffine in coordinates $[d]$, such that $\sigma(x_{[k] \setminus \{d+1\}}) = \mu(x_{[k] \setminus \{d+1\}})$ for each $x_{[k] \setminus \{d+1\}} \in S'$.\end{corollary}

\begin{proof}By induction on $i \in [0,d]$ we show that there is a set $S_i \subseteq S$ of density $\delta_i = \exp\Big(-O((\log \delta^{-1} c^{-1})^{O(1)})\Big)$ such that $\sigma|_{S_i}$ is a Freiman homomorphism in directions $1,2,\dots,i$. For the base case take $S_0 = S$ and $\delta_0 = \delta$. Assume that the claim holds for some $i \in [0,d-1]$ and let $S_i$ be the relevant subset. By Lemma~\ref{convImpliesLin}, there are at least $\delta_i^4 c^8 |G_{i+1}|^3 |G_{[k] \setminus \{i + 1, d + 1\}}|$ $(i+1)$-additive quadruples in $S_i$ that are respected by $\sigma$. Let $T$ be the set of all $x_{[k] \setminus \{i + 1, d + 1\}} \in G_{[k] \setminus \{i + 1, d + 1\}}$ for which there are at least $\frac{1}{2}\delta_i^4 c^8 |G_{i+1}|^3$ such additive quadruples coming from $(S_i)_{x_{[k] \setminus \{i + 1, d + 1\}}}$. Then $|T| \geq \frac{1}{2}\delta_i^4 c^8 |G_{[k] \setminus \{i + 1, d + 1\}}|$. By Corollary~\ref{FreBSG}, for each $x_{[k] \setminus \{i + 1, d + 1\}} \in T$ we may find an affine map $\alpha \colon G_{i+1} \to G_{d+1}$ and a set $Y_{x_{[k] \setminus \{i + 1, d + 1\}}} \subseteq (S_i)_{x_{[k] \setminus \{i + 1, d + 1\}}} $ of size $\exp\Big(-O((\log \delta^{-1} c^{-1})^{O(1)})\Big) |G_{i+1}|$, such that $\sigma_{x_{[k] \setminus \{i + 1, d + 1\}}}(y_{i+1}) = \alpha(y_{i+1})$ for all $y_{j+1} \in Y_{x_{[k] \setminus \{i + 1, d + 1\}}}$. Thus, taking $S_{i+1} = \bigcup_{x_{[k] \setminus \{i + 1, d + 1\}} \in T} \{x_{[k] \setminus \{i + 1, d + 1\}}\} \times Y_{x_{[k] \setminus \{i + 1, d + 1\}}} \subset S_i$ completes the proof of the inductive step.\\
\indent Once we have obtained the set $S_{d}$, for each $x_{[d + 2, k]} \in G_{[d + 2, k]}$, we see that the map $y_{[d]} \in (S_{d})_{x_{[d + 2, k]}} \mapsto \sigma(y_{[d]}, x_{[d + 2, k]}))$ is a multi-homomorphism. By averaging, we may find a set $X \subseteq G_{[d + 2, k]}$ of size $|X| \geq \frac{\delta_d}{2} |G_{[d + 2, k]}|$ such that each $x_{[d + 2, k]} \in X$ satisfies $|(S_{d})_{x_{[d + 2, k]}}| \geq \frac{\delta_d}{2} |G_{[d]}|$. For each $x_{[d + 2, k]} \in X$, by Theorem~\ref{multiaffineInvThm}, there is a global multiaffine map $\mu_{x_{[d + 2, k]}} \colon G_{[d]} \to G_k$ such that $\mu_{x_{[d + 2, k]}}(y_{[d]}) = \sigma(y_{[d]}, x_{[d + 2, k]})$ holds for at least $\Big(\exp^{(D^{\mathrm{mh}}_{d})}(O(\delta_{d}^{-1}))\Big)^{-1} |G_{[d]}|$ of $y_{[d]} \in (S_{d})_{x_{[d + 2, k]}}$. To completes the proof, take $S' = \cup_{x_{[d + 2, k]} \in X} \{x_{[d + 2, k]}\} \times (S_{d})_{x_{[d + 2, k]}}$, which is a subset of $S$ of claimed density, and write $\mu(y_{[d]}, x_{[d + 2, k]}) = \mu_{x_{[d + 2, k]}}(y_{[d]})$, to give a map which is multiaffine in coordinates $[d]$ and coincides with $\sigma$ on $S'$.\end{proof}

We apply the corollary above iteratively to cover most of pairs $(x_{[k] \setminus \{d + 1\}}, r_{d + 1})$ where $\Big|\widehat{g_{x_{[k] \setminus \{d + 1\}}}}(r_{d+1})\Big|\geq c$ by few structured pieces.

\begin{corollary}\label{FClinFinal}Let $f, k, d$ and $g$ be as in Corollary~\ref{FClinMidStep}, and let $\varepsilon > 0$. Then there exist a positive integer $m=\exp^{(D^{\mathrm{mh}}_{d} + 1)}\Big(O((\log \varepsilon^{-1})^{O(1)})\Big)$, maps $\mu_1, \dots, \mu_m \colon G_{[k] \setminus \{d + 1\}} \to G_{d+1}$, multiaffine in coordinates $[d]$, and a set $S \subset G_{[k] \setminus \{d+1\}}$ of size at least $(1-\varepsilon) |G_{[k] \setminus \{d+1\}}|$ such that $r_{d+1} \in \Big\{\mu_1(x_{[k] \setminus \{d+1\}}), \dots, \mu_m(x_{[k] \setminus \{d+1\}})\Big\}$ for every $x_{[k] \setminus \{d+1\}} \in S$ and every $r_{d+1} \in G_{d+1}$ such that $|\widehat{g_{x_{[k] \setminus \{d+1\}}}}(r_{d+1})| \geq \varepsilon$.\end{corollary}

\begin{proof}We iteratively define multiaffine maps $\mu_1, \dots, \mu_m \colon G_{[k] \setminus \{d + 1\}} \to G_{d + 1}$ as follows. If at the $i$\textsuperscript{th} step there is a set $A \subset  G_{[k] \setminus \{d + 1\}}$ of size at least $\varepsilon | G_{[k] \setminus \{d + 1\}}|$ such that for each $x_{[k] \setminus \{d + 1\}} \in A$ there exists an element $\sigma(x_{[k] \setminus \{d + 1\}}) \in G_{d + 1} \setminus \Big\{\mu_1(x_{[k] \setminus \{d + 1\}}), \dots, \mu_i(x_{[k] \setminus \{d + 1\}})\Big\}$ such that $|\widehat{g_{x_{[k] \setminus \{d + 1\}}}}(\sigma(x_{[k] \setminus \{d + 1\}}))| \geq \varepsilon$, then we may apply Corollary~\ref{FClinMidStep} to find a map $\mu_{i+1} \colon G_{[k]\setminus \{d + 1\}} \to G_{d + 1}$, multiaffine in coordinates $[d]$, such that $\sigma(x_{[k] \setminus \{d + 1\}}) = \mu_{i+1}(x_{[k] \setminus \{d + 1\}})$ holds for $\Big(\exp^{(D^{\mathrm{mh}}_{d} + 1)}\Big(O((\log \varepsilon^{-1})^{O(1)})\Big)\Big)^{-1} |G_{[k] \setminus \{d + 1\}}|$ elements $x_{[k] \setminus \{d + 1\}}$ of $A$.\\
\indent Since by Lemma~\ref{LargeFCs} the number of $\varepsilon$-large Fourier coefficients of each $g_{x_{[k] \setminus \{d + 1\}}}$ is at most $\varepsilon^{-2}$, this procedure terminates after at most $\varepsilon^{-2} \exp^{(D^{\mathrm{mh}}_{d} + 1)}\Big(O((\log \varepsilon^{-1})^{O(1)})\Big)$ steps, as desired.\end{proof}

\subsection{Obtaining the approximation}

We prove Theorem~\ref{MixedConvApprox} by gradually obtaining more and more structured approximations for mixed convolutions $\bigconv{d} \dots \bigconv{1} f$. The next result is the main proposition, from which Theorem~\ref{MixedConvApprox} follows as the special case $d = k$.

\begin{proposition}\label{MixedConvApproxAuxProp} Let $G_1, \dots, G_k$ be finite-dimensional vector spaces over $\mathbb{F}_p$ and let $f\colon G_{[k]} \to \mathbb{D}$. Let $d \leq k$ and $\varepsilon > 0$. Then there exist 
\begin{itemize}
\item a positive integer $r =\exp^{\big(D^{\mathrm{mh}}_{d-1} + D^{\mathrm{mh}}_{d - 2} + \dots + D^{\mathrm{mh}}_{1} + 2d\big)}\Big(O(\varepsilon^{-O(1)})\Big)$, 
\item a map $\phi \colon G_{[k]} \to \mathbb{F}_p^r$, which is multiaffine in coordinates $[d]$, and
\item map $c : G_{[d + 1, k]} \times \mathbb{F}_p^r \to \mathbb{D}$
\end{itemize}
such that
\[\bigconv{d}\dots\bigconv{1} f(x_{[k]}) \apps{\varepsilon}_{L^2, x_{[k]}} c(x_{[d + 1, k]}, \phi(x_{[k]})).\]
\end{proposition}

The formulation of proposition allows a proof by induction on $d$.\\

\begin{proof} We prove the claim by induction on $k$. The base $d = 1$ follows from Corollary~\ref{basicLqapp} applied to each $x_{[2,k]} \in G_{[2,k]}$.\\
\indent Assume that the claim holds for some $d < k$, giving suitable $\phi$ and $c$. Let $\eta > 0$ be a parameter to be chosen later.\\

\noindent\textbf{Step 1. Finding multilinear structure in large Fourier coefficients.} Apply Corollary~\ref{FClinFinal} to find a positive integer $m=\exp^{(D^{\mathrm{mh}}_{d} + 1)}\Big(O((\log \eta^{-1})^{O(1)})\Big)$, maps $\Phi_1, \dots ,\Phi_m : G_{[k] \setminus \{d + 1\}} \to G_{d + 1}$, multiaffine in coordinates $1, \dots, d$,  and a set $S \subset G_{[k] \setminus \{d + 1\}}$ of size at least $(1-\eta) |G_{[k] \setminus \{d + 1\}}|$ such that $r_{d + 1} \in \Big\{\Phi_1(x_{[k] \setminus \{d + 1\}}), \dots, \Phi_m(x_{[k] \setminus \{d + 1\}})\Big\}$ for every $x_{[k] \setminus \{d + 1\}} \in S$ and every $r_{d + 1} \in G_{d + 1}$ such that $|\fco (\bigconv{d} \dots {\bigconv{1}})_{x_{[k] \setminus \{d + 1\}}} \fcc (r_{d + 1})| \geq \eta$.\\

\noindent\textbf{Step 2. Controlling the vanishing linear combinations.} Apply Theorem~\ref{newthmvanishinglcs} for each $x_{[d + 2, k]} \in G_{[d + 2, k]}$ to the maps $y_{[d]} \mapsto \Phi_i(y_{[d]}, x_{[d + 2, k]})$, where $i \in [m]$. We obtain a positive integer $t \leq (O(m \log_p \eta^{-1}))^{O(1)}$, map $\beta : G_{[k] \setminus \{d + 1\}} \to \mathbb{F}_p^t$ which is multiaffine in coordinates $[d]$, a union of layers $U_{x_{[d + 2, k]}}$ of the map $y_{[d]} \mapsto \beta(y_{[d]}, x_{[d + 2, k]})$ of size $|U_{x_{[d + 2, k]}}| \geq (1 - \eta) |G_{[d + 2, k]}|$, such that for each layer $L \subseteq U_{x_{[d + 2, k]}}$ we have a subspace $\Lambda_{x_{[d + 2, k]}, L} \leq \mathbb{F}_p^m$ such that for each $y_{[d]} \in L$ the set of vanishing linear combinations of elements $\Phi_1(y_{[d]}, x_{[d + 2, k]}), \dots, \Phi_m(y_{[d]}, x_{[d + 2, k]})$ is precisely $\Lambda_{x_{[d + 2, k]}, L}$.\\

\noindent\textbf{Step 3. Applying the 1-dimensional convolution approximation.} Take any $x_{[d+2, k]}$. Take an arbitrary layer $L$ inside $U_{x_{[d + 2, k]}}$. Looking at vanishing linear combinations $\Phi_i(x_{[k] \setminus \{d + 1\}}) - \Phi_j(x_{[k] \setminus \{d + 1\}})$ for different pairs $(i, j)$, we may find a set $I_{x_{[d + 2, k]}, L} \subseteq [m]$ such that $\Phi_i(x_{[k] \setminus \{d + 1\}})$ are distinct as $i$ ranges over $I_{x_{[d + 2, k]}, L}$ and cover the whole set $\Big\{\Phi_1(x_{[k] \setminus \{d + 1\}}), \dots, \Phi_m(x_{[k] \setminus \{d + 1\}})\Big\}$ for every $x_{[d]} \in L$. Lemma~\ref{basicL2app} implies that
\begin{align}&\Big(\bigconv{d + 1} \bigconv{d} \dots \bigconv{1} f\Big)_{x_{[k]\setminus \{d + 1\}}}(x_{d + 1})\nonumber\\
&\hspace{4cm}\apps{\eta}_{L^2, x_{d + 1}} \sum_{i \in I_{x_{[d + 2, k]}, L}} \Big|\fco(\bigconv{d} \dots \bigconv{1} f)_{x_{[k]\setminus \{d + 1\}}}\fcc(\Phi_i(x_{[k] \setminus \{d + 1\}}))\Big|^2 \omega^{\Phi_i(x_{[k] \setminus \{d + 1\}}) \cdot x_{d + 1}}\label{step3approx1dimconv}\end{align}
holds for every $x_{[d]} \in L$.\\

\noindent\textbf{Step 4. Removing Fourier coefficients.} In this step, we replace the $\fco(\bigconv{d} \dots \bigconv{1} f)_{x_{[k]\setminus \{d + 1\}}}\fcc(\Phi_i(x_{[k] \setminus \{d + 1\}}))$ by more structured terms. 

\begin{claim*}
    Let $\Psi : G_{[k] \setminus \{d + 1\}} \to G_{d + 1}$ be a map which is multiaffine in coordinates $[d]$. Let $\xi > 0$. Then, there exist a set $A \subseteq G_{[k] \setminus \{d + 1\}}$ of size $|A| \geq (1- \xi)|G_{[k] \setminus \{d + 1\}}|$, a positive integer $s \leq \exp^{\big(D^{\mathrm{mh}}_{d-1} + D^{\mathrm{mh}}_{d - 2} + \dots + D^{\mathrm{mh}}_{1} + 2d\big)}\Big(O({\xi}^{-O(1)})\Big)$, a map $\psi : G_{[k] \setminus \{d + 1\}} \to \mathbb{F}_p^s$, multiaffine in coordinates $[d]$, and a map $c' : G_{[d + 2, k]} \times \mathbb{F}_p^s \to \mathbb{C}$ such that 
    \[\bigg|\Big|\fco\Big(\bigconv{d} \dots \bigconv{1} f\Big)_{x_{[k]\setminus \{d + 1\}}}\fcc(\Psi(x_{[k] \setminus \{d + 1\}}))\Big|^2\,\, - \,\,c'(x_{[d + 2, k]}, \psi(x_{[k]\setminus \{d + 1\}}))\bigg| \leq \xi\]
    for all $x_{[k] \setminus \{d + 1\}} \in A$.
\end{claim*}

\begin{proof}
    By definition 
    \begin{align}&\Big|\fco\Big(\bigconv{d} \dots \bigconv{1} f\Big)_{x_{[k]\setminus \{d + 1\}}}\fcc(\Psi(x_{[k] \setminus \{d + 1\}}))\Big|^2\nonumber\\
    &\hspace{2cm}= \exx_{y_{d + 1}, z_{d + 1} \in G_{d + 1}} \bigconv{d} \dots \bigconv{1} f(x_{[k] \setminus \{d + 1\}}, y_{d + 1})\overline{\bigconv{d} \dots \bigconv{1} f(x_{[k] \setminus \{d + 1\}}, z_{d + 1})} \omega^{\Psi(x_{[k] \setminus \{d + 1\}}) \cdot (z_{d + 1 } - y_{d + 1})}.\label{sec4dapproxih}\end{align}

    Let $\xi' > 0$ be a parameter to be chosen later. Use induction hypothesis with parameter $\xi' > 0$ instead of $\varepsilon$ to obtain $r \leq \exp^{\big(D^{\mathrm{mh}}_{d-1} + D^{\mathrm{mh}}_{d - 2} + \dots + D^{\mathrm{mh}}_{1} + 2d\big)}\Big(O({\xi'}^{-O(1)})\Big)$ and maps $\phi, c$ with properties in the proposition such that
    \[\bigconv{d}\dots\bigconv{1} f(x_{[k]}) \apps{\xi'}_{L^2, x_{[k]}} c(x_{[d + 1, k]}, \phi(x_{[k]})).\]

    By averaging, we may find a set $A_1 \subseteq G_{[k] \setminus \{d + 1\}}$ of size $|A_1| \geq (1 - \xi') |G_{[k] \setminus \{d + 1\}}|$ such that for each $x_{[k] \setminus \{d + 1\}} \in A_1$ we have
    \[\exx_{y_{d + 1}} \bigg|\bigconv{d}\dots\bigconv{1} f(x_{[k]\setminus \{d + 1\}}, y_{d+1})- c(y_{d + 1}, x_{[d + 2, k]}, \phi(x_{[k] \setminus \{d+1\}}, y_{d+1}))\bigg|^2 \leq \xi'.\]
    
    By~\eqref{sec4dapproxih} and Cauchy--Schwarz inequality, we get
    \begin{align*}
        &\bigg|\Big|\fco\Big(\bigconv{d} \dots \bigconv{1} f\Big)_{x_{[k]\setminus \{d + 1\}}}\fcc(\Psi(x_{[k] \setminus \{d + 1\}}))\Big|^2\\
        &\hspace{1cm}- \exx_{y_{d + 1}, z_{d + 1} \in G_{d + 1}} c(y_{d + 1}, x_{[d + 2, k]}, \phi(x_{[k] \setminus \{d+1\}}, y_{d+1})) \overline{\bigconv{d} \dots \bigconv{1} f(x_{[k] \setminus \{d + 1\}}, z_{d + 1})} \omega^{\Psi(x_{[k] \setminus \{d + 1\}}) \cdot (z_{d + 1 } - y_{d + 1})}\bigg|\\
        = &\bigg|\exx_{y_{d + 1} \in G_{d + 1}} \Big(\bigconv{d}\dots\bigconv{1} f(x_{[k]\setminus \{d + 1\}}, y_{d+1}) - c(y_{d + 1}, x_{[d + 2, k]}, \phi(x_{[k] \setminus \{d+1\}}, y_{d+1})) \Big)\\
        &\hspace{6cm}\bigg(\exx_{z_{d + 1} \in G_{d + 1}}\overline{\bigconv{d} \dots \bigconv{1} f(x_{[k] \setminus \{d + 1\}}, z_{d + 1})} \omega^{\Psi(x_{[k] \setminus \{d + 1\}}) \cdot (z_{d + 1 } - y_{d + 1})}\bigg)\bigg|\\
        \leq& \sqrt{\xi'}.
    \end{align*}

    Using a similar argument to replace the term $\overline{\bigconv{d} \dots \bigconv{1} f(x_{[k] \setminus \{d + 1\}}, z_{d + 1})}$ by $\overline{c(z_{d + 1}, x_{[d + 2, k]}, \phi(x_{[k] \setminus \{d+1\}}, z_{d+1}))}$ and the triangle inequality, we get 

    \begin{align*}
        &\bigg|\Big|\fco\Big(\bigconv{d} \dots \bigconv{1} f\Big)_{x_{[k]\setminus \{d + 1\}}}\fcc(\Psi(x_{[k] \setminus \{d + 1\}}))\Big|^2\\
        &\hspace{1cm}- \exx_{y_{d + 1}, z_{d + 1} \in G_{d + 1}} c(y_{d + 1}, x_{[d + 2, k]}, \phi(x_{[k] \setminus \{d+1\}}, y_{d+1})) \overline{c(z_{d + 1}, x_{[d + 2, k]}, \phi(x_{[k] \setminus \{d+1\}}, z_{d+1}))} \\
        &\hspace{8cm}\omega^{\Psi(x_{[k] \setminus \{d + 1\}}) \cdot (z_{d + 1 } - y_{d + 1})}\bigg|\\
        \leq & 2\sqrt{\xi'}.
    \end{align*}

    Let $q$ be a positive integer. Take elements $y^1_{d + 1}, \dots , y^q_{d + 1}, z^1_{d+1}, \dots, z^q_{d+1} \in G_{d+1}$ uniformly and independently at random. For each $x_{[k] \setminus \{d + 1\}} \in G_{[k] \setminus \{d + 1\}}$, by Lemma~\ref{randomsampling}, the inequality
    \begin{align*}
        &\bigg|\frac{1}{q^2}\sum_{i, j \in [q]} c(y^i_{d + 1}, x_{[d + 2, k]}, \phi(x_{[k] \setminus \{d+1\}}, y^i_{d+1})) \overline{c(z^j_{d + 1}, x_{[d + 2, k]}, \phi(x_{[k] \setminus \{d+1\}}, z^j_{d+1}))} \\
        &\hspace{8cm}\omega^{\Psi(x_{[k] \setminus \{d + 1\}}) \cdot (z^j_{d + 1 } - y^i_{d + 1})}\\
        &\hspace{2cm} -\exx_{y_{d + 1}, z_{d + 1} \in G_{d + 1}} c(y_{d + 1}, x_{[d + 2, k]}, \phi(x_{[k] \setminus \{d+1\}}, y_{d+1})) \overline{c(z_{d + 1}, x_{[d + 2, k]}, \phi(x_{[k] \setminus \{d+1\}}, z_{d+1}))} \\
        &\hspace{8cm}\omega^{\Psi(x_{[k] \setminus \{d + 1\}}) \cdot (z_{d + 1 } - y_{d + 1})}
        \bigg| \leq \xi'
    \end{align*}
    holds with probability at least $1 - 4\exp(-{\xi'}^2 q/8)$. Choose $q = 32 {\xi'}^{-3}$, so that this probability is at least $1 - \xi'$.\\
    \indent Hence, there is a choice of $y^1_{d + 1}, \dots , y^q_{d + 1}, z^1_{d+1}, \dots, z^q_{d+1} \in G_{d+1}$ such that the inequality above holds simultaneously for at least $(1 - \xi')|G_{[k] \setminus \{d + 1\}}|$ of points $x_{[k] \setminus \{d + 1\}} \in G_{[k] \setminus \{d + 1\}}$. Let $A_2$ be the set of all such $x_{[k] \setminus \{d + 1\}}$.\\

    Write $A = A_1 \cap A_2$. If $x_{[k] \setminus \{d + 1\}} \in A$, then we get 
    \begin{align*}
        &\bigg|\Big|\fco\Big(\bigconv{d} \dots \bigconv{1} f\Big)_{x_{[k]\setminus \{d + 1\}}}\fcc(\Psi(x_{[k] \setminus \{d + 1\}}))\Big|^2\\
        &\hspace{1cm}- \frac{1}{q^2}\sum_{i, j \in [q]} c(y^i_{d + 1}, x_{[d + 2, k]}, \phi(x_{[k] \setminus \{d+1\}}, y^i_{d+1})) \overline{c(z^j_{d + 1}, x_{[d + 2, k]}, \phi(x_{[k] \setminus \{d+1\}}, z^j_{d+1}))} \\
        &\hspace{8cm}\omega^{\Psi(x_{[k] \setminus \{d + 1\}}) \cdot (z^j_{d + 1 } - y^i_{d + 1})}\bigg|\\
        \leq & 2\sqrt{\xi'} + \xi' \leq 3\sqrt{\xi'}.
    \end{align*}

    Let $s = 2qr + 2q$. Define $\psi : G_{[k] \setminus \{d + 1\}} \to \mathbb{F}_p^s$ by
    \begin{align*}&\psi(x_{[k] \setminus \{d + 1\}}) = (\phi(x_{[k] \setminus \{d+1\}}, y^1_{d+1}), \dots, \phi(x_{[k] \setminus \{d+1\}}, z^q_{d+1}),\\
    &\hspace{4cm}\Psi(x_{[k] \setminus \{d + 1\}}) \cdot y^1_{d+1}, \dots, \Psi(x_{[k] \setminus \{d + 1\}}) \cdot z^q_{d+1})\end{align*}
    which is multiaffine in coordinates $[d]$. Write an element of $\mathbb{F}_p^s$ as $(\lambda_{[q]}, \lambda'_{[q]}, \mu_{[q]}, \mu'_{[q]})$, where $\lambda_i, \lambda'_j \in \mathbb{F}_q^r$ and $\mu_i, \mu_j' \in \mathbb{F}_p$. The vectors $\lambda_i, \lambda'_j$ are values taken by $\phi$ and elements $\mu_i, \mu'_j$ are values taken by $\Psi \cdot y^i_{d+1}$ and $\Psi \cdot z^j_{d+1}$. Define a map $c' : G_{[d + 2, k]} \times \mathbb{F}_p^s \to \mathbb{C}$ by
    \[c'\Big(x_{[d + 2, k]}, (\lambda_{[q]}, \lambda'_{[q]}, \mu_{[q]}, \mu'_{[q]})\Big) =  \frac{1}{q^2}\sum_{i, j \in [q]} c(y^i_{d + 1}, x_{[d + 2, k]}, \lambda_i) \overline{c(z^j_{d + 1}, x_{[d + 2, k]}, \lambda'_j)} \omega^{\mu'_j - \mu_i}.\]
    With these definitions, we have
    \begin{align*}&c'(x_{[d + 2, k]}, \psi(x_{[k] \setminus \{d + 1\}})) = \frac{1}{q^2}\sum_{i, j \in [q]} c(y^i_{d + 1}, x_{[d + 2, k]}, \phi(x_{[k] \setminus \{d+1\}}, y^i_{d+1})) \overline{c(z^j_{d + 1}, x_{[d + 2, k]}, \phi(x_{[k] \setminus \{d+1\}}, z^j_{d+1}))} \\
    &\hspace{8cm}\omega^{\Psi(x_{[k] \setminus \{d + 1\}}) \cdot (z^j_{d + 1 } - y^i_{d + 1})}.\end{align*}
    By the work above, provided we take $\xi' = 2^{-4}\xi^2$, the claim follows.
\end{proof}

Apply the claim above for each $\Phi_i$ and error parameter $\eta/m$, giving a set $A_i$, integer
\[s \leq \exp^{\big(D^{\mathrm{mh}}_{d-1} + D^{\mathrm{mh}}_{d - 2} + \dots + D^{\mathrm{mh}}_{1} + 2d\big)}\Big(O({(\eta/m)}^{-O(1)})\Big)\]
and maps $c_i : G_{[d + 2, k]} \times \mathbb{F}_p^s \to \mathbb{D}$ and $\psi_i : G_{[k]\setminus \{d + 1\}} \to \mathbb{F}_p^s$.\\

\noindent\textbf{Step 5. Obtaining the desired approximation.} Let $B \subseteq G_{[k] \setminus \{d + 1\}}$ be the set 
\[B = \Big(\cap_{i \in [m]} A_i\Big) \cap \Big(\cup_{x_{[d + 2, k]} \in G_{[d + 2, k]}} U_{x_{[d + 2, k]}}\Big).\]
Note that $|B| \geq (1 - 2\eta)$.\\
\indent Recall that $U_{x_{[d + 2, k]}}$ is a union of layers of the map $y_{[d]} \mapsto \beta(y_{[d]}, x_{[d + 2, k]})$. With this in mind, we misuse the notation and for a value $\lambda \in $ of this map, we write $\lambda \in U_{x_{[d + 2, k]}}$ if the layer corresponding to $\lambda$ is actually contained inside $ U_{x_{[d + 2, k]}}$, that is $\{y_{[d]} \in G_{[d]} : \beta(y_{[d]}, x_{[d + 2, k]}) = \lambda\} \subseteq U_{x_{[d + 2, k]}}$. We also write $ I_{x_{[d + 2, k]}, \lambda}$ instead of $I_{x_{[d + 2, k]}, L}$, when $L$ is the layer $\{y_{[d]} \in G_{[d]} : \beta(y_{[d]}, x_{[d + 2, k]}) = \lambda\}$.\\

Define the map $\phi : G_{[k]} \to \mathbb{F}_p^{t + ms + m}$ by
\[\phi(x_{[k]}) = \Big(\beta(x_{[k] \setminus \{d + 1\}}), \psi_1(x_{[k] \setminus \{d + 1\}}), \dots, \psi_m(x_{[k] \setminus \{d + 1\}}), \Phi_1(x_{[k] \setminus \{d + 1\}}) \cdot x_{d + 1}, \Phi_m(x_{[k] \setminus \{d + 1\}}) \cdot x_{d + 1}\Big).\]
Note that $\phi$ is multiaffine in coordinates $[d + 1]$.\\
\indent We organize vectors in $\mathbb{F}_p^{t + ms + m}$ as $(\lambda, \mu_1, \dots, \mu_m, \nu_1, \dots, \nu_m)$, where $\lambda \in \mathbb{F}_p^t$, $\mu_i \in \mathbb{F}_p^s$ and $\nu_i \in \mathbb{F}_p$.\\
Define also the map $c : G_{[d+2, k]} \to \mathbb{F}_p^{t + ms + m} \to \mathbb{C}$
\[c\Big(x_{[d + 2, k]}, (\lambda, \mu_1, \dots, \mu_m, \nu_1, \dots, \nu_m)\Big) = \id(\lambda \in U_{x_{[d + 2, k]}}) \sum_{i \in I_{x_{[d + 2, k]}, \lambda}} c_i(x_{[d + 2, k]}, \mu_i) \omega^{\nu_i}.\]
Note that the codomain of $c$ is $\mathbb{C}$, we shall pass to the unit disk at the end of the proof. By the work above, for each $x_{[k] \setminus \{d + 1\}} \in B$, from~\eqref{step3approx1dimconv}
\[\Big(\bigconv{d + 1} \bigconv{d} \dots \bigconv{1} f\Big)_{x_{[k]\setminus \{d + 1\}}}(x_{d + 1}) \apps{2\eta}_{L^2, x_{d + 1}} c \circ \phi(x_{[k]}).\]
Finally, set $\eta = \varepsilon/2$ and replace $c$ by $c/|c|$ for values greater than 1 in modulus. By Fact~\ref{resttodiskfact}, this only improves the approximation.\end{proof}

\subsection{Further convolutions}
We now show that, as one would expect, further convolutions can only help with the approximation.

\begin{theorem}\label{strongMixedApprox}Let $f \colon G_{[k]} \to \mathbb{D}$, let $d_1, \dots, d_r \in [k]$ be directions and let $\varepsilon > 0$. Then there exist
\begin{itemize}
\item a positive integer $\ell = \exp^{\big(D^{\mathrm{mh}}_{k-1} + D^{\mathrm{mh}}_{k - 2} + \dots + D^{\mathrm{mh}}_{1} + 2k + r\big)}\Big(O(2^{O(r)}\varepsilon^{-O(1)})\Big)$, 
\item a map $c: \mathbb{F}_p^\ell \to \mathbb{D}$, and
\item multiaffine map $\phi \colon G_{[k]} \to \mathbb{F}_p^\ell$ 
\end{itemize}
such that
\[\bigconv{d_r} \dots \bigconv{d_1}\bigconv{k}\dots\bigconv{1} f \apps{\varepsilon}_{L^2} c \circ \phi.\]
\end{theorem}

\begin{proof}Without loss of generality $\varepsilon \leq 1$. Apply Theorem~\ref{MixedConvApprox} for the error parameter $\varepsilon/4^r$ to find 
\[\ell^{(0)} = \exp^{\big(D^{\mathrm{mh}}_{d-1} + D^{\mathrm{mh}}_{d - 2} + \dots + D^{\mathrm{mh}}_{1} + 2d\big)}\Big(O(2^{O(r)}\varepsilon^{-O(1)})\Big),\]
a multiaffine map $\phi^{(0)} \colon G_{[k]} \to \mathbb{F}^{\ell^{(0)}}_p$ such that
\[\bigconv{k}\dots\bigconv{1} f\apps{\varepsilon/4^r}_{L^{2}} c^{(0)} \circ\phi^{(0)}.\]
Write $\varepsilon_s = \varepsilon / 4^{r-s}$. By induction on $s \in [0,r]$, we now show that there exist a positive integer $\ell^{(s)} = \exp^{(s)}\Big(2^{O(s)}\ell^{(0)}\Big)$, multiaffine map $\phi^{(s)}\colon G_{[k]} \to \mathbb{F}_p^{\ell^{(s)}}$ and a map $c^{(s)} : \mathbb{F}_p^{\ell^{(s)}} \to \mathbb{D}$ such that
\[\bigconv{d_s} \dots \bigconv{d_1}\bigconv{k}\dots\bigconv{1} f\apps{\varepsilon_s}_{L^{2}} c^{(s)}\circ \phi^{(s)}.\]
The base case is already proved. Suppose that the claim holds for some $s \in [0, r-1]$. Let $g(x_{[k]}) = c^{(s)}\circ \phi^{(s)}(x_{[k]}))$. Apply Lemma~\ref{2pConvpBnd} to obtain that
\[\bigconv{d_{s+1}} \dots \bigconv{d_1}\bigconv{k}\dots\bigconv{1} f \apps{2\varepsilon_{s}}_{L^{2}} \bigconv{d_{s+1}} g.\]
By Proposition~\ref{ExpSumConvComp}, there exist a positive integer
\[\ell^{(s+1)} = \Big((\ell^{(s)}\log_p (2 \varepsilon^{-1}_s))^{O(1)}\Big)^{O(1)} \leq \exp^{(s + 1)}\Big(2^{O(s + 1)}\ell^{(0)}\Big),\]
multiaffine map $\phi^{(s + 1)}\colon G_{[k]} \to \mathbb{F}_p^{\ell^{(s + 1)}}$ and a map $c^{(s+1)} : \mathbb{F}_p^{\ell^{(s + 1)}} \to \mathbb{D}$ such that
\[\bigconv{d_{s+1}} g  \apps{2\varepsilon_{s}}_{L^{2}} c^{(s+1)} \circ \phi^{(s+1)}.\]
The claim follows after an application of the triangle inequality for the $L^{2}$ norm.\end{proof}

\section{The existence of respected arrangements}\label{sec5labelexresparr}

We now give a formal definition of an \emph{arrangement} of points in $G_{[k]}$. We begin by defining an \emph{$\emptyset$-arrangement} of lengths $l_{[k]} \in G_{[k]}$ to be a sequence of length 1 that consists of the single term $l_{[k]}$. Then, given a sequence $d_1, \dots, d_r$ of elements of $[k]$, a \emph{$(d_r, d_{r-1}, \dots, d_1)$-arrangement} of lengths $l_{[k]} \in G_{[k]}$ is a sequence $(q_1,q_2)$ of length $2^r$ obtained by  concatenating two $(d_{r-1}, \dots, d_1)$-arrangements $q_1$ and $q_2$ (for $r = 1$, $q_1$ and $q_2$ are $\emptyset$-arrangements), where $q_1$ has lengths $(l_{[d_r-1]}, l_{d_r} + y, l_{[d_r+1,k]})$ and $q_2$ has lengths $(l_{[d_r-1]}, y, l_{[d_r+1,k]})$ for some $y \in G_{d_r}$.\\

If additionally we are given a map $\phi$ and an arrangement $q$ whose points lie in $\dom \phi$, we define $\phi(q)$ recursively as $\phi(l_{[k]})$ if $q$ is an $\emptyset$-arrangement, and $\phi(q) = \phi(q_1) - \phi(q_2)$ if $q$ is the concatenation of $q_1$ and $q_2$ as above. Recall that a multi-$k$-homomorphism is a function that restricts to a Freiman homomorphism of order $k$ whenever all but one of the coordinates are fixed (so what we have been calling a multi-homomorphism is a multi-2-homomorphism). The main result of this section is the following theorem which says that, provided we are given a multi-homomorphism $\phi$, for many choices of lengths $l_{[k]}$ there is a value $v$ such that many arrangements of lengths $l_{[k]}$ have $\phi$-value equal to $v$. 

\begin{theorem}\label{exArrThm}Let $A \subset G_{[k]}$ be a set of density $\delta$, and let $\phi \colon A \to H$ be a multi-homomorphism. Then there are
\begin{itemize}
\item a set $B \subset G_{[k]}$ of density $\Omega(\delta^{O(1)})$, and
\item a map $\psi \colon B \to H$,
\end{itemize}
such that for each $l_{[k]} \in B$ there are $\Omega(\delta^{O(1)}|G_{[k]}|^{2^{2k} - 1})$ $k$-tuples $(q^{(1)}, \dots, q^{(k)})$ with the property that $q^{(i)}$ is an $(i-1, i-2, \dots, 1, k, \dots, 1, k, \dots, i)$-arrangement\footnote{There are $2k$ directions in the description of the arrangement.} of lengths $l_{[k]}$ whose points lie inside $A$ and $\phi(q^{(i)}) = \psi(l_{[k]})$.
\end{theorem} 

For this proof, we also need a more structured version of arrangements. To this end, we define a \emph{grid} $\Gamma$ to be a product $X_1 \tdt X_k$, where each $X_i = (x_{i, 1}, x_{i,2}, \dots, x_{i,r_i})$ is a tuple in $G_i$. Thus, a grid $\Gamma$ consists of $r_1 \cdots r_k$ points, each of the form $\Gamma_{j_{[k]}} = (x_{1,j_1}, \dots, x_{k, j_k})$, where $j_i \in [r_i]$ (the indexing in $\Gamma$ is inherited from indexing in each $X_i$). We shall consider only grids such that the cardinality of each $X_i$ is a power of 2. We define the \emph{$d$-halves} of a grid $\Gamma$ to be the pair of grids $X_{[d-1]} \times Y \times X_{[d+1,k]}$ and $X_{[d-1]} \times Z \times X_{[d+1,k]}$, where $Y$ is the tuple consisting of the first $r_d/2$ elements of $X_d$ and $Z$ is the tuple of the last $r_d/2$ elements of $X_d$. The \emph{lengths} of a grid are given by a sequence $l_{[k]}$ that is defined recursively for each $d \in [k]$ by $l_d = l(X_d) = l(Y_d) - L(Z_d)$, where $Y_d$ is the first half of the tuple $X_d$ and $Z_d$ is the second half, provided $|X_d| \geq 2$\chg{, where $|X_d|$ stands for the length of the tuple $X_d$}. If $|X_d| = 1$, we take $l_d$ to be the single element in $X_d$. (Thus, $l_d$ is a $\pm 1$-combination of the elements of $X_d$, and the signs are given by the Morse sequence.)\\

Now let $A \subset G_{[k]}$ be a subset of density $\delta$, and fix some directions $d_{[r]}$ and constants $\eta_{[r]}$. For each $i \in [r]$, let $s^i_{[k]}$ be the sequence of quantities $s^i_j = 2^{|\{j' \in [i] \colon d_{j'} = j\}|}$ for $j \in [k]$. We say that a grid $\Gamma = X_{[k]}$ is $i$-\emph{adequate} if $|X_j| = s^i_j$ for all $j \in [k]$ and all points of the grid lie in $A$. We now recursively define objects that we call $(d_{[i]}, \eta_{[i]})$-candidate grids and $(d_{[i]}, \eta_{[i]})$-good grids with respect to $A$. First, each point (or more precisely singleton grid) that lies in $A$ is $(\emptyset, \emptyset)$-\emph{good}. Next, a grid $\Gamma = X_{[k]}$ is a $(d_{[i]}, \eta_{[i]})$-\emph{candidate} if it is $i$-adequate and both its $d_i$-halves are $(d_{[i-1]}, \eta_{[i-1]})$-good. Secondly, $\Gamma = X_{[k]}$ is $(d_{[i]}, \eta_{[i]})$-\emph{good}, if it is a $(d_{[i]}, \eta_{[i]})$-candidate and the number of $(d_{[i]}, \eta_{[i]})$-candidate grids of the form $X_{[d_i-1]} \times Y \times X_{[d_i + 1, k]}$ with $l(Y) = l(X_{d_i})$ is at least $\eta_i |G_{d_i}|^{|X_{d_i}| - 1}$.

\begin{lemma}\label{adeqGrid}The number of $i$-adequate grids is at least $\delta^{2^i} \prod_{j \in [k]} |G_j|^{s^i_j}$.\end{lemma}

\begin{proof}We prove the claim by induction on $i$. For $i = 0$, the claim is trivial. Suppose the claim holds for some $i \geq 0$. Let $n_{X_{[k] \setminus \{d_i\}}}$ be the number of $(i-1)$-adequate grids of the form $X_1 \tdt X_{d_i - 1} \times Y_{d_i} \times X_{d_i + 1} \tdt X_k$ for a suitable tuple $Y_{d_i}$ in $G_{d_i}$. By the induction hypothesis, 
\[\sum_{X_{[k] \setminus \{d_i\}}} n_{X_{[k] \setminus \{d_i\}}} \geq \delta^{2^{i-1}} \prod_{j \in [k]} |G_j|^{s^{i-1}_j},\]
where $X_j$ in the sum ranges over all $s^{i-1}_j$-tuples in $G_j$. We are interested in the quantity
\[\sum_{X_{[k] \setminus \{d_i\}}} n_{X_{[k] \setminus \{d_i\}}}^2.\]
The desired bound follows from Cauchy--Schwarz inequality.\end{proof}

\begin{lemma}\label{goodGrid}Suppose that the numbers $\eta_{[r]}$ satisfy that $\eta_{i+1} \geq 4\eta_i$. Then the number of $i$-adequate grids $\Gamma = X_{[k]}$ that are not $(d_{[i]}, \eta_{[i]})$-good is at most $2\eta_i \prod_{j \in [k]} |G_j|^{s^i_j}$.\end{lemma}

\begin{proof}We prove the claim by induction on $i$. Since all points in $A$ are $(\emptyset, \emptyset)$-\emph{good}, the base case (that is, the case $i=0$) trivially holds. Assume that the claim holds for some $i-1$, so the number of $i$-adequate grids that are not $(d_{[i]}, \eta_{[i]})$-candidates is at most $4\eta_{i-1}\prod_{j \in [k]} |G_j|^{s^i_j}$. To count $i$-adequate grids that are $(d_{[i]}, \eta_{[i]})$-candidates but not $(d_{[i]}, \eta_{[i]})$-good, we set $n_{X_{[k] \setminus \{d_i\}}; l}$ to be the number of $i$-adequate grids of the form $X_1 \tdt X_{d_i - 1} \times Y_{d_i} \times X_{d_i + 1} \tdt X_k$ such that $l(Y_{d_i}) = l$, and we set $n'_{X_{[k] \setminus \{d_i\}}; l}$ to be the number of $i$-adequate grids of the form $X_1 \tdt X_{d_i - 1} \times Y_{d_i} \times X_{d_i + 1} \tdt X_k$ such that $l(Y_{d_i}) = l$ that are additionally $(d_{[i]}, \eta_{[i]})$-candidates. Then 
\begin{equation}\label{candGridEqn}\sum_{\substack{X_{[k] \setminus \{d_i\}}\\l \in G_{d_i}}} \Big(n_{X_{[k] \setminus \{d_i\}}; l} - n'_{X_{[k] \setminus \{d_i\}}; l}\Big) \leq 4\eta_{i-1}\prod_{j \in [k]} |G_j|^{s^i_j},\end{equation}
where $X_j$ in the sum ranges over all $s^{i}_j$-tuples in $G_j$. Observe that for fixed $X_{[k] \setminus \{d_i\}}$ and $l$, if $n'_{X_{[k] \setminus \{d_i\}}; l} \geq \eta_i |G_{d_i}|^{s^i_{d_i}- 1}$, then in fact all $(d_{[i]}, \eta_{[i]})$-candidate grids of the form $X_1 \tdt X_{d_i - 1} \times Y_{d_i} \times X_{d_i + 1} \tdt X_k$ with $l(Y_{d_i}) = l$ are $(d_{[i]}, \eta_{[i]})$-good. Hence, the number of $i$-adequate grids that are $(d_{[i]}, \eta_{[i]})$-candidates but are not $(d_{[i]}, \eta_{[i]})$-good is at most
\[\sum_{\substack{X_{[k] \setminus \{d_i\}}\\l \in G_{d_i}\\n'_{X_{[k] \setminus \{d_i\}}; l} \leq \eta_i |G_{d_i}|^{s^i_{d_i}- 1}}} n_{X_{[k] \setminus \{d_i\}}; l},\] 
which by~\eqref{candGridEqn} is at most
\[4\eta_{i-1}\prod_{j \in [k]} |G_j|^{s^i_j} + \sum_{\substack{X_{[k] \setminus \{d_i\}}\\l \in G_{d_i}\\n'_{X_{[k] \setminus \{d_i\}}; l}\ \leq\ \eta_i |G_{d_i}|^{s^i_{d_i}- 1}}} n'_{X_{[k] \setminus \{d_i\}}; l} \quad \leq\quad  (4\eta_{i-1} + \eta_i)\prod_{j \in [k]} |G_j|^{s^i_j}.\]
The proof is now complete.\end{proof}

For a map $\phi \colon A \to H$ and a grid $\Gamma = X_1 \tdt X_k$ with points that lie in $A$, we define $\phi(\Gamma)$ recursively as follows. If $Y_d$ and $Z_d$ are the first and second half of $X_d$, then $\phi(\Gamma) = \phi(X_1 \tdt X_{d-1} \times Y_d \times X_{d+1} \tdt X_k) - \phi(X_1 \tdt X_{d-1} \times Z_d \times X_{d+1} \tdt X_k)$. If $\Gamma = \{x_{[k]}\}$ is a singleton, then $\phi(\Gamma) = \phi(x_{[k]})$. Note that $\phi(\Gamma)$ is well-defined. To see this, write $|e|$ for the number of ones in a binary sequence $e$. Enumerate each tuple $X_i$ of size $2^{a_i}$ using binary sequences $e \in \{0,1\}^{a_i}$, ordered by $\sum_{j \in [a_i]} e_j 2^j$. Thus $X_i = (x_{i, e})_{e \in \{0,1\}^{a_i}}$. Then $\phi(\Gamma)$ becomes 
\[\phi(\Gamma) = \sum_{e^1 \in \{0,1\}^{a_1}, \dots, e^k \in \{0,1\}^{a_k}} (-1)^{|e^1| + \dots + |e^k|} \phi(x_{1,e^1}, \dots, x_{k, e^k}),\]
which is independent of the choice of the order of directions in computing the value.\\

\indent The relevance of $(d_{[i]}, \eta_{[i]})$-candidates and $(d_{[i]}, \eta_{[i]})$-good grids stems from the following fact.

\begin{lemma}\label{gridToArr}Let $i\geq 0$, let $\mathcal{G}_i$ be the set of $(d_{[i]}, \eta_{[i]})$-good grids, and let $\phi \colon A \to H$ be a multi-$2^i$-homomorphism. Then for each $\Gamma \in \mathcal{G}_i$ there is a set $\mathcal{A}^i_{\Gamma}$ of $(d_{[i]})$-arrangements whose points lie in $A$ and whose lengths are the same as those of $\Gamma$, such that
\begin{itemize}
\item[\textbf{(i)}] $|\mathcal{A}^i_{\Gamma}| \geq \prod_{j \in [i]} (\eta_j |G_{d_j}|)^{2^{i - j}}$, and
\item[\textbf{(ii)}] $\phi(q) = \phi(\Gamma)$ for every $q \in \mathcal{A}^i_{\Gamma}$.
\end{itemize}
\end{lemma}

\begin{proof}For $i = 0$, $\Gamma$ is a single point and we simply set $\mathcal{A}^0_{\Gamma} = \{\Gamma\}$. Now suppose that the claim holds for some $i-1 \geq 0$. Let $\Gamma = X_1 \tdt X_k$ be a $(d_{[i]}, \eta_{[i]})$-good grid with lengths $l_{[k]}$. We define $\mathcal{A}^i_{\Gamma}$ as follows. First, let $M_{d_i}$ be the set of all $m_{d_i} \in G_{d_i}$ such that there are tuples $Y_{d_i}, Z_{d_i}$ in $G_{d_i}$ with $l(Y_{d_i}) = l_{d_i} + m_{d_i}$ and $l(Z_{d_i}) = m_{d_i}$ and $X_1 \tdt X_{d_i - 1} \times Y_{d_i} \times X_{d_i + 1} \tdt X_{k}$ and $X_1 \tdt X_{d_i - 1} \times Z_{d_i} \times X_{d_i + 1} \tdt X_{k}$ are $(d_{[i-1]}, \eta_{[i-1]})$-good. Fix arbitrary such $(d_{[i-1]}, \eta_{[i-1]})$-good grids for $m_{d_i}$ and denote them by $\Gamma_1(m_{d_i})$ and $\Gamma_2(m_{d_i})$ respectively. Since $\Gamma$ is $(d_{[i]}, \eta_{[i]})$-good, we have $|M_{d_i}| \geq \eta_i |G_{d_i}|$. Then define $\mathcal{A}^i_{\Gamma}$ to be the set of all concatenations $(q_1, q_2)$ where $m_{d_i} \in M_{d_i}$, $q_1 \in \mathcal{A}^{i-1}_{\Gamma_1(m_{d_i})}$ and $q_2 \in \mathcal{A}^{i-1}_{\Gamma_2(m_{d_i})}$. It remains to check that this collection of sets satisfies the properties claimed.\\
\noindent\textbf{Property (i).} Note that if $q = (q_1, q_2) \in \mathcal{A}^i_{\Gamma}$, then $m_{d_i}$ is determined by the corresponding length of $q_2$, so each $q$ comes from exactly one $m_{d_i}$. By the inductive hypothesis, we have 
\[|\mathcal{A}^i_{\Gamma}| = \sum_{m_{d_i} \in M_{d_i}} |\mathcal{A}^{i-1}_{\Gamma_1(m_{d_i})}||\mathcal{A}^{i-1}_{\Gamma_2(m_{d_i})}| \geq \eta_i |G_{d_i}| \cdot \Big( \prod_{j \in [i-1]} (\eta_j |G_{d_j}|)^{2^{i - 1 - j}}\Big)^2 = \prod_{j \in [i]} (\eta_j |G_{d_j}|)^{2^{i - j}}.\]
\noindent\textbf{Property (ii).} Let $q = (q_1, q_2) \in \mathcal{A}^i_\Gamma$. Then there are $d_i$-halves $\Gamma_1$ and $\Gamma_2$ of some $\Gamma' = X_1 \tdt X_{d_i - 1} \times Y_{d_i} \times X_{d_i + 1} \tdt X_{k}$ where $l(Y_{d_i}) = l(X_{d_i})$ such that $q_1 \in  \mathcal{A}^{i-1}_{\Gamma_1}$ and $q_2 \in  \mathcal{A}^{i-1}_{\Gamma_2}$. Since $\phi$ is a $2^i$-homomorphism in direction $d_i$, we have $\phi(\Gamma) = \phi(\Gamma')$. By definition, $\phi(\Gamma') = \phi(\Gamma_1) - \phi(\Gamma_2)$. By the inductive hypothesis and the definition of $\phi$ on arrangements, we get that $\phi(q) = \phi(q_1) - \phi(q_2) = \phi(\Gamma_1) - \phi(\Gamma_2)$, as required.\end{proof}

\begin{proof}[Proof of Theorem~\ref{exArrThm}]Set $\eta_i = \frac{8^i\delta^{4^k}}{8^{2k+1}k}$. Let
\[(d^i_{2k}, \dots, d^i_1) = (i-1, i-2, \dots, 1, k, \dots, 1, k, \dots, i).\]
Then, by Lemmas~\ref{adeqGrid} and~\ref{goodGrid}, the number of $4 \times \dots \times 4$ grids $\Gamma$ that are $(\eta_{[2k]}, d^i_{[2k]})$-good for all $i \in [k]$ (the property of being $2k$-adequate with respect to the sequence $d^i_{[2k]}$ is the same for each $i \in [k]$) is at least $\frac{\delta^{4^k}}{2}|G_{[k]}|^4$. Hence, there is a set $B \subset G_{[k]}$ of density at least $\frac{\delta^{4^k}}{2}$ such that for each $l_{[k]} \in B$, there is a grid $\Gamma$ of lengths $l_{[k]}$ which is $(\eta_{[2k]}, d^i_{[2k]})$-good for all $i \in [k]$. Define $\psi(l_{[k]}) = \phi(\Gamma)$ for such $\Gamma$. Apply Lemma~\ref{gridToArr} to finish the proof.\end{proof}

\section{Densification of respected tuples of arrangements}\label{sec6densification}

In this section, we show that it is possible to pass to a subset where almost all arrangements of same length have the same $\phi$ value. Recall the $\con$ notation from the preliminary section of the paper and the discussion around~\eqref{fullConEqn}.

\begin{theorem}\label{densificationThm}Let $A \subset G_{[k]}$ be a set of density $\delta$, and let $\phi \colon A \to H$ be a multi-$2^{2k}$-homomorphism. Assume that $|G_i| \geq k p^{k(3^{k} + 1)}\delta^{- \con_{k, p}}$ for each $i \in [k]$, and let $\varepsilon > 0$. For each subset $A' \subset A$ let $\mathcal{Q}=\mathcal{Q}(A')$ be the set of all $k$-tuples $(q_1, \dots, q_k)$ for which each $q_i$ is an $([i,1], [k,1], [k,i + 1])$-arrangement with points in $A'$ and $q_1, \dots, q_k$ have the same lengths. Then there exists a subset $A'\subset A$ such that 
\[|\{(q_1, \dots, q_k) \in \mathcal{Q} \colon \phi(q_1) = \dots = \phi(q_k)\}| \,\,\geq (1-\varepsilon) |\mathcal{Q}|\,\, \geq (\delta\varepsilon)^{O(1)}|G_{[k]}|^{4^k}.\] 
\end{theorem}

We say that $(q_1, \dots, q_k)$ is \emph{generic} if after omitting the first point\footnote{Recall that an arrangement is a sequence, so has an indexing of points.} from each of $q_2, \dots, q_k$ the remaining $k 4^k - (k-1)$ points are linearly independent as elements of $G_1 \otimes \dots \otimes G_k$. We also say that $(q_1, \dots, q_k)$ is \emph{respected} if $\phi(q_1) = \dots = \phi(q_k)$. Thus, the theorem above claims that we can guarantee that there are several tuples of arrangements, and most of them are respected.

\begin{lemma}\label{nongenNum}The number of non-generic $(q_1, \dots, q_k)$ is at most $p^{k (4^k + 1)} \Big(\frac{1}{|G_1|} + \dots + \frac{1}{|G_k|}\Big) |G_{[k]}|^{4^k}$.\end{lemma}

Before we proceed with the proof, we explain how to parametrize $([i,1], [k,1], [k,i + 1])$-arrangements. For fixed lengths $l_{[k]}$ and parameters $u^{i,1} \in G_i, u^{i, 2} \in G_{i-1}^{\{0,1\}}, \dots, u^{i, 2k} \in G_{i+1}^{\{0,1\}^{2k-1}}$, we can define an $([i,1], [k,1], [k,i + 1])$-arrangement with points $(x^{i, \varepsilon})_{\varepsilon \in \{0,1\}^{2k}}$ by
\begin{align*}x^{i, \varepsilon}_{i + 1 - j} =\, &\varepsilon_{j + k} (\varepsilon_j l_{i + 1 - j} + u^{i, j}_{\varepsilon|_{[j-1]}}) + u^{i, j+k}_{\varepsilon|_{[j + k-1]}}\\
=\, &\varepsilon_{j + k}\varepsilon_{j}l_{i + 1 - j} + \varepsilon_{j + k} u^{i, j}_{\varepsilon|_{[j-1]}} + u^{i, j+k}_{\varepsilon|_{[j + k-1]}} \end{align*}
for $j \in [k]$, where the arithmetic in the coordinate index is carried out modulo $k$ and the points are ordered according to their image under the map $\varepsilon \mapsto (1-\varepsilon_1)2^{2k-1} + (1-\varepsilon_2)2^{2k-2} + \dots + (1-\varepsilon_{2k})$. Thus, the first point is indexed by $(1,1,\dots,1)$. The parameters $u^{i,1}, \dots, u^{i, 2k}$ arise naturally out of the recursive definition of arrangements.

\begin{proof}[Proof of Lemma~\ref{nongenNum}]Let $\mathcal{I} = ([k] \times \{0,1\}^{2k}) \setminus ([2,k] \times \{(1, \dots, 1)\})$. Write $S_d \subset G_d$ for the multiset of parameters that belong to $G_d$. That is, $S_d$ consists of $l_d$ and $u^{i,j}_\varepsilon$ for each $i \in [k]$, each $j \in [2k]$, and each $\varepsilon \in \{0,1\}^{j-1}$ such that $k | i + 1 - j - d$. We claim that if the parameters in $S_d$ are linearly independent for every $d \in [k]$, then the points $(x^{i, \varepsilon})_{(i, \varepsilon) \in \mathcal{I}}$ are linearly independent in $G_1 \otimes \dots \otimes G_k$. Assume that fact for now. Note that for each $d \in [k]$, there are at most $p^{k (4^k + 1)}$ different linear combinations that can be satisfied by elements in $S_d$. Thus, there are at most
\[p^{k (4^k + 1)} \Big(\frac{1}{|G_1|} + \dots + \frac{1}{|G_k|}\Big) |G_{[k]}|^{4^k}\]
choices of parameters such that some $S_d$ is not linearly independent.\\
\indent Thus, assume the linear independence of each $S_d$ and therefore that each $S_d$ is a proper set instead of a multiset.\\
\indent In what follows we shall write $u^{i, j}_{\varepsilon|_{[j-1]}}$ even when $j \leq 0$, when we interpret it to be $l_d$, where $d \in [k]$ satisfies $k | i + 1 - j - d$. We now partition $\mathcal{I} = \mathcal{I}_{1} \cup \dots \cup \mathcal{I}_{2k + 1}$ as follows. For $j \in [2k]$, set $\mathcal{I}_j = \Big\{\Big(i, (\varepsilon', 0, 1, \dots, 1)\Big) \colon i \in [k], \varepsilon' \in \{0,1\}^{2k - j}\Big\}$ (the zero is followed by $j-1$ ones) and $\mathcal{I}_{2k+1} = \{(1, (1, 1, \dots, 1))\}$. We prove that the given points are linearly independent by showing that the points $(x^{i, \varepsilon})_{(i, \varepsilon) \in \mathcal{I}_{1} \cup \dots \cup \mathcal{I}_j}$ are linearly independent for each $j$. But if we set $j' = 2k + 1 - j$ and take $r \in [0, k-1]$ such that $k| i + 1 - j' + r$, this claim is immediate from the fact that
\[u^{i, j'-r-1}_{\varepsilon|_{[j'-r-2]}} \otimes \dots \otimes u^{i, j'-k+1}_{\varepsilon|_{[j'-k]}} \otimes u^{i, j'}_{\varepsilon|_{[j'-1]}} \otimes u^{i, j'-1}_{\varepsilon|_{[j'-2]}} \otimes \dots \otimes u^{i, j'-r}_{\varepsilon|_{[j'-r-1]}}\]
appears only in $x^{i, \varepsilon}$ for $(i, \varepsilon) \in \mathcal{I}_j$, and in no other point of index in $\mathcal{I}_{1} \cup \dots \cup \mathcal{I}_j$ (we use $l_1 \otimes \dots \otimes l_k$ for $\mathcal{I}_{2k+1}$). The listed tensor products of elements are independent in $G_1 \otimes \dots \otimes G_k$ since each $S_d$ is a linearly independent set.\end{proof}

Theorem~\ref{densificationThm} is proved using an algebraic variant of the dependent random choice method.  

\begin{proposition}\label{DRCmap}Suppose that $\pi \colon H \to \mathbb{F}_p^t$ and $\psi \colon G_{[k]} \to \mathbb{F}_p^t$ are linear and multilinear maps chosen independently and uniformly at random. Suppose that $(q_1, \dots, q_k)$ is generic and that each $q_i$ has lengths $l_{[k]}$. Say that a point $x_{[k]}$ is \emph{kept} if $\pi \circ \phi(x_{[k]}) = \psi(x_{[k]})$. Then 
\[\mathbb{P}\Big((q_1, \dots, q_k)\text{ \emph{kept}}\Big)\begin{cases}=p^{-t(k 4^k - (k-1))}\ \text{ if }\phi(q_1) = \dots = \phi(q_k),\\\leq p^{-t(k 4^k - (k-2))}\ \text{ otherwise.}\end{cases}\]
\end{proposition}

\begin{proof}Let us begin by fixing an arbitrary affine map $\pi \colon H \to \mathbb{F}_p^t$ and conditioning on that as the $\pi$ that is chosen. Let $x^i$ be the first point of $q_i$ for $i \in [2,k]$. Since the tuple of arrangements is generic, and multilinear maps on $G_{[k]}$ are linear on $G_1 \otimes \dots \otimes G_k$, 
\[\mathbb{P}\Big((\forall x \in q_1 \cup \dots \cup q_k \setminus \{x^2, \dots, x^k\})\hspace{2pt}\pi(\phi(x)) = \psi(x)\Big) = p^{-t(k 4^k - (k-1))}.\]
Hence, for any choice of $\pi$
\begin{align*}&\mathbb{P}\Big((q_1, \dots, q_k)\hspace{2pt}\text{kept}\Big)\\
=\,&\mathbb{P}\Big(\Big((\forall x \in q_1 \cup \dots \cup q_k \setminus \{x^2, \dots, x^k\})\hspace{2pt}\pi(\phi(x)) = \psi(x)\Big)\text{ and }\Big((\forall i \in [2,k]) \pi(\phi(x^i)) = \psi(x^i)\Big)\Big)\\
=\,&\mathbb{P}\Big(\Big((\forall x \in q_1 \cup \dots \cup q_k \setminus \{x^2, \dots, x^k\})\hspace{2pt}\pi(\phi(x)) = \psi(x)\Big)\text{ and }\Big((\forall i \in [2,k]) \pi(\phi(q_i)) = \psi(l_{[k]})\Big)\Big)\\
&\hspace{2cm}\text{(from the first part of the condition we have that }\pi(\phi(q_1)) = \psi(l_{[k]}))\\ 
=\,&\mathbb{P}\Big(\Big((\forall x \in q_1 \cup \dots \cup q_k \setminus \{x^2, \dots, x^k\})\hspace{2pt}\pi(\phi(x)) = \psi(x)\Big)\text{ and }\Big((\forall i \in [2,k]) \pi(\phi(q_i)) =  \pi(\phi(q_1))\Big)\Big)\\\\
=\,&p^{-t(k 4^k - (k-1))}\, \mathbbm{1}\Big(\pi(\phi(q^1)) = \dots = \pi(\phi(q^k))\Big).\end{align*}
Now let us also take $\pi$ uniformly at random. Then by the above,
\[\mathbb{P}\Big((q_1, \dots, q_k)\hspace{2pt}\text{kept}\Big) = p^{-t(k 4^k - (k-1))} \mathbb{P}\Big(\pi(\phi(q^1)) = \dots = \pi(\phi(q^k))\Big).\]
The probability that $\pi(\phi(q^1)) = \dots = \pi(\phi(q^k))$ is 1 if the tuple is respected and at most $p^{-t}$ otherwise.\end{proof}

\begin{proof}[Proof of Theorem~\ref{densificationThm}]Let $N$ be the number of $k$-tuples $(q_1, \dots, q_k)$ such that $q_i$ is an $([i,1], [k,1], [k, i+1])$-arrangement with points lying inside $A$ and $q_1, \dots, q_k$ have the same lengths. Let $N_{\text{resp}}$ be the number of those $(q_1, \dots, q_k)$ such that $\phi(q_1) = \dots = \phi(q_k)$ and let $N_{\text{nongen}}$ be the number of those $(q_1, \dots, q_k)$ that are not generic. By Lemma~\ref{nongenNum}
\[N_{\text{nongen}} \leq p^{k (4^k + 1)} \Big(\frac{1}{|G_1|} + \dots + \frac{1}{|G_k|}\Big) |G_{[k]}|^{4^k}.\]
We apply Theorem~\ref{exArrThm}. We obtain the inequality
\begin{equation}\label{arrBounds}|G_{[k]}|^{4^{k}} \geq N \geq N_{\text{resp}} = \Omega(\delta^{O(1)}) |G_{[k]}|^{4^{k}}.\end{equation}
Let $t \in \mathbb{N}$ and let $\pi \colon H \to \mathbb{F}_p^t$ and $\psi \colon G_{[k]} \to \mathbb{F}_p^t$ be linear and multilinear maps chosen independently and uniformly at random. Let $A' = \{x_{[k]} \in A \colon \pi(\phi(x_{[k]})) = \psi(x_{[k]})\}$ and, in a similar way to the above, let $N'$ be the number of tuples $(q_1, \dots, q_k)$ such that $q_i$ is an $([i,1], [k,1], [k, i+1])$-arrangement with points lying inside $A'$ and $q_1, \dots, q_k$ have the same lengths, and let $N'_{\text{resp}}$ be the number of such tuples for which $\phi(q_1) = \dots = \phi(q_k)$. By Proposition~\ref{DRCmap} and inequality~\eqref{arrBounds}, we see that
\begin{align*}\ex\Big[ N'_{\text{resp}} - \varepsilon^{-1}(N' - N'_{\text{resp}})\Big]\,\, \geq&\,\, p^{t\big((k-1)-k4^k\big)} (N_{\text{resp}} - N_{\text{nongen}}) - \varepsilon^{-1} p^{t\big((k-2)-k4^{k}\big)} |G_{[k]}|^{4^{k}}\\
\geq &\,\,p^{t\big((k-1)-k4^{k}\big)} \delta^{O(1)}\Big(1 - p^{-t} \delta^{-O(1)} \varepsilon^{-1} \Big)|G_{[k]}|^{4^{k}},\end{align*}
provided that $|G_i| \geq k p^{k(4^{k} + 1)}\delta^{-\con_{k,p}}$ for each $i \in [k]$. Pick $t = O(\log \delta^{-1} + \log \varepsilon^{-1})$ such that $p^{-t} \delta^{-O(1)} \varepsilon^{-1} \leq 1/2$. Then there is a choice of $A'$ such that $N'_{\text{resp}} \geq (1-\varepsilon) N'$ and $N'_{\text{resp}} = \delta^{O(1)} \varepsilon^{O(1)}|G_{[k]}|^{4^k}$, as claimed.\end{proof}

\section{Obtaining a nearly multiaffine piece}\label{sec7labnearlymultiaffinepiece}

The goal of this section is to obtain a highly structured set of lengths $L$ on which the map $\psi$ given by common $\phi$-values of arrangements (as in Theorem~\ref{exArrThm}) is actually multiaffine, in the sense that for each line in one of the $k$ directions there is a global affine map that coincides with $\psi$ on the intersection with $L$. The structure we are after is given by subsets of density $1-o(1)$ inside varieties of bounded codimension. The main result will be obtained by applying the following proposition in each direction.

\begin{proposition}\label{nearlyMaffStep}Let $A \subset G_{[k]}$ be a set of density $\delta$ and let $\phi \colon A \to H$. Let $\varepsilon \in (0, 10^{-4})$ be given. Let $\mathcal{Q}$ be the set of all $(q_1, \dots, q_k)$ such that $q_i$ is an $([i,1], [k,1], [k,i + 1])$-arrangement with points in $A$ and $q_1, \dots, q_k$ have the same lengths. Write $\mathcal{Q}_{x_{[k]}}$ for the set of $(q_1, \dots, q_k) \in \mathcal{Q}$ such that each $q_i$ has lengths $x_{[k]}$. Let $X \subset G_{[k]}$ be a subset such that for each $x_{[k]} \in X$, there is a value $\psi(x_{[k]}) \in H$ such that
\[\Big|\Big\{(q_1, \dots, q_k) \in \mathcal{Q}_{x_{[k]}} \colon \phi(q_1) = \dots = \phi(q_k) = \psi(x_{[k]})\Big\}\Big| \geq (1-\varepsilon) |\mathcal{Q}_{x_{[k]}}|,\]
and 
\[\sum_{x_{[k]} \in X} |\mathcal{Q}_{x_{[k]}}| \geq (1-\varepsilon) |\mathcal{Q}|.\]
Let $d \in [k]$ be a direction. Then there exist a positive integer $t = \exp^{\big(D^{\mathrm{mh}}_{k-1} + D^{\mathrm{mh}}_{k - 2} + \dots + D^{\mathrm{mh}}_{1} + 3k\big)}\Big(O((\delta\varepsilon)^{-O(1)})\Big)$, a multiaffine map $\alpha \colon G_{[k]} \to \mathbb{F}_p^t$, and a collection of disjoint sets $(S^\lambda)_{\lambda \in \mathbb{F}_p^t}$, such that $S^\lambda \subset \{x_{[k]} \in G_{[k]} \colon \alpha(x_{[k]}) = \lambda\} \cap X$, $\psi|_{S^\lambda}$ is affine in direction $d$ for each $\lambda \in \mathbb{F}_p^t$ (in the sense that for each $x_{[k]\setminus \{d\}} \in G_{[k]\setminus \{d\}}$ there is an affine map $\rho \colon G_d \to H$ such that $\rho(y_d) = \psi(y_d)$ for all $y_d \in S^\lambda_{x_{[k] \setminus \{d\}}}$), and
\[\sum_{x_{[k]} \in \cup_{\lambda \in \mathbb{F}^t} S^\lambda} |\mathcal{Q}_{x_{[k]}}|\ = \Big(1-O(\sqrt[8]{\varepsilon})\Big) |\mathcal{Q}|.\]
\end{proposition}

The main result of this section follows easily from the proposition above, by applying it to each direction $d \in [k]$.

\begin{theorem}\label{nearlyMultThm}Suppose that the assumptions of Proposition~\ref{nearlyMaffStep} hold. Then there exist a non-empty variety $V$ of codimension $ \exp^{\big(D^{\mathrm{mh}}_{k-1} + D^{\mathrm{mh}}_{k - 2} + \dots + D^{\mathrm{mh}}_{1} + 3k\big)}\Big(O((\delta\varepsilon)^{-O(1)})\Big)$ and a subset $B \subset V \cap X$, of size $(1-O(\sqrt[8]{\varepsilon})) |V|$, such that $\psi|_B$ is multiaffine (in the sense that for each $d \in [k]$ and each $x_{[k]\setminus \{d\}} \in G_{[k]\setminus \{d\}}$ there is an affine map $\rho \colon G_d \to H$ such that $\rho(y_d) = \psi(y_d)$ for all $y_d \in B_{x_{[k] \setminus \{d\}}}$).\end{theorem}

\begin{proof}[Proof of Proposition~\ref{nearlyMaffStep}]Write 
\begin{align*}f^{(1, i)}(x_{[k]}) =& \bigconv{i}\dots\bigconv{1}\bigconv{k}\dots\bigconv{1}\bigconv{k}\dots\bigconv{i + 1} A(x_{[k]}),\\
f^{(2, i)}(x_{[k]}) =& \bigconv{i-1}\dots\bigconv{1}\bigconv{k}\dots\bigconv{1}\bigconv{k}\dots\bigconv{i + 1} A(x_{[k]})\end{align*}
and
\[\bm{g}_i = |G_{i+1}|^{2^{2k-2} + 2^{k-2}} |G_{i+2}|^{2^{2k-3} + 2^{k-3}} \dots |G_{i-1}|^{ 2^{k}  + 1} |G_i|^{2^{k-1}}.\]
The number of $([i,1], [k,1], [k,i + 1])$-arrangements in $A$ of lengths $x_{[k]}$ is exactly
\[f^{(1, i)}(x_{[k]}) \cdot \bm{g}_i^2 |G_i|,\]
and the number of $([i-1,1], [k,1], [k,i + 1])$-arrangements in $A$ of lengths $x_{[k]}$ is
\[f^{(2, i)}(x_{[k]}) \cdot \bm{g}_i.\]
Thus, 
\[\mathcal{Q}_{x_{[k]}} = \bigg(\prod_{i \in [k]} f^{(1, i)}(x_{[k]})\bigg) \cdot |G_{[k]}|^{2^{2k} - 1}.\]

\chg{Recall from the definition of arrangements that an $([i,1], [k,1], [k,i + 1])$-arrangement $q$ of lengths $l_{[k]}$ arises as a concatenation $(q^1, q^2)$, where $q^1$ and $q^2$ are $([i-1,1], [k,1], [k,i + 1])$-arrangements of lengths $(l_{[k] \setminus \{i\}}, l_i + y_i)$ and $(l_{[k] \setminus \{i\}}, y_i)$ respectively. We refer to $q^1$ as the \textit{first half} and $q^2$ as the \textit{second half}.}\\

\chg{Observe that we may give a natural probabilistic interpretation to the expressions above. Namely, define random variables $\bm{\mathsf{X}}, \bm{\mathsf{Q}}_1, \dots, \bm{\mathsf{Q}}_k$ as follows. First, let $\bm{\mathsf{X}}$ be a point in $G_{[k]}$ chosen uniformly at random. Then, for each $i$, independently choose an $([i,1], [k,1], [k,i + 1])$-arrangement $ \bm{\mathsf{Q}}_i$ uniformly among all such configurations in $G_{[k]}$ whose lengths are $\bm{\mathsf{X}}$. Additionally, let $E$ be the event that all points of each $ \bm{\mathsf{Q}}_i$ lie in the set $A$. Notice that we may express some terms above quite naturally using this language. Namely, the probability that $\bm{\mathsf{Q}}_i$ has all its points in $A$, conditioned on $\bm{\mathsf{X}} = x_{[k]}$, equals $f^{(1, i)}(x_{[k]})$ and $\frac{|\mathcal{Q}_{x_{[k]}}|}{|\mathcal{Q}|}$ becomes $\mathbb{P}(\bm{\mathsf{X}} = x_{[k]} | E)$. We shall use this language to provide additional explanation of our steps in the proof.}\\

Let $Y \subset G_{[k]}$ be the set of all $x_{[k]} \in G_{[k]}$ such that there is a value $\theta(x_{[k]}) \in H$ with the property that $\phi(q) = \theta(x_{[k]})$ for proportion at least $\frac{3}{4}$ of the $([d-1,1], [k,1], [k,d + 1])$-arrangements $q$ of lengths $x_{[k]}$ with points in $A$. This value $\theta(x_{[k]})$ is clearly unique. Note also that for each $x_{[k]} \in X$, the proportion of $([d,1], [k,1], [k,d + 1])$-arrangements $q$ in $A$ of lengths $x_{[k]}$ such that $\phi(q) = \psi(x_{[k]})$ is at least $1-\varepsilon$.\\

\indent Write $\bm{w}(x_{[k]}) = \prod_{i \in [k] \setminus \{d\}} f^{(1, i)}(x_{[k]}) \in [0,1]$. We think of this quantity as the `weight' of the point $x_{[k]}$. As a consequence\footnote{We do not use the full strength of the Theorem~\ref{exArrThm}: we merely need the existence of many tuples of arrangements with the same lengths. In particular, the map $\phi$ plays no role in this deduction, so we do not have to fulfil the requirements on $\phi$ in Theorem~\ref{exArrThm}.} of Theorem~\ref{exArrThm}, we have that
\begin{equation}\label{tuplesArrDenseEq}\exx_{x_{[k]}} \bm{w}(x_{[k]}) f^{(1, d)}(x_{[k]}) = \Omega(\delta^{O(1)}).\end{equation}

\noindent\textbf{Step 1.} \chg{In this step, we essentially show that $Y$ is large. More precisely, using the probabilistic language, we show that, conditioned on the event $E$, the probability that the lengths of both halves of $\bm{\mathsf{Q}}_d$ belong to $Y$ is close to 1.}\\
\indent Suppose that $x_{[k]} \in X$. Then each $y_d \notin Y_{x_{[k] \setminus \{d\}}}$ produces at least
\[\frac{1}{4} f^{(2, d)}(x_{[k] \setminus \{d\}}, x_d + y_d) f^{(2, d)}(x_{[k] \setminus \{d\}}, y_d) \bm{g}_d^2\]
$([d,1], [k,1], [k,d + 1])$-arrangements $q$ of lengths $x_{[k]}$ such that $\phi(q) \not= \psi(x_{[k]})$. Thus,
\[|G_d|^{-1}\sum_{y_d \in Y_{x_{[k] \setminus \{d\}}}}  f^{(2, d)}(x_{[k] \setminus \{d\}}, x_d + y_d) f^{(2, d)}(x_{[k] \setminus \{d\}}, y_d)  \geq (1-4\varepsilon) f^{(1,d)}(x_{[k]}).\]
Similarly, each $x_d + y_d \notin Y_{x_{[k] \setminus \{d\}}}$ produces at least
\[\frac{1}{4} f^{(2, d)}(x_{[k] \setminus \{d\}}, x_d + y_d) f^{(2, d)}(x_{[k] \setminus \{d\}}, y_d) \bm{g}_d^2\]
$([d,1], [k,1], [k,d + 1])$-arrangements $q$ of lengths $x_{[k]}$ such that $\phi(q) \not= \psi(x_{[k]})$. Thus
\begin{align}\sum_{x_{[k]} \in X} \bm{w}(x_{[k]})& \sum_{y_d \in Y_{x_{[k] \setminus \{d\}}} \cap (Y_{x_{[k] \setminus \{d\}}} - x_d)}  f^{(2, d)}(x_{[k] \setminus \{d\}}, x_d + y_d) f^{(2, d)}(x_{[k] \setminus \{d\}}, y_d)\nonumber\\
\geq &(1-8\varepsilon)|G_d| \sum_{x_{[k]} \in X}\bm{w}(x_{[k]}) f^{(1,d)}(x_{[k]})\nonumber\\
\geq &(1-9\varepsilon)|G_d| \sum_{x_{[k]} \in G_{[k]}}\bm{w}(x_{[k]}) f^{(1,d)}(x_{[k]})\nonumber\\
= &(1-9\varepsilon) \sum_{x_{[k]} \in G_{[k]}} \bm{w}(x_{[k]}) \sum_{y_d \in G_d} f^{(2, d)}(x_{[k] \setminus \{d\}}, x_d + y_d) f^{(2, d)}(x_{[k] \setminus \{d\}}, y_d).\label{yLargeEqn}\end{align}

\textbf{Step 2.} Recall that $\theta$ is defined as the most frequent value of $\phi(q)$, where $q$ ranges over $([d-1,1], [k,1], [k,d + 1])$-arrangements in $A$ of lengths $x_{[k]}$. In this step, we relate $\theta$ and $\psi$. \chg{Using the probabilistic language, and defining $\bm{\mathsf{Y}}_d$ to be the $d$\textsuperscript{th} length of the second half of $\bm{\mathsf{Q}}_d$, we now show that
\[\mathbb{P}\Big(\theta(\bm{\mathsf{X}}_{[k] \setminus \{d\}}, \bm{\mathsf{X}}_d + \bm{\mathsf{Y}}_d) - \theta(\bm{\mathsf{X}}_{[k] \setminus \{d\}}, \bm{\mathsf{Y}}_d) = \psi(\bm{\mathsf{X}})\,\Big|\,E\Big) \geq 1 - 13 \varepsilon,\]
where the argument of $\theta$ is tacitly understood to be in $Y$.}\\
\indent Note also that when $y_d \in Y_{x_{[k] \setminus \{d\}}} \cap (Y_{x_{[k] \setminus \{d\}}} - x_d)$, then we get at least $\frac{9}{16} f^{(2, d)}(x_{[k] \setminus \{d\}}, x_d + y_d) f^{(2, d)}(x_{[k] \setminus \{d\}}, y_d) \bm{g}_d^2$ arrangements $q$ of lengths $x_{[k]}$ such that $\phi(q) = \theta(x_{[k] \setminus \{d\}}, x_d + y_d) - \theta(x_{[k] \setminus \{d\}}, y_d)$. Hence, if $\theta(x_{[k] \setminus \{d\}}, x_d + y_d) - \theta(x_{[k] \setminus \{d\}}, y_d) \not= \psi(x_{[k]})$, then these arrangements satisfy that $\phi(q) \not= \psi(x_{[k]})$. Thus, using~\eqref{yLargeEqn}, we get
\begin{align}&\sum_{x_{[k]} \in X} \bm{w}(x_{[k]}) \sum_{y_d \in Y_{x_{[k] \setminus \{d\}}} \cap (Y_{x_{[k] \setminus \{d\}}} - x_d)}  \mathbbm{1}\Big(\theta(x_{[k] \setminus \{d\}}, x_d + y_d) - \theta(x_{[k] \setminus \{d\}}, y_d) = \psi(x_{[k]})\Big)\nonumber\\
&\hspace{9cm}f^{(2, d)}(x_{[k] \setminus \{d\}}, x_d + y_d) f^{(2, d)}(x_{[k] \setminus \{d\}}, y_d) \nonumber\\
&\nonumber\\
&\hspace{2cm}\geq (1-13 \varepsilon) \sum_{x_{[k]} \in G_{[k]}} \bm{w}(x_{[k]}) \sum_{y_d \in G_d} f^{(2, d)}(x_{[k] \setminus \{d\}}, x_d + y_d) f^{(2, d)}(x_{[k] \setminus \{d\}}, y_d).\label{approximateEqThetaPsi}\end{align}

\textbf{Step 3.} \chg{In this step we apply the approximation theorem for mixed convolutions, which we use to give strong approximations for functions $f^{(1, i)}$ for $i \not= d$ and for the function $f^{(2, d)}(x_{[k]})$ by auxiliary functions that are constant along layers of multiaffine maps. We also deduce basic estimates on the approximants (see~\eqref{approxWEqn},~\eqref{f2cApprox},~\eqref{cDenseEqn} and~\eqref{cDenseWrtTildeW}). In a later step, we exploit these approximants in expressions such as~\eqref{approximateEqThetaPsi}.}\\
\indent Let $\xi > 0$ be a constant to be specified later. For each $i \in [k] \setminus \{d\}$, apply Theorem~\ref{strongMixedApprox} to $A$ for the $L^k$\footnote{Note that for any function $g$ such that $\|g\|_{L^\infty} \leq 2$ we have $\|g\|_{L^k} \leq 2\|g\|_{L^2}^{2/k}$, so Theorem~\ref{strongMixedApprox} gives $L^k$ approximation provided we somewhat decrease the approximation parameter.} norm to obtain $l^{(i)} = \exp^{\big(D^{\mathrm{mh}}_{k-1} + D^{\mathrm{mh}}_{k - 2} + \dots + D^{\mathrm{mh}}_{1} + 3k\big)}\Big(O(\xi^{-O(1)})\Big)$, a map $c^{(i)} \mathbb{F}_p^{l^{(i)}} \to \mathbb{D}$, and a multiaffine map $\phi^{(i)} \colon G_{[k]} \to \mathbb{F}_p^{l^{(i)}}$, such that
\[f^{(1, i)}(x_{[k]}) = \bigconv{i}\dots\bigconv{1}\bigconv{k}\dots\bigconv{1}\bigconv{k}\dots\bigconv{i + 1} A(x_{[k]}) \apps{\xi}_{L^k,x_{[k]}} c^{(i)} \circ \phi^{(i)}(x_{[k]})\]
and apply Theorem~\ref{strongMixedApprox} one more time to $A$ to obtain a positive integer \[l^{(2, d)} = \exp^{\big(D^{\mathrm{mh}}_{k-1} + D^{\mathrm{mh}}_{k - 2} + \dots + D^{\mathrm{mh}}_{1} + 3k\big)}\Big(O(\xi^{-O(1)})\Big),\] a map $c^{(2, d)} : \mathbb{F}_p^{l^{(2,d)}} \to \mathbb{D}$, and a multiaffine map $\phi^{(2,d )} \colon G_{[k]} \to \mathbb{F}_p^{l^{(2,d)}}$, such that
\[f^{(2, d)}(x_{[k]}) = \bigconv{d-1}\dots\bigconv{1}\bigconv{k}\dots\bigconv{1}\bigconv{k}\dots\bigconv{d + 1} A(x_{[k]}) \apps{\xi}_{L^k,x_{[k]}} c^{(2, d)} \circ \phi^{(2, d)}(x_{[k]}).\]
By the Cauchy--Schwarz inequality (applied several times), we have that 
\begin{equation}\label{arrDenseIneq}\exx_{x_{[k]}} f^{(1,d)}(x_{[k]}) \geq \delta^{2^{3k}}\quad \text{ and }\quad \exx_{x_{[k]}} f^{(2,d)}(x_{[k]}) \geq \delta^{2^{3k-1}}.\end{equation}
Write $l = \Big(\sum_{i \in [k] \setminus \{d\}} l^{(i)}\Big) + l^{(2,d)}$. Define maps $\phi \colon G_{[k]} \to \mathbb{F}^l$ and $c, \tilde{\bm{w}}\colon \mathbb{F}^l \to \mathbb{D}$ by
\[\phi = (\phi^{(1)}, \dots, \phi^{(d-1)}, \phi^{(2,d)}, \phi^{(d+1)}, \dots, \phi^{(k)}),\]
\[c(\lambda) = c^{(2, d)}(\lambda_{s_d+ 1}, \lambda_{s_d+ 2}, \dots, \lambda_{s_d+ l^{(2,d)}})\]
and
\[\tilde{\bm{w}}(\lambda) = \prod_{i \in [k] \setminus \{d\}} c^{(i)}(\lambda_{s_i+ 1}, \lambda_{s_i+ 2}, \dots, \lambda_{s_i+ l^{(i)}}),\]
where the $s_i$ are offsets such that $s_i + [l^{(i)}]$ are indices inside $[l]$ that correspond to the map $\phi^{(i)}$ in the definition of $\phi$ (and $s_d$ corresponds to $\phi^{(2,d)}$). Then
\begin{equation}\Big\|f^{(2, d)}(x_{[k]}) - c \circ \phi(x_{[k]})\Big\|_{L^k, x_{[k]}} \leq \xi\label{approx2i}\end{equation}
and
\[\Big\|\bm{w} - \tilde{\bm{w}}\circ \phi\Big\|_{L^{1}} = \Big\|\prod_{i \in [k] \setminus \{d\}} f^{(1, i)} - \prod_{i \in [k] \setminus \{d\}} c^{(i)} \circ \phi^{(i)}\Big\|_{L^{1}}.\]
Writing $c^{(i)} \circ \phi^{(i)} = f^{(1, i)} - \Big(f^{(1, i)} - c^{(i)} \circ \phi^{(i)}\Big)$, expanding and using the triangle inequality, we get
\[\Big\|\bm{w} - \tilde{\bm{w}}\circ \phi\Big\|_{L^{1}} \leq \sum_{I \subsetneq [k] \setminus \{d\}} \Big\|\Big(\prod_{i \in I} f^{(1, i)}\Big) \cdot \Big(\prod_{i \in [k] \setminus \{d\} \setminus I} \Big(f^{(1, i)} - c^{(i)} \circ \phi^{(i)}\Big)\Big)\Big\|_{L^{1}}.\]
Using the fact that $\|f^{(1, i)}\|_{L^\infty} \leq 1$, we bound this from above by
\[\sum_{I \subsetneq [k] \setminus \{d\}} \Big\|\prod_{i \in [k] \setminus \{d\} \setminus I} \Big(f^{(1, i)} - c^{(i)} \circ \phi^{(i)}\Big)\Big\|_{L^{1}}.\]
We then apply by H\"{o}lder's inequality to conclude that 
\begin{align} \Big\|\bm{w} - \tilde{\bm{w}}\circ \phi\Big\|_{L^{1}}  \leq &\sum_{I \subsetneq [k] \setminus \{d\}} \prod_{i \in [k] \setminus \{d\} \setminus I}  \Big\|f^{(1, i)} - c^{(i)} \circ \phi^{(i)}\Big\|_{L^{k}}\nonumber\\
\leq &2^k \xi.\label{approxWEqn}\end{align}
Without loss of generality we may assume that $c \colon \mathbb{F}^l \to [0,1]$, since we have $f^{(2, d)}(x_{[k]}) \in [0,1]$ and thus we may simply replace $c(\lambda)$ by $\max \{\min\{\operatorname{Re} c(\lambda), 1\}, 0\}$. Such a change does not worsen the bound in~\eqref{approx2i}. Similarly, without loss of generality $\tilde{\bm{w}} \colon \mathbb{F}^l \to [0,1]$.\\

Note that
\begin{align}\exx_{x_{[k] \setminus \{d\}}} \exx_{y_d, z_d \in G_d} &\Big|f^{(2, d)}(x_{[k] \setminus \{d\}}, y_d) f^{(2, d)}(x_{[k] \setminus \{d\}}, z_d) - c(\phi(x_{[k]\setminus \{d\}}, y_d)) c(\phi(x_{[k]\setminus \{d\}}, z_d))\Big|\nonumber\\
\leq & \exx_{x_{[k] \setminus \{d\}}} \exx_{y_d, z_d \in G_d}  \Big|f^{(2, d)}(x_{[k] \setminus \{d\}}, y_d) - c(\phi(x_{[k]\setminus \{d\}}, y_d)) \Big| \Big|f^{(2, d)}(x_{[k] \setminus \{d\}}, z_d)\Big|\nonumber\\
&\hspace{1cm}+\exx_{x_{[k] \setminus \{d\}}} \exx_{y_d, z_d \in G_d}  \Big|c(\phi(x_{[k]\setminus \{d\}}, y_d))\Big| \Big|f^{(2, d)}(x_{[k] \setminus \{d\}}, z_d) - c(\phi(x_{[k]\setminus \{d\}}, z_d))\Big|\nonumber\\
\leq & 2\exx_{x_{[k] \setminus \{d\}}} \exx_{y_d \in G_d}  \Big|f^{(2, d)}(x_{[k] \setminus \{d\}}, y_d) - c(\phi(x_{[k]\setminus \{d\}}, y_d)) \Big|\nonumber\\
\leq&2\xi.\label{f2cApprox}\end{align}

By~\eqref{tuplesArrDenseEq} we also have that
\begin{align}\exx_{x_{[k] \setminus \{d\}}} \exx_{y_d, z_d \in G_d} &\bm{w}(x_{[k]\setminus \{d\}}, y_d - z_d) c(\phi(x_{[k]\setminus \{d\}}, y_d)) c(\phi(x_{[k]\setminus \{d\}}, z_d))\nonumber \\
\geq &\exx_{x_{[k] \setminus \{d\}}} \exx_{y_d, z_d \in G_d} \bm{w}(x_{[k]\setminus \{d\}}, y_d - z_d)  f^{(2, d)}(x_{[k]\setminus \{d\}}, y_d)  f^{(2, d)}(x_{[k]\setminus \{d\}}, z_d) - 2\xi\nonumber\\
=&\exx_{x_{[k]}} \bm{w}(x_{[k]}) f^{(1, d)}(x_{[k]})\nonumber - 2\xi\\
=&\ \Omega(\delta^{O(1)}),\label{cDenseEqn}\end{align}
provided that $\xi$ is sufficiently small, namely $\xi \leq \cons\,\delta^{\con}$. (Once again, recall that $\con$ and $\cons$ indicate positive constants, as explained in the notational part of the preliminary section of the paper; see expression~\eqref{fullConEqn} and discussion surrounding it.) Using~\eqref{approxWEqn}, we also have
\begin{equation}\label{cDenseWrtTildeW}\exx_{x_{[k] \setminus \{d\}}} \exx_{y_d, z_d \in G_d} \tilde{\bm{w}}(\phi(x_{[k]\setminus \{d\}}, y_d - z_d)) c(\phi(x_{[k]\setminus \{d\}}, y_d)) c(\phi(x_{[k]\setminus \{d\}}, z_d)) =\Omega(\delta^{O(1)})\end{equation}
once again provided $\xi \leq \cons\,\delta^{\con}$ (with slightly modified implicit constants).\\

\noindent\chg{\textbf{Step 4. Controlling the values of $\phi$.} Some of the averaging arguments over the layers of $\phi$ will require us to know exactly what the image of the affine map $y_d \mapsto \phi(x_{[k] \setminus \{d\}}, y_d)$ is. In this step, we find layers of a further low-codimensional multiaffine map, on which the images of these maps are the prescribed cosets of subspaces. The conclusion is~\eqref{imagePhi}.}\\
\indent Write $\phi\colon G_{[k]} \to \mathbb{F}_p^l$ as $\phi(x_{[k]}) = \phi'(x_{[k]}) + \tau(x_{[k] \setminus \{d\}})$, where $\phi'$ is linear in the $d$\textsuperscript{th} coordinate and $\tau$ is multiaffine. Thus, $\phi_j'(x_{[k]}) = \Phi_j(x_{[k] \setminus \{d\}}) \cdot x_d$ for a multiaffine map $\Phi_j \colon G_{[k] \setminus \{d\}} \to G_d$. \chg{The image of $y_d \mapsto \phi(x_{[k] \setminus \{d\}}, y_d)$ is related to the set of vanishing linear combinations of elements $\Phi_j(x_{[k] \setminus \{d\}})$ for $j \in [l]$. We rely on Theorem~\ref{newthmvanishinglcs} to understand sets of such linear combinations.}\\
\indent Applying Theorem~\ref{newthmvanishinglcs} to $\Phi_1, \dots, \Phi_l$, we obtain a positive integer $t = O\Big((l \log_p\xi^{-1})^{O(1)}\Big)$ and a multiaffine map $\beta \colon G_{[k] \setminus \{d\}} \to \mathbb{F}_p^t$ and a union $U$ of layers of $\beta$ of size

\begin{equation}|U| \geq (1-\xi) |G_{[k] \setminus \{d\}}|\label{uSizeBound}\end{equation}
such that for each layer $L_\mu = \{x_{[k]\setminus\{d\}} \in G_{[k]\setminus\{d\}}:\beta(x_{[k]\setminus\{d\}}) = \mu\} \subset U$ there is a subspace $\Lambda_\mu \leq \mathbb{F}_p^l$ such that
\[\{\lambda \in \mathbb{F}_p^l \colon \lambda \cdot \Phi(x_{[k] \setminus \{d\}}) = 0\} = \Lambda_\mu\]
for every $x_{[k] \setminus \{d\}} \in L_\mu$. Write $M$ for the set of $\mu \in \mathbb{F}_p^t$ such that $L_\mu \subset U$. This implies that for each $\mu \in M$ and each $x_{[k]\setminus\{d\}}\in L_\mu$, 
\[\img \Big(y_d \mapsto \phi'(x_{[k] \setminus \{d\}}, y_d)\Big) = \Lambda_\mu^\perp.\]
Thus, when $x_{[k] \setminus \{d\}} \in L_\mu$ and $\tau(x_{[k] \setminus \{d\}}) = \lambda$, then  
\begin{equation}\label{imagePhi}\img \Big(y_d \mapsto \phi(x_{[k] \setminus \{d\}}, y_d)\Big) = \lambda + \Lambda_\mu^\perp.\end{equation}

\chg{\textbf{Step 5.}} In this step, we move from using $\bm{w}$ as our system of weights to using $\tilde{\bm{w}}$, which has more algebraic structure. This will allow us to use Lemma~\ref{3approxHom} to find a structured set with an affine map that coincides with the map $y_d \mapsto \psi(x_{[k] \setminus \{d\}}, y_d)$.\\
\indent \chg{First, we derive a variant of inequality~\eqref{approximateEqThetaPsi} in which we replace $f^{(2, d)}$ by its approximation $c \circ \phi$.}\\

\begin{claima*}
    \chg{We have
    \begin{align*}&\sum_{\substack{x_{[k]} \in X\\x_{[k] \setminus \{d\}} \in U}} \bm{w}(x_{[k]}) \sum_{y_d \in Y_{x_{[k] \setminus \{d\}}} \cap (Y_{x_{[k] \setminus \{d\}}} - x_d)}  \mathbbm{1}\Big(\theta(x_{[k] \setminus \{d\}}, x_d + y_d) - \theta(x_{[k] \setminus \{d\}}, y_d) = \psi(x_{[k]})\Big)\times\nonumber\\
&\hspace{10cm}c(\phi(x_{[k]\setminus \{d\}}, x_d + y_d)) c(\phi(x_{[k]\setminus \{d\}}, y_d)) \nonumber\\
&\nonumber\\
&\hspace{2cm}\geq(1-13\varepsilon - O(\delta^{-O(1)})\xi)  \sum_{x_{[k]} \in G_{[k]}} \bm{w}(x_{[k]}) \sum_{y_d \in G_d} c(\phi(x_{[k]\setminus \{d\}}, x_d + y_d)) c(\phi(x_{[k]\setminus \{d\}}, y_d))\nonumber.\end{align*}
    }
\end{claima*}

\begin{proof}
    \chg{ We rely on inequalities~\eqref{approximateEqThetaPsi},~\eqref{f2cApprox},~\eqref{cDenseEqn} and~\eqref{uSizeBound}. By~\eqref{f2cApprox}, $c\circ \phi$ and $f^{(2, d)}$ differ by a small amount, which gives us the inequality }
    \chg{\begin{align*}
        &\sum_{\substack{x_{[k]} \in X\\x_{[k] \setminus \{d\}} \in U}} \bm{w}(x_{[k]}) \sum_{y_d \in Y_{x_{[k] \setminus \{d\}}} \cap (Y_{x_{[k] \setminus \{d\}}} - x_d)}  \mathbbm{1}\Big(\theta(x_{[k] \setminus \{d\}}, x_d + y_d) - \theta(x_{[k] \setminus \{d\}}, y_d) = \psi(x_{[k]})\Big)\ \times\nonumber\\
&\hspace{8cm}c(\phi(x_{[k]\setminus \{d\}}, x_d + y_d)) c(\phi(x_{[k]\setminus \{d\}}, y_d)) \nonumber\\
&\\
&\geq\sum_{\substack{x_{[k]} \in X\\x_{[k] \setminus \{d\}} \in U}} \bm{w}(x_{[k]}) \sum_{y_d \in Y_{x_{[k] \setminus \{d\}}} \cap (Y_{x_{[k] \setminus \{d\}}} - x_d)}  \mathbbm{1}\Big(\theta(x_{[k] \setminus \{d\}}, x_d + y_d) - \theta(x_{[k] \setminus \{d\}}, y_d) = \psi(x_{[k]})\Big)\ \times\nonumber\\
&\hspace{8cm}f^{(2, d)}(x_{[k] \setminus \{d\}}, x_d + y_d) f^{(2, d)}(x_{[k] \setminus \{d\}}, y_d) \hspace{1cm}-2\xi |G_d||G_{[k]}|.\nonumber
\end{align*}
}

\vspace{\baselineskip}

\chg{Using~\eqref{uSizeBound}, we may remove the condition $x_{[k] \setminus \{d\}} \in U$ from the first sum, so the expression above is at least
\begin{align*}
&\sum_{x_{[k]} \in X} \bm{w}(x_{[k]}) \sum_{y_d \in Y_{x_{[k] \setminus \{d\}}} \cap (Y_{x_{[k] \setminus \{d\}}} - x_d)}  \mathbbm{1}\Big(\theta(x_{[k] \setminus \{d\}}, x_d + y_d) - \theta(x_{[k] \setminus \{d\}}, y_d) = \psi(x_{[k]})\Big)\ \times\nonumber\\
&\hspace{8cm}f^{(2, d)}(x_{[k] \setminus \{d\}}, x_d + y_d) f^{(2, d)}(x_{[k] \setminus \{d\}}, y_d) \hspace{1cm}-3\xi |G_d||G_{[k]}|.
\end{align*}}

\chg{Using~\eqref{approximateEqThetaPsi}, we may remove the condition $\theta(x_{[k] \setminus \{d\}}, x_d + y_d) - \theta(x_{[k] \setminus \{d\}}, y_d) = \psi(x_{[k]})$, giving a lower bound of}

\chg{\begin{align*}
& (1-13 \varepsilon) \sum_{x_{[k]} \in G_{[k]}}\bm{w}(x_{[k]}) \sum_{y_d \in G_d} f^{(2, d)}(x_{[k] \setminus \{d\}}, x_d + y_d) f^{(2, d)}(x_{[k] \setminus \{d\}}, y_d) \hspace{0.2cm}-\hspace{0.2cm}3\xi |G_d||G_{[k]}|.
\end{align*}}

\chg{Using the approximation~\eqref{f2cApprox} for the second time, we get that the expression above is at least}

\chg{\begin{align*}
& (1-13 \varepsilon) \sum_{x_{[k]} \in G_{[k]}} \bm{w}(x_{[k]}) \sum_{y_d \in G_d} c(\phi(x_{[k]\setminus \{d\}}, x_d + y_d)) c(\phi(x_{[k]\setminus \{d\}}, y_d)) \hspace{0.2cm}-\hspace{0.2cm}5\xi |G_d||G_{[k]}|.\end{align*}}

\chg{Finally, inequality~\eqref{cDenseEqn} shows that the main sum is reasonably large, so the error term $5\xi |G_d||G_{[k]}|$ can be subsumed into the first factor, giving the final lower bound}
\chg{\begin{align*}&(1-13\varepsilon - O(\delta^{-O(1)})\xi)  \sum_{x_{[k]} \in G_{[k]}} \bm{w}(x_{[k]}) \sum_{y_d \in G_d} c(\phi(x_{[k]\setminus \{d\}}, x_d + y_d)) c(\phi(x_{[k]\setminus \{d\}}, y_d)).\qedhere
    \end{align*}}    
\end{proof}

\chg{Let us now rewrite the bound in the Claim A in a different form. We split the sums further according to the values of the maps $\theta, \beta$ and $\phi'$. First, using~\eqref{imagePhi}, the left hand side of the inequality in Claim A can be written as}

\chg{\begin{align*}&\sum_{\substack{x_{[k]} \in X\\x_{[k] \setminus \{d\}} \in U}} \bm{w}(x_{[k]}) \sum_{y_d \in Y_{x_{[k] \setminus \{d\}}} \cap (Y_{x_{[k] \setminus \{d\}}} - x_d)}  \mathbbm{1}\Big(\theta(x_{[k] \setminus \{d\}}, x_d + y_d) - \theta(x_{[k] \setminus \{d\}}, y_d) = \psi(x_{[k]})\Big)\ \times\nonumber\\
&\hspace{8cm}c(\phi(x_{[k]\setminus \{d\}}, x_d + y_d)) c(\phi(x_{[k]\setminus \{d\}}, y_d))\\ 
&\\
 =&\sum_{\substack{\lambda \in \mathbb{F}_p^l\\\mu \in M}} \sum_{\substack{x_{[k] \setminus \{d\}}\in G_{[k] \setminus \{d\}}\\\tau(x_{[k] \setminus \{d\}}) = \lambda\\\beta(x_{[k] \setminus \{d\}}) = \mu}} \sum_{x_d \in X_{x_{[k] \setminus \{d\}}}}\bm{w}(x_{[k]}) \sum_{y_d \in Y_{x_{[k] \setminus \{d\}}} \cap (Y_{x_{[k] \setminus \{d\}}} - x_d)}  \mathbbm{1}\Big(\theta(x_{[k] \setminus \{d\}}, x_d + y_d) - \theta(x_{[k] \setminus \{d\}}, y_d) = \psi(x_{[k]})\Big)\ \times\nonumber\\
&\hspace{8cm}c(\phi(x_{[k]\setminus \{d\}}, x_d + y_d)) c(\phi(x_{[k]\setminus \{d\}}, y_d)) \nonumber\\
&\\    =&\sum_{\substack{\lambda \in \mathbb{F}_p^l\\\mu \in M}} \sum_{\nu_1 \in \Lambda^\perp_\mu} \sum_{\substack{x_{[k] \setminus \{d\}} \in G_{[k] \setminus \{d\}}\\\tau(x_{[k] \setminus \{d\}}) = \lambda\\\beta(x_{[k] \setminus \{d\}}) = \mu}}\sum_{\nu_2 \in \Lambda^\perp_\mu}c(\lambda + \nu_1 + \nu_2) c(\lambda + \nu_2)  \sum_{\substack{x_d \in X_{x_{[k] \setminus \{d\}}}\\\phi'(x_{[k]}) = \nu_1}} \bm{w}(x_{[k]})\nonumber\\
&\hspace{4cm}\sum_{\substack{y_d \in Y_{x_{[k] \setminus \{d\}}} \cap (Y_{x_{[k] \setminus \{d\}}} - x_d)\\\phi'(x_{[k] \setminus \{d\}}, y_d) = \nu_2}}  \mathbbm{1}\Big(\theta(x_{[k] \setminus \{d\}}, x_d + y_d) - \theta(x_{[k] \setminus \{d\}}, y_d) = \psi(x_{[k]})\Big).
\end{align*}}

\chg{Similarly, the right hand side (ignoring the $(1 - o(1))$ factor) can be written as }
\chg{\begin{align*}
    \sum_{x_{[k]} \in G_{[k]}} &\bm{w}(x_{[k]}) \sum_{y_d \in G_d} c(\phi(x_{[k]\setminus \{d\}}, x_d + y_d)) c(\phi(x_{[k]\setminus \{d\}}, y_d))\nonumber\\
    =&\sum_{\substack{\lambda \in \mathbb{F}_p^l\\\mu \in \mathbb{F}_p^t}} \sum_{\substack{x_{[k] \setminus \{d\}} \in G_{[k] \setminus \{d\}}\\\tau(x_{[k] \setminus \{d\}}) = \lambda\\\beta(x_{[k] \setminus \{d\}}) = \mu}} \sum_{x_d \in G_d} \bm{w}(x_{[k]}) \sum_{y_d \in G_d} c(\phi(x_{[k]\setminus \{d\}}, x_d + y_d)) c(\phi(x_{[k]\setminus \{d\}}, y_d))\nonumber\\
\geq&\sum_{\substack{\lambda \in \mathbb{F}_p^l\\\mu \in M}} \sum_{\nu_1 \in \Lambda^\perp_\mu}\sum_{\substack{x_{[k] \setminus \{d\}} \in G_{[k] \setminus \{d\}}\\\tau(x_{[k] \setminus \{d\}}) = \lambda\\\beta(x_{[k] \setminus \{d\}}) = \mu}}\nonumber\sum_{\nu_2 \in \Lambda^\perp_\mu}c(\lambda + \nu_1 + \nu_2) c(\lambda + \nu_2) \sum_{\substack{x_d \in G_d\\\phi'(x_{[k]}) = \nu_1}} \bm{w}(x_{[k]}) \sum_{\substack{y_d \in G_d\\\phi'(x_{[k] \setminus \{d\}}, y_d) = \nu_2}} 1.
\end{align*}}

\chg{Thus, Claim A can be combined with these two conclusions to give}

\chg{\begin{align}\sum_{\substack{\lambda \in \mathbb{F}_p^l\\\mu \in M}} &\sum_{\nu_1 \in \Lambda^\perp_\mu} \sum_{\substack{x_{[k] \setminus \{d\}} \in G_{[k] \setminus \{d\}}\\\tau(x_{[k] \setminus \{d\}}) = \lambda\\\beta(x_{[k] \setminus \{d\}}) = \mu}}\sum_{\nu_2 \in \Lambda^\perp_\mu}c(\lambda + \nu_1 + \nu_2) c(\lambda + \nu_2)  \sum_{\substack{x_d \in X_{x_{[k] \setminus \{d\}}}\\\phi'(x_{[k]}) = \nu_1}} \bm{w}(x_{[k]})\nonumber\\
&\hspace{4cm}\sum_{\substack{y_d \in Y_{x_{[k] \setminus \{d\}}} \cap (Y_{x_{[k] \setminus \{d\}}} - x_d)\\\phi'(x_{[k] \setminus \{d\}}, y_d) = \nu_2}}  \mathbbm{1}\Big(\theta(x_{[k] \setminus \{d\}}, x_d + y_d) - \theta(x_{[k] \setminus \{d\}}, y_d) = \psi(x_{[k]})\Big)\\
\geq &(1-13 \varepsilon - O(\delta^{-O(1)}) \sum_{\substack{\lambda \in \mathbb{F}_p^l\\\mu \in M}} \sum_{\nu_1 \in \Lambda^\perp_\mu}\sum_{\substack{x_{[k] \setminus \{d\}} \in G_{[k] \setminus \{d\}}\\\tau(x_{[k] \setminus \{d\}}) = \lambda\\\beta(x_{[k] \setminus \{d\}}) = \mu}}\nonumber\\
&\hspace{4cm}\sum_{\nu_2 \in \Lambda^\perp_\mu}c(\lambda + \nu_1 + \nu_2) c(\lambda + \nu_2) \sum_{\substack{x_d \in G_d\\\phi'(x_{[k]}) = \nu_1}} \bm{w}(x_{[k]}) \sum_{\substack{y_d \in G_d\\\phi'(x_{[k] \setminus \{d\}}, y_d) = \nu_2}} 1.\label{approximateEqThetaPsi2}\end{align}}

Using~\eqref{approxWEqn} and~\eqref{cDenseWrtTildeW}, we deduce that
\begin{align}\sum_{\substack{\lambda \in \mathbb{F}_p^l\\\mu \in M}}& \sum_{\nu_1 \in \Lambda^\perp_\mu} \sum_{\substack{x_{[k] \setminus \{d\}} \in G_{[k] \setminus \{d\}}\\\tau(x_{[k] \setminus \{d\}}) = \lambda\\\beta(x_{[k] \setminus \{d\}}) = \mu}}\sum_{\nu_2 \in \Lambda^\perp_\mu}c(\lambda + \nu_1 + \nu_2) c(\lambda + \nu_2)  \tilde{\bm{w}}(\lambda + \nu_1) \nonumber\\
&\hspace{1cm}\sum_{\substack{x_d \in X_{x_{[k] \setminus \{d\}}}\\\phi'(x_{[k]}) = \nu_1}}  \sum_{\substack{y_d \in Y_{x_{[k] \setminus \{d\}}} \cap (Y_{x_{[k] \setminus \{d\}}} - x_d)\\\phi'(x_{[k] \setminus \{d\}}, y_d) = \nu_2}}  \mathbbm{1}\Big(\theta(x_{[k] \setminus \{d\}}, x_d + y_d) - \theta(x_{[k] \setminus \{d\}}, y_d) = \psi(x_{[k]})\Big)\nonumber\\
\geq &(1-13 \varepsilon - O(\delta^{-O(1)})\xi)\sum_{\substack{\lambda \in \mathbb{F}_p^l\\\mu \in M}} \sum_{\nu_1 \in \Lambda^\perp_\mu}\sum_{\substack{x_{[k] \setminus \{d\}} \in G_{[k] \setminus \{d\}}\\\tau(x_{[k] \setminus \{d\}}) = \lambda\\\beta(x_{[k] \setminus \{d\}}) = \mu}}\nonumber\\
&\hspace{4cm}\sum_{\nu_2 \in \Lambda^\perp_\mu}c(\lambda + \nu_1 + \nu_2) c(\lambda + \nu_2)  \tilde{\bm{w}}(\lambda + \nu_1) \sum_{\substack{x_d \in G_d\\\phi'(x_{[k]}) = \nu_1}} \sum_{\substack{y_d \in G_d\\\phi'(x_{[k] \setminus \{d\}}, y_d) = \nu_2}} 1\label{approximateEqThetaPsi2Tilde}.\end{align}

Provided that $\xi \leq \cons\,\delta^{\con}$, we may replace the constant at the beginning of the penultimate line by $(1-14\varepsilon)$.\\

\chg{\textbf{Step 6.} In this step, we apply Lemma~\ref{3approxHom} which provides us with the desired affine maps.} Recall that $U = \{x_{[k] \setminus \{d\}} \in G_{[k] \setminus \{d\}} \colon \beta(x_{[k] \setminus \{d\}}) \in M\}$. \chg{Next, we define a set $\tilde{X}$, which will consist of all pairs $(x_{[k] \setminus \{d\}}, \nu_1)$ for which we apply  Lemma~\ref{3approxHom} inside the coset $\{x_d \in G_d \colon \phi'(x_{[k]}) = \nu_1\}$.} To be precise, we define $\tilde{X}$ to be the set of all pairs $(x_{[k] \setminus \{d\}}, \nu_1) \in U \times \mathbb{F}_p^l$ such that $\nu_1 \in \Lambda^\perp_{\beta(x_{[k] \setminus \{d\}})}$ and
\begin{align*}&\sum_{\nu_2 \in \Lambda^\perp_{\beta(x_{[k] \setminus \{d\}})}}c(\tau(x_{[k] \setminus \{d\}}) + \nu_1 + \nu_2) c(\tau(x_{[k] \setminus \{d\}}) + \nu_2) \tilde{\bm{w}}(\tau(x_{[k] \setminus \{d\}}) + \nu_1) \sum_{\substack{x_d \in X_{x_{[k] \setminus \{d\}}}\\\phi'(x_{[k]}) = \nu_1}} \sum_{\substack{y_d \in Y_{x_{[k] \setminus \{d\}}} \cap (Y_{x_{[k] \setminus \{d\}}} - x_d)\\\phi'(x_{[k] \setminus \{d\}}, y_d) = \nu_2}}\\
&\hspace{6cm}\mathbbm{1}\Big(\theta(x_{[k] \setminus \{d\}}, x_d + y_d) - \theta(x_{[k] \setminus \{d\}}, y_d) = \psi(x_{[k]})\Big)\\
\geq&(1-\sqrt{\varepsilon}) \sum_{\nu_2 \in \Lambda^\perp_{\beta(x_{[k] \setminus \{d\}})}} c(\tau(x_{[k] \setminus \{d\}}) + \nu_1 + \nu_2) c(\tau(x_{[k] \setminus \{d\}}) + \nu_2) \tilde{\bm{w}}(\tau(x_{[k] \setminus \{d\}}) + \nu_1) \sum_{\substack{x_d \in G_d\\\phi'(x_{[k]}) = \nu_1}} \sum_{\substack{y_d \in G_d\\\phi'(x_{[k] \setminus \{d\}}, y_d) = \nu_2}} 1 \\
>& 0.
\end{align*}

For each $(x_{[k] \setminus \{d\}}, \nu_1) \in \tilde{X}$, average over $\nu_2 \in\Lambda^\perp_{\beta(x_{[k] \setminus \{d\}})}$ to find $\nu_2$ such that
\begin{align*}\sum_{\substack{x_d \in X_{x_{[k] \setminus \{d\}}}\\\phi'(x_{[k]}) = \nu_1}}&\sum_{\substack{y_d \in Y_{x_{[k] \setminus \{d\}}} \cap (Y_{x_{[k] \setminus \{d\}}} - x_d)\\\phi'(x_{[k] \setminus \{d\}}, y_d) = \nu_2}} \mathbbm{1}\Big(\theta(x_{[k] \setminus \{d\}}, x_d + y_d) - \theta(x_{[k] \setminus \{d\}}, y_d) = \psi(x_{[k]})\Big) \\
&\geq (1-\sqrt{\varepsilon})\sum_{\substack{x_d \in G_d\\\phi'(x_{[k]}) = \nu_1}} \sum_{\substack{y_d \in G_d\\\phi'(x_{[k] \setminus \{d\}}, y_d) = \nu_2}} 1\quad > \ 0.\end{align*}
Apply Lemma~\ref{3approxHom} to find a subset $X'_{x_{[k] \setminus \{d\}}, \nu_1} \subset \{x_d \in X_{x_{[k] \setminus \{d\}}} \colon \phi'(x_{[k]}) = \nu_1\}$ such that 
\[|X'_{x_{[k] \setminus \{d\}}, \nu_1}| \geq (1-2\sqrt[4]{\varepsilon}) \Big|\{x_d \in G_d \colon \phi'(x_{[k]}) = \nu_1\}\Big|\] 
and there is an affine map $\rho \colon G_d \to H$ such that $\rho(y_d) = \psi(x_{[k] \setminus \{d\}}, y_d)$ for all $y_d \in X'_{x_{[k] \setminus \{d\}}, \nu_1}$. Note that we require $\varepsilon < 10^{-4}$.\\

\chg{Having defined the sets $X'_{x_{[k] \setminus \{d\}}, \nu_1}$, we now derive a variant of inequality~\eqref{approximateEqThetaPsi2Tilde}, where the condition on equality of maps $\theta$ and $\psi$ is replaced by $x_d$ belonging to $X'_{x_{[k] \setminus \{d\}}, \nu_1}$.} Hence, when $(x_{[k] \setminus \{d\}}, \nu_1) \in \tilde{X}$, by definition we have in particular
\begin{align*}&\Big|\Big\{x_d \in X_{x_{[k] \setminus \{d\}}}\colon \phi'(x_{[k]}) = \nu_1\Big\} \Big|\sum_{\nu_2 \in \Lambda^\perp_{\beta(x_{[k] \setminus \{d\}})}}c(\tau(x_{[k] \setminus \{d\}}) + \nu_1 + \nu_2) c(\tau(x_{[k] \setminus \{d\}}) + \nu_2)\tilde{\bm{w}}(\tau(x_{[k] \setminus \{d\}}) + \nu_1)\\
&\hspace{4cm} \sum_{\substack{y_d \in G_d\\\phi'(x_{[k] \setminus \{d\}}, y_d) = \nu_2}} 1\\
=&\sum_{\nu_2 \in \Lambda^\perp_{\beta(x_{[k] \setminus \{d\}})}}c(\tau(x_{[k] \setminus \{d\}}) + \nu_1 + \nu_2) c(\tau(x_{[k] \setminus \{d\}}) + \nu_2) \tilde{\bm{w}}(\tau(x_{[k] \setminus \{d\}}) + \nu_1) \sum_{\substack{x_d \in X_{x_{[k] \setminus \{d\}}}\\\phi'(x_{[k]}) = \nu_1}} \sum_{\substack{y_d \in G_d\\\phi'(x_{[k] \setminus \{d\}}, y_d) = \nu_2}} 1\\
\geq&(1-\sqrt{\varepsilon}) \sum_{\nu_2 \in \Lambda^\perp_{\beta(x_{[k] \setminus \{d\}})}} c(\tau(x_{[k] \setminus \{d\}}) + \nu_1 + \nu_2) c(\tau(x_{[k] \setminus \{d\}}) + \nu_2) \tilde{\bm{w}}(\tau(x_{[k] \setminus \{d\}}) + \nu_1) \sum_{\substack{x_d \in G_d\\\phi'(x_{[k]}) = \nu_1}} \sum_{\substack{y_d \in G_d\\\phi'(x_{[k] \setminus \{d\}}, y_d) = \nu_2}} 1\\
=&(1-\sqrt{\varepsilon})\Big|\Big\{x_d \in G_d \colon\phi'(x_{[k]}) = \nu_1\Big\}\Big| \sum_{\nu_2 \in \Lambda^\perp_{\beta(x_{[k] \setminus \{d\}})}} c(\tau(x_{[k] \setminus \{d\}}) + \nu_1 + \nu_2) c(\tau(x_{[k] \setminus \{d\}}) + \nu_2)  \tilde{\bm{w}}(\tau(x_{[k] \setminus \{d\}}) + \nu_1)\\
&\hspace{4cm}\sum_{\substack{y_d \in G_d\\\phi'(x_{[k] \setminus \{d\}}, y_d) = \nu_2}} 1
 \quad > 0.
\end{align*}
\chg{Thus, for fixed $(x_{[k] \setminus \{d\}}, \nu_1) \in \tilde{X}$, inequality~\eqref{approximateEqThetaPsi2Tilde} essentially becomes}
\begin{align}& |X'_{x_{[k] \setminus \{d\}}, \nu_1}| \sum_{\nu_2 \in \Lambda^\perp_{\beta(x_{[k] \setminus \{d\}})}}c(\tau(x_{[k] \setminus \{d\}}) + \nu_1 + \nu_2) c(\tau(x_{[k] \setminus \{d\}}) + \nu_2)\tilde{\bm{w}}(\tau(x_{[k] \setminus \{d\}}) + \nu_1) \sum_{\substack{y_d \in G_d\\\phi'(x_{[k] \setminus \{d\}}, y_d) = \nu_2}} 1\nonumber\\
&\geq(1-3\sqrt[4]{\varepsilon}) \Big|\{x_d \in G_d\colon \phi'(x_{[k]}) = \nu_1\}\Big|  \sum_{\nu_2 \in \Lambda^\perp_{\beta(x_{[k] \setminus \{d\}})}}c(\tau(x_{[k] \setminus \{d\}}) + \nu_1 + \nu_2) c(\tau(x_{[k] \setminus \{d\}}) + \nu_2)\tilde{\bm{w}}(\tau(x_{[k] \setminus \{d\}}) + \nu_1)\nonumber\\
&\hspace{7cm}\sum_{\substack{y_d \in G_d\\\phi'(x_{[k] \setminus \{d\}}, y_d) = \nu_2}} 1\nonumber\\
&>0.\label{XprimeDefnEqn}
\end{align}
\chg{\textbf{Step 7.} In this step, we sum over all $(x_{[k] \setminus \{d\}}, \nu_1) \in \tilde{X}$.} \chg{To ease the notation, define auxiliary terms} $a^{(1)}_{x_{[k] \setminus\{d\}}, \nu_1}, a^{(2)}_{x_{[k] \setminus\{d\}}, \nu_1}$ and $a^{(3)}_{x_{[k] \setminus\{d\}}, \nu_1}$ by
\begin{align*}a^{(1)}_{x_{[k] \setminus\{d\}}, \nu_1} = &\sum_{\nu_2 \in \Lambda^\perp_{\beta(x_{[k] \setminus\{d\}})}}c(\tau(x_{[k] \setminus\{d\}}) + \nu_1 + \nu_2) c(\tau(x_{[k] \setminus\{d\}}) + \nu_2) \tilde{\bm{w}}(\tau(x_{[k] \setminus \{d\}}) + \nu_1) \sum_{\substack{x_d \in G_d\\\phi'(x_{[k]}) = \nu_1}} \sum_{\substack{y_d \in G_d\\\phi'(x_{[k] \setminus \{d\}}, y_d) = \nu_2}} 1,\\
a^{(2)}_{x_{[k] \setminus\{d\}}, \nu_1} = &\sum_{\nu_2 \in \Lambda^\perp_{\beta(x_{[k] \setminus\{d\}})}}c(\tau(x_{[k] \setminus\{d\}}) + \nu_1 + \nu_2) c(\tau(x_{[k] \setminus\{d\}}) + \nu_2)  \tilde{\bm{w}}(\tau(x_{[k] \setminus \{d\}}) + \nu_1) \sum_{\substack{x_d \in X_{x_{[k] \setminus \{d\}}}\\\phi'(x_{[k]}) = \nu_1}}\\
&\hspace{2cm}\sum_{\substack{y_d \in Y_{x_{[k] \setminus \{d\}}} \cap (Y_{x_{[k] \setminus \{d\}}} - x_d)\\\phi'(x_{[k] \setminus \{d\}}, y_d) = \nu_2}}  \mathbbm{1}\Big(\theta(x_{[k] \setminus \{d\}}, x_d + y_d) - \theta(x_{[k] \setminus \{d\}}, y_d) = \psi(x_{[k]})\Big)\\
a^{(3)}_{x_{[k] \setminus\{d\}}, \nu_1} = &|X'_{x_{[k] \setminus \{d\}}, \nu_1}| \sum_{\nu_2 \in \Lambda^\perp_{\beta(x_{[k] \setminus \{d\}})}}c(\tau(x_{[k] \setminus \{d\}}) + \nu_1 + \nu_2) c(\tau(x_{[k] \setminus \{d\}}) + \nu_2) \tilde{\bm{w}}(\tau(x_{[k] \setminus \{d\}}) + \nu_1) \sum_{\substack{y_d \in G_d\\\phi'(x_{[k] \setminus \{d\}}, y_d) = \nu_2}} 1.\end{align*}
By definition of $\tilde{X}$, we have that when $(x_{[k] \setminus \{d\}}, \nu_1) \in \tilde{X}$, then $a^{(2)}_{x_{[k] \setminus\{d\}}, \nu_1} \geq (1-\sqrt{\varepsilon}) a^{(1)}_{x_{[k] \setminus\{d\}}, \nu_1} > 0$ and from~\eqref{XprimeDefnEqn} we have $a^{(3)}_{x_{[k] \setminus\{d\}}, \nu_1} \geq (1-3\sqrt[4]{\varepsilon}) a^{(1)}_{x_{[k] \setminus\{d\}}, \nu_1}$. Using this notation~\eqref{approximateEqThetaPsi2Tilde} becomes
\[\sum_{x_{[k] \setminus \{d\}} \in U} \sum_{\nu_1 \in \Lambda^\perp_{\beta(x_{[k] \setminus \{d\}})}} a^{(2)}_{x_{[k] \setminus\{d\}}, \nu_1} \geq (1-14 \varepsilon) \sum_{x_{[k] \setminus \{d\}} \in U} \sum_{\nu_1 \in \Lambda^\perp_{\beta(x_{[k] \setminus \{d\}})}} a^{(1)}_{x_{[k] \setminus\{d\}}, \nu_1}.\]
This further gives 
\begin{align*}14\varepsilon \sum_{\substack{x_{[k] \setminus \{d\}} \in U\\\nu_1 \in \Lambda^\perp_{\beta(x_{[k] \setminus\{d\}})}}} a^{(1)}_{x_{[k] \setminus\{d\}}, \nu_1} \geq &\sum_{\substack{x_{[k] \setminus \{d\}} \in U\\\nu_1 \in \Lambda^\perp_{\beta(x_{[k] \setminus\{d\}})}}} (a^{(1)}_{x_{[k] \setminus\{d\}}, \nu_1} - a^{(2)}_{x_{[k] \setminus\{d\}}, \nu_1}) \\ 
\geq& \sum_{\substack{x_{[k] \setminus \{d\}} \in U\\\nu_1 \in \Lambda^\perp_{\beta(x_{[k] \setminus\{d\}})}\\    (x_{[k] \setminus \{d\}}, \nu_1) \notin \tilde{X}}} (a^{(1)}_{x_{[k] \setminus\{d\}}, \nu_1} - a^{(2)}_{x_{[k] \setminus\{d\}}, \nu_1}) \geq \sqrt{\varepsilon} \sum_{\substack{x_{[k] \setminus \{d\}} \in U\\\nu_1 \in \Lambda^\perp_{\beta(x_{[k] \setminus\{d\}})}\\ (x_{[k] \setminus \{d\}}, \nu_1) \notin \tilde{X}}} a^{(1)}_{x_{[k] \setminus\{d\}}, \nu_1}.\end{align*}
Hence,
\begin{align*}\sum_{\substack{x_{[k] \setminus \{d\}} \in U\\\nu_1 \in \Lambda^\perp_{\beta(x_{[k] \setminus\{d\}})}\\ (x_{[k] \setminus \{d\}}, \nu_1) \in \tilde{X}}} a^{(3)}_{x_{[k] \setminus\{d\}}, \nu_1} &\geq (1-3\sqrt[4]{\varepsilon})\sum_{\substack{x_{[k] \setminus \{d\}} \in U\\\nu_1 \in \Lambda^\perp_{\beta(x_{[k] \setminus\{d\}})}\\ (x_{[k] \setminus \{d\}}, \nu_1) \in \tilde{X}}} a^{(1)}_{x_{[k] \setminus\{d\}}, \nu_1}\\
&\geq (1-3\sqrt[4]{\varepsilon}) (1-14\sqrt{\varepsilon})\sum_{\substack{x_{[k] \setminus \{d\}} \in U\\\nu_1 \in \Lambda^\perp_{\beta(x_{[k] \setminus\{d\}})}}} a^{(1)}_{x_{[k] \setminus\{d\}}, \nu_1}.\end{align*}
Simplify the bound slightly by using $(1-3\sqrt[4]{\varepsilon}) (1-14\sqrt{\varepsilon}) \geq 1-20\sqrt[4]{\varepsilon}$ and expand out to get
\begin{align}&\sum_{(x_{[k] \setminus \{d\}}, \nu_1) \in \tilde{X}} \sum_{x_d \in X'_{x_{[k] \setminus \{d\}}, \nu_1}} \sum_{\nu_2 \in \Lambda^\perp_{\beta(x_{[k] \setminus \{d\}})}}  \sum_{\substack{y_d \in G_d\\\phi'(x_{[k] \setminus \{d\}}, y_d) = \nu_2}} c(\phi(x_{[k] \setminus \{d\}}, x_d + y_d)) c(\phi(x_{[k] \setminus \{d\}}, y_d))  \tilde{\bm{w}}(\phi(x_{[k]}))\nonumber\\
&=\sum_{(x_{[k] \setminus \{d\}}, \nu_1) \in \tilde{X}} |X'_{x_{[k] \setminus \{d\}}, \nu_1}| \sum_{\nu_2 \in \Lambda^\perp_{\beta(x_{[k] \setminus \{d\}})}}c(\tau(x_{[k] \setminus \{d\}}) + \nu_1 + \nu_2) c(\tau(x_{[k] \setminus \{d\}}) + \nu_2) \tilde{\bm{w}}(\tau(x_{[k] \setminus \{d\}}) + \nu_1) \nonumber\\
&\hspace{2cm}\sum_{\substack{y_d \in G_d\\\phi'(x_{[k] \setminus \{d\}}, y_d) = \nu_2}} 1\nonumber\\
&\geq (1-20\sqrt[4]{\varepsilon})\sum_{\substack{x_{[k] \setminus \{d\}} \in U\\\nu_1 \in \Lambda^\perp_{\beta(x_{[k] \setminus\{d\}})}}} \sum_{\nu_2 \in \Lambda^\perp_{\beta(x_{[k] \setminus\{d\}})}}c(\tau(x_{[k] \setminus\{d\}}) + \nu_1 + \nu_2) c(\tau(x_{[k] \setminus\{d\}}) + \nu_2)  \tilde{\bm{w}}(\tau(x_{[k] \setminus \{d\}}) + \nu_1) \nonumber\\
&\hspace{2cm}\sum_{\substack{x_d \in G_d\\\phi'(x_{[k]}) = \nu_1}}\sum_{\substack{y_d \in G_d\\\phi'(x_{[k] \setminus \{d\}}, y_d) = \nu_2}} 1\nonumber\\
&=(1-20\sqrt[4]{\varepsilon})\sum_{\substack{x_{[k] \setminus \{d\}} \in U\\\nu_1 \in \Lambda^\perp_{\beta(x_{[k] \setminus\{d\}})}}} \sum_{\substack{x_d \in G_d\\ \phi'(x_{[k]}) = \nu_1}} \sum_{\nu_2 \in \Lambda^\perp_{\beta(x_{[k] \setminus\{d\}})}} \sum_{\substack{y_d \in G_d\\\phi'(x_{[k] \setminus \{d\}}, y_d) = \nu_2}} c(\phi(x_{[k] \setminus \{d\}}, x_d + y_d)) c(\phi(x_{[k] \setminus \{d\}}, y_d))\tilde{\bm{w}}(\phi(x_{[k]})).\label{expandedCEqn}
\end{align}

\noindent\textbf{Step 8.} We now stop using $c$ and return to using $f^{(2,d)}$. We have
\begin{align*}&\sum_{\substack{\lambda \in \mathbb{F}_p^l\\\mu \in M}} \sum_{\nu_1 \in \Lambda^\perp_\mu}\sum_{\substack{x_{[k] \setminus \{d\}} \in \tilde{X}_{\nu_1}\\\tau(x_{[k] \setminus \{d\}}) = \lambda\\\beta(x_{[k] \setminus \{d\}}) = \mu}} \sum_{x_d \in X'_{x_{[k] \setminus \{d\}}, \nu_1}} \tilde{\bm{w}}(\phi(x_{[k]})) f^{(1,d)}(x_{[k]}) |G_d|\\
&\hspace{2cm}=\sum_{(x_{[k] \setminus \{d\}}, \nu_1) \in \tilde{X}} \sum_{x_d \in X'_{x_{[k] \setminus \{d\}}, \nu_1}} \tilde{\bm{w}}(\phi(x_{[k]})) f^{(1,d)}(x_{[k]}) |G_d|\\
&\hspace{2cm}=\sum_{(x_{[k] \setminus \{d\}}, \nu_1) \in \tilde{X}} \sum_{x_d \in X'_{x_{[k] \setminus \{d\}}, \nu_1}} \tilde{\bm{w}}(\phi(x_{[k]})) \sum_{y_d \in G_d} f^{(2,d)}(x_{[k]\setminus \{d\}}, x_d + y_d)f^{(2,d)}(x_{[k]\setminus \{d\}}, y_d).\\
\end{align*}
Using~\eqref{f2cApprox}, we see that this is at least
\begin{align*}&\sum_{(x_{[k] \setminus \{d\}}, \nu_1) \in \tilde{X}} \sum_{x_d \in X'_{x_{[k] \setminus \{d\}}, \nu_1}} \tilde{\bm{w}}(\phi(x_{[k]})) \sum_{y_d \in G_d} c(\phi(x_{[k] \setminus \{d\}}, x_d + y_d)) c(\phi(x_{[k] \setminus \{d\}}, y_d)) - 2\xi |G_{[k]}||G_d|\\
&=\sum_{(x_{[k] \setminus \{d\}}, \nu_1) \in \tilde{X}} \sum_{x_d \in X'_{x_{[k] \setminus \{d\}}, \nu_1}} \tilde{\bm{w}}(\phi(x_{[k]})) \sum_{\nu_2 \in \Lambda^\perp_{\beta(x_{[k] \setminus \{d\}})}}  \sum_{\substack{y_d \in G_d\\\phi'(x_{[k] \setminus \{d\}}, y_d) = \nu_2}} c(\phi(x_{[k] \setminus \{d\}}, x_d + y_d)) c(\phi(x_{[k] \setminus \{d\}}, y_d))\\
&\hspace{12cm}-2\xi |G_{[k]}||G_d|\\
&\geq(1-20\sqrt[4]{\varepsilon})\sum_{\substack{x_{[k] \setminus \{d\}} \in U\\\nu_1 \in \Lambda^\perp_{\beta(x_{[k] \setminus\{d\}})}}} \sum_{\substack{x_d \in G_d\\ \phi'(x_{[k]}) = \nu_1}} \tilde{\bm{w}}(\phi(x_{[k]})) \sum_{\nu_2 \in \Lambda^\perp_{\beta(x_{[k] \setminus\{d\}})}} \sum_{\substack{y_d \in G_d\\\phi'(x_{[k] \setminus \{d\}}, y_d) = \nu_2}} c(\phi(x_{[k] \setminus \{d\}}, x_d + y_d)) c(\phi(x_{[k] \setminus \{d\}}, y_d))\\
&\hspace{12cm}-2\xi |G_{[k]}||G_d|\\
\end{align*}
by~\eqref{expandedCEqn}. Using~\eqref{f2cApprox} in a similar way to the way we used it above, we have that this is at least
\begin{align*}&(1-20\sqrt[4]{\varepsilon})\sum_{\substack{x_{[k] \setminus \{d\}} \in U\\\nu_1 \in \Lambda^\perp_{\beta(x_{[k] \setminus\{d\}})}}} \sum_{\substack{x_d \in G_d\\ \phi'(x_{[k]}) = \nu_1}}\tilde{\bm{w}}(\phi(x_{[k]})) \sum_{\nu_2 \in \Lambda^\perp_{\beta(x_{[k] \setminus\{d\}})}} \sum_{\substack{y_d \in G_d\\\phi'(x_{[k] \setminus \{d\}}, y_d) = \nu_2}}  f^{(2,d)}(x_{[k] \setminus \{d\}}, x_d + y_d) f^{(2,d)}(x_{[k] \setminus \{d\}}, y_d)\\
&\hspace{12cm}-4\xi |G_{[k]}||G_d|\\
\end{align*}
which by~\eqref{uSizeBound} is at least
\begin{align*}
(1-20\sqrt[4]{\varepsilon})\sum_{x_{[k]} \in G_{[k]}} \sum_{y_d \in G_d} &\tilde{\bm{w}}(\phi(x_{[k]}))f^{(2,d)}(x_{[k] \setminus \{d\}}, x_d + y_d) f^{(2,d)}(x_{[k] \setminus \{d\}}, y_d) - 5\xi |G_{[k]}| |G_d|\\
=&(1-20\sqrt[4]{\varepsilon})\sum_{x_{[k]} \in G_{[k]}} \tilde{\bm{w}}(\phi(x_{[k]}))f^{(1,d)}(x_{[k]})|G_d|\, - \,5\xi |G_{[k]}| |G_d|.\\
\end{align*}
Provided that $\xi \leq \cons\,\delta^{\con}$, we may use~\eqref{tuplesArrDenseEq} and~\eqref{approxWEqn} to deduce that this is at least
\[(1-20\sqrt[4]{\varepsilon} - O(\delta^{-O(1)})\xi)\sum_{x_{[k]} \in G_{[k]}} \bm{w}(x_{[k]})f^{(1,d)}(x_{[k]})|G_d|.\]

Further, assuming $\xi \leq \cons\,\delta^{\con}\,\varepsilon$ allows us to replace $1-20\sqrt[4]{\varepsilon} - O(\delta^{-O(1)})\xi$ by $1-21\sqrt[4]{\varepsilon}$, simplifying the bound above.\\

\noindent\textbf{Step 9.} We now stop using $\tilde{\bm{w}}$ and return to using $\bm{w}$. Using~\eqref{tuplesArrDenseEq} and~\eqref{approxWEqn}, we obtain
\begin{align}&\sum_{\substack{\lambda \in \mathbb{F}_p^l\\\mu \in M}} \sum_{\nu_1 \in \Lambda^\perp_\mu}\sum_{\substack{x_{[k] \setminus \{d\}} \in \tilde{X}_{\nu_1}\\\tau(x_{[k] \setminus \{d\}}) = \lambda\\\beta(x_{[k] \setminus \{d\}}) = \mu}} \sum_{x_d \in X'_{x_{[k] \setminus \{d\}}, \nu_1}} \bm{w}(x_{[k]}) f^{(1,d)}(x_{[k]})\nonumber\\
&\hspace{2cm}\geq(1-21\sqrt[4]{\varepsilon} - O(\delta^{-O(1)})\xi)\sum_{x_{[k]} \in G_{[k]}} \bm{w}(x_{[k]}) f^{(1,d)}(x_{[k]}).\label{step6arransAlmostAll}
\end{align}
Again, assuming $\xi \leq \cons\,\delta^{\con}\,\varepsilon$ allows simplification of the constant in the last line to $1-22\sqrt[4]{\varepsilon}$.\\

\noindent\textbf{Step 10.} Finally, we put the sets $X'_{x_{[k] \setminus \{d\}}, \nu_1}$ together and organize them in the desired form. Let $P$ be the set of all $(\lambda, \mu, \nu_1) \in \mathbb{F}_p^l \times M \times \mathbb{F}_p^l$ such that $\nu_1 \in \Lambda^\perp_\mu$ and
\[\sum_{\substack{x_{[k] \setminus \{d\}} \in \tilde{X}_{\nu_1}\\\tau(x_{[k] \setminus \{d\}}) = \lambda\\\beta(x_{[k] \setminus \{d\}}) = \mu}} \sum_{x_d \in X'_{x_{[k] \setminus \{d\}}, \nu_1}} \bm{w}(x_{[k]})f^{(1,d)}(x_{[k]}) \geq (1-\sqrt[8]{\varepsilon}) \sum_{\substack{x_{[k] \setminus \{d\}} \in G_{[k] \setminus \{d\}}\\\tau(x_{[k] \setminus \{d\}}) = \lambda\\\beta(x_{[k] \setminus \{d\}}) = \mu}} \sum_{\substack{x_d \in G_d\\ \phi'(x_{[k]}) = \nu_1}} \bm{w}(x_{[k]}) f^{(1,d)}(x_{[k]}).\]
Define a multiaffine map $\gamma \colon G_{[k]} \to \mathbb{F}_p^{l} \times \mathbb{F}_p^{t} \times \mathbb{F}_p^{l}$ by
\[\gamma(x_{[k]}) = \Big(\tau(x_{[k] \setminus \{d\}}), \beta(x_{[k] \setminus \{d\}}), \phi'(x_{[k]})\Big).\]
For $(\lambda, \mu, \nu_1) \in P$, define $S^{(\lambda, \mu, \nu_1)}$ as 
\[S^{(\lambda, \mu, \nu_1)} = \Big\{x_{[k]} \in G_{[k]} \colon x_{[k] \setminus \{d\}} \in \tilde{X}_{\nu_1}, x_d \in X'_{x_{[k] \setminus \{d\}}, \nu_1}, \tau(x_{[k] \setminus \{d\}}) = \lambda, \beta(x_{[k] \setminus \{d\}}) = \mu\Big\},\]
and for completeness set $S^{(\lambda, \mu, \nu_1)} = \emptyset$ for the remaining choices of $(\lambda, \mu, \nu_1) \in \mathbb{F}_p^l \times \mathbb{F}_p^t \times \mathbb{F}_p^l$. Clearly, for all $(\lambda, \mu, \nu_1) \in \mathbb{F}_p^l \times \mathbb{F}_p^t \times \mathbb{F}_p^l$, we have
\[S^{(\lambda, \mu, \nu_1)} \subset X \cap \{x_{[k]} \in G_{[k]} \colon \gamma(x_{[k]}) = (\lambda, \mu, \nu_1)\}.\]
By the way we obtained sets $X'_{x_{[k] \setminus \{d\}}}$, we see that $\psi|_{S^{(\lambda, \mu, \nu_1)}}$ is affine in direction $d$ in the sense described in the statement of the proposition. It remains to check that
\begin{equation}\label{step7lastClaim}\sum_{(\lambda, \mu, \nu_1) \in P} \sum_{x_{[k]} \in S^{(\lambda, \mu, \nu_1)}} \bm{w}(x_{[k]}) f^{(1,d)}(x_{[k]}) = (1 - O(\sqrt[8]{\varepsilon})) \sum_{x_{[k]} \in G_{[k]}} \bm{w}(x_{[k]}) f^{(1,d)}(x_{[k]}).\end{equation}
Inequality~\eqref{step6arransAlmostAll} implies that
\begin{align*}\sqrt[8]{\varepsilon}&\sum_{\substack{\lambda \in \mathbb{F}_p^l\\\mu \in M\\\nu_1 \in \Lambda^\perp_\mu}} \mathbbm{1}\Big((\lambda, \mu, \nu_1) \notin P\Big)\sum_{\substack{x_{[k] \setminus \{d\}} \in \tilde{X}_{\nu_1}\\\tau(x_{[k] \setminus \{d\}}) = \lambda\\\beta(x_{[k] \setminus \{d\}}) = \mu}} \sum_{x_d \in X'_{x_{[k] \setminus \{d\}}, \nu_1}} \bm{w}(x_{[k]})f^{(1,d)}(x_{[k]})\\
\leq&\sqrt[8]{\varepsilon}\sum_{\substack{\lambda \in \mathbb{F}_p^l\\\mu \in M\\\nu_1 \in \Lambda^\perp_\mu}} \mathbbm{1}\Big((\lambda, \mu, \nu_1) \notin P\Big) \sum_{\substack{x_{[k] \setminus \{d\}} \in G_{[k] \setminus \{d\}}\\\tau(x_{[k] \setminus \{d\}}) = \lambda\\\beta(x_{[k] \setminus \{d\}}) = \mu}} \sum_{\substack{x_d \in G_d\\ \phi'(x_{[k]}) = \nu_1}} \bm{w}(x_{[k]}) f^{(1,d)}(x_{[k]})\\
\leq &\sum_{\substack{\lambda \in \mathbb{F}_p^l\\\mu \in M\\\nu_1 \in \Lambda^\perp_\mu}} \mathbbm{1}\Big((\lambda, \mu, \nu_1) \notin P\Big) \Bigg(\sum_{\substack{x_{[k] \setminus \{d\}} \in G_{[k] \setminus \{d\}}\\\tau(x_{[k] \setminus \{d\}}) = \lambda\\\beta(x_{[k] \setminus \{d\}}) = \mu}} \sum_{\substack{x_d \in G_d\\ \phi'(x_{[k]}) = \nu_1}} \bm{w}(x_{[k]}) f^{(1,d)}(x_{[k]}) \\
&\hspace{8cm}-\sum_{\substack{x_{[k] \setminus \{d\}} \in \tilde{X}_{\nu_1}\\\tau(x_{[k] \setminus \{d\}}) = \lambda\\\beta(x_{[k] \setminus \{d\}}) = \mu}} \sum_{x_d \in X'_{x_{[k] \setminus \{d\}}, \nu_1}} \bm{w}(x_{[k]})f^{(1,d)}(x_{[k]})\Bigg)\\
\leq &\sum_{\substack{\lambda \in \mathbb{F}_p^l\\\mu \in M\\\nu_1 \in \Lambda^\perp_\mu}} \Bigg(\sum_{\substack{x_{[k] \setminus \{d\}} \in G_{[k] \setminus \{d\}}\\\tau(x_{[k] \setminus \{d\}}) = \lambda\\\beta(x_{[k] \setminus \{d\}}) = \mu}} \sum_{\substack{x_d \in G_d\\ \phi'(x_{[k]}) = \nu_1}} \bm{w}(x_{[k]}) f^{(1,d)}(x_{[k]}) - \sum_{\substack{x_{[k] \setminus \{d\}} \in \tilde{X}_{\nu_1}\\\tau(x_{[k] \setminus \{d\}}) = \lambda\\\beta(x_{[k] \setminus \{d\}}) = \mu}} \sum_{x_d \in X'_{x_{[k] \setminus \{d\}}, \nu_1}} \bm{w}(x_{[k]})f^{(1,d)}(x_{[k]})\Bigg)\\
\leq&\,22\sqrt[4]{\varepsilon} \sum_{x_{[k]} \in G_{[k]}} \bm{w}(x_{[k]}) f^{(1,d)}(x_{[k]}).\hspace{4cm}\text{(using~\eqref{step6arransAlmostAll})}\end{align*}

The desired inequality~\eqref{step7lastClaim} now follows from this and~\eqref{step6arransAlmostAll} and the proof is complete. We had the requirement that $\xi\leq\mathbf{c}\delta^\con\epsilon$, which we may satisfy with a choice $\xi = \Omega(\delta^{O(1)}) \varepsilon$.\end{proof}

\begin{proof}[Proof of Theorem~\ref{nearlyMultThm}]For each $d \in [k]$, apply Proposition~\ref{nearlyMaffStep} in direction $d$ to obtain a positive integer $t^{(d)} = \exp^{\big((2k + 1)(D^{\mathrm{mh}}_{k-1} + 2)\big)}\Big(O((\delta\varepsilon)^{-O(1)})\Big)$, a multiaffine map $\alpha^{(d)} \colon G_{[k]} \to \mathbb{F}_p^{t^{(d)}}$, and a collection of disjoint sets $(S^{(d), \lambda})_{\lambda \in \mathbb{F}_p^{t^{(d)}}}$ such that $S^{(d), \lambda} \subset \{x_{[k]} \in G_{[k]} \colon \alpha^{(d)}(x_{[k]}) = \lambda\} \cap X$, $\psi|_{S^{(d), \lambda}}$ is affine in direction $d$ for each $\lambda \in \mathbb{F}_p^{t^{(d)}}$, and
\[\sum_{x_{[k]} \in \bigcup_{\lambda \in \mathbb{F}^{t^{(d)}}}\, S^{(d), \lambda}} |\mathcal{Q}_{x_{[k]}}| = \Big(1-O(\sqrt[8]{\varepsilon})\Big) |\mathcal{Q}|,\]
where $\mathcal{Q}$ and $\mathcal{Q}_{x_{[k]}}$ have the same meaning as in Proposition~\ref{nearlyMaffStep}. Write $\bm{\alpha} = (\alpha^{(1)}, \dots, \alpha^{(k)})$. For each tuple $\bm{\lambda} = (\lambda^{(1)}, \dots, \lambda^{(k)})$ such that $\lambda^{(d)} \in \mathbb{F}^{t^{(d)}}$, define $S^{\bm{\lambda}} \subset G_{[k]}$ by
\[S^{\bm{\lambda}} = \bigcap_{d \in [k]} S^{(d), \lambda^{(d)}} \subset X \cap \{x_{[k]}:\bm{\alpha}(x_{[k]}) = \bm{\lambda}(x_{[k]})\}.\]
Then for each $\bm{\lambda}$, $\psi|_{S^{\bm{\lambda}}}$ is multiaffine in the sense explained in the statement of the theorem, and
\begin{equation}\label{nearlyMaffPiecesEqn}\sum_{\bm{\lambda} \in \mathbb{F}_p^{t^{(1)}} \oplus \dots \oplus \mathbb{F}_p^{t^{(k)}}} \sum_{x_{[k]} \in S^{\bm{\lambda}}} |\mathcal{Q}_{x_{[k]}}| = \Big(1-O(\sqrt[8]{\varepsilon})\Big) \sum_{\bm{\lambda} \in \mathbb{F}_p^{t^{(1)}} \oplus \dots \oplus \mathbb{F}_p^{t^{(k)}}} \sum_{\substack{x_{[k]} \in G_{[k]}\\\bm{\alpha}(x_{[k]}) = \bm{\lambda}}} |\mathcal{Q}_{x_{[k]}}|.\end{equation}
Let $\xi > 0$. As in the proof of Proposition~\ref{nearlyMaffStep} (\textbf{Step 3}), for each $i \in [k]$, apply Theorem~\ref{strongMixedApprox} to $A$ for the $L^k$ norm to \chg{obtain a positive integer $m^{(i)} =  \exp^{\big(D^{\mathrm{mh}}_{k-1} + D^{\mathrm{mh}}_{k - 2} + \dots + D^{\mathrm{mh}}_{1} + 3k\big)}\Big(O(2^{O(k)}\xi^{-O(1)})\Big)$, a map $c^{(i)}: \mathbb{F}_p^{m^{(i)}} \to \mathbb{D}$, and a multiaffine map $\phi^{(i)} \colon G_{[k]} \to \mathbb{F}_p^{m^{(i)}}$ such that}
\[f^{(1, i)} = \bigconv{i}\dots\bigconv{1}\bigconv{k}\dots\bigconv{1}\bigconv{k}\dots\bigconv{i + 1} A \apps{\xi}_{L^k} c^{(i)} \circ\phi^{(i)}.\]
Write $\bm{\phi} = (\phi^{(1)}, \dots, \phi^{(k)})$ and define $c \colon \mathbb{F}_p^{m^{(1)}} \oplus \dots \oplus \mathbb{F}_p^{m^{(k)}} \to [0,1]$ by
\[c(\mu^{(1)}, \dots, \mu^{(k)}) = \min\Big\{\max\Big\{\operatorname{Re} \Big(\prod_{i \in [k]} c^{(i)}(\mu^{(i)})\Big), 0\Big\}, 1\Big\},\]
where $\mu^{(d)} \in \mathbb{F}_p^{m^{(d)}}$. Similarly to~\eqref{approxWEqn}, we obtain that
\[\exx_{x_{[k]} \in G_{[k]}} \Big||\mathcal{Q}_{x_{[k]}}| - c(\bm{\phi}(x_{[k]}))|G_{[k]}|^{2^{2k} - 1} \Big| \leq 2^k \xi \chg{|G_{[k]}|^{2^{2k} - 1}}.\]
Recall also from~\eqref{tuplesArrDenseEq} that
\[\exx_{x_{[k]} \in G_{[k]}} |\mathcal{Q}_{x_{[k]}}| = \Omega(\delta^{O(1)})|G_{[k]}|^{2^{2k}}.\]
Returning to~\eqref{nearlyMaffPiecesEqn}, we obtain 
\[\sum_{\bm{\lambda} \in \mathbb{F}_p^{t^{(1)}} \oplus \dots \oplus \mathbb{F}_p^{t^{(k)}}} \sum_{x_{[k]} \in S^{\bm{\lambda}}} c(\bm{\phi}(x_{[k]})) = \Big(1-O(\sqrt[8]{\varepsilon}) - O(\delta^{-O(1)})\xi\Big) \sum_{\bm{\lambda} \in \mathbb{F}_p^{t^{(1)}} \oplus \dots \oplus \mathbb{F}_p^{t^{(k)}}} \sum_{\substack{x_{[k]} \in G_{[k]}\\\bm{\alpha}(x_{[k]}) = \bm{\lambda}}} c(\bm{\phi}(x_{[k]})) > 0.\]
Pick $\xi = \Omega(\delta^{O(1)})\varepsilon$ so that $1-O(\sqrt[8]{\varepsilon}) - O(\delta^{-O(1)})\xi$ becomes just $1-O(\sqrt[8]{\varepsilon})$. Average over $\bm{\lambda}\in \mathbb{F}_p^{t^{(1)}} \oplus \dots \oplus \mathbb{F}_p^{t^{(k)}}$ and $\bm{\mu} \in \mathbb{F}_p^{m^{(1)}} \oplus \dots \oplus \mathbb{F}_p^{m^{(k)}}$ to find values such that
\[\sum_{\substack{x_{[k]} \in S^{\bm{\lambda}}\\\bm{\phi}(x_{[k]}) = \bm{\mu}}} c(\bm{\mu}) = (1-O(\sqrt[8]{\varepsilon})) \sum_{\substack{x_{[k]} \in G_{[k]}\\\bm{\alpha}(x_{[k]}) = \bm{\lambda}\\\bm{\phi}(x_{[k]}) = \bm{\mu}}} c(\bm{\mu}) > 0,\]
which gives 
\[|S^{\bm{\lambda}} \cap \bm{\phi}^{-1}(\bm{\mu})|\,\, = (1-O(\sqrt[8]{\varepsilon}))\, |\bm{\alpha}^{-1}(\bm{\lambda}) \cap \bm{\phi}^{-1}(\bm{\mu})|\,\, > 0,\] 
as desired.
\end{proof}

\section{Biaffine maps on biaffine varieties}\label{biaffinemapssection}\label{sec8labelbiaffinemapstheory}

In this section, we study bihomomorphisms defined (typically) on very dense subsets of biaffine varieties. Most of the results that follow are similar to those in~\cite{U4paper}, but are more streamlined. 

\subsection{Quasirandomness of biaffine varieties}
A very useful property that some but not all biaffine varieties have is that if we regard them as bipartite graphs, then those bipartite graphs are quasirandom. In such a situation we shall call the varieties themselves quasirandom. More precisely we make the following definition.
\begin{defin}\label{qrDefin}Let $C_1 \subset G_1, C_2 \subset G_2$ be cosets of some subspaces inside $G_1$ and $G_2$, let $\beta \colon G_1 \times G_2 \to H$ be a biaffine map, and let $\lambda \in H$. The quadruple $(\beta, \lambda, C_1, C_2)$ is \emph{$\eta$-quasirandom with density $\delta$} if the variety $V = \{(x,y) \in G_1 \times G_2 \colon \beta(x,y) = \lambda\} \cap (C_1 \times C_2)$ satisfies
\begin{itemize}
\item for at least a $ 1-\eta$ proportion of the elements $x \in C_1$
\[|V_{x}| = \delta |C_2|,\]
\item for at least a $ 1-\eta$ proportion of the pairs $(x_1, x_2) \in C_1\times C_1$
\[|V_{x_1} \cap V_{x_2}| = \delta^2 |C_2|.\]
\end{itemize}
If the cosets $C_1, C_2$, the map $\beta$ and the element $\lambda$ are clear from the context, we say that $V$ is \emph{$\eta$-quasirandom with density $\delta$}.\end{defin}

\noindent\textbf{Remark.} For now, we shall only consider quasirandom varieties with $\delta > 0$. As we shall see later, when we have control over the dimension of $H$, we will be able to deduce that quasirandom varieties with $\delta = 0$ are in fact empty sets.\\

Note that here we use the notation $V_x$ without specifying the index of the coordinate explicitly. Since we only have two variables, we shall use $V_x$ to denote the slice with the first coordinate fixed and $V_y$ for the slice with the second coordinate fixed. Later, when the arguments become more involved, we shall use a more elaborate notation, but we prefer this notation for its simplicity for now.\\

This implies that the balanced function $\nu$ defined by $\nu(x,y) = V(x,y) - \delta$ satisfies
\begin{align*}\exx_{\substack{x_1, x_2 \in C_1\\y_1, y_2 \in C_2}} \nu(x_1, y_1) &\nu(x_2, y_1) \nu(x_1, y_2) \nu(x_2, y_2)\\
= &\exx_{x_1, x_2 \in C_1} \Big|\exx_{y \in C_2} \nu(x_1, y)\nu(x_2, y)\Big|^2 \\
= &\exx_{x_1, x_2 \in C_1} \Big|\exx_{y \in C_2} \mathbbm{1}(\beta(x_1, y) = \lambda)\mathbbm{1}(\beta(x_2, y) = \lambda) - \delta \mathbbm{1}(\beta(x_1, y) = \lambda) - \delta\mathbbm{1}(\beta(x_2, y) = \lambda) + \delta^2\Big|^2 \\
= &\exx_{x_1, x_2 \in C_1} \Big||C_2|^{-1} |V_{x_1} \cap V_{x_2}| - \delta |C_2|^{-1} |V_{x_1}| - \delta |C_2|^{-1} |V_{x_2}| + \delta^2\Big|^2 \\
\leq &12\eta,\end{align*}
and therefore that $\|\nu\|_{\square} \leq 2 \eta^{1/4}$.\\

\indent Using this property, we may deduce quasirandomness in direction $G_2$.

\begin{lemma}\label{yQR}Let $\beta, \lambda, C_1, C_2, V, \delta, \eta$ be as in Definition~\ref{qrDefin}. For a proportion of at least $1-8\delta^{-2}\sqrt[4]{\eta}$ of $y \in C_2$, we have $|V_{y}| = \delta |C_1|$, and for a proportion of at least $1-32\delta^{-4}\sqrt[4]{\eta}$ of the pairs $(y_1, y_2) \in C_2\times C_2$, we have $|V_{y_1} \cap V_{y_2}| = \delta^2 |C_1|$.\end{lemma}

\begin{proof}We have
\begin{align*}\exx_{y \in C_2}\Big||C_1|^{-1}|V_y| - \delta \Big|^2 = &\exx_{y \in C_2} \Big|\exx_{x \in C_1} \nu(x,y)\Big|^2 \\
= &\exx_{x_1, x_2 \in C_1} \exx_{y \in C_2} \nu(x_1,y) \nu(x_2,y)\\
\leq & \exx_{x_2 \in C_1} \Big| \exx_{x_1 \in C_1} \exx_{y \in C_2} \nu(x_1, y) \nu(x_2, y)\Big|\\
\leq &\|\nu\|_{\square},\end{align*}

where we applied Corollary~\ref{boxUnifCor} in the last line. Since $V_y$ is also a coset of a subspace in $G_1$ and $\delta$ is already known to be the density of some coset (and thus a non-positive power of $p$), either $|V_y| = \delta |C_1|$ or $\Big||C_1|^{-1}|V_y| - \delta \Big| \geq \frac{1}{2}\delta$. The first part of the claim now follows.\\

Similarly, we have
\begin{align*}\exx_{y_1, y_2 \in C_2}\Big||C_1|^{-1}|V_{y_1} \cap V_{y_2}| - \delta^2 \Big|^2 = &\exx_{y_1, y_2 \in C_2}\Big| \exx_{x \in C_1} V(x, y_1) V(x, y_2) - \delta^2\Big|^2\\
= &\exx_{y_1, y_2 \in C_2}\Big| \exx_{x \in C_1} \nu(x, y_1) V(x, y_2) + \delta \nu(x,y_2)\Big|^2\\
\leq & 2\exx_{y_1, y_2 \in C_2}\Big| \exx_{x \in C_1} \nu(x, y_1) V(x, y_2)\Big|^2 + 2\exx_{y_1, y_2 \in C_2}\Big| \exx_{x \in C_1}\delta \nu(x,y_2)\Big|^2\\
= &2\exx_{y_1, y_2 \in C_2} \exx_{x_1, x_2 \in C_1} \nu(x_1, y_1) V(x_1, y_2) \nu(x_2, y_1) V(x_2, y_2)\\
&\hspace{2cm}+2\exx_{y_1, y_2 \in C_2} \exx_{x_1, x_2 \in C_1}\delta^2 \nu(x_1,y_2)\nu(x_2, y_2)\\
\leq &4\|\nu\|_{\square},\end{align*}

where we used Lemma~\ref{boxCS} twice in the last step.\end{proof}

We now prove some useful properties of quasirandom varieties. The first one says that the intersection of $k$ columns most often has density exactly $\delta^k$.

\begin{lemma}\label{regInt1}Let $\beta, \lambda, C_1, C_2, V, \delta, \eta$ be as in Definition~\ref{qrDefin}. Pick $x_1, \dots, x_k$ independently and uniformly from $C_1$. Then
\[\mathbb{P}\Big(|V_{x_1} \cap \dots \cap V_{x_k}| = \delta^k |C_2|\Big) \geq 1 - 8 k^2 \delta^{-2k} \sqrt[4]{\eta}.\]
\end{lemma}

\begin{proof}We have
\begin{align*}\exx_{x_1, \dots, x_k \in C_1}& \Big||C_2|^{-1}\hspace{2pt}|V_{x_1} \cap \dots \cap V_{x_k}| - \delta^k \Big|^2\\
=& \exx_{x_1, \dots, x_k \in C_1} \Big|\exx_{y \in C_2} V(x_1, y) \cdots V(x_k,y) - \delta^k\Big|^2\\
=& \exx_{x_1, \dots, x_k \in C_1} \Big|\sum_{i \in [k]} \exx_{y \in C_2} \delta^{i-1} \nu(x_i, y) V(x_{i+1}, y)\cdots V(x_k,y)\Big|^2\\
&\hspace{2cm}\text{(by the Cauchy--Schwarz inequality)}\\
\leq &\exx_{x_1, \dots, x_k \in C_1}  k\sum_{i \in [k]} \Big|\exx_{y \in C_2} \delta^{i-1} \nu(x_i, y) V(x_{i+1}, y)\cdots V(x_k,y)\Big|^2\\
=&k\sum_{i \in [k]}\exx_{x_1, \dots, x_k \in C_1}\exx_{y,z \in C_2} \delta^{2i-2} \nu(x_i, y)\nu(x_i, z) V(x_{i+1}, y)V(x_{i+1}, z)\cdots V(x_k,y)V(x_k,z)\\
\leq&k \sum_{i \in [k]} \exx_{x_1, \dots,x_{i-1}, x_{i+1}, \dots x_k \in C_1} \exx_{y \in C_2}\Big|\exx_{x_i \in C_1, z \in C_2} \delta^{2i-2} \nu(x_i, y)\hspace{3pt}\nu(x_i, z)\\
&\hspace{7cm}V(x_{i+1}, y)V(x_{i+1}, z)\cdots V(x_k,y)V(x_k,z)\Big|\\
&\hspace{2cm}\text{(by Corollary~\ref{boxUnifCor}: the only term that depends on both $x_i$ and $z$ is $\nu(x_i,z)$)}\\
\leq& k^2 \|\nu\|_{\square}\\
\leq& 2k^2 \sqrt[4]{\eta}.\qedhere
\end{align*}
\end{proof}

The next lemma is a generalization of the previous one in the sense that we consider intersections of $k$ randomly chosen columns with a fixed set. 

\begin{lemma}\label{regInt2}Let $\beta, \lambda, C_1, C_2, V, \delta, \eta$ be as in Definition~\ref{qrDefin}. Let $S \subset C_2$ and let $\varepsilon \in [20\delta^{-2}\sqrt[4]{\eta}, 1]$. Pick $x_1, \dots, x_k$ independently and uniformly from $C_1$. Then
\[\mathbb{P}\Big(\Big||S \cap V_{x_1} \cap \dots \cap V_{x_k}| - \delta^k |S|\Big| \geq \varepsilon|C_2|\Big) = O(\varepsilon^{-2} \delta^{-4}\sqrt[4]{\eta}).\]
\end{lemma}

\begin{proof}Let $N$ be the random variable $|S \cap V_{x_1} \cap \dots \cap V_{x_k}|$. (The expectation notation $\mathbb{E}$ in this proof has its usual, probabilistic meaning.) We have
\begin{align*} \Big|\mathbb{E} N - \delta^k |S|\Big|=& \Big|\sum_{y \in S} \mathbb{P}(y \in V_{x_1} \cap \dots \cap V_{x_k}) - \delta^k\Big|\\
=& \Big|\sum_{y \in S}\mathbb{P}(x_1, \dots, x_k \in V_{y}) -  \delta^{k}\Big|\\
=& \Big|\sum_{y \in S} (|C_1|^{-1} |V_{y}|)^k - \delta^{k}\Big| \\
&\hspace{2cm}\text{(by Lemma~\ref{yQR})}\\
\leq & 8\delta^{-2}\sqrt[4]{\eta}|C_2|.\end{align*}
Next, we estimate the variance of $N$.
\begin{align*}\operatorname{var} N =& \mathbb{E}(N^2) - (\mathbb{E} N)^2\\
 =& \mathbb{E} \Big(\sum_{y \in S} \mathbbm{1}(y \in V_{x_1} \cap \dots \cap V_{x_k})\Big)^2 - \Big(\mathbb{E} \sum_{y \in S} \mathbbm{1}(y \in V_{x_1} \cap \dots \cap V_{x_k})\Big)^2\\
=& \sum_{y_1, y_2 \in S} \mathbb{E} \mathbbm{1}(y_1, y_2 \in V_{x_1} \cap \dots \cap V_{x_k}) - \mathbb{E}\Big(\mathbbm{1}(y_1 \in V_{x_1} \cap \dots \cap V_{x_k})\Big)\mathbb{E}\Big(\mathbbm{1}(y_2 \in V_{x_1} \cap \dots \cap V_{x_k})\Big)\\
=& \sum_{y_1, y_2 \in S}  \mathbb{P}(y_1, y_2 \in V_{x_1} \cap \dots \cap V_{x_k}) - \mathbb{P}(y_1 \in V_{x_1} \cap \dots \cap V_{x_k})\mathbb{P}(y_2 \in V_{x_1} \cap \dots \cap V_{x_k})\\
=& \sum_{y_1, y_2 \in S}  \mathbb{P}(x_1, \dots, x_k \in V_{y_1} \cap V_{y_2}) - \mathbb{P}(x_1, \dots, x_k \in V_{y_1})\mathbb{P}(x_1, \dots, x_k \in V_{y_2})\\
=& \sum_{y_1, y_2 \in S} \Big(\frac{|V_{y_1} \cap V_{y_2}|}{|C_1|}\Big)^k - \Big(\frac{|V_{y_1}|}{|C_1|}\Big)^k \Big(\frac{|V_{y_2}|}{|C_1|}\Big)^k\\
&\hspace{2cm}\text{(by Lemma~\ref{yQR})}\\ 
\leq& 100 \delta^{-4}\sqrt[4]{\eta}|C_2|^2.\end{align*}
In the last line we  considered separately those pairs $(y_1, y_2)$ where $|V_{y_1}|\not= \delta |C_1|$, those where $|V_{y_2}|\not= \delta|C_1|$, and those where $|V_{y_1} \cap V_{y_2}| \not= \delta^2 |C_1|$. We now use Chebyshev's inequality to get 
\begin{align*}\mathbb{P}\Big(\Big|N - \delta^k |S|\Big| \geq \varepsilon |C_2|\Big) &\leq \mathbb{P}\Big(|N - \mathbb{E} N| \geq (\varepsilon - 8\delta^{-2}\sqrt[4]{\eta})|C_2|\Big)\\
&\leq \frac{\operatorname{var} N}{(\varepsilon - 8\delta^{-2}\sqrt[4]{\eta})^2|C_2|^2}\\
&\leq \frac{1000 \delta^{-4}\sqrt[4]{\eta}}{\varepsilon^2}\\ 
&= O(\varepsilon^{-2} \delta^{-4}\sqrt[4]{\eta}),\end{align*}
which concludes the proof.\end{proof}

We need the following corollaries of Lemmas~\ref{regInt1} and~\ref{regInt2}.

\begin{corollary}\label{oneDenseColumnCor}Let $\beta, \lambda, C_1, C_2, V, \delta, \eta$ be as in Definition~\ref{qrDefin}. Let $r$ be the codimension of $\beta$ (that is, $\dim H$, where $H$ is the codomain of $\beta$). Let $x_0 \in C_1$ and let $S \subset V_{x_0 }$ be such that $|S|\, \geq (1-\varepsilon) |V_{x_0 }|\, > 0$. Let $k \in \mathbb{N}$. Suppose that $x_1, \dots, x_k$ are chosen uniformly and independently from $C_1$. Then, provided that $\varepsilon \geq 20 \delta^{-2}p^{(k+1)r} \sqrt[4]{\eta}$,
\[\mathbb{P}\Big(|S \cap V_{x_1 } \cap \dots \cap V_{x_k }| \geq (1-\varepsilon)|V_{x_0 } \cap V_{x_1 } \cap \dots \cap V_{x_k }|\Big) = 1 -  O(\varepsilon^{-2}\delta^{-4}p^{2(k+1)r}\sqrt[4]{\eta}).\]\end{corollary}

\chg{\begin{proof}
    Let us apply Lemma~\ref{regInt2} with error parameter $p^{-(k+1)r} \varepsilon$ to the set $V_{x_0} \setminus S$, where the assumption that $\varepsilon \geq 20 \delta^{-2}p^{(k+1)r} \sqrt[4]{\eta}$ ensures that the technical requirement in the statement of that lemma is met. Then
    \[\mathbb{P}\Big(\Big||(V_{x_0} \setminus S) \cap V_{x_1 } \cap \dots \cap V_{x_k}| - \delta^k |V_{x_0} \setminus S|\Big| \leq p^{-(k+1)r} \varepsilon |C_2|\Big) \geq 1 - O(\varepsilon^{-2} \delta^{-4}p^{2(k+1)r}\sqrt[4]{\eta}).\]
    Take any $x_1, \dots, x_k$ that satisfy the event above. It follows that 
    \begin{equation} |S \cap V_{x_1 } \cap \dots \cap V_{x_k}| \geq |V_{x_0} \cap V_{x_1 } \cap \dots \cap V_{x_k}| - p^{-(k+1)r}\varepsilon |C_2|. \label{cor67simplifiedeq} \end{equation}
    Observe that either $V_{x_0} \cap V_{x_1} \cap \dots \cap V_{x_k}$ is an empty set, so the desired conclusion holds trivially, or the set $V_{x_0} \cap V_{x_1} \cap \dots \cap V_{x_k}$ is a non-empty intersection of $k+1$ cosets of some subspaces of codimension at most $r$ inside $C_2$. In the latter case, we have $|V_{x_0} \cap V_{x_1} \cap \dots \cap V_{x_k}| \geq p^{-(k+1)r} |C_2|$, so lower bound in inequality~\eqref{cor67simplifiedeq} is at least $(1 - \varepsilon) |V_{x_0} \cap V_{x_1 } \cap \dots \cap V_{x_k}|$, as required.
\end{proof}}

Recall that $\con$ and $\cons$ indicate positive constants, as explained in the notational part of the preliminary section of the paper (see expression~\eqref{fullConEqn} and discussion surrounding it). 

\begin{corollary}\label{qrPairsCor}Let $P \subset C_1^2$. Then provided that $\eta \leq \cons\, \delta^{32}$, for all but $O(\delta^{-8} \sqrt[8]{\eta} |C_2|)$ of the elements $y \in C_2$, we have\footnote{If we were not to use $\cons$ and $O(\cdot)$ notations for implicit constants this corollary would have taken the following form.\\
\indent There are a sufficiently small positive constant $c_0$ and a sufficiently large positive constant $C_0$ such that  the following holds. Let $P \subset C_1^2$. If $\eta \leq c_0\,\delta^{32}$ then for all but $C_0 \delta^{-8} \sqrt[8]{\eta} |C_2|$ of the elements $y \in C_2$, we have $\Big||P \cap (V_y \times V_y)| - \delta^2 |P|\Big| \leq \sqrt[16]{\eta}|C_1|^2$.\\
Using both $\cons$ and $O(\cdot)$ notations helps us reduce the number of explicit short calculations which typically compare the quasirandomness constant $\eta$ with density $\delta$ or codimension of $\beta$. Although such calculations are simple, they are numerous, so this combination of notations, although non-standard, helps to improve the readability of the proofs.}
\[\Big||P \cap (V_y \times V_y)| - \delta^2 |P|\Big| \leq \sqrt[16]{\eta}|C_1|^2.\] \end{corollary}

\begin{proof}Pick $y \in C_2$ uniformly at random. Let $N$ be the random variable $|P \cap (V_y \times V_y)|$. Then
\[\ex N = \sum_{(x_1, x_2) \in P} \mathbb{P}(x_1, x_2 \in V_{ y}) = \sum_{(x_1, x_2) \in P} \mathbb{P}(y \in V_{x_1 } \cap V_{x_2 }) = \sum_{(x_1, x_2) \in P} \frac{|V_{x_1 } \cap V_{x_2 }|}{|C_2|}.\]
Using Lemma~\ref{regInt1}, we get that
\[\Big|\ex N - \delta^2 |P|\Big| = O(\delta^{-4} \sqrt[4]{\eta}) |C_1|^2.\]
Next, we estimate the second moment of $N$. We have
\[\ex N^2 = \sum_{(x_1, x_2), (x_3, x_4) \in P} \mathbb{P}(x_1, x_2, x_3, x_4 \in V_{ y}) = \sum_{(x_1, x_2), (x_3, x_4) \in P} \frac{|V_{x_1 } \cap V_{x_2 } \cap V_{x_3 } \cap V_{x_4 }|}{|C_2|}.\]
Using Lemma~\ref{regInt1} another time, we get that
\[\Big|\ex N^2 - \delta^4 |P|^2\Big| = O(\delta^{-8} \sqrt[4]{\eta}) |C_1|^4.\]
By Markov's inequality, provided $\sqrt[8]{\eta} \geq \con\, \delta^{-4} \sqrt[4]{\eta}$, we obtain
\begin{align*}\mathbb{P}\Big(\Big|N - \delta^2 |P|\Big| \geq& \sqrt[16]{\eta}|C_1|^2\Big) \leq \mathbb{P}\Big(\Big|N - \ex N\Big| \geq \frac{1}{2}\sqrt[16]{\eta}|C_1|^2\Big) = \mathbb{P}\Big(\Big|N - \ex N\Big|^2 \geq \frac{1}{4}\sqrt[8]{\eta}|C_1|^4\Big)\\
 \leq& \frac{O(\ex N^2 - (\ex N)^2)}{\sqrt[8]{\eta}|C_1|^4} \leq \frac{O(\delta^{-8} \sqrt[4]{\eta}) |C_1|^4}{\sqrt[8]{\eta}|C_1|^4} \leq O(\delta^{-8} \sqrt[8]{\eta}).\qedhere\end{align*}
\end{proof}

We also note that a union of a small number of quasirandom pieces is still quasirandom, with a slightly worse quasirandomness parameter.

\begin{lemma}\label{simultQuasiRand}Let $C_1$ be a coset in $G_1$, let $U, W \leq G_2$ be subspaces such that $U \cap W = \{0\}$ and $\dim W = d$, let $w_0 + W$ be a coset in $G_2$, and let $\beta \colon G_1 \times G_2 \to H$ be a biaffine map. Let $\lambda \in H$. Suppose that 
\[V^{w} = \{(x_1,x_2):\beta(x_1,x_2) = \lambda\} \cap (C_1 \times (w + U))\] 
is non-empty and $\eta$-quasirandom with density $\delta$ for all $w \in w_0 + W$. Then 
\[V^{w_0 + W} =  \{(x_1,x_2):\beta(x_1,x_2) = \lambda\} \cap (C_1 \times (w_0 + W + U))\] 
is $(p^d\eta)$-quasirandom with density $\delta$.\end{lemma}

\begin{proof}By definition of quasirandomness, for each $w \in w_0 + W$, there are at least $(1 - \eta)|C_1|$ elements $x \in C_1$ such that $|V^{w}_x| = \delta |w + U|$. Thus, for at least $(1 - p^d \eta)|C_1|$ elements $x \in C_1$, for each $w \in w_0 + W$, $|V^{w}_x| = \delta |w + U|$. For such an $x$ we thus have $|V^{w_0 + W}_x| = \sum_{w \in w_0 + W} |V^{w}_x| = \sum_{w \in w_0 + W} \delta |w + U| = \delta |w_0 + W + U|$. A similar bound holds for pairs in $C_1$, so the larger variety is also quasirandom.\end{proof}

\subsection{Convolutional extensions of biaffine maps}

For a subset $S \subset G_1 \times G_2$, we write $S_{u \bullet} = \{v \in G_2 \colon (u,v) \in S\}$ and $S_{\bullet v} = \{u \in G_2 \colon (u,v) \in S\}$. We need this additional notation since previously it was understood that variables $x_i$ belonged to $G_1$ and variables $y_i$ belonged to $G_2$, and hence that $S_x$ meant $S_{x\bullet}$ and $S_y$ meant $S_{\bullet y}$. \chg{Frequently, we shall refer to the sets $S_{x\bullet}$ as the \textit{columns} of $S$ and the sets $S_{\bullet y}$ as the rows of $S$.}\\

This subsection is devoted to the proof of a result that essentially says the following. Suppose that $S$ is a subset of a quasirandom variety $B$ with the property that each column in $S$ is nearly the whole column of $B$. We do not make any assumptions on the structure of rows, which is crucial. Let $\phi \colon S \to H$ be a bihomomorphism.\footnote{We actually need a slightly stronger assumption on the order of Freiman homomorphisms in direction $G_1$.} Then the map obtained by convolving in direction $G_2$ (in the sense of the proof of Lemma~\ref{easyExtn}) produces a bihomorphism on a set whose columns are the same as the columns of the variety, but at the cost of removing a very small number of columns in $S$.  

\begin{theorem}\label{biaffineExtnConv}For every $k \in \mathbb{N}$ there is a constant $\varepsilon_0 = \varepsilon_0(k) > 0$ such that the following holds. Let $u_0 + U$ be a coset in $G_1$, let $v_0 + V$ be a coset in $G_2$ and let $\beta \colon G_1 \times G_2 \to \mathbb{F}^r$ be a biaffine map. Let $\lambda \in \mathbb{F}^r$. Suppose that 
\[B = \{(x,y):\beta(x,y) = \lambda\} \cap ((u_0 + U) \times (v_0 + V))\] 
is non-empty and $\eta$-quasirandom with density $\delta$. Let $X \subset u_0 + U$ and $S \subset B$ be such that $|S_{x \bullet}| \geq (1-\varepsilon_0) |B_{x \bullet}|$ for each $x \in X$. Let $\phi \colon S \to H$ and suppose that $\phi$ is a $6 \cdot 2^k$-homomorphism in direction $G_1$ and a 2-homomorphism in direction $G_2$.

Then provided $|U| \geq \eta^{-\con_{k, p}}$, there exist a subset $X_{\text{ext}} \subset X$ such that $|X \setminus X_{\text{ext}}| = O_{k,p}(p^{O_{k,p}(r)}\eta^{\Omega_{k,p}(1)} |U|)$, and a map $\phi^{\text{conv}} \colon (X_{\text{ext}} \times (v_0 + V)) \cap B \to H$, with the following properties.
\begin{itemize}
\item[\textbf{(i)}] $\phi^{\text{conv}}$ is a $2^k$-homomorphism in direction $G_1$.
\item[\textbf{(ii)}] $\phi^{\text{conv}}$ is a $2$-homomorphism in direction $G_2$.
\item[\textbf{(iii)}] For each $(x,y) \in (X_{\text{ext}} \times (v_0 + V)) \cap B$, whenever $z_1, z_2 \in S_{x \bullet}$ are such that $z_1 + z_2 - y \in S_{x \bullet}$, we have
\[\phi^{\text{conv}}(x,y) = \phi(x,z_1) + \phi(x,z_2) - \phi(x, z_1 + z_2 - y).\]
\end{itemize}
\end{theorem}

\noindent\textbf{Remark.} We may take the constant $\varepsilon_0$ to be $2^{-k}/100$.

\begin{proof}Set $\varepsilon_0 = 2^{-k}/100$. All the implicit constants in the asymptotic notation in the proof depend on $k$ and $p$ only, which to make the proof easier to read we do not write explicitly. We may immediately observe that for each $x \in X$, we can extend the map $y \mapsto \phi(x,y)$, defined on $S_{x \bullet}$, to a 2-homomorphism $\phi^{\text{conv}}_x \colon B_{x \bullet} \to H$, using Lemma~\ref{easyExtn}. If we set $\phi^{\text{conv}}(x,y) = \phi^{\text{conv}}_x(y)$, then the map $\phi^{\text{conv}}$ readily satisfies properties \textbf{(ii)} and \textbf{(iii)}. It remains to find $X_{\text{ext}}$ such that the restriction of $\phi^{\text{conv}}$ to $(X_{\text{ext}} \times (v_0 + V)) \cap B$ also has property \textbf{(i)}.\\

Before proceeding let us outline the strategy of the proof. It turns out that in order for a subset $X_{\text{ext}} \subset X$ to be a set of the desired kind it suffices that $X_{\text{ext}}$ has the property that every point $x \in X_{\text{ext}}$ belongs to a dense collection of additive quadruples $x_{[4]}$ in $X_{\text{ext}}$ such that $|\bigcap_{i \in [4]} S_{x_i \bullet}| \geq (1-10\varepsilon_0) |\bigcap_{i \in [4]} B_{x_i \bullet}|$ (we call such additive quadruples \emph{good}). Hence, the first step is to obtain such a set, and the second step is to show that \textbf{(i)} follows. We now prove a claim that illustrates the relevance of good quadruples and then provide a brief sketch of the second step of the proof.\\

\begin{claima*}\label{newobsthm70}
    \chg{Whenever $m\leq 6\cdot 2^k$ and $q_1, \dots, q_{2m} \in X$ satisfy $\sum_{i \in [m]} q_i = \sum_{i \in [m+1,2m]} q_i$ and
\[\bigg|\bigcap_{i \in [2m]} S_{q_i \bullet}\bigg| \geq \frac{3}{4} \bigg|\bigcap_{i \in [2m]} B_{q_i \bullet}\bigg|,\]
then $\sum_{i \in [m]} \phi^{\text{conv}}(q_i, y) = \sum_{i \in [m+1,2m]} \phi^{\text{conv}}(q_i, y)$ for every $y \in \bigcap_{i \in [2m]} B_{q_i \bullet}$.}
\end{claima*}

\vspace{\baselineskip}
\chg{
\noindent\textbf{Remark.} The case of good additive quadruples corresponds to the $m = 2$ case of the claim above. The reason for having longer tuples in the claim is that in the proof, starting from an arbitrary additive tuple in the set $X_{\text{ext}}$, which is not necessarily good, we shall use good additive quadruples to obtain tuples that satisfy the condition in the claim.
}

\begin{proof}
     \chg{Note that $C = \bigcap_{i \in [2m]} B_{q_i \bullet}$ is a non-empty coset, and that for each $z_1 \in C' = \bigcap_{i \in [2m]} S_{q_i \bullet}$, we have that $C' - z_1 + y \subset C$, so $C' - z_1 + y$ and $C'$ intersect, at some $z_2$, say. Thus $z_2, z_1 + z_2 - y \in C'$ as well, so 
\begin{align*}\sum_{i \in [m]} \phi^{\text{conv}}(q_i, y) =& \sum_{i \in [m]} \Big(\phi(q_i, z_1) + \phi(q_i, z_2) - \phi(q_i, z_1 + z_2 - y)\Big)\\
= &\Big(\sum_{i \in [m]}\phi(q_i, z_1)\Big) + \Big(\sum_{i \in [m]}\phi(q_i, z_2)\Big) - \Big(\sum_{i \in [m]}\phi(q_i, z_1 + z_2 - y)\Big)\\
&\hspace{2cm}\text{(since $\phi$ is an $m$-homomorphism in direction $G_1$ on $S$ and}\\
&\hspace{3cm}\text{all points in the arguments of $\phi$ belong to $S$)}\\
= &\Big(\sum_{i \in [m + 1, 2m]}\phi(q_i, z_1)\Big) + \Big(\sum_{i \in [m + 1, 2m]}\phi(q_i, z_2)\Big) - \Big(\sum_{i \in [m+1,2m]}\phi(q_i, z_1 + z_2 - y)\Big)\\
=& \sum_{i \in [m+1, 2m]} \Big(\phi(q_i, z_1) + \phi(q_i, z_2) - \phi(q_i, z_1 + z_2 - y)\Big)\\
=&\sum_{i \in [m+1, 2m]} \phi^{\text{conv}}(q_i, y).\qedhere
\end{align*}}
\end{proof}

\indent\chg{\textbf{Sketch of the second step.} For simplicity we just indicate how to show that $\phi^{\text{conv}}$ is a $2$-homomorphism in direction $G_1$.} Let $x_1, x_2, x_3, x_4 \in X_{\text{ext}}$ be an arbitrary additive quadruple and let $\tilde{y}\in \bigcap_{i \in [4]} B_{x_i \bullet}$. \chg{(In the actual argument, we shall use three different choices of $\tilde{y}$.)} Our goal is to show that
\[\phi^{\text{conv}}(x_1,\tilde{y}) + \phi^{\text{conv}}(x_2,\tilde{y}) - \phi^{\text{conv}}(x_3, \tilde{y}) - \phi^{\text{conv}}(x_4,\tilde{y}) = 0.\]
By using the fact that each $x_i$ belongs to many good additive quadruples, we are able to find many choices of $u_1, v_1, \dots, u_4, v_4 \in X_{\text{ext}}$ such that $u_i + v_i - x_i \in X_{\text{ext}}$ and, \chg{due to Claim A}, $\phi^{\text{conv}}(x_i,\tilde{y}) = \phi^{\text{conv}}(u_i,\tilde{y}) + \phi^{\text{conv}}(v_i,\tilde{y}) - \phi^{\text{conv}}(u_i + v_i - x_i,\tilde{y})$. Hence, if we write $\sigma \colon [4] \to \{-1,1\}$ for the function given by $\sigma(1) = \sigma(2) = 1$ and $\sigma(3) = \sigma(4) = -1$ we get
\[\sum_{i \in [4]} \sigma(i) \phi^{\text{conv}}(x_i,\tilde{y}) = \sum_{i \in [4]} \sigma(i) \Big(\phi^{\text{conv}}(u_i,\tilde{y}) + \phi^{\text{conv}}(v_i,\tilde{y}) - \phi^{\text{conv}}(u_i + v_i - x_i,\tilde{y})\Big).\]
The gain of this step is that we relate the fixed expression $\sum_{i \in [4]} \sigma(i) \phi^{\text{conv}}(x_i,\tilde{y})$ to many expressions where all points in the argument of $\phi^{\text{conv}}$ are now `flexible'. Next, we observe that $u_1 + v_1 - x_1$ and $u_2 + v_2 - x_2$ can be further replaced by new values $u_1 + v_1 - x_1 + t$ and $u_2 + v_2 - x_2 + t$, while still preserving the value of the linear combination of evaluations of $\phi^{\text{conv}}$. This argument, \chg{based on the manoeuvre that exploits the existence of many good additive quadruples}, increases the number of degrees of freedom that are allowed in sequences of points that we evaluate $\phi^{\text{conv}}$ at. We repeat such steps until we are able to say that $\sum_{i \in [4]} \sigma(i) \phi^{\text{conv}}(x_i,\tilde{y})$ equals a related expression that involves points of an additive 12-tuple in the first coordinate, for a dense proportion of additive 12-tuples. Finally, we show that in fact a vast majority of such 12-tuples in the row corresponding to $y$ are respected by $\phi^{\text{conv}}$, which completes the proof.\\
\indent Note however that we in fact have to show that $\phi^{\text{conv}}$ is a $2^k$-homomorphism in direction $G_1$, not merely a 2-homomorphism. The main ideas sketched above still suffice, but the notation in the actual proof becomes much more involved.\\

\noindent\chg{\textbf{Step 1. Defining the set $X_{\text{ext}}$.}} We say that an ordered quadruple $(x_1, x_2, x_3, x_4) \in (u_0 + U)^4$ is an \emph{additive quadruple} if $x_1 + x_2 = x_3 + x_4$. Let $\xi > 0$ be a constant to be chosen later. We iteratively remove elements from $X$. At the $i$\textsuperscript{th} step, we remove an element $r_i$ if it belongs to fewer than $\xi |U|^2$ additive quadruples whose points have not yet been removed from $X$. The procedure terminates if there is no such element. Let $r_1, r_2, \dots, r_m$ be the elements of $X$ that were removed, in this order. In particular, this means that for each $i_1 \in [m]$, there are at most $\xi |U|^2$ choices $(i_2, i_3, i_4)$ such that $i_1 < i_2 < i_3 < i_4 \leq m$ and the four elements $r_{i_{1}}, r_{i_2}, r_{i_3}$ and $r_{i_4}$ in some order form an additive quadruple. On the other hand, there are at least $m^4/ |U| - 24|U|^2$ additive quadruples consisting of distinct elements in $\{r_1, \dots, r_m\}$, from which we deduce that
\[24 \xi m |U|^2 \geq m^4/ |U| - 24|U|^2.\]
Thus, $m \leq \max\{4 \xi^{1/3} |U|, \xi^{-1}\}$. Let $X'$ be the modified set. Then each $x \in X'$ belongs to at least $\xi |U|^2$ additive quadruples in $X'$.\\

Say that an additive quadruple $x_{[4]}$ is \emph{good} if $|\bigcap_{i \in [4]} S_{x_i \bullet}| \geq (1-10\varepsilon_0) |\bigcap_{i \in [4]} B_{x_i \bullet}|$. Otherwise say that $x_{[4]}$ is \emph{bad}.

\begin{claimb*}\label{badquadsbound}The number of bad additive quadruples $x_{[4]}$ with elements in $X'$  is $O(p^{O(r)}\sqrt[4]{\eta} |U|^3)$.\end{claimb*}

\begin{proof}Observe that a quadruple $x_{[4]}$ of elements in $X'$ is automatically good when 
\[\bigg|S_{x_i \bullet} \cap \Big(\bigcap_{j \in [4] \setminus \{i\}} B_{x_j \bullet} \Big)\bigg| \geq (1 - 2\varepsilon_0) \bigg|\bigcap_{i \in [4]} B_{x_i \bullet}\bigg|\] 
for each $i \in [4]$. Using this observation, the claim follows from Corollary~\ref{oneDenseColumnCor} and the fact that $B_{x_1 \bullet} \cap B_{x_2 \bullet} \cap B_{x_3 \bullet} \cap B_{x_4 \bullet} = B_{x_1 \bullet} \cap B_{x_2 \bullet} \cap B_{x_3 \bullet} = \dots = B_{x_2 \bullet} \cap B_{x_3 \bullet} \cap B_{x_4 \bullet}$, provided $\eta \leq \cons p^{- \con\,r}$.\end{proof}

Let $X''$ be the set of all $x \in X'$ that belong to at most $\sqrt[8]{\eta}|U|^2$ bad additive quadruples whose elements lie in $X'$. Then $|X' \setminus X''| = O(p^{O(r)}\sqrt[8]{\eta} |U|)$. Thus,
\begin{equation}|X''| \geq |X| - m - O(p^{O(r)}\sqrt[8]{\eta} |U|) \geq |X| - \xi^{-1} - O((p^{O(r)}\sqrt[8]{\eta} + \xi^{1/3})|U|).\label{xextsizediff}\end{equation}
We may without loss of generality assume that $|X''| \geq \sqrt[8]{\eta} |U|$, otherwise take $X'' = \emptyset$ and the desired claim is immediately satisfied, which we now justify. Throughout the proof we will have requirements that $\eta \leq \cons\, p^{- \con\, r}$ and $\xi \geq \con\,\eta^{\cons} p^{\con\, r}$ hold. We may assume that the first inequality is satisfied as we may simply modify the implicit constants in the bound $|X \setminus X_{\text{ext}}| = O_{k,p}(p^{O_{k,p}(r)}\eta^{\Omega_{k,p}(1)} |U|)$ in the statement of the theorem so that should the first inequality fail, the theorem itself becomes vacuous. Hence, the only actual requirement becomes $\xi \geq \con\,\eta^{\cons} p^{\con\, r}$, so the final choice of $\xi$ will be $C_1 \eta^{c_2}$ for some constants $C_1, c_2 > 0$ (depending possibly only on $p$ and $k$). The term $\xi^{-1}$ in~\eqref{xextsizediff} gives rise to the requirement that $|U| \geq \eta^{-\con}$ in the statement of the theorem. Hence, our choice of $\xi$ and the assumption $|U| \geq \eta^{-\con}$ in the statement allow us to obtain the bound $\xi^{-1} + O((p^{O(r)}\sqrt[8]{\eta} + \xi^{1/3})|U|) \leq O(p^{O(r)}\eta^{\Omega(1)} |U|)$. Thus we assume that $|X''| \geq \sqrt[8]{\eta} |U|$. This discussion also shows that $|X \setminus X''| \leq  O(p^{O(r)}\eta^{\Omega(1)} |U|)$ holds, which is the desired bound.\\
\indent In particular, each $x \in X''$ still belongs to $\Big(\xi - O(p^{O(r)}\sqrt[8]{\eta})\Big) |U|^2$ additive quadruples with elements in $X''$. We write $X_{\text{ext}}$ for $X''$ and prove that $X_{\text{ext}}$ has the claimed properties.\\

\noindent\chg{\textbf{Step 2. Proof that $X_{\text{ext}}$ satisfies property (i).}} Let $a_e, b_e \in X_{\text{ext}}$ for $e \in \{0,1\}^k$ be such that $\sum_{e \in \{0,1\}^k} a_e =\sum_{e \in \{0,1\}^k} b_e$ and let $y_0 \in \bigcap_{e \in \{0,1\}^k} (B_{a_e \bullet} \cap B_{b_e\bullet})$. \chg{Our goal in this step is to show that
\[\sum_{e \in \{0,1\}^k} \phi^{\text{conv}}(a_e, y_0) = \sum_{e \in \{0,1\}^k} \phi^{\text{conv}}(b_e, y_0).\]}
For $x \in X_{\text{ext}}$, let $P_x = \{(x_1, x_2) \in X_{\text{ext}} \colon x_1 + x_2 - x \in X_{\text{ext}}\}$, \chg{which is the set of pairs representing the additive quadruples in $X_{\text{ext}}$ that involve $x$}. From Corollary~\ref{qrPairsCor}, provided $\eta \leq \cons\, \delta^{32}$ we see that for all but $O(\delta^{-8} \sqrt[8]{\eta} |V|)$ elements $y \in B_{a_e \bullet}$,
\[\Big||P_{a_e} \cap (B_{\bullet y} \times B_{\bullet y})| - \delta^2 |P_{a_e}|\Big| \leq \sqrt[16]{\eta}|U|^2.\]
A similar property holds for each $b_e$.\\
\indent Note that $\delta \geq p^{-r}$ since every non-empty column $B_{x \bullet}$ is an at most $r$-codimensional coset inside $v_0 + V$. By Lemma~\ref{yQR} we know that $B$ is $(32 \delta^{-4} \sqrt[4]{\eta})$-quasirandom with density $\delta$. Applying Lemma~\ref{regInt2} with quasirandomness in direction $G_2$, provided $\eta \leq \cons\, p^{- \con\, r}$, we obtain
\[\Big|\Big\{y \in v_0 + V \colon \Big||B_{\bullet y} \cap X_{\text{ext}}| - \delta |X_{\text{ext}}|\Big| \leq \sqrt[64]{\eta} |U|\Big\}\Big| = (1-O(p^{O(r)}\sqrt[32]{\eta}))|v_0 + V|.\]

Hence, we obtain a set $Y \subset \Big(\bigcap_{e \in \{0,1\}^k} B_{a_e \bullet}\Big) \cap \Big(\bigcap_{e \in \{0,1\}^k} B_{b_e \bullet}\Big)$ such that 
\begin{enumerate}
\item[\textbf{(i)}] $\bigg|\Big(\bigcap_{e \in \{0,1\}^k} B_{a_e \bullet}\Big) \cap \Big(\bigcap_{e \in \{0,1\}^k} B_{b_e \bullet}\Big)\bigg| - |Y| = O(p^{O(r)}\sqrt[32]{\eta})|V|$, \chg{i.e., $Y$ is very dense in the given coset,}
\item[\textbf{(ii)}] $\Big||B_{\bullet y} \cap X_{\text{ext}}| - \delta |X_{\text{ext}}|\Big| \leq \sqrt[64]{\eta} |U|$ for every $y\in Y$, \chg{i.e., the rows indexed by $Y$ contain the expected number of elements of $X_{\text{ext}}$,} and
\item[\textbf{(iii)}] $ \Big||P_{a_e} \cap B_{\bullet y}^2| - \delta^2 |P_{a_e}|\Big| \leq \sqrt[16]{\eta}|U|^2$ and $\Big||P_{b_e} \cap B_{\bullet y}^2| - \delta^2 |P_{b_e}|\Big| \leq \sqrt[16]{\eta}|U|^2$ for every $y\in Y$ and every $e\in\{0,1\}^k$, \chg{i.e., the rows indexed by $Y$ contain the expected number of additive quadruples in $X_{\text{ext}}$ that involve the given points $a_e$ and $b_e$}.
\end{enumerate}
\chg{In other words, $\Big(\bigcap_{e \in \{0,1\}^k} B_{a_e \bullet}\Big) \cap \Big(\bigcap_{e \in \{0,1\}^k} B_{b_e \bullet}\Big)$ is the set of $y$-coordinates of all rows containing the given points $a_e$ and $b_e$, and $Y$ is the subset of of well-behaved rows of that coset.}\\
\indent Note that $\Big(\bigcap_{e \in \{0,1\}^k} B_{a_e \bullet}\Big) \cap \Big(\bigcap_{e \in \{0,1\}^k} B_{b_e \bullet}\Big)$ is a non-empty since it contains $y_0$. As it is a coset of an at most $(2^{k+1} \cdot r)$-codimensional subspace of $V$, we see that $\Big|\Big(\bigcap_{e \in \{0,1\}^k} B_{a_e \bullet}\Big) \cap \Big(\bigcap_{e \in \{0,1\}^k} B_{b_e \bullet}\Big)\Big| \geq p^{-\cons r} |V|$. In particular, for a proportion $1 - O(p^{O(r)}\sqrt[32]{\eta})$ of choices of $y_1, y_2 \in \Big(\bigcap_{e \in \{0,1\}^k} B_{a_e \bullet}\Big) \cap \Big(\bigcap_{e \in \{0,1\}^k} B_{b_e \bullet}\Big)$, we have $y_1, y_2, y_1 + y_2 - y_0 \in Y$. Fix any such choice of $y_1, y_2$ and let $\tilde{y} \in \{y_1, y_2, y_1 + y_2 - y_0\}$. \chg{In the rest of the argument, we show that 
\[\sum_{e \in \{0,1\}^k} \phi^{\text{conv}}(a_e, \tilde{y}) = \sum_{e \in \{0,1\}^k} \phi^{\text{conv}}(b_e, \tilde{y})\]
for all three choices of $\tilde{y}$. Since $\phi^{\text{conv}}$ is a 2-homomorphism in direction $G_2$, the same equality will hold for $y_0$ in place of $\tilde{y}$, completing the proof.}\\
\indent \chg{Let us focus on the row indexed by $\tilde{y}$.} Set $W = B_{\bullet \tilde{y}} - a_e$ (for an arbitrary $e$ -- this is independent of the choice of $e$ as $W$ is simply the subspace of which $B_{\bullet \tilde{y}}$ is a coset), $\tilde{X} = X_{\text{ext}} \cap B_{\bullet \tilde{y}}$ and recursively define sets $\mathcal{S}_i \subset \tilde{X}^{2^{k+2}} \times W^{2^{k+1} - 2^{k + 2 - i}}$, for $i \in [k+2]$ as follows. \chg{The sets $\mathcal{S}_i$ correspond to applications of the manoeuvre involving good additive quadruples, which was explained in the sketch of the argument. The definitions will be somewhat intense notationally, but the underlying principle is the same.} We set 
\begin{align*}\mathcal{S}_1 = \Big\{\Big(u_e, v_e, w_e, z_e \colon e \in \{0,1\}^k\Big) \in \tilde{X}^{2^{k+2}} \colon \Big(\forall e \in \{0,1\}^k\Big)\hspace{2pt}&(a_e, u_e + v_e - a_e, u_e, v_e)\in {\tilde{X}}^4\text{ and is good,}\\
&(b_e, w_e + z_e - b_e, w_e, z_e)\in \tilde{X}^4\text{ and is good}\Big\}.\end{align*}
Thus, the defining condition in the set $\mathcal{S}_1$ is that for each $e \in \{0,1\}^k$ all points in the two quadruples lie in $\tilde{X}$, which is non-trivial information only for $u_e + v_e - a_e$ and $w_e + z_e - b_e$, and that both additive quadruples are good. \chg{In particular, the set $\mathcal{S}_1$ correspond to pass from points $a_e$ to triples $u_e + v_e - a_e, u_e$ and $v_e$, making all the points in the configuration flexible.} Further sets $\mathcal{S}_i$ will have similar conditions in their definition. For $i \in [k]$, once $\mathcal{S}_i$ has been defined, we set
\begin{align*}\mathcal{S}_{i+1} = \Big\{\Big((u_e,& v_e, w_e, z_e \colon e \in \{0,1\}^k), (d^{(1)}_e, f^{(1)}_e \colon e \in \{0,1\}^{k-1}), \dots, (d^{(i)}_e, f^{(i)}_e \colon e \in \{0,1\}^{k - i})\Big)\\
&\in \tilde{X}^{2^{k+2}}\times W^{2^{k+1} - 2^{k + 1 - i}} \colon \Big((u_e, v_e, w_e, z_e \colon e \in \{0,1\}^k), (d^{(1)}_e, f^{(1)}_e \colon e \in \{0,1\}^{k-1}), \dots,\\
&\hspace{5cm} (d^{(i-1)}_e, f^{(i-1)}_e \colon e \in \{0,1\}^{k + 1 - i})\Big) \in \mathcal{S}_i,\\
&\Big(\forall e \in \{0,1\}^{k-i}\Big)\hspace{2pt}(u_{e, 0, 0, \dots, 0} + d^{(1)}_{e, 0, 0, \dots, 0} + \dots + d^{(i-1)}_{e, 0},\,\, u_{e, 1, 0, \dots, 0} + d^{(1)}_{e, 1, 0, \dots, 0} + \dots + d^{(i-1)}_{e, 1},\\
&\hspace{1cm}u_{e, 0, 0, \dots, 0} + d^{(1)}_{e, 0, 0, \dots, 0} + \dots + d^{(i-1)}_{e, 0} + d^{(i)}_e,\,\, u_{e, 1, 0, \dots, 0} + d^{(1)}_{e, 1, 0, \dots, 0} + \dots + d^{(i-1)}_{e, 1} - d^{(i)}_e)\in {\tilde{X}}^4\\
&\hspace{0.5cm}(w_{e, 0, 0, \dots, 0} + f^{(1)}_{e, 0, 0, \dots, 0} + \dots + f^{(i-1)}_{e, 0},\,\, w_{e, 1, 0, \dots, 0} + f^{(1)}_{e, 1, 0, \dots, 0} + \dots + f^{(i-1)}_{e, 1},\\
&\hspace{1cm}w_{e, 0, 0, \dots, 0} + f^{(1)}_{e, 0, 0, \dots, 0} + \dots + f^{(i-1)}_{e, 0} + f^{(i)}_e,\,\, w_{e, 1, 0, \dots, 0} + f^{(1)}_{e, 1, 0, \dots, 0} + \dots + f^{(i-1)}_{e, 1} - f^{(i)}_e)\in {\tilde{X}}^4\\
&\text{and are good}\Big\}.\end{align*}
\chg{The tuples in the set $\mathcal{S}_{i+1}$ are obtained by applying several manoeuvres to the tuples in the set $\mathcal{S}_i$, leading to an increase in the number of degrees of freedom. This is reflected in the fact that tuples in the set $\mathcal{S}_{i+1}$ arise as concatenations of tuples in the set $\mathcal{S}_{i}$ and additional elements $d^{(i)}_e, f^{(i)}_e$.} Finally, define 
\begin{align*}\mathcal{S}_{k+2} = \Big\{\Big((u_e,& v_e, w_e, z_e \colon e \in \{0,1\}^k), (d^{(1)}_e, f^{(1)}_e \colon e \in \{0,1\}^{k-1}), \dots, (d^{(k)}, f^{(k)}), g\Big)\\
&\in \tilde{X}^{2^{k+2}} \times W^{2^{k+1} - 1} \colon \Big((u_e, v_e, w_e, z_e \colon e \in \{0,1\}^k), (d^{(1)}_e, f^{(1)}_e \colon e \in \{0,1\}^{k-1}), \dots,\\
&(d^{(k)}, f^{(k)})\Big) \in \mathcal{S}_{k+1},\\
&(u_{0,\dots,0} + d^{(1)}_{0,\dots,0} + \dots + d^{(k)},\,\, w_{0,\dots,0} + f^{(1)}_{0,\dots,0} + \dots + f^{(k)} + g, \\
&\hspace{1cm}u_{0,\dots,0} + d^{(1)}_{0,\dots,0} + \dots + d^{(k)} + g,\,\, w_{0,\dots,0} + f^{(1)}_{0,\dots,0} + \dots + f^{(k)})\in \tilde{X}^4\text{ and is good}
\Big\}.
\end{align*}

\chg{As with the previous sets, $\mathcal{S}_{k+2}$ arises from a final manoeuvre applied to the configurations obtained in the last step.}\\
\noindent These sets are dense, as the next claim shows. \chg{The proof is straightforward; it is a matter of a careful unpacking of definitions and it uses the abundance of good additive quadruples.}

\begin{claimc*}For each $i \in [k+2]$, provided $\xi \geq \con\,\eta^{\cons} p^{\con\, r}$, we have
\[|\mathcal{S}_i| = p^{-O(r)} \xi^{O(1)} |W|^{2^{k+2} + 2^{k+1} - 2^{k + 2 - i}}.\]
\end{claimc*}

\begin{proof}We prove the claim by induction on $i$. We begin by proving the base case $i=1$. Recall that $|P_x| \geq \Big(\xi - O(p^{O(r)}\sqrt[8]{\eta})\Big) |U|^2$ for every point $x \in X_{\text{ext}}$. Note also that $x$ belongs to at most $\sqrt[8]{\eta} |U|^2$ bad additive quadruples with elements in $X_{\text{ext}}$. Write $P^{\text{good}}_x \subset P_x$ for the set of all $(x_1, x_2) \in X_{\text{ext}}$ such that $(x, x_1 + x_2 - x, x_1 , x_2)$ is good. Thus  $|P_x \setminus P^{\text{good}}_x| \leq \sqrt[8]{\eta} |U|^2$. Furthermore, if $x, x_1, x_2 \in B_{\bullet \tilde{y}}$ then we also have $x_1 + x_2 - x \in B_{\bullet \tilde{y}}$. Hence, the size $|\mathcal{S}_1|$ is simply the product $\Big(\prod_{e \in \{0,1\}^k} |P^{\text{good}}_{a_e} \cap B_{\bullet \tilde{y}}^2|\Big) \cdot \Big(\prod_{e \in \{0,1\}^k} |P^{\text{good}}_{b_e} \cap B_{\bullet \tilde{y}}^2|\Big)$. Provided $\xi \geq \con\,\eta^{\cons} p^{\con\, r}$, using the property \textbf{(iii)} of the set $Y$ above we get the desired bound.\\

To show the bound in the case $i = 2$, for fixed values $v_e, z_e \in \tilde{X}, s_{e'}, t_{e'} \in 2a_{0,0, \dots, 0} + W$ (again the choice of index in $a_{0,0, \dots, 0}$ is irrelevant, and any $a_e$ would do), for each $e \in \{0,1\}^k$, $e' \in \{0,1\}^{k-1}$, let 
\begin{align*}\mathcal{S}_{\subalign{&1\\&v, z, s, t}} = \Big\{(u_e, w_e\colon e \in \{0,1\}^k) \in \tilde{X}^{2^{k+1}} \colon\Big(u_e&, v_e, w_e, z_e \colon e \in \{0,1\}^k\Big) \in \mathcal{S}_1\\
&\land\,\, (\forall e' \in \{0,1\}^{k-1})\,\, u_{(e', 0)} + u_{(e', 1)} = s_{e'}, w_{(e', 0)} + w_{(e', 1)} = t_{e'}\Big\}.\end{align*}
When $(u_e, w_e\colon e \in \{0,1\}^k), (u'_e, w'_e\colon e \in \{0,1\}^k) \in \mathcal{S}_{\subalign{&1\\&v,z,s,t}}$, for each $e' \in \{0,1\}^{k-1}$ we have $u_{(e', 0)} + u_{(e', 1)} = u'_{(e', 0)} + u'_{(e', 1)}$, so we can write $d^{(1)}_{e'} = u'_{(e', 0)} - u_{(e', 0)} \in W$ to get $u'_{(e', 0)} = u_{(e', 0)} + d^{(1)}_{e'}, u'_{(e', 1)} = u_{(e', 1)} - d^{(1)}_{e'}$. Similarly, we can write $w'_{(e', 0)} = w_{(e', 0)} + f^{(1)}_{e'}, w'_{(e', 1)} = w_{(e', 1)} - f^{(1)}_{e'}$ with $f^{(1)}_{e'} \in W$. Thus, every pair $(u_e, w_e\colon e \in \{0,1\}^k), (u'_e, w'_e\colon e \in \{0,1\}^k) \in \mathcal{S}_{\subalign{&1\\&v,z,s,t}}$ gives rise to a sequence $\bm{s} = \Big((u_e, v_e, w_e, z_e \colon e \in \{0,1\}^k), (d^{(1)}_e, f^{(1)}_e \colon e \in \{0,1\}^{k-1})\Big)$ such that $(u_e, v_e, w_e, z_e \colon e \in \{0,1\}^k) \in \mathcal{S}_1$, and the quadruples $\Big(u_{(e', 0)}, u_{(e', 1)}, u_{(e', 0)} + d^{(1)}_{e'}, u_{(e', 1)} - d^{(1)}_{e'}\Big)$ and $\Big(w_{(e', 0)}, w_{(e', 1)}, w_{(e', 0)} + f^{(1)}_{e'}, w_{(e', 1)} - f^{(1)}_{e'}\Big)$ both belong to $\tilde{X}^4$. In order for the sequence $\bm{s}$  to belong to $\mathcal{S}_2$ the additive quadruples we have just mentioned need to be good as well. But we may use Claim B to simply bound the number of sequences $\bm{s}$ for which one of the given additive quadruples is bad. Thus, by the Cauchy--Schwarz inequality and Claim B, 
\[|\mathcal{S}_2| = \sum_{v, z, s, t}\bigg|\mathcal{S}_{\subalign{&1\\&v, z, s, t}}\bigg|^2 - O(p^{O(r)}\sqrt[4]{\eta}) |W|^{2^{k+2} + 2^{k+1} - 2^k} = \xi^{O(1)} p^{-O(r)}|W|^{2^{k+2} + 2^{k}},\]
using $\xi \geq \con\,\eta^{\cons} p^{\con\, r}$ again.\\

Assume now that the claim holds for some $i \in [2,k]$. That is, assume that 
\[|\mathcal{S}_i| = p^{-O(r)}\xi^{O(1)}|W|^{2^{k+2} + 2^{k+1} - 2^{k + 2 - i}}.\]
Similarly to the above, for fixed values $(u_e, v_e, w_e, z_e \colon e \in \{0,1\}^k) \in \tilde{X}^{2^{k+2}}$, $(d^{(1)}_e, f^{(1)}_e \colon e \in \{0,1\}^{k-1}) \in W^{2^{k}}$, $\dots$, $(d^{(i-2)}_e, f^{(i-2)}_e \colon e \in \{0,1\}^{k+2 - i}) \in W^{2^{k + 3 - i}}$, $s_{e}, t_{e} \in 2a_{0,0, \dots,0} + W$ for $e \in \{0,1\}^{k - i}$, we set
\begin{align*}\mathcal{S}_{\subalign{&i\\&u, v, z, w\\&d^{[i-2]}, f^{[i-2]}, s, t}} = \Big\{(d^{(i-1)}_e,& f^{(i-1)}_e \colon e \in \{0,1\}^{k + 1 - i}) \in W^{2^{k+2-i}} \colon \Big((u_e, v_e, w_e, z_e \colon e \in \{0,1\}^k),\\
&(d^{(1)}_e, f^{(1)}_e \colon e \in \{0,1\}^{k-1}), \dots, (d^{(i-1)}_e, f^{(i-1)}_e \colon e \in \{0,1\}^{k + 1 - i})\Big) \in \mathcal{S}_i\\
&\hspace{1cm}\land \Big(\forall e \in \{0,1\}^{k-i}\Big)\hspace{2pt}(u_{e, 0, 0, \dots, 0} + d^{(1)}_{e, 0, 0, \dots, 0} + \dots + d^{(i-1)}_{e, 0}\\
&\hspace{2cm}+ u_{e, 1, 0, \dots, 0} + d^{(1)}_{e, 1, 0, \dots, 0} + \dots + d^{(i-1)}_{e, 1} = s_e)\\
&\hspace{1cm}\land \Big(\forall e \in \{0,1\}^{k-i}\Big)\hspace{2pt}(w_{e, 0, 0, \dots, 0} + f^{(1)}_{e, 0, 0, \dots, 0} + \dots + f^{(i-1)}_{e, 0}\\
&\hspace{2cm} + w_{e, 1, 0, \dots, 0} + f^{(1)}_{e, 1, 0, \dots, 0} + \dots + f^{(i-1)}_{e, 1} = t_e)\Big\}.\end{align*}

When $(d^{(i-1)}_e, f^{(i-1)}_e \colon e \in \{0,1\}^{k + 1 - i}), ({d'}^{(i-1)}_e, {f'}^{(i-1)}_e \colon e \in \{0,1\}^{k + 1 - i}) \in \mathcal{S}_{\subalign{&i\\&u, v, z, w\\&d^{[i-2]}, f^{[i-2]}, s, t}} $, for each $e \in \{0,1\}^{k - i}$ we have 
\begin{align*}u_{e, 0, 0, \dots, 0} &+ d^{(1)}_{e, 0, 0, \dots, 0} + \dots + d^{(i-2)}_{e, 0, 0} + d^{(i-1)}_{e, 0} + u_{e, 1, 0, \dots, 0} + d^{(1)}_{e, 1, 0, \dots, 0} + \dots + d^{(i-2)}_{e, 1, 0}+ d^{(i-1)}_{e, 1}\\
=&u_{e, 0, 0, \dots, 0} + d^{(1)}_{e, 0, 0, \dots, 0} + \dots + d^{(i-2)}_{e, 0, 0} + {d'}^{(i-1)}_{e, 0} + u_{e, 1, 0, \dots, 0} + d^{(1)}_{e, 1, 0, \dots, 0} + \dots + d^{(i-2)}_{e, 1, 0} + {d'}^{(i-1)}_{e, 1}\end{align*}
so we can find $d^{(i)}_e \in W$ such that ${d'}^{(i-1)}_{e, 0} = d^{(i-1)}_{e, 0} + d^{(i)}_e$ and ${d'}^{(i-1)}_{e, 1} = d^{(i-1)}_{e, 1} - d^{(i)}_e$. Similarly, we can write ${f'}^{(i-1)}_{e, 0} = f^{(i-1)}_{e, 0} + f^{(i)}_e$ and ${f'}^{(i-1)}_{e, 1} = f^{(i-1)}_{e, 1} - f^{(i)}_e$ for some $f^{(i)}_e \in W$. Hence, every pair $\mathcal{S}_{\subalign{&i\\&u, v, z, w\\&d^{[i-2]}, f^{[i-2]}, s, t}}$ gives us a sequence in $\mathcal{S}_{i+1}$ as long as all relevant additive quadruples are good. As in the previous step, we use Claim B to bound the number of sequences of desired shape that contain a bad additive quadruple. Thus, by the Cauchy--Schwarz inequality and Claim B,
\begin{align*}|\mathcal{S}_{i+1}| = &\sum_{\subalign{& u, v, z, w\\&d^{[i-2]}, f^{[i-2]}, s, t}}\bigg|\mathcal{S}_{\subalign{&i\\&u,v, z, w\\&d^{[i-2]}, f^{[i-2]}, s, t}}\bigg|^2 - O(p^{O(r)}\sqrt[4]{\eta}) |W|^{2^{k+2} + 2^{k+1} - 2^{k + 1 - i}}\\
= &p^{-O(r)}\xi^{O(1)}|W|^{2^{k+2} + 2^{k + 1} - 2^{k+1 - i}},\end{align*}
provided $\xi \geq \con\, \eta^{\cons} p^{\con\, r}$.\\

Finally, assume that the claim holds for $i = k + 1$. For fixed values $(u_e, v_e, w_e, z_e \colon e \in \{0,1\}^k) \in \tilde{X}^{2^{k+2}}$, $(d^{(1)}_e, f^{(1)}_e \colon e \in \{0,1\}^{k-1}) \in W^{2^{k}}$, $\dots$, $(d^{(k-1)}_e, f^{(k-1)}_e \colon e \in \{0,1\}) \in W^{4}$, $s \in W$, we set
\begin{align*}\mathcal{S}_{\subalign{&k+1\\&u,v, z, w\\&d^{[k-1]}, f^{[k-1]}, s}} = \Big\{(d^{(k)},& f^{(k)}) \in W^{2} \colon \Big((u_e, v_e, w_e, z_e \colon e \in \{0,1\}^k),(d^{(1)}_e, f^{(1)}_e \colon e \in \{0,1\}^{k-1}), \dots,\\
&\hspace{11cm} (d^{(k)}, f^{(k)})\Big) \in \mathcal{S}_{k+1}\\
&\hspace{1cm}\land \Big(u_{0,\dots,0} + d^{(1)}_{0,\dots,0} + \dots + d^{(k)}\Big) - \Big(w_{0,\dots,0} + f^{(1)}_{0,\dots,0} + \dots + f^{(k)}\Big) = s\Big\}.\end{align*}
When $(d^{(k)}, f^{(k)}), ({d'}^{(k)}, {f'}^{(k)}) \in \mathcal{S}_{\subalign{&k+1\\& u, v, z, w\\&d^{[k-1]}, f^{[k-1]}, s}}$, then, similarly to before, $d^{(k)} - f^{(k)} = {d'}^{(k)} - {f'}^{(k)}$, so there exists $g \in W$ such that ${d'}^{(k)} = d^{(k)} + g$ and ${f'}^{(k)} = f^{(k)} + g$. To finish the proof, apply the Cauchy--Schwarz inequality and Claim B again, obtaining 
\begin{align*}|\mathcal{S}_{k+2}| = &\sum_{\subalign{&u,v, z, w\\&d^{[k-1]}, f^{[k-1]}, s}}\bigg|\mathcal{S}_{\subalign{&k+1\\&u, v, z, w\\&d^{[k-1]}, f^{[k-1]}, s}}\bigg|^2 - O(p^{O(r)}\sqrt[4]{\eta}) |W|^{2^{k+2} + 2^{k+1} - 1}\\ 
= &p^{-O(r)}\xi^{O(1)}|W|^{2^{k+2} + 2^{k + 1} - 1},\end{align*}
provided $\xi \geq \con\,\eta^{\cons} p^{\con\, r}$.\end{proof}

Recall that Claim~B shows that a vast majority of additive quadruples in $X'$ are good. The final ingredient in the proof that we need is a very similar property for more involved sequences of elements $\tilde{X}$. The points of such sequences come from expressions appearing in the definition of the sets $\mathcal{S}_i$. Going back to the sketch of the proof, the sequences in the next claim are the the first coordinates in the final arrangements of points that we obtain in the procedure that increases the number of degrees of freedom.

\begin{claimd*}\label{g11tuplesCl1}The number of 
\[\Big((u_e, v_e, w_e, z_e \colon e \in \{0,1\}^k), (d^{(1)}_e, f^{(1)}_e \colon e \in \{0,1\}^{k-1}), \dots, (d^{(k)}, f^{(k)}), g\Big) \in (\tilde{X})^{2^{k+2}} \times W^{2^{k+1} - 1}\]
such that
\begin{align}\Big|\Big(\bigcap_{e \in \{0,1\}^k} &S_{u_e + v_e - a_e\bullet}\Big) \cap \Big(\bigcap_{e \in \{0,1\}^k} S_{v_e \bullet}\Big) \cap \Big(\bigcap_{e \in \{0,1\}^k} S_{w_e + z_e - b_e \bullet}\Big) \cap \Big(\bigcap_{e \in \{0,1\}^k} S_{z_e \bullet}\Big)\nonumber\\
&\hspace{1cm}\cap \bigg(\bigcap_{i \in [k]}\Big(\bigcap_{e \in \{0,1\}^{k-i}} S_{(u_{e, 1, 0, \dots, 0} + d^{(1)}_{e, 1, 0, \dots, 0} + \dots + d^{(i-1)}_{e, 1} - d^{(i)}_e)\bullet}\Big)\bigg)\nonumber\\
&\hspace{1cm}\cap \bigg(\bigcap_{i \in [k]}\Big(\bigcap_{e \in \{0,1\}^{k-i}} S_{(w_{e, 1, 0, \dots, 0} + f^{(1)}_{e, 1, 0, \dots, 0} + \dots + f^{(i-1)}_{e, 1} - f^{(i)}_e) \bullet}\Big)\bigg)\nonumber\\
&\hspace{1cm} \cap S_{(u_{0,\dots,0} + d^{(1)}_{0,\dots,0} + \dots + d^{(k)} + g) \bullet} \cap S_{(w_{0,\dots,0} + f^{(1)}_{0,\dots,0} + \dots + f^{(k)} + g) \bullet}\Big|\nonumber\\
&< (1-12\cdot 2^{k} \varepsilon_0)\Big|\Big(\bigcap_{e \in \{0,1\}^k} B_{u_e + v_e - a_e \bullet}\Big) \cap \Big(\bigcap_{e \in \{0,1\}^k} B_{v_e \bullet}\Big) \cap \Big(\bigcap_{e \in \{0,1\}^k} B_{w_e + z_e - b_e \bullet}\Big) \cap \Big(\bigcap_{e \in \{0,1\}^k} B_{z_e \bullet}\Big)\nonumber\\
&\hspace{3cm}\cap \bigg(\bigcap_{i \in [k]}\Big(\bigcap_{e \in \{0,1\}^{k-i}} B_{(u_{e, 1, 0, \dots, 0} + d^{(1)}_{e, 1, 0, \dots, 0} + \dots + d^{(i-1)}_{e, 1} - d^{(i)}_e) \bullet}\Big)\bigg)\nonumber\\
&\hspace{3cm}\cap \bigg(\bigcap_{i \in [k]}\Big(\bigcap_{e \in \{0,1\}^{k-i}} B_{(w_{e, 1, 0, \dots, 0} + f^{(1)}_{e, 1, 0, \dots, 0} + \dots + f^{(i-1)}_{e, 1} - f^{(i)}_e) \bullet}\Big)\bigg)\nonumber\\
&\hspace{3cm} \cap B_{(u_{0,\dots,0} + d^{(1)}_{0,\dots,0} + \dots + d^{(k)} + g) \bullet} \cap B_{(w_{0,\dots,0} + f^{(1)}_{0,\dots,0} + \dots + f^{(k)} + g) \bullet}\Big|,\label{g11tuplesCl1Eqn}
\end{align}
and all elements that appear in the subscript of $S$ belong to $\tilde{X}$, is $O(p^{O(r)}\sqrt[4]{\eta}|U|^{2^{k+2} + 2^{k+1} - 1})$.\end{claimd*}

\begin{proof}Note that when the elements $(u_e - a_e, v_e, w_e - b_e, z_e \colon e \in \{0,1\}^k),$ $(d^{(1)}_e, f^{(1)}_e \colon e \in \{0,1\}^{k-1}), \dots,$ $(d^{(k)},$ $f^{(k)}),$ $g$ are linearly independent, then the elements in the subscript of the set $S$ on the left-hand-side, omitting the last one, are linearly independent as well. On the other hand, there is a linear relationship that all elements in the subscripts always satisfy, namely:
\begin{align*}0 = &\Big(\sum_{e \in \{0,1\}^k} u_e + v_e - a_e\Big) - \Big(\sum_{e \in \{0,1\}^k} v_e \Big) - \Big(\sum_{e \in \{0,1\}^k} w_e + z_e - b_e \Big) + \Big(\sum_{e \in \{0,1\}^k} z_e\Big)\\
&\hspace{2cm}-\sum_{i \in [k]}\Big(\sum_{e \in \{0,1\}^{k-i}} u_{e, 1, 0, \dots, 0} + d^{(1)}_{e, 1, 0, \dots, 0} + \dots + d^{(i-1)}_{e, 1} - d^{(i)}_e\Big)\\
&\hspace{2cm}+\sum_{i \in [k]}\Big(\sum_{e \in \{0,1\}^{k-i}} w_{e, 1, 0, \dots, 0} + f^{(1)}_{e, 1, 0, \dots, 0} + \dots + f^{(i-1)}_{e, 1} - f^{(i)}_e\Big)\\
&\hspace{2cm}-\Big(u_{0,\dots,0} + d^{(1)}_{0,\dots,0} + \dots + d^{(k)} + g\Big) +\Big(w_{0,\dots,0} + f^{(1)}_{0,\dots,0} + \dots + f^{(k)} + g\Big).\end{align*}
We used the fact that $\sum_{e \in \{0,1\}^k} a_e = \sum_{e \in \{0,1\}^k} b_e$.\\
\indent The rest of the proof is almost identical to that of Claim B. Like that one, this claim follows from Corollary~\ref{oneDenseColumnCor} (being applied to all variants of~\eqref{g11tuplesCl1Eqn} where all but one occurrence of the set $S$ on the left has been changed to $B$, and the constant $1-12\cdot 2^{k} \varepsilon_0$ in~\eqref{g11tuplesCl1Eqn} is set to $1-2\varepsilon_0$).\end{proof}

We now use the structure obtained to show that $\sum_{e \in \{0,1\}^k}\phi^{\text{conv}}(a_e, \tilde{y}) = \sum_{e \in \{0,1\}^k}\phi^{\text{conv}}(b_e, \tilde{y})$. Once this has been proved, recalling that $\tilde{y}$ was arbitrary among $\{y_1, y_2, y_1 + y_2 - y_0\}$ and using the fact that $\phi^{\text{conv}}$ is a 2-homomorphism in direction $G_2$, we get $\sum_{e \in \{0,1\}^k}\phi^{\text{conv}}(a_e, y_0) = \sum_{e \in \{0,1\}^k}\phi^{\text{conv}}(b_e, y_0)$, as desired.\\

Take $\Big((u_e, v_e, w_e, z_e \colon e \in \{0,1\}^k), (d^{(1)}_e, f^{(1)}_e \colon e \in \{0,1\}^{k-1}), \dots, (d^{(k)}, f^{(k)}), g\Big) \in \mathcal{S}_{k+2}$ such that~\eqref{g11tuplesCl1Eqn} fails (i.e.\ the reverse inequality holds). Applying Claim A and recalling that $\varepsilon_0 = 2^{-k}/100$, we get 
\begin{align}\sum_{e \in \{0,1\}^k} \phi^{\text{conv}}(v_e, \tilde{y}) &- \sum_{e \in \{0,1\}^k} \phi^{\text{conv}}(u_e + v_e - a_e, \tilde{y})\nonumber\\ 
&+ \sum_{i \in [k]} \sum_{e \in \{0,1\}^{k-i}} \phi^{\text{conv}}(u_{e, 1, 0, \dots, 0} + d^{(1)}_{e, 1, 0, \dots, 0} + \dots+ d^{(i-1)}_{e, 1} - d^{(i)}_e, \tilde{y})\nonumber\\
&\hspace{4cm}+ \phi^{\text{conv}}(u_{0,\dots,0} + d^{(1)}_{0,\dots,0} + \dots + d^{(k)} + g, \tilde{y})\nonumber\\
=&\sum_{e \in \{0,1\}^k} \phi^{\text{conv}}(z_e, \tilde{y}) - \sum_{e \in \{0,1\}^k} \phi^{\text{conv}}(w_e + z_e - b_e, \tilde{y})\nonumber\\
&\hspace{3cm} + \sum_{i \in [k]} \sum_{e \in \{0,1\}^{k-i}} \phi^{\text{conv}}(w_{e, 1, 0, \dots, 0} + f^{(1)}_{e, 1, 0, \dots, 0} + \dots + f^{(i-1)}_{e, 1} - f^{(i)}_e, \tilde{y})\nonumber\\
&\hspace{7cm} + \phi^{\text{conv}}(w_{0,\dots,0} + f^{(1)}_{0,\dots,0} + \dots + f^{(k)} + g, \tilde{y})\label{centralGoodEq}
\end{align}
and whenever $x_{[4]}$ is a good additive quadruple in $B_{\bullet \tilde{y}}$, $\phi^{\text{conv}}(x_1, \tilde{y}) + \phi^{\text{conv}}(x_2, \tilde{y}) = \phi^{\text{conv}}(x_3, \tilde{y}) + \phi^{\text{conv}}(x_4, \tilde{y})$, so in particular $\phi^{\text{conv}}(a_e, \tilde{y}) + \phi^{\text{conv}}(u_e + v_e - a_e, \tilde{y}) = \phi^{\text{conv}}(u_e, \tilde{y}) +\phi^{\text{conv}}(v_e, \tilde{y})$ and when $e \in \{0,1\}^{k-i}$, 
\begin{align*}&\phi^{\text{conv}}(u_{e, 0, 0, \dots, 0} + d^{(1)}_{e, 0, 0, \dots, 0} + \dots + d^{(i-1)}_{e, 0}, \tilde{y}) + \phi^{\text{conv}}(u_{e, 1, 0, \dots, 0} + d^{(1)}_{e, 1, 0, \dots, 0} + \dots + d^{(i-1)}_{e, 1}, \tilde{y})\nonumber\\
&\hspace{1cm}=\phi^{\text{conv}}(u_{e, 0, 0, \dots, 0} + d^{(1)}_{e, 0, 0, \dots, 0} + \dots + d^{(i-1)}_{e, 0} + d^{(i)}_e, \tilde{y}) + \phi^{\text{conv}}(u_{e, 1, 0, \dots, 0} + d^{(1)}_{e, 1, 0, \dots, 0} + \dots + d^{(i-1)}_{e, 1} - d^{(i)}_e, \tilde{y})\end{align*}
and so on. (We look at all good quadruples listed in the definitions of the sets $\mathcal{S}_i$.) After straightforward algebraic manipulation, we deduce that 
\begin{align*}\sum_{e \in \{0,1\}^k} \phi^{\text{conv}}(a_e, \tilde{y}) = & \sum_{e \in \{0,1\}^k} \phi^{\text{conv}}(v_e, \tilde{y}) - \sum_{e \in \{0,1\}^k} \phi^{\text{conv}}(u_e + v_e - a_e, \tilde{y})+ \sum_{e \in \{0,1\}^k} \phi^{\text{conv}}(u_e, \tilde{y}) \\
= &  \sum_{e \in \{0,1\}^k} \phi^{\text{conv}}(v_e, \tilde{y}) - \sum_{e \in \{0,1\}^k} \phi^{\text{conv}}(u_e + v_e - a_e, \tilde{y})\\
&\hspace{5cm} + \sum_{e \in \{0,1\}^{k - 1}} \Big(\phi^{\text{conv}}(u_{e, 0}, \tilde{y}) + \phi^{\text{conv}}(u_{e, 1}, \tilde{y})\Big)\\
= &  \sum_{e \in \{0,1\}^k} \phi^{\text{conv}}(v_e, \tilde{y}) - \sum_{e \in \{0,1\}^k} \phi^{\text{conv}}(u_e + v_e - a_e, \tilde{y})\\
 &\hspace{5cm} + \sum_{e \in \{0,1\}^{k - 1}} \Big(\phi^{\text{conv}}(u_{e, 1} - d^{(1)}_e, \tilde{y}) + \phi^{\text{conv}}(u_{e, 0} + d^{(1)}_e, \tilde{y})\Big)\\
= & \dots\\
&\hspace{2cm}\text{(the central equality is precisely~\eqref{centralGoodEq})}\\
= & \dots\\
=&\sum_{e \in \{0,1\}^k} \phi^{\text{conv}}(b_e, \tilde{y}),\end{align*}
where at each step we first use the good quadruples involving the variables $u_{\bcdot}\,\,$ and $d_{\bcdot}\,\,$ listed in the definitions of the sets $\mathcal{S}_1, \dots, \mathcal{S}_{k+2}$, until we get the central equality, and then we use the good quadruples involving the variables $w_{\bcdot}\,\,$ and $f_{\bcdot}\,\,$ listed in the definitions of the sets $\mathcal{S}_{k+2}, \dots, \mathcal{S}_1$ (note the reversed order of the sets) to reach the expression in the final line, thus completing the proof. Finally, recall that we had requirements that $\eta \leq \cons\, p^{- \con\, r}$ and $\xi \geq \con\,\eta^{\cons} p^{\con\, r}$. Recall also the discussion following~\eqref{xextsizediff}, by which we may assume that the first required bound $\eta \leq \cons\, p^{- \con\, r}$ is satisfied. Hence, we may pick $\xi = C_1 \eta^{c_2}$ for some constants $C_1, c_2 > 0$ so that the required bounds on $\xi$ are satisfied.\end{proof}

\subsection{Extending biaffine maps defined on quasirandom varieties}

The main aim of this subsection is to prove that a bihomomorphism defined on almost all of a quasirandom biaffine variety agrees almost everywhere with a biaffine map defined \chg{at all points of} that variety. \chg{This is the content of Proposition~\ref{biaffineAlmostToFullExtn}.}\\

First we prove an auxiliary result that can be informally stated as follows: whenever we have a very dense subset $S$ of a quasirandom variety $B$ which is not an intersection of a product $X \times Y$ with $B$, then we may find a $4 \times 4$ grid $\Gamma \subset B$, whose first coordinates form an additive quadruple and second coordinates also form an additive quadruple, such that $S$ contains exactly 15 out of 16 points of $\Gamma$. \chg{Note that this result is solely about the structure of points inside the variety and that there are no bihomomorphisms involved.}

\begin{proposition}\label{GridWoPt} Let $\varepsilon_0 = \frac{1}{64}$. Let $u_0 + U$ be a coset in $G_1$, let $v_0 + V$ be a coset in $G_2$, and let $\beta \colon G_1 \times G_2 \to \mathbb{F}_p^r$ be a biaffine map. Let $\lambda \in \mathbb{F}_p^r$. Suppose that 
\[B = \{(x,y) \in G_1 \times G_2:\beta(x,y) = \lambda\} \cap ((u_0 + U) \times (v_0 + V))\] 
is $\eta$-quasirandom with density $\delta > 0$. Provided that $\eta \leq \cons\, p^{- \con\,r}$, the following holds.\\
\indent Let $X \subset u_0 + U, Y \subset v_0 + V$ and $S \subset B$ be such that $|S_{x \bullet} \cap Y| \geq (1-\varepsilon_0) |B_{x \bullet}|$ for each $x \in X$, and $|S_{\bullet y} \cap X| \geq (1-\varepsilon_0) |B_{\bullet y}|$ for each $y \in Y$. Then either
\begin{itemize}
\item[\textbf{(i)}] there exist $x_1, x_2, x_3, x_4 \in X$ and $y_1, y_2, y_3, y_4 \in Y$ such that $x_1 + x_2 = x_3 + x_4$, $y_1 + y_2 = y_3 + y_4$ and $(x_i,y_j) \in S$ for $(i,j) \not= (1,1)$, but $(x_1, y_1) \in B \setminus S$, or
\item[\textbf{(ii)}] there exist $X' \subset X$ and $Y' \subset Y$, such that $\frac{|X \setminus X'|}{|u_0 + U|}, \frac{|Y \setminus Y'|}{|v_0 + V|} \leq O(\eta^{1/32})$ and $(X' \times Y') \cap B \subset S$.
\end{itemize} 
\end{proposition}

\chg{\textbf{Remark.} As mentioned before the statement of the proposition, the conclusion \textbf{(i)} results in a $4 \times 4$ grid with precisely 15 points inside $S$, and the conclusion \textbf{(ii)} yields a product structure.}

\begin{proof} Without loss of generality $S \subset X \times Y$, as we may replace $S$ by $S \cap (X \times Y)$ without affecting the assumptions or structures in \textbf{(i)} and \textbf{(ii)}. First remove those $x \in X$ such that $|B_{x \bullet}| \not= \delta|v_0 + V|$ and those $y \in Y$ such that $|B_{\bullet y}| \not= \delta |u_0 + U|$. By $\eta$-quasirandomness, we have removed at most $\eta |u_0 + U|$ elements from $|X|$ and by Lemma~\ref{yQR}, we have removed at most $8\delta^{-2}\sqrt[4]{\eta} |v_0 + V|$ elements from $Y$. Misusing the notation slightly, keep writing $X$, $Y$ and $S$ for the modified sets. Note also that $S_{x\bullet }$ and $S_{\bullet y}$ might have become slightly smaller for $x \in X$, $y \in Y$. If we write $\varepsilon$ for $\varepsilon_0 + 8\delta^{-3}\sqrt[4]{\eta}$, then we still have
\[(\forall x \in X)\,\,|S_{x \bullet} \cap Y| \geq (1-\varepsilon) |B_{x \bullet}|\hspace{1cm}\text{and}\hspace{1cm}(\forall y \in Y)\,\,|S_{\bullet y} \cap X| \geq (1-\varepsilon)|B_{\bullet y}|.\]
The new value $\varepsilon$ lies in $(0, 2\varepsilon_0)$. Let $M = ((X \times Y )\cap B) \setminus S$ \chg{be the set of points in the product that are missing from $S$}. We consider two cases, depending on whether $M$ is large or small. (The bounds on $|M|$ in the two cases overlap when $\eta \leq \cons\, p^{- \con\,r}$.) \chg{In the first case, when the set of missing points $M$ is somewhat large, we shall find structure described in part \textbf{(i)} of the conclusion of the proposition, and in the second case, when $M$ is quite small, we shall find the structure from the other part of the conclusion.}\\[6pt]

\noindent \textbf{Case 1: $|M| > \eta^{1/8}|u_0 + U||v_0+V|$.} 

Assume that there is no structure of the kind described in case \textbf{(i)} of the conclusion of this proposition. Start by counting, for each $y \in Y$, the number of quadruples $(x_1, x_2, x_3, x_4) \in (u_0 + U)^4$ such that $(x_1, y) \in M$, $x_1 + x_2 = x_3 + x_4$ and $(x_2, y), (x_3, y), (x_4, y) \in S$, \chg{i.e., the additive quadruples in the row indexed by $y$ for which the first element is not in $S$, but the remaining three are}. This is 
\begin{align*}\sum_{\substack{x_1 \in M_{\bullet y}\\x_2, x_3 \in S_{\bullet y}}} S_{\bullet y}(x_1 + x_2 - x_3) = &\sum_{\substack{x_1 \in M_{\bullet y}, x_2 \in S_{\bullet y}}} |S_{\bullet y} \cap ((x_1 + x_2) - S_{\bullet y})|\\
\geq &\sum_{\substack{x_1 \in M_{\bullet y}, x_2 \in S_{\bullet y}}} (1-2\varepsilon)|B_{\bullet y}|\\
= & (1-2\varepsilon)|M_{\bullet y}||S_{\bullet y}| |B_{\bullet y}|\\
\geq &(1-3\varepsilon)|M_{\bullet y}| \delta^2|u_0 + U|^2,\end{align*}
where we used the fact that $S_{\bullet y}, ((x_1 + x_2) - S_{\bullet y}) \subset B_{\bullet y}$. Hence, the total number of pairs $(x_{[4]}, y)$, where $x_{[4]}$ is a quadruple satisfying properties above in the row $B_{\bullet y}$, is at least
\begin{equation}\label{yxquadmeqn}(1-3\varepsilon)|M| \delta^2|u_0 + U|^2.\end{equation}
\indent Note that since $\beta$ is biaffine, when $x_1 + x_2 = x_3 + x_4$ we have 
\begin{equation}\label{shortIntnEqn}B_{x_1\bullet } \cap B_{x_2\bullet } \cap B_{x_3\bullet } \cap B_{x_4\bullet } = B_{x_1\bullet } \cap B_{x_2\bullet } \cap B_{x_3\bullet }.\end{equation}
Applying Lemma~\ref{regInt1}, we get that for all but $O(\delta^{-6}\sqrt[4]{\eta} |U|^{3})$ triples $x_{[3]}$ in $(u_0 + U)^3$, the size of $|B_{x_1\bullet } \cap B_{x_2\bullet } \cap B_{x_3\bullet }|$ is precisely $\delta^3 |v_0+V|$. Combining this fact with~\eqref{yxquadmeqn}, the number of quadruples $(x_{[4]}) \in (u_0 + U)^4$ such that $|\bigcap_{i \in [4]} B_{\bullet x_i}| = \delta^{3} |v_0 + V|$, $x_1 + x_2 = x_3 + x_4$ and there exists $y$ such that $(x_1, y) \in M$ and $(x_i, y) \in S$ for $i \in [2,4]$, is at least
\begin{equation}\label{badQuadsEqn}\frac{1}{\delta^3 |v_0+V|}(1-3\varepsilon)\delta^2 |M| |U|^2 - O(\delta^{-O(1)}\sqrt[4]{\eta} |U|^3) \geq \frac{1}{2\delta |v_0+V|} |M| |U|^2 \geq \Omega(\eta^{1/8}|U|^3),\end{equation}
since $\varepsilon < \frac{1}{8}$ and provided $\eta \leq \cons \delta^{\con}$ (recall that since the codimension of the map $\beta$ is $r$, the quasirandomness of the variety $B$ implies the density $\delta$ is at least $p^{-r}$).\\
\indent Pick any such $x_{[4]}$, and let $A = B_{x_1\bullet} \cap B_{x_2\bullet} \cap B_{x_3 \bullet} \cap B_{x_4\bullet}$. Consider the set of points $P = \{x_1, x_2, x_3, x_4\} \times A \subset B$, and the set $S \cap P$. Observe that if $|S \cap P| \geq \frac{15}{16} |P|$, then we obtain the structure described in the case \textbf{(i)}, which contradicts our assumption. Indeed, consider the set $R$ of all $y \in A$ such that $x_1, x_2, x_3, x_4 \in S_{\bullet y}$. By the density assumption on $S \cap P$, we have that $|R| \geq \frac{3}{4}|A|$. Let $y_0$ be such that $(x_1, y_0) \in M$, and $x_2, x_3, x_4 \in S_{\bullet y_0}$. Since $|R| \geq \frac{3}{4} |A|$, we may find $y_1, y_2 \in R$, such that $y_1 + y_2 - y_0 \in R$ as well, by the usual argument: first pick arbitrary $y_1 \in R$, then, since $R \cap (R + y_0 - y_1) \not= \emptyset$, pick arbitrary $y_2 \in R \cap (R + y_0 - y_1)$. This produces the structure in the case \textbf{(i)} of the statement of the proposition, which we are assuming not to exist and is therefore a contradiction.\\

\indent Hence, for each $x_{[4]}$ considered above we actually have that $|S \cap P| < \frac{15}{16} |P|$, where $P$ is the relevant small grid. By the pigeonhole principle, we get that for some $i \in [4]$, $|S_{x_i \bullet} \cap B_{x_1 \bullet} \cap \dots \cap B_{x_4 \bullet}| < \frac{15}{16} |B_{x_1 \bullet} \cap \dots \cap B_{x_4 \bullet}|$. Combining this with the bound~\eqref{badQuadsEqn}, we obtain that for some $i \in [4]$ there are at least $\Omega(\eta^{1/8}|U|^{3})$ quadruples $(x_{[4]}) \in X^4$ such that $x_1 + x_2 = x_3 + x_4$ and
\[|S_{x_i \bullet} \cap B_{x_1 \bullet} \cap \dots \cap B_{x_4 \bullet}| < \frac{15}{16} |B_{x_1 \bullet} \cap \dots \cap B_{x_4 \bullet}|\, = \frac{15}{16}\delta^3 |v_0 + V|.\]
We may assume without loss of generality that $i=1$ (we no longer use the fact that $(x_1, y_0) \in M$ for some $y_0$). Recalling once again~\eqref{shortIntnEqn}, we rephrase this as follows. There are $\Omega(\eta^{1/8}|U|^3)$ triples $(x_{[3]}) \in X^3$ such that
\[|S_{x_1 \bullet } \cap B_{x_2\bullet } \cap B_{x_3 \bullet}| \leq \frac{15}{16} |B_{x_1 \bullet } \cap B_{x_2\bullet } \cap B_{x_3 \bullet}| \,= \frac{15}{16}\delta^3 |v_0 + V|.\]
Average once again, this time over $x_1$, to conclude that there is $x_1 \in X$ such that for $\Omega(\eta^{1/8}|U|^2)$ pairs $(x_2, x_3) \in (u_0 + U)^2$, we have
\[|S_{x_1 \bullet} \cap B_{x_2 \bullet} \cap B_{x_3 \bullet}| \leq \frac{15}{16}|B_{x_1 \bullet} \cap B_{x_2 \bullet} \cap B_{x_3 \bullet}| \,= \frac{15}{16}\delta^3 |v_0 + V|.\]

\indent Recall that $|S_{x_1 \bullet}| \geq \frac{31}{32} |B_{x_1 \bullet}| = \frac{31}{32} \delta |v_0+V|$. We may therefore deduce that
\[\mathbb{P}_{x_2,x_3 \in u_0 + U}\Big(\Big||S_{x_1 \bullet} \cap B_{x_2 \bullet} \cap B_{x_3 \bullet}| - \delta^2 |S_{x_1 \bullet}|\Big| \geq \frac{1}{64}\delta^3|v_0+V|\Big) = \Omega(\eta^{1/8}).\]

We now apply Lemma~\ref{regInt2}, which bounds this probability from above by $O(\sqrt[4]{\eta} \delta^{-O(1)})$, provided $\eta \leq \cons \delta^{\con}$ (so that the technical condition in that lemma is satisfied). This is a contradiction, again provided $\eta \leq \cons \delta^{\con}$, finishing the proof in this case.\\ [6pt]

\noindent\textbf{Case 2: $|M| \leq \eta^{1/16} \delta |u_0 + U| |v_0 + V|.$} 

Let $X'$ be the set of all $x \in X$ with the property that $|M_{x \bullet}| \leq \eta^{1/32} |B_{x \bullet}|$. Then $\eta^{1/32} \delta |v_0+V||X \setminus X'| \leq |M|$, from which it follows that $|X \setminus X'| = O(\eta^{1/32}|u_0 + U|)$. Similarly, let $Y'$ be the set of all $y \in Y$ such that $|M_{\bullet y}| \leq \eta^{1/32}|B_{\bullet y}|$ and therefore $|Y \setminus Y'| = O(\eta^{1/32} |v_0 + V|)$. We claim that $(X' \times Y') \cap B \subset S$.\\
\indent Suppose to the contrary that there exists $(x_1, y_1) \in ((X' \times Y') \cap B) \setminus S$. Take $x_2, x_3$ in $B_{\bullet y_1}$ uniformly and independently at random. Then by the usual counting arguments we have that 
\begin{align}\mathbb{P}_{x_2, x_3 \in B_{\bullet y_1}}\Big(x_2, x_3, x_1 + x_2 - x_3 \in X' \cap S_{\bullet y_1}\Big) &\geq 1- \mathbb{P}_{x_2 \in B_{\bullet y_1}}\Big(x_2 \notin X' \cap S_{\bullet y_1}\Big)- \mathbb{P}_{x_3 \in B_{\bullet y_1}}\Big(x_3 \notin X' \cap S_{\bullet y_1}\Big)\nonumber\\
&\hspace{2cm}  - \mathbb{P}_{x_2, x_3 \in B_{\bullet y_1}}\Big(x_1 + x_2 - x_3 \notin X' \cap S_{\bullet y_1}\Big)\nonumber\\
&\geq 1 - 3\frac{|B_{\bullet y_1}| - |X' \cap S_{\bullet y_1}|}{|B_{\bullet y_1}|}\nonumber\\
&\geq 1 - 6\varepsilon,\label{probEq1}\end{align}
provided $\eta \leq \cons \delta^{\con}$.

By Lemma~\ref{regInt2}, we have that 
\begin{equation}\label{probEq2}\mathbb{P}_{x'_2, x'_3 \in u_0 + U} \Big( \Big| |(B_{x_1 \bullet} \cap Y') \cap B_{x'_2 \bullet} \cap B_{x'_3 \bullet}| - \delta^2 |B_{x_1 \bullet} \cap Y'|\Big| \leq \frac{\delta^3}{8} |v_0 + V|\Big) = 1 - O(\delta^{-O(1)}\sqrt[4]{\eta}).\end{equation}
The technical condition in Lemma~\ref{regInt2} is satisfied provided $\eta \leq \cons \delta^{\con}$.\\
\indent Applying Lemma~\ref{regInt2} another time shows that
\begin{equation}\mathbb{P}_{x'_2, x'_3 \in u_0 + U}\Big(\Big||B_{x_1 \bullet} \cap B_{x'_2 \bullet} \cap B_{x'_3 \bullet}| - \delta^2 |B_{x_1 \bullet}|\Big| \leq \frac{\delta^3}{8} |v_0 + V|\Big) = 1 - O(\delta^{-O(1)}\sqrt[4]{\eta})\label{probEq3prelim}\end{equation}
and the technical condition in Lemma~\ref{regInt2} is again satisfied provided $\eta \leq \cons \delta^{\con}$. We now recall that $x_1 \in X' \subset X$ so in particular $|B_{x_1 \bullet}| = \delta |v_0 + V|$. Furthermore, since $\delta$ is a non-positive power of $p$ and the set $B_{x_1 \bullet} \cap B_{x'_2 \bullet} \cap B_{x'_3 \bullet}$ is a coset of a subspace in $v_0 + V$, its size can either be exactly $\delta^3 |v_0 +V|$ or at most $\frac{1}{p}\delta^3 |v_0 +V|$ or at least $p\delta^3 |v_0 +V|$. Hence, we may strengthen~\eqref{probEq3prelim} to
\begin{equation}\label{probEq3}\mathbb{P}_{x'_2, x'_3 \in u_0 + U}\Big(|B_{x_1 \bullet} \cap B_{x'_2 \bullet} \cap B_{x'_3 \bullet}| = \delta^3 |v_0 + V|\Big) = 1 - O(\delta^{-O(1)}\sqrt[4]{\eta}).\end{equation}

\indent Recall that $y_1 \in Y' \subset Y$ so $|B_{\bullet y_1}| = \delta|u_0 + U|$. By~\eqref{probEq1},~\eqref{probEq2} and~\eqref{probEq3}, we have a choice of $x_2, x_3 \in B_{\bullet y_1}$ such that $x_2, x_3, x_1 + x_2 - x_3 \in X' \cap S_{\bullet y_1}$, $|B_{x_1 \bullet} \cap B_{x_2 \bullet} \cap B_{x_3 \bullet}| = \delta^3|v_0 + V|$ and
\begin{align*}|(B_{x_1 \bullet} \cap Y') \cap B_{x_2 \bullet} \cap B_{x_3 \bullet}| \geq &\delta^2|B_{x_1 \bullet} \cap Y'| - \frac{\delta^3}{8}|v_0 + V|\\
\geq &\delta^2\Big(|B_{x_1 \bullet} \cap Y| - |Y \setminus Y'|\Big) - \frac{\delta^3}{8}|v_0 + V|\\
\geq &\delta^2\Big((1 - \varepsilon)|B_{x_1 \bullet}| - O(\eta^{1/32}|v_0 + V|)\Big) - \frac{\delta^3}{8}|v_0 + V|\\
\geq& \Big(\frac{7}{8} - \varepsilon - O(\delta^{-1}\eta^{1/32})\Big) \delta^3|v_0 + V|\\
\geq & \frac{13}{16}|B_{x_1 \bullet} \cap B_{x_2 \bullet} \cap B_{x_3 \bullet}|,\end{align*}
provided $\eta \leq \cons \delta^{\con}$.\\

Since for each $x \in X'$ we have $|(Y \cap B_{x \bullet}) \setminus S_{x \bullet}| = |M_{x \bullet}| \leq \eta^{1/32}|v_0+V|$, we obtain
\begin{align*}|Y' \cap S_{x_1 \bullet} \cap S_{x_2 \bullet}\cap S_{x_3 \bullet}\cap S_{x_1 + x_2 - x_3 \bullet}| &\geq |Y' \cap B_{x_1 \bullet} \cap B_{x_2 \bullet} \cap B_{x_3 \bullet}\cap B_{x_1 + x_2 - x_3 \bullet}|\\
&\hspace{1cm}- |(Y \cap B_{x_1 \bullet})\setminus S_{x_1 \bullet}| - |(Y \cap B_{x_2 \bullet}) \setminus S_{x_2 \bullet}|\\
&\hspace{1cm} - |(Y \cap B_{x_3 \bullet})\setminus S_{x_3 \bullet}| - |(Y \cap B_{x_1 + x_2 - x_3 \bullet})\setminus S_{x_1 + x_2 - x_3 \bullet}|\\
&\geq \Big(\frac{13}{16} - 4\delta^{-3}\eta^{1/32}\Big)|B_{x_1  \bullet} \cap B_{x_2 \bullet} \cap B_{x_3 \bullet}\cap B_{x_1 + x_2 - x_3 \bullet}|\\
&\geq \frac{3}{4} |B_{x_1 \bullet} \cap B_{x_2 \bullet} \cap B_{x_3 \bullet}\cap B_{x_1 + x_2 - x_3 \bullet}|,\end{align*}
provided $\eta \leq \cons \delta^{\con}$.\\
\indent As usual, by a simple intersection argument we obtain $y_2, y_3$ such that $y_2, y_3, y_1 + y_2 - y_3  \in Y' \cap S_{x_1 \bullet} \cap S_{x_2 \bullet} \cap S_{x_3 \bullet}\cap S_{x_1 + x_2 - x_3 \bullet}$. Then $\{x_1, x_2, x_3, x_1 + x_2 - x_3\} \times \{y_1, y_2, y_3, y_1 + y_2 - y_3\}$ gives the structure described in the case \textbf{(i)} of the conclusion of the theorem, which finishes the proof.\end{proof}

The next proposition is about extending bi-homomorphisms whose domain is of the form $B \cap (X \times Y)$ for very dense subsets $X \subset u_0 + U$ and $Y \subset v_0 + V$. It says that we may extend the domain by essentially filling in the columns $B_{x \bullet}$ for almost every $x \in X$.

\begin{proposition}\label{biaffineExtnStep1}There exists a constant $\varepsilon_0> 0$ such that the following holds. Let $u_0, U, v_0, V, \beta, r, \eta, \delta, B$ be as in Proposition~\ref{GridWoPt} (including the assumption $\eta \leq \cons\, p^{- \con\,r}$). Suppose that $X \subset u_0 + U$ and $Y \subset v_0 + V$ are such that $|X| \geq (1 - \varepsilon_0)|u_0 + U|$ and $|Y| \geq (1 - \varepsilon_0)|v_0 + V|$. Let $S = B \cap (X \times Y)$ and let $\phi \colon S \to H$ be a bi-homomorphism. Then, there is a subset $X' \subset X$ such that $|X \setminus X'| = O(\eta^{\Omega(1)} |u_0 + U|)$, a subset $S' \subset (X' \times (v_0 +V)) \cap B$ such that $|((X' \times (v_0 +V)) \cap B) \setminus S'| = O(\eta^{\Omega(1)} |U||V|)$, and a bi-homomorphism $\phi^{\text{ext}} \colon S' \to H$ such that $\phi = \phi^{\text{ext}}$ on $S \cap S'$.\end{proposition}

\begin{proof}Let $\varepsilon'_0$ be the absolute constant from Lemma~\ref{easyExtnDoubleApprox}, which we assume to be less than 1. We set $\varepsilon_0 = \frac{\varepsilon'_0}{100}$. Let $X' = \{x \in X \colon |B_{x \bullet} \cap Y| \geq (1-2\varepsilon_0)\text{ and }|B_{x \bullet}| = \delta |V|\}$. By $\eta$-quasirandomness and Lemma~\ref{regInt2}, we have that $|X \setminus X'| = O(\delta^{-6}\sqrt[4]{\eta}|U|)$. For each $x \in X'$, we may use Lemma~\ref{easyExtn} to find an affine map $\phi^{\text{ext}}_{x} \colon B_{x \bullet} \to H$ that extends $\phi(x, \bullet)$. Define a map $\phi^{\text{ext}} \colon (X' \times (v_0 +V)) \cap B \to H$ by setting $\phi^{\text{ext}}(x,y) = \phi^{\text{ext}}_x(y)$. We now show that $\phi^{\text{ext}}$ respects the vast majority of horizontal additive quadruples.\\
\indent We say that an additive quadruple $x_{[4]}$ in $X'$ (i.e., a quadruple $x_{[4]}$ such that $x_1 + x_2 = x_3 + x_4$) is \emph{good} if $|Y \cap (\bigcap_{i \in [4]} B_{x_i \bullet})| \geq \frac{3}{4}|\bigcap_{i \in [4]} B_{x_i \bullet}|$. By Lemmas~\ref{regInt1} and~\ref{regInt2}, all but $O(\delta^{-O(1)}\sqrt[4]{\eta}|U|^3)$ of the additive quadruples in $X'$ are good. Note that if $x_{[4]}$ is good and $y \in \bigcap_{i \in [4]} B_{x_i \bullet}$, then there are $y_1, y_2$ such that $y_1, y_2, y_1 + y_2 - y \in Y \cap (\bigcap_{i \in [4]} B_{x_i \bullet})$, which implies (writing $\sigma(1) = \sigma(2) = 1, \sigma(3) = \sigma(4) = -1$) that
\begin{align*}\sum_{i \in [4]} \sigma(i) \phi^{\text{ext}}(x_i, y) = &\sum_{i \in [4]} \sigma(i) \Big(\phi(x_i, y_1) + \phi(x_i, y_2) - \phi(x_i, y_1 + y_2 - y)\Big)\\
=&\Big(\sum_{i \in [4]} \sigma(i) \phi(x_i, y_1)\Big) + \Big(\sum_{i \in [4]} \sigma(i) \phi(x_i, y_2)\Big) - \Big(\sum_{i \in [4]} \sigma(i) \phi(x_i, y_1 + y_2 - y)\Big) = 0.\end{align*} 
Let $\mathcal{S}$ be the set of all $(x_{[4]}, y)$ such that $x_{[4]}$ is an additive quadruple in $X'$, $y \in \bigcap_{i \in [4]} B_{x_i \bullet}$, and $\sum_{i \in [4]} \sigma(i) \phi^{\text{ext}}(x_i, y) \not= 0$. Then $|\mathcal{S}| = O(\delta^{-O(1)}\sqrt[4]{\eta}|U|^3|V|)$.\\
\indent Recall that $\varepsilon'_0$ is the absolute constant from Lemma~\ref{easyExtnDoubleApprox}. By Lemmas~\ref{yQR} and~\ref{regInt2}, provided $\eta \leq \cons\, p^{- \con\,r}$, we have that $|B_{\bullet y} \cap X'| \geq (1 - \varepsilon'_0) |B_{\bullet y}|$ for all $y \in v_0 + V$ except at most $O(\delta^{-O(1)}\sqrt[16]{\eta}|v_0 + V|)$ of them. Now let
\[Z = \{y \in v_0 + V \colon |\mathcal{S}_y| \leq \sqrt[8]{\eta} |U|^3 \text{ and }|B_{\bullet y} \cap X'| \geq (1 - \varepsilon'_0) |B_{\bullet y}|\}.\]
By the work above we have that $|Z| = (1 - O(\delta^{-O(1)}\sqrt[16]{\eta}))|V|$.\\
\indent On the other hand, for each $y \in Z$, $\phi^{\text{ext}}(\cdot, y)$ respects all but at most $\sqrt[8]{\eta} |U|^3$ additive quadruples in $X' \cap B_{\bullet y}$. Use Lemma~\ref{easyExtnDoubleApprox} to find a subset $S'_y \subset X' \cap B_{\bullet y}$ such that $|(X' \cap B_{\bullet y}) \setminus S'_y| \leq \sqrt[32]{\eta}$ and $\phi^{\text{ext}}(\bullet, y)$ is given by an affine map defined on $B_{\bullet y}$ restricted to $S'_y$. If we define $S' = \cup_{y \in Z} S'_y \times \{y\}$, then $\phi^{\text{ext}}$ is a bi-homomorphism on $S'$, and $S'$ has the  properties claimed.\end{proof}

We now prove an easy extension result in the case when the domain of a bi-homomorphism is an extremely dense subset of a variety; that is, when the density is not merely $1-c$ for a small positive constant $c$, but instead when $c$ is small even when compared to $p^{-\con r}$, where $r$ is the codimension of the variety.

\begin{proposition}\label{biaffineExtnStep2}Let $u_0, U, v_0, V, \beta, r, \eta, \delta, B$ be as in Proposition~\ref{GridWoPt} (including the assumption $\eta \leq \cons\, p^{- \con\,r}$). Let $\xi \in (0,\cons\, p^{-\con\, r})$. Let $S \subset B$ be a subset of size at least $(1-\xi) |B|$ and let $\phi \colon S \to H$ be a bi-homomorphism. Then there is a bi-homomorphism $\phi^{\text{ext}} \colon B \to H$ such that $\phi = \phi^{\text{ext}}$ on a set $S' \subset S$ such that $|S'| = (1 - O(\xi p^{3r} + \eta p^{4r})) |B|$.\end{proposition}

\begin{proof}Let $X =\Big\{x \in u_0 + U\colon |S_{x \bullet}| \geq \Big(1-\frac{1}{16}p^{-3r}\Big) |B_{x \bullet}|\Big\}$. By $\eta$-quasirandomness and averaging, we have that $|X| = (1- O(\xi p^{3r}) -O(\eta p^{4r})) |u_0 + U|$. Using Lemma~\ref{easyExtn}, for each $x \in X$ let $\phi^{\text{ext}(1)}(x, \cdot)$ be the affine extension of $\phi(x, \cdot)$. We claim that $\phi^{\text{ext}(1)}$ is a homomorphism in direction $G_1$ as well. To this end, take any $x_1, x_2, x_3, x_4 \in X$ such that $x_1 + x_2 = x_3 + x_4$ and let $y \in \bigcap_{i \in [4]} B_{x_i \bullet}$ be arbitrary. Since $|\bigcap_{i \in [4]} S_{x_i \bullet}| \geq \frac{3}{4} |\bigcap_{i \in [4]} B_{x_i \bullet}|$, there are $y_1, y_2 \in \bigcap_{i \in [4]} S_{x_i \bullet}$ such that additionally $y_1 + y_2 - y \in \bigcap_{i \in [4]} S_{x_i \bullet}$. Thus,
\begin{align*}\phi^{\text{ext}(1)}&(x_1, y) + \phi^{\text{ext}(1)}(x_2, y) - \phi^{\text{ext}(1)}(x_3, y) - \phi^{\text{ext}(1)}(x_4, y)\\ 
&= \Big(\phi(x_1, y_1) + \phi(x_1, y_2) - \phi(x_1, y_1 + y_2 - y)\Big) + \cdots - \Big(\phi(x_4, y_1) + \phi(x_4, y_2) - \phi(x_4, y_1 + y_2 - y)\Big)\\
&= 0,\end{align*}
as claimed. Hence $\phi^{\text{ext}(1)} \colon (X \times (v_0 + V)) \cap B \to H$ is a bi-homomorphism.\\
\indent Next, for each $y \in v_0 + V$, note that $|X \cap B_{\bullet y}| \geq |B_{\bullet y}| - |(u_0 + U)\setminus X| = (1 - O(\xi p^{4r} + \eta p^{5r}))|B_{\bullet y}| \geq \frac{9}{10}|B_{\bullet y}|$ and extend $\phi^{\text{ext}(1)}(\cdot, y)$ from $X \cap B_{\bullet y}$ to $B_{\bullet y}$, again using Lemma~\ref{easyExtn}. As above, we see that this extension is a bi-homomorphism, by using the fact that for any additive quadruple $y_{[4]}$ in $v_0 + V$ we have $|X \cap (\cap_{i \in [4]}B_{\bullet y_i})| \geq (1 - O(\xi p^{6r} + \eta p^{7r})) |\cap_{i \in [4]} B_{\bullet y_i}|$.\end{proof}

\chg{Finally, we arrive at the main result of this subsection. We shall show that a bi-homomorphism defined on a subset $S$ of density $1-c$ inside a variety $B$ can be extended to the whole of $B$ (at the cost of removing few points in $S$) even if $c$ is merely a small positive constant. This is a significantly stronger result than the previous proposition, which requires $c$ to be a very small quantity when compared to $p^{-r}$, where $r$ is the codimension of the variety.}

\begin{proposition}\label{biaffineAlmostToFullExtn}Let $u_0, U, v_0, V, \beta, r, \eta, \delta, B$ be as in Proposition~\ref{GridWoPt} (including the assumption $\eta \leq \cons\, p^{- \con\,r}$). Let $S \subset B$ be a subset of size at least $(1-\varepsilon) |B|$. Let $\phi \colon S \to H$ be a bi-$2$-homomorphism. Then, there is a subset $S' \subset S$ of size $(1-O(\varepsilon) - O(\eta^{\Omega(1)})) |B|$ and a bi-homomorphism $\psi \colon B \to H$ such that $\phi = \psi$ on $S'$.\end{proposition}

\begin{proof}\chg{Before proceeding with the argument, we provide a brief overview. We shall use the structure in case \textbf{(i)} of the conclusion of Proposition~\ref{GridWoPt} to fill in the gaps in the set $S$ until we obtain the product structure described in case \textbf{(ii)}. That is, each time we find a grid $\{x_1, x_2, x_3, x_4\} \times \{y_1, \dots, y_4\}$ for which all points other than $(x_1, y_1)$ belong to $S$, and both $(x_1, \dots, x_4)$ and $(y_1, \dots, y_4)$ are additive quadruples, we may extend $\phi$ to the point $(x_1, y_1)$ as $\phi(x_3, y_1) + \phi(x_4, y_1) - \phi(x_2, y_1)$, and the grid structure will guarantee that the map $\phi$ is still a bihomomorphism and that $\phi(x_3, y_1) + \phi(x_4, y_1) - \phi(x_2, y_1) = \phi(x_1, y_3) + \phi(x_1, y_4) - \phi(x_1, y_2)$.\\
\indent Once we obtain the product structure, we use Proposition~\ref{biaffineExtnStep1} twice, once in direction $G_1$ and once in direction $G_2$, to extend $\phi$ further to almost all of $B$. Finally, we apply Proposition~\ref{biaffineExtnStep2} to finish the proof.}\\
\indent \chg{We now turn to the details. First, we need to find suitable sets $X$ and $Y$ so that we may apply Proposition~\ref{GridWoPt}.} Start by setting $X_0 = \{x \in u_0 + U \colon |B_{x \bullet}| = \delta |V|\}$ and $Y_0 = \{y \in v_0 + V \colon |B_{\bullet y}| = \delta |U|\}$. By $\eta$-quasirandomness and Lemma~\ref{yQR}, we have that $|(u_0 + U) \setminus X_0| \leq \eta |U|$ and $|(v_0 + V) \setminus Y_0| \leq O(\delta^{-4} \sqrt[4]{\eta} |V|)$. We shall modify sets $X_0$ and $Y_0$ so that the conditions in Proposition~\ref{GridWoPt} are satisfied. Let $\varepsilon_0 = \frac{1}{64}$ be as in that proposition.\\ 
\indent Let $E_1 \subset X_0$ be defined by $E_1 = \{x \in X_0 \colon |S_{x \bullet}| \geq (1-\varepsilon_0/3) |B_{x \bullet}|\}$. By simple averaging, we get
\[|X_0 \setminus E_1|\, \cdot \frac{\varepsilon_0}{3} \delta |v_0 + V|\, \leq \sum_{x \in X_0 \setminus E_1} |B_{x \bullet} \setminus S_{x \bullet}|\, \leq |B \setminus S| \leq \varepsilon |B| \,\leq \varepsilon(\delta + \eta) |u_0 + U| |v_0 + V|.\]
Provided $\eta \leq \delta$ we obtain $|X_0 \setminus E_1| \leq 6 \varepsilon_0^{-1} \varepsilon |u_0 + U|$. Similarly, we define $E_2 \subset Y_0$ be defined by $E_2 = \{y \in Y_0 \colon |S_{\bullet y}| \geq (1-\varepsilon_0/3) |B_{\bullet y}|\}$ and an analogous averaging argument gives $|Y_0 \setminus E_2| \leq 6 \varepsilon_0^{-1} \varepsilon |v_0 + V|$. Finally, define $X = \{x \in E_1 \colon |S_{x \bullet} \cap E_2| \geq (1-2\varepsilon_0/3) |B_{x \bullet}|\}$ and $Y = \{y \in E_2 \colon |S_{\bullet y} \cap E_1| \geq (1-2\varepsilon_0/3) |B_{\bullet y}|\}$.

\begin{claima*}We have $|E_1 \setminus X| \leq O(\delta^{-O(1)} \sqrt[16]{\eta}|u_0 + U|)$ and $|E_2 \setminus Y| \leq O(\delta^{-O(1)} \sqrt[16]{\eta}|v_0 + V|)$.\end{claima*}

\begin{proof}Note that $|(v_0 + V) \setminus E_2| \leq \Big(6 \varepsilon_0^{-1} \varepsilon  + \delta^{-4} \sqrt[4]{\eta}\Big) |v_0 + V|$. If $\eta \leq \cons\, \delta^{\con}$ and $\varepsilon \leq \cons$, we have $|(v_0 + V) \setminus E_2| \leq \frac{\varepsilon_0}{6}|v_0 + V|$. By Lemma~\ref{regInt2} we have that for all but at most $O(\delta^{-O(1)} \sqrt[4]{\eta} |u_0 + U|)$ of $x \in u_0 + U$ we have $|B_{x \bullet} \setminus E_2| = |B_{x \bullet} \cap ((v_0 + V) \setminus E_2)|  \leq \frac{\varepsilon_0 \delta}{3} |v_0 + V|$. If such an $x$ additionally lies in $E_1$, then we have
\[|S_{x \bullet} \cap E_2| \geq |S_{x \bullet}| - |B_{x \bullet} \setminus E_2| \geq (1-2\varepsilon_0/3) |B_{x \bullet}|,\]
hence such an $x$ actually lies in $X$. Thus $|E_1 \setminus X| \leq O(\delta^{-O(1)} \sqrt[4]{\eta} |u_0 + U|)$.\\
\indent For the second inequality, we follow the same steps, the only difference being that, by Lemma~\ref{yQR}, the quasirandomness parameter for $B$ in direction $G_2$ can be taken to be $32 \delta^{-4} \sqrt[4]{\eta}$. We have $|(u_0 + U) \setminus E_1| \leq \frac{\varepsilon_0}{6}|u_0 + U|$ and by Lemma~\ref{regInt2} for all but at most $O(\delta^{-O(1)} \sqrt[16]{\eta} |v_0 + V|)$ of $y \in v_0 + V$ we have $|B_{\bullet y} \setminus E_1| = |B_{\bullet y} \cap ((u_0 + U) \setminus E_1)|  \leq \frac{\varepsilon_0 \delta}{3} |u_0 + U|$. Hence, when such a $y$ additionally lies in $E_2$, then we have
\[|S_{\bullet y} \cap E_1| \geq |S_{\bullet y}| - |B_{\bullet y} \setminus E_1| \geq (1-2\varepsilon_0/3) |B_{\bullet y}|,\] 
which implies that $y \in Y$. Thus, $|E_2 \setminus Y| \leq O(\delta^{-O(1)} \sqrt[16]{\eta} |v_0 + V|)$.\end{proof}

Now let $x \in X$. Then $|S_{x \bullet} \cap Y| \geq |S_{x \bullet} \cap E_2| - |E_2 \setminus Y| \geq (1-2\varepsilon_0/3 -  O(\delta^{-O(1)} \sqrt[16]{\eta})) |B_{x \bullet}|$, by the claim above. Provided $\eta \leq \cons\, \delta^{\con}$, we obtain $|S_{x \bullet} \cap Y| \geq (1 - \varepsilon_0)|B_{x \bullet}|$. Similarly, when $y \in Y$, then we have $|S_{\bullet y} \cap X| \geq (1-\varepsilon_0)|B_{\bullet y}|$. Let us also record the bounds on sizes of $X$ and $Y$
\begin{align}\label{xysizebounds}|X| \geq & |u_0 + U| - |(u_0 + U) \setminus X_0| - |X_0 \setminus E_1| - |E_1 \setminus X| \nonumber\\
&\hspace{3cm}\geq (1 - \eta - 6 \varepsilon_0^{-1} \varepsilon - O(\delta^{-O(1)} \sqrt[16]{\eta}))|u_0 + U|,\text{ and}\nonumber\\
|Y| \geq & |v_0 + V| - |(v_0 + V) \setminus Y_0| - |Y_0 \setminus E_2| - |E_2 \setminus Y| \nonumber\\
&\hspace{3cm}\geq (1 - O(\delta^{-O(1)} \sqrt[4]{\eta}) - 6 \varepsilon_0^{-1} \varepsilon - O(\delta^{-O(1)} \sqrt[16]{\eta}))|v_0 + V|.
\end{align}

The sets $X$ and $Y$ now satisfy the conditions in Proposition~\ref{GridWoPt}. We start a procedure where we apply Proposition~\ref{GridWoPt} iteratively until we obtain the structure described in \textbf{(ii)} of that proposition. When we get structure described in \textbf{(i)}, i.e.\ when there are $x_1, x_2, x_3, x_4 \in X$, $y_1, y_2, y_3, y_4 \in Y$ such that $x_1 + x_2 = x_3 + x_4$, $y_1 + y_2 = y_3 + y_4$ and $(x_i,y_j) \in S$ for $(i,j) \not= (1,1)$, but $(x_1, y_1) \in B \setminus S$, we may add $(x_1, y_1)$ to $S$ and extend $\phi$ to $(x_1, y_1)$ by setting $\phi(x_1, y_1) = \phi(x_3, y_1) + \phi(x_4, y_1) - \phi(x_2, y_1)$. Note that
\begin{align*}\phi(x_3, y_1) + \phi(x_4, y_1) - \phi(x_2, y_1) = &\Big(\phi(x_3, y_3) + \phi(x_3, y_4) - \phi(x_3, y_2)\Big) + \Big(\phi(x_4, y_3) + \phi(x_4, y_4) - \phi(x_4, y_2)\Big)\\
&\hspace{5cm} - \Big(\phi(x_2, y_3) + \phi(x_2, y_4) - \phi(x_2, y_2)\Big)\\
= & \phi(x_1, y_3) + \phi(x_1, y_4) - \phi(x_1, y_2).\end{align*}
By Lemma~\ref{easyExtn} applied to $\phi(x, \cdot)$ on the set $S_{x \bullet}$ and to $\phi(\cdot, y)$ on the set $S_{\bullet y}$, we see that the extension is still a bi-homomorphism, so we may proceed.\\ 

Suppose that we have finally obtained the structure described in \textbf{(ii)}. Thus, we may assume that there are $X^{(1)} \subset X, Y^{(1)} \subset Y$, such that $\frac{|X \setminus X^{(1)}|}{|u_0 + U|}, \frac{|Y \setminus Y^{(1)}|}{|v_0 + V|} \leq O(\eta^{1/32})$ and a bi-homomorphism $\phi^{(1)} \colon (X^{(1)} \times Y^{(1)}) \cap B \to H$, such that $\phi^{(1)} = \phi$ on $S^{(1)} = S \cap (X^{(1)} \times Y^{(1)})$. Note that the size of $S^{(1)}$ is
\begin{align}|S^{(1)}| \,=\, &|S \cap (X^{(1)} \times Y^{(1)})| \nonumber\\
\geq\, &|S| - \Big|B \cap \Big(((u_0 + U) \setminus X^{(1)}) \times (v_0 + V)\Big)\Big| - \Big|B \cap \Big((u_0 + U) \times ((v_0 + V)\setminus Y^{(1)}) \Big)\Big|\nonumber\\
&\hspace{4cm}\text{(using: $|B_{x \bullet}| = \delta |v_0 + V|$ and $|B_{\bullet y}| = \delta |u_0 + U|$ for most $x$ and $y$)}\nonumber\\
\geq\,& (1-\varepsilon)|B| \,-\, |(u_0 + U) \setminus X^{(1)}| \cdot \delta |v_0 + V| \,-\, \eta |u_0 + U| |v_0 + V| \nonumber\\
&\hspace{3cm}- |(v_0 + V)\setminus Y^{(1)}| \cdot \delta |v_0 + V| \,-\, O(\delta^{-O(1)} \sqrt[4]{\eta} |u_0+U| |v_0 + V|)\nonumber\\
\geq\,& (1-\varepsilon)|B|\, -\, \Big(12\varepsilon_0^{-1} \varepsilon + O(\delta^{-O(1)} \sqrt[32]{\eta})\Big)\delta |u_0 + U||v_0 + V|\nonumber\\
\geq\,&(1-O(\varepsilon) - O(\eta^{\Omega(1)}))|B|,\label{s1bndeqn}
\end{align}
provided $\eta \leq \cons\,\delta^{\con}$. In the rest of the proof, we show that we may essentially extend $\phi^{(1)}$ to the whole of $B$.\\

Let $\varepsilon_0' > 0$ be the constant that appears in the statement of Proposition~\ref{biaffineExtnStep1}. Note that $|X^{(1)}| \geq (1-\varepsilon'_0)|u_0 + U|$ and $|Y^{(1)}| \geq (1-\varepsilon'_0)|v_0 + V|$, provided $\varepsilon \leq \cons$ and $\eta \leq \cons\,\delta^{\con}$. Apply Proposition~\ref{biaffineExtnStep1} to get a subset $X^{(2)} \subset X^{(1)}$ such that $|X^{(2)} \setminus X^{(1)}| = O(\eta^{\Omega(1)} |U|)$, a subset $S^{(2)} \subset (X^{(2)} \times (v_0 +V)) \cap B$ such that $|((X^{(2)} \times (v_0 +V)) \cap B) \setminus S^{(2)}| = O(\eta^{\Omega(1)} |U||V|)$ and a bi-homomorphism $\phi^{(2)} \colon S^{(2)} \to H$ such that $\phi = \phi^{(2)}$ on $S^{(1)} \cap S^{(2)}$.\\

Now, as in the first step, we find subsets $X^{(3)} \subset X^{(2)}$ and $Y^{(3)} \subset v_0 + V$ that satisfy the conditions of Proposition~\ref{GridWoPt}. Thanks to the additional structure we have, the argument is simpler this time. Recall the facts that $|((X^{(2)} \times (v_0 +V)) \cap B) \setminus S^{(2)}| = O(\eta^{\Omega(1)} |U||V|)$ and that $|B_{x \bullet}| = \delta |v_0 + V|$ for all $x \in X^{(2)}$. Provided that $\eta \leq \cons\,\delta^{\con}$, by averaging we may find a subset $X^{(3)} \subset X^{(2)}$ such that $|X^{(2)} \setminus X^{(3)}| \leq O(\eta^{\Omega(1)} |u_0 + U|)$ and for each $x \in X^{(3)}$ we have $|S^{(2)}_{x \bullet}| \geq (1 - O(\eta^{\Omega(1)})) |B^{(2)}_{x \bullet}|$. Recall from~\eqref{xysizebounds} that, provided that $\eta \leq \cons\,\delta^{\con}$ and $\varepsilon \leq \cons$, the size of $X$ is $|X| \geq (1-\varepsilon_0/8)|u_0 + U|$. Further $|X \setminus X^{(3)}| \leq O(\eta^{\Omega(1)} |u_0 + U|)$, so, again provided that $\eta \leq \cons\,\delta^{\con}$, we have $|X^{(3)}| \geq (1- \varepsilon_0/4)|u_0 + U|$. By Lemmas~\ref{yQR} and~\ref{regInt2} we have that $|X^{(3)} \cap B_{\bullet y}| \geq (1-\varepsilon_0/2) |B_{\bullet y}|$ holds for all but at most $O(\delta^{-O(1)}\eta^{\Omega(1)}|v_0 + V|)$ of $y \in v_0 + V$. On the other hand, by a simple averaging argument we infer that $|(X^{(3)} \cap B_{\bullet y}) \setminus S^{(2)}_{\bullet y}| \leq O(\eta^{\Omega(1)} |u_0 + U|)$ holds for all but at most $O(\eta^{\Omega(1)} |v_0 + V|)$ of $y \in v_0 + V$. In conclusion, we obtain $Y^{(3)} \subset v_0 + V$ of size $|Y^{(3)}| \geq (1 - O(\eta^{\Omega(1)}))|v_0 + V|$ such that for each $y \in Y^{(3)}$ we simultaneously have $|X^{(3)} \cap B_{\bullet y}| \geq (1-\varepsilon_0/2) |B_{\bullet y}|$ and $|(X^{(3)} \cap B_{\bullet y}) \setminus S^{(2)}_{\bullet y}| \leq O(\eta^{\Omega(1)} |u_0 + U|)$. In particular, provided $\eta \leq \cons\,\delta^{\con}$, $|X^{(3)} \cap S^{(2)}_{\bullet y}| \geq (1-\varepsilon_0) |B_{\bullet y}|$ holds when $y \in Y^{(3)}$. Finally, for each $x \in X^{(3)}$ we have $|S^{(2)}_{x \bullet} \cap Y^{(3)}| \geq (1 - O(\delta^{-O(1)}\eta^{\Omega(1)})) |B^{(2)}_{x \bullet}|$.\\ 

Hence, provided that $\eta \leq \cons\,\delta^{\con}$, $S^{(2)}$, $X^{(3)}$ and $Y^{(3)}$ satisfy the conditions of Proposition~\ref{GridWoPt}. Again, iteratively apply Proposition~\ref{GridWoPt}, as before. Each time that we obtain the structure described in part \textbf{(i)} of the conclusion of that proposition, we extend the domain of $\phi^{(2)}$, as we did previously. Eventually, we get the structure in part \textbf{(ii)} and we end up with sets $X^{(4)} \subset X^{(3)}, Y^{(4)} \subset Y^{(3)}$ such that $\frac{|X^{(3)} \setminus X^{(4)}|}{|u_0 + U|}, \frac{|Y^{(3)} \setminus Y^{(4)}|}{|v_0 + V|} = O(\eta^{1/32})$, and a bi-homomorphism $\phi^{(4)} \colon (X^{(4)} \times Y^{(4)}) \cap B \to H$, such that $\phi^{(4)} = \phi$ on $S^{(4)} = S^{(1)} \cap S^{(2)} \cap (X^{(4)} \times Y^{(4)})$.\\

Now reverse the roles of directions $G_1$ and $G_2$. The price we pay is the somewhat weaker quasirandomness parameter of $O(\delta^{-4} \sqrt[4]{\eta})$ instead of $\eta$, by Lemma~\ref{yQR}. Recall that $\varepsilon_0' > 0$ is the constant that appears in the statement of Proposition~\ref{biaffineExtnStep1}. Note that $|X^{(4)}| \geq (1-\varepsilon'_0)|u_0 + U|$ and $|Y^{(4)}| \geq (1-\varepsilon'_0)|v_0 + V|$, provided $\varepsilon \leq \cons$ and $\eta \leq \cons\,\delta^{\con}$. Apply Proposition~\ref{biaffineExtnStep1} to get subset $Y^{(5)} \subset Y^{(4)}$ such that $|Y^{(5)} \setminus Y^{(4)}| = O(\eta^{\Omega(1)} |V|)$, a subset $S^{(5)} \subset ((u_0 +U) \times Y^{(5)}) \cap B$ such that $|(((u_0 + U) \times Y^{(5)}) \cap B) \setminus S^{(5)}| = O(\eta^{\Omega(1)} |U||V|)$ and a bi-homomorphism $\phi^{(5)} \colon S^{(5)} \to H$ such that $\phi = \phi^{(5)}$ on $S^{(4)} \cap S^{(5)}$. But, note that $|(v_0 + V) \setminus Y^{(5)}| = O(\eta^{\Omega(1)}) |V|$, so in fact $|S^{(5)}| = (1 - O(\eta^{\Omega(1)})) |B|$. We may apply Proposition~\ref{biaffineExtnStep2} to find a bi-homomorphism $\phi^{(6)} \colon B \to H$ such that $\phi = \phi^{(6)}$ on a subset $S^{(6)} \subset S^{(4)} \cap S^{(5)}$ such that $|(S^{(4)} \cap S^{(5)}) \setminus S^{(6)}| \leq O(\eta^{\Omega(1)} |U||V|)$. To finish the proof, we estimate the size of $S^{(6)}$
\begin{align*}|S^{(6)}| \geq\,& |S^{(4)} \cap S^{(5)}| - O(\eta^{\Omega(1)} |U||V|) \\
\geq\, &|S^{(1)} \cap S^{(2)} \cap (X^{(4)} \times Y^{(5)})| - O(\eta^{\Omega(1)} |U||V|)\\
\geq\,& |S^{(1)} \cap (X^{(4)} \times Y^{(5)})| - O(\eta^{\Omega(1)} |U||V|)\\
\geq\, & |S^{(1)}| - \delta |V| \cdot |X^{(1)} \setminus X^{(4)}| - \delta |U| \cdot |Y^{(1)} \setminus Y^{(4)}| - O(\eta^{\Omega(1)} |U||V|)\\
&\hspace{4cm}\text{(by~\eqref{s1bndeqn} and $\eta \leq \cons\,\delta^{\con}$)}\\
\geq & (1-O(\varepsilon) - O(\eta^{\Omega(1)}))|B|.\end{align*}

Note that the assumption $\varepsilon \leq \cons$ can be omitted by increasing the implicit constant in the $O(\varepsilon)$ term above.
\end{proof}

\subsection{Extending biaffine maps from subvarieties of codimension 1}

In this subsection we study bi-homomorphisms defined on whole biaffine varieties and consider how to extend them to varieties of lower codimension. Note also that in this case a bi-homomorphism is the same as a biaffine map, as all rows and columns of the domain are in fact cosets of subspaces.

\begin{proposition}\label{1codimExtnBiaffineProp}Let $u_0 + U$ and $v_0 + V$ be cosets in $G_1$ and $G_2$, respectively, and let $\beta \colon G_1 \times G_2 \to \mathbb{F}_p^r$ be a biaffine map. Suppose that there exists $\delta > 0$ such that for each $\lambda \in \mathbb{F}_p^r$ the variety $B^\lambda = ((u_0 + U) \times (v_0 + V)) \cap \{(x,y):\beta(x,y) = \lambda\}$ is either $\eta$-quasirandom with density $\delta$ or empty. Let $\lambda \in \mathbb{F}_p^r$ be such that $B = B^\lambda$ is non-empty. Let
\[B^{\text{ext}} = ((u_0 + U) \times (v_0 + V)) \cap \{(x,y) \in G_1 \times G_2 \colon (\forall i \in [2,r])\,\, \beta_i(x,y) = \lambda_i\}.\]
Let $\phi \colon B \to H$ be a biaffine map. Then, provided $\eta \leq \cons p^{-\con\,r}$, there is a biaffine map $\phi^{\mathrm{ext}} \colon B^{\mathrm{ext}} \to H$ such that $\phi(x,y) = \phi^{\mathrm{ext}}(x,y)$ for all but $O(\eta^{\Omega(1)} |U||V|)$ elements $(x,y) \in B$.\end{proposition}

\indent\chg{Note that the variety $B^{\text{ext}}$ is defined by the biaffine forms $\beta_2, \dots, \beta_r$, so $B = B^{\text{ext}} \cap \{(x,y):\beta_1(x,y) = \lambda_1\}$. Hence, the proposition essentially allows us to reduce the codimension of the domain of $\phi$ by 1.}\\
\indent We note that the assumption that all non-empty layers have the same density $\delta$ is automatically satisfied provided they are all quasirandom; we postpone this argument to the proof of Theorem~\ref{biaffineExtnFullThm}.

\begin{proof}If $B^{\text{ext}} = B$, the claim is trivial, so suppose the opposite. When $\mu_1 \in \mathbb{F}_p$, we also write $B^{\mu_1} = \{(x,y) \in B^{\text{ext}} \colon \beta_1(x,y) = \mu_1\}$. (If we were to use the notation from the statement, this variety would have been denoted $B^{(\mu_1, \lambda_2, \dots, \lambda_r)}$.) The fact that $B^{\text{ext}} \not= B$ implies that $B^{\mu_1} \not= \emptyset$ for some $\mu_1 \not= \lambda_1$. By assumption $B^{\mu_1}$ is $\eta$-quasirandom with density $\delta$. Provided $\eta < 1/2$ we may find $x \in u_0 + U$ such that $|B_{x \bullet}| = |B^{\mu_1}_{x \bullet}| = \delta |v_0 + V|$. In particular, this means that $\beta_1(x,y)$ takes all values in $\mathbb{F}_p$ as $y$ ranges over $B^{\text{ext}}_{x \bullet}$. Hence, the variety $B^{\mu}$ is non-empty for each $\mu \in \mathbb{F}_p$ and is thus $\eta$-quasirandom with density $\delta$.\\
\indent\chg{The map $\phi$ will be extended in three stages. Initially, we choose a suitable point $(a,b) \in B^{\text{ext}} \setminus B$ and put $\phi^{\mathrm{ext}}(a,b) = h_0$, for an arbitrary $h_0 \in H$. In the first stage, we extend $\phi$ to the whole column $B^{\text{ext}}_{a \bullet}$ so that it remains affine in direction $G_2$. Next, given any $z \in (B^{\text{ext}} \setminus B)_{a\bullet}$, we extend $\phi$ to whole rows $B^{\text{ext}}_{\bullet z}$, using the previously defined value $\phi^{\mathrm{ext}}(a,z)$. Finally, in the third stage, for a vast majority of $x \in u_0 + U$ the intersection $(B^{\text{ext}} \setminus B)_{x \bullet} \cap (B^{\text{ext}} \setminus B)_{a\bullet}$ will be non-empty, meaning that the extension $\phi^{\mathrm{ext}}$ is previously defined at some point $(x,z) \in B^{\text{ext}} \setminus B$, which allows us to extend $\phi$ affinely from $B_{x \bullet}$ to $B^{\text{ext}}_{x \bullet}$.\\
\indent To show that this construction works, we need to check carefully that the final map is a well-defined bihomomorphism. The hardest part of the argument is to check various cocycle type equations for $\phi$ inside $B$. The necessary equalities are proved in Claims A and C below. However, the nature of the construction is such that these cocycle-type equations would require several cases in their proofs, so we opt to define $\phi^{\mathrm{ext}}$ as a single expression that works for all three stages simultaneously (see~\eqref{bigpsidefneqn}). This is slightly less intuitive than the above explanation of the construction, but results in a more direct argument.}\\
\indent We now turn to the proof of the proposition. Note that without loss of generality we may assume that $\lambda_1 = 0$ (since shifting $\beta_1$ by a constant keeps it biaffine). Let $(a,b) \in B^{1}$ be a point to be specified later. Let $h_0 \in H$ be arbitrary. For $(x,y) \in B^{\text{ext}}$, let $\mu = \beta_1(x,y)$, and let $z \in v_0 + V$ be such that $(a,z), (x,z) \in B^{1}$. We define
\begin{align}\label{bigpsidefneqn}\psi(x,y; a,b; z,s,u,v,w) = &\phi\Big(x, v + y - w - \mu(z - w)\Big) - \phi(x, v) - (\mu - 1)\phi(x,w)\nonumber\\
&\hspace{1cm}+ \mu\bigg(\phi(u + x - a, z) - \phi(u,z) + \phi(a,s + z - b) - \phi(a,s) + h_0\bigg),\end{align}
where $v, w \in B_{x \bullet}, u \in B_{\bullet z}$, and $s \in B_{a \bullet}$ are arbitrary. We shall pick $(a,b)$ such that for almost every $(x,y) \in B^{\text{ext}}$ there is a value $\phi^{\text{ext}}(x,y)$ for which $\psi(x,y; a,b; z, s, u, v, w) = \phi^{\text{ext}}(x,y)$ for almost all permissible choices of $z, s, u, v, w$. Moreover, viewed as a map on the set of points where it is defined, $\phi^{\text{ext}}$ will be a bi-homomorphism.\\
\indent We remark that all points at which $\phi$ is evaluated in the expression~\eqref{bigpsidefneqn} actually belong to $B$. To see this, first observe that by assumption the points $(x,y),$ $(x,v)$, $(x,w)$, $(x,z),$ $(u, z),$ $(a,z),$ $(a,s),$ $(a,z)$ and $(a,b)$ all belong to $B^{\text{ext}}$. Since $B^{\text{ext}}$ is a biaffine variety, it follows that the arguments of $\phi$ in~\eqref{bigpsidefneqn} are points inside $B^{\text{ext}}$. It remains to see that the $\beta_1$ value at all those points is exactly $0$. The assumptions give that $\beta_1(x,y) = \mu$, $\beta_1(x,v) =$ $\beta_1(x,w) =$ $\beta_1(u,z)=$ $ \beta_1(a,s) = 0$, and $\beta_1(x,z) = $ $\beta_1(a,z) =$ $\beta_1(a,b) = 1$. The fact that $\beta_1 = 0$ at the points of interest follows after a simple calculation.\\
\indent Also, when $(x,y) \in B$, we in fact have $\mu = \lambda_1$ and the expression~\eqref{bigpsidefneqn} becomes simply
\[\psi(x,y; a,b; z,s,u,v,w) = \phi(x, v + y - w) - \phi(x, v) + \phi(x,w) = \phi(x,y).\]

For fixed $x,y,a,b$ the number of choices of the other arguments such that all assumptions above are satisfied is $|B_{a \bullet}||B_{x \bullet}|^2 \sum_{z \in B^{1}_{x \bullet} \cap B^{1}_{a \bullet}} |B_{\bullet z}|$. Also, $\psi$ does not depend on the choice of $s,u,v,w$, so we may write $\psi(x,y; a,b; z)$ instead. We show this only for the parameter $s$: the cases of other parameters are very similar. Suppose that $x,$ $y,$ $a,$ $b,$ $z,$ $s,$ $u,$ $v$ and $w$ are as above and that $s' \in B_{a \bullet}$. Then 
\begin{align*}\psi(x,y; a,b; z,s,u,v,w)& - \psi(x,y; a,b; z,s',u,v,w)\\
 =\,& \mu\Big(\phi(a,s + z - b) - \phi(a,s) - \phi(a,s' + z - b) - \phi(a,s')\Big)\\
 =\,& 0,\end{align*}
since $\phi$ is a bi-homomorphism.\\

The rest of the proof will depend on several claims. The first one shows that for almost every choice of $a,b,x,y,z$ with the properties above we may find plenty of parameters $u, v, w, s$ which together satisfy the rest of the requirements, and the claim also shows that $\psi(x,y; a,b; z)$ is typically independent of $z$.

\begin{claima*}For all but $O(\delta^{-O(1)}\sqrt[16]{\eta} |U|^2 |V|^4)$ choices of $(x,y; a, b, z_1, z_2)$ such that $(x,y) \in B^{\text{ext}}$ and $(a,b), (a,z_1), (x,z_1),(a,z_2), (x,z_2) \in B^{1}$, we have $B_{x \bullet}, B_{\bullet z_1} \cap B_{\bullet z_2}, B_{a \bullet} \not= \emptyset$ and $\psi(x,y; a,b; z_1) = \psi(x,y; a,b; z_2)$.\end{claima*}

\begin{proof}[Proof of Claim A] \chg{The proof reduces to showing that, for suitable $s, u, v, w$, the cocycle-type expression
\begin{align*}
    &\phi\Big(x, v + y - w - \mu (z_1 - w)\Big) - \phi\Big(x, v + y - w - \mu (z_2 - w)\Big)\\
&\hspace{2cm}+\mu\bigg(\phi(u + x - a, z_1) - \phi(u,z_1) - \phi(u + x - a, z_2) \\
&\hspace{3cm}+ \phi(u,z_2) + \phi(a,s + z_1 - b) - \phi(a,s + z_2 - b)\bigg)
\end{align*}
vanishes. To achieve this, we need to find special elements $e, e' \in u_0 + U$ and $f, f' \in V_0 + V$ which will allow us exploit the fact that $\phi$ is a bihomomorphism.}\\
\indent By Lemmas~\ref{regInt1} and~\ref{yQR} we have that $|B_{x \bullet} \cap B_{a \bullet}| = \delta^2 |V|$ and $|B_{\bullet z_1} \cap B_{\bullet z_2} \cap B_{\bullet b}| = \delta^3 |U|$ for all but $O(\delta^{-O(1)}\sqrt[16]{\eta} |U|^2 |V|^4)$ of the sextuples considered. Thus, assume that these equalities hold.\\
\indent Since $|B_{\bullet z_1} \cap B_{\bullet z_2} \cap B_{\bullet b}| = \delta^3 |U|$, by Lemmas~\ref{yQR} and~\ref{regInt2} there are $v, w \in B_{x \bullet},$ $s \in B_{a \bullet}$ such that $Z_1 = B_{\bullet z_1} \cap B_{\bullet z_2} \cap B_{\bullet v} \cap B_{\bullet w} \cap B_{\bullet b} \cap B_{\bullet s}$ is non-empty, provided that $\eta \leq \cons \,\delta^{\con}$. Note that $Z_1$ is a non-empty and at most $(6r)$-codimensional coset inside $u_0 + U$. Hence $|Z_1| \geq p^{-6r} |u_0 + U|$. From this fact and Lemma~\ref{regInt2}, again provided that $\eta \leq \cons \,\delta^{\con}$, we may find elements $u \in B_{\bullet z_1} \cap B_{\bullet z_2},$ $e \in Z_1$ and $e' \in Z_1$ such that $Z_2 = B_{x \bullet} \cap B_{a \bullet} \cap B_{u \bullet} \cap B^{1}_{e \bullet} \cap B_{e' \bullet}$ and $Z_3 = B_{x \bullet} \cap B_{a \bullet} \cap B_{u \bullet} \cap B_{e \bullet} \cap B_{e' \bullet}$ are non-empty.\footnote{\label{foot1}Here we actually first use Lemma~\ref{regInt2} to show that for a vast majority of $u \in B_{\bullet z_1} \cap B_{\bullet z_2}$ and $e' \in Z_1$ we have $\tilde{Z}_2 = B_{x \bullet} \cap B_{a \bullet} \cap B_{u \bullet} \cap B_{e' \bullet} \not= \emptyset$. Then for any such choice of $u$ and $e'$ the set $\tilde{Z}_2$ has size $|\tilde{Z}_2| \geq p^{-4r}|v_0 + V|$. Applying Lemma~\ref{regInt2} twice to the set $\tilde{Z}_2$, once for the variety $B$ and the second time for $B^{1}$ (which we recall is also $\eta$-quasirandom with density $\delta$), we see that for a vast majority of $e \in Z_1$, the sets $Z_2 = \tilde{Z}_2 \cap B^{1}_{e \bullet}$ and $Z_3 = \tilde{Z}_2 \cap B_{e \bullet}$ are also non-empty.} Finally, take $f \in Z_2$ and $f' \in Z_3$.\\
\indent \chg{Having found the special elements $e,e', f, f'$, the rest of the proof is algebraic manipulation, in which we exploit the fact that $\phi$ is a bihomomorphism.} All points below at which $\phi$ is evaluated lie in $B$. \chg{From the definition of $\psi$}, we have
\begin{align*}\psi(x,&y; a,b; z_1) - \psi(x,y; a,b; z_2)\\ 
&= \phi\Big(x, v + y - w - \mu (z_1 - w)\Big) - \phi\Big(x, v + y - w - \mu (z_2 - w)\Big)\\
&\hspace{2cm}+\mu\bigg(\phi(u + x - a, z_1) - \phi(u,z_1) - \phi(u + x - a, z_2) + \phi(u,z_2)\\
&\hspace{4cm} + \phi(a,s + z_1 - b) - \phi(a,s + z_2 - b)\bigg).
\end{align*}

\chg{Let us now use $e$ and $e'$, via the easy identities $x - a = (x + e - e') - (a + e - e')$ and $a = (a + e - e') - e + e'$. We also rely on the fact that $\phi$ is a bihomomorphism a few times. The previous expression equals}

\begin{align*}&\phi(x,v) - \phi\Big(x, v + \mu (z_1 - z_2)\Big)\\
&\hspace{0.5cm}+\mu\bigg(\phi(u + (x + e - e') - (a + e - e'), z_1) - \phi(u,z_1) - \phi(u + (x + e - e') - (a + e - e'), z_2)\\
&\hspace{4cm}+ \phi(u,z_2) + \phi(a + e - e', s + z_1 - b) - \phi(e, s + z_1 - b) + \phi(e', s + z_1 - b) \\
&\hspace{4cm}- \phi(a + e - e', s + z_2 - b) + \phi(e, s + z_2 - b) - \phi(e', s + z_2 - b)\bigg).\end{align*}
\chg{Introducing $e$ and $e'$ in a similar way to the second term in the expression, it becomes}
\begin{align*}&\phi(x,v) - \phi\Big(x + e - e', v + \mu (z_1 - z_2)\Big) + \phi\Big(e, v + \mu (z_1 - z_2)\Big)\\
&\hspace{2cm}- \phi\Big(e', v + \mu (z_1 - z_2)\Big)\\
&\hspace{2cm}+\mu\bigg(\phi(u + (x + e - e') - (a + e - e'), z_1) - \phi(u,z_1)\\
&\hspace{4cm} - \phi(u + (x + e - e') - (a + e - e'), z_2) + \phi(u,z_2)\\
&\hspace{4cm} + \phi(a + e - e', s + z_1 - b) - \phi(e, s + z_1 - b) + \phi(e', s + z_1 - b)\\
 &\hspace{4cm}- \phi(a + e - e', s + z_2 - b) + \phi(e, s + z_2 - b) - \phi(e', s + z_2 - b)\bigg).
\end{align*}

\chg{In the next step of the manipulation, we use the elements $f$ and $f'$. Focus on the term $\phi\Big(x + e - e', v + \mu (z_1 - z_2)\Big)$. Note that, due to our choice of $e$, $e'$, $f$ and $f'$, we have 
\[\beta(x + e - e', v) = \beta(x + e - e', z_1 - f + f') = \beta(x + e - e', z_2 -  f + f') = 0.\]
Using the fact the $\phi$ is a 2-homomorphism in the first variable, we may rewrite the main expression as}

\begin{align*}&\phi(x,v) -\bigg(\phi\Big(x + e - e', v\Big) + \mu \phi\Big(x + e - e', z_1 - f + f'\Big) - \mu\phi\Big(x + e - e', z_2- f+f'\Big)\bigg) \\
&\hspace{2cm} + \phi\Big(e, v + \mu (z_1 - z_2)\Big) - \phi\Big(e', v + \mu (z_1 - z_2)\Big)\\
&\hspace{2cm}+\mu\bigg(\phi(u + (x + e - e') - (a + e - e'), z_1) - \phi(u,z_1)\\
&\hspace{4cm} - \phi(u + (x + e - e') - (a + e - e'), z_2) + \phi(u,z_2)\\
&\hspace{4cm} + \phi(a + e - e', s + z_1 - b) - \phi(e, s + z_1 - b) + \phi(e', s + z_1 - b)\\
 &\hspace{4cm}- \phi(a + e - e', s + z_2 - b) + \phi(e, s + z_2 - b) - \phi(e', s + z_2 - b)\bigg).
\end{align*}

\noindent\chg{We apply the same trick to the other $\phi$ terms whose arguments involve $e - e'$, to obtain}
 
\begin{align*}
&\phi(x,v) + \phi\Big(e, v + \mu (z_1 - z_2)\Big) - \phi\Big(e', v + \mu (z_1 - z_2)\Big)\\
&\hspace{1cm}- \bigg(\phi\Big(x + e - e', v\Big) + \mu \phi\Big(x + e - e', z_1 - f + f'\Big) - \mu\phi\Big(x + e - e', z_2- f+f'\Big)\bigg)\\
&\hspace{2cm}+\mu\bigg(\phi(u + (x + e - e') - (a + e - e'), z_1- f + f') - \phi(u,z_1)\\
&\hspace{3cm} - \phi(u + (x + e - e') - (a + e - e'), z_2-  f + f') + \phi(u,z_2)\\
&\hspace{4cm} + \phi(a + e - e', z_1 - f + f') - \phi(a + e - e', z_2- f + f') \\
&\hspace{4cm} - \phi(e, s + z_1 - b) + \phi(e', s + z_1 - b) + \phi(e, s + z_2 - b) - \phi(e', s + z_2 - b)\bigg).
\end{align*}

\chg{Using several times the fact that $\phi$ is affine in the first variable, we may simplify the $\phi$ terms that have $z_1 -  f + f'$ and $z_2 - f + f'$ as their second argument and the expression above reduces to}

\begin{align*}&\phi(x,v) + \phi\Big(e, v + \mu (z_1 - z_2)\Big) - \phi\Big(e', v + \mu (z_1 - z_2)\Big) - \phi\Big(x + e - e', v\Big)\\
&\hspace{2cm}+\mu\bigg(\phi(u, z_1-  f + f') - \phi(u,z_1) - \phi(u, z_2-  f + f') + \phi(u,z_2)\\
&\hspace{4cm} - \phi(e, s + z_1 - b) + \phi(e', s + z_1 - b) + \phi(e, s + z_2 - b) - \phi(e', s + z_2 - b)\bigg).
\end{align*}
\chg{Applying the fact that $\phi$ is affine in the second variable to terms of the form $\phi(u, \bullet)$, we obtain}
\begin{align*}
&\phi(e',v) - \phi(e,v) + \phi\Big(e, v + \mu (z_1 - z_2)\Big) - \phi\Big(e', v + \mu(z_1 - z_2)\Big) \\
&\hspace{2cm}+\mu\bigg(\phi(e', s + z_1 - b) - \phi(e', s + z_2 - b) - \phi(e, s + z_1 - b) + \phi(e, s + z_2 - b) \bigg).
\end{align*}
\chg{Since $\phi$ is affine in the second variable, focusing on the terms of the form $\phi(e, \bullet)$ and $\phi(e', \bullet)$, we see that the expression vanishes, as required.}
\end{proof}

\chg{Having proved Claim A, for fixed $a$ and $b$ we shall define the value of the extension map at $(x,y)$ to be the most frequent value of $\psi(x,y; a,b; z)$. This requires another averaging argument.}\\ 
\indent For each $(a,b) \in B^{\mu_1}$, let $S_{(a,b)}$ be the set of all $(x,y) \in B^{\text{ext}}$ such that $B^{1}_{x \bullet} \cap B^{1}_{a \bullet} \not= \emptyset$ and for all but at most $\sqrt[32]{\eta} |V|^2$ choices of $z_1, z_2 \in B^{1}_{x \bullet} \cap B^{1}_{a \bullet}$, we have $B_{x \bullet}, B_{\bullet z_1} \cap B_{\bullet z_2}, B_{a \bullet} \not= \emptyset$ and $\psi(x,y; a,b; z_1) = \psi(x,y; a,b; z_2)$. Firstly, since $b \in B^{1}_{a \bullet}$, this is a non-empty coset and thus we have that $|B^{1}_{a \bullet}| \geq p^{-r} |v_0 + V|$. By Lemma~\ref{regInt2}, $B^{1}_{x \bullet} \cap B^{1}_{a \bullet} \not= \emptyset$ holds for all but at most $O(p^{O(r)}\delta^{-O(1)}\sqrt[4]{\eta})|u_0 + U|$ elements $x \in u_0 + U$, provided that $\eta \leq \cons\, p^{-\con r}$. Thus, by Claim A,
\begin{equation}\label{totalSsize}\sum_{(a,b) \in B^{1}} |S_{(a,b)}| = (1 - O(p^{O(r)}\delta^{-O(1)}\sqrt[32]{\eta}))|B^{1}||B^{\text{ext}}|.\end{equation}
For each $(x,y) \in S_{(a,b)}$, define $\phi^{\text{ext}}_{a,b}(x,y)$ to be the most frequent value of $\psi(x,y; a,b; z)$. (Note that there is a unique such value as long as $\eta \leq \cons\,p^{-\con r}$.)\\

The next claim shows that $\phi^{\text{ext}}_{a,b}$ is necessarily a homomorphism in direction $G_2$.

\begin{claimb*}\label{biaffine1codimExtnCl1}For all tuples $(a,b, x, y_1, \dots, y_4)$ such that $(x, y_i) \in S_{(a,b)}$ and $y_1 + y_2 = y_3 + y_4$, we have that
\[\phi^{\mathrm{ext}}_{a,b}(x, y_1) + \phi^{\mathrm{ext}}_{a,b}(x, y_2) = \phi^{\mathrm{ext}}_{a,b}(x, y_3) + \phi^{\mathrm{ext}}_{a,b}(x,y_4).\]
\end{claimb*} 

\begin{proof}[Proof of Claim B]Take $z \in B^{1}_{x \bullet} \cap B^{1}_{a \bullet}$ such that $\phi^{\text{ext}}_{a,b}(x, y_i) = \psi(x,y_i; a,b; z)$ for each $i \in [4]$ and $B_{x \bullet}, B_{\bullet z}, B_{a \bullet} \not= \emptyset$, which we may do provided that $\eta \leq \cons\, p^{-\con r}$. Take any $v, w \in B_{x \bullet}, u \in B_{\bullet z}, s \in B_{a \bullet}$. Let $\tau_i = \beta_1(x, y_i)$. These satisfy $\tau_1 + \tau_2 = \tau_3 + \tau_4$. Let $\sigma \colon [4] \to \{-1,1\}$ be defined by $\sigma(1) = \sigma(2) = 1, \sigma(3) = \sigma(4) = -1$. Then by definition~\eqref{bigpsidefneqn} of $\psi$,
\begin{align*}\phi^{\text{ext}}_{a,b}&(x, y_1) + \phi^{\text{ext}}_{a,b}(x, y_2) - \phi^{\text{ext}}_{a,b}(x, y_3) - \phi^{\text{ext}}_{a,b}(x,y_4)\,\,=\,\,\sum_{i \in [4]} \sigma(i) \phi^{\text{ext}}_{a,b}(x, y_i)\,\,=\,\,\sum_{i \in [4]} \sigma(i)\psi(x,y_i; a,b; z)\\
=&\sum_{i \in [4]} \sigma(i) \phi\Big(x, v + y_i - w - \tau_i (z - w)\Big) - \Big(\sum_{i \in [4]} \sigma(i)\Big)\phi(x, v) - \Big(\sum_{i \in [4]} \sigma(i)(\tau_i - 1)\Big) \phi(x,w)\nonumber\\
&\hspace{1cm}+\Big(\sum_{i \in [4]} \sigma(i) \tau_i\Big)\bigg(\phi(u + x - a, z) - \phi(u,z) + \phi(a,s + z - b) - \phi(a,s) + h_0\bigg)\\
=&0,\end{align*}
since $\phi$ is a bi-homomorphism, $\sum_{i \in [4]} \sigma(i) = 0$, $\sum_{i \in [4]} \sigma(i) y_i = 0$ and $\sum_{i \in [4]} \sigma(i) \tau_i = 0$.\end{proof}

Before showing that $\phi^{\text{ext}}_{a,b}$ is almost a homomorphism in direction $G_1$ (which will be our final claim in this proof), we need one more auxiliary claim that relates values of $\phi$ at different points in $B$. It is a slight generalization of a lemma in~\cite{extnPaper} which was proved for varieties defined by bilinear maps instead of biaffine maps.

\begin{claimc*}Suppose that $x, e, e', g, g',s \in u_0 + U$ and $w,z, f, h, h',t \in v_0 + V$ are such that $\beta_1 = 0$ for all pairs inside $\{x, e, e', g, g',s\} \times \{w,z,f,h,h',t\}$ except the values $\beta_1(x,z) = \beta_1(e,f) = \beta_1(g,h) = \beta_1(g', h') = 1$. Then, for each $\tau \in \mathbb{F}_p$ we have
\begin{align}&\phi(x + \tau(e - e'), \tau (z - w) + h + h' - f\Big) = \phi(x,h + h' - f)\nonumber\\
&\hspace{1cm}+ \tau \phi\Big(x + s - g, z - w + h\Big) -\tau \phi(x,h)\nonumber\\
&\hspace{1cm} + \tau\Big(\phi\Big(e - e' + g', h + h' - f\Big) - \phi\Big(s + g' - g, h + h' - f\Big)\Big)\nonumber\\
&\hspace{1cm}+ \tau \Big( -\phi(s,z) + \phi(s,w)-\phi(s,f) + \phi(s,h') + \phi(g,z) - \phi(g,w)  + \phi(g,f) -\phi(g,h')\Big)\nonumber\\ 
&\hspace{1cm}+\tau^2\Big( \phi(e,z) - \phi(e,w) - \phi(e', z) + \phi(e',w)\Big).\label{biaffineidentitylemma}\end{align}
\end{claimc*} 

Note that if we knew that $\phi$ was a global biaffine map, then the equality~\eqref{biaffineidentitylemma} would trivially hold. However, the fact that $\phi$ is defined only on a subset of $(u_0 + U) \times (v_0 + V)$ makes the result non-trivial.\\

\indent As before, we note that all points in the argument of $\phi$ belong to $B$ (and in the rest of this subsection by writing $\phi(q,r)$ we tacitly imply that $(q,r) \in B$). The easiest way to see this is to observe the following. Suppose that $\alpha$ is an affine map on a coset $C$. Let $\lambda_1, \dots, \lambda_r \in \mathbb{F}_p$ be such that $\sum_{i \in [r]} \lambda_i = 1$ and let $x_1, \dots, x_r \in C$. Then $\sum_{i \in [r]} \lambda_i x_i \in C$ and $\alpha\Big(\sum_{i \in [r]} \lambda_i x_i\Big) = \sum_{i \in [r]} \lambda_i \alpha(x_i)$. Hence, even though we need to pay additional care in algebraic manipulations when working with affine maps, such maps behave linearly when the linear combinations satisfy $\sum_{i \in [r]} \lambda_i = 1$. For example, to see that $(x + s - g, z - w + h) \in B$ holds under the assumptions in Claim C, since $\beta_1$ is biaffine, we have
\begin{align*}&\beta_1(x + s - g, z - w + h) = \beta_1(x, z - w + h) + \beta_1(s, z - w + h) - \beta_1(g, z - w + h) \\
&\hspace{2cm}= \beta_1(x,z) - \beta_1(x,w) + \beta_1(x,h) + \beta_1(s,z) - \beta_1(s,w) + \beta_1(s,h) - \beta_1(g,z) + \beta_1(g,w) - \beta_1(g,h)\\
&\hspace{2cm}= 0.\end{align*}
Similarly, we may similarly relate $\phi$ values of different points using the observation above, since $\phi$ is also biaffine. For example, our first step in the proof of Claim C is
\begin{align*}&\phi(x + (\tau + 1)(e - e'), (\tau + 1)(z - w) + h + h' - f\Big)\\
=\,&\phi\Big(x + \tau(e - e') + s - g, (\tau + 1)(z - w) + h + h' - f\Big) + \phi\Big(e - e' + g', (\tau + 1)(z - w) + h + h' - f\Big)\\
&\hspace{1cm}- \phi\Big(s + g' - g, (\tau + 1)(z - w) + h + h' - f\Big).\end{align*}
We will perform such steps without explicit comments in the the proof below. \chg{However, to make it easier to follow the argument, we introduce the following notation. Each time we perform a step like the above, we write $\phi(x_1 + x_2 - x_3, y) = \phi(\du{x_1}, y) + \phi(\du{x_2}, y) - \phi(\du{x_3}, y)$. The double underline indicates to which argument of $\phi$ we apply the affine assumption. Each step of the argument will have at most one such operation, unless stated otherwise.}

\begin{proof}[Proof of Claim C]We prove the claim by induction on $\tau \in \{0,1,2 \dots, p-1\}$. The base case $\tau = 0$ is trivial. Suppose the claim holds for some $\tau \in \{0,1,\dots, p-2\}$. Then
\begin{align*}&\phi(x + (\tau + 1)(e - e'), (\tau + 1)(z - w) + h + h' - f\Big)\\
=\,&\phi\Big(\du{x + \tau(e - e') + s - g}, (\tau + 1)(z - w) + h + h' - f\Big) + \phi\Big(\du{e - e' + g'}, (\tau + 1)(z - w) + h + h' - f\Big)\\
&\hspace{1cm}- \phi\Big(\du{s + g' - g}, (\tau + 1)(z - w) + h + h' - f\Big)\\
=\,&\phi\Big(x + \tau(e - e') + s - g, \du{\tau (z - w) + 2t - f}\Big) + \phi\Big(x + \tau(e - e') + s - g, \du{z - w + h}\Big)\\
&\hspace{4cm} - \phi\Big(x + \tau(e - e') + s - g, \du{2t - h'}\Big)\\
&\hspace{1cm} +\phi\Big(e - e' + g', (\tau + 1)(z - w) + h + h' - f\Big) - \phi\Big(s + g' - g, (\tau + 1)(z - w) + h + h' - f\Big)\\
=\,&\phi\Big(\du{x + \tau(e - e')}, \tau (z - w) + 2t - f\Big) + \phi\Big(\du{s}, \tau z - \tau w + 2t - f\Big) - \phi\Big(\du{g}, \tau z - \tau w + 2t - f\Big)\\
&\hspace{1cm}+ \phi\Big(x + \tau(e - e') + s - g, z - w + h\Big) - \phi\Big(x + \tau(e - e') + s - g, 2t - h'\Big)\\
&\hspace{1cm} +\phi\Big(e - e' + g', (\tau + 1)(z - w) + h + h' - f\Big) - \phi\Big(s + g' - g, (\tau + 1)(z - w) + h + h' - f\Big)\\
=\,&\phi\Big(x + \tau(e - e'), \du{\tau (z - w) + h + h' - f}\Big) + \phi\Big(x + \tau(e - e'), \du{2t - h}\Big) - \phi\Big(x + \tau(e - e'), \du{h'}\Big)\\
&\hspace{1cm}+ \phi\Big(s, \tau z - \tau w + 2t - f\Big) - \phi\Big(g, \tau z - \tau w + 2t - f\Big)\\
&\hspace{1cm}+ \phi\Big(x + \tau(e - e') + s - g, z - w + h\Big) - \phi\Big(x + \tau(e - e') + s - g, 2t - h'\Big)\\
&\hspace{1cm} +\phi\Big(e - e' + g', (\tau + 1)(z - w) + h + h' - f\Big) - \phi\Big(s + g' - g, (\tau + 1)(z - w) + h + h' - f\Big)\\
=\,&\phi\Big(x + \tau(e - e'), \tau (z - w) + h + h' - f\Big) + \phi\Big(x + \tau(e - e'), 2t - h\Big) - \phi\Big(x + \tau(e - e'), h'\Big)\\
&\hspace{1cm}+ \phi\Big(s, \tau z - \tau w + 2t - f\Big) - \phi\Big(g, \tau z - \tau w + 2t - f\Big)\\
&\hspace{1cm}+ \phi\Big(\du{x + s - g}, z - w + h\Big) +\tau \phi\Big(\du{e}, z - w + h\Big)\\
&\hspace{1cm} -\tau \phi\Big(\du{e'}, z - w + h\Big) - \phi\Big(x + \tau(e - e') + s - g, 2t - h'\Big)\\
&\hspace{1cm} +\phi\Big(e - e' + g', (\tau + 1)(z - w) + h + h' - f\Big) - \phi\Big(s + g' - g, (\tau + 1)(z - w) + h + h' - f\Big)\\
=\,&\phi\Big(x + \tau(e - e'), \tau (z - w) + h + h' - f\Big) + \phi\Big(x + \tau(e - e'), 2t - h\Big) - \phi\Big(x + \tau(e - e'), h'\Big)\\
&\hspace{1cm}+ \phi\Big(s, \tau z - \tau w + 2t - f\Big) - \phi\Big(g, \tau z - \tau w + 2t - f\Big)\\
&\hspace{1cm}+ \phi\Big(x + s - g, z - w + h\Big) +\tau \phi\Big(e, z - w + h\Big)\\
&\hspace{1cm} -\tau \phi\Big(e', z - w + h\Big) - \phi\Big(x + \tau(e - e') + s - g, 2t - h'\Big)\\
&\hspace{1cm} + (\tau + 1)\phi\Big(e - e' + g', \du{z} \Big) - (\tau + 1)\phi\Big(e - e' + g', \du{w}\Big)\\
&\hspace{1cm} + \phi\Big(e - e' + g', \du{h + h' - f}\Big)- \phi\Big(s + g' - g, (\tau + 1)(z - w) + h + h' - f\Big)\\
=\,&\phi\Big(x + \tau(e - e'), \tau (z - w) + h + h' - f\Big) + \phi\Big(x + \tau(e - e'), 2t - h\Big) - \phi\Big(x + \tau(e - e'), h'\Big)\\
&\hspace{1cm}+ \phi\Big(s, \tau z - \tau w + 2t - f\Big) - \phi\Big(g, \tau z - \tau w + 2t - f\Big)\\
&\hspace{1cm}+ \phi\Big(x + s - g, z - w + h\Big) +\tau \phi\Big(e, z - w + h\Big)\\
&\hspace{1cm} -\tau \phi\Big(e', z - w + h\Big) - \phi\Big(x + \tau(e - e') + s - g, 2t - h'\Big)\\
&\hspace{1cm} + (\tau + 1)\phi\Big(e - e' + g', z \Big) - (\tau + 1)\phi\Big(e - e' + g', w\Big)\\
&\hspace{1cm} + \phi\Big(e - e' + g', h + h' - f\Big) - (\tau + 1)\phi\Big(s + g' - g, \du{z}\Big)\\
&\hspace{1cm} + (\tau + 1)\phi\Big(s + g' - g, \du{w}\Big) - \phi\Big(s + g' - g, \du{h + h' - f}\Big).
\end{align*}
\chg{Observe that terms other than $\phi\Big(x + \tau(e - e'), \tau (z - w) + h + h' - f\Big)$, $\phi\Big(x + s - g, z - w + h\Big)$, $\phi\Big(e - e' + g', h + h' - f\Big)$ and $\phi\Big(s + g' - g, h + h' - f\Big)$ can be fully expanded as linear combinations of $\phi(u, v)$ for obvious choices of $u, v$ among $x, e, e', \dots$, as $\beta_1$ vanishes at all resulting points. Using this, the expression becomes}
\begin{align*}
&\phi\Big(x + \tau(e - e'), \tau (z - w) + h + h' - f\Big)\\
&\hspace{1cm}+ 2\phi(x, t) - \phi(x,h) + 2\tau \phi(e, t) - \tau \phi(e,h) - 2\tau\phi(e',t) + \tau\phi(e',h)\\
&\hspace{1cm}- \phi(x, h') - \tau \phi(e,h') + \tau \phi(e', h')\\
&\hspace{1cm}+ \tau \phi(s, z) - \tau \phi(s,w) + 2\phi(s,t) - \phi(s,f)\\
&\hspace{1cm}-\tau\phi(g,z) + \tau \phi(g,w) - 2\phi(g,t) + \phi(g,f)\\
&\hspace{1cm}+ \phi\Big(x + s - g, z - w + h\Big) + \tau \phi(e, z) - \tau\phi(e, w) + \tau\phi(e,h) - \tau\phi(e', z) + \tau\phi(e', w) - \tau \phi(e', h)\\
&\hspace{1cm}- 2\phi(x,t) + \phi(x,h') - 2\tau\phi(e,t) + \tau \phi(e, h') + 2\tau \phi(e', t)\\
&\hspace{2cm} - \tau \phi(e', h') - 2\phi(s,t) + \phi(s,h') + 2\phi(g,t) - \phi(g,h')\\
&\hspace{1cm} + \phi\Big(e - e' + g', h + h' - f\Big) + (\tau + 1) \phi(e,z) - (\tau + 1)\phi(e', z) + (\tau + 1)\phi(g', z) \\
&\hspace{2cm}- (\tau + 1)\phi(e,w) + (\tau + 1)\phi(e',w) - (\tau + 1)\phi(g', w)\\
&\hspace{1cm}- \phi\Big(s + g' - g, h + h' - f\Big) - (\tau + 1)\phi(s, z) -(\tau + 1)\phi(g', z) + (\tau + 1)\phi(g,z)\\
&\hspace{2cm} + (\tau + 1)\phi(s, w) + (\tau + 1)\phi(g', w) - (\tau + 1)\phi(g,w).\end{align*}

\chg{Canceling out certain terms, we obtain}
\begin{align*}&\phi\Big(x + \tau(e - e'), \tau (z - w) + h + h' - f\Big)\\
&\hspace{1cm}+ \phi\Big(x + s - g, z - w + h\Big) + \phi\Big(e - e' + g', h + h' - f\Big) - \phi\Big(s + g' - g, h + h' - f\Big)\\
&\hspace{1cm}-\phi(x,h)-\phi(s,z) + \phi(s,w)-\phi(s,f) + \phi(g,z) - \phi(g,w)  + \phi(g,f) + \phi(s,h') -\phi(g,h')\\ 
&\hspace{1cm}+(2\tau+1) \phi(e,z) - (2\tau + 1) \phi(e,w) - (2\tau + 1) \phi(e', z) + (2\tau + 1) \phi(e',w). 
\end{align*}
The result follows after applying the induction hypothesis.\end{proof}

As announced, our final claim is that $\phi^{\text{ext}}_{a,b}$ is essentially a homomorphism in direction $G_1$ as well.

\begin{claimd*}\label{biaffine1codimExtnCl2}For all but $O(\delta^{-O(1)}\sqrt[4]{\eta} |U|^4|V|^2)$ tuples $(a,b, x_1, \dots, x_4, y)$ such that $(x_i, y) \in S_{(a,b)}$ and $x_1 + x_2 = x_3 + x_4$, we have that
\[\phi^{\mathrm{ext}}_{a,b}(x_1, y) + \phi^{\mathrm{ext}}_{a,b}(x_2, y) = \phi^{\mathrm{ext}}_{a,b}(x_3, y) + \phi^{\mathrm{ext}}_{a,b}(x_4,y).\]
\end{claimd*} 

\begin{proof}[Proof of Claim D]By Lemma~\ref{regInt1}, we have $|B^{1}_{a\bullet} \cap (\bigcap_{i \in [4]} B^{1}_{x_i \bullet})| = |B_{a\bullet} \cap (\bigcap_{i \in [4]} B_{x_i \bullet})| = \delta^4 |V|$ for all but $O(\delta^{O(1)}\sqrt[4]{\eta}|U|^4)$ choices of $(x_{[4]}, a)$ such that $x_1 + x_2 = x_3 + x_4$. Suppose therefore that $(a,b, x_1, \dots, x_4, y)$ is such that $(x_i, y) \in S_{(a,b)}$, $x_1 + x_2 = x_3 + x_4$ and $|B^{1}_{a\bullet} \cap (\bigcap_{i \in [4]} B^{1}_{x_i \bullet})| = \delta^4 |V|$. Since $(x_i, y) \in S_{(a,b)}$ and $\eta \leq \cons\,\delta^{\con}$, there are at least $\Omega(p^{-5r}|v_0 + V|)$ choices of $z \in B^{1}_{a\bullet} \cap (\bigcap_{i \in [4]} B^{1}_{x_i \bullet})$ such that $B_{\bullet z} \not= \emptyset$ and $\phi^{\text{ext}}_{a,b}(x_i,y) = \psi(x_i,y; a,b; z)$ holds for each $i \in [4]$. Let $Z$ be the set of all such $z$. We now find auxiliary elements in $u_0 + U$ and $v_0 + V$ that will allow us to perform algebraic manipulations and in particular enable us to apply Claim C. Let $\tau_i = \beta_1(x_i, y)$.\\
\indent  By Lemma~\ref{regInt2} and $\eta \leq \cons\,\delta^{\con}$, there is a choice of $z \in Z$ such that $B_{\bullet y} \cap B_{\bullet z} \not= \emptyset$. Fix such a $z$. In particular, $|B_{\bullet y} \cap B_{\bullet z}| \geq p^{-2r} |u_0 + U|$. By Lemma~\ref{regInt2} and $\eta \leq \cons\,\delta^{\con}$ there are $u,$ $e,$ $e',$ $g,$ $g',$ $s \in B_{\bullet y} \cap B_{\bullet z}$ such that
\[B_{x_1 \bullet} \cap \dots \cap B_{x_4 \bullet} \cap B_{a \bullet} \cap B_{u \bullet}  \cap B^{\mu_1}_{e \bullet} \cap B^{\mu_2}_{e' \bullet} \cap B^{\mu_3}_{g \bullet}  \cap B^{\mu_4}_{g' \bullet} \cap B_{s \bullet} \not= \emptyset\]
for each $(\mu_1, \dots, \mu_4) \in \Big\{(1,0,0,0), (0,1,0,0), (0,0,1,0), (0,0,0,1)\Big\}$.\footnote{To be more precise, we need to apply Lemma~\ref{regInt2} twice for each choice of $(\mu_1,\dots,\mu_4)$, since we have intersections of two varieties $B$ and $B^1$. This is similar to the argument in footnote~\ref{foot1} in the proof of Claim A.} In conclusion, we may find $w,$ $v,$ $f,$ $f',$ $h,$ $h',$ $t \in v_0 + V$ such that
\begin{equation}\text{$\beta_1 = 0$ for all pairs in $\{x_1, x_2, x_3, x_4, a, e, e', g, g', s\} \times \{y, z, w,v,f,f',h,h',t\}$}\label{auxValuesEqn1}\end{equation}
except that
\begin{equation}\beta_1(x_i, y) = \tau_i,\,\,\beta_1(x_i,z) = \beta_1(e,f) = \beta_1(e',f') = \beta_1(g, h) = \beta_1(g',h') = 1\text{ and possibly }\beta_1(a,y)\not= 0\label{auxValuesEqn2}\end{equation}
where the value $\beta_1(a,y)$ is irrelevant.\\

Let $\sigma \colon [4] \to \{-1,1\}$ be defined by $\sigma(1) = \sigma(2) = 1, \sigma(3) = \sigma(4) = -1$. Then by definition~\eqref{bigpsidefneqn} of $\psi$, 
\begin{align*}\sum_{i \in [4]} \sigma(i) \phi^{\text{ext}}_{a,b}(x_i,y)=& \sum_{i \in [4]} \sigma(i)\psi(x_i,y; a,b; z,s,u,v,w)\\
=&\sum_{i \in [4]} \sigma(i) \phi\Big(x_i, v + y - w - \tau_i (z - w)\Big) - \Big(\sum_{i \in [4]} \sigma(i) \phi(x_i, v)\Big)\\
&\hspace{2cm} - \Big(\sum_{i \in [4]} \sigma(i) (\tau_i - 1)\phi(x_i,w)\Big) + \Big(\sum_{i \in [4]} \sigma(i) \tau_i \phi(u + x_i - a, z)\Big)\\
&\hspace{2cm} - \sum_{i \in [4]} \sigma(i) (\tau_i - 1)\Big(- \phi(u,z) + \phi(a,s + z - b) - \phi(a,s) + h_0\Big)\\
=& \sum_{i \in [4]} \sigma(i) \phi\Big(x_i, v + y - w - \tau_i (z - w)\Big) - \Big(\sum_{i \in [4]} \sigma(i) (\tau_i - 1)\phi(x_i,w)\Big)\\
&\hspace{2cm}+\Big(\sum_{i \in [4]} \sigma(i) \tau_i \phi(u + x_i - a, z)\Big)\\
=& \sum_{i \in [4]} \sigma(i) \bigg(\phi\Big(x_i, y - \tau_i (z - w)\Big) + \phi(x_i, v) - \phi(x_i, w)\bigg) - \Big(\sum_{i \in [4]} \sigma(i) (\tau_i - 1)\phi(x_i,w)\Big)\\
&\hspace{2cm}+\Big(\sum_{i \in [4]} \sigma(i) \tau_i \phi(u + x_i - a, z)\Big)\\
=& \sum_{i \in [4]} \sigma(i) \phi\Big(x_i, y - \tau_i (z - w)\Big) - \Big(\sum_{i \in [4]} \sigma(i) \tau_i \phi(x_i,w)\Big) + \Big(\sum_{i \in [4]} \sigma(i) \tau_i \phi(u + x_i - a, z)\Big),\end{align*}
since $\phi$ is a bi-homomorphism, $\sum_{i \in [4]} \sigma(i) = 0$, $\sum_{i \in [4]} \sigma(i) x_i = 0$ and $\sum_{i \in [4]} \sigma(i) \tau_i = 0$.\\

Recall that we have auxiliary elements of $u_0 + U$ and $v_0 + V$ satisfying~\eqref{auxValuesEqn1} and~\eqref{auxValuesEqn2}. Then we may proceed with manipulating the expression above as 
\begin{align*}&\sum_{i \in [4]} \sigma(i) \phi\Big(x_i, y - \tau_i (z - w)\Big) - \Big(\sum_{i \in [4]} \sigma(i) \tau_i \phi(x_i,w)\Big) + \Big(\sum_{i \in [4]} \sigma(i) \tau_i \phi(u + x_i - a, z)\Big)\\
=&\sum_{i \in [4]} \sigma(i) \Big(\phi\Big(x_i + \tau_i(e - e'), y - \tau_i (z - w)\Big) - \tau_i\phi\Big(e, y - \tau_i (z - w)\Big) + \tau_i \phi\Big(e', y - \tau_i (z - w)\Big)\Big)\\
&\hspace{1cm} - \Big(\sum_{i \in [4]} \sigma(i) \tau_i \phi(x_i,w)\Big) + \Big(\sum_{i \in [4]} \sigma(i) \tau_i \phi(u + x_i - a, z)\Big)\\
=&\sum_{i \in [4]} \sigma(i) \Big(\phi\Big(x_i + \tau_i(e - e'), y - \tau_i (z - w)\Big) - \tau_i\phi(e, y) + \tau_i \phi(e', y)\Big)\\
&\hspace{1cm} - \sum_{i \in [4]} \sigma(i) \Big(\tau_i \phi(x_i,w) + \tau_i^2 \phi(e, w) - \tau_i^2 \phi(e', w)\Big) + \Big(\sum_{i \in [4]} \sigma(i) \tau_i \phi(u + x_i - a, z) + \tau_i^2 \phi(e, z) - \tau_i^2 \phi(e', z)\Big)\\
=&\sum_{i \in [4]} \sigma(i) \Big(\phi\Big(x_i + \tau_i(e - e'), y - f' + h\Big) + \phi\Big(x_i + \tau_i(e - e'), f' - f + h'\Big)\\
&\hspace{4cm} - \phi\Big(x_i + \tau_i(e - e'), \tau_i (z - w) + h + h' - f\Big)\Big)\\
&\hspace{1cm} - \sum_{i \in [4]} \sigma(i) \Big(\tau_i \phi(x_i,w) + \tau_i^2 \phi(e, w) - \tau_i^2 \phi(e', w)\Big) + \Big(\sum_{i \in [4]} \sigma(i) \tau_i \phi(u + x_i - a, z) + \tau_i^2 \phi(e, z) - \tau_i^2 \phi(e', z)\Big)\\
=&-\sum_{i \in [4]} \sigma(i) \phi\Big(x_i + \tau_i(e - e'), \tau_i (z - w) + h + h' - f\Big)\Big)\\
&\hspace{1cm} - \sum_{i \in [4]} \sigma(i) \Big(\tau_i \phi(x_i,w) + \tau_i^2 \phi(e, w) - \tau_i^2 \phi(e', w)\Big) + \Big(\sum_{i \in [4]} \sigma(i) \tau_i \phi(u + x_i - a, z) + \tau_i^2 \phi(e, z) - \tau_i^2 \phi(e', z)\Big).\end{align*}

Applying Claim C (and using all the auxiliary elements from~\eqref{auxValuesEqn1} that we have not yet used), we may use the identity~\eqref{biaffineidentitylemma} for $\phi\Big(x_i + \tau_i(e - e'), \tau_i (z - w) + h + h' - f\Big)$ for each $i \in [4]$ so that the expression above becomes (after most of the terms cancel because $\phi$ is a bi-homomorphism and $\sum_{i \in [4]} \sigma(i) \tau_i = 0$)

\begin{align*}&-\Big(\sum_{i \in [4]} \sigma(i)\tau_i \phi\Big(x_i + s - g, z - w + h\Big)\Big) + \Big(\sum_{i \in [4]} \sigma(i)\tau_i \phi(x_i, h)\Big) - \Big(\sum_{i \in [4]} \sigma(i) \tau_i \phi(x_i,w) \Big)\\
&\hspace{10cm} + \Big(\sum_{i \in [4]} \sigma(i) \tau_i \phi(x_i + u - a, z)\Big)\\
=\,&-\Big(\sum_{i \in [4]} \sigma(i)\tau_i \phi\Big(x_i + u - a, z - w + h\Big) + \sigma(i)\tau_i \phi\Big(s, z - w + h\Big) - \sigma(i)\tau_i \phi\Big(u - a + g, z - w + h\Big)\Big)\\
&\hspace{1cm}  + \Big(\sum_{i \in [4]} \sigma(i)\tau_i \phi(x_i + u - a, h)\Big) - \Big(\sum_{i \in [4]} \sigma(i) \tau_i \phi(x_i + u - a, w) \Big) + \Big(\sum_{i \in [4]} \sigma(i) \tau_i \phi(x_i + u - a, z)\Big)\\
=\,&\sum_{i \in [4]} \sigma(i)\tau_i\Big(-\phi\Big(x_i + u - a, z - w + h\Big) + \phi(x_i + u - a, h) - \phi(x_i + u - a, w) + \phi(x_i + u - a, z)\Big)\\
=\,&0,\end{align*}
which finally shows that 
\[\phi^{\text{ext}}_{a,b}(x_1, y) + \phi^{\text{ext}}_{a,b}(x_2, y) = \phi^{\text{ext}}_{a,b}(x_3, y) + \phi^{\text{ext}}_{a,b}(x_4,y).\qedhere\]
\end{proof}

Combining~\eqref{totalSsize} with Claims B and D, we conclude that there is a choice of $(a,b) \in B^{1}$ such that $|S_{(a,b)}| = (1 - O(\delta^{-O(1)} \eta^{\Omega(1)}))|B^{\text{ext}}|$ and $\phi^{\text{ext}}_{a,b}$ respects all but $O(\delta^{-O(1)}\eta^{\Omega(1)} |U|^3|V|)$ of the additive quadruples in direction $G_1$ and all additive quadruples in direction $G_2$. After averaging to remove at most columns $O(\delta^{-O(1)}\eta^{\Omega(1)} |v_0 + V|)$ columns, so that the remaining ones have at most $O(\delta^{-O(1)}\eta^{\Omega(1)}|u_0+U|^3)$ additive quadruples in direction $G_1$ not respected by $\phi^{\text{ext}}_{a,b}$, we may apply Corollary~\ref{approxF2homm} to make $\phi^{\text{ext}}_{a,b}$ a bi-homomorphism on a subset of $S_{(a,b)}$ of size $(1 - O(\delta^{-O(1)} \eta^{\Omega(1)}))|B^{\text{ext}}|$. Finally, apply Proposition~\ref{biaffineExtnStep2} to finish the proof.\end{proof}

\subsection{The main biaffine extension result}

We conclude this section with the final result on extending biaffine maps defined on almost all points of biaffine varieties. The result is that such maps extend to global biaffine maps, with a small error set. Moreover, we have the control over the size of the error set.

\begin{theorem}\label{biaffineExtnFullThm}Let $u_0 + U$ and $v_0 + V$ be cosets in $G_1$ and $G_2$, respectively, and let $\beta \colon G_1 \times G_2 \to \mathbb{F}_p^r$ be a biaffine map. Suppose that for each $\lambda \in \mathbb{F}_p^r$ the variety $B^\lambda = ((u_0 + U) \times (v_0 + V)) \cap \{(x,y):\beta(x,y) = \lambda\}$ is $\eta$-quasirandom with density $\delta_\lambda \geq 0$. Let $\lambda \in \mathbb{F}_p^r$ be such that $B = B^\lambda$ is non-empty.\\
\indent Let $S \subset B$ be a subset of size at least $(1-\varepsilon)|B|$ and let $\phi \colon S \to H$ be a biaffine map. Provided that $\eta \leq \cons\, \delta^{\con} p^{-\con\,r}$, we may find a global biaffine map $\Phi \colon (u_0 + U) \times (v_0 + V) \to H$ such that $\Phi(x,y) = \phi(x,y)$ for every $(x,y)$ in a subset $S' \subset S$ of size $(1 - O(\varepsilon^{\Omega(1)}) - O(\eta^{\Omega(1)})) |B|$.\end{theorem}

Before we start the proof, let us observe that if $B^{\mu}$ is $\eta$-quasirandom with density 0, it is in fact empty (this fact was mentioned after Definition~\ref{qrDefin}). Suppose on the contrary that some $(x,y)$ belongs to $B^{\mu}$. Then $|B^\mu_{\bullet y}| \geq p^{-r}|v_0 + V|$. Hence, for at least $p^{-r}|u_0 + U|$ of $x' \in u_0 + U$, we have $|B^{\mu}_{x' \bullet}| > 0$, which is a contradiction with $\eta$-quasirandomness provided $\eta < p^{-r}$.

\begin{proof} We first show that there is some $\delta > 0$ such that $\delta_\lambda \in\{0,\delta\}$ for each $\lambda$. To this end, suppose that $B^{\mu}$ is $\eta$-quasirandom with density $\delta_1 > 0$ and $B^{\nu}$ is $\eta$-quasirandom with density $\delta_2 > 0$. Then for all but $O(\eta |u_0 + U|)$ elements $x\in u_0 + U$ we have $|B^{\mu}_{x \bullet}| = \delta_1 |v_0 + V|$ and $|B^{\nu}_{x \bullet}| = \delta_2 |v_0 + V|$. But $B^{\mu}_{x \bullet}$ and $B^{\nu}_{x \bullet}$ are cosets of the same subspace, so $|B^{\mu}_{x \bullet}| = |B^{\nu}_{x \bullet}|$, which proves $\delta_1 = \delta_2$.\\

Next, for each $i \in [r]$ define biaffine map $\beta^{(i)} \colon G_1 \times G_2 \to \mathbb{F}_p^{[i,r]}$ by $\beta^{(i)}_j = \beta_j$ for $j \in [i,r]$. Thus $\beta^{(1)} = \beta$. We note that all layers of $\beta^{(i)}$ are $(O(p^{6r}\sqrt[4]{\eta})$-quasirandom. Indeed, if we write $\delta_\lambda$ for quasirandomness density of layer $B^\lambda$, and if $\mu \in \mathbb{F}_p^{[i,r]}$, we see that 
\[|\{y \in v_0 + V \colon \beta^{(i)}(x,y) = \mu\}| = \sum_{\nu \in \mathbb{F}_p^{[i-1]}} |B^{(\nu, \mu)}_{x \bullet}| = \Big(\sum_{\nu \in \mathbb{F}_p^{[i-1]}} \delta_{(\nu,\mu)}\Big)|v_0 + V|\] 
for all but at most $p^{i-1}\eta |u_0 + U|$ elements $x \in u_0 + U$. Write $\tilde{\delta} = \sum_{\nu \in \mathbb{F}_p^{[i-1]}} \delta_{(\nu,\mu)}$.\\

\chg{Let us now check the second part of the quasirandomness. For two elements $x, x' \in u_0 + U$ we have
\[|\{y \in v_0 + V \colon \beta^{(i)}(x,y) = \beta^{(i)}(x',y) = \mu\}|\ \ = \sum_{\nu, \nu' \in \mathbb{F}_p^{[i-1]}} |B^{(\nu, \mu)}_{x \bullet} \cap B^{(\nu', \mu)}_{x' \bullet}|.\]
Fix any $x \in u_0 + U$ such that $|B^{(\nu, \mu)}_{x \bullet}| = \delta_{(\nu, \mu)} |v_0 + V|$ holds for all $\nu \in \mathbb{F}_p^{[i-1]}$. By the previous argument, this holds for all but at most $p^{2i-2}\eta |u_0 + U|$ elements $x \in u_0 + U$. Using Lemma~\ref{regInt2} for each $\nu'$, we have $|B^{(\nu, \mu)}_{x \bullet} \cap B^{(\nu', \mu)}_{x' \bullet}| = \delta_{(\nu', \mu)} |B^{(\nu, \mu)}_{x \bullet}| =  \delta_{(\nu, \mu)} \delta_{(\nu', \mu)}|v_0 + V| $ for all but at most $O(p^{-6r}\sqrt[4]{\eta}|u_0 + U|)$ choices of $x' \in u_0 + U$. Hence, 
\[|\{y \in v_0 + V \colon \beta^{(i)}(x,y) = \beta^{(i)}(x',y) = \mu\}|\ =\ \tilde{\delta}^2 |v_0 + V|\]
for all but at most $O(p^{6r}\sqrt[4]{\eta}|u_0 + U|^2)$  pairs $(x, x') \in u_0 + U$, completing the check of quasirandomness.}\\

Now apply Proposition~\ref{biaffineAlmostToFullExtn}, and then apply Proposition~\ref{1codimExtnBiaffineProp} $r$ times (at the $i$\textsuperscript{th} step using the map $\beta^{(i)}$ in the definition of biaffine varieties) to complete the proof. The work above shows that all layers of $\beta^{(i)}$ are quasirandom, and the argument at the beginning of this proof then shows that all non-empty layers among them have the same density. Hence the assumptions of Proposition~\ref{1codimExtnBiaffineProp} are satisfied at each step.\end{proof}

\section{A simultaneous biaffine regularity lemma}\label{sec9labelsimbiaffinereglemma}

Throughout this section, we shall frequently consider sequences that agree with an element $x$ on all coordinates but two, and take prescribed values on one or two of those two coordinates. In order to notate these, we introduced the following conventions earlier, which we now recall. The notation $(x_{[k] \setminus \{d_1, d_2\}},^{d_1}\!y,^{d_2}\!z)$ stands for the sequence $u_{[k]}$ such that $u_i=x_i$ for $i \in [k] \setminus \{d_1, d_2\}$, $u_{d_1}=y$, and $u_{d_2}=z$. Similarly, the notation $(x_{[k] \setminus \{d_1, d_2\}},^{d_1}\!y)$ stands for the sequence $u_{[k]\setminus\{d_2\}}$ such that $u_i=x_i$ for $i \in [k] \setminus \{d_1, d_2\}$ and $u_{d_1}=y$, while $(x_{[k] \setminus \{d_1, d_2\}},^{d_2}\!z)$ stands for the sequence $u_{[k]\setminus\{d_1\}}$ such that $u_i=x_i$ for $i \in [k] \setminus \{d_1, d_2\}$ and $u_{d_2}=z$.

We shall adopt the further convention that unless we specify to the contrary, additional coordinates are inserted in the obvious order. So we shall often write $(x_{[k] \setminus \{d_1, d_2\}},y,z)$ instead of $(x_{[k] \setminus \{d_1, d_2\}},^{d_1}\!y,^{d_2}\!z)$ and $(x_{[k] \setminus \{d_1, d_2\}},y)$ instead of $(x_{[k] \setminus \{d_1, d_2\}},^{d_1}\!y)$. However, we cannot abbreviate the expression $(x_{[k] \setminus \{d_1, d_2\}},^{d_2}\!z)$ in this way, since the $z$ goes in the `second slot' rather than the first. (Usually there will not, strictly speaking, be any ambiguity to avoid since the sequence will be the argument of a multilinear map and we will have specified the domain of the map, but we do not want to rely on the reader's memory to that extent.) \\

Recall that a collection $\mathcal{G} \subset \mathcal{P}([k])$ is a down-set if it is closed under taking subsets. Recall also that a multiaffine map $\alpha \colon G_{[k]} \to \mathbb{F}^r$ is $\mathcal{G}$-supported if it can be written in the form $\alpha(x_{[k]}) = \sum_{I \in \mathcal{G}} \alpha_I(x_I)$ for some multilinear maps $\alpha_I \colon G_I \to \mathbb{F}^r$.\\

Broadly speaking, the main theorem of this section allows us to take a multiaffine variety, fix two directions, and decompose almost all 2-dimensional layers obtained by fixing the remaining coordinates into large quasirandom pieces. \chg{Since the statement is very technical, we shall follow it with a more intuitive interpretation of the various objects that play a role in the theorem.} The precise statement is as follows.

\begin{theorem}\label{simReg}Let $\eta > 0$ and let $d_1, d_2 \in [k]$ be two coordinates. Let $\beta^1 \colon G_{[k] \setminus \{d_2\}} \to \mathbb{F}_p^{r_1}$, $\beta^2 \colon G_{[k] \setminus \{d_1\}} \to \mathbb{F}_p^{r_2}$ and $\beta^{12} \colon G_{[k]} \to \mathbb{F}_p^{r_{12}}$ be multiaffine maps. Let $\mathcal{G}$ be a down-set such that $\beta^{12}$ is $\mathcal{G}$-supported. For each $i \in [r_{12}]$ write the map $\beta^{12}_i(x_{[k]})$ as $\alpha_i(x_{[k] \setminus \{d_2\}}) + \alpha'_i(x_{[k]\setminus \{d_1, d_2\}}) + (A_i(x_{[k] \setminus \{d_2\}}) + A'_i(x_{[k] \setminus\{d_1, d_2\}})) \cdot x_{d_2}$, where the maps $A_i \colon G_{[k] \setminus\{d_2\}} \to G_{d_2}$, $A_i' \colon G_{[k] \setminus\{d_1, d_2\}} \to G_{d_2}$, $\alpha_i\colon G_{[k] \setminus\{d_2\}} \to \mathbb{F}_p$ and $\alpha'_i\colon G_{[k] \setminus\{d_1, d_2\}} \to \mathbb{F}_p$ are multiaffine and in addition $A_i$ and $\alpha_i$ are linear (as opposed to merely affine) in coordinate $d_1$. Let $\mathcal{G}' = \{S \subset [k] \setminus \{d_2\} \colon S \cup \{d_1, d_2\} \in \mathcal{G}\}$ and $\mathcal{G}'' = \{S \subset [k] \setminus \{d_1, d_2\} \colon S \cup \{d_1, d_2\} \in \mathcal{G}\}$. Then, there exist 
\begin{itemize}
\item positive integers $m, t = O((r_{12} + r_1 + r_2 + \log_{p}\eta^{-1})^{O(1)})$,
\item $\phi=(\phi_1,\dots,\phi_{r_{12}})$, where each $\phi_j\colon G_{[k] \setminus \{d_2\}} \to \mathbb{F}_p^m$ is a $\mathcal{G}'$-supported multiaffine map that is linear in coordinate $d_1$, 
\item a $\mathcal{G}''$-supported  multiaffine map $\gamma \colon G_{[k] \setminus \{d_1, d_2\}} \to \mathbb{F}_p^t$, and
\item a union $F \subset G_{[k] \setminus \{d_1,d_2\}}$ of layers of $\gamma$ of size $|F| \leq \eta|G_{[k] \setminus \{d_1,d_2\}}|$,
\end{itemize}
such that for each layer $L$ of $\gamma$ not in $F$, there is a subspace $\Lambda \leq \mathbb{F}_p^{r_{12}}$ such that for each $x_{[k] \setminus \{d_1,d_2\}} \in L$, $u_0 \in G_{d_1}, v_0 \in G_{d_2}$ and $\tau \in \mathbb{F}_p^{r_{12}}$, the biaffine variety
\begin{align}&\bigg[\bigg(u_0 + \Big(\{y \in G_{d_1} \colon (\forall \lambda \in \Lambda)\,\, \lambda \cdot \phi(x_{[k] \setminus \{d_1, d_2\}}, \ls{d_1} y) = 0\} \nonumber\\
&\hspace{3cm}\cap \{y \in G_{d_1} \colon \beta^1(x_{[k] \setminus \{d_1, d_2\}}, \ls{d_1} y) =\beta^1(x_{[k] \setminus \{d_1, d_2\}}, \ls{d_1} 0)\}\nonumber\\
&\hspace{3cm}\cap \{y \in G_{d_1} \colon (\forall i \in [r_{12}])\,\, \alpha_i(x_{[k] \setminus \{d_1, d_2\}}, \ls{d_1} y) = 0\}\Big)\bigg) \nonumber\\
&\hspace{2cm}\times \bigg(v_0 + \Big(\langle A'_1(x_{[k] \setminus \{d_1,d_2\}}), \dots, A'_{r_{12}}(x_{[k] \setminus \{d_1, d_2\}})\rangle^\perp\nonumber\\
&\hspace{4cm} \cap \langle A_1(x_{[k] \setminus \{d_1, d_2\}}, \ls{d_1} {u_0}), \dots, A_{r_{12}}(x_{[k] \setminus \{d_1,d_2\}}, \ls{d_1} {u_0})\rangle^\perp\nonumber\\
&\hspace{4cm}\cap \{z \in G_{d_2} \colon \beta^2(x_{[k] \setminus \{d_1, d_2\}}, \ls{d_2}{z}) = \beta^2(x_{[k] \setminus \{d_1, d_2\}}, \ls{d_2} 0)\}\Big)\bigg)\bigg]\nonumber\\
&\hspace{1cm} \cap \{(y,z) \in G_{d_1}\times G_{d_2} \colon \beta^{12}(x_{[k] \setminus \{d_1, d_2\}}, y, z) = \tau\}\label{verylongbiaffinevar}\end{align}
is either $\eta$-quasirandom with density $|\Lambda|p^{-{r_{12}}}$ or empty.
\end{theorem}

\noindent\chg{\textbf{Interpretation.} The principal objective of the theorem is to find quasirandom biaffine slices (in directions $d_1$ and $d_2$) of layers of the map $\beta^{12}$. Thus, we first need to split its multilinear parts depending on whether variables $x_{d_1}$ and $x_{d_2}$ play a role, which leads to the decomposition
\[\beta^{12}_i(x_{[k]}) = \alpha_i(x_{[k] \setminus \{d_2\}}) + \alpha'_i(x_{[k]\setminus \{d_1, d_2\}}) + (A_i(x_{[k] \setminus \{d_2\}}) + A'_i(x_{[k] \setminus\{d_1, d_2\}})) \cdot x_{d_2}.\]
The maps $\beta^1$ and $\beta^2$ are present for technical reasons, and they do not play a crucial role as each of them misses one of the variables $x_{d_1}$ and $x_{d_2}$. Furthermore, via the down-set $\mathcal{G}$, we also keep track of which sets of indices of variables appear in the multilinear parts of $\beta^{12}$.\\
\indent The outputs of the theorem are further multiaffine maps $\phi$ and $\gamma$ of bounded codimension, such that $\phi$ does not depend on $x_{d_2}$ and $\gamma$ does not depend on $x_{d_1}$ and $x_{d_2}$. Observe that all terms, except $\{y \in G_{d_1} \colon (\forall \lambda \in \Lambda)\,\, \lambda \cdot \phi(x_{[k] \setminus \{d_1, d_2\}}, \ls{d_1} y) = 0\}$, in the definition of the variety~\eqref{verylongbiaffinevar} come from $\beta^{12}, \beta^1$ and $\beta^2$. Here, the purpose of the map $\phi$ is to provide us with an additional subspace, on whose cosets the biaffine variety in~\eqref{verylongbiaffinevar} becomes quasirandom. On the other hand, $\gamma$ serves to control the set of points $x_{[k] \setminus \{d_1, d_2\}}$ for which the biaffine variety in~\eqref{verylongbiaffinevar} fails to be quasirandom.}\\

\noindent\textbf{Remark.} Note that the codimension in direction $d_2$ is almost unchanged -- it doubles at worst. In particular, it does not depend on $\eta$.\\

We begin with a lemma that will play an important role in the proof. For both the lemma and the proof we shall consider just the case $d_1 = k - 1$ and $d_2 = k$, which clearly loses no generality. For $\Lambda \leq \mathbb{F}_p^{r_{12}}$, we define the set
\[V_{\Lambda} = \{x_{[k-1]} \in G_{[k-1]} \colon (\forall \lambda \in \Lambda)\,\,\lambda \cdot A(x_{[k-1]}) =0 \}.\]
We take a moment to recall the convention on how to interpret the expression $\lambda\cdot A(x_{[k-1]})$. Note that $A=(A_1,\dots,A_{r_{12}})$ is a (multiaffine) map from $G_{[k-1]}$ to $\mathbb F_p^{r_{12}}$. Then $\lambda \cdot A$ is the multiaffine map $\sum_{i=1}^{r_{12}}\lambda_iA_i$. \\
\indent\chg{For each $x_{[k-1]}$, it is of crucial importance to keep track of all vanishing linear combinations of elements $A_1(x_{[k-1]}), \dots, A_{r_{12}}(x_{[k-1]})$. The set $V_{\Lambda}$ (a variety in fact) consists of those $x_{[k-1]}$ for which the given $r_{12}$ elements satisfy all equations corresponding to linear  combinations in $\Lambda$. The next lemma tells us that precise control over the vanishing linear combinations can be used to deduce quasirandomness of the biaffine varieties of interest.}\\
\indent The notation in the lemma below and its proof is inherited from Theorem~\ref{simReg}, so $\beta^{12}, \alpha, \alpha'$, etc.\ have the same meaning.

\begin{lemma}\label{qrEmptyClaim}Let $x_{[k-2]} \in G_{[k-2]}$ and $\Lambda \leq \mathbb{F}_p^{r_{12}}$ be such that $\eta p^{-5r_{12} -r_1-r_2}|(V_{\Lambda})_{x_{[k-2]}}|\hspace{3pt} > |(V_M)_{x_{[k-2]}}|$ for all $M\leq\mathbb{F}_p^{r_{12}}$ such that $\Lambda\lneqq M$. Then for each $u_0 \in G_{k-1}$, $v_0 \in G_k$ and $\tau \in \mathbb{F}_p^{r_{12}}$, the biaffine variety 
\begin{align*}B = &\bigg[\bigg(u_0 + \Big((V_\Lambda)_{x_{[k-2]}} \cap \{y \in G_{k-1} \colon \beta^1(x_{[k-2]}, y) = \beta^1(x_{[k-2]}, 0)\}\\
&\hspace{9cm}\cap \{y \in G_{k-1} \colon (\forall i \in [r_{12}])\,\,\alpha_i(x_{[k-2]}, y) = 0\}\Big)\bigg) \\
&\hspace{2cm}\times \bigg(v_0 + \Big(\langle A'_1(x_{[k-2]}), \dots, A'_{r_{12}}(x_{[k-2]})\rangle^\perp \cap \langle A_1(x_{[k-2]}, u_0), \dots, A_{r_{12}}(x_{[k-2]}, u_0)\rangle^\perp\\
&\hspace{9cm}\cap \{z \in G_{k} \colon \beta^2(x_{[k-2]},\ls{k} z) = \beta^2(x_{[k-2]},\ls{k} 0)\}\Big)\bigg)\bigg]\\
&\hspace{1cm} \cap \{(y,z) \in G_{k-1} \times G_k \colon \beta^{12}(x_{[k-2]}, y,z) = \tau\}\end{align*}
is either $\eta$-quasirandom with density $\delta = p^{-r_{12}} |\Lambda|$ or empty.\end{lemma}

\noindent\chg{\textbf{Remark.} The assumption on $V_{\Lambda}$ and $V_{M}$ for all $\Lambda\lneqq M$ means that for most $x_{k-1} \in (V_{\Lambda})_{x_{[k-2]}}$ the vanishing linear combinations of $A_1(x_{[k-1]}), \dots, A_{r_{12}}(x_{[k-1]})$ are precisely the linear combinations in $\Lambda$.}

\begin{proof} Suppose that for the given $x_{[k-2]}$, $\Lambda$, $u_0, v_0$ and $\tau$, $B$ is neither $\eta$-quasirandom with density $\delta$ nor empty. Write 
\begin{align*}Y = (V_\Lambda)_{x_{[k-2]}} \hspace{3pt}\cap&\hspace{3pt} \{y \in G_{k-1} \colon \beta^1(x_{[k-2]}, y) = \beta^1(x_{[k-2]}, 0)\}\\
\hspace{3pt}\cap &\hspace{3pt}\{y \in G_{k-1} \colon (\forall i \in [r_{12}])\,\,\alpha_i(x_{[k-2]}, y) = 0\},\end{align*}
which is a subspace of $G_{k-1}$, and 
\begin{align}Z = \langle A'_1(x_{[k-2]}), \dots, A'_{r_{12}}(x_{[k-2]})\rangle^\perp\hspace{3pt} \cap& \hspace{3pt}\langle A_1(x_{[k-2]}, u_0), \dots, A_{r_{12}}(x_{[k-2]}, u_0)\rangle^\perp\nonumber\\
 \cap&\hspace{3pt} \{z \in G_{k} \colon \beta^2(x_{[k-2]},\ls{k} z) = \beta^2(x_{[k-2]},\ls{k} 0)\},\label{ZspaceDef}\end{align}
which is a subspace of $G_k$. \chg{Note that these two subspaces appear in the statement in the definition of the variety $B$, which can be expressed as $B = \{(y, z) \in (u_0 + Y) \times (v_0 + Z) \colon \beta^{12}(x_{[k-2]}, y, z) = \tau\}$.}\\
\indent\chg{The key observation in the proof is that if the set of linear combinations of $A_1(x_{[k-2]}, y), \dots, A_{r_{12}}(x_{[k-2]}, y)$ that belongs to $Z^\perp$ is precisely $\Lambda$ for most $y \in u_0 + U$, and if a similar condition holds for pairs $(y, y') \in (u_0 + U)^2$, then $B$ is quasirandom. This is the content of Claim A below. However, before proving the claim, some preliminary calculations relating various values of $A_1, \dots, A_{r_{12}}$ are required.}\\

\indent For given $y \in u_0 + Y$, we have
\begin{align}\label{ByEqn}B_{y \bullet} = \hspace{1pt}&\Big\{z \in v_0 + Z \colon \beta^{12}(x_{[k-2]}, y, z) = \tau\}\nonumber\\
=\hspace{1pt}&\Big\{z \in v_0 + Z \colon (\forall i \in [r_{12}]) \hspace{3pt}\alpha_i(x_{[k-2]}, y) + \alpha'_i(x_{[k-2]}) + (A_i(x_{[k-2]}, y) + A'_i(x_{[k-2]})) \cdot z = \tau_i\Big\}\nonumber\\
=\hspace{1pt}&\Big\{z \in v_0 + Z \colon (\forall i \in [r_{12}])\hspace{3pt} A_i(x_{[k-2]}, y) \cdot z = \tau_i - A'_i(x_{[k-2]}) \cdot v_0 - \alpha_i(x_{[k-2]}, u_0) - \alpha'_i(x_{[k-2]})\Big\},\end{align}
where we used the fact that $z-v_0 \in Z \subset \langle A'_1(x_{[k-2]}), \dots, A'_{r_{12}}(x_{[k-2]})\rangle^\perp$ to replace $A'_i(x_{[k-2]}) \cdot z$ by $A'_i(x_{[k-2]}) \cdot v_0$ in the last equality.\\

Recall that $B$ is non-empty, and let $(y,z) \in B$. Let $\lambda \in \Lambda$. Then
\begin{equation}\sum_{i \in [r_{12}]} \lambda_i A_i(x_{[k-2]}, u_0) \cdot v_0=\sum_{i \in [r_{12}]} \lambda_i A_i(x_{[k-2]}, u_0) \cdot z,\label{qrlemmauvtoyz}\end{equation}
using the fact that $z-v_0 \in Z \subset \langle A_1(x_{[k-2]}, u_0), \dots, A_{r_{12}}(x_{[k-2]}, u_0)\rangle^\perp$. But $y - u_0 \in (V_\Lambda)_{x_{[k-2]}}$, so~\eqref{qrlemmauvtoyz} equals $\sum_{i \in [r_{12}]} \lambda_i A_i(x_{[k-2]}, y) \cdot z$, so by~\eqref{ByEqn} we end up with the equality  
\begin{align}\label{uvEqn}
\sum_{i \in [r_{12}]} \lambda_i A_i(x_{[k-2]}, u_0) \cdot v_0=- \sum_{i \in [r_{12}]}\lambda_i \Big(A'_i(x_{[k-2]}) \cdot v_0 + \alpha_i(x_{[k-2]}, u_0) + \alpha'_i(x_{[k-2]}) - \tau_i\Big).
\end{align}


From~\eqref{ZspaceDef} we have
\begin{align*}Z^\perp = \langle A'_1(x_{[k-2]}), \dots, A'_{r_{12}}(x_{[k-2]})\rangle \hspace{3pt} + & \hspace{3pt}\langle A_1(x_{[k-2]}, u_0), \dots, A_{r_{12}}(x_{[k-2]}, u_0)\rangle\\
 +&\hspace{3pt} \{z \in G_{k} \colon \beta^2(x_{[k-2]},\ls{k} z) = \beta^2(x_{[k-2]},\ls{k} 0)\}^\perp.\end{align*}
Observe that for each $y \in u_0 + Y$ and each pair $(y_1, y_2) \in (u_0 + Y)^2$, 
\begin{equation}\label{lambdaInBy}\{\lambda \in \mathbb{F}_p^{r_{12}} \colon \lambda \cdot A(x_{[k-2]}, y) \in Z^\perp\} \supseteq \Lambda\end{equation}
and
\[\{(\lambda, \mu) \in \mathbb{F}_p^{r_{12}} \times \mathbb{F}_p^{r_{12}} \colon \lambda \cdot A(x_{[k-2]}, y_1) + \mu \cdot A(x_{[k-2]}, y_2) \in Z^\perp\} \supseteq \Lambda \times \Lambda.\]
Recall from the statement that $\delta = p^{-r_{12}}|\Lambda|$. \chg{We now arrive at the key claim that will allow us to deduce quasirandomness of $B$, and hence obtain a contradiction.}

\begin{claima*}
\begin{itemize}
\item[\textbf{(i)}] If $\{\lambda \in \mathbb{F}_p^{r_{12}} \colon \lambda \cdot A(x_{[k-2]}, y) \in Z^\perp\} =\Lambda$, then either $|B_{y \bullet}| = \delta |v_0 + Z|$ or $|B_{y \bullet}| = 0$.\footnote{In fact, most of the time the second possibility cannot occur, but we do not need to prove this fact directly, since it will follow from \textbf{(ii)}.}
\item[\textbf{(ii)}] If $\{(\lambda, \mu) \in \mathbb{F}_p^{r_{12}} \times \mathbb{F}_p^{r_{12}} \colon \lambda \cdot A(x_{[k-2]}, y_1) + \mu \cdot A(x_{[k-2]}, y_2) \in Z^\perp\} = \Lambda \times \Lambda$, then $|B_{y_1 \bullet} \cap B_{y_2 \bullet}| = \delta^2 |v_0 + Z|$.
\end{itemize}
\end{claima*}

\begin{proof}[Proof of Claim A]
To prove \textbf{(i)}, suppose that $B_{y \bullet}$ is non-empty. From~\eqref{ByEqn}, $|B_{y \bullet}| = |\{z \in Z \colon (\forall i \in [r_{12}])\,\, A_i(x_{[k-2]}, y) \cdot z = 0\}|$. Consider the linear map $\psi \colon Z \to \mathbb{F}_p^{r_{12}}$, $\psi \colon z \mapsto (A_i(x_{[k-2]}, y) \cdot z \colon i \in [r_{12}])$. By the rank-nullity theorem, $|B_{y \bullet}| / |Z| = |\img \psi|^{-1} = |(\img \psi)^\perp| p^{-r_{12}}$. Hence, 
\begin{align*}|B_{y \bullet}| = &\Big|\Big\{\lambda \in \mathbb{F}_p^{r_{12}} \colon (\forall z \in Z) \sum_{i \in [r_{12}]} \lambda_i A_i(x_{[k-2]}, y) \cdot z = 0\Big\}\Big| p^{-r_{12}} |v_0 + Z|\\
 = &\Big|\Big\{\lambda \in \mathbb{F}_p^{r_{12}} \colon \sum_{i \in [r_{12}]} \lambda_i A_i(x_{[k-2]}, y) \in Z^\perp\Big\}\Big| p^{-r_{12}} |v_0 + Z|\\
= &|\Lambda| p^{-r_{12}} |v_0 + Z| = \delta |v_0 + Z|.\end{align*}
We now turn to \textbf{(ii)}. First we show that if $B_{y_1 \bullet} \cap B_{y_2 \bullet} \not= \emptyset$, then $|B_{y_1  \bullet} \cap B_{y_2 \bullet}| = \delta^2 |v_0 + Z|$. We argue similarly to \textbf{(i)}. From~\eqref{ByEqn}, $|B_{y_1  \bullet} \cap B_{y_2  \bullet}| = |\{z \in Z \colon (\forall i \in [r_{12}]) A_i(x_{[k-2]}, y_1) \cdot z = A_i(x_{[k-2]}, y_2) \cdot z = 0\}|$. Consider the linear map $\psi \colon Z \to \Big(\mathbb{F}_p^{r_{12}}\Big)^2$ given by
\[\psi \colon z \mapsto \Big((A_i(x_{[k-2]}, y_1) \cdot z \colon i \in [r_{12}]), (A_i(x_{[k-2]}, y_2) \cdot z \colon i \in [r_{12}])\Big).\] 
By the rank-nullity theorem, $|B_{y_1 \bullet} \cap B_{y_2 \bullet}| / |Z| = |\img \psi|^{-1} = |(\img \psi)^\perp| p^{-2r_{12}}$. Hence, 
\begin{align*}|B_{y_1 \bullet} \cap B_{y_2 \bullet}| = &\Big|\Big\{(\lambda, \mu) \in \mathbb{F}_p^{r_{12}} \times \mathbb{F}_p^{r_{12}} \colon (\forall z \in Z) \sum_{i \in [r_{12}]} (\lambda_i A_i(x_{[k-2]}, y_1) + \mu_iA_i(x_{[k-2]}, y_2)) \cdot z = 0\Big\}\Big| p^{-2r_{12}} |v_0 + Z|\\
 = &\Big|\Big\{(\lambda, \mu) \in \mathbb{F}_p^{r_{12}} \times \mathbb{F}_p^{r_{12}} \colon \sum_{i \in [r_{12}]} (\lambda_i A_i(x_{[k-2]}, y_1) + \mu_i A_i(x_{[k-2]}, y_2)) \in Z^\perp\Big\}\Big| p^{-2r_{12}} |v_0 + Z|\\
= &|\Lambda|^2 p^{-2r_{12}} |v_0 + Z| = \delta^2 |v_0 + Z|.\end{align*}
To finish the proof of \textbf{(ii)}, we apply Lemma~\ref{solCrit}. We just need to check that for $(\lambda, \mu) \in \Lambda \times \Lambda$, 
\begin{align*}\sum_{i \in [r_{12}]} &\lambda_i (A'_i(x_{[k-2]}) \cdot v_0 + \alpha_i(x_{[k-2]}, u_0) + \alpha'_i(x_{[k-2]}) + A_i(x_{[k-2]}, y_1) \cdot v_0 - \tau_i)\\
+&\sum_{i \in [r_{12}]} \mu_i (A'_i(x_{[k-2]}) \cdot v_0 + \alpha_i(x_{[k-2]}, u_0) + \alpha'_i(x_{[k-2]}) + A_i(x_{[k-2]}, y_2) \cdot v_0 - \tau_i) = 0.\end{align*}
Since $y_1 - u_0, y_2 - u_0 \in Y \subset (V_\Lambda)_{x_{[k-2]}}$, this is equivalent to 
\begin{align*}\sum_{i \in [r_{12}]} &\lambda_i (A'_i(x_{[k-2]}) \cdot v_0 + \alpha_i(x_{[k-2]}, u_0) + \alpha'_i(x_{[k-2]}) + A_i(x_{[k-2]}, u_0) \cdot v_0 - \tau_i)\\
+&\sum_{i \in [r_{12}]} \mu_i (A'_i(x_{[k-2]}) \cdot v_0 + \alpha_i(x_{[k-2]}, u_0) + \alpha'_i(x_{[k-2]}) + A_i(x_{[k-2]}, u_0) \cdot v_0 - \tau_i) = 0,\end{align*}
which, by~\eqref{uvEqn}, indeed holds.\end{proof}

Having proved observations \textbf{(i)} and \textbf{(ii)}, we now suppose that 
\[\{(\lambda, \mu) \in \mathbb{F}_p^{r_{12}} \times \mathbb{F}_p^{r_{12}} \colon \lambda \cdot A(x_{[k-2]}, y_1) + \mu \cdot A(x_{[k-2]}, y_2) \in Z^\perp\} = \Lambda \times \Lambda\]
for a proportion at least $1-\eta$ of the pairs $(y_1, y_2)$ in $u_0 + Y$. 
By property \textbf{(ii)}, we have in particular that $B_{y_1 \bullet} \not=\emptyset$. Also, by~\eqref{lambdaInBy} we see that 
\[\{\lambda \in \mathbb{F}_p^{r_{12}} \colon \lambda \cdot A(x_{[k-2]}, y) \in Z^\perp\} =\Lambda\]
for a proportion at least $1-\eta$ of elements $y \in u_0 + Y$.
Hence, by properties \textbf{(i)} and \textbf{(ii)} again, we have that $|B_{y  \bullet}| = \delta |v_0 + Z|$ for a proportion at least $1-\eta$ of elements $y \in u_0 + Y$, and $|B_{y_1 \bullet} \cap B_{y_2 \bullet}| = \delta^2 |v_0 + Z|$ for a proportion at least $1-\eta$ of the pairs $(y_1, y_2) \in (u_0 + Y)^2$. Thus, $B$ is $\eta$-quasirandom with density $\delta$, which is a contradiction.\\

\indent Therefore, if the given biaffine variety fails to be $\eta$-quasirandom with density $\delta$, then without loss of generality there is a pair $(\lambda, \mu) \in \mathbb{F}_p^{r_{12}}\times \mathbb{F}_p^{r_{12}}$, where $\mu \notin \Lambda$, such that
\[\lambda \cdot A(x_{[k-2]}, y_1) + \mu \cdot A(x_{[k-2]}, y_2) \in Z^\perp\]
for at least an $\eta p^{-2r_{12}}$ proportion of the pairs $(y_1, y_2) \in (u_0 + Y)^2$. Since $\dim Z^\perp \leq r_2 + 2r_{12}$, by writing $d = y_2 - y_1$ and averaging we find that there exist $y_1 \in u_0 + Y$ and $w \in Z^\perp$ such that for at least $\eta p^{-4r_{12} - r_2} |Y|$ elements $d \in Y$, we have
\[(\lambda +\mu) \cdot A(x_{[k-2]}, y_1) + \mu \cdot A(x_{[k-2]}, d) = w.\]
By taking differences between values of this expression for different choices of $d$, we conclude that 
\[\mu \cdot A(x_{[k-2]}, d') = 0\]
for at least $\eta p^{-4r_{12}-r_2} |Y|$ elements $d' \in Y \subset (V_\Lambda)_{x_{[k-2]}}$. This implies that
\[|(V_M)_{x_{[k-2]}}|\,\geq \eta p^{-4r_{12}-r_2} |Y|\, \geq \eta p^{-5r_{12} -r_1 -r_2}|(V_\Lambda)_{x_{[k-2]}}|\]
for $M = \Lambda + \langle \mu \rangle$, which is a contradiction.\end{proof}

\begin{proof}[Proof of Theorem \ref{simReg}]
\chg{Recall that our key goal is to have strong control over the sets $V_\Lambda$, which were defined as 
\[V_{\Lambda} = \{x_{[k-1]} \in G_{[k-1]} \colon (\forall \lambda \in \Lambda)\,\,\lambda \cdot A(x_{[k-1]}) =0 \}.\]
We begin the proof by passing from the maps $A_i$ to suitable multiaffine maps $\phi_i$ with codomains of bounded dimension.}\\

Note that $A_i$ is $\mathcal{G}'$-supported for each $i \in [r_{12}]$, where $\mathcal{G}'$ was defined in the statement as $ \{S \subset [k] \setminus \{k\} \colon S \cup \{k-1, k\} \in \mathcal{G}\}$. Let $m \in \mathbb{N}$ be a parameter to be specified later. Apply Lemma~\ref{BohrApproxSim} to $A_{[r_{12}]}$ to find $\mathcal{G}'$-supported multiaffine maps $\phi_{1}, \dots, \phi_{r_{12}} \colon G_{[k-1]} \to \mathbb{F}_p^m$, linear in coordinate $k-1$, such that for each $\lambda \in \mathbb{F}_p^{r_{12}}$, 
\begin{equation}\{x_{[k-1]} \in G_{[k-1]} \colon \lambda \cdot A(x_{[k-1]}) = 0\} \subset \{x_{[k-1]} \in G_{[k-1]} \colon \lambda \cdot \phi(x_{[k-1]}) = 0\}\label{phiPropEqnOuterApp}\end{equation}
and
\begin{equation}\Big|\{x_{[k-1]} \in G_{[k-1]} \colon \lambda \cdot \phi(x_{[k-1]}) = 0\} \setminus \{x_{[k-1]} \in G_{[k-1]} \colon \lambda \cdot A(x_{[k-1]}) = 0\}\Big|\,\, \leq\,\, p^{r_{12} - m} |G_{[k-1]}|.\label{phiPropEqnOuterApp2}\end{equation}

\chg{Having found the maps $\phi_i$, we want to use them instead of the maps $A_i$. This is possible if for a given point $x_{[k-1]}$ the vanishing linear combinations of $A_1(x_{[k-1]}), \dots, A_{r_{12}}(x_{[k-1]})$ coincide with those of $\phi_1(x_{[k-1]}), \dots, \phi_{r_{12}}(x_{[k-1]})$. With this in mind, in order to control the extent to which the described property fails,} let $F_1$ be the set of all $x_{[k-2]} \in G_{[k-2]}$ with the property that there is $\Lambda \leq \mathbb{F}_p^{r_{12}}$ such that 
\[|\{y \in G_{k-1} \colon (\forall \lambda \in \Lambda) \lambda \cdot \phi(x_{[k-2]}, y) = 0\}| \geq \eta^{r_{12}} p^{-5r_{12}^2 - r_1 r_{12} - r_2r_{12}}|G_{k-1}|,\] 
but 
\[(V_{\Lambda})_{x_{[k-2]}} \subsetneq \{y \in G_{k-1} \colon (\forall \lambda \in \Lambda)\,\, \lambda \cdot \phi(x_{[k-2]}, y) = 0\}.\]
\noindent\chg{\textbf{Overview of the rest of the proof.} The rest of the proof proceeds as follows. First, we need to show that $F_1$ is small. Next, since it is not obvious whether $F_1$ has the desired algebraic structure, we need to pass to a small union of layers of varieties of bounded codimension $\tilde{F}_1$ that contains $F$. The final `forbidden set' $F$ will be $F = \tilde{F}_1 \cup F_2$, where $F_2$ is another small union of layers of varieties of bounded codimension which controls additional vanishing linear combinations of considered maps. Once $F$ is suitably defined, we may use Lemma~\ref{qrEmptyClaim} to complete the proof.}\\

\noindent\chg{\textbf{Bounding the size of $F_1$.}} Choose a subspace $\Lambda \leq \mathbb{F}_p^{r_{12}}$, and let $F^\Lambda_1$ be the set of $x_{[k-2]} \in F_1$ such that this $\Lambda$ is a witness to $x_{[k-2]}$ being in $F_1$. Then
\begin{align*}|F_1^\Lambda| \frac{1}{2}\eta^{r_{12}}p^{-5r_{12}^2 - r_1 r_{12} - r_2r_{12}}|G_{k-1}| &\leq \sum_{x_{[k-2]} \in F^\Lambda_1} \Big|\Big\{y \in G_{k-1} \colon (\forall \lambda \in \Lambda)\,\, \lambda \cdot \phi(x_{[k-2]}, y) = 0\Big\}\,\setminus\, (V_{\Lambda})_{x_{[k-2]}}\Big|\\
&\leq \Big|\Big\{x_{[k-1]} \in G_{[k-1]} \colon (\forall \lambda \in \Lambda) \lambda \cdot \phi(x_{[k-1]}) = 0\Big\} \,\setminus\, V_\Lambda\Big|\\
&=\Big|\bigcup_{\mu \in \Lambda}\Big\{x_{[k-1]} \in G_{[k-1]} \colon (\forall \lambda \in \Lambda) \lambda \cdot \phi(x_{[k-1]}) = 0\Big\} \\
&\hspace{3cm}\setminus \Big\{x_{[k-1]} \in G_{[k-1]} \colon \mu \cdot A(x_{[k-1]}) = 0\Big\}\Big|\\
&\leq \Big|\bigcup_{\mu \in \Lambda} \{x_{[k-1]} \in G_{[k-1]} \colon \mu \cdot \phi(x_{[k-1]}) = 0\} \\
&\hspace{3cm}\setminus \{x_{[k-1]} \in G_{[k-1]} \colon \mu \cdot A(x_{[k-1]}) = 0\}\Big|\\
&\leq p^{2r_{12} - m} |G_{[k-1]}|,\end{align*}
from which (combined with the fact that the number of subspaces $\Lambda \leq \mathbb{F}_p^{r_{12}}$ is at most $p^{r_{12}^2}$) we deduce that $|F_1| \leq 2\eta^{-r_{12}}p^{6r_{12}^2 + r_1 r_{12} + r_2r_{12} + 2r_{12} - m}|G_{[k-2]}|$.\\

\noindent\chg{\textbf{Containing $F_1$ inside a small union of varieties.}} We now approximate $F_1$ by layers of a multiaffine map of bounded codimension. Apply Theorem~\ref{simVarAppThm} to $A_{[r_{12}]}$ to find a positive integer $s=O\Big((r_{12} + m)^{O(1)}\Big)$ and a $\mathcal{G}'$-supported multiaffine map $\theta \colon G_{[k-1]} \to \mathbb{F}_p^{s}$ such that for each $\lambda \in \mathbb{F}_p^{r_{12}}$, the layers of $\theta$ internally $(p^{-m})$-approximate the set $\{x_{[k-1]} \in G_{[k-1]} \colon \lambda \cdot A(x_{[k-1]}) = 0\}$. Hence, the layers of $\theta$ internally $(p^{r_{12}^2-m)}$-approximate $V_\Lambda$ and the layers of $\phi$ externally $p^{2 r^2_{12} - m}$-approximate $V_\Lambda$.\\
\indent Let $\xi > 0$. Apply Theorem~\ref{simFibresThm} to both $\phi$ and $\theta$. Note that $\mathcal{G}' \cap \mathcal{P}([k-2]) = \mathcal{G}''$. We obtain positive integers $t^{(1,1)}, t^{(1,2)} =O\Big((r_{12} + m + \log_{p} \xi^{-1})^{O(1)}\Big)$, $\mathcal{G}''$-supported multiaffine maps $\gamma^{(1,1)} \colon G_{[k-2]} \to \mathbb{F}_p^{t^{(1,1)}}$ and $\gamma^{(1,2)} \colon G_{[k-2]} \to \mathbb{F}_p^{t^{(1,2)}}$, a union $U^1$ of layers of $\gamma^{(1,1)}$, and a union $U^2$ of layers of $\gamma^{(1,2)}$, such that $|U^1|, |U^2|\, \geq (1-\xi)|G_{[k-2]}|$ and for every layer $L$ of $\gamma^{(1,1)}$ inside $U^1$, there is a map $c^{(1,1)} \colon \mathbb{F}_p^{[r_{12}] \times [m]} \to [0,1]$ such that
\begin{equation}\Big(\forall \lambda \in \mathbb{F}_p^{[r_{12}] \times [m]}\Big)\hspace{2pt}(\forall x_{[k-2]} \in L) \hspace{6pt}\Big||G_{k-1}|^{-1} |\{y_{k-1} \in G_{k-1} \colon \phi(x_{[k-2]}, y_{k-1}) = \lambda\}| - c^{(1,1)}(\lambda)\Big| \leq \xi\label{simRegPhiRegEqn}\end{equation}
and for every layer $L$ of $\gamma^{(1,2)}$ inside $U^2$, there is a map $c^{(1,2)} \colon \mathbb{F}_p^{s} \to [0,1]$ such that
\begin{equation}\Big(\forall \lambda \in \mathbb{F}_p^{s}\Big)\hspace{2pt}(\forall x_{[k-2]} \in L) \hspace{6pt} \Big||G_{k-1}|^{-1} |\{y_{k-1} \in G_{k-1} \colon \theta(x_{[k-2]}, y_{k-1}) = \lambda\}| - c^{(1,2)}(\lambda)\Big| \leq \xi.\label{simRegThetaRegEqn}\end{equation}
Write $\gamma^{(1)} = (\gamma^{(1,1)}, \gamma^{(1,2)})$, which is also a $\mathcal{G}''$-supported multiaffine map on $G_{[k-2]}$. Let $U = U^1 \cap U^2$, which is a union of layers of $\gamma^{(1)}$. Let $\Lambda \leq \mathbb{F}_p^{r_{12}}$. Let $\tilde{F}^\Lambda_1$ be the union of all layers $L$ of $\gamma^{(1)}$ such that $L \subseteq U$ and $L \cap F^\Lambda_1 \not= \emptyset$. We claim that $\tilde{F}^\Lambda_1$ is small (to be precise, we prove the bound~\eqref{tildef1smalleqn}). Let $L$ be any layer of $\gamma^{(1)}$ inside $\tilde{F}^\Lambda_1$. Then there is some $\tilde{x}_{[k-2]} \in L \cap F^\Lambda_1$. By definition of $F^\Lambda_1$, we have 
\[|\{y \in G_{k-1} \colon (\forall \lambda \in \Lambda)\,\,\lambda \cdot \phi(\tilde{x}_{[k-2]}, y) = 0\}| \geq \eta^{r_{12}} p^{-5r_{12}^2 - r_1 r_{12} - r_2r_{12}}|G_{k-1}|,\] 
but 
\[(V_{\Lambda})_{\tilde{x}_{[k-2]}} \subsetneq \{y \in G_{k-1} \colon (\forall \lambda \in \Lambda)\,\,\lambda \cdot \phi(\tilde{x}_{[k-2]}, y) = 0\}.\] 
Let $M_1 \subseteq (\mathbb{F}^m_p)^{r_{12}}$ and $M_2 \subseteq \mathbb{F}_p^{s}$ be sets of values such that $\bigcup_{\lambda \in M_1}\{x_{[k-1]} \in G_{[k-1]}:\phi(x_{[k-1]}) = \lambda\}$ externally $(p^{2r^2_{12}-m})$-approximates $V_\Lambda$ and $\bigcup_{\lambda \in M_2} \{x_{[k-1]} \in G_{[k-1]}\colon\theta(x_{[k-1]}) = \lambda\}$ internally $(p^{r^2_{12}-m})$-approximates $V_\Lambda$. By~\eqref{phiPropEqnOuterApp} and~\eqref{phiPropEqnOuterApp2} we may take $M_1$ to be the set of $r_{12}$-tuples of elements $\mu_1, \dots, \mu_{r_{12}} \in \mathbb{F}_p^m$ such that $\lambda_1 \mu_1 + \dots + \lambda_{r_{12}} \mu_{r_{12}} = 0$ for every $\lambda \in \Lambda$. In particular,
\begin{equation}\label{smallDiffApprox}\Big|\Big(\bigcup_{\lambda \in M_1}\{x_{[k-1]} \in G_{[k-1]}:\phi(x_{[k-1]}) = \lambda\}\Big) \setminus \Big(\bigcup_{\lambda \in M_2} \{x_{[k-1]} \in G_{[k-1]}:\theta(x_{[k-1]}) = \lambda\}\Big)\Big| \leq 2p^{2r^2_{12}-m} |G_{[k-1]}|.\end{equation}
Observe that the fibre size $|\{y_{k-1} \in G_{k-1} \colon \phi(z_{[k-2]}, y_{k-1}) = \lambda\}|$ is either 0 or a positive integer in $\{p^{-mr_{12}} |G_{k-1}|, p^{-mr_{12} + 1} |G_{k-1}|, \dots, |G_{k-1}|\}$ and an analogous property holds for the fibre size $|\{y_{k-1} \in G_{k-1} \colon \theta(x_{[k-2]}, y_{k-1}) = \lambda\}|$ (just with different negative powers of $p$). Since $L$ is a layer of $\gamma^{(1)}$ contained in $U$, provided we take $\xi = \frac{1}{4}p^{-s-r_{12}m}$, from~\eqref{simRegPhiRegEqn} and~\eqref{simRegThetaRegEqn} we conclude that the pair of sequences 
\[\bigg(\Big(|\{y_{k-1} \in G_{k-1} \colon \phi(x_{[k-2]}, y_{k-1}) = \lambda\}| \colon \lambda \in (\mathbb{F}^m_p)^{r_{12}}\Big), \Big(|\{y_{k-1} \in G_{k-1} \colon \theta(z_{[k-2]}, y_{k-1}) = \lambda\}| \colon \lambda \in \mathbb{F}_p^s\Big)\bigg)\]
is the same for all $z_{[k-2]} \in L$. Recall that we fixed some $\tilde{x}_{[k-2]} \in L \cap F^\Lambda_1$ earlier. Thus, for all $z_{[k-2]} \in L$ we have that
\begin{align*}\Big|\Big(\bigcup_{\mu \in M_1}\{x_{[k-1]} \in G_{[k-1]}:\phi(x_{[k-1]}) = \mu\}\Big)_{z_{[k-2]}}\Big| &=\Big|\Big(\bigcup_{\mu \in M_1}\{x_{[k-1]} \in G_{[k-1]}:\phi(x_{[k-1]}) = \mu\}\Big)_{\tilde{x}_{[k-2]}}\Big|\\ 
&=\Big|\Big(\bigcap_{\lambda \in \Lambda}\{x_{[k-1]}\in G_{[k-1]}:\lambda \cdot \phi(x_{[k-1]}) = 0\}\Big)_{\tilde{x}_{[k-2]}}\Big|\\ 
&\geq \eta^{r_{12}} p^{-5r_{12}^2 - r_1 r_{12} - r_2r_{12}}|G_{k-1}|,\\
\end{align*}
\[\Big|\Big(\bigcup_{\lambda \in M_2}\{x_{[k-1]}\in G_{[k-1]}:\theta(x_{[k-1]}) = \lambda\}\Big)_{z_{[k-2]}}\Big| = \Big|\Big(\bigcup_{\lambda \in M_2}\{x_{[k-1]}\in G_{[k-1]}:\theta(x_{[k-1]}) = \lambda\}\Big)_{\tilde{x}_{[k-2]}}\Big|\]
and
\[\Big(\bigcup_{\lambda \in M_2}\{x_{[k-1]}\in G_{[k-1]}:\theta(x_{[k-1]}) = \lambda\}\Big)_{z_{[k-2]}} \subseteq\Big(\bigcup_{\lambda \in M_1}\{x_{[k-1]}\in G_{[k-1]}:\phi(x_{[k-1]}) = \lambda\}\Big)_{z_{[k-2]}}.\]
 
Hence,
\begin{align*}\Big|\Big(\bigcup_{\lambda \in M_1}\{&x_{[k-1]}\in G_{[k-1]}:\phi(x_{[k-1]}) = \lambda\}\Big)_{z_{[k-2]}} \setminus \Big(\bigcup_{\lambda \in M_2} \{x_{[k-1]}\in G_{[k-1]}:\theta(x_{[k-1]}) = \lambda\}\Big)_{z_{[k-2]}}\Big|\\
&=\Big|\Big(\bigcup_{\lambda \in M_1}\{x_{[k-1]}\in G_{[k-1]}:\phi(x_{[k-1]}) = \lambda\}\Big)_{\tilde{x}_{[k-2]}} \setminus \Big(\bigcup_{\lambda \in M_2} \{x_{[k-1]}\in G_{[k-1]}:\theta(x_{[k-1]}) = \lambda\}\Big)_{\tilde{x}_{[k-2]}}\Big|\\
&\geq \Big|\Big(\bigcap_{\lambda \in \Lambda}\{x_{[k-1]}\in G_{[k-1]}:  \lambda \cdot \phi(x_{[k-1]}) = 0\}\Big)_{\tilde{x}_{[k-2]}} \setminus (V_{\Lambda})_{\tilde{x}_{[k-2]}}\Big|\hspace{2cm}\text{(since $\tilde{x}_{[k-2]} \in F^\Lambda_1$)}\\
&\geq\frac{1}{2}\eta^{r_{12}} p^{-5r_{12}^2 - r_1 r_{12} - r_2r_{12}}|G_{k-1}|.\end{align*}
By~\eqref{smallDiffApprox}, we see that the total size of all layers $L$ of $\gamma^{(1)}$ contained in $U$ that intersect $F^\Lambda_1$ is
\begin{equation}\label{tildef1smalleqn}|\tilde{F}^\Lambda_1| \leq 4\eta^{-r_{12}}p^{7r_{12}^2 + r_1 r_{12} + r_2r_{12}-m}|G_{[k-2]}|.\end{equation}
Hence, $F_1$ is contained in $\tilde{F}_1 = (\bigcup_{\Lambda \leq \mathbb{F}_p^{r_{12}}} \tilde{F}^\Lambda_1) \cup (G_{[k-2]} \setminus U)$, which is still a union of layers of $\gamma^{(1)}$, and (using the fact that the number of subspaces $\Lambda \leq \mathbb{F}^{r_{12}}_p$ is at most $p^{r_{12}^2}$)
\[|\tilde{F}_1| \leq 6\eta^{-r_{12}}p^{9r_{12}^2 + r_1 r_{12} + r_2r_{12} - m}|G_{[k-2]}|.\]

\noindent\chg{\textbf{Controlling further vanishing linear combinations.}} Since $\phi_{ij}$ is linear in coordinate $k-1$ for all $i \in [r_{12}], j \in [m]$, we may find a $\mathcal{G}''$-supported multiaffine map $\Phi_{ij} \colon G_{[k-2]} \to G_{k-1}$ such that $\phi_{ij}(x_{[k-1]}) = \Phi_{ij}(x_{[k-2]}) \cdot x_{k-1}$. Notice that the set
\[\{\mu \in \mathbb{F}_p^{[r_{12}] \times [m]} \colon \mu \cdot \Phi(x_{[k-2]}) = 0\}\]
determines the sizes $|\{y \in G_{k-1} \colon (\forall \lambda \in \Lambda) \lambda \cdot \phi(x_{[k-2]}, y) = 0\}|$ for all $\Lambda \leq \mathbb{F}_p^{r_{12}}$.\\

Apply Theorem~\ref{simVarAppThm} to $\Phi_{[r_{12}] \times [m]}$ to find a $\mathcal{G}''$-supported multiaffine map $\gamma^{(2)} \colon G_{[k-2]} \to \mathbb{F}_p^t$, for $t \leq O((r_{12} + m)^{O(1)})$ such that the layers of $\gamma^{(2)}$ internally and externally $(p^{-m - r_{12}m - r_{12}^2 m^2})$-approximate the sets $\{x_{[k-2]} \in G_{[k-2]} \colon \mu \cdot \Phi(x_{[k-2]}) = 0\}$ for each $\mu \in \mathbb{F}_p^{[r_{12}] \times [m]}$. Thus, the layers of $\gamma^{(2)}$ internally and externally $(p^{-m - r_{12}^2 m^2})$-approximate the sets
\begin{align*}&\Big\{x_{[k-2]} \in G_{[k-2]} \colon \{\mu \in \mathbb{F}_p^{[r_{12}] \times [m]} \colon \mu \cdot \Phi(x_{[k-2]}) = 0\} = M\Big\}\\
&\hspace{0.5cm}=\Big\{x_{[k-2]} \in G_{[k-2]} \colon (\forall \mu \in M) \mu \cdot \Phi(x_{[k-2]}) = 0\Big\} \setminus \bigg(\bigcup_{\lambda \in \mathbb{F}_p^{[r_{12}] \times [m]} \setminus M}\Big\{x_{[k-2]} \in G_{[k-2]} \colon \lambda \cdot \Phi(x_{[k-2]}) = 0\Big\}\bigg)\end{align*}
for each subspace $M \leq \mathbb{F}_p^{[r_{12}] \times [m]}$. These sets partition $G_{[k-2]}$. Let $F_2$ be the union of layers of $\gamma^{(2)}$ that are not fully contained in a set of this form. Then $|F_2| \leq p^{-m}|G_{[k-2]}|$.\\

\noindent\chg{\textbf{Defining $F$ and completing the proof.}} Let $F = \tilde{F}_1 \cup F_2$, which is a union of layers of the map $\gamma = (\gamma^{(1)}, \gamma^{(2)})$, which is a $\mathcal{G}''$-supported multiaffine map on $G_{[k-2]}$. Let $L$ be a layer of $\gamma$ such that $L \not\subseteq F$. In particular, $L \not\subseteq F_2$, so $L$ is fully contained in the set 
\[\Big\{x_{[k-2]} \in G_{[k-2]} \colon \{\mu \in \mathbb{F}_p^{[r_{12}] \times [m]} \colon \mu \cdot \Phi(x_{[k-2]}) = 0\} = M\Big\}\]
for some subspace $M \leq \mathbb{F}_p^{[r_{12}] \times [m]}$. This implies that the sequence
\[\Big(|\{y \in G_{k-1} \colon (\forall \lambda \in \Lambda) \,\,\lambda \cdot \phi(x_{[k-2]}, y) = 0\}|\colon \Lambda \leq \mathbb{F}_p^{r_{12}}\Big)\]
is the same for all $x_{[k-2]} \in L$. Hence, there exists $\Lambda \leq \mathbb{F}_p^{r_{12}}$ such that for each $x_{[k-2]} \in L$,
\[|\{y \in G_{k-1} \colon (\forall \lambda \in \Lambda)\,\,\lambda \cdot \phi(x_{[k-2]}, y) = 0\}|\ \geq\, \eta^{r_{12}} p^{-5r^2_{12}-r_{12}r_1-r_{12}r_2}|G_{k-1}|\]
and for each subspace $\Lambda' \gneq \Lambda$
\[\eta p^{-5r_{12}-r_1-r_2}|\{y \in G_{k-1} \colon (\forall \lambda \in \Lambda)\,\, \lambda \cdot \phi(x_{[k-2]}, y) = 0\}|\ > |\{y \in G_{k-1} \colon (\forall \lambda \in \Lambda')\,\, \lambda \cdot \phi(x_{[k-2]}, y) = 0\}|.\]
By the way we defined $\tilde{F}_1$, it must be that $L \subset U$ and $L \cap F_1^\Lambda = \emptyset$. From the definition of $F_1^\Lambda$, we deduce that for each $x_{[k-2]} \in L$,
\[(V_\Lambda)_{x_{[k-2]}} = \{y \in G_{k-1} \colon (\forall \lambda \in \Lambda)\,\, \lambda \cdot \phi(x_{[k-2]}, y) = 0\}.\]
On the other hand, from~\eqref{phiPropEqnOuterApp}, for each $x_{[k-2]} \in L$ we also have
\[|(V_{\Lambda'})_{x_{[k-2]}}|\ \leq\, |\{y \in G_{k-1} \colon (\forall \lambda \in \Lambda')\,\, \lambda \cdot \phi(x_{[k-2]}, y) = 0\}|\]
for all $\Lambda' \gneq \Lambda$. Therefore Lemma~\ref{qrEmptyClaim} applies, and finally, for each choice of $x_{[k-2]} \in L$, $u_0 \in G_{k-1}$, $v_0 \in G_k$ and $\tau \in \mathbb{F}_p^{r_{12}}$, the biaffine variety
\begin{align*}&\bigg[\bigg(u_0 + \Big(\{y \in G_{k-1} \colon (\forall \lambda \in \Lambda) \lambda \cdot \phi(x_{[k-2]}, y) = 0\} \cap \{y \in G_{k-1} \colon \beta^1(x_{[k-2]}, y) = \beta^1(x_{[k-2]}, 0)\}\\
&\hspace{4cm}\cap \{y \in G_{k-1} \colon (\forall i \in [r_{12}])\alpha_i(x_{[k-2]}, y) = 0\}\Big)\bigg) \\
&\hspace{2cm}\times \bigg(v_0 + \Big(\langle A'_1(x_{[k-2]}), \dots, A'_{r_{12}}(x_{[k-2]})\rangle^\perp \cap \langle A_1(x_{[k-2]}, u_0), \dots, A_{r_{12}}(x_{[k-2]}, u_0)\rangle^\perp\\
&\hspace{4cm}\cap \{z \in G_{k} \colon \beta^2(x_{[k-2]},\ls{k} z) = \beta^2(x_{[k-2]},\ls{k} 0)\}\Big)\bigg)\bigg]\\
&\hspace{1cm} \cap \{(y,z) \in G_{k-1} \times G_k \colon \beta^{12}(x_{[k-2]}, y,z) = \tau\}\end{align*}
is either $\eta$-quasirandom with density $p^{-r_{12}}|\Lambda|$ or empty. We may choose $m = O\Big((r_{12} + r_1 + r_2 + \log_{p} \eta^{-1})^{O(1)}\Big)$ and ensure that $|F|\ \leq \eta |G_{[k-2]}|$.\end{proof}

\subsection{Convolutional extensions of multi-homomorphisms}

Recall from the introduction that when $\alpha \colon G_{[k]} \to \mathbb{F}^r$ is a multiaffine map such that each component $\alpha_i$ is multilinear on $G_{I_i}$, where $I_i$ is the set of coordinates on which $\alpha_i$ depends, we say that $\alpha$ is mixed-linear. Further, we say that a variety is mixed-linear if it is a layer of a mixed-linear map.\\
\indent The next theorem tells us that if we are given a multi-$D$-homomorphism defined on a subset $S$ of a variety $V$ and a set $X \subset G_{[k] \setminus \{d\}}$ of $(1-\varepsilon)$-dense columns for sufficiently small $\varepsilon > 0$ (that is $X = \Big\{x_{[k] \setminus \{d\}} \in G_{[k] \setminus \{d\}} \colon |S_{x_{[k] \setminus \{d\}}}| \geq (1-\varepsilon) |V_{x_{[k] \setminus \{d\}}}|\Big\}$), then convolving in direction $d$ turns $\phi$ into multi-$D'$-homomorphism $\psi$ for some $D'$ slightly smaller than $D$. The gain in the theorem is that the domain of $\psi$ is given by `filling' the $(1-\varepsilon)$-dense columns from $S_{x_{[k] \setminus \{d\}}}$ to $V_{x_{[k] \setminus \{d\}}}$, with a minor price that we need to remove a very small portion of columns and that we need to intersect $V$ with a further lower order variety of bounded codimension. Crucially, we do not make any additional assumptions on $X$ such as being $1-o(1)$ dense in a variety of bounded codimension. 

\begin{theorem}\label{mainExtensionConvStep}Let $D, k,  \in \mathbb{N}$. Then there exists $\varepsilon_0 = \varepsilon_0(D,k) > 0$, depending only on $D$ and $k$, with the following property. Suppose that $S \subset V \subset G_{[k]}$ and that $V$ is a mixed-linear variety of codimension at most $r$. Let $\mathcal{G}$ be a down-set such that $V$ is $\mathcal{G}$-supported. Write $\mathcal{G}' = \{J \subset [k-1] \colon J \cup \{k\} \in \mathcal{G}\}$. Let $X \subset G_{[k-1]}$ be a set such that $|S_{x_{[k-1]}}|\ \geq (1-\varepsilon_0) |V_{x_{[k-1]}}|\ > 0$ for every $x_{[k-1]} \in X$,
and suppose that $\phi \colon S \to H$ is a multi-$(20 D)$-homomorphism. Let $\xi > 0$. Then, provided $\dim G_i \geq \con_{D, p} \Big(r + \log_p \xi^{-1}\Big)^{\con_{D, p}}$ for each $i \in [k-1]$, there exist a positive integer $s =O_{D, p}\Big((r + \log_{p}\xi^{-1})^{O_{D, p}(1)}\Big)$, a $\mathcal{G}'$-supported multiaffine map $\gamma \colon G_{[k-1]} \to \mathbb{F}_p^s$, a collection of values $\Gamma \subset \mathbb{F}_p^s$, and a set $X' \subset X$, such that
\begin{itemize}
\item[\textbf{(i)}] $|X \setminus X'| \leq \xi |G_{[k-1]}|$,
\item[\textbf{(ii)}] $|\{x_{[k-1]} \in G_{[k-1]} \colon \gamma(x_{[k-1]}) \in \Gamma\}|\ \geq (1-\xi)|G_{[k-1]}|$,
\item[\textbf{(iii)}] for each $\lambda \in \Gamma$, the map $\phi^{\text{ext}} \colon ((X' \cap \gamma^{-1}(\lambda)) \times G_k) \cap V \to H$ given by the formula
\[\phi^{\text{ext}}(x_{[k-1]},y_1 + y_2 - y_3) = \phi(x_{[k-1]},y_1) + \phi(x_{[k-1]},y_2) - \phi(x_{[k-1]},y_3),\]
for $x_{[k-1]} \in X' \cap \gamma^{-1}(\lambda)$ and $y_1, y_2, y_3 \in S_{x_{[k-1]}}$, is well-defined, has the domain claimed, and is a multi-$D$-homomorphism.
\end{itemize}
\end{theorem}

\noindent\chg{\textbf{Remark.} The assumption that $\phi$ is a multi-$(20 D)$-homomorphism can be weakened to being a multi-$(12 D)$-homomorphism, but we find this constant slightly neater for later computations.}

\begin{proof}\chg{The main tool for obtaining the desired convolutional extension is Theorem~\ref{biaffineExtnConv}, which provides us with a suitable map when we are working inside a quasirandom biaffine variety. Thus, in order to use that theorem, we need to apply the simultaneous biaffine regularity lemma, which is Theorem~\ref{simReg}. Finally, after Theorem~\ref{biaffineExtnConv} is applied, because of the introduction of various multiaffine varieties during the proof, we need to ensure that the final set of columns has the structure described in the conclusion.}\\

\noindent\chg{\textbf{Preparation and application of Theorem~\ref{simReg}.}} Fix any $d \in [k - 1]$. Since $V$ is a $\mathcal{G}$-supported mixed-linear variety of codimension $r$, we can find multiaffine maps $\beta^1 \colon G_{[k-1]} \to \mathbb{F}_p^{r_1}$, $\beta^2 \colon G_{[k]\setminus \{d\}} \to \mathbb{F}_p^{r_2}$ and $\beta^{12} \colon G_{[k]} \to \mathbb{F}_p^{r_{12}}$ such that 
\begin{itemize}
\item $r_1 + r_2 + r_{12} \leq r$, 
\item there are values $\lambda^1 \in \mathbb{F}_p^{r_1}, \lambda^2 \in \mathbb{F}_p^{r_2}, \lambda^{12} \in \mathbb{F}_p^{r_{12}}$ for which 
\[V = \{x_{[k]} \in G_{[k]} \colon \beta^1(x_{[k-1]}) = \lambda^1, \beta^2(x_{[k] \setminus \{d\}}) = \lambda^2, \beta^{12}(x_{[k]}) = \lambda^{12}\},\] 
\item $\beta^{12}$ is linear in coordinates $G_d$ and $G_k$, and 
\item $\beta^{12}$ is $\{J \subset [k] \colon J \cup \{d,k\} \in \mathcal{G}\}$-supported.
\end{itemize}

Write $\beta^{12}_i(x_{[k]}) = A_i(x_{[k-1]}) \cdot x_k$, for multiaffine maps $A_i \colon G_{[k-1]} \to G_k$, where the $A_i$ are additionally linear in coordinate $d$ and $\mathcal{G}'$-supported.\\
\indent Let $\eta > 0$ be a constant to be chosen later. Apply Theorem~\ref{simReg} with $d_1 = d$ and $d_2 = k$ to $\beta^1, \beta^2$ and $\beta^{12}$. We obtain
\begin{itemize}
\item positive integers $m, s = O\Big((r + \log_{p}\eta^{-1})^{O(1)}\Big)$,
\item $\mathcal{G}'$-supported\footnote{This $\mathcal{G}'$ is the same as in the statement of this theorem, not the one given by Theorem~\ref{simReg}. In fact the former collections of sets contains the latter.} multiaffine maps $\psi^{(d)}_1, \dots, \psi^{(d)}_{r_{12}} \colon G_{[k-1]} \to \mathbb{F}_p^m$, which are linear in coordinate $d$,
\item  $\mathcal{G}''$-supported multiaffine maps $\gamma^{(d)} \colon G_{[k-1] \setminus \{d\}} \to \mathbb{F}_p^s$, where 
\[\mathcal{G}'' = \{J \subset [k-1] \setminus \{d\} \colon J \cup \{d, k\} \in \mathcal{G}\},\]
\item a collection of values $\Gamma^{(d)} \subset \mathbb{F}^s$ such that 
\[|\{x_{[k-1] \setminus \{d\}} \in G_{[k-1] \setminus \{d\}} \colon \gamma^{(d)}(x_{[k-1] \setminus \{d\}}) \in \Gamma^{(d)}\}| \geq (1-\eta) |G_{[k-1] \setminus \{d\}}|,\]
\end{itemize}
with the property that for each $\mu \in \Gamma^{(d)}$, there is a subspace $\Lambda^\mu \leq \mathbb{F}^{r_{12}}$ such that for each $x_{[k-1] \setminus \{d\}} \in (\gamma^{(d)})^{-1}(\mu)$, $u_0 \in G_{d}, v_0 \in G_{k}$ and $\tau \in \mathbb{F}_p^{r_{12}}$, the biaffine variety
\begin{align*}&\bigg[\bigg(u_0 + \Big(\{y \in G_{d} \colon (\forall \lambda \in \Lambda^\mu) \,\,\lambda \cdot \psi^{(d)}(x_{[k-1] \setminus \{d\}},\ls{d} y) = 0\} \\
&\hspace{3cm}\cap \{y \in G_{d} \colon \beta^1(x_{[k-1] \setminus \{d\}},\ls{d} y) =\beta^1(x_{[k-1] \setminus \{d\}},\ls{d} 0)\}\Big)\bigg)\\
&\hspace{2cm}\times \bigg(v_0 + \Big(\langle A_1(x_{[k-1] \setminus \{d\}},\ls{d} u_0), \dots, A_{r_{12}}(x_{[k-1] \setminus \{d\}},\ls{d} u_0)\rangle^\perp\\
&\hspace{4cm}\cap \{z \in G_{k} \colon \beta^2(x_{[k-1] \setminus \{d\}}, \ls k z) = \beta^2(x_{[k-1] \setminus \{d\}}, \ls k 0)\}\Big)\bigg)\bigg]\\
&\hspace{1cm} \cap \{(y,z) \in G_{d}\times G_{k} \colon \beta^{12}(x_{[k-1] \setminus \{d\}}, y, z) = \tau\}\end{align*}
is either $\eta$-quasirandom with density $|\Lambda^\mu|p^{-r_{12}}$ or empty.\\

\noindent\chg{\textbf{Applying Theorem~\ref{biaffineExtnConv}.} Our next step is to apply Theorem~\ref{biaffineExtnConv}. However, because of the introduction of additional multiaffine forms $\psi^{(d)}_i$, we need to modify the sets $X$ and $S$ slightly. Theorem~\ref{biaffineExtnConv} will eventually be applied to closely related sets defined in~\eqref{biaffineconvapplsets}, for the biaffine variety in~\eqref{biaffineconvapplvar}.}\\
\indent Let $x_{[k-1] \setminus \{d\}} \in (\gamma^{(d)})^{-1}(\Gamma^{(d)})$ be such that $X_{x_{[k-1] \setminus \{d\}}} \not=\emptyset$. This implies $V_{x_{[k-1] \setminus \{d\}}} \not=\emptyset$, and in particular we have that $\Big\{y \in G_{d} \colon \beta^1(x_{[k-1] \setminus \{d\}},\ls{d} y) = \lambda^1\Big\}$ and $\Big\{z \in G_{k} \colon \beta^2(x_{[k-1] \setminus \{d\}}, \ls k z) = \lambda^2\Big\}$ are non-empty cosets in $G_{d}$ and $G_k$ respectively. Write $\mu = \gamma^{(d)}(x_{[k-1] \setminus \{d\}}) \in \Gamma^{(d)}$ and define
\[U_{x_{[k-1] \setminus \{d\}}} = \Big\{y \in G_{d} \colon (\forall \lambda \in \Lambda^\mu)\,\, \lambda \cdot \psi^{(d)}(x_{[k-1] \setminus \{d\}},\ls{d} y) = 0\Big\} \cap \Big\{y \in G_{d} \colon \beta^1(x_{[k-1] \setminus \{d\}},\ls{d} y) =\beta^1(x_{[k-1] \setminus \{d\}},\ls{d} 0)\Big\}.\]
Take a coset $W_{x_{[k-1] \setminus \{d\}}}$ inside $G_d$ such that
\begin{equation}\label{wcosetintroeqn}U_{x_{[k-1] \setminus \{d\}}}  \oplus W_{x_{[k-1] \setminus \{d\}}} = \Big\{y \in G_{d} \colon \beta^1(x_{[k-1] \setminus \{d\}},\ls{d} y) = \lambda^1\Big\}.\end{equation}
For each $u_0\in W_{x_{[k-1] \setminus \{d\}}}$, write 
\begin{align*}Y_{x_{[k-1] \setminus \{d\}}, u_0} = &\Big\langle A_1(x_{[k-1] \setminus \{d\}},\ls{d} u_0), \dots, A_{r_{12}}(x_{[k-1] \setminus \{d\}},\ls{d} u_0)\Big\rangle^\perp\\
&\hspace{2cm} \cap \Big\{z \in G_{k} \colon \beta^2(x_{[k-1] \setminus \{d\}}, \ls k z) = \beta^2(x_{[k-1] \setminus \{d\}}, \ls k 0)\Big\}.\end{align*}
Let $T_{x_{[k-1] \setminus \{d\}}, u_0}$ be a coset in $G_k$ such that
\[Y_{x_{[k-1] \setminus \{d\}}, u_0} \oplus T_{x_{[k-1] \setminus \{d\}}, u_0} = \{z \in G_{k} \colon \beta^2(x_{[k-1] \setminus \{d\}}, \ls k z) = \lambda^2\},\]
and let $C_{x_{[k-1] \setminus \{d\}}, u_0}$ be the set of all $v_0 \in T_{x_{[k-1] \setminus \{d\}}, u_0}$ such that 
\[B^{v_0} = \Big((u_0 + U_{x_{[k-1] \setminus \{d\}}}) \times (v_0 + Y_{x_{[k-1] \setminus \{d\}}, u_0})\Big) \cap \{(y,z) \in G_{d}\times G_k \colon \beta^{12}(x_{[k-1] \setminus \{d\}}, y, z) = \lambda^{12}\}\]
is non-empty, and therefore $\eta$-quasirandom with density $|\Lambda^\mu|p^{-r_{12}}$. By quasirandomness, we see that for a proportion at least $1 - p^{r_2 + r_{12}}\eta$ of $y \in u_0 + U_{x_{[k-1] \setminus \{d\}}}$ we have that $B^{v_0}_{y \bullet}$ is non-empty for every $v_0 \in C_{x_{[k-1] \setminus \{d\}}, u_0}$. On the other hand, for a fixed $y$, the set of all $v_0 \in T_{x_{[k-1] \setminus \{d\}}, u_0}$ such that $B^{v_0}_{y \bullet}$ is non-empty is a coset. Thus $C_{x_{[k-1] \setminus \{d\}}, u_0}$ is a coset in $G_k$, and by Lemma~\ref{simultQuasiRand},
\[\Big((u_0 + U_{x_{[k-1] \setminus \{d\}}}) \times (C_{x_{[k-1] \setminus \{d\}}, u_0} + Y_{x_{[k-1] \setminus \{d\}}, u_0})\Big) \cap \{(y,z) \in G_{d}\times G_k \colon \beta^{12}(x_{[k-1] \setminus \{d\}}, y, z) = \lambda^{12}\}\]
is $\eta p^{r}$-quasirandom with density $|\Lambda^\mu|p^{-r_{12}}$.\\

We shall apply Theorem~\ref{biaffineExtnConv} to sets
\begin{equation}\tilde{X}_{u_0, x_{[k-1] \setminus \{d\}}} = X_{x_{[k-1] \setminus \{d\}}} \cap\hspace{3pt} (u_0 + U_{x_{[k-1] \setminus \{d\}}}),\hspace{6pt} \tilde{S}_{u_0, x_{[k-1] \setminus \{d\}}} = S_{x_{[k-1] \setminus \{d\}}} \cap\hspace{3pt} \Big((u_0 + U_{x_{[k-1] \setminus \{d\}}})\times G_k\Big)\label{biaffineconvapplsets}\end{equation}
and the variety
\begin{align}\tilde{V}_{u_0, x_{[k-1] \setminus \{d\}}} = & \Big((u_0 + U_{x_{[k-1] \setminus \{d\}}}) \times (C_{x_{[k-1] \setminus \{d\}}, u_0} + Y_{x_{[k-1] \setminus \{d\}}, u_0})\Big) \nonumber\\
&\hspace{4cm}\cap \{(y,z) \in G_{d}\times G_k \colon \beta^{12}(x_{[k-1] \setminus \{d\}}, y, z) = \lambda^{12}\}\nonumber\\
=&\Big((u_0 + U_{x_{[k-1] \setminus \{d\}}}) \times \{z \in G_{k} \colon \beta^2(x_{[k-1] \setminus \{d\}}, \ls k z) = \lambda^2\}\Big)\nonumber\\
&\hspace{4cm}\cap \{(y,z) \in G_{d}\times G_k \colon \beta^{12}(x_{[k-1] \setminus \{d\}}, y, z) = \lambda^{12}\}\nonumber\\
=&\Big((u_0 + U_{x_{[k-1] \setminus \{d\}}}) \times G_k \Big) \,\cap\, V_{x_{[k-1] \setminus \{d\}}}.\label{biaffineconvapplvar}\end{align}
We now briefly check that the assumptions in Theorem~\ref{biaffineExtnConv} are satisfied. We alredy know that $\tilde{V}_{u_0, x_{[k-1] \setminus \{d\}}}$ is $\eta p^{r}$-quasirandom with density $|\Lambda^\mu|p^{-r_{12}}$. When $y \in u_0 + U_{x_{[k-1] \setminus \{d\}}}$ then $\beta^1(x_{[k-1] \setminus \{d\}},\ls{d} y) = \lambda^1$, so we have $(\tilde{V}_{u_0, x_{[k-1] \setminus \{d\}}})_{y\bullet} = V_{x_{[k-1] \setminus \{d\}}, \ls{d} y}$. Hence, for each $y \in \tilde{X}_{u_0, x_{[k-1] \setminus \{d\}}}$ we have 
\[\Big|\Big(\tilde{S}_{u_0, x_{[k-1] \setminus \{d\}}}\Big)_{y \bullet}\Big| \geq (1 - \varepsilon_0) \Big|\Big(\tilde{V}_{u_0, x_{[k-1] \setminus \{d\}}}\Big)_{y\bullet}\Big|.\]
Note that we need $\varepsilon_0(D)$ in the statement of this theorem to be at most $\varepsilon_0(20D)$ from Theorem~\ref{biaffineExtnConv}.\\

\indent Lastly, provided that $|G_d| \geq (p^r\eta)^{- \con_{D, p}}\, p^{rm}$, which is sufficient to guarantee the technical condition $|U_{x_{[k-1] \setminus \{d\}}}| \geq (p^r\eta)^{- \con_{D, p}}$ required by Theorem~\ref{biaffineExtnConv}, all the assumptions in that theorem hold, so we obtain a subset $X^{(d, u_0)}_{x_{[k-1] \setminus \{d\}}} \subset X_{x_{[k-1] \setminus \{d\}}} \cap (u_0 + U_{x_{[k-1] \setminus \{d\}}})$ such that
\begin{equation}|(X_{x_{[k-1] \setminus \{d\}}} \cap (u_0 + U_{x_{[k-1] \setminus \{d\}}})) \setminus X^{(d, u_0)}_{x_{[k-1] \setminus \{d\}}}|\, =\, O_{D, p}(p^{O_{D, p}(r)} \eta^{\Omega_{D, p}(1)}) |U_{x_{[k-1] \setminus \{d\}}}|\label{2dextnsetdiff}\end{equation}
and $X^{(d, u_0)}_{x_{[k-1] \setminus \{d\}}}$ has the property regarding extensions of maps described in that theorem; the extended domain is
\begin{equation}\label{extendedDomainEqn} (X^{(d, u_0)}_{x_{[k-1] \setminus \{d\}}} \times G_k) \cap \tilde{V}_{u_0, x_{[k-1] \setminus \{d\}}} = (X^{(d, u_0)}_{x_{[k-1] \setminus \{d\}}} \times G_k) \cap V_{x_{[k-1] \setminus \{d\}}}.\end{equation}

\vspace{\baselineskip}

\noindent\chg{\textbf{Ensuring that the set of full columns has the desired structure.}} Recall that we only considered $x_{[k-1] \setminus \{d\}}$ such that $X_{x_{[k-1] \setminus \{d\}}} \not= \emptyset$. The assumptions in the statement of this theorem imply that $V_{x_{[k-1]\setminus\{d\}}, y_d} \not=\emptyset$ for each $y_d \in X_{x_{[k-1] \setminus \{d\}}}$. Hence $\beta^1(x_{[k-1]\setminus\{d\}}, y_d) = \lambda^1$, so $X_{x_{[k-1] \setminus \{d\}}} \subset \{y \in G_d \colon \beta^1(x_{[k-1] \setminus \{d\}},\ls{d} y) = \lambda^1\}$. Recall from~\eqref{wcosetintroeqn} that the latter set is equal to $U_{x_{[k-1] \setminus \{d\}}}  \oplus W_{x_{[k-1] \setminus \{d\}}}$. Hence
\begin{align}\Big|X_{x_{[k-1] \setminus \{d\}}} \,\setminus\,\Big(\bigcup_{u_0 \in W_{x_{[k-1] \setminus \{d\}}}} X^{(d, u_0)}_{x_{[k-1] \setminus \{d\}}}\Big)\Big|\,\, =&\,\, \sum_{u_0 \in W_{x_{[k-1] \setminus \{d\}}}} \Big|\Big(X_{x_{[k-1] \setminus \{d\}}} \cap  (u_0 + U_{x_{[k-1] \setminus \{d\}}})\Big)\,\setminus\, X^{(d, u_0)}_{x_{[k-1] \setminus \{d\}}}\Big| \nonumber\\
\leq&\,\, O_{D, p}(p^{O_{D, p}(r)} \eta^{\Omega_{D, p}(1)}) |W_{x_{[k-1] \setminus \{d\}}}||U_{x_{[k-1] \setminus \{d\}}}|\nonumber\\
&\hspace{6cm}\text{(by~\eqref{2dextnsetdiff})}\nonumber\\
=&\,\,O_{D, p}(p^{O_{D, p}(r)} \eta^{\Omega_{D, p}(1)}) |W_{x_{[k-1] \setminus \{d\}}} + U_{x_{[k-1] \setminus \{d\}}}|\nonumber\\
=&\,\,O_{D, p}(p^{O_{D, p}(r)} \eta^{\Omega_{D, p}(1)}) \Big|\Big\{y \in G_d \colon \beta^1(x_{[k-1] \setminus \{d\}},\ls{d} y) = \lambda^1\Big\}\Big|.\label{xandxdfiberdiffeqn}\end{align}

Let
\begin{equation}X^{(d)} = \bigcup_{\substack{x_{[k-1] \setminus \{d\}} \in G_{[k-1] \setminus \{d\}}\\\text{such that}\\X_{x_{[k-1] \setminus \{d\}}} \not=\emptyset\\\gamma^{(d)}(x_{[k-1] \setminus \{d\}}) \in \Gamma^{(d)}}} \{x_{[k-1] \setminus \{d\}}\} \times \Big(\bigcup_{u_0 \in W_{x_{[k-1] \setminus \{d\}}}} X^{(d, u_0)}_{x_{[k-1] \setminus \{d\}}}\Big).\label{xddefineqnextnthm}\end{equation}
\chg{(Our final set of columns will be $X'$, defined as the intersection $\bigcap_{d \in [k-1]} X^{(d)}$.)}\\
\indent Then $X^{(d)} \subset X$ and
\begin{align*}|X \setminus X^{(d)}| \,\,\leq &\Big|\Big(X \cap ((\gamma^{(d)})^{-1}(\Gamma^{(d)}) \times G_d)\Big) \setminus X^{(d)}\Big| + \Big|X \setminus ((\gamma^{(d)})^{-1}(\Gamma^{(d)}) \times G_d)\Big|\\
\leq &\sum_{\substack{x_{[k-1] \setminus \{d\}} \in G_{[k-1] \setminus \{d\}}\\\text{such that}\\X_{x_{[k-1] \setminus \{d\}}} \not=\emptyset\\\gamma^{(d)}(x_{[k-1] \setminus \{d\}}) \in \Gamma^{(d)}}} \Big|X_{x_{[k-1] \setminus \{d\}}} \,\setminus\,\Big(\bigcup_{u_0 \in W_{x_{[k-1] \setminus \{d\}}}} X^{(d, u_0)}_{x_{[k-1] \setminus \{d\}}}\Big)\Big| \,\, + \,\, \Big|G_{[k-1]} \setminus ((\gamma^{(d)})^{-1}(\Gamma^{(d)}) \times G_d)\Big|\\
&\hspace{10cm}\text{(by~\eqref{xddefineqnextnthm})}\\
\leq & O_{D, p}(p^{O_{D, p}(r)} \eta^{\Omega_{D,p}(1)})|G_{[k-1]}| \,\, + \eta |G_{[k-1]}|\hspace{4cm}\text{(by~\eqref{xandxdfiberdiffeqn})}\\
\leq & O_{D, p}(p^{O_{D, p}(r)} \eta^{\Omega_{D,p}(1)})|G_{[k-1]}|.\end{align*}

For each $\nu \in \Gamma^{(d)}$ and $\nu' \in (\mathbb{F}_p^{r_{12}})^m$ define the set $\tilde{X}^{(d)}_{\nu, \nu'}$ as
\[\tilde{X}^{(d)}_{\nu, \nu'} = \bigg[\bigg(X^{(d)} \,\cap\, \Big((\gamma^{(d)})^{-1}(\nu)\times G_d\Big) \,\cap\, \bigcap_{\lambda \in \Lambda^\nu} (\lambda \cdot \psi^{(d)})^{-1}(\nu')\bigg) \times G_k\bigg] \cap V.\]
Such a set is of the form $(X^{(d)} \times G_k) \cap V$ intersected with a layer of a lower order multiaffine map, which is the structure required by part \textbf{(iii)} of the conclusion of the theorem. We claim that if $x_{[k-1] \setminus \{d\}} \in G_{[k-1] \setminus \{d\}}$ is such that $(\tilde{X}^{(d)}_{\nu, \nu'})_{x_{[k-1] \setminus \{d\}}} \not= \emptyset$, then 
\begin{equation}\label{tildexsectioneqn}(\tilde{X}^{(d)}_{\nu, \nu'})_{x_{[k-1] \setminus \{d\}}} = ((X^{(d, u_0)}_{x_{[k-1] \setminus \{d\}}}) \times G_k) \cap V_{x_{[k-1] \setminus \{d\}}}\end{equation}
for a suitable $u_0 \in W_{x_{[k-1] \setminus \{d\}}}$. This coincides with~\eqref{extendedDomainEqn}, which is the domain of the extended map given by Theorem~\ref{biaffineExtnConv}. To prove this claim, note that $(\tilde{X}^{(d)}_{\nu, \nu'})_{x_{[k-1] \setminus \{d\}}} \not= \emptyset$ implies in particular that $\gamma^{(d)}(x_{[k-1] \setminus \{d\}}) = \nu \in \Gamma^{(d)}$ and $\emptyset \not= X^{(d)}_{x_{[k-1] \setminus \{d\}}} \subset X_{x_{[k-1] \setminus \{d\}}}$, so by~\eqref{xddefineqnextnthm} we have
\begin{align}(\tilde{X}^{(d)}_{\nu, \nu'})_{x_{[k-1] \setminus \{d\}}} = \bigg[\bigg(\Big(\bigcup_{u_0 \in W_{x_{[k-1] \setminus \{d\}}}} X^{(d, u_0)}_{x_{[k-1] \setminus \{d\}}}\Big) \,\cap\, \Big\{y_d \in G_d \colon (\forall \lambda \in \Lambda^\nu)\,\,\lambda &\cdot \psi^{(d)}(x_{[k-1] \setminus \{d\}}, \ls{d} y_d) = \nu' \Big\}\bigg) \times G_k\bigg]\nonumber\\
&\cap\, V_{x_{[k-1] \setminus \{d\}}}.\label{tildexdnunu}\end{align}
But $V_{x_{[k-1] \setminus \{d\}}} \subset \Big\{y_d \in G_d \colon \beta^1(x_{[k-1] \setminus \{d\}}, \ls{d} y_d) = \lambda^1\Big\} \times G_k$, so we have that 
\begin{align*}(\tilde{X}^{(d)}_{\nu, \nu'})_{x_{[k-1] \setminus \{d\}}} \subset &\bigg(\Big(\bigcup_{u_0 \in W_{x_{[k-1] \setminus \{d\}}}} X^{(d, u_0)}_{x_{[k-1] \setminus \{d\}}}\Big) \times G_k\bigg)\\
&\cap \,\, \bigg(\Big\{y_d \in G_d \colon \beta^1(x_{[k-1] \setminus \{d\}}, \ls{d} y_d) = \lambda^1 \,\,\land \,\,(\forall \lambda \in \Lambda^\nu)\,\,\lambda \cdot \psi^{(d)}(x_{[k-1] \setminus \{d\}}, \ls{d} y_d) = \nu'\Big\}\times G_k\bigg).\end{align*}
The set $\Big\{y_d \in G_d \colon \beta^1(x_{[k-1] \setminus \{d\}}, \ls{d} y_d) = \lambda^1 \,\,\land \,\,(\forall \lambda \in \Lambda^\nu)\,\,\lambda \cdot \psi^{(d)}(x_{[k-1] \setminus \{d\}}, \ls{d} y_d) = \nu'\Big\}$ is actually a coset of $U_{x_{[k-1] \setminus \{d\}}}$ and is contained inside $\{y \in G_d \colon \beta^1(x_{[k-1] \setminus \{d\}}, \ls{d} y) = \lambda^1\} = W_{x_{[k-1] \setminus \{d\}}} + U_{x_{[k-1] \setminus \{d\}}}$, so there is some $u_0 \in W_{x_{[k-1] \setminus \{d\}}}$ such that
\begin{align*}(\tilde{X}^{(d)}_{\nu, \nu'})_{x_{[k-1] \setminus \{d\}}}\,\, \subset\,\, &\bigg(\Big(\bigcup_{u'_0 \in W_{x_{[k-1] \setminus \{d\}}}} X^{(d, u'_0)}_{x_{[k-1] \setminus \{d\}}}\Big) \times G_k\bigg)\,\,\cap \,\, \bigg(\Big(u_0 + U_{x_{[k-1] \setminus \{d\}}}\Big) \times G_k\bigg)\\
=&X^{(d, u_0)}_{x_{[k-1] \setminus \{d\}}} \times G_k.\end{align*}
Combining this with~\eqref{tildexdnunu} we deduce~\eqref{tildexsectioneqn}.\\

\indent In conclusion, the extension map defined by the rule described in the statement is well-defined on $\tilde{X}^{(d)}_{\nu, \nu'}$, affine in direction $k$, and a $D$-homomorphism in direction $d$ for every $\nu \in \Gamma^{(d)}$ and $\nu' \in (\mathbb{F}_p^{r_{12}})^m$.\\

We may apply the above steps for all $d \in [k-1]$ to obtain the relevant $\gamma^{(d)}, \Gamma^{(d)}, \psi^{(d)}, X^{(d)}$. Let $I^{(d)}$ be the codomain of $\psi^{(i)}$ and view each $\gamma^{(d)}$ as a multiaffine map on $G_{[k-1]}$ by mapping $x_{[k-1]} \mapsto \gamma^{(d)}(x_{[k-1] \setminus \{d\}})$. From this perspective, $\gamma^{(d)}$ is $\mathcal{G}'$-supported. Finally, take $X' = \bigcap_{d \in [k-1]} X^{(d)}$, $\gamma = (\gamma^{(1)}, \psi^{(1)}, \dots, \gamma^{(k-1)}, \psi^{(k-1)})$ (which is $\mathcal{G}'$-supported) and $\Gamma = \Gamma^{(1)} \times I^{(1)} \times \dots \times \Gamma^{(k-1)}  \times I^{(k-1)}$. We may take $\eta \geq \Omega_{D, p}((p^{-r}\xi)^{O_{D, p}(1)})$ to finish the proof.\end{proof}

\subsection{Biaffine structure in higher dimensions}

The next theorem tells us that under suitable conditions, if $\phi$ is a multi-homomorphism defined on almost all of a low-codimensional variety, then there is a subvariety $V'$ of bounded codimension and a subset $X'$ that contains almost all points of $V'$ such that for every 2-dimensional axis-aligned cross-section of $\phi$ agrees on $X'$ with a global biaffine map. In other words, we end up in a slightly stronger position than before: the restrictions to 2-dimensional cross-sections are not just bihomomorphisms but bihomomorphisms that can be extended to the whole of the corresponding cross-section of $G_{[k]}$.

\begin{theorem}\label{globalBiaffineInAllPlanes}Let $\mathcal{G} \subset \mathcal{P}([k])$ be a down-set and let $V$ be a $\mathcal{G}$-supported mixed-linear variety of codimension at most $r$. Let $\mathcal{H}$ be the down-set obtained by removing maximal elements from $\mathcal{G}$. Let $X \subset V$ be a set of size at least $(1-\varepsilon) |V|$. Let $\phi \colon X \to H$ be a multi-homomorphism and let $\xi > 0$. Then there exist an $\mathcal{H}$-supported mixed-linear variety $W$ of codimension $O\Big((r + \log_p \xi^{-1})^{O(1)}\Big)$, such that $V' = V \cap W \not=\emptyset$, and a set $X' \subset V' \cap X$ of size $(1-O(\varepsilon^{\Omega(1)}) - O(\xi^{\Omega(1)})) |V'|$ such that for any two coordinates $d_1 \not= d_2$ and every $x_{[k] \setminus \{d_1, d_2\}} \in G_{[k] \setminus \{d_1, d_2\}}$, there is a global biaffine map $\phi^{\mathrm{glob}}_{x_{[k] \setminus \{d_1, d_2\}}} \colon G_{d_1} \times G_{d_2} \to H$ with the property that 
\[\phi^{\mathrm{glob}}_{x_{[k] \setminus \{d_1, d_2\}}}(y_{d_1}, y_{d_2}) = \phi(x_{[k] \setminus \{d_1, d_2\}}, y_{d_1}, y_{d_2})\]
for every $(y_{d_1}, y_{d_2}) \in X'_{x_{[k] \setminus \{d_1, d_2\}}}$.

\end{theorem}

\noindent\textbf{Remark.} By modifying $\xi$ slightly, we may improve the bound to $|X'| =(1-O(\varepsilon^{\Omega(1)}) - \xi) |V'|$ instead of the claimed bound.

\begin{proof}The approach is quite similar to that of the proof of Theorem~\ref{mainExtensionConvStep}. \chg{As expected, we shall apply the simultaneous biaffine regularity lemma, i.e. Theorem~\ref{simReg}, and once we have quasirandom biaffine varieties in the 2-dimensional axis-aligned cross-section, we shall use Theorem~\ref{biaffineExtnFullThm} to find the desired global biaffine maps. However, because of the introduction of additional multiaffine maps, there will be several steps in which we prepare for the application of Theorem~\ref{biaffineExtnFullThm}.}\\

 Let $(d^{(1)}_1, d^{(1)}_2),$ $(d^{(2)}_1, d^{(2)}_2), \dots,$ $(d^{(k')}_1, d^{(k')}_2)$ be an enumeration of all pairs of directions in $[k]$, where $k' = \binom{k}{2}$. By induction on $i \in \{0,1, \dots, k'\}$, we shall show that there exist a $\mathcal{H}$-supported mixed-linear variety $W$ of codimension $(r + \log_p \xi^{-1})^{O(1)}$, such that $V' = V \cap W \not=\emptyset$, and a set $X' \subset V' \cap X$ of size $|X'| = (1-O(\varepsilon^{\Omega(1)})-O(\xi^{\Omega(1)})) |V'|$, such that for any two coordinates $(d_1, d_2) \in \{(d^{(1)}_1, d^{(1)}_2), (d^{(2)}_1, d^{(2)}_2), \dots, (d^{(i)}_1, d^{(i)}_2)\}$, the conclusion in the statement holds.\\

The base case $i = 0$ is trivial. Now assume that the claim holds for some $i \geq 0$ and let $W, V', X'$ be the relevant objects. Let $\xi' = O(\xi^{\Omega(1)}), \varepsilon' = O(\varepsilon^{\Omega(1)})$ be such that
\begin{equation}|X'| \geq (1 - \varepsilon' - \xi')|V'|.\label{globalbiaffinexprimesize}\end{equation}
\noindent\chg{\textbf{Application of Theorem~\ref{simReg}.}} Write $d_1 = d_1^{(i+1)}$ and $d_2 = d_2^{(i+1)}$. Since $V$ and $W$ are mixed-linear so is $V'$, and we may find $\mathcal{G}$-supported multiaffine maps $\beta^1 \colon G_{[k]\setminus \{d_2\}} \to \mathbb{F}_p^{r_1}$, $\beta^2 \colon G_{[k]\setminus \{d_1\}} \to \mathbb{F}_p^{r_2}$ and $\beta^{12} \colon G_{[k]} \to \mathbb{F}_p^{r_{12}}$ such that $r_1 + r_2 + r_{12} \leq \codim V + \codim W \leq (r + \log_p \xi^{-1})^{O(1)}$, there are values $\lambda^1 \in \mathbb{F}_p^{r_1}, \lambda^2 \in \mathbb{F}_p^{r_2}, \lambda^{12} \in \mathbb{F}_p^{r_{12}}$ for which $V' = \{x_{[k]} \in G_{[k]} \colon \beta^1(x_{[k] \setminus \{d_2\}}) = \lambda^1, \beta^2(x_{[k] \setminus \{d_1\}}) = \lambda^2, \beta^{12}(x_{[k]}) = \lambda^{12}\}$, and $\beta^{12}$ is linear in coordinates $G_{d_1}$ and $G_{d_2}$.\\

Let $\mathcal{G}' = \{S \subset [k] \setminus \{d_2\} \colon S \cup \{d_1, d_2\} \in \mathcal{G}\}$ and $\mathcal{G}'' = \{S \subset [k] \setminus \{d_1, d_2\} \colon S \cup \{d_1, d_2\} \in \mathcal{G}\}$. These down-sets do not have any maximal elements of $\mathcal{G}$, so both are contained in $\mathcal{H}$. Note that $\beta^{12}_i(x_{[k]}) = A_i(x_{[k] \setminus \{d_2\}}) \cdot x_{d_2}$ for some multiaffine $A_i \colon G_{[k] \setminus \{d_2\}} \to G_{d_2}$ that is linear in coordinate $d_1$. Since $\beta_i^{12}$ is $\mathcal{G}$-supported, we have that $A_i$ is $\mathcal{G}'$-supported. Apply Theorem~\ref{simReg} to $\beta^1, \beta^2, \beta^{12}$ with parameter $\eta > 0$ in directions $d_1$ and $d_2$. Then there exist
\begin{itemize}
\item positive integers $m, t = O\Big((r_{12} + r_1 + r_2 + \log_{p}\eta^{-1})^{O(1)}\Big)$,
\item $\mathcal{G}'$-supported multiaffine maps $\phi_1, \dots, \phi_{r_{12}} \colon G_{[k] \setminus \{d_2\}} \to \mathbb{F}_p^m$ that are linear in coordinate $d_1$,
\item a $\mathcal{G}''$-supported  multiaffine map $\gamma \colon G_{[k] \setminus \{d_1, d_2\}} \to \mathbb{F}_p^t$, and
\item a collection of values $\Gamma \subset \mathbb{F}_p^t$ such that $|\gamma^{-1}(\Gamma)| \geq (1-\eta)|G_{[k] \setminus \{d_1,d_2\}}|$,
\end{itemize}
such that for each $\mu \in \Gamma$, there is a subspace $\Lambda^\mu \leq \mathbb{F}_p^{r_{12}}$ such that for each $x_{[k] \setminus \{d_1,d_2\}}\in\gamma^{-1}(\mu)$, each $u_0 \in G_{d_1}, v_0 \in G_{d_2}$ and each $\tau \in \mathbb{F}_p^{r_{12}}$, the biaffine variety
\begin{align}&\bigg[\bigg(u_0 + \Big(\{y \in G_{d_1} \colon (\forall \lambda \in \Lambda^\mu)\,\, \lambda \cdot \phi(x_{[k] \setminus \{d_1, d_2\}}, \ls{d_1} y) = 0\} \nonumber\\
&\hspace{3cm}\cap \{y \in G_{d_1} \colon \beta^1(x_{[k] \setminus \{d_1, d_2\}}, \ls{d_1} y) =\beta^1(x_{[k] \setminus \{d_1, d_2\}}, \ls{d_1} 0)\}\Big)\bigg)\nonumber \\
&\hspace{2cm}\times \bigg(v_0 + \Big(\langle A_1(x_{[k] \setminus \{d_1, d_2\}}, \ls{d_1} {u_0}), \dots, A_{r_{12}}(x_{[k] \setminus \{d_1,d_2\}},\ls{d_1} {u_0})\rangle^\perp\nonumber \\
&\hspace{4cm}\cap \{z \in G_{d_2} \colon \beta^2(x_{[k] \setminus \{d_1, d_2\}}, \ls{d_2} z) = \beta^2(x_{[k] \setminus \{d_1, d_2\}}, \ls{d_2} 0)\}\Big)\bigg)\bigg]\nonumber \\
&\hspace{1cm} \cap \{(y,z) \in G_{d_1}\times G_{d_2} \colon \beta^{12}(x_{[k] \setminus \{d_1, d_2\}}, y, z) = \tau\}\label{biaffinevaruvtaueqn}\end{align}
is either $\eta$-quasirandom with density $|\Lambda^\mu|p^{-{r_{12}}}$ or empty.\\

\noindent\chg{\textbf{Preparation for the application of Theorem~\ref{biaffineExtnFullThm}.}} Let $x_{[k] \setminus \{d_1, d_2\}} \in \gamma^{-1}(\Gamma)$ be such that $V'_{x_{[k] \setminus \{d_1, d_2\}}} \not=\emptyset$. Write $\mu = \gamma(x_{[k] \setminus \{d_1, d_2\}})$ and
\begin{align*}U_{x_{[k] \setminus \{d_1, d_2\}}} = \Big\{y \in G_{d_1} &\colon (\forall \lambda \in \Lambda^\mu)\,\, \lambda \cdot \phi(x_{[k] \setminus \{d_1, d_2\}}, \ls{d_1} y) = 0\Big\}\\
& \cap \Big\{y \in G_{d_1} \colon \beta^1(x_{[k] \setminus \{d_1, d_2\}}, \ls{d_1} y) =\beta^1(x_{[k] \setminus \{d_1, d_2\}}, \ls{d_1} 0)\Big\}.\end{align*}
Since $V'_{x_{[k] \setminus \{d_1, d_2\}}} \not=\emptyset$ we also have that $\Big\{y \in G_{d_1} \colon \beta^1(x_{[k] \setminus \{d_1, d_2\}}, \ls{d_1} y) = \lambda^1\Big\}$ is a non-empty coset of the subspace $\Big\{y \in G_{d_1} \colon \beta^1(x_{[k] \setminus \{d_1, d_2\}}, \ls{d_1} y) = \beta^1(x_{[k] \setminus \{d_1, d_2\}}, \ls{d_1} 0)\Big\}$. Thus, we may find a coset $W_{x_{[k] \setminus \{d_1, d_2\}}}$ inside $G_{d_1}$ such that
\[U_{x_{[k] \setminus \{d_1, d_2\}}}  \oplus W_{x_{[k] \setminus \{d_1, d_2\}}} = \Big\{y \in G_{d_1} \colon \beta^1(x_{[k] \setminus \{d_1, d_2\}}, \ls{d_1} y) = \lambda^1\Big\}.\]
For each $u_0\in W_{x_{[k] \setminus \{d_1, d_2\}}}$, write 
\begin{align*}Y_{x_{[k] \setminus \{d_1, d_2\}}, u_0} = &\Big\langle A_1(x_{[k] \setminus \{d_1, d_2\}}, \ls{d_1}{u_0}), \dots, A_{r_{12}}(x_{[k] \setminus \{d_1, d_2\}}, \ls{d_1}{u_0})\Big\rangle^\perp\\
&\hspace{2cm} \cap \Big\{z \in G_{d_2} \colon \beta^2(x_{[k] \setminus \{d_1, d_2\}}, \ls{d_2} z) = \beta^2(x_{[k] \setminus \{d_1, d_2\}}, \ls{d_2} 0)\Big\}.\end{align*}
With this notation, the biaffine variety in~\eqref{biaffinevaruvtaueqn}, which we shall denote $B^{u_0, v_0, \tau}$,\footnote{The biaffine variety $B^{u_0, v_0, \tau}$ depends also on $x_{[k] \setminus \{d_1, d_2\}}$, but we do not make this fact explicit in the notation as expressions such as $B^{u_0, v_0, \tau}_{x_{[k] \setminus \{d_1, d_2\}}}$ would be confusing and cumbersome, yet we stress this dependence.} may be written as
\[B^{u_0, v_0, \tau} = \Big((u_0 + U_{x_{[k] \setminus \{d_1, d_2\}}}) \times (v_0 + Y_{x_{[k] \setminus \{d_1, d_2\}}, u_0})\Big) \cap \{(y,z) \in G_{d_1}\times G_{d_2} \colon \beta^{12}(x_{[k] \setminus \{d_1, d_2\}}, y, z) = \tau\}.\]
Using the fact that $V'_{x_{[k] \setminus \{d_1, d_2\}}} \not=\emptyset$ one more time, we see that $\{z \in G_{d_2} \colon \beta^2(x_{[k] \setminus \{d_1, d_2\}}, \ls{d_2} z) = \lambda^2\}$ is a non-empty coset of the subspace $\{z \in G_{d_2} \colon \beta^2(x_{[k] \setminus \{d_1, d_2\}}, \ls{d_2} z) = \beta^2(x_{[k] \setminus \{d_1, d_2\}}, \ls{d_2} 0)\}$. Hence, there exists a coset $T_{x_{[k] \setminus \{d_1, d_2\}}, u_0}$ in $G_{d_2}$ such that
\[Y_{x_{[k] \setminus \{d_1, d_2\}}, u_0} \oplus T_{x_{[k] \setminus \{d_1, d_2\}}, u_0} = \{z \in G_{d_2} \colon \beta^2(x_{[k] \setminus \{d_1, d_2\}}, \ls{d_2} z) = \lambda^2\}.\]
Let $C_{x_{[k] \setminus \{d_1, d_2\}}, u_0}$ be the set of all $v_0 \in T_{x_{[k] \setminus \{d_1, d_2\}}, u_0}$ such that $B^{u_0, v_0, \lambda^{12}}$ is non-empty and hence $\eta$-quasirandom with density $|\Lambda^\mu|p^{-r_{12}}$. By quasirandomness, we see that for a proportion at least $1 - p^{r_2 + r_{12}}\eta$ of $y \in u_0 + U_{[k] \setminus \{d_1, d_2\}}$ we have that $B^{u_0, v_0, \lambda^{12}}_{y \bullet}$ is non-empty for every $v_0 \in C_{x_{[k] \setminus \{d_1, d_2\}}, u_0}$. On the other hand, for a fixed $y$, the set of all $v_0 \in T_{x_{[k] \setminus \{d_1, d_2\}}, u_0}$ such that $B^{u_0, v_0, \lambda^{12}}_{y \bullet}$ is non-empty is a coset. Thus $C_{x_{[k] \setminus \{d_1, d_2\}}, u_0}$ is a coset in $G_k$.\\
\indent For each $v_0 \in C_{x_{[k] \setminus \{d_1, d_2\}}, u_0}$, let $\tilde{\Lambda}(v_0)$ be the set of all $\tau \in \mathbb{F}_p^{r_{12}}$ such that $B^{u_0, v_0, \tau}$ is non-empty (and hence $\eta$-quasirandom with density $|\Lambda^\mu|p^{-r_{12}}$). (The set $\tilde{\Lambda}(v_0)$ depends on $x_{[k] \setminus \{d_1, d_2\}}$ and $u_0$ as well, but we opt not to include these in the subscript as $\tilde{\Lambda}(v_0)$ will only be used temporarily in order to conclude the fact~\eqref{bu0tauqr} below.)

\begin{claima*}\begin{itemize}
\item[\textbf{(i)}] For each $v_0 \in C_{x_{[k] \setminus \{d_1, d_2\}}, u_0}$, $\tilde{\Lambda}(v_0)$ is a non-empty coset inside $\mathbb{F}_p^{r_{12}}$.
\item[\textbf{(ii)}] For each $v_0 \in C_{x_{[k] \setminus \{d_1, d_2\}}, u_0}$, $\tilde{\Lambda}(v_0)$ is the same, and we may write simply $\tilde{\Lambda}$.
\end{itemize}\end{claima*}

\begin{proof}[Proof of Claim A]\textbf{(i)} We already know that $\lambda^{12} \in \tilde{\Lambda}(v_0)$, so the set is non-empty. On the other hand, using quasirandomness, provided that $\eta p^{r_{12}} < 1$, we may find $y \in u_0 + U_{x_{[k] \setminus \{d_1, d_2\}}}$ such that $B^{u_0, v_0, \tau}_{y \bullet} = \emptyset$ when $\tau \notin \tilde{\Lambda}(v_0)$ and $B^{u_0, v_0, \tau}_{y \bullet} \not= \emptyset$ when $\tau \in \tilde{\Lambda}(v_0)$. Hence, $\tilde{\Lambda}(v_0)$ is precisely the image of an affine map $z \mapsto \beta^{12}(y,z)$, with domain $v_0 + Y_{x_{[k] \setminus \{d_1, d_2\}}, u_0}$. It follows that $\tilde{\Lambda}(v_0)$ is indeed a coset inside $\mathbb{F}_p^{r_{12}}$.\\
\noindent\textbf{(ii)} Proceeding further, note that in fact we may pick $y \in u_0 + U_{x_{[k] \setminus \{d_1, d_2\}}}$ so that the property above holds simultaneously for all $v_0 \in C_{x_{[k] \setminus \{d_1, d_2\}}, u_0}$, provided that $\eta p^{2r_{12}} < 1$. We may then observe that for each $v_0 \in C_{x_{[k] \setminus \{d_1, d_2\}}, u_0}$, the set $\tilde{\Lambda}(v_0)$ is a coset of the subspace $I \leq \mathbb{F}_p^{r_12}$ defined as the image of linear map $z \mapsto \beta^{12}(y,z) - \beta^{12}(y,0)$, with domain $Y_{x_{[k] \setminus \{d_1, d_2\}}, u_0}$. But $\lambda^{12} \in \tilde{\Lambda}(v_0)$, which implies that $\tilde{\Lambda}(v_0) = \lambda^{12} + I$, which is independent of the choice of $v_0$.\end{proof}

By Claim A, we see that for each $u_0 \in W_{x_{[k] \setminus \{d_1, d_2\}}}$ and every $\tau \in \mathbb{F}_p^{r_{12}}$ we have that either $B^{u_0, v_0, \tau}$ is $\eta$-quasirandom with density $|\Lambda^\mu|p^{-r_{12}}$ for every $v_0 \in C_{x_{[k] \setminus \{d_1, d_2\}}, u_0}$ (this happens when $\tau \in \tilde{\Lambda}$) or $B^{u_0, v_0, \tau} = \emptyset$ for every $v_0 \in C_{x_{[k] \setminus \{d_1, d_2\}}, u_0}$ (this happens when $\tau \notin \tilde{\Lambda}$). By Lemma~\ref{simultQuasiRand}, we see that for each $u_0 \in W_{x_{[k] \setminus \{d_1, d_2\}}}$ and $\tau \in \tilde{\Lambda}$
\begin{align*}B^{u_0, \tau}=&\Big((u_0 + U_{x_{[k] \setminus \{d_1, d_2\}}}) \times (C_{x_{[k] \setminus \{d_1, d_2\}}, u_0} + Y_{x_{[k] \setminus \{d_1, d_2\}}, u_0})\Big)\\
&\hspace{2cm} \cap \{(y,z) \in G_{d_1}\times G_{d_2} \colon \beta^{12}(x_{[k] \setminus \{d_1, d_2\}}, y, z) = \tau\}\end{align*}
is an $\eta p^{r_{12}}$-quasirandom biaffine variety\footnote{As with biaffine varieties $B^{u_0, v_0, \tau}$, we stress that the variety $B^{u_0, \tau}$ depends also on $x_{[k] \setminus \{d_1, d_2\}}$, even though this is not apparent in the notation. We allow expressions such as $\sum_{x_{[k] \setminus \{d_1, d_2\}}} |B^{u_0, \tau}|$, where $B^{u_0, \tau}$ always means the variety $B^{u_0, \tau}$ for the given $x_{[k] \setminus \{d_1, d_2\}}$.} with density $|\Lambda^\mu|p^{-r_{12}}$. On the other hand, when $\tau \notin \tilde{\Lambda}$, the set given by this expression is empty. Hence, 
\begin{equation}\label{bu0tauqr}(\forall \tau \in \mathbb{F}_p^{r_{12}})\hspace{6pt}B^{u_0, \tau}\text{ is $\eta p^{r_{12}}$-quasirandom}.\end{equation}
Note that one also has
\begin{align*}B^{u_0, \lambda^{12}}=&\bigcup_{v_0 \in C_{x_{[k] \setminus \{d_1, d_2\}}}}\,B^{u_0, v_0, \lambda^{12}}\\
=&\bigcup_{v_0 \in T_{x_{[k] \setminus \{d_1, d_2\}}}}\,B^{u_0, v_0, \lambda^{12}}\hspace{2cm}\text{($B^{u_0, v_0, \lambda^{12}} = \emptyset$ for $v_0 \in T_{x_{[k] \setminus \{d_1, d_2\}}} \setminus C_{x_{[k] \setminus \{d_1, d_2\}}}$)}\\
=&\Big[\Big(u_0 + U_{x_{[k] \setminus \{d_1, d_2\}}}\Big) \times \Big\{z \in G_{d_2} \colon \beta^2(x_{[k] \setminus \{d_1, d_2\}}, \ls{d_2} z) = \lambda^2\Big\}\Big]\\
&\hspace{2cm} \cap \{(y,z) \in G_{d_1}\times G_{d_2} \colon \beta^{12}(x_{[k] \setminus \{d_1, d_2\}}, y, z) = \lambda^{12}\}.\end{align*}

By Lemma~\ref{varsizelemma}, we have $|V'| \geq p^{-k (r + \log_p \xi^{-1})^{O(1)}} |G_{[k]}|$. Returning to~\eqref{globalbiaffinexprimesize} we see that
\begin{align*}|(\gamma^{-1}(\Gamma) \times G_{d_1} \times G_{d_2}) \cap X'| \,\geq &\,|X'| - |G_{[k] \setminus \{d_1,d_2\}} \setminus \gamma^{-1}(\Gamma)|\,|G_{d_1}|\,|G_{d_2}|\\
 \geq&\, (1 - \varepsilon' - \xi')|V'| - \eta |G_{[k]}| \\
\geq &\,\Big(1 - \varepsilon' - \xi' - \eta p^{k (r + \log_p \xi^{-1})^{O(1)}}\Big)|V'|.\end{align*}
In particular $(\gamma^{-1}(\Gamma) \times G_{d_1} \times G_{d_2}) \cap V' \not= \emptyset$, as long as $\varepsilon', \xi' < 1/4$ (which we may assume as otherwise the statement of the theorem becomes trivial) and $\eta \leq \frac{1}{2} p^{-k (r + \log_p \xi^{-1})^{O(1)}}$. Write $\varepsilon'' = \varepsilon' + \xi' + \eta p^{k (r + \log_p \xi^{-1})^{O(1)}}$. By averaging, we may find $\mu \in \Gamma$ such that 
\[|(\gamma^{-1}(\mu) \times G_{d_1} \times G_{d_2}) \cap X'| \geq (1 - \varepsilon'') |(\gamma^{-1}(\mu) \times G_{d_1} \times G_{d_2}) \cap V'| > 0.\] 
Let $S$ be the set of all pairs $(x_{[k] \setminus \{d_1,d_2\}}, u_0)$ such that $\gamma(x_{[k] \setminus \{d_1,d_2\}}) = \mu$, $u_0\in W_{x_{[k] \setminus \{d_1, d_2\}}}$, $B^{u_0, \lambda^{12}}$ is non-empty, and 
\[|X'_{x_{[k] \setminus \{d_1,d_2\}}} \cap ((u_0 + U_{x_{[k] \setminus \{d_1, d_2\}}}) \times G_{d_2})| \geq (1-\sqrt{\varepsilon''})|B^{u_0, \lambda^{12}}|.\] 
Note that $\bigcup_{u_0 \in W_{x_{[k] \setminus \{d_1, d_2\}}}} B^{u_0, \lambda^{12}} = (V')_{x_{[k] \setminus \{d_1, d_2\}}}$ for each element $x_{[k] \setminus \{d_1, d_2\}} \in \gamma^{-1}(\mu)$. 
Therefore 
\begin{align*}&\varepsilon''\Big|\Big(\gamma^{-1}(\mu) \times G_{d_1} \times G_{d_2}\Big) \cap V'\Big|\\
&\hspace{2cm}\geq \Big|\Big(\gamma^{-1}(\mu) \times G_{d_1} \times G_{d_2}\Big) \cap V' \setminus X'\Big|\\ 
&\hspace{2cm}= \sum_{x_{[k] \setminus \{d_1,d_2\}} \in \gamma^{-1}(\mu)} \Big|(V')_{x_{[k] \setminus \{d_1,d_2\}}} \setminus (X')_{x_{[k] \setminus \{d_1,d_2\}}}\Big|\\
&\hspace{2cm}= \sum_{\substack{x_{[k] \setminus \{d_1,d_2\}} \in \gamma^{-1}(\mu)\\u_0 \in W_{x_{[k] \setminus \{d_1, d_2\}}}}} \Big|B^{u_0, \lambda^{12}} \setminus \Big((X')_{x_{[k] \setminus \{d_1,d_2\}}}\, \cap \, ((u_0 + U_{x_{[k] \setminus \{d_1, d_2\}}}) \times G_{d_2})\Big)\Big|\\
&\hspace{2cm}\geq \sum_{\substack{x_{[k] \setminus \{d_1,d_2\}} \in \gamma^{-1}(\mu)\\u_0 \in W_{x_{[k] \setminus \{d_1, d_2\}}}\\\text{such that}\\(x_{[k] \setminus \{d_1,d_2\}}, u_0) \notin S}} \Big|B^{u_0, \lambda^{12}} \setminus \Big((X')_{x_{[k] \setminus \{d_1,d_2\}}}\, \cap \,((u_0 + U_{x_{[k] \setminus \{d_1, d_2\}}}) \times G_{d_2})\Big)\Big|\\  
&\hspace{2cm}\geq \sqrt{\varepsilon''} \sum_{\substack{x_{[k] \setminus \{d_1,d_2\}} \in \gamma^{-1}(\mu)\\u_0 \in W_{x_{[k] \setminus \{d_1, d_2\}}}\\\text{such that}\\(x_{[k] \setminus \{d_1,d_2\}}, u_0) \notin S}} \Big|B^{u_0, \lambda^{12}}\Big|\end{align*}
from which we obtain
\begin{align}&\sum_{(x_{[k] \setminus \{d_1,d_2\}}, u_0) \in S}  \Big|X'_{x_{[k] \setminus \{d_1,d_2\}}} \cap \Big((u_0 + U_{x_{[k] \setminus \{d_1, d_2\}}}) \times G_{d_2}\Big)\Big|\nonumber\\
&\hspace{2cm}\geq (1-\sqrt{\varepsilon''}) \sum_{(x_{[k] \setminus \{d_1,d_2\}}, u_0) \in S} \Big|B^{u_0, \lambda^{12}}\Big|\nonumber\\
&\hspace{2cm} \geq (1-2\sqrt{\varepsilon''}) \Big|\Big(\gamma^{-1}(\mu) \times G_{d_1} \times G_{d_2}\Big) \cap V'\Big|.\label{unionofxprimebound}\end{align}

\noindent\chg{\textbf{Applying Theorem~\ref{biaffineExtnFullThm}.}} Now, provided that $\eta \leq \cons p^{-\con r_{12}}$, apply Theorem~\ref{biaffineExtnFullThm} to $X'_{x_{[k] \setminus \{d_1,d_2\}}} \cap ((u_0 + U_{x_{[k] \setminus \{d_1, d_2\}}}) \times G_{d_2})$ and $B^{u_0, \lambda^{12}}$ for every $(x_{[k] \setminus \{d_1,d_2\}}, u_0) \in S$. The quasirandomness condition in the theorem is satisfied by~\eqref{bu0tauqr}, as long as $\eta \leq \cons p^{-\con r_{12}}$. We obtain $\tilde{X}'_{x_{[k] \setminus \{d_1,d_2\}}, u_0} \subset X'_{x_{[k] \setminus \{d_1,d_2\}}} \cap ((u_0 + U_{x_{[k] \setminus \{d_1, d_2\}}}) \times G_{d_2})$ such that
\begin{equation}\Big|\tilde{X}'_{x_{[k] \setminus \{d_1,d_2\}}, u_0} \Big|\,\geq\,\Big(1-O({(\varepsilon'')}^{\Omega(1)}) - O((\eta p^{r_{12}})^{\Omega(1)})\Big)|B^{u_0, \lambda^{12}}|\label{xtildexprimeboundeqn}\end{equation}
and a global biaffine map $\phi^{\text{glob}}_{x_{[k] \setminus \{d_1, d_2\}}, u_0} \colon G_{d_1} \times G_{d_2} \to H$ with the property that 
\begin{equation}\phi^{\text{glob}}_{x_{[k] \setminus \{d_1, d_2\}}, u_0}(y_{d_1}, y_{d_2}) = \phi(x_{[k] \setminus \{d_1, d_2\}}, y_{d_1}, y_{d_2})\label{globalbiaffineu0map}\end{equation}
for every $(y_{d_1}, y_{d_2}) \in \tilde{X}'_{x_{[k] \setminus \{d_1, d_2\}}, u_0}$.
\\

\noindent\chg{\textbf{Completing the proof.}} To finish the proof, we need to specify the set $X''$ and variety $V''$, which have the desired properties. (The final definitions are~\eqref{doubleprimeXdefcrosssec} and~\eqref{doubleprimeVdefcrosssec}.) Firstly, we need to find a subset $\tilde{S} \subset \gamma^{-1}(\mu)$ and to associate an element $u_0(x_{[k] \setminus \{d_1,d_2\}})$ with each $x_{[k] \setminus \{d_1,d_2\}} \in \tilde{S}$ such that $(x_{[k] \setminus \{d_1,d_2\}}, u_0(x_{[k] \setminus \{d_1,d_2\}})) \in S$ and
\[\bigcup_{x_{[k] \setminus \{d_1,d_2\}} \in \tilde{S}} \{x_{[k] \setminus \{d_1,d_2\}}\} \times \tilde{X}'_{x_{[k] \setminus \{d_1, d_2\}}, u_0(x_{[k] \setminus \{d_1, d_2\}})}\]
is a very dense subset of a bounded codimension subvariety of $V'$. To this end, fix a basis $\nu^1, \dots, \nu^l$ of $\Lambda^\mu$ and observe that the cosets of $U_{x_{[k] \setminus \{d_1, d_2\}}}$ that partition $\Big\{y \in G_{d_1} \colon \beta^1(x_{[k] \setminus \{d_1, d_2\}}, \ls{d_1} y) = \lambda^1\Big\}$ are sets of the form
\[\Big\{y \in G_{d_1} \colon \beta^1(x_{[k] \setminus \{d_1, d_2\}}, \ls{d_1} y) = \lambda^1 \land (\forall i \in [l])\,\,\nu^i \cdot \phi(x_{[k] \setminus \{d_1, d_2\}}, \ls{d_1} y) = \rho_i\Big\}\]
for all choices of $\rho \in \mathbb{F}_p^l$. Hence, for each $\rho \in \mathbb{F}_p^l$, we define $\tilde{S}_\rho$ as the set of all $(x_{[k] \setminus \{d_1,d_2\}}, u_0) \in S$ such that 
\[u_0 + U_{x_{[k] \setminus \{d_1, d_2\}}} = \Big\{y \in G_{d_1} \colon \beta^1(x_{[k] \setminus \{d_1, d_2\}}, \ls{d_1} y) = \lambda^1 \land (\forall i \in [l])\,\, \nu^i \cdot \phi(x_{[k] \setminus \{d_1, d_2\}}, \ls{d_1} y) = \rho_i\Big\}.\]
Thus, $S = \bigcup_{\rho \in \mathbb{F}_p^l} \tilde{S}_\rho$ is a partition, and the sets $\tilde{S}_\rho$ are candidates for the set $\tilde{S}$ and the mapping $x_{[k] \setminus \{d_1,d_2\}} \mapsto u_0(x_{[k] \setminus \{d_1,d_2\}})$ that we want. Note that for each $x_{[k] \setminus \{d_1,d_2\}} \in G_{[k] \setminus \{d_1,d_2\}}$ there is at most one value $u_0 \in W_{x_{[k] \setminus \{d_1, d_2\}}}$ such that $(x_{[k] \setminus \{d_1,d_2\}}, u_0) \in \tilde{S}_\rho$.\\

Note that we may partition $(\gamma^{-1}(\mu) \times G_{d_1} \times G_{d_2}) \cap V'$ as
\[(\gamma^{-1}(\mu) \times G_{d_1} \times G_{d_2}) \cap V' = \bigcup_{\rho \in \mathbb{F}_p^l} \bigg[ (\gamma^{-1}(\mu) \times G_{d_1} \times G_{d_2}) \cap V' \cap \bigg(\Big(\bigcap_{i \in [l]} (\nu^i \cdot \phi)^{-1}(\rho_i) \Big)\times G_{d_2}\bigg)\bigg]\] 
and observe that 
\begin{align*}
&\bigcup_{(x_{[k] \setminus \{d_1,d_2\}}, u_0) \in \tilde{S}_\rho} \{x_{[k] \setminus \{d_1,d_2\}}\} \times \tilde{X}'_{x_{[k] \setminus \{d_1, d_2\}}, u_0}\\
\subset\,&\bigcup_{(x_{[k] \setminus \{d_1,d_2\}}, u_0) \in \tilde{S}_\rho} \{x_{[k] \setminus \{d_1,d_2\}}\} \times \Big(X'_{x_{[k] \setminus \{d_1,d_2\}}} \cap ((u_0 + U_{x_{[k] \setminus \{d_1, d_2\}}}) \times G_{d_2})\Big)\\
&\hspace{3cm}\subset\,\, (\gamma^{-1}(\mu) \times G_{d_1} \times G_{d_2}) \cap V' \cap \bigg(\Big(\bigcap_{i \in [l]} (\nu^i \cdot \phi)^{-1}(\rho_i)\Big) \times G_{d_2}\bigg)\\
\end{align*}
for each $\rho$. Thus, by using inequalities~\eqref{unionofxprimebound} and~\eqref{xtildexprimeboundeqn} and averaging, we obtain $\rho \in \mathbb{F}_p^l$ such that 
\begin{equation}\label{doubleprimeXdefcrosssec} X'' = \bigcup_{(x_{[k] \setminus \{d_1,d_2\}}, u_0) \in \tilde{S}_\rho} \{x_{[k] \setminus \{d_1,d_2\}}\} \times \tilde{X}'_{x_{[k] \setminus \{d_1, d_2\}}, u_0}\end{equation}
and 
\begin{equation}\label{doubleprimeVdefcrosssec} V'' = \Big(\gamma^{-1}(\mu) \times G_{d_1} \times G_{d_2}\Big) \cap V' \cap \bigg(\Big(\bigcap_{i \in [l]} (\nu^i \cdot \phi)^{-1}(\rho_i)\Big) \times G_{d_2}\bigg)\end{equation} 
such that $V'' \not= \emptyset$ and $|X''| \geq \Big(1-O({(\varepsilon'')}^{\Omega(1)}) - O((\eta p^{r_{12}})^{\Omega(1)})\Big) |V''|$. Further, let $x_{[k] \setminus \{d_1,d_2\}} \in G_{[k] \setminus \{d_1, d_2\}}$ be such that $X''_{x_{[k] \setminus \{d_1,d_2\}}}\not=\emptyset$. Then $X''_{x_{[k] \setminus \{d_1,d_2\}}} = \tilde{X}'_{x_{[k] \setminus \{d_1, d_2\}}, u_0}$ for the element $u_0$ such that $(x_{[k] \setminus \{d_1,d_2\}}, u_0) \in \tilde{S}_\rho$. By~\eqref{globalbiaffineu0map} there is a global biaffine map $\phi^{\text{glob}}_{x_{[k] \setminus \{d_1, d_2\}}} \colon G_{d_1} \times G_{d_2} \to H$ (which was denoted by $\phi^{\text{glob}}_{x_{[k] \setminus \{d_1, d_2\}}, u_0}$ in~\eqref{globalbiaffineu0map}) with the property that 
\[\phi^{\text{glob}}_{x_{[k] \setminus \{d_1, d_2\}}}(y_{d_1}, y_{d_2}) = \phi(x_{[k] \setminus \{d_1, d_2\}}, y_{d_1}, y_{d_2})\]
for every $(y_{d_1}, y_{d_2}) \in X''_{x_{[k] \setminus \{d_1,d_2\}}}$. On the other hand, if  $X''_{x_{[k] \setminus \{d_1,d_2\}}} =\emptyset$ take an arbitrary global biaffine map $\phi^{\text{glob}}_{x_{[k] \setminus \{d_1, d_2\}}} \colon G_{d_1} \times G_{d_2} \to H$ for completeness. We may choose $\eta = \Omega\Big(p^{-O\big((r + \log_p \xi^{-1})^{O(1)}\big)}\Big)$ so that the necessary bounds are satisfied.\\

Hence, $X''$ and $V''$ have the claimed properties, except that $V''$ possibly is not mixed-linear. To make it mixed-linear, first recall that $V'$ is already mixed-linear. Look at the multilinear parts of maps $\gamma$ and $\nu^i \cdot \phi$ and observe that $V''$ becomes a union of layers of a mixed-linear multiaffine map and average over these layers. \end{proof}

\section{Extending maps from nearly full varieties}\label{sec10labelextendingmultiaffine}

\subsection{Extensions from 1-codimensional subvarieties}

The next theorem is a generalization of Proposition~\ref{biaffineAlmostToFullExtn} to the multivariate case. Its statement is rather technical so we describe it informally beforehand. Let $V = \alpha^{-1}(\nu)$ be a variety defined by multilinear forms $\alpha_i \colon G_{I_i} \to \mathbb{F}_p$ for $i \in [r]$, where $\nu \in \mathbb{F}_p^r$. Assume that $\alpha_1$ is quasirandom with respect to the other forms, \chg{by which we mean that all linear combinations of $\alpha_1$ (which receives coefficient 1, say) and the other multilinear forms $\alpha_i$ with the same set of variables have a sufficiently small bias,} and consider the variety defined by all forms but the first one, that is, $V^{[2,r]} = \{x_{[k]} \in G_{[k]} \colon \alpha_2(x_{[k]}) = \nu_2, \dots, \alpha_r(x_{[k]}) = \nu_r\}$.\\
\indent Now suppose that $\phi$ is a multiaffine map defined on a $1-o(1)$ proportion of $V$. Then, we may extend $\phi$ to most of $V^{[2,r]} \cap W$, for a lower order variety $W$, using an explicit formula. 

\begin{theorem}\label{extn1codim}For  every positive integer $k$ there are constants $C_k, D_k \geq 1$ such that the following statement holds. Let $\mathcal{G} \subset \mathcal{P}([k])$ be a down-set and let $\mathcal{G}'$ be the down-set obtained by removing the maximal elements from $\mathcal{G}$. Suppose that we are given multilinear forms $\alpha_i \colon G_{I_i} \to \mathbb{F}_p$, where $I_i \in \mathcal{G}$, for $i \in [r]$. Let $\nu\in\mathbb{F}_p^r$ and let $V = \big\{x_{[k]} \in G_{[k]} \colon (\forall i \in [r])\,\,\alpha_i(x_{I_i})=\nu_i\big\}$. Let $\xi > 0$. Suppose that $I_1$ is a maximal element of $\mathcal{G}$ and that
\[\bias (\alpha_1 - \lambda \cdot \alpha) \leq p^{-C_k (r + \log_p \xi^{-1})^{D_k}}\]
for every $\lambda \in \mathbb{F}_p^{[2,r]}$ such that $\lambda_i =0$ whenever $I_i \not= I_1$. Let $X \subset V$ be a set of size at least $(1-\varepsilon) |V|$ and let $\phi \colon X \to H$ be a multiaffine map. Write $I_1 = \{c_1, c_2, \dots, c_m\}$. Let $h_0 \colon G_{[k] \setminus I_1}\to H$ be an arbitrary multiaffine map. Then there are a $\mathcal{G}'$-supported multiaffine variety $W$ of codimension $O\Big((r + \log_p \varepsilon^{-1} + \log_p \xi^{-1})^{O(1)}\Big)$ such that the variety $V'$ defined by $\{x_{[k]} \in G_{[k]} \colon (\forall i \in [2, r])\,\,\alpha_i(x_{I_i}) = \nu_i\} \cap W$ is non-empty, a set $X' \subset V'$ of size
$(1-O(\varepsilon^{\Omega(1)}) - O(\xi^{\Omega(1)})) |V'|$, a multiaffine map $\psi \colon X' \to H$, a point $a_{[k]} \in G_{[k]}$, and $\mu_0 \in \mathbb{F}_p \setminus \{\nu_1\}$ such that for each $x_{[k]} \in X'$,
\begin{itemize}
\item[\textbf{(i)}] when $x_{[k]} \in V$, then $x_{[k]} \in X$ and $\psi(x_{[k]}) = \phi(x_{[k]})$, and
\item[\textbf{(ii)}] when $\alpha_1(x_{I_1}) = \mu$ for some $\mu \not= \nu_1$, for $\Omega\Big(p^{-O((r + \log_p \xi^{-1})^{O(1)})}\Big)|G_{I_1}|$ choices of $u_{I_1} \in G_{I_1}$ we have
\begin{align*}\psi(x_{[k]}) &= \phi\Big(x_{[k] \setminus \{c_m\}}, x_{c_m} - \frac{\mu- \nu_1}{\mu_0 - \nu_1}(a_{c_m} - u_{c_m})\Big)\\
&\hspace{1cm} + \frac{\mu- \nu_1}{\mu_0 - \nu_1}\bigg(h_0(x_{[k] \setminus I_1}) - \phi(x_{[k] \setminus \{c_m\}}, u_{c_m})\\
&\hspace{4cm} + \sum_{i \in [m-1]}\phi(x_{[k] \setminus \{c_i, \dots, c_{m}\}}, u_{c_i} + x_{c_i} - a_{c_i}, a_{\{c_{i+1}, \dots, c_m\}})\\
&\hspace{4cm} - \sum_{i \in [m-1]}\phi(x_{[k] \setminus \{c_i, \dots, c_{m}\}}, u_{c_i}, a_{\{c_{i+1}, \dots, c_m\}})\bigg),\end{align*}
and additionally all points in the arguments of $\phi$ belong to $X$.
\end{itemize}
\end{theorem}

\noindent\textbf{Remark.} By modifying $\xi$ appropriately, we may strengthen the conclusion slightly to
\begin{equation}|X'| \geq (1-O(\varepsilon^{\Omega(1)}) - \xi) |V'|.\label{extn1codimXprimeDens}\end{equation}

\begin{proof} Let us first give a brief sketch of the argument.\\

\noindent\chg{\textbf{Sketch of the proof.} Our goal is to extend the map $\phi \colon X \to H$ to a multiaffine map to a very dense subset of $V'$, modulo lower order varieties. Note that $\alpha_1$ plays no role in the definition of $V'$ and thus $V'$ has smaller codimension than $V$. We shall assume that $I_1 = [d_0 + 1, k]$. The extension map $\phi^{\text{ext}}$ will be defined in three stages. The first one is the trivial one, where $\phi^{\text{ext}}(x_{[k]})$ is set to be $ \phi(x_{[k]})$ for the part of the domain that intersects $X$. The second stage, which is the key part of the definition, defines $\phi^{\text{ext}}$ at almost all points inside a variety $\{x_{[k]} \in V' : \alpha_1(x_{I_1}) = \mu_0\}$, for some $\mu_0 \not= \nu_1$ (again, modulo lower order varieties). Note that this is a variety very closely related to $V$. The only difference is the value of $\alpha_1$. This time, we find a suitable a point $a_{[k]}$ and then the value $\phi^{\text{ext}}(x_{[k]})$ is defined by starting from $(x_{[d_0]}, a_{[d_0 + 1, k]})$ and changing the coordinates $a_i$ to $x_i$, one by one, for $i \geq d_0 + 1$. That is, once $\phi^{\text{ext}}(x_{[i]}, a_{[i+1,k]})$ is defined, we set 
\begin{equation}\label{phiextcodim1defnsecond}\phi^{\text{ext}}(x_{[i+1]}, a_{[i+2,k]}) = \phi^{\text{ext}}(x_{[i]}, a_{[i+1,k]}) + \phi(x_{[i]}, u_{i+1} + x_{i+1} - a_{i+1}, a_{[i+2,k]}) - \phi(x_{[i]}, u_{i+1}, a_{[i+2,k]}),\end{equation}
where $u_{i+1}$ is such that the points $(x_{[i]}, u_{i+1} + x_{i+1} - a_{i+1}, a_{[i+2,k]})$ and $(x_{[i]}, u_{i+1}, a_{[i+2,k]})$ both lie in $X$.\\
\indent Firstly, this leads us to the requirement that each point $x_{[d]}$ of the extended domain has the property that $|X_{x_{[k] \setminus \{d\}}}| \geq (1-o(1)) |V_{x_{[k] \setminus \{d\}}}|$, which is part of the definition of the set $Y_d$ below. Having a very dense $X_{x_{[k] \setminus \{d\}}}$ is crucial for the definition in~\eqref{phiextcodim1defnsecond}. Secondly, we need to ensure that we may reach all points $x_{[k]}$ of the extended domain via paths in which we change coordinates one by one and still remain in the extended domain. To this end, we apply Theorem~\ref{connectedVeryDensePiece}, and the set $Z$ below arises in this way.\\
\indent Finally, in the third stage, we extend $\phi$ to those points for which $\alpha_1(x_{[k]}) \not= \nu_1, \mu_0$. In the previous two stages, we already covered points with $\alpha_1(x_{[k]}) \in \{\nu_1, \mu_0\}$, so for the given $x_{[k]}$, it is easy to find an element $u_k \in G_k$ such that $(x_{[k-1]}, u_k), \Big(x_{[k-1]}, x_k - \frac{\mu- \nu_1}{\mu_0 - \nu_1}(a_k - u_k)\Big) \in X$, so that $\phi^{\text{ext}}(x_{[k]})$ can be defined by a formula similar to the one in the second stage (see~\eqref{thirdpieceeqn}).\\
\indent Once the map $\phi^{\text{ext}}$ is defined, we need to check that it is indeed a multi-homomorphism. To carry out this check, we rely on the biaffine theory, and use Theorem~\ref{globalBiaffineInAllPlanes} to ensure that $\phi$ agrees with global biaffine maps on 2-dimensional cross-sections. Thanks to this property, many cocycle-type identities for $\phi$ becomes significantly simpler to prove.}\\

\noindent\chg{\textbf{Finding global biaffine maps.}} Suppose that $I_i = I_1$ if and only if $i \in [r_0]$. We begin by applying Theorem~\ref{globalBiaffineInAllPlanes}. Misusing the notation by still writing $X$ and $V$ for the smaller set and variety produced by that theorem, we may assume that a set $X \subset V$ of size at least $(1-\varepsilon') |V|$ is given, where $\varepsilon' = O(\varepsilon^{\Omega(1)}) + \xi$, $\phi \colon X \to H$ is multiaffine and extends on each axis-aligned plane to a global biaffine map, and $V$ is defined by $s \leq C_k(r + \log_p \xi^{-1})^{D_k}$ multilinear forms, where the first $r$ are the given ones, and the others are $\mathcal{G}'$-supported, where $C_k$ and $D_k$ are the implicit constants appearing in the theorem. Write $V^{[2,s]} = \{x_{[k]} \in G_{[k]} \colon \alpha_2(x_{[k]}) = \nu_2, \dots\}$ for the variety where $\alpha_1$ is not used.\\

\noindent\chg{\textbf{Defining the sets $Y_d$.}} Let $\eta \in (0,1)$. For simplicity of notation assume without loss of generality that $I_1 = [d_0 + 1, k]$. For each $d \in [d_0 + 1, k]$, let $Y_d$ be the set of all $x_{[k]} \in V^{[2,s]} \setminus V$ such that $|X_{x_{[k] \setminus \{d\}}}| \geq (1-\eta) |V_{x_{[k] \setminus \{d\}}}| > 0$.\\ 

\indent\chg{Using quasirandomness of the form $\alpha_1$, we show that each $Y_d$ is very dense inside $V^{[2,s]} \setminus V$.}

\begin{claima*}If $\bias (\alpha_1 - \lambda \cdot \alpha) \leq \frac{\xi}{2}p^{-(k+1)s}$ for every $\lambda \in \mathbb{F}_p^{[2,r_0]}$, then
\[|Y_d|\ \geq (1 - 2p\eta^{-1}\varepsilon' - 2\xi) |V^{[2,s]} \setminus V|\]
for each $d \in [d_0 + 1, k]$.\end{claima*}

\begin{proof}[Proof of Claim A]Without loss of generality $d\in I_i$ if and only if $i\in [s_0]$. Write $\alpha_i(x_{I_i}) = A_i(x_{I_i \setminus \{d\}}) \cdot x_d$ for $i \in [s_0]$. Note that $A_1(x_{I_1 \setminus \{d\}})$ is independent of $A_2(x_{I_2 \setminus \{d\}}), \dots, A_{s_0}(x_{I_{s_0} \setminus \{d\}})$ for most values of $x_{[k] \setminus \{d\}} \in G_{[k] \setminus \{d\}}$. Indeed, for any $\lambda \in \mathbb{F}_p^{[2,s_0]}$, we have
\begin{align*}|G_{[k] \setminus \{d\}}|^{-1}\Big|\Big\{x_{[k] \setminus \{d\}} &\in G_{[k] \setminus \{d\}} \colon A_1(x_{I_1 \setminus \{d\}}) = \sum_{i \in [2,s_0]} \lambda_i A_i(x_{I_{i} \setminus \{d\}})\Big\}\Big|\\
=&\exx_{x_{[k]}} \omega^{\big(A_1(x_{I_1 \setminus \{d\}}) - \sum_{i \in [2,s_0]} \lambda_i A_i(x_{I_{i} \setminus \{d\}})\big) \cdot x_d} \\
=&\exx_{x_{[k]}} \omega^{\alpha_1(x_{I_1}) - \sum_{i \in [2,s_0]} \lambda_i \alpha_i(x_{I_i})}\\
\leq&\exx_{x_{[k] \setminus I_1}} \Big|\exx_{x_{I_1}} \omega^{\alpha_1(x_{I_1}) - \sum_{i \in [2,s_0]} \lambda_i \alpha_i(x_{I_i})}\Big|\\
\leq &\bias \Big(\alpha_1 - \sum_{i \in [2,r_0]} \lambda_i \alpha_i\Big)\\
\leq &\frac{\xi}{2} p^{-(k+1)s},
\end{align*} 
where the second inequality follows from Lemma~\ref{biasHomog} and the third is true by hypothesis.

There are $p^{s_0-1}$ choices of $\lambda$ above, so $|V_{x_{[k] \setminus \{d\}}}| = \frac{1}{p}|V^{[2,s]}_{x_{[k] \setminus \{d\}}}|$ for all but at most $\frac{\xi}{2} p^{-ks} |G_{[k] \setminus \{d\}}|$ elements $x_{[k] \setminus \{d\}} \in G_{[k] \setminus \{d\}}$. Let $Z$ be the set of such $x_{[k] \setminus \{d\}}$. Then
\begin{align}|(V^{[2,s]} \setminus V) \setminus Y_d|\ \leq & \sum_{x_{[k] \setminus \{d\}} \in Z} |(V^{[2,s]} \setminus V)_{x_{[k] \setminus \{d\}}}| \cdot \mathbbm{1}\Big(|X_{x_{[k] \setminus \{d\}}}| \leq (1-\eta) |V_{x_{[k] \setminus \{d\}}}|\Big) + |G_{[k] \setminus \{d\}} \setminus Z| \cdot |G_d|\nonumber\\
= & \sum_{x_{[k] \setminus \{d\}} \in Z} (p-1)|V_{x_{[k] \setminus \{d\}}}| \cdot \mathbbm{1}\Big(|X_{x_{[k] \setminus \{d\}}}| \leq (1-\eta) |V_{x_{[k] \setminus \{d\}}}|\Big) + |G_{[k] \setminus \{d\}} \setminus Z| \cdot |G_d|\nonumber\\
\leq & \sum_{x_{[k] \setminus \{d\}} \in Z} (p-1) \eta^{-1} |V_{x_{[k] \setminus \{d\}}} \setminus X_{x_{[k] \setminus \{d\}}}| + |G_{[k] \setminus \{d\}} \setminus Z| \cdot |G_d|\nonumber\\
\leq &  (p-1)\eta^{-1} |V \setminus X| + \xi p^{-ks}|G_{[k]}| \nonumber\\
\leq & ((p-1)\eta^{-1}\varepsilon' + \xi) |V|,\label{v2sboundvarineq}\end{align}
where we used Lemma~\ref{varsizelemma} to see that $|V| \geq p^{-ks} |G_{[k]}|$.

Note that $V^{[2,s]} \setminus V$ constains a layer of $\alpha$, so $|V^{[2,s]} \setminus V| \geq p^{-ks}|G_{[k]}|$ by Lemma~\ref{varsizelemma}. Lastly, we have that $|V \cap (Z \times G_d)| = \frac{1}{p-1}|(V^{[2,s]} \setminus V) \cap (Z \times G_d)|$ from which we conclude that
\[|V| \,\,\leq \, |V \cap (Z \times G_d)| + \xi p^{-ks} |G_{[k]}|\,\, \leq\, |(V^{[2,s]} \setminus V) \cap (Z \times G_d)| + \xi p^{-ks} |G_{[k]}| \,\,\leq\, 2|V^{[2,s]} \setminus V|.\]
Combining this inequality with~\eqref{v2sboundvarineq}, we complete the proof of the claim. \end{proof}

Let $\mu_0 \in \mathbb{F}_p \setminus \{\nu_1\}$ be arbitrary. Write $V^{\mu_0} =  \{x_{[k]} \in V^{[2,s]} \colon \alpha_1(x_{[k]}) = \mu_0\}$. Let $Y = \bigcap_{d \in I_1} Y_d$. Applying Claim A to each $Y_d$, we have that $|Y|\ \geq (1 - 2kp\eta^{-1}\varepsilon' - 2k\xi) |V^{[2,s]} \setminus V|$. Let 
\[\tilde{Y} = \Big\{x_{[k]} \in Y \colon \alpha_1(x_{[k]}) = \mu_0, |Y_{x_{[k-1]}}| \geq (1 - \eta) |(V^{[2,s]} \setminus V)_{x_{[k-1]}}|\Big\}.\] 
Using the quasirandomness of $\alpha_1$ with respect to the other forms as in the proof of the claim, one has that $|V_{x_{[k-1]}}| = |V^{\mu_0}_{x_{[k-1]}}| = \frac{1}{p}|V^{[2,s]}_{x_{[k-1]}}|$ for all but at most $\frac{\xi}{2} p^{-ks} |G_{[k-1]}|$ elements $x_{[k-1]} \in G_{[k-1]}$. Let $Q$ be the set of such $x_{[k-1]}$. We deduce that
\begin{align}|V^{\mu_0}\setminus \tilde{Y}|\ \leq &\sum_{x_{[k-1]} \in Q} |V^{\mu_0}_{x_{[k-1]}}|\hspace{2pt} \mathbbm{1}\Big(|Y_{x_{[k-1]}}| \leq (1 - \eta) |(V^{[2,s]} \setminus V)_{x_{[k-1]}}|\Big) + \sum_{x_{[k-1]} \in G_{[k-1]} \setminus Q} |G_k|\nonumber\\
\leq& \sum_{x_{[k-1]} \in Q} \frac{1}{p-1}|(V^{[2,s]} \setminus V)_{x_{[k-1]}}| \hspace{2pt} \mathbbm{1}\Big(|Y_{x_{[k-1]}}| \leq (1 - \eta) |(V^{[2,s]} \setminus V)_{x_{[k-1]}}|\Big) + \sum_{x_{[k-1]} \in G_{[k-1]} \setminus Q} |G_k|\nonumber\\
\leq & \sum_{x_{[k-1]} \in Q} \frac{1}{p-1}\eta^{-1}|((V^{[2,s]} \setminus V) \setminus Y)_{x_{[k-1]}}| + \frac{\xi}{2} p^{-ks}|G_{[k]}|\nonumber\\
\leq & \eta^{-1} (2kp\eta^{-1}\varepsilon' + 2k\xi) \frac{1}{p-1} |V^{[2,s]} \setminus V| + \frac{\xi}{2}|V^{\mu_0}|.\label{vmu0v2sbndineq}\end{align}
Since $|V^{\mu_0}| \geq p^{-ks}|G_{[k]}|$ we get
\[|V^{[2,s]} \setminus V| \,\,\leq\, |(V^{[2,s]} \setminus V) \cap (Q \times G_k)| + \xi p^{-ks}|G_{[k]}| \,\,=\, (p-1)|V^{\mu_0} \cap(Q\times G_k)| + \xi p^{-ks}|G_{[k]}|\,\, \leq \,p |V^{\mu_0}|.\]
Together with~\eqref{vmu0v2sbndineq}, this inequality implies that $|\tilde{Y}|\ \geq (1 - 4kp\eta^{-2}\varepsilon' - 5k\eta^{-1}\xi) |V^{\mu_0}|$.\\

\noindent\chg{\textbf{Application of Theorem~\ref{connectedVeryDensePiece}.}} Applying Theorem~\ref{connectedVeryDensePiece} to $\tilde{Y} \subset V^{\mu_0}$, we find a lower order variety $U \subset G_{[k-1]}$ of codimension $O((r + \log_p \xi^{-1})^{O(1)})$, a subset $Z \subset \tilde{Y} \cap (U \times G_k)$, and an element $a_{[k]} \in Z$ such that $V^{\mu_0} \cap (U \times G_k) \not= \emptyset$,
\begin{equation}\label{1codimExtnzdensEqn}|Z|\ \geq (1 - \tilde{\eta})|V^{\mu_0} \cap (U \times G_k)|,\end{equation}
where $\tilde{\eta} \leq O((\eta^{-2}{\varepsilon'})^{\Omega(1)}) - O((\eta^{-1}\xi)^{\Omega(1)})$, and
\begin{equation}(\forall x_{[k]} \in Z)\,(\forall i \in [k]) \hspace{3pt} (x_{[i]}, a_{[i+1,k]}) \in  Z.\label{zbelongalldirmove}\end{equation}
Moreover, there is $\delta > 0$ such that
\begin{equation}(\forall x_{[k-1]} \in U \cap V^{[k-1]}) \hspace{3pt}|V^{\mu_0}_{x_{[k-1]}}| = \delta |G_k|,\label{samecolsizestildez}\end{equation} where \chg{we recall the nonstandard notation} 
\[V^{[k-1]} = \{x_{[k-1]} \in G_{[k-1]} \colon (\forall i \in [s])\ I_i \subset [k-1]\implies \alpha_i(x_{I_i}) = \nu_i\}.\]
Let $h_0 \colon G_{[k] \setminus I_1}\to H$ be an arbitrary multiaffine map. We now define a multi-homomorphism $\phi^{\text{ext}}$ on a very dense subset of $V^{[2,s]} \cap (U \times G_k)$. Later we shall strengthen this new map to a multiaffine one.\\

\noindent\chg{\textbf{Definition of the domain.}} The domain of $\phi^{\text{ext}}$ will be the set $\tilde{Z} \subset V^{[2,s]} \cap (U \times G_k)$, defined as follows. Note that
\begin{align*}V^{[2,s]} \cap (U \times G_k) = \Big(V^{[2,s]} \cap (U \times G_k) \cap \alpha_1^{-1}(\nu_1)\Big)\,&\cup\,\Big(V^{[2,s]} \cap (U \times G_k) \cap \alpha_1^{-1}(\mu_0)\Big)\\
&\cup\,\Big(V^{[2,s]} \cap (U \times G_k) \cap \alpha_1^{-1}(\mathbb{F}_p \setminus \{\nu_1, \mu_0\})\Big)\end{align*}
is a partition $V^{[2,s]} \cap (U \times G_k)$. We define $\tilde{Z}$ on each of these pieces by setting $\tilde{Z} \cap \alpha_1^{-1}(\nu_1) = X \cap (U \times G_k)$, $\tilde{Z} \cap\alpha_1^{-1}(\mu_0)= Z$ and $\tilde{Z} \cap\alpha_1^{-1}(\mathbb{F}_p \setminus \{\mu_0, \nu_1\}) = Y \cap (Z_{a_k} \times G_k) \cap \alpha_1^{-1}(\mathbb{F}_p \setminus \{\mu_0, \nu_1\})$.\\

\noindent\chg{\textbf{Checking the density of the domain.}} Before constructing the map $\phi^{\text{ext}} \colon \tilde{Z} \to H$ we show that $\tilde{Z}$ is very dense in $V^{[2,s]} \cap (U \times G_k)$. To this end, first observe that $Z_{a_k}$ is very dense in $U \cap V^{[k-1]}$. Indeed, by~\eqref{samecolsizestildez},
\begin{align}|Z_{a_k}| \cdot \delta |G_k| \,\,\geq\, |Z| \,\,\geq&\, (1 - \tilde{\eta})|V^{\mu_0} \cap (U \times G_k)|\nonumber\\
 =&\, (1 - \tilde{\eta}) |U \cap V^{[k-1]}| \cdot \delta |G_k|.\label{zakbound}\end{align}

 Observe that when $x_{[k-1]} \in Z_{a_k}$, then we have $(x_{[k-1]}, a_k) \in \tilde{Y} \subset Y_k$, so $|Y_{x_{[k-1]}}|\ \geq (1 - \eta) |(V^{[2,s]} \setminus V)_{x_{[k-1]}}|$ and $|X_{x_{[k-1]}}|\ \geq (1-\eta) |V_{x_{[k-1]}}| > 0$. Using this, we conclude that 
\begin{align}|\tilde{Z}|\ =&\,|X \cap (U \times G_k)|\, + \,|Z|\, +\, |Y \cap (Z_{a_k} \times G_k) \cap \alpha_1^{-1}(\mathbb{F}_p \setminus \{\mu_0, \nu_1\})|\nonumber\\
\geq&\,|Z| \,+\,\sum_{x_{[k-1]} \in Z_{a_k}} \Big(|X_{x_{[k-1]}}|\,+\,|(Y \cap \alpha_1^{-1}(\mathbb{F}_p \setminus \{\mu_0, \nu_1\}))_{x_{[k-1]}}|\Big)\nonumber\\
\geq&\,|Z| \,+\,\sum_{x_{[k-1]} \in Z_{a_k}} \Big(|X_{x_{[k-1]}}|\,+\,|Y_{x_{[k-1]}}|\,-\,|\alpha_1^{-1}(\mu_0)_{x_{[k-1]}}|\Big)\nonumber\\
\geq&\,|Z| \,+\,\sum_{x_{[k-1]} \in Z_{a_k}} \Big((1-\eta) |V_{x_{[k-1]}}|\,+\, (1 - \eta) |(V^{[2,s]} \setminus V)_{x_{[k-1]}}|\,-\,|V^{\mu_0}_{x_{[k-1]}}|\Big)\nonumber\\
\geq&\,(1 - \tilde{\eta})|V^{\mu_0} \cap (U \times G_k)| \, +  \,(1-\tilde{\eta})\sum_{x_{[k-1]} \in U \cap V^{[k-1]}} \Big((1-\eta) |V_{x_{[k-1]}}|\,+\, (1 - \eta) |(V^{[2,s]} \setminus V)_{x_{[k-1]}}|\,-\,|V^{\mu_0}_{x_{[k-1]}}|\Big)\nonumber\\
&\hspace{5cm}\text{(by~\eqref{1codimExtnzdensEqn},~\eqref{samecolsizestildez} and~\eqref{zakbound})}\nonumber\\
\geq & \Big(1 - O(\eta) - O((\eta^{-2}{\varepsilon'})^{\Omega(1)}) - O((\eta^{-1}\xi)^{\Omega(1)})\Big)|V^{[2,s]} \cap (U \times G_k)|.\label{1codimExtnTzdensEqn}\\
&\hspace{5cm}\text{(by~\eqref{samecolsizestildez})}\nonumber\end{align}
\phantom{a}\\
We now proceed to define $\phi^{\text{ext}}$ on each of the three pieces of $\tilde{Z}$.\\

\noindent\textbf{Definition on the first piece.} When $x_{[k]} \in \tilde{Z} \cap \alpha_1^{-1}(\nu_1)$ then $x_{[k]} \in X$ and we set $\phi^{\text{ext}}(x_{[k]}) = \phi(x_{[k]})$.\\[3pt]
\noindent\textbf{Definition on the second piece.} When $x_{[k]} \in \tilde{Z} \cap \alpha_1^{-1}(\mu_0)$, we have that $x_{[k]} \in Z$, so by~\eqref{zbelongalldirmove} we also have $(x_{[i]}, a_{[i+1,k]}) \in  Z$ for each $i \in [k]$. The definition of $\phi^{\text{ext}}$ is more involved in this case and we proceed in stages. In particular, in order to specify $\phi^{\text{ext}}(x_{[k]})$ we need to define $\phi^{\text{ext}}(x_{[i]}, a_{[i+1,k]})$ for all $i \in [d_0, k]$. First, we define $\phi^{\text{ext}}$ on the set $\tilde{Z} \cap \alpha_1^{-1}(\mu_0) \cap (Z_{a_{[d_0 + 1, k]}} \chg{\times \{a_{[d_0 + 1, k]}\}})$, which is the set of points in the second piece of $Z$ that are of the form $(x_{[d_0]}, a_{[d_0 + 1, k]})$. We set $\phi^{\text{ext}}(x_{[d_0]}, a_{[d_0 + 1, k]}) = h_0(x_{[d_0]})$. We now consider points of the form $(x_{[i]}, a_{[i+1,k]})$ where $i$ grows by one in each step.\\
\indent Suppose that we have defined $\phi^{\text{ext}}$ on $\tilde{Z} \cap \alpha_1^{-1}(\mu_0) \cap (Z_{a_{[i + 1, k]}} \times G_{[i + 1, k]})$ for some $i \in [d_0, k-1]$ and that $(x_{[i+1]}, a_{[i+2,k]}) \in \tilde{Z} \cap \alpha_1^{-1}(\mu_0)$ is given and that $x_{i+1} \not= a_{i+1}$. We set
\[\phi^{\text{ext}}(x_{[i+1]}, a_{[i+2,k]}) = \phi^{\text{ext}}(x_{[i]}, a_{[i+1,k]}) + \phi(x_{[i]}, u_{i+1} + x_{i+1} - a_{i+1}, a_{[i+2,k]}) - \phi(x_{[i]}, u_{i+1}, a_{[i+2,k]}),\]
where $u_{i+1}$ is chosen so that the last two points in the argument of $\phi$ lie in $X$. Since $(x_{[i+1]}, a_{[i+2,k]}) \in Z \subset \tilde{Y} \subset Y_{i+1}$, we know that $|X_{x_{[i]}, a_{[i+2,k]}}| \geq (1-\eta) |V_{x_{[i]}, a_{[i+2,k]}}|$, so we may find such a $u_{i+1}$ as long as $\eta < \frac{1}{10}$. In fact there are at least $\frac{1}{2}|V_{x_{[i]}, a_{[i+2,k]}}| \geq \frac{1}{2}p^{-s} |G_{i+1}|$ such choices of $u_{i+1}$. Note also that this is well-defined (i.e.\ independent of a particular choice of $u_{i+1}$) since $\phi$ is a multiaffine map. Also note that the displayed inequality is still true if $x_{i+1} = a_{i+1}$.\\[3pt]
\noindent\textbf{Definition on the third piece.} Finally, when $x_{[k]} \in Y$ is such that $\alpha_1(x_{[k]}) = \mu \not= \mu_0, \nu_1$ and $(x_{[k-1]}, a_k) \in Z$, we put
\begin{equation}\phi^{\text{ext}}(x_{[k]}) = \frac{\mu- \nu_1}{\mu_0 - \nu_1}(\phi^{\text{ext}}(x_{[k-1]}, a_k) - \phi(x_{[k-1]}, u_k)) + \phi\Big(x_{[k-1]}, x_k - \frac{\mu- \nu_1}{\mu_0 - \nu_1}(a_k - u_k)\Big),\label{thirdpieceeqn}\end{equation}
where again we choose $u_k$ so that $(x_{[k-1]}, u_k), \Big(x_{[k-1]}, x_k - \frac{\mu- \nu_1}{\mu_0 - \nu_1}(a_k - u_k)\Big) \in X$. As above, we know that $x_{[k]} \in Y_k$, so $|X_{x_{[k-1]}}| \geq (1-\eta)|V_{x_{[k-1]}}|$, which implies that there is a desired $u_k$ (in fact there are at least $\frac{1}{2}|V_{x_{[k-1]}}| \geq \frac{1}{2}p^{-s} |G_{k}|$ such choices of $u_{k}$) and $\phi^{\text{ext}}(x_{[k]})$ is well-defined (i.e.\ we get the same value for any $u_k$ with the required property).\\

Note that, allowing a misuse of notation, we may use the same equation for $x_{[k]} \in V$, since the right-hand-side reduces to $\phi(x_{[k]})$ (even though $(x_{[k-1]}, a_k)$ might not belong to $Z$ in this case). Furthermore, when $\mu = \mu_0$, this definition coincides with the definition from the second step (with $i = k-1$). Hence, equality~\eqref{thirdpieceeqn} holds for all $x_{[k]} \in \tilde{Z}$ with the convention that if $\mu \not= \nu_1$ (which is equivalent to $\frac{\mu- \nu_1}{\mu_0 - \nu_1} \not= 0$) then $(x_{[k-1]}, a_k)$ lies in the domain of $\phi^{\text{ext}}$. If $\mu = \nu_1$ then we do not claim anything about the point $(x_{[k-1]}, a_k)$, but we keep it to simplify the notation as $\phi^{\text{ext}}(x_{[k-1]}, a_k)$ makes no contribution to the expression~\eqref{thirdpieceeqn} in this case.\\
\indent Also, expanding the definitions of $\phi^{\text{ext}}$ we get the formula in the case \textbf{(ii)} of the conclusion of the theorem and we get the desired number of possible choices of $u_{[d_0 + 1,k]}$.\\

We now prove that this extension is a multi-homomorphism.\\

Let $d \in [k]$, and for $i\in[4]$ let $(x_{[k] \setminus \{d\}}, y^i_d) \in \tilde{Z}$ be points such that $y^1_d + y^2_d = y^3_d + y^4_d$. Write $\sgn \colon [4] \to \{-1,1\}$ for the function defined by $\sgn(1) = \sgn(2) = 1, \sgn(3) = \sgn(4) = -1$, so that $\sum_{i \in [4]} \sgn(i) y^i_d = 0$. Let $\mu_i$ be the value of $\alpha_1$ at the $i$\textsuperscript{th} point. That is,
\[\mu_i = \alpha_1\Big((x_{[k] \setminus \{d\}}, y^i_d)|_{I_1}\Big) = \begin{cases}\alpha_1(x_{I_1 \setminus \{d\}}, y^i_d),\hspace{1cm}&\text{when }d \geq d_0 + 1\\\alpha_1(x_{I_1}),&\text{when }d \leq d_0.\end{cases}\]
Since $\alpha_1$ is multilinear, these values also satisfy $\sum_{i \in [4]} \sgn(i) \mu_i = 0$. Note also that we may assume that not all of the four given points belong to $X$, as $\phi^{\text{ext}}  = \phi$ on $X$ and we know that $\phi$ is multiaffine. Hence, at least one $\mu_i$ is not equal to $\nu_1$.\\

\noindent\textbf{The map $\phi^{\mathrm{ext}}$ is a homomorphism in direction $k$.} Assume first that $d = k$. By definition of $\phi^{\text{ext}}$ and the fact that~\eqref{thirdpieceeqn} holds for all $x_{[k]} \in \tilde{Z}$, there are $u^1_k, u^2_k, u^3_k, u^4_k \in G_k$ (if $\mu_i = \nu_1$ for some $i$ we continue with the misuse of notation mentioned earlier) such that 

\begin{align*}&\sum_{i \in [4]} \sgn(i) \phi^{\text{ext}}(x_{[k-1]}, y^i_k) = \sum_{i \in [4]} \sgn(i) \frac{\mu_i - \nu_1}{\mu_0 - \nu_1} \phi^{\text{ext}}(x_{[k-1]}, a_k)\\
&\hspace{2cm} - \sum_{i \in [4]} \sgn(i) \frac{\mu_i - \nu_1}{\mu_0 - \nu_1}\phi(x_{[k-1]}, u^i_k) + \sum_{i \in [4]} \sgn(i) \phi\Big(x_{[k-1]}, y^i_k - \frac{\mu_i - \nu_1}{\mu_0 - \nu_1}(a_k - u^i_k)\Big).\end{align*}
Since the map $z_k \mapsto \phi(x_{[k-1]}, z_k)$ is a restriction of an affine map with domain $G_k$, and $\sum_{i \in [4]} \sgn(i) \frac{\mu_i - \nu_1}{\mu_0 - \nu_1} = 0$, the whole expression is zero, as desired.\\

\noindent\textbf{The map $\phi^{\mathrm{ext}}$ is a homomorphism in directions $[d_0 +1, k-1]$.} Next, assume that $d \in [d_0 + 1, k-1]$. By downwards induction on $d' \in [d, k-1]$ we show the following claim.

\begin{claimb*}Let $d' \in [d, k-1]$. There are elements $t_d, u^{[4]}_d, v^{[4]}_d$ in $G_{d}$ such that $(x_{[d'] \setminus \{d\}}, u^i_d, a_{[d' + 1,k]}),$ $(x_{[d'] \setminus \{d\}}, v^i_d, a_{[d'+1,k]}),$ $(x_{[d'] \setminus \{d\}},$ $u^i_d + v^i_d + t_d - y^i_d,$ $a_{[d' + 1,k]}) \in X$ for each $i \in [4]$ with $\mu_i \not= \nu_1$ and 

\begin{align}&\sum_{i \in [4]} \sgn(i) \phi^{\text{ext}}(x_{[k] \setminus \{d\}}, y^i_d) = \sum_{i \in [4]} \sgn(i) \frac{\mu_i - \nu_1}{\mu_0 - \nu_1} \phi^{\text{ext}}(x_{[d'] \setminus \{d\}}, y^i_d, a_{[d' + 1, k]})\nonumber\\
&\hspace{2cm} - \sum_{i \in [4]} \sgn(i) \frac{\mu_i - \nu_1}{\mu_0 - \nu_1}\Big(\phi(x_{[d'] \setminus \{d\}}, u^i_d, a_{[d' + 1,k]}) + \phi(x_{[d'] \setminus \{d\}}, v^i_d, a_{[d'+1,k]})\nonumber\\
&\hspace{8cm} - \phi(x_{[d'] \setminus \{d\}}, u^i_d + v^i_d + t_d - y^i_d, a_{[d' + 1,k]})\Big),\label{claimbext1codimlongformula}\end{align}
where $(x_{[d'] \setminus \{d\}}, y^i_d, a_{[d' + 1, k]}) \in Z$ for each $i \in [4]$ with $\mu_i \not= \nu_1$ (again, when $\mu_i = \nu_1$, we continue with the misuse of notation and write formal expressions $\phi^{\text{ext}}(x_{[d'] \setminus \{d\}}, y^i_d, a_{[d' + 1, k]})$, $\phi(x_{[d'] \setminus \{d\}}, u^i_d, a_{[d' + 1,k]})$, etc.,\ which are multiplied by a zero scalar and thus make no contribution).\end{claimb*}

\noindent\chg{\textbf{Remark.} The goal of this claim is to move from values $\phi^{\text{ext}}(x_{[k] \setminus \{d\}}, y^i_d)$ to $\phi^{\text{ext}}(x_{[d -1]}, y^i_d, a_{[d + 1, k]})$, which is the case $d' = d$. The discrepancy between linear combinations of these values is given by the final sum occurring in expression~\eqref{claimbext1codimlongformula} above, which is itself a linear combinations of values of $\phi$ at points in $X$. This case of the formula will be used in~\eqref{dirdhomcheck1codimextn} to show that $\phi^{\mathrm{ext}}$ is a homomorphism in direction $d$, as required.} 

\begin{proof}[Proof of Claim B]We begin the proof by dealing with the base case $d' = k-1$. By definition, there are $u^{[4]}_k$ in $G_k$ such that 

\begin{align*}&\sum_{i \in [4]} \sgn(i) \phi^{\text{ext}}(x_{[k] \setminus \{d\}}, y^i_d) = \sum_{i \in [4]} \sgn(i) \frac{\mu_i - \nu_1}{\mu_0 - \nu_1} \phi^{\text{ext}}(x_{[k-1] \setminus \{d\}}, y^i_d, a_k)\\
&\hspace{2cm} - \sum_{i \in [4]} \sgn(i) \frac{\mu_i - \nu_1}{\mu_0 - \nu_1}\phi(x_{[k-1] \setminus \{d\}}, y^i_d, u^i_k) + \sum_{i \in [4]} \sgn(i) \phi\Big(x_{[k-1] \setminus \{d\}}, y^i_d, x_k - \frac{\mu_i - \nu_1}{\mu_0 - \nu_1}(a_k - u^i_k)\Big),\end{align*}
where $(x_{[k-1] \setminus \{d\}}, y^i_d, a_k) \in Z \subset V^{\mu_0}$ holds when $\mu_i \not= \nu_1$ (recall that there is at least one such $i \in [4]$). In particular, for such an index $i$, $\alpha_1(x_{[d_0 + 1, k-1] \setminus \{d\}}, y^i_d, a_k) = \mu_0$. Note also that $V^{\mu_0}_{x_{[k-1] \setminus \{d\}}, a_k} \not= \emptyset$. Moreover, since $(x_{[k-1] \setminus \{d\}}, y^i_d, a_k) \in Z \subset Y_d$, we also have $V_{x_{[k-1] \setminus \{d\}}, a_k} \not= \emptyset$. Thus, being non-empty, the sets $V_{x_{[k-1] \setminus \{d\}}, a_k}$ and $V^{\mu_0}_{x_{[k-1] \setminus \{d\}}, a_k}$ are cosets of the same subspace and there exists some element $t_d \in G_d$ such that $V_{x_{[k-1] \setminus \{d\}}, a_k} = V^{\mu_0}_{x_{[k-1] \setminus \{d\}}, a_k} - t_d$. In particular, $(x_{[k-1] \setminus \{d\}}, y^i_d - t_d, a_k) \in V$ for each $i \in [4]$ such that $\mu_i \not= \nu_1$.\\

For fixed $x_{[k-1] \setminus \{d\}}$, there is a global biaffine map $\phi^{\text{glob}}_{x_{[k-1] \setminus \{d\}}}$ that extends $\phi$ from $X_{x_{[k-1] \setminus \{d\}}}$ to $G_d \times G_k$. Thus, we get
\begin{align*}\sum_{i \in [4]}& \sgn(i) \phi^{\text{ext}}(x_{[k] \setminus \{d\}}, y^i_d)\\ 
= &\sum_{i \in [4]} \sgn(i) \frac{\mu_i - \nu_1}{\mu_0 - \nu_1} \phi^{\text{ext}}(x_{[k-1] \setminus \{d\}}, y^i_d, a_k)\\
&\hspace{2cm} - \sum_{i \in [4]} \sgn(i) \frac{\mu_i - \nu_1}{\mu_0 - \nu_1}\phi^{\text{glob}}_{x_{[k-1] \setminus \{d\}}}(y^i_d, u^i_k) + \sum_{i \in [4]} \sgn(i) \phi^{\text{glob}}_{x_{[k-1] \setminus \{d\}}}\Big(y^i_d, x_k - \frac{\mu_i - \nu_1}{\mu_0 - \nu_1}(a_k - u^i_k)\Big)\\
= &\sum_{i \in [4]} \sgn(i) \frac{\mu_i - \nu_1}{\mu_0 - \nu_1} \phi^{\text{ext}}(x_{[k-1] \setminus \{d\}}, y^i_d, a_k)\\
&\hspace{2cm} - \sum_{i \in [4]} \sgn(i) \frac{\mu_i - \nu_1}{\mu_0 - \nu_1}\phi^{\text{glob}}_{x_{[k-1] \setminus \{d\}}}(y^i_d, a_k)\\
= &\sum_{i \in [4]} \sgn(i) \frac{\mu_i - \nu_1}{\mu_0 - \nu_1} \phi^{\text{ext}}(x_{[k-1] \setminus \{d\}}, y^i_d, a_k)\\
&\hspace{2cm} - \sum_{i \in [4]} \sgn(i) \frac{\mu_i - \nu_1}{\mu_0 - \nu_1}\phi^{\text{glob}}_{x_{[k-1] \setminus \{d\}}}(y^i_d - t_d, a_k)\\
= &\sum_{i \in [4]} \sgn(i) \frac{\mu_i - \nu_1}{\mu_0 - \nu_1} \phi^{\text{ext}}(x_{[k-1] \setminus \{d\}}, y^i_d, a_k)\\
&\hspace{1cm} - \sum_{i \in [4]} \sgn(i) \frac{\mu_i - \nu_1}{\mu_0 - \nu_1}\Big(\phi(x_{[k-1] \setminus \{d\}}, \tilde{u}^i_d, a_k) + \phi(x_{[k-1] \setminus \{d\}}, \tilde{v}^i_d, a_k) - \phi(x_{[k-1] \setminus \{d\}}, \tilde{u}^i_d + \tilde{v}^i_k + t_d - y^i_d, a_k)\Big),\end{align*}
where in the last line, for each $i \in [4]$ with $\mu_i \not= \nu_1$, we found elements $\tilde{u}^i_d, \tilde{v}^i_d \in X_{x_{[k-1] \setminus \{d\}}, a_k}$ such that $\tilde{u}^i_d + \tilde{v}^i_k + t_d - y^i_d \in X_{x_{[k-1] \setminus \{d\}}, a_k}$, using the facts that $(x_{[k-1] \setminus \{d\}}, y^i_d - t_d, a_k) \in V$, $(x_{[k-1] \setminus \{d\}}, y^i_d, a_k) \in Y_d$, and $\eta < 1/10$.\\

Now assume that the claim holds for some $d' \in [d+1, k-1]$. Then there are elements $t_d, u^{[4]}_d, v^{[4]}_d$ in $G_{d}$ with the properties claimed. By definition of $\phi^{\text{ext}}$ and the fact that $(x_{[d'] \setminus \{d\}}, y^i_d, a_{[d' + 1, k]}) \in Z$ when $\mu_i \not= \nu_1$, there are $w^{[4]}_{d'}$ in $G_{d'}$ such that 
\begin{align*}&\sum_{i \in [4]} \sgn(i) \phi^{\text{ext}}(x_{[k] \setminus \{d\}}, y^i_d) = \sum_{i \in [4]} \sgn(i) \frac{\mu_i - \nu_1}{\mu_0 - \nu_1} \Big(\phi^{\text{ext}}(x_{[d'-1] \setminus \{d\}}, y^i_d, a_{[d', k]})\\
&\hspace{5cm} + \phi(x_{[d'-1] \setminus \{d\}}, y^i_d, x_{d'} + w^i_{d'} - a_{d'}, a_{[d'+1, k]}) - \phi(x_{[d'-1] \setminus \{d\}}, y^i_d, w^i_{d'}, a_{[d'+1, k]})\Big)\\
&\hspace{2cm} - \sum_{i \in [4]} \sgn(i) \frac{\mu_i - \nu_1}{\mu_0 - \nu_1}\Big(\phi(x_{[d'] \setminus \{d\}}, u^i_d, a_{[d' + 1,k]}) + \phi(x_{[d'] \setminus \{d\}}, v^i_d, a_{[d'+1,k]})\\
&\hspace{10cm} - \phi(x_{[d'] \setminus \{d\}}, u^i_d + v^i_d + t_d - y^i_d, a_{[d' + 1,k]})\Big),\end{align*}
where on the right hand side the arguments of $\phi^{\text{ext}}$ lie in $Z$ and the arguments of $\phi$ lie in $X$ when $\mu_i \not= \nu_1$ (when $\mu_i = \nu_1$ scalar $\frac{\mu_i - \nu_1}{\mu_0 - \nu_1}$ is zero so there is no contribution for such indices $i$).\\
\indent As before, since $(x_{[d'-1] \setminus \{d\}}, y^i_d, a_{[d', k]}) \in Z \subset V^{\mu_0} \cap Y_d$ when $\mu_i \not= \nu_1$, we may pick $\tilde{t}_d \in G_d$ such that $(x_{[d'-1] \setminus \{d\}}, y^i_d - \tilde{t}_d, a_{[d', k]}) \in V$ for each $i \in [4]$ with $\mu_i \not= \nu_1$. Using the global biaffine map $\phi^{\text{glob}} = \phi^{\text{glob}}_{x_{[d'-1] \setminus \{d\}}, a_{[d'+1, k]}}$ on $G_{d} \times G_{d'}$, we get
\begin{align*}&\sum_{i \in [4]} \sgn(i) \phi^{\text{ext}}(x_{[k] \setminus \{d\}}, y^i_d)\\
&\hspace{1cm}=\sum_{i \in [4]} \sgn(i) \frac{\mu_i - \nu_1}{\mu_0 - \nu_1} \phi^{\text{ext}}(x_{[d'-1] \setminus \{d\}}, y^i_d, a_{[d', k]})\\
&\hspace{2cm}+\sum_{i \in [4]} \sgn(i) \frac{\mu_i - \nu_1}{\mu_0 - \nu_1} \Big(\phi^{\text{glob}}(y^i_d, x_{d'} + w^i_{d'} - a_{d'}) - \phi^{\text{glob}}(y^i_d, w^i_{d'})\\
&\hspace{3cm} - \phi^{\text{glob}}(u^i_d, x_{d'}) - \phi^{\text{glob}}(v^i_d, x_{d'}) + \phi^{\text{glob}}(u^i_d + v^i_d + t_d - y^i_d, x_{d'})\Big)\\
&\hspace{1cm}= \sum_{i \in [4]} \sgn(i) \frac{\mu_i - \nu_1}{\mu_0 - \nu_1} \phi^{\text{ext}}(x_{[d'-1] \setminus \{d\}}, y^i_d, a_{[d', k]})\\
&\hspace{2cm}-\sum_{i \in [4]} \sgn(i) \frac{\mu_i - \nu_1}{\mu_0 - \nu_1}\phi^{\text{glob}}(y^i_d, a_{d'})\\
&\hspace{1cm}= \sum_{i \in [4]} \sgn(i) \frac{\mu_i - \nu_1}{\mu_0 - \nu_1} \phi^{\text{ext}}(x_{[d'-1] \setminus \{d\}}, y^i_d, a_{[d', k]})\\
&\hspace{2cm}-\sum_{i \in [4]} \sgn(i) \frac{\mu_i - \nu_1}{\mu_0 - \nu_1}\phi^{\text{glob}}(y^i_d - \tilde{t}_d, a_{d'})\\
&\hspace{1cm}= \sum_{i \in [4]} \sgn(i) \frac{\mu_i - \nu_1}{\mu_0 - \nu_1} \phi^{\text{ext}}(x_{[d'-1] \setminus \{d\}}, y^i_d, a_{[d', k]})\\
&\hspace{2cm}-\sum_{i \in [4]} \sgn(i) \frac{\mu_i - \nu_1}{\mu_0 - \nu_1}\Big(\phi(x_{[d'-1] \setminus \{d\}}, \tilde{u}^i_d, a_{[d',k]})+ \phi(x_{[d'-1] \setminus \{d\}}, \tilde{v}^i_k, a_{[d',k]})\\
&\hspace{7cm} - \phi(x_{[d'-1] \setminus \{d\}}, \tilde{u}^i_d + \tilde{v}^i_k + \tilde{t}_d - y^i_d, a_{[d',k]})\Big),\end{align*}
where in the last line for each $i \in [4]$ with $\mu_i \not= \nu_1$ we again found elements $\tilde{u}^i_d, \tilde{v}^i_d \in X_{x_{[d'-1] \setminus \{d\}}, a_{[d',k]}}$ such that $\tilde{u}^i_d + \tilde{v}^i_k + \tilde{t}_d - y^i_d \in X_{x_{[d'-1] \setminus \{d\}}, a_{[d',k]}}$, using the facts that $(x_{[d'-1] \setminus \{d\}}, y^i_d, a_{[d', k]}) \in Y_d$ and $\eta < 1/10$.\end{proof}

Now that the claim has been proved, apply it with $d' = d$ to get 
\begin{align}&\sum_{i \in [4]} \sgn(i) \phi^{\text{ext}}(x_{[k] \setminus \{d\}}, y^i_d)\nonumber\\ 
&\hspace{1cm}= \sum_{i \in [4]} \sgn(i) \frac{\mu_i - \nu_1}{\mu_0 - \nu_1} \phi^{\text{ext}}(x_{[d-1]}, y^i_d, a_{[d + 1, k]})\nonumber\\
&\hspace{5cm} - \sum_{i \in [4]} \sgn(i) \frac{\mu_i - \nu_1}{\mu_0 - \nu_1}\Big(\phi(x_{[d-1]}, u^i_d, a_{[d + 1,k]}) + \phi(x_{[d-1]}, v^i_d, a_{[d+1,k]})\nonumber\\
&\hspace{10cm} - \phi(x_{[d-1]}, u^i_d + v^i_d + t_d - y^i_d, a_{[d + 1,k]})\Big)\nonumber\\
&\hspace{1cm} = \sum_{i \in [4]} \sgn(i) \frac{\mu_i - \nu_1}{\mu_0 - \nu_1} \Big(\phi^{\text{ext}}(x_{[d-1]}, a_{[d, k]}) + \phi(x_{[d-1]}, y^i_d + w^i_d - a_d, a_{[d+1,k]}) - \phi(x_{[d-1]}, w^i_d, a_{[d+1,k]})\Big)\nonumber\\
&\hspace{2cm} - \sum_{i \in [4]} \sgn(i) \frac{\mu_i - \nu_1}{\mu_0 - \nu_1}\Big(\phi(x_{[d-1]}, u^i_d, a_{[d + 1,k]}) + \phi(x_{[d-1]}, v^i_d, a_{[d+1,k]})\nonumber\\
&\hspace{10cm}  - \phi(x_{[d-1]}, u^i_d + v^i_d + t_d - y^i_d, a_{[d + 1,k]})\Big)\nonumber\\
&\hspace{1cm} = \Big(\sum_{i \in [4]} \sgn(i) \frac{\mu_i - \nu_1}{\mu_0 - \nu_1}\Big) \phi^{\text{ext}}(x_{[d-1]}, a_{[d, k]})\nonumber\\
&\hspace{2cm} - \sum_{i \in [4]} \sgn(i) \frac{\mu_i - \nu_1}{\mu_0 - \nu_1}\Big(\phi(x_{[d-1]}, w^i_d, a_{[d+1,k]}) - \phi(x_{[d-1]}, y^i_d + w^i_d - a_d, a_{[d+1,k]})\nonumber\\
&\hspace{3cm} + \phi(x_{[d-1]}, u^i_d, a_{[d + 1,k]}) + \phi(x_{[d-1]}, v^i_d, a_{[d+1,k]}) - \phi(x_{[d-1]}, u^i_d + v^i_d + t_d - y^i_d, a_{[d + 1,k]})\Big)\nonumber\\
&\hspace{1cm}=0\label{dirdhomcheck1codimextn}
\end{align}
since $\sum_{i \in [4]} \sgn(i) \frac{\mu_i - \nu_1}{\mu_0 - \nu_1}  = 0$ and $\phi$ is the restriction of a global affine map in each direction, in particular in direction $d$. (The second equality above uses the definition of $ \phi^{\text{ext}}(x_{[d-1]}, y^i_d, a_{[d + 1, k]})$, where $(x_{[d-1]}, y^i_d, a_{[d + 1, k]}) \in Z$, with an auxiliary element $w_d^i \in G_d$.)\\

\noindent\textbf{The map $\phi^{\mathrm{ext}}$ is a homomorphism in the directions $[d_0]$.} Finally, suppose that $d \in [d_0]$. Recall that in this case $\mu_1 = \mu_2 = \mu_3 = \mu_4 = \alpha_1(x_{I_1})$ and that without loss of generality $\alpha_1(x_{I_1}) \not= \nu_1$, as otherwise all four points lie in $X$ and $\phi^{\text{ext}} = \phi$ on $X \cap \tilde{Z}$, and the latter map is multiaffine. By induction on $d' \in [d_0,k]$ we show the following claim.
\begin{claimc*}For $d' \in [d_0,k]$, we have $\sum_{i \in [4]} \sgn(i) \phi^{\text{ext}}(x_{[d'] \setminus \{d\}}, y^i_d, a_{[d'+1,k]}) = 0.$\end{claimc*}
Note that $(x_{[d'] \setminus \{d\}}, y^i_d, a_{[d'+1,k]}) \in Z$ for $d' \in [d, k-1]$.

\begin{proof}[Proof of Claim C]The base case is the case $d' = d_0$, when $\phi^{\text{ext}}(x_{[d_0] \setminus \{d\}}, y^i_d, a_{[d_0+1,k]})$ becomes $h_0(x_{[d_0] \setminus \{d\}}, y^i_d)$, and $h_0$ is multiaffine. Assume now that the claim holds for some $d' \in [d_0,k-2]$: we shall treat the case $d' = k-1$ separately. (When $d' \in [d_0,k-2]$ we immediately know that $(x_{[d' + 1] \setminus \{d\}}, y^i_d, a_{[d'+2,k]}) \in Z$, when $d' = k-1$ this need not hold.) Recall that for fixed $x_{[k] \setminus \{d, d' + 1\}}$, there is a global biaffine map $\phi^{\text{glob}}_{x_{[k] \setminus \{d, d' + 1\}}}$ that extends $\phi$ from $X_{x_{[k] \setminus \{d, d'+1\}}}$ to $G_d \times G_{d'+1}$. Then, by definition of $\phi^{\text{ext}}$ on $Z$, we have, for some elements $u^{[4]}_{d'+1}$, that
\begin{align*}\sum_{i \in [4]} \sgn(i) \phi^{\text{ext}}(x_{[d' + 1] \setminus \{d\}}, y^i_d, a_{[d'+2,k]})
&=\sum_{i \in [4]} \sgn(i) \phi^{\text{ext}}(x_{[d'] \setminus \{d\}}, y^i_d, a_{[d'+1,k]})\\ 
&\hspace{1cm}+ \sum_{i \in [4]} \sgn(i)\phi(x_{[d'] \setminus \{d\}}, y^i_d, u^i_{d'+1} + x_{d'+1} - a_{d'+1}, a_{[d'+2,k]})\\
&\hspace{2cm}- \sum_{i \in [4]} \sgn(i)\phi(x_{[d'] \setminus \{d\}}, y^i_d, u^i_{d'+1}, a_{[d'+2,k]}),\\
\end{align*}
which by the inductive hypothesis is equal to
\begin{align*}
&\sum_{i \in [4]} \sgn(i)\phi(x_{[d'] \setminus \{d\}}, y^i_d, u^i_{d'+1} + x_{d'+1} - a_{d'+1}, a_{[d'+2,k]})- \sum_{i \in [4]} \sgn(i)\phi(x_{[d'] \setminus \{d\}}, y^i_d, u^i_{d'+1}, a_{[d'+2,k]})\\
&\hspace{1cm}=\sum_{i \in [4]} \sgn(i)\phi^{\text{glob}}_{x_{[d'] \setminus \{d\}}, a_{[d'+2,k]}}(y^i_d, u^i_{d'+1} + x_{d'+1} - a_{d'+1})- \sum_{i \in [4]} \sgn(i)\phi^{\text{glob}}_{x_{[d'] \setminus \{d\}}, a_{[d'+2,k]}}(y^i_d, u^i_{d'+1})\\
&\hspace{1cm}=0.
\end{align*}
We now prove the induction step when $d' = k-1$. Recall that $\mu_1 = \mu_2 = \mu_3 = \mu_4 = \alpha_1(x_{I_1}) \not= \nu_1$ and write $\mu = \alpha_1(x_{I_1})$. By definition of $\phi^{\text{ext}}$, we have for some $u^{[4]}_k$
\begin{align*}\sum_{i \in [4]} \sgn(i) \phi^{\text{ext}}(x_{[k] \setminus \{d\}}, y^i_d) &= \sum_{i \in [4]} \sgn(i) \frac{\mu- \nu_1}{\mu_0 - \nu_1}\Big(\phi^{\text{ext}}(x_{[k-1] \setminus \{d\}}, y^i_d, a_k) - \phi(x_{[k-1] \setminus \{d\}}, y^i_d, u^i_k)\Big)\\
&\hspace{2cm}+ \sum_{i \in [4]} \sgn(i)\phi\Big(x_{[k-1] \setminus \{d\}}, y^i_d, x_k - \frac{\mu- \nu_1}{\mu_0 - \nu_1}(a_k - u^i_k)\Big),\\
\end{align*}
which by the inductive hypothesis is equal to 
\begin{align*} 
&-\sum_{i \in [4]} \sgn(i) \frac{\mu- \nu_1}{\mu_0 - \nu_1} \phi(x_{[k-1] \setminus \{d\}}, y^i_d, u^i_k)+ \sum_{i \in [4]} \sgn(i)\phi\Big(x_{[k-1] \setminus \{d\}}, y^i_d, x_k - \frac{\mu- \nu_1}{\mu_0 - \nu_1}(a_k - u^i_k)\Big)\\
&\hspace{1cm}=  -\sum_{i \in [4]} \sgn(i) \frac{\mu- \nu_1}{\mu_0 - \nu_1} \phi^{\text{glob}}_{x_{[k-1] \setminus \{d\}}}(y^i_d, u^i_k)+ \sum_{i \in [4]} \sgn(i)\phi^{\text{glob}}_{x_{[k-1] \setminus \{d\}}}\Big(y^i_d, x_k - \frac{\mu- \nu_1}{\mu_0 - \nu_1}(a_k - u^i_k)\Big)\\
&\hspace{1cm}=0,\\
\end{align*}
completing the proof of the claim.\end{proof}

\indent When $d' = k$, we get $\sum_{i \in [4]} \sgn(i) \phi^{\text{ext}}(x_{[k] \setminus \{d\}}, y^i_d) = 0$, which completes the proof that $\phi^{\text{ext}}$ is a multi-homomorphism on $\tilde{Z}$.\\

Recall from~\eqref{1codimExtnTzdensEqn} that $|\tilde{Z}| = (1 - O(\eta) - O((\eta^{-2}{\varepsilon'})^{\Omega(1)}) - O((\eta^{-1}\xi)^{\Omega(1)}))|V^{[2,s]} \cap (U \times G_k)|$ and that $V^{[2,s]} \cap (U \times G_k) \not= \emptyset$, where $\varepsilon' = O(\varepsilon^{\Omega(1)}) + \xi$. We may pick $\eta \geq \con\, \Big(\varepsilon^{\cons } + \xi^{\cons }\Big)$ so that the bound above becomes
\[|\tilde{Z}| \geq (1-O(\varepsilon^{\Omega(1)}) - O(\xi^{\Omega(1)}))|V^{[2,s]} \cap (U \times G_k)|.\]
(Note that the necessary bounds involving $\eta$ were $\eta < 1/10$, so the choice of $\eta$ made here will satisfy those requirements provided $\varepsilon$ and $\xi$ are smaller than some positive absolute constant $c$. Of course if $\xi > c$ or $\varepsilon > c$ then the claim is vacuous.)\\
\indent Apply Proposition~\ref{MltHommToMltAff} to make the multi-homomorphism become a multiaffine map and to finish the proof.\end{proof}
\bigskip

\subsection{Generating structure by convolving}

\indent\chg{The next theorem tells us that if $\phi$ is a multi-$D$-homomorphism defined on a $1-\varepsilon$ proportion of the points of a variety $V$ and if $\xi > 0$ is arbitrarily small, then, by convolving, we may extend $\phi$ to $1-\xi$ of $V \cap W$, for a lower order variety $W$ of bounded codimension. The gain in the theorem is that we may take $\xi$ to be significantly smaller than $\varepsilon$, thus significantly decreasing the set of points in the ambient variety that are missing from the domain. The additional price we pay is that this extension is a multi-$D'$-homomorphism, where $D'$ is slightly smaller than $D$. As usual, the conclusion holds modulo a lower order variety.}\\

\chg{In order to relate the constructed map to the given one,} in a similar way to the way we defined arrangements when we stated Theorem~\ref{exArrThm}, we now define a closely related structure that we call a \emph{tri-arrangement}. We say that a singleton sequence consisting of a single point $x_{[k]} \in G_{[k]}$ is an \emph{$\emptyset$-tri-arrangement} of lengths $x_{[k]}$. For $i \in [k]$, a \emph{$(k,k-1, \dots, i)$-tri-arrangement} $q$ of lengths $x_{[k]}$ is any concatenation $(q_1, q_2, q_3)$ where $q_1, q_2$ and $q_3$ are $(k,k-1, \dots, i+1)$-tri-arrangements of lengths $(x_{[k] \setminus \{i\}}, u_{i})$, $(x_{[k] \setminus \{i\}}, v_{i})$ and $(x_{[k] \setminus \{i\}}, w_{i})$, respectively, such that $u_i + v_i - w_i = x_i$. When $\phi$ is a map defined at points of the tri-arrangement $q$ above, we recursively define its $\phi$ value by $\phi(q) = \phi(q_1) + \phi(q_2) - \phi(q_3)$, and we set $\phi$ value of the singleton sequence $\{x_{[k]}\}$ to be simply $\phi(x_{[k]})$.\\

\noindent\textbf{Remark.} Additionally to $k$ and $p$, the implicit constants in bounds in the theorem below depend on $D$. However, we shall apply the theorem with $D = O_{k,p}(1)$, so we opt not to stress the dependence on $D$ for the sake of readability.\\

\begin{theorem}\label{bogolyubovSingleStep}\chg{Let $D$ be a positive integer and let $\mathcal{G} \subset \mathcal{P}[k]$ be a down-set. Write $\mathcal{G}'$ for the down-set obtained by removing maximal subsets from $\mathcal{G}$. Let $V$ be a non-empty $\mathcal{G}$-supported variety of codimension $r$.\\
\indent Let $X \subset V$ be a set of size at least $(1-\varepsilon) |V|$.  Let $\phi \colon X \to H$ be a multi-$(D \cdot 20^{k})$-homomorphism and let $\xi > 0$. Suppose that $\dim G_i \geq \con \Big(r + \log_p \xi^{-1}\Big)^{\con}$ for each $i \in [k]$. Then there exist
\begin{itemize}
\item a $\mathcal{G}'$-supported variety $U \subset G_{[k]}$ of codimension $O((r + \log_p \xi^{-1})^{O(1)})$,
\item a subset $Z \subset V \cap U$, and
\item a multi-$D$-homomorphism $\psi \colon Z \to H$
\end{itemize}
such that}
\begin{itemize}
\item[\textbf{(i)}] \chg{the variety $V \cap U$ is non-empty,}
\item[\textbf{(ii)}] \chg{$|(V \cap U) \setminus Z|\ \leq \xi |V \cap U|$,}
\item[\textbf{(iii)}] \chg{for each $x_{[k]} \in Z$, there are at least $\Omega\Big(p^{-O\big((r + \log_p \xi^{-1})^{O(1)}\big)}|G_{1}|^{2}|G_{2}|^{2 \cdot 3} \cdots |G_{k}|^{2 \cdot 3^{k - 1}}\Big)$ $(k,k-1, \dots, 1)$-tri-arrangements $q$ with points in $X$ of lengths $x_{[k]}$ such that $\phi(q) = \psi(x_{[k]})$.}
\end{itemize}
\end{theorem}

\chg{The proof will be based on directional convolutions. Thus, we formulate a more general proposition, which is suitable for an inductive proof. The above theorem is just the special case of the proposition below when $i = 0$. Since the proposition is more technical than the theorem, we follow it with a remark about its interpretation.} \chg{We remind the reader once again that the notation $V^{I}$ stands for the natural variety inside $G_I$ associated with $V$, instead of the traditional meaning of the set of functions $I \to V$. The formal definition is found in the statement below.}

\begin{proposition}\label{bogolyubovSingleStepInduction}Let $D$ be a positive integer and let $\mathcal{G} \subset \mathcal{P}[k]$ be a down-set. Write $\mathcal{G}'$ for the down-set obtained by removing maximal subsets from $\mathcal{G}$. Let $\alpha \colon G_{[k]} \to \mathbb{F}_p^{r_0}$ be a multiaffine map such that each form $\alpha_i$ is multilinear on some $G_{I_i}$ for a maximal set $I_i \in \mathcal{G}$. Let $V = \alpha^{-1}(\nu^{(0)}) \cap V'$ be a non-empty variety of codimension $r \geq r_0$, where $V'$ is some $\mathcal{G}'$-supported variety. Given a subset $I \subset [k]$, write $V^{I}$ for the set $\Big\{x_{I} \in G_{I} \colon (\forall j \in [r_0] \colon I_j \subset I)\,\, \alpha_j(x_{I_j}) = \nu^{(0)}_j\Big\}$.\\
\indent Let $X \subset V$ be a set of size at least $(1-\varepsilon) |V|$. Let $\phi \colon X \to H$ be a multi-$(D \cdot 20^{k})$-homomorphism, let $i \in [0, k]$, and let $\xi > 0$. Suppose that $\dim G_i \geq \con \Big(r + \log_p \xi^{-1}\Big)^{\con}$ for each $i \in [k]$. Then there exist
\begin{itemize}
\item a $(\mathcal{G}' \cap \mathcal{P}[i])$-supported variety $W \subset G_{[i]}$ of codimension $O((r + \log_p \xi^{-1})^{O(1)})$,
\item a subset $Y \subset W \cap V^{[i]}$,
\item a $\mathcal{G}'$-supported variety $U \subset G_{[k]}$ of codimension $O((r + \log_p \xi^{-1})^{O(1)})$,
\item a subset $Z \subset (Y \times G_{[i+1,k]}) \cap V \cap U$, and
\item a multi-$(D \cdot 20^{i})$-homomorphism $\psi \colon Z \to H$
\end{itemize}
such that
\begin{itemize}
\item[\textbf{(i)}] the variety $((W \cap V^{[i]}) \times G_{[i+1,k]}) \cap V \cap U$ is non-empty,
\item[\textbf{(ii)}] $|((W \cap V^{[i]} \setminus Y) \times G_{[i+1,k]}) \cap V \cap U|\ = (O(\varepsilon^{\Omega(1)}) + \xi)|((W \cap V^{[i]}) \times G_{[i+1,k]}) \cap V \cap U|$,
\item[\textbf{(iii)}] $|((Y \times G_{[i+1,k]}) \cap V \cap U) \setminus Z|\ \leq \xi |((W \cap V^{[i]}) \times G_{[i+1,k]}) \cap V \cap U|$,
\item[\textbf{(iv)}] for each $x_{[k]} \in Z$, there are at least $\Omega\Big(p^{-O\big((r + \log_p \xi^{-1})^{O(1)}\big)}|G_{i+1}|^{2}|G_{i+2}|^{2 \cdot 3} \cdots |G_{k}|^{2 \cdot 3^{k - i - 1}}\Big)$ $(k,k-1, \dots, i+1)$-tri-arrangements $q$ with points in $X$ of lengths $x_{[k]}$ such that $\phi(q) = \psi(x_{[k]})$.
\end{itemize}
\end{proposition}

\vspace{\baselineskip}

\noindent\chg{\textbf{Interpretation of the terms in the statement.} We are given a multi-$(D \cdot 20^{k})$-homomorphism $\phi$, whose domain $X$ is $(1-\varepsilon)$-dense inside a variety $V$. Since we need control the coordinate dependencies carefully, we write $V = \alpha^{-1}(\nu^{(0)}) \cap V'$, for a lower order variety $V'$ and define varieties $V^{I}$.\\
\indent The outcome of proposition is the result of convolutional extensions (as in Theorem~\ref{mainExtensionConvStep}) of $\phi$ applied to directions $i + 1, \dots, k$. The new domain is $Z$. As a result of convolutional extensions, for many $x_{[i]}$, we have that
\begin{equation}\label{zverydenseineqinter}|Z_{x_{[i]}}| \geq (1 - \xi) |(V \cap U)_{x_{[i]}}|,\end{equation}
meaning that the density of the domain has increased in the desired fashion on suitable varieties. On the other hand, the set $Y$ controls the points $x_{[i]}$ for which we have~\eqref{zverydenseineqinter}. Since we did not apply convolutions in directions in $[i]$ so far, we can only guarantee that $Y$ does not lose too much in density inside its ambient variety. The precise quantification of this property is given by part \textbf{(ii)} of the conclusion.}\\

\noindent\textbf{Remark.} Note that $V^{[i]} \times G_{[i+1,k]} \supset V$ so $V^{[i]}$ is redundant in some of the expressions above, e.g.\ $((W \cap V^{[i]}) \times G_{[i+1,k]}) \cap V \cap U = (W \times G_{[i+1,k]}) \cap V \cap U$. However, we keep $V^{[i]}$ in such expressions to make them more consistent and easier to compare. The main reason to include $V^{[i]}$ in the first place is that the set $Y$ is a subset of $W \cap V^{[i]}$, which makes the expression $((W \cap V^{[i]} \setminus Y) \times G_{[i+1,k]}) \cap V \cap U$ more suggestive than $((W \setminus Y) \times G_{[i+1,k]}) \cap V \cap U$.\\

\noindent\textbf{Remark about the case $i = 0$.} Note that we allow $i = 0$, in which case the variety $W$ and the set $Y$ play no significant role as $Y \times G_{[k]}$ is either $\emptyset$ or $G_{[k]}$. In that case $Z$ becomes a subset of $V \cap U$. To deduce Theorem~\ref{bogolyubovSingleStep}, we just need the term $O(\varepsilon^{\Omega(1)}) + \xi)$ in part \textbf{(ii)} of the conclusion to be less than 1, to ensure that $Y \times G_{[k]} = G_{[k]}$. We remark that the proof simplifies in this case as some steps become trivial, e.g. \textbf{Regularization} below.

\begin{proof}[Proof of Proposition~\ref{bogolyubovSingleStepInduction}]Without loss of generality, $V'$ is mixed-linear, thus making $V$ mixed-linear as well. We prove the claim by downwards induction on $i$. For the base case $i = k$, simply take $W = V'$, $Y = X$, $U = G_{[k]}$, $Z = X$ and $\psi = \phi$.\\

\noindent\chg{\textbf{Applying the induction hypothesis.}} Assume now that the claim holds for some $i+1 \in [1,k]$. Let $\xi_0 > 0$ be a parameter to be chosen later. Apply the inductive hypothesis for the parameter $\xi_0$ instead of $\xi$. There exist a $(\mathcal{G}' \cap \mathcal{P}[i+1])$-supported variety $W^{\text{ind}} \subset G_{[i+1]}$ of codimension $O((r + \log_p \xi_0^{-1})^{O(1)})$, a subset $Y^{\text{ind}} \subset W^{\text{ind}} \cap V^{[i+1]}$, a $\mathcal{G}'$-supported variety $U^{\text{ind}} \subset G_{[k]}$ of codimension $O((r + \log_p \xi_0^{-1})^{O(1)})$, a subset $Z^{\text{ind}} \subset V \cap U^{\text{ind}} \cap (Y^{\text{ind}} \times G_{[i+2,k]})$, and a multi-$(D \cdot 20^{i+1})$-homomorphism $\psi^{\text{ind}} \colon Z^{\text{ind}} \to H$ with the properties described. \chg{In the inductive step, we will apply convolutional extension in direction $i + 1$.}\\

\noindent\chg{\textbf{Ensuring that $W^{\text{ind}}$ and $U^{\text{ind}}$ are mixed-linear.}} We now show that without loss of generality $W^{\text{ind}}$ and $U^{\text{ind}}$ can be taken to be mixed-linear. Mixed-linearity is necessary for the argument in order to apply Theorem~\ref{mainExtensionConvStep}. By looking at the multilinear parts of the maps that define these two varieties, we find $r_1 \leq 2^k \codim W^{\text{ind}}$, $r_2 \leq 2^k \codim U^{\text{ind}}$, a $(\mathcal{G}' \cap \mathcal{P}[i + 1])$-supported mixed-linear map $\gamma^1 \colon G_{[i+1]} \to \mathbb{F}_p^{r_1}$, a $\mathcal{G}'$-supported mixed-linear map $\gamma^2 \colon G_{[k]} \to \mathbb{F}_p^{r_2}$, and collections of values $M_1 \subset \mathbb{F}_p^{r_1}, M_2 \subset \mathbb{F}_p^{r_2}$ such that $W^{\text{ind}} = (\gamma^1)^{-1}(M_1)$ and $U^{\text{ind}} = (\gamma^2)^{-1}(M_2)$. We may also assume that $(\gamma^1)^{-1}(\mu^1)\not=\emptyset$ for each $\mu^1 \in M_1$ and $(\gamma^2)^{-1}(\mu^2) \not=\emptyset$ for each $\mu^2 \in M_2$.\\
\indent Let $0 < \tilde{\varepsilon} = O(\varepsilon^{\Omega(1)}) + \xi_0$ be such that 
\[|((W^{\text{ind}} \cap V^{[i+1]} \setminus Y^{\text{ind}}) \times G_{[i+2,k]}) \cap V \cap U^{\text{ind}}|\ \leq \tilde{\varepsilon}|((W^{\text{ind}} \cap V^{[i+1]}) \times G_{[i+2,k]}) \cap V \cap U^{\text{ind}}|.\]
\chg{We shall pass to layers $(\gamma^1)^{-1}(\mu^1)$ and $(\gamma^2)^{-1}(\mu^2)$ for suitable $\mu^1 \in M_1$ and $\mu^2 \in M_2$. To that end, we define sets $P_1, P_2, P_3 \subseteq M_1 \times M_2$ that will guarantee that properties \textbf{(i)}, \textbf{(ii)} and \textbf{(iii)} of the conclusion, which followed from the induction hypothesis, still hold.}\\
\indent Let $P_1$ be the set of all $(\mu^1, \mu^2) \in M_1 \times M_2$ such that
\[\Big(((\gamma^1)^{-1}(\mu^1) \cap V^{[i+1]}) \times G_{[i+2,k]}\Big) \cap V \cap (\gamma^2)^{-1}(\mu^2) \not= \emptyset,\]
let $P_2$ be the set of all $(\mu^1, \mu^2) \in M_1 \times M_2$ such that
\[\Big|\Big(((\gamma^1)^{-1}(\mu^1) \cap V^{[i+1]}\setminus Y^{\text{ind}}) \times G_{[i+2,k]}\Big) \cap V \cap (\gamma^2)^{-1}(\mu^2)\Big| \leq 3\tilde{\varepsilon}\Big|\Big(((\gamma^1)^{-1}(\mu^1)  \cap V^{[i+1]}) \times G_{[i+2,k]}\Big) \cap V \cap (\gamma^2)^{-1}(\mu^2)\Big|,\]
and let $P_3$ be the set of all $(\mu^1, \mu^2) \in M_1 \times M_2$ such that
\[\Big|\Big(((Y^{\text{ind}} \cap (\gamma^1)^{-1}(\mu^1)) \times G_{[i+2,k]}) \cap V \cap (\gamma^2)^{-1}(\mu^2) \Big) \setminus Z^{\text{ind}}\Big|\ \leq 3\xi_0 \Big|\Big(((\gamma^1)^{-1}(\mu^1)  \cap V^{[i+1]}) \times G_{[i+2,k]}\Big) \cap V \cap (\gamma^2)^{-1}(\mu^2)\Big|.\]
By the definitions of the sets $P_1, P_2$ and $P_3$ we have
\[\sum_{(\mu^1, \mu^2) \in (M_1 \times M_2) \setminus P_1} \Big|\Big(((\gamma^1)^{-1}(\mu^1)  \cap V^{[i+1]}) \times G_{[i+2,k]}\Big) \cap V \cap (\gamma^2)^{-1}(\mu^2) \Big|\ = 0,\]
\begin{align*}
&\sum_{(\mu^1, \mu^2) \in (M_1 \times M_2) \setminus P_2} \Big|\Big(((\gamma^1)^{-1}(\mu^1)  \cap V^{[i+1]}) \times G_{[i+2,k]}\Big) \cap V \cap (\gamma^2)^{-1}(\mu^2) \Big|\\
&\hspace{3cm}\leq \frac{{\tilde{\varepsilon}}^{-1}}{3} \sum_{(\mu^1, \mu^2) \in (M_1 \times M_2) \setminus P_2}\Big|\Big(((\gamma^1)^{-1}(\mu^1) \cap V^{[i+1]} \setminus Y^{\text{ind}}) \times G_{[i+2,k]}\Big) \cap V \cap (\gamma^2)^{-1}(\mu^2) \Big|\\
&\hspace{3cm} \leq \frac{{\tilde{\varepsilon}}^{-1}}{3} |((W^{\text{ind}} \cap V^{[i+1]} \setminus Y^{\text{ind}}) \times G_{[i+2,k]}) \cap V \cap U^{\text{ind}}|\\
&\hspace{3cm} \leq \frac{1}{3} \Big|((W^{\text{ind}} \cap V^{[i+1]}) \times G_{[i+2,k]}) \cap V \cap U^{\text{ind}}\Big|,\\
\end{align*}
and
\begin{align*}
&\sum_{(\mu^1, \mu^2) \in (M_1 \times M_2) \setminus P_3} \Big|\Big(((\gamma^1)^{-1}(\mu^1)  \cap V^{[i+1]}) \times G_{[i+2,k]}\Big) \cap V \cap (\gamma^2)^{-1}(\mu^2) \Big| \leq\\
&\hspace{3cm} \leq \frac{{\xi_0}^{-1}}{3} \sum_{(\mu^1, \mu^2) \in (M_1 \times M_2) \setminus P_3} \Big|\Big(((Y^{\text{ind}} \cap (\gamma^1)^{-1}(\mu^1)) \times G_{[i+2,k]}) \cap V \cap (\gamma^2)^{-1}(\mu^2) \Big) \setminus Z^{\text{ind}}\Big|\\
&\hspace{3cm} \leq \frac{{\xi_0}^{-1}}{3}\Big|((Y^{\text{ind}} \times G_{[i+2,k]}) \cap V \cap U^{\text{ind}}) \setminus Z^{\text{ind}}\Big|\\
&\hspace{3cm} \leq \frac{1}{3} \Big|((W^{\text{ind}} \cap V^{[i+1]}) \times G_{[i+2,k]}) \cap V \cap U^{\text{ind}}\Big|.\\
\end{align*}
Since
\[\sum_{(\mu^1, \mu^2) \in M_1 \times M_2}  \Big|\Big(((\gamma^1)^{-1}(\mu^1)  \cap V^{[i+1]}) \times G_{[i+2,k]}\Big) \cap V \cap (\gamma^2)^{-1}(\mu^2) \Big| = \Big|((W^{\text{ind}} \cap V^{[i+1]}) \times G_{[i+2,k]}) \cap V \cap U^{\text{ind}}\Big|\]
we deduce that there exists a pair $(\mu^1, \mu^2) \in P_1 \cap P_2 \cap P_3$. Misuse the notation and write $W^{\text{ind}}$ for $(\gamma^1)^{-1}(\mu^1)$, $U^{\text{ind}}$ for $(\gamma^2)^{-1}(\mu^2)$, $Y^{\text{ind}}$ for $Y^{\text{ind}} \cap (\gamma^1)^{-1}(\mu^1)$, $Z^{\text{ind}}$ for $Z^{\text{ind}} \cap ((\gamma^1)^{-1}(\mu^1)  \times G_{[i+2,k]}) \cap (\gamma^2)^{-1}(\mu^2)$ and $\xi_0$ instead of $3\xi_0$ (this will not affect the bounds in the statement that depend on $\xi$).\\
\indent Let $r^{\text{ind}} = r + r_1 + r_2 = O((r + \log_p \xi_0^{-1})^{O(1)})$.\\

\noindent\chg{\textbf{Preparation for the convolutional extension.} Our goal is to apply Theorem~\ref{mainExtensionConvStep}. This theorem will eventually be applied to the map $\psi^{\text{ind}}$, on the domain $Z^{\text{ind}}$, inside the ambient variety $(W^{\text{ind}} \times G_{[i+2,k]}) \cap V \cap U^{\text{ind}}$ and the set of $(1-\varepsilon_0)$-dense columns $\on{DC}$, which is to be defined in this step of the proof.\\
\indent We shall first apply Theorem~\ref{fibresThm} to ensure that the fibres in direction $i + 1$ of the variety $V \cap U^{\text{ind}} \cap (W^{\text{ind}} \times G_{[i+2,k]})$ have the same size. We then perform an averaging argument to find a variety $T$ in $G_{[k] \setminus \{i + 1\}}$ which controls these fibres and also ensures that properties \textbf{(i)}, \textbf{(ii)} and \textbf{(iii)} still hold. The variety $T$ is defined in~\eqref{tdefnafterlongavg}.\\
\indent In order to be able to conclude the properties \textbf{(i)}, \textbf{(ii)} and \textbf{(iii)} at the end of the proof, we also introduce additional subsets $A, F \subseteq T$, defined in~\eqref{highdimconvadefn} and~\eqref{highdimconvfdefn} respectively, which control the density of $Y$ and $Z$ in suitable fibres.}\\
\indent Apply Theorem~\ref{fibresThm} to the variety\footnote{More precisely, since $V \cap U^{\text{ind}} \cap (W^{\text{ind}} \times G_{[i+2,k]})$ is a mixed-linear variety, we may see it as defined to be $(\beta^1)^{-1}(\mu^1) \cap (\beta^2)^{-1}(\mu^2)$, where $\beta^1$ is a mixed-linear map all of whose components depend linearly on $G_{i+1}$ and $\beta^2$ is a mixed-linear map that does not depend on $G_{i+1}$. Then apply the theorem to the variety $(\beta^1)^{-1}(\mu^1)$ to obtain the relevant layers $L_1, \dots, L_m \subset G_{[k]\setminus \{i+1\}}$ of some other lower order multiaffine map $\gamma$. For the initial variety, we are then interested in the layers of $(\gamma, \beta^2)$, which are of the form $(\beta^2)^{-1}(\mu^2) \cap L_i$. In our case, $(\beta^2)^{-1}(\mu^2)$ is of the form $V^{[k] \setminus \{i+1\}}$ intersected with a mixed-linear lower order variety. (We have put these details in a footnote so as not to interrupt the flow of the argument.)\label{linearconstfootnote}} 
\[V \cap U^{\text{ind}} \cap (W^{\text{ind}} \times G_{[i+2,k]})\]
in direction $G_{i+1}$ and with parameter $\eta > 0$, to find a positive integer $s = O\Big((r^{\text{ind}} + \log_p \eta^{-1})^{O(1)}\Big)$, a $\mathcal{G}'$-supported multiaffine map $\gamma \colon G_{[k] \setminus \{i+1\}} \to \mathbb{F}_p^{s}$ (which without loss of generality is mixed-linear), a collection of values $M \subset \mathbb{F}_p^{s}$ such that $|\gamma^{-1}(M)|\ \geq (1-\eta) |G_{[k]\setminus \{i+1\}}|$, and a map $c \colon M \to [0,1]$ such that
\begin{equation} \Big|\Big(V \cap U^{\text{ind}} \cap (W^{\text{ind}} \times G_{[i+2,k]})\Big)_{x_{[k] \setminus \{i+1\}}}\Big| = c(\mu) |G_{i+1}|\label{cOfMuEqnProp}\end{equation}
for every $\mu \in M$ and every $x_{[k] \setminus \{i+1\}} \in \gamma^{-1}(\mu)\cap V^{[k] \setminus \{i+1\}}$. Let $\varepsilon' > 0$ be a parameter to be chosen later and let $M'$ be the set of all $\mu \in M$ such that
\begin{align}\Big|\Big((W^{\text{ind}} \cap V^{[i+1]} &\setminus Y^{\text{ind}}) \times G_{[i+2,k]}\Big)  \cap \Big(\gamma^{-1}(\mu) \times G_{i+1}\Big) \cap V \cap U^{\text{ind}}\Big|\nonumber\\
&\leq \varepsilon' \Big|\Big((W^{\text{ind}} \cap V^{[i+1]}) \times G_{[i+2,k]}\Big) \cap \Big(\gamma^{-1}(\mu) \times G_{i+1}\Big) \cap V \cap U^{\text{ind}}\Big| \not= 0\label{MprimeDefnEqn}\end{align}
(note that we insist that the variety on the right-hand-side is non-empty) and let $M''$ be the set of all $\mu \in M$ such that
\begin{align}\Big|\Big((Y^{\text{ind}} \times G_{[i+2,k]})& \cap V \cap U^{\text{ind}} \cap (\gamma^{-1}(\mu) \times G_{i+1})\Big) \setminus Z^{\text{ind}}\Big|\nonumber\\
&\leq \sqrt{\xi_0} \Big|\Big((W^{\text{ind}} \cap V^{[i+1]}) \times G_{[i+2,k]}\Big) \cap \Big(\gamma^{-1}(\mu) \times G_{i+1}\Big) \cap V \cap U^{\text{ind}}\Big| \not= 0,\label{MdoubleprimeDefnEqn}\end{align}
(again we insist that the variety on the right-hand-side is non-empty). Then
\begin{align*}\sum_{\mu \in M \setminus M'} &\Big|\Big((W^{\text{ind}} \cap V^{[i+1]}) \times G_{[i+2,k]}\Big) \cap \Big(\gamma^{-1}(\mu) \times G_{i+1}\Big) \cap V \cap U^{\text{ind}}\Big|\\
 &\leq {\varepsilon'}^{-1} \sum_{\mu \in M \setminus M'} \Big|\Big((W^{\text{ind}} \cap V^{[i+1]} \setminus Y^{\text{ind}}) \times G_{[i+2,k]}\Big)  \cap \Big(\gamma^{-1}(\mu) \times G_{i+1}\Big) \cap V \cap U^{\text{ind}}\Big|\\
&\leq {\varepsilon'}^{-1} \Big|\Big((W^{\text{ind}} \cap V^{[i+1]} \setminus Y^{\text{ind}}) \times G_{[i+2,k]}\Big) \cap V \cap U^{\text{ind}}\Big|\\
&= O({\varepsilon'}^{-1}(\varepsilon^{\Omega(1)} + \xi_0)) \Big|((W^{\text{ind}} \cap V^{[i+1]}) \times G_{[i+2,k]}) \cap V \cap U^{\text{ind}}\Big|.\end{align*}
We may take $\varepsilon' = O(\varepsilon^{\Omega(1)} + \xi_0)$ so that the $O({\varepsilon'}^{-1}(\varepsilon^{\Omega(1)} + \xi_0))$ term becomes less than $1/3$.\\
\indent Similarly, 
\begin{align*}\sum_{\mu \in M' \setminus M''} &\Big|\Big((W^{\text{ind}} \cap V^{[i+1]}) \times G_{[i+2,k]}\Big) \cap \Big(\gamma^{-1}(\mu) \times G_{i+1}\Big) \cap V \cap U^{\text{ind}}\Big|\\ 
&\leq \xi_0^{-1/2} \sum_{\mu \in M' \setminus M''} \Big|\Big(Y^{\text{ind}} \times G_{[i+2,k]}\Big)  \cap \Big((\gamma^{-1}(\mu) \times G_{i+1}) \cap V \cap U^{\text{ind}}\Big) \setminus Z^{\text{ind}}\Big|\\ 
&\leq \xi_0^{-1/2} \Big|(Y^{\text{ind}} \times G_{[i+2,k]}) \cap V \cap U^{\text{ind}} \setminus Z^{\text{ind}}\Big|\\
&\leq \xi_0^{1/2} \Big|\Big((W^{\text{ind}} \cap V^{[i+1]}) \times G_{[i+2,k]}\Big) \cap V \cap U^{\text{ind}}\Big|.\end{align*}
Finally,
\begin{align*}\sum_{\mu \in \mathbb{F}_p^{s} \setminus M} &\Big|\Big((W^{\text{ind}} \cap V^{[i+1]}) \times G_{[i+2,k]}\Big) \cap \Big(\gamma^{-1}(\mu) \times G_{i+1}\Big) \cap V \cap U^{\text{ind}}\Big|\ \leq \sum_{\mu \in \mathbb{F}_p^{s} \setminus M} |\gamma^{-1}(\mu)| |G_{i+1}| \,\,\leq\, \eta |G_{[k]}|\\
&\leq \eta p^{k r^{\text{ind}}}\Big|\Big((W^{\text{ind}} \cap V^{[i+1]}) \times G_{[i+2,k]}\Big) \cap V \cap U^{\text{ind}}\Big|.\hspace{2cm}\text{(by Lemma~\ref{varsizelemma})}\end{align*}
Thus, provided we take $\eta = \frac{1}{3} p^{-2k r^{\text{ind}}}$ and $\xi_0 \leq \cons$, and provided that $\varepsilon$ is smaller than some sufficiently small positive absolute constant (or else the claim of the proposition is vacuous, as we may take $Y = Z = \emptyset$), there exists $\mu_1 \in M' \cap M''$, which we fix. Let $\delta = c(\mu_1) > 0$.\\ 

Let $\varepsilon_0 = \varepsilon_0(D \cdot 20^{i+1}, k)$ be the constant from Theorem~\ref{mainExtensionConvStep}. Set 
\begin{equation}\label{tdefnafterlongavg} T = \{x_{[k] \setminus \{i+1\}} \in  V^{[k] \setminus \{i+1\}} \colon \gamma(x_{[k] \setminus \{i+1\}}) = \mu_1\},\end{equation}
which is a mixed-linear variety. By~\eqref{cOfMuEqnProp}, we have for every $x_{[k] \setminus \{i+1\}} \in T$, that
\begin{equation}\label{deltaEqnPropTElt}\Big|\Big(V \cap U^{\text{ind}} \cap (W^{\text{ind}} \times G_{[i+2,k]})\Big)_{x_{[k] \setminus \{i+1\}}}\Big| = \delta |G_{i+1}|.\end{equation}
Define also
\begin{equation}A = \Big\{x_{[k] \setminus \{i+1\}} \in T \colon \Big|\Big(W^{\text{ind}} \cap V^{[i+1]} \setminus Y^{\text{ind}}\Big)_{x_{[i]}} \cap (V \cap U^{\text{ind}})_{x_{[k] \setminus \{i+1\}}}\Big| \leq \frac{\varepsilon_0}{4} \Big|W^{\text{ind}}_{x_{[i]}} \cap (V \cap U^{\text{ind}})_{x_{[k] \setminus \{i+1\}}}\Big|\Big\}\label{highdimconvadefn}\end{equation}
and
\begin{equation}F = \Big\{x_{[k] \setminus \{i+1\}} \in T \colon \Big|Y^{\text{ind}}_{x_{[i]}} \cap \Big(V \cap U^{\text{ind}} \setminus Z^{\text{ind}}\Big)_{x_{[k] \setminus \{i+1\}}}\Big| \geq \frac{\varepsilon_0}{4} \Big|W^{\text{ind}}_{x_{[i]}} \cap (V \cap U^{\text{ind}})_{x_{[k] \setminus \{i+1\}}}\Big|\Big\}.\label{highdimconvfdefn}\end{equation}
\chg{Hence, for $x_{[k] \setminus \{i + 1\}} \in A \setminus F$, we have that the fibres in direction $i + 1$ of $Y^{\text{ind}}$ and $Z^{\text{ind}}$ are quite dense. Convolutional extension in direction $i + 1$ will ensure that these fibres become the full cosets of the ambient variety.} Note also that
\[V \cap (\gamma^{-1}(\mu_1) \times G_{i+1}) = V \cap (V^{[k] \setminus \{i+1\}} \times G_{i+1}) \cap (\gamma^{-1}(\mu_1) \times G_{i+1}) = V \cap (T \times G_{i+1}).\]
Then from~\eqref{MprimeDefnEqn},~\eqref{MdoubleprimeDefnEqn} and~\eqref{deltaEqnPropTElt}
\begin{align*}|T \setminus A| = &\frac{1}{\delta |G_{i+1}|} \sum_{x_{[k] \setminus \{i+1\}} \in T \setminus A} \Big|\Big(V \cap U^{\text{ind}} \cap (W^{\text{ind}} \times G_{[i+2,k]})\Big)_{x_{[k] \setminus \{i+1\}}}\Big|\hspace{2cm}\text{(by~\eqref{deltaEqnPropTElt})}\\
\leq &\frac{4\varepsilon_0^{-1}}{\delta |G_{i+1}|} \sum_{x_{[k] \setminus \{i+1\}} \in T \setminus A} \Big|\Big(W^{\text{ind}} \cap V^{[i+1]} \setminus Y^{\text{ind}}\Big)_{x_{[i]}} \cap (V \cap U^{\text{ind}})_{x_{[k] \setminus \{i+1\}}}\Big|\\
\leq &\frac{4\varepsilon_0^{-1}}{\delta |G_{i+1}|} \sum_{x_{[k] \setminus \{i+1\}} \in T} \Big|\Big(W^{\text{ind}} \cap V^{[i+1]} \setminus Y^{\text{ind}}\Big)_{x_{[i]}} \cap (V \cap U^{\text{ind}})_{x_{[k] \setminus \{i+1\}}}\Big|\\
= &\frac{4\varepsilon_0^{-1}}{\delta |G_{i+1}|}\Big|\Big((W^{\text{ind}} \cap V^{[i+1]} \setminus Y^{\text{ind}}) \times G_{[i+2,k]}\Big) \cap V \cap U^{\text{ind}} \cap (T \times G_{i+1})\Big|\\
\leq &\frac{4\varepsilon_0^{-1}}{\delta |G_{i+1}|} O(\varepsilon^{\Omega(1)} + \xi_0) \Big|\Big((W^{\text{ind}} \cap V^{[i+1]}) \times G_{[i+2,k]}\Big) \cap \Big(T \times G_{i+1}\Big) \cap V \cap U^{\text{ind}}\Big|\hspace{2cm}\text{(by~\eqref{MprimeDefnEqn})}\\
= &\frac{O(\varepsilon^{\Omega(1)} + \xi_0)}{\delta |G_{i+1}|} \sum_{x_{[k] \setminus \{i+1\}} \in T} \Big|\Big(V \cap U^{\text{ind}} \cap (W^{\text{ind}} \times G_{[i+2,k]})\Big)_{x_{[k] \setminus \{i+1\}}}\Big|\\
 =& O(\varepsilon^{\Omega(1)} + \xi_0)|T|\hspace{2cm}\text{(by~\eqref{deltaEqnPropTElt})}\end{align*}
and
\begin{align*}|F| = &\frac{1}{\delta |G_{i+1}|} \sum_{x_{[k] \setminus \{i+1\}} \in F} \Big|\Big(V \cap U^{\text{ind}} \cap (W^{\text{ind}} \times G_{[i+2,k]})\Big)_{x_{[k] \setminus \{i+1\}}}\Big|\hspace{2cm}\text{(by~\eqref{deltaEqnPropTElt})}\\
\leq &\frac{4\varepsilon_0^{-1}}{\delta |G_{i+1}|} \sum_{x_{[k] \setminus \{i+1\}} \in F} \Big|Y^{\text{ind}}_{x_{[i]}} \cap \Big(V \cap U^{\text{ind}} \setminus Z^{\text{ind}}\Big)_{x_{[k] \setminus \{i+1\}}}\Big|\\
\leq &\frac{4\varepsilon_0^{-1}}{\delta |G_{i+1}|} \sum_{x_{[k] \setminus \{i+1\}} \in T} \Big|Y^{\text{ind}}_{x_{[i]}} \cap \Big(V \cap U^{\text{ind}} \setminus Z^{\text{ind}}\Big)_{x_{[k] \setminus \{i+1\}}}\Big|\\
= &\frac{4\varepsilon_0^{-1}}{\delta |G_{i+1}|} \Big|(Y^{\text{ind}} \times G_{[i+2,k]}) \cap V \cap U^{\text{ind}} \cap (T \times G_{i+1}) \setminus Z^{\text{ind}}\Big|\\
\leq & \frac{4\varepsilon_0^{-1} \sqrt{\xi_0} }{\delta |G_{i+1}|}\Big|\Big((W^{\text{ind}} \cap V^{[i+1]}) \times G_{[i+2,k]}\Big) \cap \Big(T \times G_{i+1}\Big) \cap V \cap U^{\text{ind}}\Big|\hspace{2cm}\text{(by~\eqref{MdoubleprimeDefnEqn})}\\
= & \frac{4\varepsilon_0^{-1} \sqrt{\xi_0} }{\delta |G_{i+1}|} \sum_{x_{[k] \setminus \{i+1\}} \in T} \Big|\Big(V \cap U^{\text{ind}} \cap (W^{\text{ind}} \times G_{[i+2,k]})\Big)_{x_{[k] \setminus \{i+1\}}}\Big|\\
=&O(\sqrt{\xi_0})|T|.\hspace{2cm}\text{(by~\eqref{deltaEqnPropTElt})}\end{align*}
Note that
\[|Z^{\text{ind}}_{x_{[k] \setminus \{i+1\}}}|\geq (1- \varepsilon_0)\Big|W^{\text{ind}}_{x_{[i]}} \cap (V \cap U^{\text{ind}})_{x_{[k] \setminus \{i+1\}}}\Big|\]
holds for each $x_{[k] \setminus \{i+1\}} \in A \setminus F$. Define the set of all $(1-\varepsilon_0)$-dense columns $\on{DC}$ by
\[\on{DC} = \Big\{x_{[k] \setminus \{i+1\}} \in T : |Z^{\text{ind}}_{x_{[k] \setminus \{i+1\}}}|\geq (1- \varepsilon_0)\Big|W^{\text{ind}}_{x_{[i]}} \cap (V \cap U^{\text{ind}})_{x_{[k] \setminus \{i+1\}}}\Big|\Big\}.\]
Hence $A \setminus F \subset \on{DC}$.\\

\noindent\chg{\textbf{Convolutional extension -- application of Theorem~\ref{mainExtensionConvStep}.}} Apply Theorem~\ref{mainExtensionConvStep} to the subset $\on{DC}$ of $T$ (the set of $(1-\varepsilon_0)$-dense columns), the subset $Z^{\text{ind}}$ as the domain of the $\psi^{\text{ind}}$, the ambient variety $(W^{\text{ind}} \times G_{[i+2,k]}) \cap V \cap U^{\text{ind}}$ containing the domain $Z^{\text{ind}}$, and the parameter $\rho > 0$. Let $\gamma^{\text{conv}}$ be the $\mathcal{G}'$-supported multiaffine map (which without loss of generality is mixed-linear) of codimension $t^{\text{conv}} = O((r^{\text{ind}} + \log_p \rho^{-1})^{O(1)})$ given by the theorem, with a set of allowed layers $(\gamma^{\text{conv}})^{-1}(\lambda)$ for $\lambda \in \Gamma$ such that $|(\gamma^{\text{conv}})^{-1}(\Gamma)| \geq (1 - \rho)|G_{[k] \setminus \{i+1\}}|$, and let $R \subset \on{DC}$ be the set of points removed from $\on{DC}$, which is such that $|R| \leq \rho |G_{[k] \setminus \{i+1\}}|$.\\
\indent \chg{We now average over $\lambda \in \Gamma$ to pass to a suitable layer of $\gamma^{\text{conv}}$, with respect to which the set $A$ is very dense and the sets $F$ and $R$ are very sparse.}\\
\indent By Lemma~\ref{varsizelemma}, we have that $|T| \geq p^{-k (r^{\text{ind}} + s)} |G_{[k] \setminus \{i+1\}}|$. Let $0 < \delta_1 \leq O(\varepsilon^{\Omega(1)} + \xi_0)$ and $0 < \delta_2 \leq O(\sqrt{\xi_0} + \rho\, p^{k(r^{\text{ind}} + s)})$ be quantities such that $|T \setminus A| \leq \delta_1 |T|$ and $|F \cup R| \leq \delta_2 |T|$. Let $\Gamma_1$ be the set of all $\lambda \in \mathbb{F}_p^{t^{\text{conv}}}$ such that $|(T \setminus A) \cap (\gamma^{\text{conv}})^{-1}(\lambda)| \leq 3\delta_1 |T \cap (\gamma^{\text{conv}})^{-1}(\lambda)|$ and let $\Gamma_2$ be the set of all $\lambda \in \mathbb{F}_p^{t^{\text{conv}}}$ such that $|(F \cup R) \cap (\gamma^{\text{conv}})^{-1}(\lambda)| \leq 3\delta_2 |T \cap (\gamma^{\text{conv}})^{-1}(\lambda)|$. Then
\begin{align*}\sum_{\lambda \notin \Gamma_1} |T \cap (\gamma^{\text{conv}})^{-1}(\lambda)| \,\,\leq\,\sum_{\lambda \notin \Gamma_1} \frac{\delta_1^{-1}}{3} |(T \setminus A) \cap (\gamma^{\text{conv}})^{-1}(\lambda)| \,\,\leq\, \frac{\delta_1^{-1}}{3} |T \setminus A| \,\,\leq\,\frac{\delta_1^{-1}}{3} \cdot \delta_1 |T|\,\, \leq \,\frac{1}{3}|T|\end{align*}
and 
\[\sum_{\lambda \notin \Gamma_2} |T \cap (\gamma^{\text{conv}})^{-1}(\lambda)| \,\,\leq\, \frac{\delta_2^{-1}}{3}\sum_{\lambda \notin \Gamma_2} |(F \cup R) \cap (\gamma^{\text{conv}})^{-1}(\lambda)| \,\,\leq\, \frac{\delta_2^{-1}}{3} |F \cup R| \,\,\leq\, \frac{\delta_2^{-1}}{3}\cdot \delta_2 |T| \,\,\leq\, \frac{1}{3}|T|.\]
Hence
\[\sum_{\lambda \in \Gamma \cap \Gamma_1 \cap \Gamma_2} |T \cap (\gamma^{\text{conv}})^{-1}(\lambda)| \geq \Big(\frac{1}{3} - \rho p^{k(r^{\text{ind}} + s)}\Big)|T|.\]

We may take $\rho \geq \Omega(\xi_0 p^{-k(r^{\text{ind}} + s)})$ so that the factor $\Big(\frac{1}{3} - \rho p^{k(r^{\text{ind}} + s)}\Big)$ on the right hand side of the expression is positive and $\delta_2 \leq O(\sqrt{\xi_0})$, which allows us to find a layer $(\gamma^{\text{conv}})^{-1}(\mu^{\text{conv}})$ such that
\begin{align*}|A \cap (\gamma^{\text{conv}})^{-1}(\mu^{\text{conv}})| \,\,=&\,\, (1 - O(\varepsilon^{\Omega(1)} + \xi_0))|T \cap (\gamma^{\text{conv}})^{-1}(\mu^{\text{conv}})|\\
|(F \cup R) \cap (\gamma^{\text{conv}})^{-1}(\mu^{\text{conv}})|\,\, =& \,\,O(\sqrt{\xi_0}) |T \cap (\gamma^{\text{conv}})^{-1}(\mu^{\text{conv}})|\end{align*}
and the variety $T \cap (\gamma^{\text{conv}})^{-1}(\mu^{\text{conv}})$ is non-empty.\\
\indent We shall abuse notation and keep writing $T$ for $T \cap (\gamma^{\text{conv}})^{-1}(\mu^{\text{conv}})$ (which remains mixed-linear), $\on{DC}$ for $\on{DC} \cap (\gamma^{\text{conv}})^{-1}(\mu^{\text{conv}})$, $A$ for $A \cap (\gamma^{\text{conv}})^{-1}(\mu^{\text{conv}})$ and $F$ for $(F \cup R) \cap (\gamma^{\text{conv}})^{-1}(\mu^{\text{conv}})$. In the new notation we have
\begin{equation}\label{newtdefn}T = V^{[k] \setminus \{i+1\}} \cap \gamma^{-1}(\mu_1) \cap (\gamma^{\text{conv}})^{-1}(\mu^{\text{conv}}).\end{equation}
We also record here the codimension bounds
\[s + t^{\text{conv}} = \codim \gamma + \codim \gamma^{\text{conv}} \leq O((r^{\text{ind}} + \log_p \xi_0^{-1})^{O(1)})\]
and recall that $\gamma$ and $\gamma^{\text{conv}}$ are mixed-linear and $\mathcal{G}'$-supported. Misusing the notation, we write $\codim T = r^{\text{ind}} + s + t^{\text{conv}}$, which is not necessarily the codimension of the variety $T$ (recall that the codimension of $T$ defined to be the minimum possible codimension of a multiaffine map which gives $T$ as a layer), but is essentially the codimension of the mixed-linear map we use here to define $T$. The notation $\codim T$ is more suggestive and will simplify the calculations that follow.\\ 
\indent Again, we have
\[A = \Big\{x_{[k] \setminus \{i+1\}} \in T \colon \Big|Y^{\text{ind}}_{x_{[i]}} \cap (V \cap U^{\text{ind}})_{x_{[k] \setminus \{i+1\}}}\Big| \geq \Big(1 - \frac{\varepsilon_0}{4}\Big) \Big|W^{\text{ind}}_{x_{[i]}} \cap (V \cap U^{\text{ind}})_{x_{[k] \setminus \{i+1\}}}\Big|\Big\},\]
(note that the elements removed from initial version~\eqref{highdimconvadefn} of $A$ are also removed from the initial version of $T$), $A \setminus F \subset \on{DC}$
\begin{equation}\label{AdenseEqn}|A| = (1 - O(\varepsilon^{\Omega(1)} + \xi_0))|T|\end{equation}
and
\begin{equation}\label{FsparseEqn}|F| = O(\sqrt{\xi_0})|T|.\end{equation}
Additionally, convolving in direction $i+1$ turns $\psi^{\text{ind}}$ into a well-defined multi-$(D \cdot 20^{i})$-homomorphism on $\Big((W^{\text{ind}} \times G_{[i+2,k]}) \cap V \cap U^{\text{ind}}\Big) \cap \Big((\on{DC} \setminus F) \times G_{i+1} \Big)$. Property \textbf{(iv)} from the statement follows directly from the inductive hypothesis and property \textbf{(iii)} of Theorem~\ref{mainExtensionConvStep} (to get the stated bounds on the number of tri-arrangements we still need to specify the parameter $\xi_0$). It remains to show that convolving in direction $i+1$ of $Z^{\text{ind}}$ gives a set with the desired structure, that is, we need to find the desired structure in the set
\begin{align}\Big((W^{\text{ind}}& \times G_{[i+2,k]}) \cap V \cap U^{\text{ind}}\Big) \cap \Big((\on{DC} \setminus F) \times G_{i+1} \Big)\label{targetdomainset}\\
= &\Big((W^{\text{ind}} \times G_{[i+2,k]}) \cap V \cap U^{\text{ind}}\cap (\on{DC} \times G_{i+1}) \Big) \hspace{3pt}\setminus\hspace{3pt} \Big((W^{\text{ind}} \times G_{[i+2,k]}) \cap V \cap U^{\text{ind}}\cap (F \times G_{i+1}) \Big).\nonumber\end{align}

\vspace{\baselineskip}

\noindent \textbf{Picking dense columns via algebraic dependent random choice.} \chg{In this step of the proof, we take dense columns in direction $i + 1$, namely, columns indexed by points $(x_{[i]}, z_{[i+2, k]})$ such that $|Y^{\text{ind}}_{x_{[i]}} \cap (V \cap U^{\text{ind}})_{x_{[i]}, z_{[i+2, k]}}| \geq (1 - o(1)) \delta |G_{i+1}|$. However, to be able to relate the chosen set of columns to varieties, we do this via algebraic dependent random choice. That is, we shall pick random elements of $G_{i+1}$ and keep track of which linear combinations end up in $Y^{\text{ind}}_{x_{[i]}} \cap (V \cap U^{\text{ind}})_{x_{[i]}, z_{[i+2, k]}}$. This will be important in later steps of the proof (see~\eqref{cslicetolayerseqn}). Let us also remark that keeping track of linear combinations instead of simply counting which randomly chosen points end up in the column improves the efficiency of this step: the simpler argument would introduce an exponential dependence.}\\
\indent Let $m \in \mathbb{N}$ be a parameter to be specified later. Choose elements $a, y_1, \dots, y_m \in G_{i+1}$ independently and uniformly at random. For a set $S \subset G_{i+1}$, we write
\[L(S) = \{\lambda \in \mathbb{F}_p^m \colon a + \lambda \cdot y \in S\}\]
and we write $l(S)$ for the cardinality of $L(S)$. \chg{Thus, $l(S)$ is the number of affine combinations of the chosen elements that end up in $S$, and we expect that $l(S)$ is concentrated around the density of $S$ times $p^m$.}\\
\indent Let
\[C = \Big\{(x_{[i]}, z_{[i+2, k]}) \in T \colon l(Y^{\text{ind}}_{x_{[i]}} \cap (V \cap U^{\text{ind}})_{x_{[i]}, z_{[i+2, k]}}) \geq \Big(1 - \frac{\varepsilon_0}{3}\Big) \delta p^m\Big\}.\]
By Lemma~\ref{randomCosetInt},
\begin{itemize}
\item if $(x_{[i]}, z_{[i+2, k]})$ is such that $|Y^{\text{ind}}_{x_{[i]}} \cap (V \cap U^{\text{ind}})_{x_{[i]}, z_{[i+2, k]}}|\ \geq \Big(1 - \frac{\varepsilon_0}{4}\Big) \delta |G_{i+1}|$, then
\begin{equation}\label{probChoiceDense}\mathbb{P}_{a, y_{[m]}}\Big((x_{[i]}, z_{[i+2, k]}) \in C\Big) = 1 - O(p^{-m} \delta^{-1});\end{equation}
\item if $(x_{[i]}, z_{[i+2, k]})$ is such that $|Y^{\text{ind}}_{x_{[i]}} \cap (V \cap U^{\text{ind}})_{x_{[i]}, z_{[i+2, k]}}|\ \leq \Big(1 - \frac{\varepsilon_0}{2}\Big) \delta |G_{i+1}|$, then
\begin{equation}\label{probChoiceSparse}\mathbb{P}_{a, y_{[m]}}\Big((x_{[i]}, z_{[i+2, k]}) \in C\Big) = O(p^{-m} \delta^{-1}).\end{equation}
\end{itemize}
By~\eqref{AdenseEqn}, the first case holds for $(1 - O(\varepsilon^{\Omega(1)} + \xi_0))|T|$ elements $(x_{[i]}, z_{[i+2, k]}) \in T$. Thus $\ex_{a, y_{[m]}} |T \setminus C| \leq O(\varepsilon^{\Omega(1)} + \xi_0 + p^{-m} \delta^{-1})|T|$.\\
\indent Let
\[C^{\text{sparse}} = \Big\{(x_{[i]}, z_{[i+2, k]}) \in C \colon \Big|Y^{\text{ind}}_{x_{[i]}} \cap (V \cap U^{\text{ind}})_{x_{[i]}, z_{[i+2, k]}}\Big| \leq \Big(1 - \frac{\varepsilon_0}{2}\Big) \delta |G_{i+1}|\Big\}.\]
By the calculation in the second case, $\ex_{a, y_{[m]}} |C^{\text{sparse}}| = O(p^{-m} \delta^{-1}) |T|$. Therefore
\[\exx_{a, y_{[m]}} |T \setminus C| + p^{m/2} |C^{\text{sparse}}|= O(\varepsilon^{\Omega(1)} + \xi_0 + p^{-m/2} \delta^{-1}) |T|,\]
so there exists a choice of $a, y_{[m]}$ such that $|C|\ \geq (1 - O(\varepsilon^{\Omega(1)} + \xi_0))|T|$ and $|C^{\text{sparse}}| \leq p^{-m/3} |T|$, provided that $p^{-m} \leq \cons\, \delta^{\con} (\varepsilon^{\con} + \xi_0^{\con})$ and $\varepsilon, \xi_0$ are smaller than some sufficiently small absolute constant. Let $C^{\text{bad}} = C \cap F$, which by~\eqref{FsparseEqn} has size at most $|F| = O(\sqrt{\xi_0})|T|$. Note also that when $x_{[k] \setminus \{i+1\}} \in C \setminus (C^{\text{sparse}} \cup C^{\text{bad}})$, then 
\[\Big|Z^{\text{ind}}_{x_{[k] \setminus \{i+1\}}}\Big| \geq (1 - \varepsilon_0)\Big|\Big(V \cap U^{\text{ind}} \cap (W^{\text{ind}} \times G_{[i+2,k]})\Big)_{x_{[k] \setminus \{i+1\}}}\Big|,\]
and in particular $x_{[k] \setminus \{i+1\}} \in \on{DC} \setminus F$. Thus, 
\begin{equation}\label{targetdomainset2} C \setminus (C^{\text{sparse}} \cup C^{\text{bad}}) \,\subset\, \on{DC} \setminus F.\end{equation}
\vspace{\baselineskip}

\noindent\textbf{Relating $C$ to varieties.} \chg{We now exploit the way we constructed $C$ to show that $C$ is related to the layers of some mixed-linear map.} Observe that we may rewrite $C$ in a different form. Let $\tau_1 \colon G_{[k]} \to \mathbb{F}^{t_1}
$ be a $\mathcal{G}$-supported multiaffine map, linear in direction $i+1$, and let $\kappa_1 \in \mathbb{F}^{t_1}$ be such that 
\[(T \times G_{i+1}) \cap (W^{\text{ind}} \times G_{[i+2,k]}) \cap V \cap U^{\text{ind}} = (T \times G_{i+1}) \cap (\tau_1)^{-1}(\kappa_1).\] 
(The fact that we may find such a map $\tau_1$ follows from the work in footnote~\ref{linearconstfootnote}.) Fix an element $x_{[i]} \in G_{[i]}$ and let $\{\lambda^1, \dots, \lambda^e\} = L(Y^{\text{ind}}_{x_{[i]}})$. Then
\begin{align}C_{x_{[i]}} = &\Big\{z_{[i+2, k]} \in T_{x_{[i]}} \colon \Big|\Big\{j \in [e] \colon a + \lambda^j \cdot y \in (V \cap U^{\text{ind}} \cap (W^{\text{ind}} \times G_{[i+2,k]}))_{x_{[i]}, z_{[i+2,k]}}\Big\}\Big| \geq \Big(1 - \frac{\varepsilon_0}{3}\Big) \delta p^m\Big\}\nonumber\\
=&\bigcup_{\substack{M \subset \{\lambda^{[e]}\}\\|M| \geq  \big(1 - \frac{\varepsilon_0}{3}\big)\delta p^m}} \Big\{z_{[i+2, k]} \in T_{x_{[i]}} \colon (\forall \mu \in M)\,\,a + \mu \cdot y \in (V \cap U^{\text{ind}}\cap (W^{\text{ind}} \times G_{[i+2,k]}))_{x_{[i]}, z_{[i+2,k]}}\Big\}\nonumber\\
=&\bigcup_{\substack{M \subset \{\lambda^{[e]}\}\\|M| \geq  \big(1 - \frac{\varepsilon_0}{3}\big)\delta p^m}}  T_{x_{[i]}} \cap \bigg(\bigcap_{\mu \in M} (V \cap U^{\text{ind}} \cap (W^{\text{ind}} \times G_{[i+2,k]}))_{x_{[i]},\ls{i+1}\,a + \mu \cdot y}\bigg).\label{cslicetolayerseqn}\end{align} 
Define a map
\[\theta(x_{[i]}, z_{[i+2,k]}) = \Big(\tau_1(x_{[i]},\ls{i+1}\,a, z_{[i+2,k]}), \tau_1(x_{[i]},\ls{i+1}\,y_1, z_{[i+2,k]}), \dots, \tau_1(x_{[i]},\ls{i+1}\,y_m, z_{[i+2,k]})\Big).\]
Then $\codim \theta \leq (m+1) r^{\text{ind}}$ and $\theta$ is $\mathcal{G}'$-supported. We write $\mathbb{F}_p^{\codim \theta}$ for the codomain of $\theta$ to avoid introducing more notation. Notice that when $(x_{[i]}, \ls{i+1}\,a + \mu \cdot y) \in W^{\text{ind}}$ (which is the case for the relevant values $\mu$ in~\eqref{cslicetolayerseqn} since $Y^{\text{ind}} \subset W^{\text{ind}}$), we have
\begin{align*}T_{x_{[i]}} \cap (V \cap U^{\text{ind}})_{x_{[i]}, \ls{i+1}\,a + \mu \cdot y} &=T_{x_{[i]}} \cap \{z_{[i+2,k]} \in G_{[i+2,k]} \colon \tau_1(x_{[i]},\ls{i+1}\,a + \mu \cdot y, z_{[i+2,k]}) = \kappa_1\}\\
&=T_{x_{[i]}} \cap \Big\{z_{[i+2,k]} \in G_{[i+2,k]} \colon \tau_1(x_{[i]},\ls{i+1}\,a, z_{[i+2,k]}) + \mu_1 \tau_1(x_{[i]},\ls{i+1}\,y_1, z_{[i+2,k]}) + \dots\\
&\hspace{8cm} + \mu_m \tau_1(x_{[i]},\ls{i+1}\,y_m, z_{[i+2,k]}) = \kappa_1\Big\}\\
&=T_{x_{[i]}} \cap \Big\{z_{[i+2,k]} \in G_{[i+2,k]} \colon \Big(1, \mu_1, \dots, \mu_m\Big) \cdot \theta(x_{[i]}, z_{[i+2,k]}) = \kappa_1\Big\}.\end{align*}
Combining this observation with~\eqref{cslicetolayerseqn} we conclude that for each $x_{[i]} \in G_{[i]}$, we have a collection $R_{x_{[i]}} \subset \mathbb{F}_p^{t_1 (m+1)}$ of values of $\theta$, such that
\begin{equation}C_{x_{[i]}} = \bigcup_{u \in R_{x_{[i]}}} T_{x_{[i]}} \cap \Big\{z_{[i+2, k]} \in G_{[i+2, k]}\colon \theta(x_{[i]}, z_{[i+2, k]}) = u\Big\}.\label{introducingrlinearcombs}\end{equation}
\vspace{\baselineskip}

\noindent\textbf{Digression.} Before proceeding any further with the proof, we pause to make a comment about the case $i=0$, elaborating on a remark preceding the proof. In that case, the equality above becomes simply $C = \bigcup_{u \in R} T \cap \theta^{-1}(u)$ for a subset $R\subset \mathbb{F}_p^{t_1 (m+1)}$. We may then immediately proceed to step~\eqref{rtildestep} below, and pick the relevant element $u \in \tilde{R}$, where $\tilde{R}$ is defined in that step -- when $i = 0$ any $u \in \tilde{R}$ that gives non-empty varieties in the end works. (We interpret $Y \times G_{[k]}$ and $W \times G_{[k]}$ as $G_{[k]}$ and we ignore $x_{[i]}$ in the subscripts.) Once that is done, recall that we only need to specify $U$ and $Z$ in this case, which we do as in~\eqref{udefnendproof} and~\eqref{zdefnendproof}, using the same argument as in cases $i \not= 0$.\\
\indent We now proceed with the proof.\\
\vspace{\baselineskip}

\noindent \chg{\textbf{Regularization of $T$ and $\theta$.} Recall that we eventually need to construct the sets $Y \subseteq G_{[i]}$ and $Z \subseteq G_{[k]}$ and show that $Z$ misses very few points inside $(Y \times G_{[i+1,k]}) \cap V \cap U$. In order to show this, we need precise control over the fibres of the variety $T$ and the layers of the map $\theta$ in directions $G_{[i + 2, k]}$. In this step, we use Theorem~\ref{simFibresThm} to achieve this control.}\\
\indent Let $\eta_2 > 0$. Apply Theorem~\ref{simFibresThm} to $T$ and $\theta$ in the directions $G_{[i+2,k]}$ with an error parameter $\eta_2 > 0$ to be chosen later.\footnote{As with a previous step in this proof, we are actually applying the theorem to the map $(\rho, \theta)$, where $\rho$ is a mixed-linear map, each component of which depends on at least one coordinate in $G_{[i+2,k]}$, such that $T = (T^{[i]} \times G_{[i+2,k]}) \cap \{\rho = \lambda\}$ for some $\lambda \in \mathbb{F}_p^{\codim \rho}$. Here, $T^{[i]} \subset G_{[i]}$ is the variety given by the multilinear forms that define $T$ and are $(\mathcal{G} \cap \mathcal{P}[i])$-supported. Hence, each $\rho_i$ depends on a set of coordinates that lies in $\mathcal{G} \setminus \mathcal{P}([i])$. Therefore, the multiaffine map given by Theorem~\ref{simFibresThm} is $\mathcal{G}'$-supported. We put these details in a footnote to keep the main line of the argument clearer.} We get a positive integer $s_2 =O\Big((\log_p \eta_2^{-1} + \codim T + \codim \theta)^{O(1)}\Big)$, a $\mathcal{G}'$-supported multiaffine map $\gamma_2 \colon G_{[i]} \to \mathbb{F}_p^{s_2}$, a collection of values $M_2 \subset \mathbb{F}_p^{s_2}$, and maps $c_\mu \colon \mathbb{F}_p^{\codim \theta} \to [0,1]$ for $\mu \in M_2$, such that 
\[|\gamma_2^{-1}(M_2)| \geq (1-\eta_2) |G_{[i]}|\]
and
\begin{align}
|T_{x_{[i]}} \cap \{z_{[i+2,k]} \in G_{[i+2,k]} \colon \theta(x_{[i]}, z_{[i+2, k]}) = u\}| \in [c_\mu(u), c_\mu(u) + \eta_2] \cdot |G_{[i+2,k]}|\label{gamma2EqnPropDensityApprox}\end{align}
for every $\mu \in M_2$, every $x_{[i]} \in \gamma_2^{-1}(\mu)\cap T^{[i]}$ and every $u \in \mathbb{F}_p^{\codim \theta}$. Moreover, provided $\eta_2 < p^{-k (\codim T + \codim \theta)}$, we may assume that if $c_\mu(u) = 0$, then actually 
\[T_{x_{[i]}} \cap \{z_{[i+2,k]} \in G_{[i+2,k]} \colon \theta(x_{[i]}, z_{[i+2, k]}) = u\} = \emptyset.\] 
Without loss of generality, when $c_\mu(u) > 0$, then in fact $c_\mu(u) \geq p^{-k (\codim T + \codim \theta)}$.\\

\chg{In the rest of the work in this step, we need to find a suitable layer of $\gamma_2^{-1}(\mu_2)$, for some $\mu_2 \in M_2$, such that $C, C^{\text{sparse}}$ and $ C^{\text{bad}}$ have the expected intersections with $\gamma_2^{-1}(\mu_2) \times G_{[i+2,k]}$.}\\

Let $0 < \varepsilon'' = O(\varepsilon^{\Omega(1)} + \xi_0)$ be a quantity such that $|C| \geq (1 - \varepsilon'')|T|$. Without loss of generality $\varepsilon'' \geq \varepsilon + \xi_0$ (although smaller $\varepsilon''$ can only help, we need such a lower bound to simplify the notation and calculations). Let $M_2'$ be the set of all $\mu \in M_2$ such that
\[|C \cap (\gamma_2^{-1}(\mu) \times G_{[i+2,k]})| \,\,\geq\, (1 - 3\varepsilon'')|T \cap (\gamma_2^{-1}(\mu) \times G_{[i+2,k]})| > 0.\]
Let $0 < \xi_0'' = p^{-m/3} + O(\sqrt{\xi_0})$ be a quantity such that $|C^{\text{sparse}} \cup C^{\text{bad}}| \leq \xi_0'' |T|$; without loss of generality $\xi_0'' \geq \xi_0$. Let $M_2''$ be the set of all $\mu \in M_2$ such that 
\[|(C^{\text{sparse}} \cup C^{\text{bad}}) \cap (\gamma_2^{-1}(\mu) \times G_{[i+2,k]})| \,\,\leq \,3 \xi_0'' |T \cap (\gamma_2^{-1}(\mu) \times G_{[i+2,k]})|.\]
We then have
\begin{align*}\sum_{\mu \in M_2 \setminus M_2'} |T \cap (\gamma_2^{-1}(\mu) \times G_{[i+2,k]})|& \leq \frac{{\varepsilon''}^{-1}}{3} \sum_{\mu \in M_2 \setminus M_2'} |(T \setminus C) \cap (\gamma_2^{-1}(\mu) \times G_{[i+2,k]})|\\
&\leq \frac{{\varepsilon''}^{-1}}{3} |T\setminus C|\\
&\leq \frac{1}{3} |T|.\end{align*}
Also,
\begin{align*}\sum_{\mu \in M'_2 \setminus M''_2} |T \cap (\gamma_2^{-1}(\mu) \times G_{[i+2,k]})|& \leq \frac{{\xi_0''}^{-1}}{3} \sum_{\mu \in M_2 \setminus M_2'} |(C^{\text{sparse}} \cup C^{\text{bad}}) \cap (\gamma_2^{-1}(\mu) \times G_{[i+2,k]})|\\
&\leq \frac{{\xi_0''}^{-1}}{3}  |C^{\text{sparse}} \cup C^{\text{bad}}|\\
&\leq \frac{1}{3} |T|.\end{align*}
Finally, 
\[\sum_{\mu \in \mathbb{F}_p^{s_2} \setminus M_2} |T \cap (\gamma_2^{-1}(\mu) \times G_{[i+2,k]})| \leq |\gamma_2^{-1}(\mathbb{F}_p^{s_2} \setminus M_2)| \cdot |G_{[i+2, k]}| \leq \eta_2 p^{k\codim T} |T|.\]
Hence, provided $\eta_2 \leq \frac{1}{4} p^{-k\codim T}$, we may select $\mu_2 \in M'_2 \cap M''_2$. Note that $\mu_2 \in M'_2$ in particular implies that $T \cap (\gamma_2^{-1}(\mu_2) \times G_{[i+2,k]}) \not= \emptyset$ and hence $T^{[i]} \cap (\gamma_2)^{-1}(\mu_2) \not= \emptyset$.\\

\vspace{\baselineskip}

\noindent\chg{\textbf{Defining the set $Y$.} We now proceed to define the set $Y \subseteq G_{[i]}$, which is one of the key objects in the conclusion of the proposition. We shall first define auxiliary sets $\tilde{Y}_1$ and $\tilde{Y}_2$ as certain subsets $T^{[i]} \cap \gamma_2^{-1}(\mu_2)$ for which $C, C^{\text{sparse}}$ and $ C^{\text{bad}}$ have well-behaved fibres. The set $Y$ will finally be obtained from $\tilde{Y}_1 \cap \tilde{Y}_2$ by exploiting the linear combinations $R_{x_{[i]}}$, introduced previously in~\eqref{introducingrlinearcombs}.}\\
\indent Let $\tilde{Y}_1$ be the set of $x_{[i]} \in T^{[i]} \cap \gamma_2^{-1}(\mu_2)$ such that $|C_{x_{[i]}}| \geq (1 - \sqrt{\varepsilon''})|T_{x_{[i]}}|$. Then
\begin{align*}\sum_{x_{[i]} \in T^{[i]} \cap (\gamma_2)^{-1}(\mu_2) \setminus \tilde{Y}_1} |T_{x_{[i]}}|& \leq  {\varepsilon''}^{-1/2}\sum_{x_{[i]} \in T^{[i]} \cap \gamma_2^{-1}(\mu_2) \setminus \tilde{Y}_1} |T_{x_{[i]}} \setminus C_{x_{[i]}}|\\
&\leq  {\varepsilon''}^{-1/2} |T \cap (\gamma_2^{-1}(\mu_2) \times G_{[i+2,k]}) \setminus C|\\
&\leq  3\sqrt{\varepsilon''} |T \cap (\gamma_2^{-1}(\mu_2) \times G_{[i+2,k]})|.\end{align*}
Let $\tilde{Y}_2$ be the set of $x_{[i]} \in T^{[i]} \cap \gamma_2^{-1}(\mu_2)$ such that $|(C^{\text{sparse}} \cup C^{\text{bad}})_{x_{[i]}}| \leq \sqrt{\xi_0''} |T_{x_{[i]}}|$. Then
\begin{align*} \sum_{x_{[i]} \in T^{[i]} \cap \gamma_2^{-1}(\mu_2) \setminus \tilde{Y}_2} |T_{x_{[i]}}| &\leq  {\xi_0''}^{-1/2}\sum_{x_{[i]} \in T^{[i]} \cap \gamma_2^{-1}(\mu_2) \setminus \tilde{Y}_2} |(C^{\text{sparse}} \cup C^{\text{bad}})_{x_{[i]}}|\\
&\leq  {\xi_0''}^{-1/2} |(C^{\text{sparse}} \cup C^{\text{bad}}) \cap (\gamma_2^{-1}(\mu_2) \times G_{[i+2,k]})|\\
&\leq  3 \sqrt{\xi_0''}|T \cap (\gamma_2^{-1}(\mu_2) \times G_{[i+2,k]})|.
\end{align*}
Let $\tilde{Y} = \tilde{Y}_1 \cap \tilde{Y}_2$. Write $\sigma$ for $\sum_{u \in \mathbb{F}_p^{\codim \theta}} c_{\mu_2}(u)$. Note that when $x_{[i]}$ is such that $(T \cap (\gamma_2^{-1}(\mu_2) \times G_{[i+2,k]}))_{x_{[i]}} \not= \emptyset$, from~\eqref{gamma2EqnPropDensityApprox} and Lemma~\ref{varsizelemma} we have 
\[(\sigma + p^{\codim \theta}\eta_2) |G_{[i+2,k]}|\,\, \geq \,|(T \cap (\gamma_2^{-1}(\mu_2) \times G_{[i+2,k]}))_{x_{[i]}}|\,\, =\, |T_{x_{[i]}}| \,\,\geq\, p^{-k \codim T}|G_{[i+2,k]}|.\]
Since $T \cap (\gamma_2^{-1}(\mu_2) \times G_{[i+2,k]}) \not= \emptyset$, provided $\eta_2 \leq \frac{1}{2}p^{-k \codim T - \codim \theta}$, we conclude that $\sigma \geq \frac{1}{2}p^{-k \codim T}$.\\
\indent Using~\eqref{gamma2EqnPropDensityApprox} twice, we have 
\begin{align*}|\tilde{Y}| (\sigma  + p^{\codim \theta} \eta_2) |G_{[i+2,k]}| &\geq \sum_{x_{[i]} \in \tilde{Y}} |T_{x_{[i]}}|\\
&\geq  (1 - 3\sqrt{\varepsilon''} - 3 \sqrt{\xi_0''}) |T \cap ((\gamma_2)^{-1}(\mu_2) \times G_{[i+2,k]})|\\
&=  (1 - 3\sqrt{\varepsilon''} - 3 \sqrt{\xi_0''}) \sum_{x_{[i]} \in T^{[i]} \cap (\gamma_2)^{-1}(\mu_2)} |T_{x_{[i]}}|\\
&\geq (1 - 3\sqrt{\varepsilon''} - 3 \sqrt{\xi_0''}) \sum_{x_{[i]} \in T^{[i]} \cap (\gamma_2)^{-1}(\mu_2)} \sigma |G_{[i+2,k]}|.
\end{align*}
It follows that
\begin{align*}
|\tilde{Y}| &\geq (1 - 3\sqrt{\varepsilon''} - 3 \sqrt{\xi_0''}) \frac{\sigma}{\sigma + p^{\codim \theta}\eta_2} |T^{[i]} \cap (\gamma_2)^{-1}(\mu_2)|\\ 
&=(1 - O(\sqrt{\varepsilon''} + \sqrt{\xi_0''})) |T^{[i]} \cap (\gamma_2)^{-1}(\mu_2)|,\\
\end{align*}
provided that $\eta_2 \leq \frac{1}{2}(\sqrt{\varepsilon''} + \sqrt{\xi_0''}) p^{-k \codim T - \codim \theta}$.\\

\indent For each $x_{[i]} \in \tilde{Y}$ let $\tilde{R}_{x_{[i]}}$ be the set of all $u \in R_{x_{[i]}}$ such that 
\begin{align}\Big|(C^{\text{sparse}} \cup C^{\text{bad}})_{x_{[i]}} \cap \Big\{z_{[i+2, k]} &\in G_{[i+2, k]}\colon \theta(x_{[i]}, z_{[i+2, k]}) = u\Big\}\Big|\nonumber\\ 
&\leq \sqrt[4]{\xi_0''}\Big|T_{x_{[i]}} \cap \Big\{z_{[i+2, k]} \in G_{[i+2, k]}\colon \theta(x_{[i]}, z_{[i+2, k]}) = u\Big\}\Big|.\label{rtildestep}\end{align}
Then by averaging,
\[\bigg|\bigcup_{u \in \tilde{R}_{x_{[i]}}} T_{x_{[i]}} \cap \Big\{z_{[i+2, k]} \in G_{[i+2, k]}\colon \theta(x_{[i]}, z_{[i+2, k]}) = u\Big\}\bigg| = (1 - O(\sqrt{\varepsilon''} + \sqrt[4]{\xi_0''}))|T_{x_{[i]}}|.\]
Pick a random value $u \in \mathbb{F}_p^{\codim \theta}$, according to the probability distribution $p_u = \sigma^{-1} c_{\mu_2}(u)$. Define $Y = \{x_{[i]} \in \tilde{Y} \colon u \in \tilde{R}_{x_{[i]}}\}$. Then
\begin{align*}\sigma \ex |Y| = &\sigma \sum_{x_{[i]} \in \tilde{Y}} \mathbb{P}(u \in \tilde{R}_{x_{[i]}})\\
 =& \sum_{x_{[i]} \in \tilde{Y}} \sum_{u \in \mathbb{F}_p^{\codim \theta}} \sigma p_u \mathbbm{1}(u \in \tilde{R}_{x_{[i]}})\\
= & \sum_{x_{[i]} \in \tilde{Y}} \sum_{u \in \mathbb{F}_p^{\codim \theta}} c_{\mu_2}(u) \mathbbm{1}(u \in \tilde{R}_{x_{[i]}})\\
\geq & \sum_{x_{[i]} \in \tilde{Y}} \sum_{u \in \mathbb{F}_p^{\codim \theta}}\frac{c_{\mu_2}(u)}{c_{\mu_2}(u) + \eta_2} |G_{[i+2,k]}|^{-1}|T_{x_{[i]}} \cap \{z_{[i+2,k]} \in G_{[i+2,k]} \colon \theta(x_{[i]}, z_{[i+2, k]}) = u\}|\mathbbm{1}(u \in \tilde{R}_{x_{[i]}})\\
\geq& (1 - \eta_2 p^{k \codim T + k \codim \theta}) \sum_{x_{[i]} \in \tilde{Y}} \sum_{u \in \tilde{R}_{x_{[i]}}} |G_{[i+2,k]}|^{-1} |T_{x_{[i]}} \cap \{z_{[i+2,k]} \in G_{[i+2,k]} \colon \theta(x_{[i]}, z_{[i+2, k]}) = u\}|\\
\geq & (1 - O(\sqrt{\varepsilon''} + \sqrt[4]{\xi_0''})) \sum_{x_{[i]} \in \tilde{Y}}  |G_{[i+2,k]}|^{-1} |T_{x_{[i]}}|\\
\geq & (1 - O(\sqrt{\varepsilon''} + \sqrt[4]{\xi_0''})) |\tilde{Y}| \sigma,\end{align*}
where the first and the last inequality both use~\eqref{gamma2EqnPropDensityApprox}, while in the second inequality we use the fact that either $c_{\mu_2}(u) \geq p^{-k (\codim T + \codim \theta)}$, or $c_{\mu_2}(u) = 0$, the latter possibility implying that
\[T_{x_{[i]}} \cap \{z_{[i+2,k]} \in G_{[i+2,k]} \colon \theta(x_{[i]}, z_{[i+2, k]}) = u\} = \emptyset.\]
Hence, there is a choice of $u$ such that $c_{\mu_2}(u) \geq p^{-k (\codim T + \codim \theta)}$ and
\begin{equation}\label{YsizeFinalEqn}|Y| = (1 - O(\sqrt{\varepsilon''} + \sqrt[4]{\xi_0''})) |T^{[i]} \cap (\gamma_2)^{-1}(\mu_2)|.\end{equation}

\vspace{\baselineskip}

\noindent\chg{\textbf{Defining the varieties $U$ and $W$.}} Recall from~\eqref{newtdefn} that $T = V^{[k] \setminus \{i+1\}} \cap \gamma^{-1}(\mu_1) \cap (\gamma^{\text{conv}})^{-1}(\mu^{\text{conv}})$ and that $\gamma$ and $\gamma^{\text{conv}}$ are mixed-linear $\mathcal{G}'$-supported multiaffine maps. Let $\tilde{T}$ be the variety defined by $\{x_{[i]} \in G_{[i]} \colon (\forall j \in \mathcal{J})\, \gamma_j(x_{[i]}) = (\mu_1)_j\,\land\,(\forall j \in \mathcal{J}^{\text{conv}})\, \gamma^{\text{conv}}_j(x_{[i]}) = \mu^{\text{conv}}_j\}$, where $\mathcal{J}$ is the set of all $j$ such that $\gamma_j$ does not depend on any coordinate among $G_{[i+2,k]}$ and $\mathcal{J}^{\text{conv}}$ is the set of all $j$ such that $\gamma^{\text{conv}}_j$ does not depend on any coordinate among $G_{[i+2,k]}$. Note that the  variety $T^{[i]}$ satisfies $T^{[i]} = \tilde{T} \cap V^{[i]}$. Let 
\begin{equation}U = ((\gamma^{-1}(\mu_1) \cap (\gamma^{\text{conv}})^{-1}(\mu^{\text{conv}}) \cap\theta^{-1}(u)) \times G_{i+1}) \cap U^{\text{ind}} \cap (W^{\text{ind}} \times G_{[i+2,k]}) \subset G_{[k]}\label{udefnendproof}\end{equation}
and
\[W = \tilde{T} \cap (\gamma_2)^{-1}(\mu_2) \subset G_{[i]},\]
both of which are $\mathcal{G}'$-supported. Then we have
\[Y \subset T^{[i]} \cap (\gamma_2)^{-1}(\mu_2) = \tilde{T} \cap V^{[i]} \cap (\gamma_2)^{-1}(\mu_2) = W \cap V^{[i]}.\]

\noindent\chg{\textbf{Checking the density of $Y$.} We may now check that $(Y \times G_{[i+1,k]}) \cap V \cap U$ is sufficiently dense inside $(W \times G_{[i+1,k]}) \cap V \cap U$, which is property \textbf{(ii)}.} Indeed, we have
\begin{align*}|(Y &\times G_{[i+1,k]}) \cap V \cap U|\\
&= |(Y \times G_{[i+1,k]}) \cap ((\gamma^{-1}(\mu_1) \cap (\gamma^{\text{conv}})^{-1}(\mu^{\text{conv}}) \cap \theta^{-1}(u)) \times G_{i+1}) \cap V \cap U^{\text{ind}} \cap (W^{\text{ind}} \times G_{[i+2,k]})|\\ 
&= |(Y \times G_{[i+1,k]}) \cap ((T \cap \theta^{-1}(u)) \times G_{i+1}) \cap V \cap U^{\text{ind}} \cap (W^{\text{ind}} \times G_{[i+2,k]})|\\
&= \sum_{x_{[k] \setminus \{i+1\}} \in (Y \times G_{[i+2,k]}) \cap T \cap \theta^{-1}(u)}\Big|\Big(V \cap U^{\text{ind}} \cap (W^{\text{ind}} \times G_{[i+2,k]})\Big)_{x_{[k] \setminus \{i+1\}}}\Big|\\
&=|(Y \times G_{[i+2,k]}) \cap T \cap \theta^{-1}(u)| \cdot \delta |G_{i+1}|\hspace{1cm}\text{(by~\eqref{deltaEqnPropTElt})}\\
&=\sum_{x_{[i]} \in Y} |T_{x_{[i]}} \cap \{z_{[i+2,k]} \in G_{[i+2,k]} \colon \theta(x_{[i]}, z_{[i+2, k]}) = u\}| \cdot \delta |G_{i+1}|\\
&\geq|Y| \delta c_{\mu_2}(u) |G_{[i+1,k]}|\hspace{1cm} \text{(by ~\eqref{gamma2EqnPropDensityApprox})}\\
&= (1 - O(\sqrt{\varepsilon''} + \sqrt[4]{\xi_0''})) \delta c_{\mu_2}(u) |T^{[i]} \cap (\gamma_2)^{-1}(\mu_2)| |G_{[i+1,k]}|\hspace{1cm}\text{(by~\eqref{YsizeFinalEqn})}\\
&=\frac{c_{\mu_2}(u)}{c_{\mu_2}(u) + \eta_2} (1 - O(\sqrt{\varepsilon''} + \sqrt[4]{\xi_0''}))\\
&\hspace{1cm} \sum_{x_{[i]} \in T^{[i]} \cap (\gamma_2)^{-1}(\mu_2)} |T_{x_{[i]}} \cap \{z_{[i+2,k]} \in G_{[i+2,k]} \colon \theta(x_{[i]}, z_{[i+2, k]}) = u\}| \cdot \delta |G_{i+1}|\hspace{1cm}\text{(by~\eqref{gamma2EqnPropDensityApprox})}\\ 
&= (1 - O(\sqrt{\varepsilon''} + \sqrt[4]{\xi_0''})) \Big|\Big((T^{[i]} \cap (\gamma_2)^{-1}(\mu_2)) \times G_{[i+2,k]}\Big) \cap T \cap \theta^{-1}(u)\Big|\cdot\delta |G_{i+1}|\\
&\hspace{3cm}\text{(provided $\eta_2 \leq (\sqrt{\varepsilon''} + \sqrt[4]{\xi_0''}) p^{-k \codim T - k\codim \theta}$)}\\
&= (1 - O(\sqrt{\varepsilon''} + \sqrt[4]{\xi_0''}))\\
&\hspace{1cm} \sum_{x_{[k] \setminus \{i+1\}} \in ((T^{[i]} \cap (\gamma_2)^{-1}(\mu_2)) \times G_{[i+2,k]}) \cap T \cap \theta^{-1}(u)} \Big|\Big(V \cap U^{\text{ind}} \cap (W^{\text{ind}} \times G_{[i+2,k]})\Big)_{x_{[k] \setminus \{i+1\}}}\Big|\hspace{1cm}\text{(by~\eqref{deltaEqnPropTElt})}\\
\\
&=(1 - O(\sqrt{\varepsilon''} + \sqrt[4]{\xi_0''}))\\ 
&\hspace{1cm}\Big|\Big((T^{[i]} \cap (\gamma_2)^{-1}(\mu_2)) \times G_{[i+1,k]}\Big) \cap \Big((T \cap\theta^{-1}(u)) \times G_{i+1}\Big) \cap V \cap U^{\text{ind}} \cap (W^{\text{ind}} \times G_{[i+2,k]})\Big|\\
&=(1 - O(\sqrt{\varepsilon''} + \sqrt[4]{\xi_0''}))\\ 
&\hspace{1cm}\Big|\Big((\tilde{T} \cap (\gamma_2)^{-1}(\mu_2)) \times G_{[i+1,k]}\Big) \cap \Big((\gamma^{-1}(\mu_1) \cap(\gamma^{\text{conv}})^{-1}(\mu^{\text{conv}})\cap\theta^{-1}(u)) \times G_{i+1}\Big)\\
&\hspace{11cm}\cap V \cap U^{\text{ind}} \cap (W^{\text{ind}} \times G_{[i+2,k]})\Big|\\
&=(1 - O(\sqrt{\varepsilon''} + \sqrt[4]{\xi_0''}))\Big|\Big(W \times G_{[i+1,k]}\Big) \cap V \cap U\Big|.\\
\end{align*}

\noindent\chg{\textbf{Defining the set $Z$.}} Finally, set
\begin{equation}Z = \Big((Y \times G_{[i+1,k]}) \cap ((T \cap \theta^{-1}(u)) \times G_{i+1}) \cap V \cap U^{\text{ind}} \cap (W^{\text{ind}} \times G_{[i+2,k]})\Big) \setminus \Big((C^{\text{sparse}} \cup C^{\text{bad}}) \times G_{i+1}\Big).\label{zdefnendproof}\end{equation}
Note that $Z$ satisfies
\begin{align*}Z &= \Big((Y \times G_{[i+1,k]}) \cap ((\gamma^{-1}(\mu_1) \cap(\gamma^{\text{conv}})^{-1}(\mu^{\text{conv}})\cap \theta^{-1}(u)) \times G_{i+1}) \cap V \cap U^{\text{ind}} \cap (W^{\text{ind}} \times G_{[i+2,k]})\Big)\\
&\hspace{13cm} \setminus \Big((C^{\text{sparse}} \cup C^{\text{bad}}) \times G_{i+1}\Big)\\
&= (Y \times G_{[i+1,k]}) \cap V \cap U\setminus \Big((C^{\text{sparse}} \cup C^{\text{bad}}) \times G_{i+1}\Big)\\
&\subset (Y \times G_{[i+1,k]}) \cap V \cap U.\end{align*}
Also, we claim that $Z$ is the subset of the domain of the new convolved map that was given by expression~\eqref{targetdomainset}. By~\eqref{targetdomainset2}, it suffices to show that $Z \subset (C \setminus (C^{\text{sparse}} \cup C^{\text{bad}})) \times G_{i+1}$. To that end, let $x_{[k]} \in Z$ be arbitrary. From the definition of $Z$, we already know that $x_{[k] \setminus \{i+1\}} \notin C^{\text{sparse}} \cup C^{\text{bad}}$. Proceeding further, we see that $x_{[i]} \in Y$, so we have $x_{[i]} \in \tilde{Y}$ and $u \in \tilde{R}_{x_{[i]}} \subset R_{x_{[i]}}$. But this further implies that $T_{x_{[i]}} \cap (\theta^{-1}(u))_{x_{[i]}} \subset C_{x_{[i]}}$. From the definition of $Z$, we have $x_{[k] \setminus \{i+1\}} \in T \cap \theta^{-1}(u)$, which completes the proof that $x_{[k] \setminus \{i+1\}} \in C$.\\ 

\noindent\chg{\textbf{Checking the density of $Z$.}} The last property of the set $Z$ is that it is very dense relative to the set $(Y \times G_{[i+1,k]}) \cap V \cap U$, which we now show. We have
\begin{align*}\Big|\Big(&(Y \times G_{[i+1,k]}) \cap V \cap U\Big) \setminus Z\Big|\\
&=\Big|\Big((Y \times G_{[i+1,k]}) \cap ((T \cap\theta^{-1}(u)) \times G_{i+1}) \cap V \cap U^{\text{ind}} \cap (W^{\text{ind}} \times G_{[i+2,k]})\Big) \cap \Big((C^{\text{sparse}} \cup C^{\text{bad}}) \times G_{i+1}\Big)\Big|\\ 
&=\Big|(Y \times G_{[i+2,k]}) \cap T \cap \theta^{-1}(u) \cap (C^{\text{sparse}} \cup C^{\text{bad}})\Big| \cdot \delta |G_{i+1}|\hspace{1cm}\text{(by~\eqref{deltaEqnPropTElt})}\\
&=\sum_{x_{[i]} \in Y} |(T \cap \theta^{-1}(u))_{x_{[i]}} \cap (C^{\text{sparse}} \cup C^{\text{bad}})_{x_{[i]}}| \cdot \delta |G_{i+1}|\\
&\leq\sqrt[4]{\xi_0''}\sum_{x_{[i]} \in Y} |(T \cap \theta^{-1}(u))_{x_{[i]}}| \cdot \delta |G_{i+1}|\hspace{1cm}\text{(since $u \in \tilde{R}_{x_{[i]}}$ for all $x_{[i]} \in Y$)}\\
&=\sqrt[4]{\xi_0''} |(Y \times G_{[i+2,k]}) \cap T \cap \theta^{-1}(u)| \cdot \delta |G_{i+1}|\\
&=\sqrt[4]{\xi_0''}\Big|(Y \times G_{[i+1,k]}) \cap ((T \cap \theta^{-1}(u)) \times G_{i+1}) \cap V \cap U^{\text{ind}} \cap (W^{\text{ind}} \times G_{[i+2,k]})\Big|\hspace{1cm}\text{(by~\eqref{deltaEqnPropTElt})}\\
&=\sqrt[4]{\xi_0''}\Big|(Y \times G_{[i+1,k]}) \cap ((\gamma^{-1}(\mu_1) \cap (\gamma^{\text{conv}})^{-1}(\mu^{\text{conv}}) \cap \theta^{-1}(u)) \times G_{i+1}) \cap V \cap U^{\text{ind}} \cap (W^{\text{ind}} \times G_{[i+2,k]})\Big|\\
&=\sqrt[4]{\xi_0''}\Big|(Y \times G_{[i+1,k]}) \cap U \cap V\Big|\\
&\leq\sqrt[4]{\xi_0''} |((W \cap V^{[i]}) \times G_{[i+1,k]}) \cap V \cap U|.\end{align*}
It remains to choose parameters $\xi_0, m$ and $\eta_2$. We set $\eta_2 = \frac{1}{2}(\varepsilon + \xi_0) p^{-k\codim T - (k+1)\codim \theta}$. To see that this choice will make the required bounds involving $\eta_2$ hold, recall that we assumed that $\varepsilon'' \geq \varepsilon + \xi_0$ and $\xi_0'' \geq \xi_0$, where $\varepsilon''$ and $\xi_0''$ were previously defined and satisfy $\varepsilon'' = O(\varepsilon^{\Omega(1)} + \xi_0)$ and $\xi_0'' = p^{-m/3} + O(\varepsilon_0^{-1} \sqrt{\xi_0}) = p^{-m/3} + O(\sqrt{\xi_0})$.\\
\indent Finally, we may choose $m \leq O(\log_p \delta^{-1} + \log_p \xi_0^{-1})$ and $\xi_0 \geq \Omega(\xi^{O(1)})$ so that all the other required bounds hold. This completes the proof.\end{proof}

\subsection{Obtaining a global multiaffine map}

Let $\phi$ be a multiaffine map defined on $1-o(1)$ of a $\mathcal{G}$-supported variety $V$. The next proposition tells us that we can use $\phi$ to define a multiaffine map $\psi$ on $1-o(1)$ of a $(\mathcal{G} \setminus \{I_0\})$-supported variety, where $I_0 \in \mathcal{G}$ is a maximal member. The value of $\psi$ is obtained by evaluation of $\phi$ on an arrangement of points. However, we allow more than a single shape of arrangement, which leads to a decomposition of the domain of $\psi$ according to which kind of arrangement was used.

\begin{proposition}\label{StabBogFnMultThm}There exists $\varepsilon_0 = \varepsilon_0(k) > 0$ with the following property.\\
\indent Let $\mathcal{G} \subset \mathcal{P}$ be a down-set and let $\mathcal{G}'$ be a down-set obtained by removing a maximal element $I_0$ from $\mathcal{G}$. Let $V \subset G_{[k]}$ be a $\mathcal{G}$-supported mixed-linear variety of codimension at most $r$ given by $V = \{x_{[k]} \in G_{[k]} \colon (\forall i \in [r])\,\, \alpha_i(x_{I_i}) = \tau_i\}$ for some multilinear forms $\alpha_i \colon G_{I_i} \to \mathbb{F}_p$, where $I_i \in \mathcal{G}$ for each $i$. Define $V^{\mathrm{lower}}$ to be the variety $\{x_{[k]} \in G_{[k]} \colon (\forall i \in [r] \colon I_i \not= I_0)\,\,\alpha_i(x_{I_i}) = \tau_i\}$. Let $X \subset V$ be a set of size at least $(1- \varepsilon_0) |V|$. Let $\phi \colon X \to H$ be a multiaffine map and let $\eta > 0$. Suppose that $\dim G_i \geq \con \Big(2^r + \log_p \xi^{-1}\Big)^{\con}$ for each $i \in [k]$. Then there exist
\begin{itemize}
\item a positive integer $s \leq p^{{(2k)}^{r + 1} - 1}$,
\item a mixed-linear $\mathcal{G}'$-supported variety $V'$ of codimension $(2 + \log_p \eta^{-1})^{2^{O(r)}}$,
\item a set $X' \subset V^{\mathrm{lower}} \cap V'$,
\item a multiaffine map $\psi \colon X' \to H$, and
\item a partition $X' = X'_1 \cup \dots \cup X'_s$,
\end{itemize}
such that 
\begin{itemize}
\item[\textbf{(i)}] $V^{\mathrm{lower}} \cap V'  \not= \emptyset$,
\item[\textbf{(ii)}] $|X'|\ \geq (1 - \eta) | V^{\mathrm{lower}} \cap V'|$,
\item[\textbf{(iii)}] for each $i \in [s]$, there exist $m = m^{(i)} \leq 3^k \cdot (2k + 1)^{r + 1}$, a collection of coefficients $\nu_{j, l} = \nu_{j, l}^{(i)} \in \mathbb{F}_p^{[0,m]}$, for $j \in [m], l \in [k]$, elements $a_{j, l} = a_{j, l}^{(i)} \in G_l$ for $j \in [m], l \in [k]$ and coefficients $\lambda_j = \lambda_j^{(i)} \in \mathbb{F}_p \setminus \{0\}$, for $j \in [m]$, such that:
\begin{itemize}
\item[\textbf{(iii.a)}] for each $x_{[k]} \in X'_i$
\begin{equation}\psi(x_{[k]}) = \sum_{j \in [m]} \lambda_j \phi\Big(\nu_{j, 1} \cdot (x_1, u^1_1, \dots, u^1_{m}) + a_{j, 1},\hspace{2pt}\dots,\hspace{2pt}\nu_{j, k} \cdot (x_k, u^k_{1}, \dots, u^k_{m}) + a_{j, k}\Big),\label{psitophiidentity}\end{equation}
where $\nu_{j, l} \cdot (x_l, u^l_{1}, \dots, u^l_{m}) = \nu_{j,l,0} x_l + \sum_{l' \in [m]} \nu_{j,l,l'} u^l_{l'}$, holds (and the arguments of $\phi$ belong to $X$) for at least
\[p^{-(2 + \log_p \eta^{-1})^{2^{O(r)}}}|G_{[k]}|^{m}\]
choices of $u^1_{1}, \dots,$ $u^1_{m} \in G_1, \dots,$ $u^k_{1}, \dots, u^k_{m} \in G_k$,
\item[\textbf{(iii.b)}] there is exactly one $j \in [m]$ such that $\nu_{j,l,0} \not= 0$ for all $l \in [k]$.
\end{itemize}
\end{itemize}
\end{proposition}

\begin{proof} For fixed $k$ and $p$, let $D_0, c_0 > 0$ be constants such that the term $O(\varepsilon^{\Omega(1)})$ replaces $D_0 \varepsilon^{c_0}$ that appears in~\eqref{extn1codimXprimeDens} in the statement of Theorem~\ref{extn1codim}. Without loss of generality we may assume that $D_0 \geq 2$ and $c_0 \leq 1/2$, as lower values of $D_0$ and higher values of $c_0$ only make the bound stronger. Let $d_0 = (2D_0)^{-c_0^{-1}}$ and $C_0 = c_0^{-1}$. Likewise, let $D_1, c_1 > 0$ be the implicit constants such that the term $O(\varepsilon^{\Omega(1)})$ replaces $D_1 \varepsilon^{c_1}$ in~\eqref{MltHommToMltAffEqn} in Proposition~\ref{MltHommToMltAff} and again assume that in fact $D_1 \geq 2$ and $c_1 \leq 1/2$. Let $d_1 = (2D_1)^{-1}$ and $C_1 = c_1^{-1}$.\\
\indent Since $\phi$ is multiaffine, it is also a multi-$2 \cdot 20^k$-homomorphism. Apply Theorem~\ref{bogolyubovSingleStep} for $i = 0$ and an error parameter $\xi_0 = \Big(\frac{d_1}{2}\Big)^{C_1} \cdot d_0^{C_1(C_0^{r-1} + C_0^{r-2} + \dots + 1)} \eta^{C_1 \cdot C_0^r}$ to the map $\phi$. (This is the reason why $\varepsilon_0$ has to be smaller than some positive constant -- the term $(O(\varepsilon^{\Omega(1)}) + \xi)$ in the bound in property \textbf{(ii)} of that theorem has to be less than 1 in order to have $Y \times G_{[k]} = G_{[k]}$.) We obtain 
\begin{itemize}
\item a $\mathcal{G}'$-supported variety $U \subset G_{[k]}$ of codimension $O((r + \log_p \xi_0^{-1})^{O(1)}) = O(C_0^{O(r)} \log^{O(1)}_p \eta^{-1})$, which is mixed-linear, without loss of generality,
\item a subset $Z \subset V \cap U$, and
\item a multi-homomorphism $\psi \colon Z \to H$,
\end{itemize}
such that
\begin{itemize}
\item[\textbf{(i)}] the variety $V \cap U$ is non-empty,
\item[\textbf{(ii)}] $|(V \cap U) \setminus Z| \leq \xi_0 |V \cap U|$,
\item[\textbf{(iii)}] for each $x_{[k]} \in Z$, there are $p^{-O\big((r + \log_p \xi_0^{-1})^{O(1)}\big)}|G_k|^{2 \cdot 3^{k - 1}}|G_{k-1}|^{2 \cdot 3^{k- 2}} \cdots |G_1|^2$ $(k,k-1, \dots, 1)$-tri-arrangements $q$ with points in $X$ of lengths $x_{[k]}$ such that $\phi(q) = \psi(x_{[k]})$.
\end{itemize}
Apply Proposition~\ref{MltHommToMltAff} to $Z \subset V \cap U$ and the map $\psi$. Abusing notation, we may assume that all the properties above hold, with $\psi$ now being multiaffine, and $|(V \cap U) \setminus Z| \leq \xi_0 |V \cap U|$ changed to $|(V \cap U) \setminus Z| \leq \xi_1 |V \cap U|$, for
\[\xi_1 = d_0^{C_0^{r-1} + C_0^{r-2} + \dots + 1} \eta^{C_0^r}.\]
Without loss of generality the forms that define $V$ are organized in such a way that $I_i = I_0$ if and only if $i \in [r_0]$ for some $r_0 \leq r$. By induction on $s \in [r_0 + 1]$ we show that there exist
\begin{itemize}
\item a $\mathcal{G}'$-supported variety $U^{\text{lower}} \subset G_{[k]}$ of codimension at most $(C_0^{O(r)} + \log^{O(1)}_p \eta^{-1})^{2^{O(s)}}$,
\item a subset $Z^{\text{dom}} \subset \{x_{[k]} \in G_{[k]} \colon (\forall i \in [s,r_0])\,\, \alpha_i(x_{I_0}) = \tau_i\} \cap U^{\text{lower}}$, and
\item a multiaffine map $\psi \colon Z^{\text{dom}} \to H$,
\end{itemize}
such that
\begin{itemize}
\item[\textbf{(i)}] the variety $\{x_{[k]} \in G_{[k]} \colon (\forall i \in [s,r_0]) \,\,\alpha_i(x_{I_0}) = \tau_i\} \cap U^{\text{lower}}$ is non-empty,
\item[\textbf{(ii)}] \begin{align*}|(\{x_{[k]} \in G_{[k]} &\colon (\forall i \in [s,r_0]) \,\,\alpha_i(x_{I_0}) = \tau_i\} \cap U^{\text{lower}}) \setminus Z^{\text{dom}}|\\
& \leq  d_0^{C_0^{r - s -1} + C_0^{r-s-2} + \dots + 1} \eta^{C_0^{r-s}} |\{x_{[k]} \in G_{[k]} \colon (\forall i \in [s,r_0]) \alpha_i(x_{I_0}) = \tau_i\} \cap U^{\text{lower}}|,\\
\end{align*}
\item[\textbf{(iii)}] we have a partition of $Z^{\text{dom}}$ into $p^{(2k)^{s} - 1}$ pieces such that properties \textbf{(iii.a)} and \textbf{(iii.b)} from the statement hold with $m^{(i)} \leq 3^k (2k)^s$ and a proportion $p^{-(C_0^{O(r)} \log^{O(1)}_p \eta^{-1})^{2^{O(s)}}}$ of the parameters making the relevant identity~\eqref{psitophiidentity} hold.
\end{itemize}
For the base case $s = 1$, we use $U^{\text{lower}} = U \cap \{x_{[k]} \in G_{[k]} \colon (\forall i \in [r_0 + 1, r]) \alpha_i(x_{I_i}) = \tau_i\}$, $Z^{\text{dom}} = Z$ with the trivial partition $Z = Z$ and with $\phi$ as above. We just need to identify the scalars to show that \textbf{(iii)} holds. We can parametrize $(k, k-1, \dots, 1)$-tri-arrangements of lengths $x_{[k]}$ using $u^d_{\varepsilon} \in G_d$ for $\varepsilon \in \{1,2\}^{\{d\}}\times \{1,2,3\}^{[d-1, 1]}$,\footnote{This is the set of sequences indexed from $d$ down to $1$, whose elements lie in $\{1,2,3\}$ except the element indexed by $d$ which can only be $1$ or $2$.} as follows. When $\varepsilon \in \{1,2,3\}^{[k,1]}$ is given, we define $y_{[k]} = y_{[k]}(\varepsilon)$, by setting 
\[y_d = \begin{cases}u^d_{\varepsilon|_{[d, 1]}},&\text{ if }\varepsilon_d = 1,2,\\
u^d_{(1, \varepsilon|_{[d-1, 1]})} + u^d_{(2, \varepsilon|_{[d-1, 1]})} - x_d,&\text{ if }\varepsilon_d = 3,\end{cases}\]
and define its corresponding weight $\lambda_\varepsilon = (-1)^{|\{d \in [k] \colon \varepsilon_d = 3\}|}$. Hence, for each $x_{[k]} \in Z$ there is a sufficient number of choices of $u^i_\varepsilon$ such that
\[\psi(x_{[k]}) = \sum_{\varepsilon \in \{1,2,3\}^{[k, 1]}} \lambda_\varepsilon \phi(y_{[k]}(\varepsilon)).\]
Note that the only time all $x_1, \dots, x_k$ are present in the coordinates of $y_{[k]}$ is when $\varepsilon = (3, \dots, 3)$. The properties in \textbf{(iii)} are easily seen to hold (note that we formally need to add auxiliary variables $u^d_\varepsilon$ for $\varepsilon \notin \{1,2\}^{\{d\}}\times \{1,2,3\}^{[d-1, 1]}$, which are not used in the expression above, to make the sequences of parameters $u^d_\bullet$ of same length for each $d$).\\

Assume now that the claim holds for some $s \in [r_0]$. Let $U^{\text{lower}}$, without loss of generality defined by a mixed-linear map of codimension $r_{\text{lower}} \leq (C_0^{O(r)} + \log^{O(1)}_p \eta^{-1})^{2^{O(s)}}\Big)$, $Z^{\text{dom}} = Z^{\text{dom}}_1 \cup \dots \cup Z^{\text{dom}}_N$, where $N \leq p^{(2k)^{s} - 1}$, and $\psi$ be the objects with the properties described for $s$. Let $C_k, D_k$ be the constants from Theorem~\ref{extn1codim} and let $\xi'_s = \frac{1}{2}d_0^{C_0^{r-s-2} + C_0^{r-s-3} + \dots + 1} \eta^{C_0^{r-s - 1}}$. If $\alpha_{s}$ satisfies 
\begin{equation}\label{largebiascase}\bias \Big(\alpha_{s} - \mu \cdot \alpha_{[s+1,r_0]} \Big) \geq p^{-C_k(r_0 + r_{\text{lower}} + \log_p {\xi'_s}^{-1})^{D_k}}\end{equation}
for some $\mu \in \mathbb{F}_p^{[s+1,r_0]}$, apply Theorem~\ref{strongInvARankThm}. Then there exists a positive integer $t \leq O\Big((r_0 + r_{\text{lower}} + \log_p {\xi'_s}^{-1})^{O(1)}\Big)$, subsets $\emptyset \not= J_i \subsetneq I_0$ and multilinear forms $\beta_i \colon G_{J_i} \to \mathbb{F}_p$, $\gamma_i \colon G_{I_0 \setminus J_i} \to \mathbb{F}_p$ for $j \in [t]$ such that
\[\alpha_{s}(x_{I_0}) - \mu \cdot \alpha_{[s+1,r_0]}(x_{I_0}) = \sum_{i \in [t]} \beta_i(x_{J_i}) \gamma_i(x_{I_0 \setminus J_i}).\]
Then the variety $\{x_{[k]} \in G_{[k]} \colon (\forall i \in [s,r_0])\,\, \alpha_i(x_{I_0}) = \tau_i\} \cap U^{\text{lower}}$ becomes a union of layers of the form
\[\{x_{[k]} \in G_{[k]} \colon (\forall i \in [s + 1, r_0])\,\, \alpha_i(x_{I_0}) = \tau_i\} \cap U^{\text{lower}} \cap \{x_{[k]} \in G_{[k]} \colon (\forall i \in [t])\,\,\beta_i(x_{J_i}) = \tau'_i, \gamma_i(x_{I_0 \setminus J_i}) = \tau''_i\}\]
for some $\tau', \tau'' \in \mathbb{F}_p^t$. We simply average over such layers to find one that has sufficiently dense intersection with $Z^{\text{dom}}$ to finish the proof in this case.\\
\indent Thus, we may assume the opposite, i.e.\ that inequality~\eqref{largebiascase} fails. Apply Theorem~\ref{extn1codim} to $Z^{\text{dom}} \subset \{x_{[k]} \in G_{[k]} \colon (\forall i \in [s,r_0]) \alpha_i(x_{I_0}) = \tau_i\} \cap U^{\text{lower}}$, the map $\psi$ and parameter $\xi'_s$. We obtain a further $\mathcal{G}'$-supported variety $W \subset G_{[k]}$ of codimension $(r + r_{\text{lower}})^{O(1)} + r C_0^{O(r)} \log^{O(1)}_p \eta^{-1}$, a subset $Z' \subset \{x_{[k]} \in G_{[k]} \colon (\forall i \in [s+1,r_0])\,\, \alpha_i(x_{I_0}) = \tau_i\} \cap U^{\text{lower}} \cap W$ of size at least
\begin{align*}(1 - D_0 &(d_0^{C_0^{r - s -1} + C_0^{r-s-2} + \dots + 1} \eta^{C_0^{r-s}})^{c_0} - \xi'_s) |\{x_{[k]} \in G_{[k]} \colon (\forall i \in [s+1,r_0])\,\, \alpha_i(x_{I_0}) = \tau_i\} \cap U^{\text{lower}} \cap W|\\
&= (1 - d_0^{C_0^{r-s-2} + C_0^{r-s-3} + \dots + 1} \eta^{C_0^{r-s -1}})|\{x_{[k]} \in G_{[k]} \colon (\forall i \in [s+1,r_0])\,\, \alpha_i(x_{I_0}) = \tau_i\} \cap U^{\text{lower}} \cap W| > 0,\end{align*}
a multiaffine map $\psi' \colon Z' \to H$, a point $a_{[k]} \in G_{[k]}$, and $\mu_0 \in \mathbb{F}_p \setminus \{\tau_s\}$, such that for each $x_{[k]} \in Z'$,
\begin{itemize}
\item if $\alpha_s(x_{I_0}) = \tau_s$, then $\psi'(x_{[k]}) = \psi(x_{[k]})$, and
\item if $\alpha_s(x_{I_0}) = \mu \not= \tau_s$, for $\Omega(p^{-O\big((r + r_{\text{lower}} + \log_p {\xi'_s}^{-1})^{O(1)}\big)}|G_{I_0}|)$ choices of $u_{I_0} \in G_{I_0}$, we have
\begin{align}\psi'(x_{[k]}) &= \psi\Big(x_{[k] \setminus \{c_t\}}, x_{c_t} - \frac{\mu- \tau_s}{\mu_0 - \tau_s}(a_{c_t} - u_{c_t})\Big)\nonumber\\
&\hspace{1cm} + \frac{\mu- \tau_s}{\mu_0 - \tau_s}\bigg( - \psi(x_{[k] \setminus \{c_t\}}, u_{c_t})\nonumber\\
&\hspace{4cm} + \sum_{i \in [t-1]}\psi(x_{[k] \setminus \{c_i, \dots, c_{t}\}}, u_{c_i} + x_{c_i} - a_{c_i}, a_{\{c_{i+1}, \dots, c_t\}})\nonumber\\
&\hspace{4cm} - \sum_{i \in [t-1]}\psi(x_{[k] \setminus \{c_i, \dots, c_{t}\}}, u_{c_i}, a_{\{c_{i+1}, \dots, c_t\}})\bigg),\label{znewptsidentity}\end{align}
where $I_0 = \{c_1, \dots, c_t\}$.\\
\end{itemize}

It remains to show the property \textbf{(iii)} for the set $Z'$. We first partition $Z'$ into sets of the form $Z' \cap \{x_{[k]} \in G_{[k]} \colon \alpha_s(x_{I_0}) = \mu\}$. Thus, for each set in the current partition, we have that one of the formulas above holds for all points $x_{[k]}$ in the set. First, when $\mu = \tau_s$, we partition $Z' \cap \{x_{[k]} \in G_{[k]} \colon \alpha_s(x_{I_0}) = \tau_s\}$ as $\bigcup_{i \in [N]} Z' \cap \{x_{[k]} \in G_{[k]} \colon \alpha_s(x_{I_0}) = \tau_s\} \cap Z^{\text{dom}}_i$. Secondly, when $\mu \not= \tau_s$, partition $Z' \cap \{x_{[k]} \in G_{[k]} \colon \alpha_s(x_{I_0}) = \mu\}$ further by looking, for each $x_{[k]}$, at most frequent choice of indices $(i_1, i_2, \dots, i_{2t}) \in [N]^{2t}$ such that the arguments of $\psi$ belong to $Z^{\text{dom}}_{i_1}, Z^{\text{dom}}_{i_2}, \dots, Z^{\text{dom}}_{i_{2t}}$ (in the order they appear in expression~\eqref{znewptsidentity}). Then the number of sets in the new partition is at most $p \cdot \Big(p^{(2k)^{s} - 1}\Big)^{2k}$. For each set in the new partition we get an identity of the form in the property \textbf{(iii)} by using~\eqref{znewptsidentity} and expressing each $\psi$ value as a linear combination of values of $\phi$ using the property \textbf{(iii)} for sets $Z^{\text{dom}}_{i_1}, Z^{\text{dom}}_{i_2}, \dots, Z^{\text{dom}}_{i_{2t}}$. In the final expression, where $\psi'(x_{[k]})$ is expressed in terms of $\phi$, the only time we get an argument of $\phi$ that depends on all $x_1, \dots, x_k$ comes from the only such argument of $\psi$ in~\eqref{znewptsidentity} (namely the first one).\end{proof}

We now arrive to the culmination of the work on extensions of multiaffine maps. Using Proposition~\ref{StabBogFnMultThm} several times, we are able to prove that a multiaffine map defined on $1-o(1)$ of a variety necessarily agrees on a large proportion of points with a global multiaffine map. Recall that $\exp^{(t)}$ stands for the $t$-fold iteration of the exponential function, i.e.\ the tower of exponentials of height $t$.

\begin{theorem}\label{veryDenseToFullMultiaffineMap}There exists $\varepsilon_0 = \varepsilon_0(k) > 0$ with the following property.\\
\indent Let $V \subset G_{[k]}$ be a multiaffine variety of codimension $r$ and let $X \subset V$ be a set of size at least $(1- \varepsilon_0) |V|$. Suppose that $\dim G_i \geq \exp^{(2^{k+1} - 1)}(\con\,r)$ for each $i \in [k]$. Let $\phi \colon X \to H$ be a multiaffine map. Then there exist a subset $X' \subset X$ of size $\Big(\exp^{(2^{k+1})}(O(r))\Big)^{-1}|X|$, a global multiaffine map $\Phi \colon G_{[k]} \to H$, a positive integer $m =\exp^{(2^{k+1} - 1)}(O(r))$, collections of coefficients $\nu_{j, l} \in \mathbb{F}_p^{[0,m]}$ for $j \in [m], l \in [k]$, elements $a_{j, l} \in G_l$ for $j \in [m], l \in [k]$, and coefficients $\lambda_j \in \mathbb{F}_p \setminus \{0\}$, $j \in [m]$, such that
\begin{itemize}
\item[\textbf{(a)}] for each $x_{[k]} \in X'$
\begin{equation}\Phi(x_{[k]}) = \sum_{j \in [m]} \lambda_j \phi\Big(\nu_{j, 1} \cdot (x_1, u^1_1, \dots, u^1_{m}) + a_{j, 1},\hspace{3pt}\dots,\hspace{3pt}\nu_{j, k} \cdot (x_k, u^k_1, \dots, u^k_{m}) + a_{j, k}\Big),\label{phitoglobalidentity}\end{equation}
where $\nu_{j, l} \cdot (x_l, u_{l, 1}, \dots, u_{l,m}) = \nu_{j,l,0} x_l + \sum_{s \in [m]} \nu_{j,l,s} u_{l,s}$, holds for at least
\[2^{-\exp^{(2^{k+1})} (O(r))}|G_{[k]}|^{m}\]
choices of $u^1_{1}, \dots,$ $u^1_{m} \in G_1, \dots,$ $u^k_{1}, \dots, u^k_{m} \in G_k$,
\item[\textbf{(b)}] there is exactly one $j \in [m]$ such that $\nu_{j,l,0} \not= 0$ for all $l \in [k]$.
\end{itemize}
\end{theorem}

\begin{proof}Let $\varepsilon_0 > 0$ be the value in Proposition~\ref{StabBogFnMultThm}. Decreasing $\varepsilon_0$ further if necessary, we may assume that $\varepsilon_0 \leq c_k$ and $1- C_k\varepsilon_0^{d_k} \geq 1/2$, where $c_k, C_k, d_k$ are the constants in Proposition~\ref{lastExtnStep}. Note that $\varepsilon_0$ is still a positive quantity depending on $k$ and $p$ only. List all subsets of $[k]$ as $I_1, \dots, I_{2^k}$ so that larger sets come first -- that is, if $I_i \supset I_j$ then $i \leq j$. Let $\mathcal{G}_i = \{I_i, \dots, I_{2^k}\}$, which is a down-set for each $i \in [2^k]$. By induction on $i \in [2^k]$, we shall show that there exist 
\begin{itemize}
\item[\textbf{(i)}] a non-empty $\mathcal{G}_i$-supported variety $V'$ of codimension at most $r_i = \exp^{(2 i)} (O(r))$,
\item[\textbf{(ii)}] a subset $X' \subset V'$ of size at least $(1 - \varepsilon_0)|V'|$,
\item[\textbf{(iii)}] a multiaffine map $\phi' \colon X' \to H$,
\item[\textbf{(iv)}] a partition $X' = X'_1 \cup \dots \cup X'_s$, where $s = \exp^{(2i)}(O(r))$, such that for each $i_0 \in [s]$, there exist $m = m^{(i_0)} = \exp^{(2i + 1)}(O(r))$, collections of coefficients $\nu_{j, l} = \nu_{j, l}^{(i_0)} \in \mathbb{F}_p^{[0,m]}$ for $j \in [m], l \in [k]$, elements $a_{j, l} = a_{j, l}^{(i_0)} \in G_l$ for $j \in [m], l \in [k]$, coefficients $\lambda_j = \lambda_j^{(i_0)} \in \mathbb{F}_p \setminus \{0\}$, for $j \in [m]$, such that:
\begin{itemize}
\item[\textbf{(iv.a)}] for each $x_{[k]} \in X'_{i_0}$
\[\phi'(x_{[k]}) = \sum_{j \in [m]} \lambda_j \phi(\nu_{j, 1} \cdot (x_1, u^1_{1}, \dots, u^1_{m}) + a_{j, 1},\hspace{2pt}\dots,\hspace{2pt}\nu_{j, k} \cdot (x_k, u^k_{1}, \dots, u^k_{m}) + a_{j, k}),\]
for $2^{-\exp^{(2i)}(O(r))}|G_{[k]}|^{m}$ choices of $u^1_{1}, \dots,$ $u^1_{m} \in G_1, \dots,$ $u^k_{1}, \dots, u^k_{m} \in G_k$,
\item[\textbf{(iv.b)}] there is exactly one $j \in [m]$ such that $\nu_{j,l,0} \not= 0$ for all $l \in [k]$.
\end{itemize}
\end{itemize}
For $i = 1$, the claim is trivial. Assume now that the claim has been proved for some $i \in [2^k - 1]$, and let $V', X' = X'_1  \cup \dots \cup X_s'$ and $\phi'$ be the relevant variety, set and map. By averaging over the layers of the multilinear parts of the multiaffine map that defines $V'$ we may without loss of generality assume that $V'$ is mixed-linear. Recall that $I_i$ is a maximal set in $\mathcal{G}_i$. Write $V' = V^{\text{max}} \cap V^{\text{lower}}$, where $V^{\text{max}}$ and $V^{\text{lower}}$ are mixed-linear, $V^{\text{max}}$ is defined by multilinear forms that depend exactly on $G_{I_i}$ and $V^{\text{lower}}$ is $\mathcal{G}_{i+1}$-supported. Apply Proposition~\ref{StabBogFnMultThm} to $X' \subset V'$, the map $\phi'$ and the parameter $\eta = \varepsilon_0$ to get a further $\mathcal{G}_{i+1}$-supported variety $U^{\text{lower}}$ of codimension $2^{2^{O(r_i)}}$ such that $V^{\text{lower}} \cap U^{\text{lower}} \not= \emptyset$, a subset $X'' \subset V^{\text{lower}} \cap U^{\text{lower}}$ of size at least $(1 - \varepsilon_0) |V^{\text{lower}} \cap U^{\text{lower}}|$, a multiaffine map $\phi'' \colon X'' \to H$, and a partition of $X''$ into $2^{2^{O(r_i)}}$ pieces $X''_e$, such that for each piece there exist $m = m^{(e)} = 2^{O(r_i)}$, collections of coefficients $\nu_{j, l} = \nu_{j, l}^{(e)} \in \mathbb{F}_p^{[0,m]}$ for $j \in [m], l \in [k]$, elements $a_{j, l} = a_{j, l}^{(e)} \in G_l$ for $j \in [m], l \in [k]$, and coefficients $\lambda_j = \lambda_j^{(e)} \in \mathbb{F}_p \setminus \{0\}$, $j \in [m]$, such that:
\begin{itemize}
\item for each $x_{[k]} \in X''_{e}$
\begin{equation}\phi''(x_{[k]}) = \sum_{j \in [m]} \lambda_j \phi'(\nu_{j, 1} \cdot (x_1, u^1_{1}, \dots, u^1_{m}) + a_{j, 1},\hspace{2pt}\dots,\hspace{2pt}\nu_{j, k} \cdot (x_k, u^k_{1}, \dots, u^k_{m}) + a_{j, k}),\label{phidprimetophiprime}\end{equation}
for $2^{-2^{2^{O(r_{i+1})}}}|G_{[k]}|^{m}$ choices of $u^1_{1}, \dots,$ $u^1_{m} \in G_1, \dots,$ $u^k_{1}, \dots, u^k_{m} \in G_k$, and
\item there is exactly one $j \in [m]$ such that $\nu_{j,l,0} \not= 0$ for all $l \in [k]$.
\end{itemize}
We partition $X''_e$ further into subsets of the following form. For each $(i_1, \dots, i_m) \in [s]^m$ we take the set of $x_{[k]} \in X''_e$ such that the argument in the $j$\textsuperscript{th} $\phi'$ term in~\eqref{phidprimetophiprime} belongs to $X'_{i_j}$ for $j \in [m]$ for at least a proportion $s^{-m}$ of allowed values of $u^1_1, \dots, u^k_m$. (If $x_{[k]}$ can be assigned to several such sets, we put it in an arbitrary one.) On each such subset, we obtain an identity as in property \textbf{(iv)} that expresses $\phi''(x_{[k]})$ in terms of $\phi$ by using \textbf{(iv.a)} for each $X'_{i_j}$, $j \in [m]$, provided by the inductive hypothesis. Again, property \textbf{(iv.b)} holds for the newly obtained expression, as there is a single term in~\eqref{phidprimetophiprime} that involves all of $x_1, \dots, x_k$ and property \textbf{(iv.b)} for the relevant $X'_{i_j}$ gives a unique $\phi$ term that has all of $x_1, \dots, x_k$ present in its argument. The claim for $i + 1$ now follows.\\

Once the claim has been proved, we use it for $i = 2^k$. We obtain a multiaffine map $\phi'$ from a subset $X' \subset G_{[k]}$ of size at least $(1-\varepsilon_0) |G_{[k]}|$, where $X'$ also has a partition $X_1' \cup \dots \cup X'_s$ that satisfies property \textbf{(iv)}. Apply Proposition~\ref{lastExtnStep} to $\phi'$ to obtain a global multiaffine map $\Phi \colon G_{[k]} \to H$ such that $\Phi(x_{[k]}) = \phi'(x_{[k]})$ holds for at least $(1- C_k\varepsilon_0^{d_k})|G_{[k]}| \geq 1/2|G_{[k]}|$ points $x_{[k]} \in X'$. Let $\tilde{X}' \subset X'$ be the set of such $x_{[k]}$. Take the largest set $X'_i \cap \tilde{X}'$ to finish the proof.\end{proof}

\section{Putting everything together}\label{sec11labelPET}

Before we finally proceed to prove Theorem~\ref{multiaffineInvThm}, we show that if a map $\phi \colon G_{[k]} \to H$ is related to a global multiaffine map $\Phi$ as in~\eqref{phitoglobalidentity}, then $\phi$ itself coincides with a global multiaffine map on a dense set.

\begin{proposition}\label{lastStepGlobalToMulti}Let $D \subset G_{[k]}$ be a set of density $\delta > 0$, let $\phi \colon D \to H$ be a multi-homomorphism and let $\Phi \colon G_{[k]} \to H$ be a global multiaffine map. Let $m, n \in \mathbb{N}$, $\lambda_i \in \mathbb{F}_p \setminus \{0\}$ for $i \in [n]$, $\nu_{i, j, l} \in \mathbb{F}_p$ for $i \in [n], j \in [k], l \in [0,m]$ and $a_{i,j} \in G_j$ for $i \in [m], j \in [k]$, be such that
\begin{equation}\label{phiPhiArrsEqn}\Phi(x_{[k]}) = \sum_{i \in [n]} \lambda_i \phi(\nu_{i, 1} \cdot (x_1, u^1_{1}, \dots, u^1_{m}) + a_{i,1},\hspace{2pt}\dots,\hspace{2pt}\nu_{i, k} \cdot (x_k, u^k_{1}, \dots, u^k_{m}) + a_{i,k}),\end{equation}
holds for at least $\delta |G_1|^{m+1} \cdots |G_k|^{m+1}$ choices of $x_i, u^i_{j} \in G_i$, $i \in [k], j \in [m]$. Assume also that $\prod_{j \in [k]} \nu_{i,j, 0} \not=0$ for exactly one $i \in [n]$. Then $\phi$ coincides with some global multiaffine map on a set of size $\delta' |G_{[k]}|$, where $\delta' = \Big(\exp^{(k \cdot D^{\mathrm{mh}}_{k-1})}(O(\delta^{-1}))\Big)^{-1}$.\end{proposition}

\begin{proof}Without loss of generality $\prod_{j \in [k]} \nu_{n,j, 0} \not=0$ and $\prod_{j \in [k]} \nu_{i,j, 0} =0$ for $i \in [n-1]$. We may therefore partition $[n-1] = I_1 \cup \dots \cup I_k$ so that when $i\in I_j$, then $\nu_{i,j, 0} = 0$.\\
By averaging, we may find $u^{[k]}_{[m]}$ such that~\eqref{phiPhiArrsEqn} holds for at least $\delta |G_{[k]}|$ elements $x_{[k]} \in G_{[k]}$. Let $X_0$ be the set of such $x_{[k]}$. We may rewrite that expression as
\[\Phi(x_{[k]}) = \sum_{i \in [n]} \lambda_i \phi(\nu_{i, 1, 0} x_1 + a'_{i, 1}, \dots, \nu_{i, k, 0} x_k + a'_{i,k})\]
for some $a'_{i,j} \in G_j$, $i \in [n], j \in [k]$. Reorganizing this further, we obtain
\begin{equation}\Phi(x_{[k]}) = \lambda_n \phi(\nu_{n, 1, 0} x_1 + a'_{n, 1}, \dots, \nu_{n, k, 0} x_k + a'_{n,k}) +  \sum_{j \in [k]} \Big(\sum_{i \in I_j} \lambda_i \phi(\nu_{i, 1, 0} x_1 + a'_{i, 1}, \dots, \nu_{i, k, 0} x_k + a'_{i,k})\Big).\label{phiPhi2Eqn}\end{equation}
We now find sets $X_0 \supset X_1 \supset \dots \supset X_k$ of sizes $|X_i| = \delta_i |G_{[k]}|$, $\delta_i \geq \Big(\exp^{(i \cdot D^{\mathrm{mh}}_{k-1})}(O(\delta^{-1}))\Big)^{-1}$, such that for each $d \in [k]$, there is a global multiaffine map $\Psi_d \colon G_{[k] \setminus \{d\}} \to H$ such that
\[(\forall x_{[k]} \in X_d)\hspace{6pt} \sum_{i \in I_d} \lambda_i \phi(\nu_{i, 1, 0} x_1 + a'_{i, 1}, \dots, \nu_{i, k, 0} x_k + a'_{i,k}) = \Psi_d(x_{[k] \setminus \{d\}}).\]
We have already defined $X_0$. Assume now that we have defined $X_0, X_1, \dots, X_{d-1}$. Let $Y \subset G_{[k] \setminus \{d\}}$ be the set of $x_{[k] \setminus \{d\}}$ such that $|(X_{d-1})_{x_{[k] \setminus \{d\}}}| \geq \frac{\delta_{d-1}}{2} |G_d|$. By averaging, $|Y| \geq \frac{\delta_{d-1}}{2}|G_{[k] \setminus \{d\}}|$. Hence, we may define a map $\psi \colon Y \to H$ by
\[\psi(x_{[k] \setminus \{d\}}) = \sum_{i \in I_d} \lambda_i \phi(\nu_{i, 1, 0} x_1 + a'_{i, 1}, \dots, \nu_{i, k, 0} x_k + a'_{i,k}),\]
noting that $\nu_{i, d, 0} = 0$ for all $i \in I_d$ so $x_d$ does not appear. Since $\phi$ is a multi-homomorphism, so is $\psi$. We may now apply Theorem~\ref{multiaffineInvThm} for the space $G_{[k] \setminus \{d\}}$ (recall that since the beginning of the work in this paper we assume Theorem~\ref{multiaffineInvThm} for smaller numbers of variables). This gives us a subset $Y' \subset Y$ of size 
\[\Big(\exp^{(D^{\mathrm{mh}}_{k-1})}(O(\delta_{d-1}^{-1}))\Big)^{-1} |G_{[k] \setminus \{d\}}|\]
and a global multiaffine map $\Psi_d \colon G_{[k] \setminus \{d\}} \to H$ such that $\psi = \Psi_d$ on $Y'$. Let $X_d = (Y' \times G_d) \cap X_{d-1} = \bigcup_{x_{[k] \setminus \{d\}} \in Y'} \{x_{[k] \setminus \{d\}}\} \times (X_{d-1})_{x_{[k] \setminus \{d\}}}$, which has size at least $\delta_d |G_{[k]}|$, where
\[\delta_d \geq \frac{\delta_{d-1}}{2} \cdot \Big(\exp^{(D^{\mathrm{mh}}_{k-1})}(O(\delta_{d-1}^{-1}))\Big)^{-1}.\]

Thus, when $x_{[k]} \in X_k$, from~\eqref{phiPhi2Eqn}, we obtain
\[\phi(\nu_{n, 1, 0} x_1 + a'_{n, 1}, \dots, \nu_{n, k, 0} x_k + a'_{n,k}) = \lambda^{-1}_n \Phi(x_{[k]})  - \lambda^{-1}_n\sum_{j \in [k]} \Psi_j(x_{[k] \setminus \{j\}}).\]
The result follows after a change of variables.\end{proof}

\begin{proof}[Proof of Theorem~\ref{multiaffineInvThm}] Suppose that $\phi \colon A \to H$ is a multi-homomorphism, where $A \subset G_{[k]}$ has density $\delta$. First, we find subsets $A_k \subset A_{k-1} \subset \dots \subset A_1 \subset A_0 = A$ of densities $\delta_i = |A_i| / |G_{[k]}|$, such that on each $A_d$, $\phi$ is actually affine in the directions $G_{[d]}$. We obtain $A_{d+1}$ from $A_d$ by applying Theorem~\ref{FreBSG} to $(\delta_d/2)$-dense columns of $A_d$ in direction $G_{d+1}$. Thus, we may take $\delta_k = \exp(-\log^{O(1)}\delta^{-1})$.\\
\indent Let $\varepsilon_0 > 0$ be as in Theorem~\ref{veryDenseToFullMultiaffineMap}. We shall apply Theorem~\ref{densificationThm}, but first we need to check that the technical condition $|G_i| \geq \con\, \delta_k^{- \con}$ for each $i \in [k]$ of that theorem is satisfied. Suppose on the contrary that $|G_i| \leq O(\delta_k^{- O(1)})$ holds for some $i$, which we may assume to be $i = k$. Then we may simply average over $G_k$ to find an element $y_k \in G_k$ such that $|A_{y_k}| \geq \delta |G_{[k-1]}|$. We may apply the inductive hypothesis to $\phi_{y_k}$ to find a global multiaffine map $\Phi \colon G_{[k-1]} \to H$ such that $\phi(x_{[k-1]}, y_k) = \Phi(x_{[k-1]})$ holds for at least 
\[\Big(\exp^{(D^{\mathrm{mh}}_{k-1})}(O(\delta^{-1}))\Big)^{-1} |G_{[k-1]}|\]
of $x_{[k-1]} \in A_{y_k}$. But, since $G_k$ is small, we simply conclude that in fact $\phi(x_{[k]}) = \Phi(x_{[k-1]})$ holds for at least
\[\Big(\exp^{(D^{\mathrm{mh}}_{k-1})}(O(\delta^{-1}))\Big)^{-1} \, \exp(-O(\log^{O(1)}\delta^{-1}))  |G_{[k]}|\]
of $x_{[k]} \in G_{[k]}$, completing the proof. Hence, we may assume that all $G_i$ are sufficiently large, as required by Theorem~\ref{densificationThm}.\\
\indent Let $\varepsilon > 0$ be a positive quantity that we shall specify later. By Theorem~\ref{densificationThm} for parameter $\varepsilon$ there exists a subset $A' \subset A_k$ such that
\[|\{(q_1, \dots, q_k) \in \mathcal{Q} \colon \phi(q_1) = \dots = \phi(q_k)\}| \,\,\geq (1-\varepsilon) |\mathcal{Q}|\,\, \geq \delta'\,|G_{[k]}|^{8^k},\] 
where $\delta' \geq \Omega((\varepsilon \delta)^{O(1)})$ and where $\mathcal{Q}$ is defined in the statement of that theorem. Write $\mathcal{Q}_{x_{[k]}}$ for the set of $(q_1, \dots, q_k) \in \mathcal{Q}$ such that each $q_i$ has lengths $x_{[k]}$. Let $X^{(1)} \subset G_{[k]}$ be the set of all $x_{[k]}$ such that $|\mathcal{Q}_{x_{[k]}}| \geq \varepsilon\delta'|G_{[k]}|^{8^k - 1}$. By definition of $X^{(1)}$ we have
\[\sum_{x_{[k]} \in X^{(1)}} |\mathcal{Q}_{x_{[k]}}|\,\, = |\mathcal{Q}| - \sum_{x_{[k]} \notin X^{(1)}} |\mathcal{Q}_{x_{[k]}}|\,\, \geq |\mathcal{Q}| - \varepsilon \delta'|G_{[k]}|^{8^k} \geq (1 - \varepsilon) |\mathcal{Q}|.\] 
Next, define $X^{(2)} \subset G_{[k]}$ be the set of all $x_{[k]}$ such that $|\{(q_1, \dots, q_k) \in \mathcal{Q}_{x_{[k]}} \colon \phi(q_1) = \dots = \phi(q_k)\}| \,\,\geq (1-\sqrt{\varepsilon}) |\mathcal{Q}_{x_{[k]}}|$. We have
\begin{align*}\sqrt{\varepsilon} \sum_{x_{[k]} \notin X^{(2)}} |\mathcal{Q}_{x_{[k]}}|\,\,\leq &\sum_{x_{[k]} \notin X^{(2)}} |\{(q_1, \dots, q_k) \in \mathcal{Q}_{x_{[k]}} \colon \phi(q_1), \dots, \phi(q_k)\text{ not all equal}\}| \\
\leq &\,\,|\{(q_1, \dots, q_k) \in \mathcal{Q} \colon \phi(q_1), \dots, \phi(q_k)\text{ not all equal}\}|\\
 \leq&\,\, \varepsilon |\mathcal{Q}|.\end{align*}
Putting $X = X^{(1)} \cap X^{(2)}$ we conclude that
\[\sum_{x_{[k]} \in X} |\mathcal{Q}_{x_{[k]}}| \geq (1 - 2\sqrt{\varepsilon}) |\mathcal{Q}|.\]

\begin{claim*} Provided $\varepsilon \leq \cons$, for each $x_{[k]} \in X^{(2)}$ we may find a value $\psi(x_{[k]}) \in H$ such that 
\[|\{(q_1, \dots, q_k) \in \mathcal{Q}_{x_{[k]}} \colon \phi(q_1) = \dots = \phi(q_k) = \psi(x_{[k]})\}| \geq (1-2\sqrt{\varepsilon}) |\mathcal{Q}_{x_{[k]}}|.\]\end{claim*}

\begin{proof}Let $n_i$ be the number of $([i,1], [k,1], [k,1], [k,i + 1])$-arrangements with points in $A'$ of lengths $x_{[k]}$. For $h \in H$, let $a_i(h)$  be the number of $([i,1], [k,1], [k,1], [k,i + 1])$-arrangements $q$ with points in $A'$ of lengths $x_{[k]}$ such that $\phi(q) = h$. Then we have $\sum_{h \in H} a_i(h) = n_i$ and 
\[\sum_{h \in H} a_1(h) \cdots a_k(h) \geq (1 - \sqrt{\varepsilon}) n_1 \cdots n_k.\]
In particular, for any $i \in [k]$, we have
\begin{align*}\Big(\max_{h \in H} a_i(h) \Big) \prod_{j \in [k] \setminus \{i\}} n_j \geq &\Big(\max_{h \in H} a_i(h) \Big) \Big(\sum_{h \in H} a_1(h) \cdots a_{i-1}(h)a_{i+1}(h) \cdots a_k(h)\Big)\\
 \geq &\sum_{h \in H} a_1(h) \cdots a_k(h)\\
 \geq &(1 - \sqrt{\varepsilon}) n_1 \cdots n_k\end{align*}
so $\max_{h \in H} a_i(h) \geq (1 - \sqrt{\varepsilon}) n_i$. Provided $\varepsilon \leq \cons$, there is a unique $h_i \in H$ where the maximum of $a_i(h)$ is attained. We claim that $h_1 = \dots = h_k$. It suffices to prove $h_1 = h_2$. Suppose that on the contrary $h_1 \not= h_2$. Then
\begin{align*} (1 - \sqrt{\varepsilon}) n_1 \cdots n_k \leq &\sum_{h \in H} a_1(h) \cdots a_k(h) \leq n_3 \cdots n_k \Big(\sum_{h \in H} a_1(h) a_2(h)\Big)\\
 = &n_3 \cdots n_k \Big(a_1(h_1)a_2(h_1) + \sum_{h \not= h_1} a_1(h) a_2(h)\Big) \leq n_3 \cdots n_k \Big(a_1(h_1) \sqrt{\varepsilon} n_2 + \sqrt{\varepsilon} n_1 \sum_{h \in H} a_2(h)\Big)\\
\leq &2 \sqrt{\varepsilon}  n_1 \cdots n_k \end{align*}
which is a contradiction, provided $\varepsilon \leq \cons$. Hence, $h_1 = \dots = h_k$ and we set $\psi(x_{[k]})$ to be this common value. To simplify the notation, write $\tilde{a}_i = a_i(\psi(x_{[k]}))$.\\
\indent Finally, observe that
\begin{align*}(1 - \sqrt{\varepsilon}) n_1 \cdots n_k \leq &\sum_{h \in H} a_1(h) \cdots a_k(h) \leq \tilde{a}_1 \cdots \tilde{a}_k + \sum_{v_1, \dots, v_k \in H \setminus \{\psi(x_{[k]})\}} a_1(v_1) \cdots a_k(v_k)\\
= & \tilde{a}_1 \cdots \tilde{a}_k + (n_1 - \tilde{a}_1) \cdots (n_k - \tilde{a}_k) \leq \tilde{a}_1 \cdots \tilde{a}_k + \varepsilon^{k/2} n_1\cdots n_k,\end{align*}
from which we get $\tilde{a}_1 \cdots \tilde{a}_k \geq (1 - 2\sqrt{\varepsilon}) n_1 \cdots n_k$, as desired.\end{proof}

For each $x_{[k]} \in X$, let $\psi(x_{[k]})$ be given by the claim above. By Theorem~\ref{nearlyMultThm} there exist a non-empty variety $V$ of codimension \chg{$r = \exp^{\big(D^{\mathrm{mult}}\big)}\Big(O((\varepsilon\delta_k)^{-O(1)})\Big)$, where $D^{\mathrm{mult}} = D^{\mathrm{mh}}_{k-1} + D^{\mathrm{mh}}_{k - 2} + \dots + D^{\mathrm{mh}}_{1} + 3k$,} and a subset $B \subset V \cap X$ of size at least $(1-O(\varepsilon^{1/16})) |V|$ such that the restriction $\psi|_B$ is a multiaffine map (in the sense of Theorem~\ref{nearlyMultThm}). Moreover, since $B \subset X$, we also know that for each $x_{[k]} \in B$ there are at least
\[\Omega((\varepsilon\delta_k)^{O(1)})|G_1|^{2^{3k-1} + 2^{2k-1} + 2^{k-1}} |G_2|^{2^{3k-2} + 2^{2k-2} + 2^{k-2}} \cdots |G_k|^{2^{2k} + 2^k + 1}\]
($[k,1], [k,1], [k,1]$)-arrangements $q$ with lengths $x_{[k]}$ whose points lie in $A'$, such that $\phi(q) = \psi(x_{[k]})$.\\
\indent We may pick $\varepsilon = \Omega(\varepsilon_0^{O(1)})$ so that $|B| \geq (1-\varepsilon_0) |V|$ and $\varepsilon \leq \cons$. Before we apply Theorem~\ref{veryDenseToFullMultiaffineMap}, we comment about its technical requirement that $\dim G_i \geq \exp^{(2^{k+1} - 1)}(\con\,r)$ for each $i \in [k]$. As above in the case of Theorem~\ref{densificationThm}, if it happens that $\dim G_i \leq \exp^{(2^{k+1} - 1)}(O(r))$ for some $i$, say $i = k$, we may apply the induction hypothesis to find a global multiaffine map $\Phi \colon G_{[k-1]} \to H$ such that $\phi(x_{[k]}) = \Phi(x_{[k-1]})$ holds for 
\[\Big(\exp^{(D^{\mathrm{mh}}_{k-1})}(O(\delta^{-1}))\Big)^{-1} \, \Big(\exp^{\big(D^{\mathrm{mult}} + 2^{k + 1} + 1\big)}(O(\delta^{-1}))\Big)^{-1}  |G_{[k]}|\]
points $x_{[k]} \in G_{[k]}$ and we are done immediately. Thus, assume that the condition on $\dim G_i$ is satisfied and apply Theorem~\ref{veryDenseToFullMultiaffineMap}. There exist a subset $B' \subset B$ of size at least $\Big(\exp^{(2^{k+1})}(O(r))\Big)^{-1}|B|$, a global multiaffine map $\Phi \colon G_{[k]} \to H$, $m \leq \exp^{(2^{k+1} - 1)}(O(r))$, collections of coefficients $\nu_{j, l}\in \mathbb{F}_p^{[0,m]}$ for $j \in [m], l \in [k]$, elements $a_{j, l} \in G_l$ for $j \in [m], l \in [k]$, and coefficients $\lambda_j \in \mathbb{F}_p \setminus \{0\}$ for $j \in [m]$ such that
\begin{itemize}
\item[\textbf{(a)}] for each $x_{[k]} \in B'$
\[\Phi(x_{[k]}) = \sum_{j \in [m]} \lambda_j \psi(\nu_{j, 1} \cdot (x_1, u^1_{1}, \dots, u^1_{m}) + a_{j, 1},\hspace{2pt}\dots,\hspace{2pt}\nu_{j, k} \cdot (x_k, u^k_{1}, \dots, u^k_{m}) + a_{j, k}),\]
for
\[2^{-\exp^{(2^{k+1})} (O(r))}|G_{[k]}|^{m}\]
choices of $u^1_{1}, \dots,$ $u^1_{m} \in G_1, \dots,$ $u^k_{1}, \dots, u^k_{m} \in G_k$, and
\item[\textbf{(b)}] there is exactly one $j \in [m]$ such that $\nu_{j,l,0} \not= 0$ for all $l \in [k]$.
\end{itemize}
Replace $\psi$ in the expression above using ($[k,1], [k,1], [k,1]$)-arrangements $q$ such that $\phi(q) = \psi(l(q))$, where $l(q) \in G_{[k]}$ are the lengths of $q$. Apply Proposition~\ref{lastStepGlobalToMulti}, where the proportion of parameters that obey the relevant equation is at least
\[\bigg(\exp^{\big(2^{k+1} + 2 + D^{\mathrm{mult}}\big)} (\log^{O(1)} \delta^{-1})\bigg)^{-1}.\]
This completes the proof, as we obtain a global multiaffine map $\Psi$ which coincides with $\phi$ for a proportion
\begin{align*}&\bigg(\exp^{(k \cdot D^{\mathrm{mh}}_{k-1})}\Big(\exp^{\big(2^{k+1} + 2 + D^{\mathrm{mult}}\big)} (\log^{O(1)} \delta^{-1})\Big)\bigg)^{-1}\\
&\hspace{2cm}\geq \Big(\exp^{\big(2k \cdot D^{\mathrm{mh}}_{k-1} + 2^{k+1} + 3k + 2\big)} (O(\delta^{-1}))\Big)^{-1} \end{align*}
of the points in $G_{[k]}$.\end{proof}

\section{Applications}\label{sec12labelapps}

\subsection{An inverse theorem for multiaffine maps over general finite fields}

In this subsection, we fix a finite field $\mathbb{F}$ of characteristic $p$, and we view $\mathbb{F}_p$ as a subfield of $\mathbb{F}$. The next result tells us that a map defined on a dense subset of $G_{[k]}$ that is $\mathbb{F}$-affine in each direction separately agrees on a further dense subset with a global $\mathbb{F}$-multiaffine map.

\begin{theorem}\label{multiaffineInvThmGen}For each $k \in \mathbb{N}$ there exists a constant $D_k$ such that the following holds. Let $G_1, \dots, G_k$ be finite-dimensional vector spaces over $\mathbb{F}$. Let $A \subset G_{[k]}$ be a set of density $\delta$, and let $\phi \colon A \to H$ be a $\mathbb{F}$-multiaffine map, which means that for each direction $d \in [k]$, and every $x_{[k] \setminus \{d\}} \in G_{[k] \setminus \{d\}}$, there is an $\mathbb{F}$-affine map $\alpha \colon G_d \to H$ such that for each $y_d \in A_{x_{[k] \setminus \{d\}}}$, $\phi(x_{[k] \setminus \{d\}}, y_d) = \alpha(y_d)$. Then, there is a global $\mathbb{F}$-multiaffine map $\Phi \colon G_{[k]} \to H$ such that $\phi(x_{[k]}) = \Phi(x_{[k]})$ for at least $\Big(\exp^{(D_k)}(O(\delta^{-1}))\Big)^{-1} |G_{[k]}|$ of $x_{[k]} \in A$.\end{theorem}

The theorem will follow from Theorem~\ref{multiaffineInvThm} and Proposition \ref{app1MainProp}. For the proposition, we need the following lemma from~\cite{extnPaper}.

\begin{lemma}[Lemma 5.1 in~\cite{extnPaper}]\label{lowRankApprox}Let $\alpha \colon G_{[k]} \to \mathbb{F}$ be a $\mathbb{F}_p$-multilinear map, which is $\mathbb{F}$-linear in the first $d-1$ coordinates. Suppose that for every $i \in [r]$ and $\lambda \in \mathbb{F}$,
\[\prank_{\mathbb{F}_p} \Big(x_{[k]} \mapsto \alpha_i(x_{[k] \setminus \{d\}}, \lambda x_d) - (\lambda \cdot \alpha)_i(x_{[k]})\Big) \leq s.\]
Then there is an $\mathbb{F}_p$-multilinear form $\sigma \colon G_{[k]} \to \mathbb{F}$, \chg{$\mathbb{F}$-}linear in the first $d$ coordinates, such that $\prank_{\mathbb{F}_p} \Big(\alpha_i - \sigma_i\Big) \leq 2sr^2$ for each $i \in [r]$.
\end{lemma}

\begin{proposition}\label{app1MainProp}Let $\mathcal{G} \subset \mathcal{P}([k])$ be a down-set with a maximal element $I_0$. Write $\mathcal{G}' = \mathcal{G} \setminus \{I_0\}$. Let $\Phi \colon G_{[k]} \to H$ be an $\mathbb{F}_p$-multiaffine $\mathcal{G}$-supported map and let $A \subset G_{[k]}$ be a set of density $\delta > 0$ such that the restriction $\Phi|_A$ is $\mathbb{F}$-multiaffine. Then there exist an $\mathbb{F}$-multiaffine map $\Theta \colon G_{[k]} \to H$ and an $\mathbb{F}_p$-multiaffine map $\alpha \colon G_{[k]} \to \mathbb{F}_p^r$, where $r = O(\log_p\delta^{-1})^{O(1)}$, such that on each layer of $\alpha$, $\Phi - \Theta$ is a $\mathcal{G}'$-supported $\mathbb{F}_p$-multiaffine map, provided that $|G_d| \geq \con\,\delta^{-\con}$ for each $d \in [k]$.\end{proposition}

\begin{proof}Since $\Phi$ is $\mathcal{G}$-supported, we may write it as $\Phi(x_{[k]}) = \Phi^{\text{ml}}(x_{I_0}) + \Phi'(x_{[k]})$ for an $\mathbb{F}_p$-multilinear map $\Phi^{\text{ml}} \colon G_{I_0} \to H$ and a $\mathcal{G}'$-supported $\mathbb{F}_p$-multiaffine map $\Phi' \colon G_{[k]} \to H$. Thus, whenever $a_{I_0}, b_{I_0} \in G_{I_0}$ and $z_{[k] \setminus I_0} \in G_{[k] \setminus I_0}$ are such that $(a_J, b_{I_0 \setminus J}, z_{[k] \setminus I_0}) \in A$ for each choice of $J \subset I_0$, then in fact
\[\sum_{J \subset I_0} (-1)^{|I_0| - |J|} \Phi(a_J, b_{I_0 \setminus J}, z_{[k] \setminus I_0}) = \Phi^{\text{ml}}(a_i - b_i \colon i \in I_0).\]
By averaging, there are positive quantities $\delta_1, \delta_2 \geq \Omega(\delta^{O(1)})$, an element $z_{[k] \setminus I_0} \in G_{[k] \setminus I_0}$ and a $\delta_1$-dense set $X \subset G_{I_0}$ such that for each $x_{I_0} \in X$, there are at least $\delta_2 |G_{I_0}|$ elements $b_{I_0} \in G_{I_0}$ such that $\Big((b + x)_J, b_{I_0 \setminus J}, z_{[k] \setminus I_0}\Big) \in A$ for each $J \subset I_0$.\\

\begin{claim*}\label{linearDirCoinc}Provided $|G_d| \geq \con\,\delta^{-\con}$, for each $d \in I_0$, there are a subset $X' \subset X$ of size at least $(|\mathbb{F}|^{-1} \delta)^{O(1)} |G_{I_0}|$ and an $\mathbb{F}$-linear map $\theta = \theta_{x_{I_0 \setminus \{d\}}}\colon G_d \to H$ for each $x_{I_0 \setminus \{d\}} \in G_{I_0 \setminus \{d\}}$ such that $\Phi^{\mathrm{ml}}(x_{I_0 \setminus \{d\}}, y_d) = \theta(y_d)$ for all $y_d \in X'_{x_{I_0 \setminus \{d\}}}$.\end{claim*}

\begin{proof}[Proof of claim] Let $x_{I_0 \setminus \{d\}}$ be such that $X_{x_{I_0 \setminus \{d\}}}$ is $\delta_1/2$-dense. Then there exists $b_{I_0} \in G_{I_0}$ such that $\Big((b + x)_I, b_{I_0 \setminus I}, z_{[k] \setminus I_0}\Big) \in A$ for each $I \subset I_0$ for at least $\Omega(\delta^{O(1)})|G_d|$ of $x_d \in X_{x_{I_0 \setminus \{d\}}}$. Let $\tilde{X}_{x_{I_0 \setminus \{d\}}}$ be the set of such $x_d$. Then for each $x_{d} \in \tilde{X}_{x_{I_0 \setminus \{d\}}}$,
\[\sum_{I \subset I_0} (-1)^{|I_0|-|I|} \Phi((b+x)_I, b_{I_0 \setminus I}, z_{[k] \setminus I_0}) = \Phi^{\text{ml}}(x_{I_0}).\]
By the properties of $\Phi$, for each $I \subset I_0 \setminus \{d\}$ there is an affine map $\theta_I \colon G_d \to H$ such that we have 
\[\theta_I(y_d) = \Phi((b + x)_I, b_{I_0 \setminus I \setminus \{d\}}, z_{[k] \setminus I_0}, y_d)\]
for all $y_d \in A_{(b + x)_I, b_{I_0 \setminus I \setminus \{d\}}, z_{[k] \setminus I_0}}$. It follows that for each $x_d \in \tilde{X}_{x_{I_0 \setminus \{d\}}}$ we have
\[\Phi^{\text{ml}}(x_{I_0}) = \sum_{I \subset (I_0 \setminus \{d\})} (-1)^{|I_0|-|I|-1} \Big(\theta_I(x_d + b_d) - \theta_I(x_d)\Big).\]
Writing $\theta \colon G_d \to H$ for the affine map
\[y_d \mapsto \sum_{I \subset (I_0 \setminus \{d\})} (-1)^{|I_0|-|I|-1} \Big(\theta_I(y_d + b_d) - \theta_I(y_d)\Big),\]
we sum up the work above to conclude that for each $x_{I_0 \setminus \{d\}}$ such that $X_{x_{I_0 \setminus \{d\}}}$ is $\delta_1/2$-dense, there are a set $\tilde{X}_{x_{I_0 \setminus \{d\}}}$ of size at least $\Omega(\delta^{O(1)})|G_d|$ and an affine map $\theta \colon G_d \to H$ such that $\Phi^{\text{ml}}(x_{I_0}) = \theta(x_d)$ for $x_d \in \tilde{X}_{x_{I_0 \setminus \{d\}}}$. We now modify $\theta$ to make it an $\mathbb{F}$-linear map.\\

Fix an $\mathbb{F}$ dot product $\cdot$ on $G_d$. Take a random element $u \in G_d$. Define a map $\theta'$ by $\theta'(x) = \theta(x) - \theta(0) + \theta(0) u \cdot x$. View $\mathbb{F}_p$ as a subset of $\mathbb{F}$. Then for each $x_d \in \tilde{X}_{x_{I_0 \setminus \{d\}}} \setminus \{0\}$, 
\[\mathbb{P}\Big(\Phi^{\text{ml}}(x_{I_0}) = \theta'(x_d)\Big) = \mathbb{P}\Big(\theta(x_d) = \theta'(x_d)\Big) = \mathbb{P}(u \cdot x = 1) = |\mathbb{F}|^{-1}.\]
There is therefore a choice of $u$ such that $\Phi^{\text{ml}}(x_{I_0}) = \theta'(x_d)$ on a set $X'_{x_{I_0 \setminus \{d\}}} \subset \tilde{X}_{x_{I_0 \setminus \{d\}}}$ of size at least $|\mathbb{F}|^{-1}(|\tilde{X}_{x_{I_0 \setminus \{d\}}}| - 1)$, as desired. This proves the claim.\end{proof}
\medskip

\indent Now define a map $\psi \colon G_{I_0} \times H \to \mathbb{F}$ by $\psi(x_{I_0}, h) = \Phi^{\text{ml}}(x_{I_0}) \cdot h$, where $\cdot$ is an $\mathbb{F}$ dot product on $H$. Then $\psi$ is an $\mathbb{F}_p$-multilinear map. Fix an $\mathbb{F}_p$-basis $e_1, \dots, e_r$ of $\mathbb{F}$ and for each $\lambda \in \mathbb{F}$ let $\lambda_i$ be the $i$\textsuperscript{th} coordinate with respect to this basis. We now claim that for each $i \in [r], d \in I_0, \lambda \in \mathbb{F}$, the $\mathbb{F}_p$-multilinear form
\[(x_{I_0}, h) \mapsto \psi_i(x_{I_0 \setminus \{d\}}, \lambda x_d, h) - (\lambda \psi)_i(x_{I_0}, h)\]
has large bias.\\
\indent Let $X'$ be the set provided by the claim. This set has size $\delta_3 |G_{I_0}|$, where $\delta_3 = \Omega(\delta^{O(1)})$. Let $Y$ be the set of all $x_{I_0 \setminus \{d\}} \in G_{I_0 \setminus \{d\}}$ such that $|X'_{x_{I_0 \setminus \{d\}}}|\ \geq \frac{\delta_3}{2} |G_d|$. Then $|Y|\ \geq \frac{\delta_3}{2} |G_{I_0 \setminus \{d\}}|$. Observe that we may in fact assume that $X'_{x_{I_0 \setminus \{d\}}}$ is an $\mathbb{F}_p$-subspace of $G_d$, for each $x_{I_0 \setminus \{d\}} \in Y$ (we no longer use the fact that $X' \subset X$). Indeed, the set of all $y_d \in G_d$ such that $\theta_{x_{I_0 \setminus \{d\}}}(y_d) = \Phi^{\text{ml}}(x_{I_0 \setminus \{d\}}, y_d)$ is an $\mathbb{F}_p$-subspace. Let $X''_{x_{I_0 \setminus \{d\}}} = \bigcap_{\mu \in \mathbb{F}^\times} \mu \cdot X'_{x_{I_0 \setminus \{d\}}}$. This is then an $\mathbb{F}$-subspace of density at least $2^{1-|\mathbb{F}|}\delta_3^{|\mathbb{F}| - 1}$.\\
\indent For each $x_{I_0 \setminus \{d\}} \in Y$ we obtain
\begin{align*}\exx_{y_d, h} \chi\Big(\psi_i(x_{I_0 \setminus \{d\}}, \lambda y_d, h)& - (\lambda \psi)_i(x_{I_0 \setminus \{d\}}, y_d, h)\Big)\\ 
&= \exx_{y_d, h} \chi\Big((\Phi^{\text{ml}}(x_{I_0 \setminus \{d\}}, \lambda y_d) \cdot h - \lambda \Phi^{\text{ml}}(x_{I_0 \setminus \{d\}}, y_d) \cdot h)_i\Big)\\
&= \exx_{y_d} \bigg(\exx_h\chi\Big((\Phi^{\text{ml}}(x_{I_0 \setminus \{d\}},\lambda y_d) \cdot h - \lambda \Phi^{\text{ml}}(x_{I_0 \setminus \{d\}}, y_d) \cdot h)_i\Big)\bigg).\\
\end{align*}
Since the expression $\ex_h\chi\Big((\Phi^{\text{ml}}(x_{I_0 \setminus \{d\}},\lambda y_d) \cdot h - \lambda \Phi^{\text{ml}}(x_{I_0 \setminus \{d\}}, y_d) \cdot h)_i\Big)$ always takes values in $[0,1]$, this is at least
\[\exx_{y_d} \mathbbm{1}(y_d \in X''_{x_{I_0 \setminus \{d\}}})\exx_h\chi\Big((\Phi^{\text{ml}}(x_{I_0 \setminus \{d\}}, \lambda y_d) \cdot h - \lambda \Phi^{\text{ml}}(x_{I_0 \setminus \{d\}}, y_d) \cdot h)_i\Big).\]
If $y_d \in X''_{x_{I_0 \setminus \{d\}}}$, then $\lambda y_d\in X''_{x_{I_0 \setminus \{d\}}}$ as well, so this is equal to
\begin{align*} \exx_{y_d} \mathbbm{1}(y_d \in X''_{x_{I_0 \setminus \{d\}}})\exx_h \chi\Big(((\theta(\lambda y_d) - \lambda \theta(y_d)) \cdot h)_i\Big) &=  |X''_{x_{I_0 \setminus \{d\}}}| / |G_d|\,\, \geq 2^{1-|\mathbb{F}|}\delta_3^{|\mathbb{F}| - 1}.
\end{align*}
After averaging over $x_{I_0 \setminus\{d\}}$, the bias we considered is at least $2^{-|\mathbf{F}|}\delta_3^{|\mathbf{F}|}$.\\

By Theorem~\ref{strongInvARankThm}, $\rho_{i,\lambda,d}(x_{I_0}, h) = \psi_i(x_{I_0 \setminus \{d\}}, \lambda x_d, h) - (\lambda  \psi)_i(x_{I_0}, h)$ has partition-rank $s = O(\log_{p} \delta^{-1})^{O(1)}$.

We now apply Lemma \ref{lowRankApprox} $|I_0|$ times to get an $\mathbb{F}$-multilinear form $\sigma \colon G_{I_0} \times H \to \mathbb{F}$ such that $\prank_{\mathbb{F}_p} (\psi_i - \sigma_i) \leq (2r)^{O(k)} s$ for each $i \in [r]$. Write $\sigma(x_{I_0}, h) = S(x_{I_0}) \cdot h$ for an $\mathbb{F}$-multilinear map $S \colon G_{I_0} \to H$. For each $i \in [r]$, we get that $\{x_{I_0} \in G_{I_0} \colon (\forall h \in H) \psi_i(x_{I_0}, h) = \sigma_i(x_{I_0}, h)\}$ contains an $\mathbb{F}_p$-multilinear variety $W_i$ of codimension $O(\log_p \delta^{-1})^{O(1)}$. Thus, $\{x_{I_0} \in G_{I_0} \colon (\forall h \in H)(\forall i \in [r]) \hspace{2pt} \psi_i(x_{I_0}, h) = \sigma_i(x_{I_0}, h)\}$ contains an $\mathbb{F}_p$-multilinear variety of codimension $s' = O(\log_p \delta^{-1})^{O(1)}$, and in particular has density at least $p^{-k s'}$. Hence, $S =  \Phi^{\text{ml}}$ on a set of density at least $p^{-ks'}$. Let $\overset{p}{\bullet}$ be an $\mathbb{F}_p$ dot product on $H$. Now consider the $\mathbb{F}_p$-multilinear forms $\psi^{\mathbb{F}_p}$ and $\sigma^{\mathbb{F}_p}$ defined by $\psi^{\mathbb{F}_p}(x_{I_0}, h) = \Phi^{\text{ml}}(x_{I_0}) \overset{p}{\bullet} h$ and $\sigma^{\mathbb{F}_p}(x_{I_0}, h) = S(x_{I_0}) \overset{p}{\bullet} h$. Apply Theorem~\ref{strongInvARankThm} to the form $\psi^{\mathbb{F}_p} - \sigma^{\mathbb{F}_p}$ to complete the proof of the proposition.\end{proof}

\begin{proof}[Proof of Theorem~\ref{multiaffineInvThmGen}] We prove the claim by induction on $k$, the base case $k = 1$ being trivial. Thus assume that $k \geq 2$ and that the claim holds for smaller values of $k$. Let $\eta > 0$ be a parameter to be specified later. We begin the proof by showing that we may assume the condition $|G_d| \geq \con\,\eta^{-\con}$ for each $d \in [k]$. Suppose on the contrary that $|G_d| \leq O(\eta^{-O(1)})$ for some $d$, say $d = k$. We simply average over $G_k$ to find an element $x_k$ such that 
\[|A_{x_k}| \geq \frac{1}{|G_k|} |A| \geq \Omega(\eta^{O(1)} |A|).\]
The inductive hypothesis applies to give an $\mathbb{F}$-multiaffine map $\Phi \colon G_{[k-1]} \to H$ such that $\Phi(y_{[k-1]}) = \phi(y_{[k-1]}, x_k)$ holds for at least $\Big(\exp^{(D_{k-1})}(O(\eta^{-1}))\Big)^{-1} |G_{[k-1]}|$ choices of $y_{[k-1]} \in G_{[k-1]}$. In particular $\phi(y_{[k]}) = \Phi(y_{[k-1]})$ holds for at least 
\[\Omega(\eta^{O(1)}) \Big(\exp^{(D_{k-1})}(O(\eta^{-1}))\Big)^{-1} |G_{[k]}|\]
elements $y_{[k]} \in G_{[k]}$ and we are done. Thus, assume that the condition on the size of $G_d$ is satisfied, provided $\eta$ is not too small.\\

From the assumptions, we deduce that $\phi$ is a Freiman homomorphism in each direction, so Theorem~\ref{multiaffineInvThm} applies and gives an $\mathbb{F}_p$-multiaffine map $\Phi \colon G_{[k]} \to H$ such that $\phi = \Phi$ on a set $B \subset A$ of size $|B| = \delta_1 |G_{[k]}|$, where $\delta_1 = \Big(\exp^{(D^{\mathrm{mh}}_k)}(O(\delta^{-1}))\Big)^{-1}$. We now use $\Phi$ to find the desired $\mathbb{F}$-multiaffine map. List all subsets $I_1, \dots, I_{2^{k}} \subset \mathcal{P}([k])$ in such a way that $I_i \supset I_j$ implies $i \leq j$. Let $\mathcal{G}_i = \{I_i, \dots, I_{2^k}\}$, which is a down-set for each $i \in [2^k]$. By induction on $i \in [2^k]$, we claim that there exist an $\mathbb{F}$-multiaffine map $\Theta^{(i)} \colon G_{[k]} \to H$ and a $\mathcal{G}_i$-supported $\mathbb{F}_p$-multiaffine map $\Psi^{(i)} \colon G_{[k]} \to H$  such that $\Phi = \Theta^{(i)} + \Psi^{(i)}$ on a subset $B^{(i)} \subset B$ of size at least $\exp\Big(-\log^{O(1)} \delta_1^{-1}\Big)|G_{[k]}|$. The base case $i = 1$ is trivial -- we simply take $\Theta^{(1)} = 0$, $\Psi^{(1)} = \Phi$ and $B^{(1)} = B$.\\
\indent Assume now that the claim holds for some $i \in [2^k-1]$, let $\Theta^{(i)}$ and $\Psi^{(i)}$ be the relevant maps, and let $B^{(i)}$ be the relevant set. Note that $\Phi$ and $\Theta^{(i)}$ are $\mathbb{F}$-multiaffine on $B^{(i)}$, which makes $\Psi^{(i)}$ $\mathbb{F}$-multiaffine on $B^{(i)}$ as well. Since $\Psi^{(i)}$ is $\mathcal{G}_i$-supported and $I_i$ is a maximal set in $\mathcal{G}_i$, we may apply Proposition~\ref{app1MainProp} to find an $\mathbb{F}$-multiaffine map $\Theta' \colon G_{I_i} \to H$ and an $\mathbb{F}_p$-multiaffine map $\alpha \colon G_{I_{[k]}} \to \mathbb{F}_p^{s'}$, where $s' = O(\log_p \delta_1^{-1})^{O(1)}$, such that on each layer of $\alpha$, $\Psi^{(i)} - \Theta'$ is a $\mathcal{G}_{i+1}$-supported $\mathbb{F}_p$-multiaffine map. (By the discussion at the beginning of the proof, we may assume that the condition on $|G_d|$ in Proposition~\ref{app1MainProp} is satisfied, by setting $\eta = \exp(-\log^{O(1)} \delta_1^{-1})$.) Take a layer $L$ of $\alpha$ such that the set $B^{(i+1)}=L \cap B^{(i)}$ has size at least $p^{-s'} |B^{(i)}|$. Then on $L$, $\Psi^{(i)} = \Theta' + \Psi^{(i+1)}$ for a $\mathcal{G}_{i+1}$-supported $\mathbb{F}_p$-multiaffine map $\Psi^{(i+1)}$. Set $\Theta^{(i+1)} = \Theta^{(i)} + \Theta'$ to finish the proof of the claim.\\
\indent Finally, apply the above claim for $i = 2^k$ to complete the proof of the theorem.\end{proof}

\subsection{The structure of approximate polynomials}
Recall that a (generalized) polynomial of degree at most $k$ from an Abelian group $G$ to an Abelian group $H$ can be defined as a function $\Phi$ with the property that $\Delta_{a_1} \dots \Delta_{a_{k+1}} \Phi(x) = 0$ for every $x,a_1,\dots,a_{k+1}\in G$, where $\Delta_af(x)$ is defined to be $f(x)-f(x-a)$. Our next main theorem states that a function that satisfies this equation for a dense set of $(x,a_1,\dots,a_{k+1})$ agrees on a dense subset of $G$ with a (classical) polynomial of degree at most~$k$.

\begin{theorem}\label{ApproxPolyThmGen}Let $G$ and $H$ be $\mathbb{F}_p$-vector spaces, with $p > k$. Suppose that $\phi \colon G \to H$ is a function with the property that $\Delta_{a_1} \dots \Delta_{a_{k+1}} \phi(x) = 0$ for at least $c|G|^{k+2}$ choices of $(x, a_1, \dots, a_{k+1}) \in G^{k+2}$. Then there is a polynomial $\Phi \colon G \to H$ of degree at most $k$ such that $\phi(x) = \Phi(x)$ for every $x$ in a subset of $G$ of density $\Omega\Big(\exp^{(O(1))} (c^{-1})\Big)$.\end{theorem}

\begin{proof} The proof will proceed by induction on $k$ and will take several steps. For $k = 0$, the result is straightforward to prove. For the rest of this subsection we shall assume that the result has been proved for $k-1$. 
\medskip

\noindent\textbf{Step 1. Finding a multiaffine map.} We may rewrite the condition as the statement that
\[\Delta_{a_1} \dots \Delta_{a_{k}} \phi(x) = \Delta_{a_1} \dots \Delta_{a_{k}} \phi(y)\]
for at least $c|G|^{k+2}$ $(k+2)$-tuples $(x, y, a_1, \dots, a_k) \in G^{k+2}$. By averaging, we may find a set $S \subset G^k$ and a map $\psi \colon S \to H$ such that $|S|\ = c^{O(1)} |G^k|$ and such that for each $a_{[k]} \in S$ we may find $c^{O(1)} |G|$ elements $x \in G$ for which
\[\Delta_{a_1} \dots \Delta_{a_{k}} \phi(x) = \psi(a_{[k]}).\]

We now prove that $\psi$ coincides with a global multiaffine map on a dense set.

\begin{claim}\label{appPolyClaim}Let $A \subset G^k$ be a set of density $\delta$ such that for each $a_{[k]} \in A$, there are at least $\delta' |G|$ elements $x \in G$ such that $\Delta_{a_1} \dots \Delta_{a_{k}} \phi(x) = \psi(a_{[k]})$. Then for each $d \in [k]$, $\psi$ respects $\Omega(\delta^5 {\delta'}^4) |G|^{k+2}$ $d$-additive quadruples.
\end{claim}

\begin{proof}Without loss of generality $d=k$. Take any $a_{[k-1]} \in G^{k-1}$ such that $|A_{a_{[k-1]}}| \geq \frac{\delta}{2}|G|$. There are at least $\frac{\delta}{2}|G|^{k-1}$ such $a_{[k-1]}$. Define $\alpha \colon G \to H$ by
\[\alpha(x) = \Delta_{a_1} \dots \Delta_{a_{k-1}} \phi(x).\]
Then there are at least $\frac{\delta \delta'}{2}|G|^2$ pairs $(b,x) \in A_{a_{[k-1]}} \times G$ such that
\[\alpha(x) - \alpha(x-b) = \psi(a_{[k-1]}, b).\]
That is, there are at least $\frac{\delta \delta'}{2}|G|^2$ of $(x,y) \in G^2$ such that $(a_{[k-1]}, x-y) \in A$ and
\[\alpha(x) - \alpha(y) = \psi(a_{[k-1]}, x-y).\] 
By the Cauchy--Schwarz inequality, there are at least $\Big(\frac{\delta \delta'}{2}\Big)^2|G|^3$ triples $(x_1, x_2,y) \in G^3$ such that $(a_{[k-1]}, x_1-y), (a_{[k-1]}, x_2-y) \in A$, $\alpha(x_1) - \alpha(y) = \psi(a_{[k-1]}, x_1-y)$, and $\alpha(x_2) - \alpha(y) = \psi(a_{[k-1]}, x_2-y)$. The last two conditions imply that
\[\alpha(x_1) - \alpha(x_2) = \psi(a_{[k-1]}, x_1-y) - \psi(a_{[k-1]}, x_2-y).\] 
Applying the Cauchy--Schwarz inequality once more, we get at least $\Big(\frac{\delta \delta'}{2}\Big)^4|G|^4$ quadruples $(x_1, x_2 ,y_1, y_2) \in G^4$ such that $(a_{[k-1]}, x_1-y_1), (a_{[k-1]}, x_2-y_1), (a_{[k-1]}, x_1-y_2), (a_{[k-1]}, x_2-y_2) \in A$,
\[\alpha(x_1) - \alpha(x_2) = \psi(a_{[k-1]}, x_1-y_1) - \psi(a_{[k-1]}, x_2-y_1),\]
and
\[\alpha(x_1) - \alpha(x_2) = \psi(a_{[k-1]}, x_1-y_2) - \psi(a_{[k-1]}, x_2-y_2).\] 
The last two conditions imply that
\[\psi(a_{[k-1]}, x_1-y_1) - \psi(a_{[k-1]}, x_2-y_1) - \psi(a_{[k-1]}, x_1-y_2) + \psi(a_{[k-1]}, x_2-y_2) = 0,\]
which completes the proof.\end{proof}

Combining this claim with Theorem~\ref{FreBSG} for each direction, we find a set $B \subset G^k$ of size $c_1 |G^k|$, where $c_1 = \exp(-O(\log c^{-1})^{O(1)})$, such that for each $d \in [k]$ and $x_{[k] \setminus \{d\}} \in G_{[k] \setminus \{d\}}$, $(\psi|_B)_{x_{[k] \setminus \{d\}}}$ coincides with an affine map $G \to H$. Thus, by Theorem~\ref{multiaffineInvThm}, there is a global multiaffine map $\Psi \colon G^k \to H$ such that
\begin{equation}\label{CubePsiEqn}\Delta_{a_1} \dots \Delta_{a_{k}} \phi(x) = \Psi(a_{[k]})\end{equation}
for $c_2 |G|^{k+1}$ choices of $(a_{[k]}, x)$, where $c_2 = \Big(\exp^{(O(1))} (O(c^{-1}))\Big)^{-1}$.\\

\noindent\textbf{Step 2. We may take $\Psi$ to be multilinear.} This follows from the following proposition.

\begin{proposition}\label{lowOrderRemovalPoly}Suppose that $\Delta_{a_1} \dots \Delta_{a_{k}} \phi(x) = \Theta(a_{[k]})$ holds for $c_0 |G|^{k+1}$ choices of $(a_{[k]}, x)$ in $G^{k+1}$, where $\Theta$ is a global multiaffine map. Then $\Delta_{a_1} \dots \Delta_{a_{k}} \phi(x) = \Theta^{\text{\emph{ml}}}(a_{[k]})$ for $c_0^{O(1)} |G|^{k+1}$ choices of $(a_{[k]}, x)$ in $G^{k+1}$.\end{proposition}

\begin{proof}Write $\Theta_0 = \Theta$, and define $\Theta_1, \dots, \Theta_k$ iteratively as $\Theta_{i-1}(a_{[k]}) = \Theta_i(a_{[k]}) + \theta_i(a_{[k] \setminus \{i\}})$, where $\Theta_i$ is linear in coordinate $i$ and $\theta_i$ is multiaffine. Actually, we obtain that $\Theta_i$ is linear in the first $i$ coordinates, and in particular, $\Theta_k$ is multilinear and therefore $\Theta^{\text{ml}} = \Theta_k$. Let $T^0$ be the set of $(x, a_{[k]}) \in G^{k+1}$ such that 
\[\Delta_{a_1} \dots \Delta_{a_{k}} \phi(x) = \Theta(a_{[k]}).\]
Thus, $|T^0| \geq c_0|G|^{k+1}$.

We claim now that for each $i \in [0, k]$, there is a set $T^i \subset G^{k+1}$ such that $|T^i| \geq c_0^{O(1)}|G|^{k+1}$ and for each $(x, a_{[k]}) \in T^i$
\begin{equation}\Delta_{a_1} \dots \Delta_{a_{k}} \phi(x) = \Theta_i(a_{[k]}),\label{CubePsiEqn2}\end{equation}
which will imply the proposition.

We prove this claim by induction on $i$. The base case is the case $i = 0$, where the claim holds trivially. Assume now that it holds for some $i \in [0, k-1]$, and fix $x_0 \in G$. Observe that if $b_{i+1}, c_{i+1} \in T^i_{x_0, a_{[k] \setminus \{i+1\}}}$, then
\begin{align*}\Theta_{i+1}(a_{[k] \setminus \{i+1\}}, b_{i+1} - c_{i+1}) = & \Theta_i(a_{[k] \setminus \{i+1\}}, b_{i+1}) - \Theta_i(a_{[k] \setminus \{i+1\}}, c_{i+1})\\
= & \Delta_{a_{k}} \dots \Delta_{a_{i+2}}\Delta_{b_{i+1}} \Delta_{a_{i}}\dots \Delta_{a_1}\phi (x_0)\\
&\hspace{2cm}- \Delta_{a_{k}} \dots \Delta_{a_{i+2}}\Delta_{c_{i+1}} \Delta_{a_{i}}\dots \Delta_{a_1}\phi (x_0)\\
= &\Delta_{a_{k}} \dots \Delta_{a_{i+2}} \Delta_{a_{i}}\dots \Delta_{a_1}\phi (x_0 - c_{i+1}) \\
&\hspace{2cm}- \Delta_{a_{k}} \dots \Delta_{a_{i+2}} \Delta_{a_{i}}\dots \Delta_{a_1}\phi (x_0 - b_{i+1})\\
= &\Delta_{a_{k}} \dots \Delta_{a_{i+2}} \Delta_{b_{i+1} - c_{i+1}} \Delta_{a_{i}}\dots \Delta_{a_1}\phi (x_0 - c_{i+1}).\\
\end{align*}

By the Cauchy--Schwarz inequality, there are $c_0^{O(1)}|G|^{k+2}$ choices of $(x_0, a_{[k] \setminus \{i+1\}}, b_{i+1}, c_{i+1})$ such that the above equality holds. The result follows after a change of variables.\end{proof}

We shall abuse notation and keep writing $\Psi$ for the modified multilinear map $\Psi^{\text{ml}}$. We end up with
\begin{equation}\label{CubePsi2}\Delta_{a_1} \dots \Delta_{a_{k}} \phi(x) = \Psi(a_{[k]})\end{equation}
which holds for a $c_3$-dense collection of $(a_{[k]}, x)$, where $c_3 = c_2^{O(1)}$.\\
\vspace{\baselineskip}

\noindent\textbf{Step 3. A symmetry argument.} What we would like to do at this point is use a polarization identity to obtain a polynomial from $\Psi$, but we cannot do that because $\Psi$ is not symmetric. In this section we show how to obtain a symmetric multilinear map from $\Psi$. The argument will have a similar flavour to an argument introduced in a slightly different context by Green and Tao in \cite{StrongU3}.

We begin by showing that, for each $i \in [k-1]$, the variety
\[\Big\{a_{[k]} \in G^k \colon \Psi(a_1, \dots, a_{i-1}, a_i, a_{i+1}, a_{i+2}, \dots, a_k) = \Psi(a_1, \dots, a_{i-1}, a_{i+1}, a_i, a_{i+1}, \dots, a_k)\Big\}\]
is dense.\\

\indent Let us focus on the first two coordinates. Write $T \subset G^{k+1}$ for the set of all $(x, a_{[k]})$ such that~\eqref{CubePsi2} holds. Fix $x \in G$ and $a_{[2,k]} \in G^{k-1}$. Then, for $u_1,v_1 \in T_{x, a_{[2,k]}}$, we have
\begin{align*}\Psi(u_1 - v_1, a_{[2,k]}) = &\Delta_{u_1} \Delta_{a_2} \dots \Delta_{a_k} \phi(x) - \Delta_{v_1} \Delta_{a_2} \dots \Delta_{a_k} \phi(x)\\
= &\Delta_{a_2} \dots \Delta_{a_k} \phi(x - v_1) - \Delta_{a_2} \dots \Delta_{a_k} \phi(x - u_1)\\
 = &\Delta_{a_3} \dots \Delta_{a_k} \phi(x - u_1 - a_2) - \Delta_{a_3} \dots \Delta_{a_k} \phi(x - u_1)\\
&\hspace{1cm} - \Delta_{a_3} \dots \Delta_{a_k} \phi(x - v_1 - a_2) + \Delta_{a_3} \dots \Delta_{a_k} \phi(x - v_1).\end{align*}
By the Cauchy--Schwarz inequality, this holds for ${c_3}^{O(1)}|G|^{k+2}$ choices of $(x, u_1, v_1, a_{[2,k]})$.\\
\indent Let $z = x - a_2 - u_1 - v_1$. Then
\begin{align*}\Psi(u_1 - v_1, x - u_1 - v_1 - z, a_{[3,k]}) = &\Delta_{a_3} \dots \Delta_{a_k} \phi(z + v_1) + \Delta_{a_3} \dots \Delta_{a_k} \phi(x - v_1)\\
&\hspace{1cm} -  \Delta_{a_3} \dots \Delta_{a_k} \phi(x - u_1) - \Delta_{a_3} \dots \Delta_{a_k} \phi(z + u_1).\end{align*}  
Hence, there is a set $T' \subset G^{k+2}$ of size ${c_3}^{O(1)}|G|^{k+2}$ whose elements $(x, u_1, v_1, z, a_{[3,k]})$ satisfy 
\begin{align*}\Psi(v_1, u_1, a_{[3,k]}) - \Psi(u_1, v_1, a_{[3,k]}) =&\Delta_{a_3} \dots \Delta_{a_k} \phi(x - v_1) + \Delta_{a_3} \dots \Delta_{a_k} \phi(z + v_1)\\
&\hspace{1cm}+ \Psi(v_1, x - v_1 - z, a_{[3,k]})\\
&\hspace{1cm}- \Delta_{a_3} \dots \Delta_{a_k} \phi(x - u_1) - \Delta_{a_3} \dots \Delta_{a_k} \phi(z + u_1)\\
&\hspace{1cm} -\Psi(u_1, x - u_1 - z, a_{[3,k]}).\end{align*}  
Fix $x, u_1, z, a_{[3,k]}$. For any $v_1, v'_1 \in T'_{(x, u_1, z, a_{[3,k]})}$ we get, after subtracting, that
\begin{align*}\Psi(v_1 - v'_1, u_1, a_{[3,k]}) - \Psi(u_1, v_1-v'_1, a_{[3,k]}) =& \Delta_{a_3} \dots \Delta_{a_k} \phi(x - v_1) + \Delta_{a_3} \dots \Delta_{a_k} \phi(z + v_1) \\
&\hspace{1cm}+\Psi(v_1, x - v_1 - z, a_{[3,k]}) -  \Delta_{a_3} \dots \Delta_{a_k} \phi(x - v'_1)\\
&\hspace{1cm} -\Delta_{a_3} \dots \Delta_{a_k} \phi(z + v'_1) - \Psi(v'_1, x - v'_1 - z, a_{[3,k]}).\end{align*}  
By the Cauchy--Schwarz inequality, we obtain a set $T'' \subset G^{k+3}$ of size ${c_3}^{O(1)}|G|^{k+3}$ such that each $(x, u_1, v_1, v'_1, z, a_{[3,k]}) \in T''$ satisfies the equality above. Now fix $(x, v_1, v'_1, z, a_{[3,k]})$. For each $u_1, u_1' \in T''_{(x, v_1, v'_1, z, a_{[3,k]})}$ we get, after subtracting, that
\begin{align*}\Psi(v_1-v'_1, &u_1-u'_1, a_{[3,k]}) - \Psi(u_1-u'_1, v_1-v'_1, a_{[3,k]}) = \\
&\Psi(v_1-v'_1, u_1, a_{[3,k]}) - \Psi(u_1, v_1-v'_1, a_{[3,k]}) - \Psi(v_1-v'_1, u'_1, a_{[3,k]}) + \Psi(u'_1, v_1-v'_1, a_{[3,k]}) = 0.\end{align*}
Hence, by Cauchy--Schwarz, 
\[\Psi(v_1-v'_1, u_1-u'_1, a_{[3,k]}) - \Psi(u_1-u'_1, v_1-v'_1, a_{[3,k]}) = 0\]
for $c_3^{O(1)}|G|^{k+2}$ choices of $u_1, u'_1, v_1, v'_1, a_{[3,k]}$. We are done after a change of variables.\\

By Claim 1.6 in~\cite{Lovett}, it follows that the variety
\[V_{\text{sym}} = \Big\{a_{[k]} \in G^k \colon (\forall i \in [k-1]) \Psi(a_1, \dots, a_i, a_{i+1}, \dots, a_k) = \Psi(a_1, \dots, a_{i+1}, a_i, \dots, a_k)\Big\}\]
has density ${c_3}^{O(1)}$.
\vspace{\baselineskip}

Let $\psi \colon G^k \times H \to \mathbb{F}_p$ be the multilinear form given by $\psi(x_{[k]}, y) = \Psi(x_{[k]}) \cdot y$. For a permutation $\pi \in \operatorname{Sym}_k$, let $\pi \circ \psi(x_{[k]}, h) = \psi(x_{\pi(1)}, \dots, x_{\pi(k)}, h)$. Then for each $\pi$, $\bias \Big(\pi \circ \psi - \psi\Big) \geq {c_3}^{O(1)}$. By Theorem~\ref{strongInvARankThm}, we have $\prank \Big(\pi \circ \psi - \psi\Big) \leq O(\log_p c_3^{-1})^{O(1)}$.\\
\indent Define a multilinear form $\theta \colon G^k \times H \to \mathbb{F}_p$ by
\[\theta(x_{[k]}, y) = \frac{1}{k!} \sum_{\pi \in \operatorname{Sym}_k} \pi \circ \psi(x_{[k]}, h).\]
Then $s = \prank (\psi - \theta) = O(\log_p c_3^{-1})^{O(1)}$, which means that there exist non-empty sets $I_i \subset [k]$, and multilinear forms $\alpha_i \colon G^{I_i} \to \mathbb{F}_p$ and $\beta_i \colon G^{[k] \setminus I_i} \times H \to \mathbb{F}_p$, for each $i \in [s]$, such that for each $(x_{[k]}, h) \in G^k \times H$ we have
\[\psi(x_{[k]}, h) - \theta(x_{[k]}, h) = \sum_{i \in [s]} \alpha_i(x_{I_i}) \beta_i(x_{[k] \setminus I_i}, h).\]
\indent Write $\theta(x_{[k]}, y) = \Theta(x_{[k]}) \cdot y$ and $\beta_i(x_{[k] \setminus I_i}, y) = B_i(x_{[k]\setminus I_i}) \cdot y$ for multilinear maps $\Theta$ and $B_i$. Then, $\Theta = \frac{1}{k!} \sum_{\pi \in \operatorname{Sym}_k} \pi \circ \Psi$, so it is symmetric, and $\Psi = \Theta + \sum_{i \in [s]} \alpha_i(x_{I_i}) B_i(x_{[k] \setminus I_i})$. By averaging over layers of $\alpha$, we obtain $\mu \in \mathbb{F}_p^s$ and a $c_3 p^{-s}$-dense set of $(a_{[k]}, x)$ such that $\alpha_i(a_{I_i}) = \mu_i$ and
\[\Delta_{a_1} \dots \Delta_{a_{k}} \phi(x) = \Theta(a_{[k]}) + \sum_{i \in [s]} \mu_i B_i(a_{[k] \setminus I_i}).\]
Applying Proposition~\ref{lowOrderRemovalPoly}, we may without loss of generality assume that the $B_i$ do not appear in the expression above, which now holds for a $c^{O(1)}_3 p^{-O(s)}$-dense set of $(a_{[k]}, x)$. Define $\phi'(x) = \phi(x) - \frac{1}{k!}\Theta(x,x,\dots,x)$. By the polarization identity,
\[\Delta_{a_1} \dots \Delta_{a_{k}}\phi'(x) = 0\]
holds on a set of parameters of density $c^{O(1)}_3 p^{-O(s)}$. We may now apply the inductive hypothesis to finish the proof of Theorem \ref{ApproxPolyThmGen}.\end{proof}

\subsection{A quantitative inverse theorem for the $U^k$ norm}

For our next application, we show that a bounded function defined on $G=\mathbb F_p^n$ with large $U^k$ norm must correlate well with a polynomial phase function of degree at most $k-1$. For each fixed $k$, the dependence on the $U^k$ norm is given by a bounded number of exponentials, but this number increases with $k$, so the dependence on $k$ is given by a tower-type function.

For a map $f \colon G \to \mathbb{C}$, and $a \in G$, we write $\mder_a f \colon G \to \mathbb{C}$ for the map $\mder_a f(x) = f(x) \overline{f(x - a)}$.

\begin{theorem}\label{UkInverseTheorem}Let $f \colon G \to \mathbb{D}$ be a function such that $\|f\|_{U^{k}} \geq \delta$. Assume that $p \geq k$. Write $\omega = \exp\Big(\frac{2 \pi i}{p}\Big)$. Then there is a polynomial $g \colon G \to \mathbb{F}_p$ of degree at most $k-1$ such that
\[\Big|\exx_x f(x) \omega^{g(x)}\Big|\ =\Big(\exp^{(O(1))} (O(\delta^{-1}))\Big)^{-1}.\]
\end{theorem}

Let us first recall the definition of higher-dimensional box norms. Let $X_1, \dots, X_k$ be arbitrary sets. The \emph{box norm} of a function $f \colon X_1 \tdt X_k \to \mathbb{C}$ (see for examle Definition B.1 in the Appendix B of~\cite{GreenTaoPrimes}) is defined by
\[\|f\|_{\square^{k}}^{2^k} = \exx_{x_1, y_1 \in X_1, \dots, x_k, y_k \in X_k} \prod_{I \subset [k]} \operatorname{Conj}^{|I|} f(x_I, y_{[k] \setminus I}).\]
Note that when $\phi \colon G_{[k]} \to \mathbb{F}$ is a multilinear form, we have $\|\chi \circ \phi\|_{\square^{k}}^{2^k} = \exx_{x_{[k]} \in G_{[k]}} \chi(\phi(x_{[k]}))$. The following is a well-known generalized Cauchy--Schwarz inequality for the box norm.

\begin{lemma}\label{unifBound}Let $f_I \colon X_1 \tdt X_k \to \mathbb{C}$ be a function for each $I \subset [k]$. Then
\[\Big|\exx_{x_1, y_1 \in X_1, \dots, x_k, y_k \in X_k} \prod_{I \subset [k]} \operatorname{Conj}^{|I|} f_I(x_I, y_{[k] \setminus I})\Big| \leq \prod_{I \subset [k]} \|f_I\|_{\square^k}.\]
\end{lemma}

This lemma has the following useful corollary.

\begin{lemma}\label{unifBound2}Let $f \colon X_1 \tdt X_k \to \mathbb{D}$ be a function and for each $i \in [k]$ let $f_i \colon \prod_{j \in [k] \setminus \{i\}} X_j \to \mathbb{D}$ be a function that does not depend on $i$\textsuperscript{th} coordinate. Then
\[\Big|\exx_{x_{[k]}} f(x_{[k]}) \prod_{i \in [k]} f_i(x_{[k] \setminus \{i\}})\Big| \leq \|f\|_{\square^k}.\]
\end{lemma}

We begin the proof of the inverse theorem by proving the following lemma, which is a straightforward generalization of lemmas that have played similar roles in proofs of earlier $U^k$ inverse theorems.

\begin{lemma}\label{multiaffineUkPassClaim}Let $S \subset G^{k-2}$ be a set of density $\delta'$ such that for each $a_{[k-2]} \in S$
\[\Big|[\mder_{a_1} \dots \mder_{a_{k-2}} f\fcc(\phi(a_{[k-2]}))\Big| \geq \delta''.\]
Then, for each $d \in [k]$, $\phi$ respects at least ${\delta'}^4 {\delta''}^8 |G|^k$ $d$-additive quadruples.\end{lemma}

\begin{proof}Since $\mder_a \mder_b = \mder_b \mder_a$, we may assume without loss of generality that $d = 1$. We have
\begin{align*}\exx_{a_{[k-2]}} &S(a_{[k-2]}) \Big|[\mder_{a_1} \dots \mder_{a_{k-2}} f\fcc(\phi(a_{[k-2]}))\Big|^2\\            
&=\exx_{a_{[k-2]}} S(a_{[k-2]}) \exx_{x,y} \mder_{a_1} \dots \mder_{a_{k-2}} f(x) \overline{\mder_{a_1} \dots \mder_{a_{k-2}}f(y)} \omega^{-\phi(a_{[k-2]}) \cdot (x - y)}\\
&=\exx_{a_{[k-2]}} S(a_{[k-2]}) \exx_{x,u} \mder_{a_1} \dots \mder_{a_{k-2}} f(x) \overline{\mder_{a_1} \dots \mder_{a_{k-2}}f(x - u)} \omega^{-\phi(a_{[k-2]}) \cdot u}\\
&\leq \exx_{a_{[2,k-2]}} \Big|\exx_{x,u} \mder_{a_2} \dots \mder_{a_{k-2}} f(x) \overline{\mder_{a_2} \dots \mder_{a_{k-2}} f(x - u)}\\
&\hspace{3cm} \exx_{a_1}  S(a_{[k-2]})  \overline{\mder_{a_2} \dots \mder_{a_{k-2}} f(x - a_1)} \mder_{a_2} \dots \mder_{a_{k-2}}f(x - a_1 - u) \omega^{-\phi(a_{[k-2]}) \cdot u}\Big|\\
&\leq \exx_{a_{[2,k-2]}} \bigg|\exx_{x,u} \Big|\exx_{a_1}  S(a_{[k-2]})  \overline{\mder_{a_2} \dots \mder_{a_{k-2}} f(x - a_1)} \mder_{a_2} \dots \mder_{a_{k-2}}f(x - a_1 - u) \omega^{-\phi(a_{[k-2]}) \cdot u}\Big|\bigg|.\end{align*}
For fixed $a_{[2,k-2]}$ and $u$, define maps 
\[g_{a_{[2,k-2]},u}(x) = \overline{\mder_{a_2} \dots \mder_{a_{k-2}} f( - x)} \mder_{a_2} \dots \mder_{a_{k-2}}f(- x - u)\]
and 
\[h_{a_{[2,k-2]},u}(x) = S(x,a_{[2, k-2]})\omega^{-\phi(x, a_{[2, k-2]}) \cdot u}.\]
By Lemma~\ref{l4bound}, we get
\[\bigg(\exx_{x}\Big|\exx_{a_1} g_{a_{[2,k-2]},u}(a_1 - x) h_{a_{[2,k-2]},u}(a_1)\Big|^2\bigg)^2 \leq \sum_r \Big|[h_{a_{[2,k-2]},u}\fcc(r)\Big|^4.\]
Returning to the inequalities above, we deduce that
\begin{align*}{\delta'}^4{\delta''}^8\leq &\bigg(\exx_{a_{[k-2]}} S(a_{[k-2]}) \Big|[\mder_{a_1} \dots \mder_{a_{k-2}} f\fcc(\phi(a_{[k-2]}))\Big|^2\bigg)^4\\
\leq &\exx_{a_{[2,k-2]}} \exx_u \sum_r \Big|\exx_{v} S(v, a_{[2, k-2]})\omega^{-\phi(v, a_{[2, k-2]}) \cdot u} \omega^{-r \cdot v}\Big|^4.\end{align*}
Some easy algebraic manipulation shows that the right-hand side expands to give $|G|^{-k}$ times the number of $d$-additive quadruples, which implies the lemma.\end{proof}

We also need another symmetry argument (closer to that of Green and Tao), which we present as a separate lemma.

\begin{lemma}[Symmetry argument]\label{symmArgumentLemma}Let $i_1, i_2 \in [k-1]$ be distinct elements and let $\psi \colon G_{[k-1]} \to \mathbb{F}_p$ be a multilinear form. Suppose that
\begin{equation}\label{singleForm} \Big|\exx_{a_{[k-1]} \in G^{k-1}} \exx_{x \in G} \mder_{a_1} \dots \mder_{a_{k-1}} f(x) \omega^{\psi(a_{[k-1]})}\Big| \geq c\end{equation}
for some $c > 0$. Then the multilinear map $\psi'$, defined by
\[\psi'(x_{[k-1]}) = \psi(x_{[k-1] \setminus \{i_1,i_2\}}, \ls{i_1}{x_{i_2}}, \ls{i_1}{x_{i_2}}),\]
that is by swapping coordinates $i_1$ and $i_2$, satisfies
\[\bias(\psi - \psi') \geq c^{O(1)}.\]
\end{lemma} 

We adopt a shorter notation for multiple derivatives, by letting $\mder_{d_{[l]}} f(x_0)$ stand for $\mder_{d_1} \dots \mder_{d_l} f(x_0)$.

\begin{proof}[Proof of Lemma~\ref{symmArgumentLemma}]For simplicity, we argue in the case $i_1 = k-2, i_2 = k-1$; the general case is analogous. By averaging, there is some $x_0 \in G$ such that
\begin{equation}\label{mainAssumptionForSymm}c \leq \Big|\exx_{a_{[k-1]} \in G^{k-1}} \mder_{a_1} \dots \mder_{a_{k-1}} f(x_0) \omega^{\psi(a_{[k-1]})}\Big|.\end{equation}

To make the expressions that follow clearer, we shall write $\chi$ for the one-dimensional character $\chi(\lambda) = \omega^{\lambda}$. Applying the Cauchy--Schwarz inequality, and making the change of variables $z=x_0 - u_{k-1} - v_{k-1} - a_{k-2}$ for the third equality below, we obtain that
\begin{align*}c^{2} &\leq \exx_{a_{[k-2]}} \Big|\exx_{a_{k-1}}  \mder_{a_{[k-2]}} f(x_0) \overline{\mder_{a_{[k-2]}} f(x_0 - a_{k-1})} \chi\Big(\psi(a_{[k-1]})\Big)\Big|^2\\
&=  \exx_{a_{[k-2]}} \exx_{u_{k-1}, v_{k-1}} \overline{\mder_{a_{[k-2]}} f(x_0 - u_{k-1})} \mder_{a_{[k-2]}} f(x_0 - v_{k-1}) \chi\Big(\psi(a_{[k-2]}, u_{k-1} - v_{k-1})\Big)\\
&=   \exx_{a_{[k-3]}} \exx_{a_{k-2}, u_{k-1}, v_{k-1}} \mder_{a_{[k-3]}} f(x_0 - u_{k-1} - a_{k-2}) \overline{\mder_{a_{[k-3]}} f(x_0 - u_{k-1})} \\
&\hspace{2cm} \overline{\mder_{a_{[k-3]}} f(x_0 - v_{k-1} - a_{k-2})} \mder_{a_{[k-3]}} f(x_0 - v_{k-1}) \chi\Big(\psi(a_{[k-2]}, u_{k-1} - v_{k-1})\Big)\\
&=   \exx_{a_{[k-3]}} \exx_{z, u_{k-1}, v_{k-1}} \mder_{a_{[k-3]}} f(z + v_{k-1}) \overline{\mder_{a_{[k-3]}} f(x_0 - u_{k-1})} \hspace{1pt} \overline{\mder_{a_{[k-3]}} f(z + u_{k-1})} \mder_{a_{[k-3]}} f(x_0 - v_{k-1})\\
&\hspace{2cm} \chi\Big(\psi(a_{[k-3]}, x_0 - u_{k-1} - v_{k-1} - z, u_{k-1} - v_{k-1})\Big)\\
&=   \exx_{a_{[k-3]}} \exx_{z, u_{k-1}, v_{k-1}}\mder_{a_{[k-3]}} f(z + v_{k-1}) \overline{\mder_{a_{[k-3]}} f(x_0 - u_{k-1})} \hspace{1pt} \overline{\mder_{a_{[k-3]}} f(z + u_{k-1})} \mder_{a_{[k-3]}} f(x_0 - v_{k-1})\\
&\hspace{2cm} \chi\Big(\psi(a_{[k-3]}, x_0 - z, u_{k-1})\Big) \overline{\chi\Big(\psi(a_{[k-3]}, x_0 - z, v_{k-1})\Big)}\\
&\hspace{2cm}\chi\Big(\psi(a_{[k-3]}, v_{k-1}, v_{k-1})\Big) \overline{\chi\Big(\psi(a_{[k-3]}, u_{k-1}, u_{k-1})\Big)}\\
&\hspace{2cm}\chi\Big(\psi(a_{[k-3]}, u_{k-1} , v_{k-1}) - \psi(a_{[k-3]}, v_{k-1} , u_{k-1})\Big).
\end{align*}

Applying the Cauchy--Schwarz inequality once again, we get
\begin{align*}c^{4} &\leq  \exx_{a_{[k-3]}, z, u_{k-1}} \Big|\exx_{v_{k-1}}  \mder_{a_{[k-3]}} f(z + v_{k-1}) \overline{\mder_{a_{[k-3]}} f(x_0 - u_{k-1})} \hspace{1pt} \overline{\mder_{a_{[k-3]}} f(z + u_{k-1})} \mder_{a_{[k-3]}} f(x_0 - v_{k-1})\\
&\hspace{2cm} \chi\Big(\psi(a_{[k-3]}, x_0 - z, u_{k-1})\Big) \overline{\chi\Big(\psi(a_{[k-3]}, x_0 - z, v_{k-1})\Big)}\\
&\hspace{2cm}\chi\Big(\psi(a_{[k-3]}, v_{k-1}, v_{k-1})\Big) \overline{\chi\Big(\psi(a_{[k-3]}, u_{k-1}, u_{k-1})\Big)}\\
&\hspace{2cm}\chi\Big(\psi(a_{[k-3]}, u_{k-1}, v_{k-1}) - \psi(a_{[k-3]}, v_{k-1}, u_{k-1})\Big)\Big|^2\\
&\leq \exx_{a_{[k-3]}, z, u_{k-1}} \Big|\exx_{v_{k-1}}  \mder_{a_{[k-3]}} f(z + v_{k-1}) \mder_{a_{[k-3]}} f(x_0 - v_{k-1})\\
&\hspace{2cm} \overline{\chi\Big(\psi(a_{[k-3]}, x_0 - z, v_{k-1})\Big)} \chi\Big(\psi(a_{[k-3]}, v_{k-1}, v_{k-1})\Big)\\
&\hspace{2cm}\chi\Big(\psi(a_{[k-3]}, u_{k-1}, v_{k-1}) - \psi(a_{[k-3]}, v_{k-1}, u_{k-1})\Big)\Big|^2\\
&=\exx_{a_{[k-3]}, z, u_{k-1}, v_{k-1}, v'_{k-1}} \mder_{a_{[k-3]}} f(z + v_{k-1}) \mder_{a_{[k-3]}} f(x_0 - v_{k-1}) \overline{\mder_{a_{[k-3]}} f(z + v'_{k-1})}\,\overline{\mder_{a_{[k-3]}} f(x_0 - v'_{k-1})}\\
&\hspace{2cm}\chi\Big(\psi(a_{[k-3]}, x_0 - z, v'_{k-1} - v_{k-1})\Big)\\
&\hspace{2cm} \chi\Big(\psi(a_{[k-3]}, v_{k-1}, v_{k-1})\Big) \overline{\chi\Big(\psi(a_{[k-3]}, v'_{k-1}, v'_{k-1})\Big)}\\
&\hspace{2cm}\chi\Big(\psi(a_{[k-3]}, u_{k-1}, v_{k-1} - v'_{k-1})- \psi(a_{[k-3]}, v_{k-1} - v'_{k-1}, u_{k-1})\Big).\end{align*}

Applying the Cauchy--Schwarz inequality one last time,
\begin{align*}c^{8} &\leq  \exx_{a_{[k-3]}, z, v_{k-1}, v'_{k-1}}\bigg|\exx_{u_{k-1}} \mder_{a_{[k-3]}} f(z + v_{k-1}) \mder_{a_{[k-3]}} f(x_0 - v_{k-1}) \overline{\mder_{a_{[k-3]}} f(z + v'_{k-1})}\,\overline{\mder_{a_{[k-3]}} f(x_0 - v'_{k-1})}\\
&\hspace{2cm}\chi\Big(\psi(a_{[k-3]}, z - x_0, v'_{k-1} - v_{k-1})\Big)\\
&\hspace{2cm} \chi\Big(\psi(a_{[k-3]}, v_{k-1}, v_{k-1})\Big) \overline{\chi\Big(\psi(a_{[k-3]}, v'_{k-1}, v'_{k-1})\Big)}\\
&\hspace{2cm}\chi\Big(\psi(a_{[k-3]}, u_{k-1} , v_{k-1} - v'_{k-1}) - \psi(a_{[k-3]}, v_{k-1} - v'_{k-1}, u_{k-1})\Big)\bigg|^2\\
&\leq\exx_{a_{[k-3]}, z, v_{k-1}, v'_{k-1}}\Big|\exx_{u_{k-1}}  \chi\Big(\psi(a_{[k-3]}, u_{k-1} , v_{k-1} - v'_{k-1})- \psi(a_{[k-3]}, v_{k-1} - v'_{k-1}, u_{k-1})\Big) \Big|^2\\
&=\exx_{a_{[k-3]}, z, v_{k-1}, v'_{k-1}, u_{k-1}, u'_{k-1}} \chi\Big(\psi(a_{[k-3]}, u_{k-1} - u'_{k-1}, v_{k-1} - v'_{k-1}) - \psi(a_{[k-3]}, v_{k-1} - v'_{k-1}, u_{k-1} - u'_{k-1})\Big)\\
&=\exx_{a_{[k-1]}} \chi\Big(\psi(a_{[k-3]}, a_{k-2}, a_{k-1}) - \psi(a_{[k-3]}, a_{k-1}, a_{k-2})\Big),\end{align*}
which completes the proof of the lemma.\end{proof}

For $\pi \in \operatorname{Sym}_{[k-1]}$, write $\psi_\pi(a_{[k-1]})$ for $\psi(a^\pi_{[k-1]})$, where $a^\pi_d=a_{\pi(d)}$ for each $d\in [k-1]$. We claim that each $\psi_\pi$ differs from $\psi$ by a multilinear form of small partition rank.

\begin{corollary}\label{symmArgumentCor}let $\psi \colon G_{[k-1]} \to \mathbb{F}_p$ be a multilinear form. Suppose that
\[\Big|\exx_{a_{[k-1]} \in G^{k-1}} \exx_{x \in G} \mder_{a_1} \dots \mder_{a_{k-1}} f(x) \omega^{\psi(a_{[k-1]})}\Big| \geq c\]
for some $c > 0$. Then for each $\pi \in \on{Sym}_{[k-1]}$ we have $\on{prank} (\psi - \psi_\pi) \leq  O\Big((\log_p c^{-1})^{O(1)}\Big)$.\end{corollary}

\begin{proof}Recall that every permutation is a composition of a bounded number of transpositions. We use induction on that number, the base case being $\pi = \on{id}$, which is trivial. Suppose now that $\pi \in \on{Sym}_{[k-1]}$ is given and that $\tau \in \on{Sym}_{[k-1]}$ is a transposition such that the claim holds for $\pi' = \tau \circ \pi$. Then there are a positive integer $r \leq  O\Big((\log_p c^{-1})^{O(1)}\Big)$, non-empty proper subsets $J_i \subset [k-1]$ and multilinear forms $\alpha_i \colon G^{J_i} \to \mathbb{F}_p$ and $\beta_i \colon G^{[k-1] \setminus J_i}$ for $i \in [r]$ such that
\[\psi(a_{[k-1]}) - \psi_{\pi'}(a_{[k-1]}) = \sum_{i \in [r]} \alpha_i(a_{J_i}) \beta_i(a_{[k-1] \setminus J_i}).\]
Going back to the assumption in the statement of the corollary, we conclude that
\begin{align*}c &\leq \Big|\exx_{a_{[k-1]} \in G^{k-1}} \exx_{x \in G} \mder_{a_1} \dots \mder_{a_{k-1}} f(x) \omega^{\psi(a_{[k-1]})}\Big|\\
&=  \Big|\exx_{a_{[k-1]} \in G^{k-1}} \exx_{x \in G} \mder_{a_1} \dots \mder_{a_{k-1}} f(x) \omega^{\psi_{\pi'}(a_{[k-1]}) + \sum_{i \in [r]} \alpha_i(a_{J_i}) \beta_i(a_{[k-1] \setminus J_i})} \Big|\\
&\leq  \sum_{\nu, \nu' \in \mathbb{F}_p^r} \Big|\exx_{a_{[k-1]} \in G^{k-1}} \exx_{x \in G} \mder_{a_1} \dots \mder_{a_{k-1}} f(x) \omega^{\psi_{\pi'}(a_{[k-1]})} \Big(\prod_{i \in [r]} \mathbbm{1}(\alpha_i(a_{J_i}) = \nu_i)\Big) \Big(\prod_{i \in [r]} \mathbbm{1}(\beta_i(a_{[k-1] \setminus J_i}) = \nu'_i)\Big)\Big|.\end{align*}
By averaging, there is a choice of $\nu, \nu' \in \mathbb{F}_p^r$ such that
\begin{align*}cp^{-2r} &\leq \Big|\exx_{a_{[k-1]} \in G^{k-1}} \exx_{x \in G} \mder_{a_1} \dots \mder_{a_{k-1}} f(x) \omega^{\psi_{\pi'}(a_{[k-1]})} \Big(\prod_{i \in [r]} \mathbbm{1}(\alpha_i(a_{J_i}) = \nu_i)\Big) \Big(\prod_{i \in [r]} \mathbbm{1}(\beta_i(a_{[k-1] \setminus J_i}) = \nu'_i)\Big)\Big|\\
&\leq  \exx_{x \in G} \Big|\exx_{a_{[k-1]} \in G^{k-1}}\mder_{a_1} \dots \mder_{a_{k-1}} f(x) \omega^{\psi_{\pi'}(a_{[k-1]})} \Big(\prod_{i \in [r]} \mathbbm{1}(\alpha_i(a_{J_i}) = \nu_i)\Big) \Big(\prod_{i \in [r]} \mathbbm{1}(\beta_i(a_{[k-1] \setminus J_i}) = \nu'_i)\Big)\Big|.\end{align*}

Define an auxiliary map $\tilde{f}_x \colon G^{k-1} \to \mathbb{D}$ by $\tilde{f}_x(a_{[k-1]}) = f(x - a_1 - \dots - a_{k-1}) \omega^{\psi_{\pi'}(a_{[k-1]})}$. Using Lemma~\ref{unifBound2} we see that 
\[c^{O(1)}p^{-O(r)}\leq \exx_x \|\tilde{f}_x\|_{\square^{k-1}}^{2^{k-1}} = \exx_{x, a_{[k-1]}}  \mder_{a_1} \dots \mder_{a_{k-1}} f(x) \omega^{\psi_{\pi'}(a_{[k-1]})}\]
after expansion and an easy algebraic manipulation. Lemma~\ref{symmArgumentLemma} implies that 
\[\bias (\psi_\pi - \psi_{\tau \circ \pi}) \geq c^{O(1)}p^{-O(r)}.\]
By Theorem~\ref{strongInvARankThm}, we get that $\prank (\psi_\pi - \psi_{\tau \circ \pi}) \leq  O\Big((\log_p c^{-1})^{O(1)}\Big)$, finishing the proof of the claim.\end{proof}

We are now ready to prove Theorem~\ref{UkInverseTheorem}.

\begin{proof}[Proof of Theorem~\ref{UkInverseTheorem}] We prove the result by induction on $k$. Recall that for any map $h \colon G \to \mathbb{D}$, we have the bound
\[\|h\|_{U^2}^4 = \sum_r |\hat{h}(r)|^4 \leq \Big(\sum_r |\hat{h}(r)|^2\Big) \max_r |\hat{h}(r)|^2 \leq \max_r |\hat{h}(r)|^2.\]
It follows that
\[\delta^{2^k} \leq \|f\|_{U^{k}}^{2^k} = \exx_{a_1, \dots, a_{k-2}} \Big\|\mder_{a_1} \dots \mder_{a_{k-2}} f\Big\|_{U^2}^4 \leq \exx_{a_1, \dots, a_{k-2}} \max_r |[\mder_{a_1} \dots \mder_{a_{k-2}} f\fcc(r)|^2.\]
Let $\phi \colon G^{k-2} \to G$ be a map defined by taking $\phi(a_{[k-2]})$ to be any $r \in G$ such that $|[\mder_{a_1} \dots \mder_{a_{k-2}} f\fcc(r)|$ is maximal. This gives us a set $A \subset G_{[k-2]}$ of size at least $\frac{\delta^{2^k}}{2}|G_{k-2}|$ such that
\[ \Big|[\mder_{a_1} \dots \mder_{a_{k-2}} f\fcc(\phi(a_{[k-2]}))\Big| \geq \delta^{2^{k-1}}/2\]
for every $a_{[k-2]} \in A$.

Applying Theorem~\ref{FreBSG} and Lemma~\ref{multiaffineUkPassClaim} in each direction, and then Theorem~\ref{multiaffineInvThm}, we obtain a global multiaffine map $\Phi \colon G^{k-2} \to G$ and a set $A' \subset G_{[k-2]}$ of size $\delta_1|G_{[k-2]}|$, where 
\[\delta_1 =\Big(\exp^{(O(1))} (O(\delta^{-1}))\Big)^{-1}\]
such that for every $a_{[k-2]} \in A'$ we have the bound
\[ \Big|[\mder_{a_1} \dots \mder_{a_{k-2}} f\fcc(\Phi(a_{[k-2]}))\Big| \geq \delta^{2^{k-1}}/2.\]
Therefore,
\[\exx_{a_{[k-1]} \in G^{k-1}} \exx_{x \in G} \mder_{a_1} \dots \mder_{a_{k-1}} f(x) \omega^{-\Phi(a_{[k-2]}) \cdot a_{k-1}}= \exx_{a_{[k-2]} \in G_{[k-2]}} \Big|[\mder_{a_1} \dots \mder_{a_{k-2}} f\fcc(\Phi(a_{[k-2]}))\Big|^2= \delta_2,\]
where $\delta_2 \geq \Omega(\delta_1 \delta^{O(1)})$.\\

Write $\phi(a_{[k-1]}) = -\Phi(a_{[k-2]}) \cdot a_{k-1}$, which is a multiaffine form. We may further rewrite $\phi$ as a sum $\phi^{\text{ml}} + \phi_1 + \dots + \phi_{k-1}$, where $\phi^{\text{ml}}$ is multilinear, and $\phi_i$ is multiaffine form that does not depend on $a_i$. Then
\begin{align*}\delta_2 \leq &\exx_{a_{[k-1]}} \exx_{x} \mder_{a_1} \dots \mder_{a_{k-1}} f(x) \chi(\phi^{\text{ml}}(a_{[k-1]})) \prod_{i \in [k-1]} \chi(\phi_i(a_{[k-1]\setminus\{i\}}))\\
\leq & \exx_x \Big|\exx_{a_{[k-1]}} f(x - a_1 - \dots - a_{k-1}) \chi(\phi^{\text{ml}}(a_{[k-1]})) \prod_{I \subsetneq [k-1]} \on{Conj}^{|I|} f\Big(x - \sum_{i \in I} a_i\Big)\, \prod_{i \in [k-1]} \chi(\phi_i(a_{[k-1]\setminus\{i\}})\Big|.\end{align*}
Define an auxiliary map $\tilde{f}_x \colon G^{k-1} \to \mathbb{D}$ by $\tilde{f}_x(a_{[k-1]}) = f(x - a_1 - \dots - a_{k-1}) \chi(\phi^{\text{ml}}(a_{[k-1]}))$. Using Lemma~\ref{unifBound2} we see that 
\begin{equation}\delta^{2^{k-1}}_2 \leq \exx_x \|\tilde{f}_x\|_{\square^{k-1}}^{2^{k-1}} = \exx_{x, a_{[k-1]}}  \mder_{a_1} \dots \mder_{a_{k-1}} f(x) \chi(\phi^{\text{ml}}(a_{[k-1]}))\label{ukmultilinearformbnd}\end{equation}
after expansion and an easy algebraic manipulation.\\

Apply Corollary~\ref{symmArgumentCor} to find a symmetric multilinear form such that $m' = \prank (\phi^{\text{ml}} - \sigma) \leq O\Big((\log_p \delta_2^{-1})^{O(1)}\Big)$. Then there are proper non-empty subsets $J_i \subsetneq [k-1]$, and multilinear forms $\alpha_i \colon G^{J_i} \to \mathbb{F}_p$ and $\beta_i \colon G^{[k-1] \setminus J_i} \to \mathbb{F}_p$ for each $i \in [m']$, such that
\[\phi^{\text{ml}}(a_{[k-1]}) = \sigma(a_{[k-1]}) + \sum_{i \in [m']} \alpha_i(a_{J_i}) \beta_i(a_{[k-1] \setminus J_i}).\]
Return to~\eqref{ukmultilinearformbnd} to obtain
\begin{align*}\delta_2^{O(1)} \leq & \exx_{x, a_{[k-1]}}  \mder_{a_1} \dots \mder_{a_{k-1}} f(x) \chi(\sigma(a_{[k-1]})) \chi\Big(\sum_{i \in [m']} \alpha_i(a_{J_i}) \beta_i(a_{[k-1] \setminus J_i})\Big)\\
= &\exx_{x, a_{[k-1]}} \sum_{\nu, \nu' \in \mathbb{F}_p^{m'}} \mder_{a_1} \dots \mder_{a_{k-1}} f(x) \chi(\sigma(a_{[k-1]})) \chi(\nu \cdot \nu') \prod_{i \in [m']} \mathbbm{1}(\alpha_i(a_{J_i}) = \nu_i)\,\mathbbm{1}(\beta_i(a_{[k-1] \setminus J_i}) = \nu'_i).\end{align*}
By averaging, there are $\nu, \nu' \in \mathbb{F}_p^{m'}$ such that 
\[\delta_2^{O(1)}p^{-2m'} \leq  \Big|\exx_{x, a_{[k-1]}}\mder_{a_1} \dots \mder_{a_{k-1}} f(x) \chi(\sigma(a_{[k-1]})) \prod_{i \in [m']} \mathbbm{1}(\alpha_i(a_{J_i}) = \nu_i)\,\mathbbm{1}(\beta_i(a_{[k-1] \setminus J_i}) = \nu'_i)\Big|.\]
Using a very similar argument to the one above, we apply Lemma~\ref{unifBound2} to conclude that
\[\delta_2^{O(1)}p^{-O(m')} \leq \exx_{x, a_{[k-1]}}\mder_{a_1} \dots \mder_{a_{k-1}} f(x) \chi(\sigma(a_{[k-1]})).\]
Since $p \geq k$, we may use the polarization identity to find a polynomial $g \colon G \to \mathbb{F}_p$ of degree at most $k-1$ such that $\sigma(x_{[k-1]}) = \Delta_{x_1} \dots \Delta_{x_{k-1}} g(y)$ (namely $g(x) = \frac{1}{(k-1)!} \sigma(x, \dots, x)$). Set $f'(x) = f(x) \chi\Big(g(x)\Big)$ to obtain the identity
\[\mder_{a_1} \dots \mder_{a_{k-1}} f(x) \chi\Big(\sigma(a_{[k-1]})\Big) = \mder_{a_1} \dots \mder_{a_{k-1}} f'(x).\]
This implies that 
\[\delta_2^{O(1)}p^{-O(m')} \leq \|f'\|_{U^{k-1}}\]
and we may use the induction hypothesis to complete the proof of the theorem.\end{proof}

\subsection{A multiaffine Bogolyubov argument}

We now strengthen Theorem~\ref{strongMixedApprox}. The main theorem of this subsection tells us, roughly speaking, that if we apply enough convolutions in each direction to a bounded function $f$, then the resulting function will be approximately constant (in an $L_\infty$ sense) on almost all the level sets of a multiaffine map to $\mathbb F_p^l$, where $l$ is bounded. (In the case $k=1$, the ``almost all" makes this statement weaker than Bogolyubov's lemma, but as we remarked in \cite{U4paper}, one cannot obtain uniform approximations on all level sets when $k\geq 2$.)

\begin{theorem}\label{AlmostLinftyApproxThm}Let $f \colon G_{[k]} \to \mathbb{D}$, let $d_1, \dots, d_r \in [k]$ directions such that $\{d_1, \dots, d_r\} = [k]$, and let $\varepsilon > 0$. Then there exist
\begin{itemize}
\item a positive integer $l = \exp^{(O(1))}\big(2^{O(r)}\varepsilon^{-O(1)}\big)$,
\item a multiaffine map $\phi \colon G_{[k]} \to \mathbb{F}^l_p$,
\item a set of values $M \subset \mathbb{F}_p^l$, such that $|\phi^{-1}(M)| \geq (1-\varepsilon) |G_{[k]}|$,
\item a map $c \colon M \to \mathbb{D}$
\end{itemize}
such that
\[\Big|\bigconv{d_r} \dots \bigconv{d_1}\bigconv{k} \dots \bigconv{1} f(x_{[k]}) - c(\phi(x_{[k]}))\Big|\leq \varepsilon\]
for every $x_{[k]}\in\phi^{-1}(M)$.
\end{theorem}

We will be almost done once we have proved a closely related statement in the case where $f(x_{[k]})$ is itself of the form $c(\phi(x_{[k]}))$. The next lemma shows such functions remain approximately of the same form if we convolve in several directions.

\begin{lemma}\label{linftyforMLmaps}Let $l \in \mathbb{N}$, let $\alpha \colon G_{[k]} \to \mathbb{F}^l_p$ be a multiaffine map, let $c \colon \mathbb{F}_p^l \to \mathbb{D}$ and let $f=c\circ\alpha$. Let $d_1, \dots, d_r \in [k]$ be directions and let $\varepsilon > 0$. Then there exist positive integers $l', s = O(l + \log_p \varepsilon^{-1})^{O(1)}$, multiaffine maps $\beta \colon G_{[k]} \to \mathbb{F}_p^s$, $\alpha' \colon G_{[k]} \to \mathbb{F}_p^{l'}$, a map $c'\colon \mathbb{F}_p^{l'} \to \mathbb{D}$, and a collection of values $B \subset \mathbb{F}_p^s$ such that 
\[|\beta^{-1}(B)| \geq (1-\varepsilon) |G_{[k]}|\]
and
\[\Big|\bigconv{d_r} \dots\bigconv{d_1} f(x_{[k]}) -  c'(\alpha'(x_{[k]}))\Big| \leq \varepsilon\]
for every $x_{[k]}\in\beta^{-1}(B)$.
\end{lemma}

\begin{proof}We prove the statement by induction on $r$ and include the trivial $r = 0$ case as the base case. Suppose that the statement has been proved for some $r$. Apply it with a parameter $\varepsilon' > 0$, and let $l',s, \alpha', \beta, c', B$ be the relevant objects. Let $g \colon G_{[k]} \to \mathbb{D}$ be given by $g(x_{[k]}) = c'(\alpha'(x_{[k]}))$. Apply Theorem~\ref{simFibresThm} to $\beta$ in direction $G_{d_{r+1}}$ with parameter $\eta_1$. We obtain a positive integer $t = O(s + \log_p \eta_1^{-1})^{O(1)}$, a multiaffine map $\rho \colon G_{[k] \setminus \{d_{r+1}\}} \to \mathbb{F}_p^t$, a collection of values $R \subset \mathbb{F}_p^t$ and a map $v \colon R \times \mathbb{F}_p^s \to \mathbb{D}$ such that
\[|\rho^{-1}(R)| \geq (1- \eta_1) |G_{[k] \setminus \{d_{r+1}\}}|\]
and
\[\big|\big\{y_{d_{r+1}} \in G_{d_{r+1}} \colon \beta(x_{[k] \setminus \{d_{r+1}\}}, y_{d_{r+1}}) = \lambda\big\}\big| = v(\rho(x_{[k] \setminus \{d_{r+1}\}}), \lambda)\]
for every $x_{[k] \setminus \{d_{r+1}\}} \in \rho^{-1}(R)$ and every $\lambda \in \mathbb{F}_p^s$.

Note that for each $\mu \in R$, the size $|\{y_{d_{r+1}} \in G_{d_{r+1}} \colon \beta(x_{[k] \setminus \{d_{r+1}\}}, y_{d_{r+1}}) \notin B\}|$ is the same for every $x_{[k] \setminus \{d_{r+1}\}} \in \rho^{-1}(\mu)$. Let $R'$ be the set of all $\mu\in R$ for which this size is at most $\frac{\varepsilon}{100}|G_{d_{r+1}}|$. We have
\[\varepsilon'|G_{[k]}| \,\,\geq \,|\beta^{-1}(\mathbb{F}_p^s \setminus B)|\,\, \geq\,  |\rho^{-1}(R \setminus R')| \cdot \frac{\varepsilon}{100}|G_{d_{r+1}}|\]
and hence
\[|\rho^{-1}(R')|\ \geq (1- \eta_1 - 100 \varepsilon^{-1} \varepsilon') |G_{[k] \setminus \{d_{r+1}\}}|.\]
Thus, whenever $\rho(x_{[k] \setminus \{d_{r+1}\}})\in R'$, we have
\begin{equation}\Big|\bigconv{d_{r+1}} \dots \bigconv{d_1} f(x_{[k] \setminus \{d_{r+1}\}}, y_{d_{r+1}}) - \bigconv{d_{r+1}} g(x_{[k] \setminus \{d_{r+1}\}}, y_{d_{r+1}})\Big| \leq 2\varepsilon' + \frac{\varepsilon}{10}.\label{genbogapproxeqn1}\end{equation}
It remains to approximate $\bigconv{d_{r+1}} g$. To this end, let $\alpha'(x_{[k]}) = \alpha''(x_{[k]}) + \gamma(x_{[k] \setminus \{d_{r+1}\}})$ for multiaffine maps $\alpha'' \colon G_{[k]} \to \mathbb{F}_p^{l'}$ and $\gamma \colon G_{[k] \setminus \{d_{r+1}\}} \to \mathbb{F}_p^{l'}$, where $\alpha''$ is additionally linear in coordinate $d_{r+1}$. Then, for each $x_{[k] \setminus \{d_{r+1}\}}$, the map $y_{d_{r+1}} \mapsto \bigconv{d_{r+1}} g(x_{[k] \setminus \{d_{r+1}\}}, y_{d_{r+1}})$ is constant on cosets of $\{y_{d_{r+1}} \in G_{d_{r+1}}\colon \alpha''(x_{[k] \setminus \{d_{r+1}\}}, y_{d_{r+1}}) = 0\}$.  Apply Theorem~\ref{simFibresThm} to $\alpha''$ in direction $G_{d_{r+1}}$ with parameter $\eta_2$. We obtain a positive integer $t' = O(l' + \log_p \eta_2^{-1})^{O(1)}$,  a multiaffine map $\rho' \colon G_{[k] \setminus \{d_{r+1}\}} \to \mathbb{F}_p^{t'}$, a collection of values $T \subset \mathbb{F}_p^{t'}$, and a map $u \colon T \times \mathbb{F}_p^{l'} \to \mathbb{D}$, such that
\[|(\rho')^{-1}(T)| \geq (1- \eta_2) |G_{[k] \setminus \{d_{r+1}\}}|\]
and
\[|\{y_{d_{r+1}} \in G_{d_{r+1}} \colon \alpha''(x_{[k] \setminus \{d_{r+1}\}}, y_{d_{r+1}}) = \mu\}| = u(\rho'(x_{[k] \setminus \{d_{r+1}\}}), \mu)\]
for every $x_{[k] \setminus \{d_{r+1}\}} \in(\rho')^{-1}(T)$ and every $\mu \in \mathbb{F}_p^{l'}$. We now use this to approximate $\bigconv{d_{r+1}} g$.\\
Let $\lambda' \in T$ be fixed. Then, for each $x_{[k] \setminus \{d_{r+1}\}} \in(\rho')^{-1}(\lambda')$ we have that the vector space $\Gamma = \{\alpha''(x_{[k] \setminus \{d_{r+1}\}}, y_{d_{r+1}}) : y_{d_{r+1}} \in G_{d_{r+1}}\}$ is the same. Hence,
\begin{align*}\bigconv{d_{r+1}} g(x_{[k]}) = &\exx_{y_{d_{r+1}}}  g(x_{[k] \setminus \{d_{r+1}\}}, x_{d_{r+1}} + y_{d_{r+1}}) \overline{g(x_{[k] \setminus \{d_{r+1}\}}, y_{d_{r+1}})}\\
= &\exx_{y_{d_{r+1}}}  c'\Big(\alpha'(x_{[k] \setminus \{d_{r+1}\}}, x_{d_{r+1}} + y_{d_{r+1}})\Big) \overline{c'\Big(\alpha'(x_{[k] \setminus \{d_{r+1}\}}, y_{d_{r+1}})\Big)}\\
= &\exx_{y_{d_{r+1}}}  c'\Big(\alpha''(x_{[k] \setminus \{d_{r+1}\}}, x_{d_{r+1}} + y_{d_{r+1}}) + \gamma(x_{[k] \setminus \{d_{r+1}\}})\Big) \overline{c'\Big(\alpha''(x_{[k] \setminus \{d_{r+1}\}}, y_{d_{r+1}}) + \gamma(x_{[k] \setminus \{d_{r+1}\}})\Big)}\\
= &\exx_{y_{d_{r+1}}}  c'\Big(\alpha''(x_{[k]}) + \alpha''(x_{[k] \setminus \{d_{r+1}\}}, y_{d_{r+1}}) + \gamma(x_{[k] \setminus \{d_{r+1}\}})\Big) \overline{c'\Big(\alpha''(x_{[k] \setminus \{d_{r+1}\}}, y_{d_{r+1}}) + \gamma(x_{[k] \setminus \{d_{r+1}\}})\Big)}\\
= & \exx_{\mu \in \Gamma} c'\Big(\alpha''(x_{[k]}) + \mu + \gamma(x_{[k] \setminus \{d_{r+1}\}})\Big) \overline{c'\Big(\mu+ \gamma(x_{[k] \setminus \{d_{r+1}\}})\Big)},\end{align*}
which is a quantity that depends only on $\alpha''(x_{[k]})$ and $\gamma(x_{[k] \setminus \{d_{r+1}\}})$. Recalling that we considered only $x_{[k] \setminus \{d_{r+1}\}} \in(\rho')^{-1}(\lambda')$ for fixed $\lambda'$, we conclude that for every $\lambda' \in T$, $\mu, \mu' \in \mathbb{F}_p^{l'}$, we may find a value $w(\lambda', \mu, \mu') \in \mathbb{D}$ such that
\begin{equation}\label{genbogapproxeqn2}\bigconv{d_{r+1}} g(x_{[k]})  = w(\lambda', \mu, \mu')\end{equation}
for all $x_{[k]} \in G_{[k]}$ such that $\rho'(x_{[k] \setminus \{d_{r+1}\}}) = \lambda', \alpha''(x_{[k]}) = \mu$, and $\gamma(x_{[k] \setminus \{d_{r+1}\}}) = \mu'$.\\

Combining~\eqref{genbogapproxeqn1} and~\eqref{genbogapproxeqn2}, for every choice of $\lambda \in R', \lambda' \in T, \mu, \mu' \in \mathbb{F}_p^{l'}$, we obtain a single value $w \in \mathbb{D}$ such that for each $x_{[k]} \in G_{[k]}$ such that $\rho(x_{[k] \setminus \{d_{r+1}\}}) = \lambda, \rho'(x_{[k] \setminus \{d_{r+1}\}}) = \lambda', \alpha''(x_{[k]}) = \mu$, and $\gamma(x_{[k] \setminus \{d_{r+1}\}}) = \mu'$, 
\[\Big|\bigconv{d_{r+1}} \dots \bigconv{d_1} f(x_{[k]}) - w\Big| \leq 2\varepsilon' + \frac{\varepsilon}{10}.\]
Pick $\varepsilon' = \frac{\varepsilon^2}{1000}, \eta_1 = \frac{\varepsilon}{100}, \eta_2 = \frac{\varepsilon}{100}$ to finish the proof.\end{proof}

\begin{proof}[Proof of Theorem \ref{AlmostLinftyApproxThm}] To reduce the theorem to Lemma~\ref{linftyforMLmaps}, we first apply Theorem~\ref{MixedConvApprox} for the $L^1$ norm with approximation parameter $2^{-r-1}\varepsilon$. That yields
\begin{itemize}
\item a positive integer $l^{(1)} =\exp^{(O(1))}\Big(O(2^{O(r)}\varepsilon^{-O(1)})\Big)$, 
\item constants $c^{(1)}_1, \dots, c^{(1)}_{l^{(1)}} \in \mathbb{D}$, and
\item multiaffine forms $\phi^{(1)}_1, \dots, \phi^{(1)}_{l^{(1)}} \colon G_{[k]} \to \mathbb{F}_p$ such that
\end{itemize}
\[\bigconv{k}\dots\bigconv{1}f \apps{2^{-r-1}\varepsilon}_{L^1} \sum_{i \in [l^{(1)}]} c^{(1)}_i \chi\circ\phi^{(1)}_i.\]
Define $g \colon G_{[k]} \to \mathbb{C}$ as follows. For given $x_{[k]}$, let $\sigma = \sum_{i \in [l^{(1)}]} c^{(1)}_i \chi(\phi_i(x_{[k]})) $. Set
\[g(x_{[k]}) = \begin{cases}\sigma,&\hspace{1cm}\text{when }|\sigma| \leq 1\\
\frac{\sigma}{|\sigma|},&\hspace{1cm}\text{when }|\sigma| > 1.\end{cases}\]
Notice that $g$ is constant on layers of $\phi$ and that $\bigconv{k} \dots \bigconv{1} f \apps{2^{-r-1}\varepsilon}_{L^1} g$.\\

Now apply Lemma~\ref{LinftyConvApprox} to obtain an approximation in the $L^\infty$ norm, though with a slightly larger approximation parameter:
\[\Big\|\bigconv{d_r} \dots \bigconv{d_1}\bigconv{k} \dots \bigconv{1} f - \bigconv{d_r} \dots \bigconv{d_1}g\Big\|_{L^{\infty}} \leq \frac{\varepsilon}{2}.\]
Finally, we may apply Lemma~\ref{linftyforMLmaps} to approximate $\bigconv{d_r} \dots \bigconv{d_1}g$ with error parameter $\varepsilon/2$. We obtain positive integers $l, s = O(l^{(1)} + \log_p \varepsilon^{-1})^{O(1)}$, multiaffine maps $\beta \colon G_{[k]} \to \mathbb{F}_p^s$, $\alpha \colon G_{[k]} \to \mathbb{F}_p^{l}$, a map $c\colon \mathbb{F}_p^{l} \to \mathbb{D}$, and a collection of values $B \subset \mathbb{F}_p^s$ such that 
\[|\beta^{-1}(B)| \geq (1-\varepsilon) |G_{[k]}|\]
and
\[\Big|\bigconv{d_r} \dots\bigconv{d_1} g(x_{[k]}) -  c(\alpha(x_{[k]}))\Big| \leq \varepsilon/2\]
for every $x_{[k]}\in\beta^{-1}(B)$.\\
\indent It follows that 
\begin{align*}&\Big|\bigconv{d_r} \dots \bigconv{d_1}\bigconv{k} \dots \bigconv{1} f(x_{[k]}) - c(\alpha(x_{[k]}))\Big|\\
&\hspace{2cm}\leq \Big|\bigconv{d_r} \dots \bigconv{d_1}\bigconv{k} \dots \bigconv{1} f(x_{[k]}) - \bigconv{d_r} \dots \bigconv{d_1}g(x_{[k]})\Big| + \Big|\bigconv{d_r} \dots\bigconv{d_1} g(x_{[k]}) -  c(\alpha(x_{[k]}))\Big|\\
&\hspace{2cm} \leq \varepsilon\end{align*}
for each $x_{[k]}\in\beta^{-1}(B)$, which completes the proof.\end{proof}

We are now ready to prove Theorem~\ref{alldirsubspacestheorem} about sets $X \subset G_{[k]}$ that are subspaces in each principal direction.

\begin{proof}[Proof of Theorem~\ref{alldirsubspacestheorem}]First of all, notice that since $X$ is a subspace in each principal direction, we have that $\on{supp} \bigconv{d} X \subset X$ for any $d \in [k]$. By the Cauchy--Schwarz inequality we have that
\[\exx_{x_{[k]}} \bigconv{k} \dots \bigconv{1}\bigconv{k} \dots \bigconv{1} X(x_{[k]}) \geq \delta^{2^{2k}},\]
where we applied $2k$ convolutions in total. Apply Theorem~\ref{AlmostLinftyApproxThm} with $\varepsilon =  \delta^{2^{2k}} / 2$. Then
\[\on{supp} \bigconv{k} \dots \bigconv{1}\bigconv{k} \dots \bigconv{1} X \subset X\]
contains a non-empty variety $V$ of codimension $r = \exp^{(O(1))}(O(\delta^{-1}))$. Apply Lemma~\ref{mlSetVarML} to get a multilinear non-empty variety $\tilde{V}$ of codimension $O(r)$ inside $X$.\end{proof}

In the case of Theorem~\ref{alldirsubspacestheorem}, unlike other results in this paper, the underlying field plays a non-trivial role and we deduce the following corollary for the case of general finite fields. Let $\mathbb{F}$ be an arbitrary field of characteristic $p$; we view $\mathbb{F}_p$ as a subfield of $\mathbb{F}$.

\begin{corollary}[General finite fields version] Let $G_1,\dots,G_k$ be $\mathbb{F}$-vector spaces and let $X \subset G_{[k]}$ be a set of density $\delta > 0$ such that for each $d \in [k]$ and each $x_{[k] \setminus \{d\}} \in G_{[k] \setminus \{d\}}$, the slice $X_{x_{[k] \setminus \{d\}}}$ is a (possibly empty) $\mathbb{F}$-subspace. Then, there exist a positive integer $r =\exp^{(O(1))}(O(\delta^{-1}))$ and an $\mathbb{F}_p$-mixed-linear map $\theta \colon G_{[k]} \to \mathbb{F}_p^r$ such that 
\[\bigcap_{d \in [k], \lambda \in \mathbb{F} \setminus \{0\}} \{x_{[k]} \in G_{[k]} \colon \theta(x_{[d-1]}, \lambda x_d, x_{[d+1,k]}) = 0\} \subset X.\]\end{corollary}

This is best possible -- it is easy to check that the set on the left-hand-side of the expression in the conclusion of the corollary is indeed an $\mathbb{F}$-subspace in each direction. It is naturally dense, since it is a non-empty $\mathbb{F}_p$-variety of bounded codimension.

\begin{proof}Apply Theorem~\ref{alldirsubspacestheorem} and let $\theta$ be an $\mathbb{F}_p$-mixed-linear map $\theta \colon G_{[k]} \to \mathbb{F}_p^r$ such that $\theta^{-1}(0) \subset X$, where $r = \exp^{(O(1))}(O(\delta^{-1}))$. Then $X$ contains the variety
\[\bigcap_{d \in [k], \lambda \in \mathbb{F} \setminus \{0\}} \{x_{[k]} \in G_{[k]} \colon \theta(x_{[d-1]}, \lambda x_d, x_{[d+1,k]}) = 0\}\]
as well.\end{proof}

\end{document}